\documentclass[9pt]{amsart}
\usepackage{amsmath}
\usepackage{amsxtra}
\usepackage{amscd}
\usepackage{amsthm}
\usepackage{amsfonts}
\usepackage{amssymb}
\usepackage{eucal}
\usepackage[all]{xy}
\usepackage{graphicx}

\newtheorem{cor}[subsubsection]{Corollary}
\newtheorem{lem}[subsubsection]{Lemma}
\newtheorem{prop}[subsubsection]{Proposition}
\newtheorem{thmconstr}[subsubsection]{Theorem-Construction}

\newtheorem{thm}[subsubsection]{Theorem}

\newtheorem{defn}[subsubsection]{Definition}

\theoremstyle{definition}

\theoremstyle{remark}
\newtheorem{rem}[subsubsection]{Remark}

\newcommand{\thmref}[1]{Theorem~\ref{#1}}
\newcommand{\thmmref}[1]{Thm.~\ref{#1}}

\newcommand{\secref}[1]{Sect.~\ref{#1}}
\newcommand{\lemref}[1]{Lemma~\ref{#1}}
\newcommand{\propref}[1]{Proposition~\ref{#1}}
\newcommand{\corrref}[1]{Cor.~\ref{#1}}
\newcommand{\corref}[1]{Corrollary~\ref{#1}}

\numberwithin{equation}{section}

\newcommand{\nc}{\newcommand}
\nc{\renc}{\renewcommand}
\nc{\ssec}{\subsection}
\nc{\sssec}{\subsubsection}
\nc{\on}{\operatorname}

\nc\ol{\overline}
\nc\wt{\widetilde}
\nc\tboxtimes{\wt{\boxtimes}}
\nc\tstar{\wt{\star}}
\nc{\alp}{a}

\nc{\ZZ}{{\mathbb Z}}
\nc{\NN}{{\mathbb N}}
\nc{\OO}{{\mathbb O}}
\renc{\SS}{{\mathbb S}}
\nc{\DD}{{\mathbb D}}
\nc{\GG}{{\mathbb G}}

\nc{\Fq}{{\mathbb F}_q}
\nc{\Fqb}{\ol{{\mathbb F}_q}}
\nc{\Ql}{\ol{{\mathbb Q}_\ell}}
\nc{\id}{\text{id}}
\nc\X{\mathcal X}

\nc{\Hom}{\on{Hom}}
\nc{\Lie}{\on{Lie}}
\nc{\Loc}{\on{Loc}}
\nc{\Pic}{\on{Pic}}
\nc{\Bun}{\on{Bun}}
\nc{\IC}{\on{IC}}
\nc{\ICs}{\on{IC}^{\frac{\infty}{2}}}
\nc{\bICs}{\overset{\bullet}{\on{IC}}{}^{\frac{\infty}{2}}}
\nc{\tICs}{\wt{\on{IC}}{}^{\frac{\infty}{2}}}
\nc{\ICsg}{\on{IC}^{\gamma+\frac{\infty}{2}}}
\nc{\ICsl}{\on{IC}^{\lambda+\frac{\infty}{2}}}
\nc{\ICslm}{\on{IC}^{\lambda+\frac{\infty}{2},-}}
\nc{\ICsm}{\on{IC}^{\frac{\infty}{2},-}}
\nc{\bICsm}{\overset{\bullet}{\on{IC}}{}^{\frac{\infty}{2},-}}
\nc{\Aut}{\on{Aut}}
\nc{\rk}{\on{rk}}
\nc{\Sh}{\on{Sh}}
\nc{\Perv}{\on{Perv}}
\nc{\pos}{{\on{pos}}}
\nc{\Conv}{\on{Conv}}
\nc{\Sph}{\on{Sph}}
\nc{\Sat}{\on{Sat}}
\nc{\Sym}{\on{Sym}}
\nc{\BunBb}{\overline{\Bun}_B}
\nc{\BunNb}{\overline{\Bun}_N}
\nc{\BunTb}{\overline{\Bun}_T}
\nc{\BunBbm}{\overline{\Bun}_{B^-}}
\nc{\BunBbel}{\overline{\Bun}_{B,el}}
\nc{\BunBbmel}{\overline{\Bun}_{B^-,el}}
\nc{\Buno}{\overset{o}{\Bun}}
\nc{\BunPb}{{\overline{\Bun}_P}}
\nc{\BunBM}{\Bun_{B(M)}}
\nc{\BunBMb}{\overline{\Bun}_{B(M)}}
\nc{\BunPbw}{{\widetilde{\Bun}_P}}
\nc{\BunBP}{\widetilde{\Bun}_{B,P}}
\nc{\GUb}{\overline{G/U}}
\nc{\GUPb}{\overline{G/U(P)}}

\nc{\Hhom}{\underline{\on{Hom}}}
\nc\syminfty{\on{Sym}^{\infty}}
\nc\lal{\ol{\kappa_x}}
\nc\xl{\ol{x}}
\nc\thl{\ol{\theta}}
\nc\nul{\ol{\nu}}
\nc\lambdal{\ol{\lambda}}
\nc\Sum\Sigma
\nc{\oX}{\overset{o}{X}{}}
\nc{\hl}{\overset{\leftarrow}h{}}
\nc{\hr}{\overset{\rightarrow}h{}}
\nc{\M}{{\mathcal M}}
\nc{\N}{{\mathcal N}}
\nc{\F}{{\mathcal F}}
\nc{\D}{{\mathcal D}}
\nc{\Q}{{\mathcal Q}}
\nc{\Y}{{\mathcal Y}}
\nc{\G}{{\mathcal G}}
\nc{\E}{{\mathcal E}}
\nc{\CalC}{{\mathcal C}}
\nc\Dh{\widehat{\D}}

\nc{\C}{{\mathcal C}}
\nc{\K}{{\mathcal K}}
\renewcommand{\H}{{\mathcal H}}

\nc{\T}{{\mathcal T}}
\nc{\V}{{\mathcal V}}
\renc{\P}{{\mathcal P}}
\nc{\A}{{\mathcal A}}
\nc{\B}{{\mathcal B}}
\nc{\U}{{\mathcal U}}

\nc{\Gr}{{\on{Gr}}}

\nc{\frn}{{\check{\mathfrak u}(P)}}

\nc{\fC}{\mathfrak C}
\nc{\p}{\mathfrak p}
\nc{\q}{\mathfrak q}
\nc\f{{\mathfrak f}}

\nc{\qo}{{\mathfrak q}}
\nc{\po}{{\mathfrak p}}
\nc{\s}{{\mathfrak s}}
\nc\w{\text{w}}

\renewcommand{\mod}{{\on{-mod}}}

\nc\mathi\iota
\nc\Spec{\on{Spec}}
\nc\Mod{\on{Mod}}
\nc{\tw}{\widetilde{\mathfrak t}}
\nc{\pw}{\widetilde{\mathfrak p}}
\nc{\qw}{\widetilde{\mathfrak q}}
\nc{\jw}{\widetilde j}

\nc{\grb}{\overline{\Gr}}
\nc{\I}{\mathcal I}

\nc{\kappach}{{\check\kappa_x}}
\nc{\Lambdach}{{\check\Lambda}{}}
\nc{\lambdach}{{\check\lambda}}
\nc{\omegach}{{\check\omega}}
\nc{\nuch}{{\check\nu}}
\nc{\etach}{{\check\eta}}
\nc{\alphach}{{\checka}}
\nc{\oblvtach}{{\check\oblvta}}
\nc{\pich}{{\check\pi}}
\nc{\ch}{{\check h}}

\nc{\Hb}{\overline{\H}}

\nc{\BA}{{\mathbb{A}}}
\nc{\BC}{{\mathbb{C}}}
\nc{\BE}{{\mathbb{E}}}
\nc{\BF}{{\mathbb{F}}}
\nc{\BG}{{\mathbb{G}}}
\nc{\BM}{{\mathbb{M}}}
\nc{\BO}{{\mathbb{O}}}
\nc{\BD}{{\mathbb{D}}}
\nc{\BN}{{\mathbb{N}}}
\nc{\BP}{{\mathbb{P}}}
\nc{\BQ}{{\mathbb{Q}}}
\nc{\BR}{{\mathbb{R}}}
\nc{\BZ}{{\mathbb{Z}}}
\nc{\BS}{{\mathbb{S}}}

\nc{\CA}{{\mathcal{A}}}
\nc{\CB}{{\mathcal{B}}}

\nc{\CE}{{\mathcal{E}}}
\nc{\CF}{{\mathcal{F}}}
\nc{\CG}{{\mathcal{G}}}
\nc{\CH}{{\mathcal{H}}}

\nc{\CL}{{\mathcal{L}}}
\nc{\CC}{{\mathcal{C}}}
\nc{\CM}{{\mathcal{M}}}
\nc{\CN}{{\mathcal{N}}}
\nc{\CK}{{\mathcal{K}}}
\nc{\CO}{{\mathcal{O}}}
\nc{\CP}{{\mathcal{P}}}
\nc{\CQ}{{\mathcal{Q}}}
\nc{\CR}{{\mathcal{R}}}
\nc{\CS}{{\mathcal{S}}}
\nc{\CT}{{\mathcal{T}}}
\nc{\CU}{{\mathcal{U}}}
\nc{\CV}{{\mathcal{V}}}
\nc{\CW}{{\mathcal{W}}}
\nc{\CX}{{\mathcal{X}}}
\nc{\CY}{{\mathcal{Y}}}
\nc{\CZ}{{\mathcal{Z}}}
\nc{\CI}{{\mathcal{I}}}
\nc{\CJ}{{\mathcal{J}}}

\nc{\csM}{{\check{\mathcal A}}{}}
\nc{\oM}{{\overset{\circ}{\mathcal M}}{}}
\nc{\obM}{{\overset{\circ}{\mathbf M}}{}}
\nc{\oCA}{{\overset{\circ}{\mathcal A}}{}}
\nc{\obA}{{\overset{\circ}{\mathbf A}}{}}
\nc{\ooM}{{\overset{\circ}{M}}{}}
\nc{\osM}{{\overset{\circ}{\mathsf M}}{}}
\nc{\vM}{{\overset{\bullet}{\mathcal M}}{}}
\nc{\nM}{{\underset{\bullet}{\mathcal M}}{}}
\nc{\oD}{{\overset{\circ}{\mathcal D}}{}}
\nc{\obD}{{\overset{\circ}{\mathbf D}}{}}
\nc{\oA}{{\overset{\circ}{\mathbb A}}{}}
\nc{\op}{{\overset{\bullet}{\mathbf p}}{}}
\nc{\cp}{{\overset{\circ}{\mathbf p}}{}}
\nc{\oU}{{\overset{\bullet}{\mathcal U}}{}}
\nc{\oZ}{{\overset{\circ}{\mathcal Z}}{}}
\nc{\ofZ}{{\overset{\circ}{\mathfrak Z}}{}}
\nc{\oF}{{\overset{\circ}{\fF}}}

\nc{\fa}{{\mathfrak{a}}}
\nc{\fb}{{\mathfrak{b}}}
\nc{\fd}{{\mathfrak{d}}}
\nc{\ff}{{\mathfrak{f}}}
\nc{\fg}{{\mathfrak{g}}}
\nc{\fgl}{{\mathfrak{gl}}}
\nc{\fh}{{\mathfrak{h}}}
\nc{\fj}{{\mathfrak{j}}}
\nc{\fl}{{\mathfrak{l}}}
\nc{\fm}{{\mathfrak{m}}}
\nc{\fn}{{\mathfrak{n}}}
\nc{\fu}{{\mathfrak{u}}}
\nc{\fp}{{\mathfrak{p}}}
\nc{\fr}{{\mathfrak{r}}}
\nc{\fs}{{\mathfrak{s}}}
\nc{\ft}{{\mathfrak{t}}}
\nc{\fz}{{\mathfrak{z}}}
\nc{\fsl}{{\mathfrak{sl}}}
\nc{\hsl}{{\widehat{\mathfrak{sl}}}}
\nc{\hgl}{{\widehat{\mathfrak{gl}}}}
\nc{\hg}{{\widehat{\mathfrak{g}}}}
\nc{\chg}{{\widehat{\mathfrak{g}}}{}^\vee}
\nc{\hn}{{\widehat{\mathfrak{n}}}}
\nc{\chn}{{\widehat{\mathfrak{n}}}{}^\vee}

\nc{\fA}{{\mathfrak{A}}}
\nc{\fB}{{\mathfrak{B}}}
\nc{\fD}{{\mathfrak{D}}}
\nc{\fE}{{\mathfrak{E}}}
\nc{\fF}{{\mathfrak{F}}}
\nc{\fG}{{\mathfrak{G}}}
\nc{\fK}{{\mathfrak{K}}}
\nc{\fJ}{{\mathfrak{J}}}
\nc{\fL}{{\mathfrak{L}}}
\nc{\fM}{{\mathfrak{M}}}
\nc{\fN}{{\mathfrak{N}}}
\nc{\fP}{{\mathfrak{P}}}
\nc{\fU}{{\mathfrak{U}}}
\nc{\fV}{{\mathfrak{V}}}
\nc{\fZ}{{\mathfrak{Z}}}

\nc{\ba}{{\mathbf{a}}}
\nc{\bb}{{\mathbf{b}}}
\nc{\bc}{{\mathbf{c}}}
\nc{\bd}{{\mathbf{d}}}
\nc{\bbf}{{\mathbf{f}}}
\nc{\be}{{\mathbf{e}}}
\nc{\bi}{{\mathbf{i}}}
\nc{\bj}{{\mathbf{j}}}
\nc{\bn}{{\mathbf{n}}}
\nc{\bo}{{\mathbf{o}}}
\nc{\bp}{{\mathbf{p}}}
\nc{\bq}{{\mathbf{q}}}
\nc{\bu}{{\mathbf{u}}}
\nc{\bv}{{\mathbf{v}}}
\nc{\bx}{{\mathbf{x}}}
\nc{\bs}{{\mathbf{s}}}
\nc{\by}{{\mathbf{y}}}
\nc{\bw}{{\mathbf{w}}}
\nc{\bA}{{\mathbf{A}}}
\nc{\bK}{{\mathbf{K}}}
\nc{\bB}{{\mathbf{B}}}
\nc{\bF}{{\mathbf{F}}}
\nc{\bC}{{\mathbf{C}}}
\nc{\bG}{{\mathbf{G}}}
\nc{\bD}{{\mathbf{D}}}
\nc{\bE}{{\mathbf{E}}}
\nc{\bH}{{\mathbf{H}}}
\nc{\bI}{{\mathbf{I}}}
\nc{\bM}{{\mathbf{M}}}
\nc{\bN}{{\mathbf{N}}}
\nc{\bO}{{\mathbf{O}}}
\nc{\bV}{{\mathbf{V}}}
\nc{\bW}{{\mathbf{W}}}
\nc{\bX}{{\mathbf{X}}}
\nc{\bZ}{{\mathbf{Z}}}
\nc{\bS}{{\mathbf{S}}}

\nc{\sA}{{\mathsf{A}}}
\nc{\sB}{{\mathsf{B}}}
\nc{\sC}{{\mathsf{C}}}
\nc{\sD}{{\mathsf{D}}}
\nc{\sF}{{\mathsf{F}}}
\nc{\sK}{{\mathsf{K}}}
\nc{\sM}{{\mathsf{M}}}
\nc{\sO}{{\mathsf{O}}}
\nc{\sW}{{\mathsf{W}}}
\nc{\sQ}{{\mathsf{Q}}}
\nc{\sP}{{\mathsf{P}}}
\nc{\sZ}{{\mathsf{Z}}}
\nc{\sV}{{\mathsf{V}}}
\nc{\sr}{{\mathsf{r}}}
\nc{\bk}{{\mathsf{k}}}
\nc{\sg}{{\mathsf{g}}}
\nc{\sff}{{\mathsf{f}}}
\nc{\sfe}{{\mathsf{e}}}
\nc{\sfj}{{\mathsf{j}}}
\nc{\sfi}{{\mathsf{i}}}
\nc{\sfb}{{\mathsf{b}}}
\nc{\sfc}{{\mathsf{c}}}
\nc{\sd}{{\mathsf{d}}}
\nc{\sv}{{\mathsf{v}}}
\nc{\sw}{{\mathsf{w}}}

\nc{\BK}{{\bar{K}}}

\nc{\tA}{{\widetilde{\mathbf{A}}}}
\nc{\tB}{{\widetilde{\mathcal{B}}}}
\nc{\tg}{{\widetilde{\mathfrak{g}}}}
\nc{\tG}{{\widetilde{G}}}
\nc{\TM}{{\widetilde{\mathbb{M}}}{}}
\nc{\tO}{{\widetilde{\mathsf{O}}}{}}
\nc{\tU}{{\widetilde{\mathfrak{U}}}{}}
\nc{\TZ}{{\tilde{Z}}}
\nc{\tx}{{\tilde{x}}}
\nc{\tbv}{{\tilde{\bv}}}
\nc{\tfP}{{\widetilde{\mathfrak{P}}}{}}
\nc{\tz}{{\tilde{\zeta}}}
\nc{\tmu}{{\tilde{\lambda}}}

\nc{\urho}{\underline{\pi}}
\nc{\uB}{\underline{B}}
\nc{\uC}{{\underline{\mathbb{C}}}}
\nc{\ui}{\underline{i}}
\nc{\uj}{\underline{j}}
\nc{\ofP}{{\overline{\mathfrak{P}}}}
\nc{\oB}{{\overline{\mathcal{B}}}}
\nc{\og}{{\overline{\mathfrak{g}}}}
\nc{\oI}{{\overline{I}}}

\nc{\eps}{\varepsilon}
\nc{\hrho}{{\hat{\pi}}}

\nc{\one}{{\mathbf{1}}}
\nc{\two}{{\mathbf{t}}}

\nc{\Rep}{{\mathop{\operatorname{\rm Rep}}}}
\nc{\Tot}{{\mathop{\operatorname{\rm Tot}}}}
\nc{\Ker}{{\mathop{\operatorname{\rm Ker}}}}
\nc{\Hilb}{{\mathop{\operatorname{\rm Hilb}}}}
\nc{\End}{{\mathop{\operatorname{\rm End}}}}
\nc{\Ext}{{\mathop{\operatorname{\rm Ext}}}}
\nc{\CHom}{{\mathop{\operatorname{{\mathcal{H}}\it om}}}}
\nc{\GL}{{\mathop{\operatorname{\rm GL}}}}
\nc{\gr}{{\mathop{\operatorname{\rm gr}}}}
\nc{\Id}{{\mathop{\operatorname{\rm Id}}}}
\nc{\de}{{\mathop{\operatorname{\rm def}}}}
\nc{\length}{{\mathop{\operatorname{\rm length}}}}
\nc{\supp}{{\mathop{\operatorname{\rm supp}}}}

\nc{\Cliff}{{\mathsf{Cliff}}}
\nc{\Fl}{\on{Fl}}
\nc{\Fib}{{\mathsf{Fib}}}
\nc{\Coh}{{\mathsf{Coh}}}
\nc{\QCoh}{{\on{QCoh}}}
\nc{\IndCoh}{{\on{IndCoh}}}
\nc{\FCoh}{{\mathsf{FCoh}}}

\nc{\reg}{{\text{\rm reg}}}

\nc{\cplus}{{\mathbf{C}_+}}
\nc{\cminus}{{\mathbf{C}_-}}
\nc{\cthree}{{\mathbf{C}_*}}
\nc{\Qbar}{{\bar{Q}}}
\nc\Eis{\on{Eis}}
\nc\Eisb{\ol\Eis{}}
\nc\Eisr{\on{Eis}^{rat}{}}
\nc\wh{\widehat}
\nc{\Def}{\on{Def_{\check{\fb}}(E)}}
\nc{\barZ}{\overline{Z}{}}
\nc{\barbarZ}{\overline{\barZ}{}}
\nc{\barpi}{\overline\iota}
\nc{\barbarpi}{\overline\barpi}
\nc{\barpip}{\overline\iota{}^+}
\nc{\barpim}{\overline\iota{}^-}

\nc{\fq}{\mathfrak q}

\nc{\fqb}{\ol{\fq}{}}
\nc{\fpb}{\ol{\fp}{}}
\nc{\fpr}{{\fp^{rat}}{}}
\nc{\fqr}{{\fq^{rat}}{}}

\nc{\hattimes}{\wh\otimes}

\nc{\bh}{{\bar{h}}}
\nc{\bOmega}{{\overline{\Omega(\check \fn)}}}

\nc{\seq}[1]{\stackrel{#1}{\sim}}

\nc{\cT}{{\check{T}}}
\nc{\cG}{{\check{G}}}
\nc{\cM}{{\check{M}}}
\nc{\cB}{{\check{B}}}
\nc{\cN}{{\check{N}}}

\nc{\ct}{{\check{\mathfrak t}}}
\nc{\cg}{{\check{\fg}}}
\nc{\cb}{{\check{\fb}}}
\nc{\cn}{{\check{\fn}}}

\nc{\cLambda}{{\check\Lambda}}

\nc{\cla}{{\check\kappa_x}}
\nc{\cmu}{{\check\lambda}}
\nc{\clambda}{{\check\lambda}}
\nc{\cnu}{{\check\nu}}
\nc{\ceta}{{\check\eta}}

\nc{\DefbE}{{\on{Def}_{\cB}(E_\cT)}}

\nc{\imathb}{{\ol{\imath}}}

\nc{\KG}{K\backslash G}
\nc{\comult}{{co\text{-}mult}}
\nc{\counit}{{co\text{-}unit}}
\nc{\uHom}{{\underline{\Maps}}}
\nc{\dgSch}{\on{Sch}}
\nc{\Sch}{\on{Sch}}
\nc{\affdgSch}{\on{Sch}^{\on{aff}}}
\nc{\affSch}{\on{Sch}^{\on{aff}}}
\nc{\Groupoids}{\on{Grpd}}
\nc{\inftygroup}{\on{Spc}}
\nc{\inftyCat}{\infty\on{-Cat}}
\nc{\StinftyCat}{\inftyCat^{\on{St}}}
\nc{\MoninftyCat}{\infty\on{-Cat}^{\on{Mon}}}
\nc{\SymMoninftyCat}{\infty\on{-Cat}^{\on{SymMon}}}
\nc{\SymMonStinftyCat}{\on{DGCat}^{\on{SymMon}}}
\nc{\MonStinftyCat}{\on{DGCat}^{\on{Mon}}}
\nc{\inftystack}{\on{Stk}}
\nc{\inftystackalg}{Stk^{1\text{-}alg}}
\nc{\inftyprestack}{\on{PreStk}}
\nc{\inftydgnearstack}{\on{NearStk}}
\nc{\inftydgstack}{\on{Stk}}
\nc{\inftydgstackalg}{DGStk^{1\text{-}alg}}
\nc{\inftydgprestack}{\on{PreStk}}
\nc{\dgindSch}{\on{indSch}}
\nc{\indSch}{{}^{\on{cl}}\!\on{indSch}}
\nc{\infSch}{\on{infSch}}
\nc{\dr}{{\on{dR}}}

\nc{\mmod}{{\on{-}\!{\mathbf{mod}}}}

\nc{\starr}{\text{\dh}}
\nc{\Spectra}{\on{Spectra}}
\nc{\Crys}{\on{Crys}}
\nc{\oblv}{{\mathbf{oblv}}}
\nc{\ind}{{\mathbf{ind}}}
\nc{\coind}{{\mathbf{coind}}}
\nc{\inv}{{\mathbf{inv}}}
\nc{\triv}{{\mathbf{triv}}}
\nc{\CMaps}{{\mathcal Maps}}
\nc{\Maps}{\on{Maps}}
\nc{\bMaps}{\mathbf{Maps}}
\nc{\BMaps}{\ul{\on{Maps}}}
\nc{\Grid}{\on{Grid}}
\nc{\hGrid}{\on{Grid}^{\geq\,\on{dgnl}}}
\nc{\Diag}{\on{Diag}}
\nc{\bDelta}{\mathbf{\Delta}}
\nc{\tCateg}{(\infty\on{-2)-Cat}}
\nc{\ul}{\underline}
\nc{\Seg}{\on{Seq}}

\nc{\triSeg}{\on{tri-Seq}}
\nc{\quadSeg}{\on{quad-Seq}}
\nc{\nSeg}{\on{n-Seq}}
\nc{\Segm}{\on{Seg}^{\on{mkd}}}
\nc{\fLm}{\fL^{\on{mkd}}}
\nc{\inftyCatm}{\inftyCat^{\on{mkd}}}
\nc{\Blocks}{\mathbf{Blocks}}
\nc{\Snakes}{\mathbf{Snakes}}

\nc{\Sets}{\on{Sets}}
\nc{\Ran}{{\on{Ran}}}
\nc{\Vect}{\on{Vect}}
\nc{\Shv}{\on{Shv}}
\nc{\unn}{\mathbf{union}}
\nc{\Spc}{\on{Spc}}
\nc{\ppart}{(\!(t)\!)}
\nc{\qqart}{[\![t]\!]}
\nc{\Dmod}{\on{D-mod}}
\nc{\cD}{\mathcal D}
\nc{\ocD}{\overset{\circ}{\cD}}
\nc{\sfp}{\mathsf{p}}
\nc{\sfq}{\mathsf{q}}
\nc{\DGCat}{\on{DGCat}}
\renc{\det}{\on{det}}
\nc{\Conf}{\on{Conf}}
\nc{\Whit}{\on{Whit}}
\nc{\Reg}{\on{Reg}}
\nc{\Res}{\on{Res}}
\nc{\BunNbom}{\overline\Bun_N^{\omega^\rho}} 
\nc{\BunNbox}{(\overline\Bun_N^{\omega^\rho})_{\infty\cdot x}} 
\nc{\BunNmbox}{(\overline\Bun_{N^-}^{\omega^\rho})_{\infty\cdot x}}
\nc{\Hecke}{\on{Hecke}}
\nc{\BHecke}{B\on{-Hecke}}
\nc{\BmHecke}{B^-\on{-Hecke}}
\nc{\bHecke}{\overset{\bullet}{\on{Hecke}}}
\nc{\bCZ}{\ol\CZ}
\nc{\oCZ}{\overset{\circ}\CZ} 
\nc{\boCZ}{\ol{\oCZ}}
\nc{\sotimes}{\overset{!}\otimes}
\nc{\SI}{\on{SI}}
\nc{\semiinf}{{\frac{\infty}{2}}}
\nc{\coInd}{\on{coInd}}
\nc{\Ind}{\on{Ind}}
\nc{\bCM}{\overset{\bullet}\CM{}}

\begin{document}

\title[The ``small" FLE]
{Metaplectic Whittaker category and quantum groups : \\
the ``small" FLE}

\author{D. Gaitsgory and S.~Lysenko}

\date{\today}

\maketitle

\tableofcontents

\section*{Introduction}

\ssec{Towards the FLE}

The present work lies in within the general paradigm of \emph{quantum geometric Langlands theory}. 
In the same way as the usual (i.e., non-quantum) geometric Langlands theory, arguably, originates from the
geometric Satake equivalence, the quantum geometric Langlands theory originates from the FLE, the
\emph{fundamental local equivalence}. In this subsection we will review what the FLE says. 

\sssec{}

The original FLE is a statement within the theory of D-modules, so it takes place over an algebraically closed
field $k$ of characteristic zero. Let $G$ be a reductive group and let $\cG$ be its Langlands dual, both considered
as groups over $k$. Let $\Lambda$ denote the coweight lattice of $G$, and let $\cLambda$ be the dual lattice,
which is by definition the coweight lattice of $\cG$. 

\medskip

We fix a \emph{level} for $G$, which is by definition, a Weyl group-invariant symmetric bilinear form
$$\kappa:\Lambda\otimes \Lambda\to k,$$

We will assume that resulting (symmetric) bilinear form
\begin{equation} \label{e:level k}
\ft\underset{k}\otimes \ft\to k
\end{equation}
is \emph{non-degenerate}, where $\ft:=k\underset{\BZ}\otimes \Lambda$ is the Lie algebra of the Cartan subgroup $T\subset G$.
The non-degeneracy of the pairing \eqref{e:level k} means that the resulting map
$$\ft\to \ft^*=:\check{\ft}$$
is an isomorphism. Consider the inverse map 
$$\check\ft\to \ft\simeq \check\ft^*,$$
which we can interpret as a (symmetric) bilinear form
\begin{equation} \label{e:dual level k}
\check\ft\underset{k}\otimes \check\ft\to k,
\end{equation}
or a symmetric bilinear form 
$$\check\kappa:\check\Lambda\otimes \check\Lambda\to k.$$

We will refer to $\check\kappa$ as the level \emph{dual} dual to $\kappa$. 

\sssec{}

We turn \eqref{e:level k} and \eqref{e:dual level k}
into Ad-invariant symmetric bilinear forms 
$$(-,-)_\kappa:\fg\otimes \fg\to k \text{ and } (-,-)_{\check\kappa}:\cg\otimes \cg\to k,$$
respectively, where the correspondence 
$$\on{SymBilin}(\fg,k)^G \, \Leftrightarrow \, \on{SymBilin}(\ft,k)^W$$
is given by
$$(-,-)\mapsto (-,-)|_{\ft}+\frac{(-,-)_{\on{Kil}}}{2},$$
where $(-,-)_{\on{Kil}}$ is the Killing form, and similarly for $\cg$.

\sssec{}

The form $(-,-)_\kappa$ gives rise to a Kac-Moody extension 
$$0\to k\to \hg_\kappa\to \fg\ppart\to 0$$
and to the category of twisted D-modules $\Dmod_{\kappa}(\Gr_G)$ on the affine Grassmannian $\Gr_G=G\ppart/G\qqart$ of
$G$. 

\medskip

We consider two categories associated with the above data:
\begin{equation} \label{e:two categories}
\on{KL}_\kappa(G):=\hg_\kappa\mod^{G\qqart}  \text{ and } \Whit_\kappa(G):=\Dmod_{\kappa}(\Gr_G)^{N\ppart,\chi}.
\end{equation}

Here the superscript $G\qqart$ stands for $G\qqart$-equivariance, i.e., $\on{KL}_\kappa(G)$ is modules over the
Harish-Chandra pair $(\hg_\kappa,G\qqart)$. The superscript $N\ppart,\chi$ stands for equivariance for $N\ppart$ 
against the non-degenerate character $\chi$ (see \secref{ss:def Whit} for the detailed definition). 

\medskip

We consider the similar categories for the Langlands dual group
$$\on{KL}_{\check\kappa}(\cG) \text{ and } \Whit_{\check\kappa}(\cG).$$

\sssec{}

We are now ready to state the FLE. It says that we have a canonical equivalence
\begin{equation} \label{e:FLE}
\Whit_\kappa(G) \simeq \on{KL}_{\check\kappa}(\cG). 
\end{equation} 

Symmetrically, we are also supposed to have an equivalence
$$\Whit_{\check\kappa}(\cG) \simeq \on{KL}_{\kappa}(G).$$

\medskip

The equivalence \eqref{e:FLE} is still conjectural, and the present work may be regarded as a step 
towards its proof. 

\sssec{}  \label{sss:fact categ}

The categories appearing on the two sides of \eqref{e:two categories} are not mere (DG) categories, but they carry 
extra structure.  Namely, the pair $k\ppart\supset k\qqart$ that appears in the definition of both sides
should be thought of as attached to a point $x$ on a curve $X$, where $t$ is a local parameter at $x$. 

\medskip

Each of the categories appearing in \eqref{e:two categories} has a structure of \emph{factorization category},
i.e., it can be more generally attached to a finite collection of points $\ul{x}\subset X$, i.e., a point of the
Ran space of $X$ (see \secref{ss:Ran}),
\begin{equation} \label{e:factorization category 1}
\ul{x}\rightsquigarrow \CC_{\ul{x}},
\end{equation}
and we have a system of isomorphisms 
\begin{equation} \label{e:factorization category 2}
\CC_{\ul{x}{}_1\sqcup \ul{x}{}_2}\simeq \CC_{\ul{x}{}_1}\otimes \CC_{\ul{x}{}_2},
\end{equation}
whenever $\ul{x}{}_1$ and $\ul{x}{}_2$ are \emph{disjoint}. 

\medskip

The additional structure involved in \eqref{e:FLE} is that it is supposed to be an equivalence of factorization categories. 

\ssec{Quantum groups perspective}

There is a third (and eventually there will be also a fourth) player in the equivalence \eqref{e:FLE}. This
third player is the category
$$\Rep_q(\cG),$$
which is the category of representations of the ``big" (i.e., Lusztig's quantum group). We will now explain 
how it fits into the picture.

\sssec{}

The category $\Rep_q(\cG)$ is of algebraic nature and can be defined over an arbitrary field of coefficients $\sfe$
(say, also assumed algebraically closed and of characterstic zero). 
The structure that $\Rep_q(\cG)$ possesses is that if a \emph{braided monoidal category}.

\medskip

Here the quantum parameter $q$ is a quadratic form on $\Lambda$ (which is the weight lattice for $\cG$)
with values in $\sfe^\times$. 

\sssec{}  \label{sss:factorization vs braided monoidal}

Assume now that $\sfe=\BC=k$. Take the curve $X$ to be $\BA^1$ and $x=0\in \BA^1$. In this case, for a given
DG category $\CC$, a braided monoidal structure on it (under appropriate finiteness conditions), via Riemann-Hilbert correspondence
gives rise to a factorization category (see \secref{sss:fact categ} above) with $\CC=\CC_{\{0\}}$. 

\sssec{}

Now, the equivalence established in the series of papers \cite{KL} can be formulated as saying that the factorization category 
corresponding by the above procedure to $\Rep_q(\cG)$ identifies with $\on{KL}_{\check\kappa}(\cG)$ for the quantum parameter
$$q=\on{exp}(2\pi i q_{\kappa}),$$
where $q_\kappa$ is the quadratic form $\Lambda\to k$ such that the associated bilinear form is $\kappa$, i.e.,
$$q_\kappa(\lambda)=\frac{\kappa(\lambda,\lambda)}{2}.$$

\medskip

In particular, we have an equivalence as plain DG categories
\begin{equation} \label{e:KL}
\on{KL}_{\check\kappa}(\cG) \simeq \Rep_q(\cG).
\end{equation}

\begin{rem}
Note that the equivalence \eqref{e:KL} does \emph{not} involve Langlands duality: on both sides we are dealing
with the same reductive group, in this case $\cG$.
\end{rem} 

\sssec{}

Thus, combining, for $\sfe=\BC=k$ we are supposed to have the equivalences
\begin{equation}  \label{e:three equivalence}
\Whit_\kappa(G) \simeq \on{KL}_{\check\kappa}(\cG)\simeq \Rep_q(\cG) . 
\end{equation} 

What we prove in this work is a result that goes a long way towards the composite equivalence
\begin{equation}  \label{e:Whit vs quant}
\Whit_\kappa(G) \simeq \Rep_q(\cG). 
\end{equation} 

The precise statement of what we actually prove will be explained in the rest of this Introduction. Here let us note the
following two of its features:

\begin{itemize}

\item The right-hand side of \emph{our} equivalence will not be $\Rep_q(\cG)$, but rather $\overset{\bullet}\fu_q(\cG)\mod$,
the category of modules over the \emph{small} quantum group. 

\item Our equivalence will be geometric (or motivic) in nature in that it will not be tied to the situation of $\sfe=\BC=k$
and neither will it rely on Riemann-Hilbert. Rather, it applies over any ground field and for an arbitrary sheaf theory
(see \secref{sss:sheaf theory}). 

\end{itemize} 

\medskip

Let us also add that the equivalence \eqref{e:Whit vs quant} had been conjectured by J.~Lurie and the first-named author
around 2007; it stands at the origin of the FLE. 

\sssec{}

One could of course try to prove the FLE \eqref{e:FLE} by combining the Kazhdan-Lusztig 
equivalence \eqref{e:KL} with the (yet to be established) equivalence \eqref{e:Whit vs quant}. 
But this is not how we plan to proceed about proving the FLE. Nor do quantum groups appear 
in the statement of the main theorem of the present work, \thmref{t:main}. 

\medskip

So in a sense, the equivalences with the category of modules over the quantum group is just an 
add-on to the FLE. Yet, it is a very useful add-on: quantum groups are ``much more finite-dimensional",
than the other objects involved, so perceiving things from the perspective of quantum groups
is convenient in that it really explains ``what is going on". We will resort to this perspective on multiple
occasions in this Introduction (notably, in \secref{ss:how}) as a guiding principle to some of the key 
constructions. 

\ssec{Let's prove the FLE: the Jacquet functor}

We will now outline a strategy towards the proof of the FLE. This strategy will not quite work
(rather one needs to do more work in order to make it work, and this will be done in a future publication).
Yet, this strategy will explain the context for what we will do in the main body of the text. 

\sssec{}

The inherent difficulty in establishing the equivalence \eqref{e:FLE} is that it involves Langlands duality:
each side is a category extracted from the geometry of the corresponding reductive group (i.e., $G$ and $\cG$,
respectively), while the relationship between these two groups is purely combinatorial (duality on the root 
data). 

\medskip

So, a reasonable idea to prove \eqref{e:FLE} would be to describe both sides in \emph{combinatorial terms},
i.e., in terms that only involve the root data and the parameter $\kappa$, with the hope that the resulting
descriptions will match up.  

\medskip

The process of expressing a category associated with the group $G$ in terms that involve just the root data
is something quite familiar in representation theory: if $\CC_G$ is a category attached to $G$ and $\CC_T$
is a similar category attached to the Cartan subgroup $T$, one seeks to construct a \emph{Jacquet functor}
$$\fJ:\CC_G\to \CC_T,$$
and express $\CC_G$ as objects of $\CC_T$ equipped with some additional structure. 

\medskip

For example, this is essentially how one proves the geometric Satake equivalence: one constructs the corresponding
Jacquet functor
$$\on{Perv}(\Gr_G)^{G\qqart}\to \on{Perv}(\Gr_T)^{T\qqart}\simeq \Rep(\cT)$$
by pull-push along the diagram
$$
\CD
\Gr_B  @>>>  \Gr _G  \\
@VVV  \\
\Gr_T.
\endCD
$$

\sssec{}

Let us try to do the same for the two sides of the FLE. First, we notice that for $G=T$, the two categories, i.e., 
\begin{equation} \label{e:FLE T}
\Whit_\kappa(T)\simeq \Dmod_\kappa(\Gr_T) \text{ and } \on{KL}_{\check\kappa}(\cT)
\end{equation}
are indeed equivalent \emph{as factorization categories} in a more or less tautological way. 

\medskip

We now have to figure out what functors to use
$$\fJ^{\Whit}: \Whit_\kappa(G)\to \Whit_\kappa(T) \text{ and }
\fJ^{\on{KL}}: \on{KL}_{\check\kappa}(\cG)\to \on{KL}_{\check\kappa}(\cT),$$
so that we will have a chance to express the $G$-categories in terms of the corresponding $T$-categories.

\sssec{}

Here the quantum group perspective will be instrumental. The functor
$$\fJ^{\on{Quant}}:\Rep_q(\cG)\to \Rep_q(\cT)$$
that we want to use, to be denoted $\fJ^{\on{Quant}}_!$, is given by
$$M\mapsto \on{C}^\cdot(U^{\on{Lus}}_q(\cN),M),$$
where $U^{\on{Lus}}_q(\cN)$ is the positive (i.e., ``upper triangular") part of the big (i.e., Lusztig's) 
quantum group. 

\medskip

The next step is to construct the functors
$$\fJ^{\Whit}_!: \Whit_\kappa(G)\to \Whit_\kappa(T) \text{ and }
\fJ^{\on{KL}}_!: \on{KL}_{\check\kappa}(\cG)\to \on{KL}_{\check\kappa}(\cT)$$
that under the equivalences \eqref{e:three equivalence} are supposed to correspond to $\fJ^{\on{Quant}}_!$.

\sssec{}

The functor
$$\fJ^{\on{KL}}_!: \on{KL}_{\check\kappa}(\cG)\to \on{KL}_{\check\kappa}(\cT)$$
is a version of the functor of semi-infinite cohomology with respect to $\cn\ppart$. 
More precisely, it is literally that when $\check\kappa$ is \emph{positive} or \emph{irrational}
and a certain non-trivial modification when $\check\kappa$ is \emph{negative rational}. 

\medskip

The precise construction of $\fJ^{\on{KL}}_!$ will be given in a subsequent publication, which 
will deal with the ``Kazhdan-Lusztig vs factorization modules" counterpart of our main theorem
(the latter deals with ``Whittaker vs factorization modules" equivalence). 

\sssec{}

Let us describe the sought-for functor 
$$\fJ^{\Whit}_!: \Whit_\kappa(G)\to \Whit_\kappa(T).$$

We will first make a naive attempt.  Namely, let 
$$\fJ^{\Whit}_*: \Whit_\kappa(G)\to \Whit_\kappa(T)$$
be the functor defined by !-pull and *-push along the diagram 
\begin{equation} \label{e:jacquet diag intro}
\CD
\Gr_{B^-}  @>>>  \Gr _G  \\
@VVV  \\
\Gr_T.
\endCD
\end{equation} 

This is a meaningful
functor, but it does not produce what we need. Namely, one can show that with respect to the equivalence 
\eqref{e:Whit vs quant}, the functor $\fJ^{\Whit}_*$ corresponds to the functor
$$\fJ^{\on{Quant}}_*: \Rep_q(\cG)\to \Rep_q(\cT),\quad M\mapsto \on{C}^\cdot(U^{\on{DK}}_q(\cN),M),$$
where $U^{\on{DK}}_q(\cN)$ is the De Concini-Kac version of $U_q(\cN)$ (we note that when $\kappa$ is
irrational, i.e., when $q$ takes values in non-roots of unity, the two versions do coincide). 

\sssec{}

To obtain the desired functor $\fJ^{\Whit}_!$ we proceed as follows, Note that 
$$\Whit_\kappa(T)\simeq \Dmod_\kappa(\Gr_T),$$
and as a plain DG category, it is equivalent to the category of $\Lambda$-graded vector spaces. 
For a given $\lambda\in \Lambda$, let us describe the $\lambda$-component $(\fJ^{\Whit}_!)^\lambda$ of
$\fJ^{\Whit}_!$. 

\medskip

Going back to $\fJ^{\Whit}_*$, its $\lambda$-component $(\fJ^{\Whit}_*)^\lambda$ is given by
$$\CF\mapsto \on{C}^\cdot(\Gr_G,\CF\sotimes (\bi^-_\lambda)_*(\omega_{S^{-,\lambda}})),$$
where 
$$N^-(\ppart)\cdot t^\lambda=:S^{-,\lambda} \overset{\bi^-_\lambda}\hookrightarrow \Gr_G.$$

\medskip

Now, $(\fJ^{\Whit}_!)^\lambda$ is given by 
$$\CF\mapsto \on{C}^\cdot(\Gr_G,\CF\sotimes (\bi^-_\lambda)_!(\omega_{S^{-,\lambda}})).$$

So, the difference between $(\fJ^{\Whit}_!)^\lambda$ and $(\fJ^{\Whit}_*)^\lambda$ is that we take
$(\bi^-_\lambda)_!(\omega_{S^{-,\lambda}})$ instead of $(\bi^-_\lambda)_*(\omega_{S^{-,\lambda}})$. 

\begin{rem}  \label{r:J !* Whit}
Let us remark that the functor that appears in the main body of this work is yet a different one, to be
denoted $\fJ^{\Whit}_{!*}$. Its $\lambda$-component $(\fJ^{\Whit}_{!*})^\lambda$ is given by
$$\CF\mapsto \on{C}^\cdot(\Gr_G,\CF\sotimes \ICslm),$$
where $\ICslm$ is a certain twisted D-module defined in \secref{s:ICs}. 

\medskip

Under the equivalence \eqref{e:Whit vs quant}, the functor $\fJ^{\Whit}_{!*}$ corresponds to the functor
\begin{equation} \label{e:small quant grp cohomology}
\fJ^{\on{Quant}}_{!*}: \Rep_q(\cG)\to \Rep_q(\cT),\quad M\mapsto \on{C}^\cdot(\fu_q(\cN),M),
\end{equation} 
where $\fu_q(\cN)$ is the positive part of the small quantum group.
\end{rem} 

\sssec{}

Having described the functor $\fJ^{\Whit}_!$, let us address the question of whether we can use it in order to express the category
$\Whit_\kappa(G)$ in terms of $\Whit_\kappa(T)$. To do so, we will again resort to the quantum group picture.

\medskip

The functor $\fJ^{\on{Quant}}_!$ has an additional structure: it is \emph{lax braided monoidal}. As such, it sends
the (monoidal) unit object $\sfe\in \Rep_q(\cG)$ to the object that we will denote
$$\Omega^{\on{Lus}}_q\in \Rep_q(\cT),$$
which has a structure of $\BE_2$-algebra. Moreover, the functor $\fJ^{\on{Quant}}_!$ can be upgraded to a functor
\begin{equation} \label{e:Lus inv quant}
(\fJ^{\on{Quant}}_!)^{\on{enh}}:\Rep_q(\cG)\to \Omega^{\on{Lus}}_q\mod_{\BE_2}(\Rep_q(\cT)).
\end{equation}

Similarly, the functors $\fJ^{\Whit}_!$ and $\fJ^{\on{KL}}_!$ have factorization structures. By applying to the unit, we obtain
\emph{factorization algebras}
\begin{equation} \label{e:Fact algebras Whit KL}
\Omega^{\on{Whit,Lus}}_\kappa \text{ and } \Omega^{\on{KL,Lus}}_{\check\kappa}
\end{equation}
in $\Whit_\kappa(T)$ and $\on{KL}_{\check\kappa}(\cT)$, respectively, and the upgrades
\begin{equation} \label{e:Lus inv Whit}
(\fJ^{\Whit}_!)^{\on{enh}}: \Whit_\kappa(G) \to \Omega^{\on{Whit,Lus}}_\kappa\on{-FactMod}(\Whit_\kappa(T))_{\on{untl}}
\end{equation}
and 
\begin{equation} \label{e:Lus inv KL}
(\fJ^{\on{KL}}_!)^{\on{enh}}: \on{KL}_{\check\kappa}(\cG)
\to \Omega^{\on{KL,Lus}}_{\check\kappa}\on{-FactMod}(\on{KL}_{\check\kappa}(\cT))_{\on{untl}},
\end{equation}
respectively, where the subscript $\on{untl}$ stands for ``unital modules". 

\medskip

One expects the functors \eqref{e:Lus inv Whit} and \eqref{e:Lus inv KL} to be equivalences \emph{if and only if}
\eqref{e:Lus inv quant} is. However, \eqref{e:Lus inv quant} is \emph{not} an equivalence, for the reasons explained
in \secref{sss:mixed} below.

\sssec{}

In a different world, if the functors \eqref{e:Lus inv Whit} and \eqref{e:Lus inv KL}  \emph{were} equivalences, one
would complete the proof of \eqref{e:FLE} by proving that under the factorization algebras \eqref{e:Fact algebras Whit KL}
correspond to one another under the equivalence 
$$\Whit_\kappa(T)\simeq \on{KL}_{\check\kappa}(\cT)$$
of \eqref{e:FLE T}.

\medskip

However, not all is lost in the real world either. In a subsequent publication, we will introduce a device that allows to
express the LHS of \eqref{e:Lus inv Whit} in terms of the RHS (and similarly, for \eqref{e:Lus inv KL}) and thereby
prove the FLE. 

\medskip

In the case of \eqref{e:Lus inv KL}, the same device should be able to give a different proof of the Kazhdan-Lusztig
equivalence \eqref{e:KL}. 

\sssec{}  \label{sss:mixed}

As was mentioned above, the functor \eqref{e:Lus inv quant} is \emph{not} an equivalence. Rather, we have an equivalence
\begin{equation} \label{e:Lus inv quant mixed}
\Rep^{\on{mxd}}_q(\cG)\simeq \Omega^{\on{Lus}}_q\mod_{\BE_2}(\Rep_q(\cT)),
\end{equation}
where $\Rep^{\on{mxd}}_q(\cG)$ is a category of modules over the ``mixed" quantum group, introduced in \cite{Ga8}. 

\medskip

The original functor \eqref{e:Lus inv quant} is the composition of \eqref{e:Lus inv quant mixed} with the restriction functor
\begin{equation} \label{e:mixed to big}
\Rep^{\on{mxd}}_q(\cG)\to \Rep_q(\cG).
\end{equation}

\begin{rem}
Recall following \cite{Ga8} that the mixed version $\Rep^{\on{mxd}}_q(\cG)$ has Lustig's algebra $U^{\on{Lus}}_q(\cN)$ 
for the positive part and the De Concini-Kac algebra $U^{\on{DK}}_q(\cN^-)$ for the negative part. However, even when
$q$ takes in non-roots of unity, the functor \eqref{e:mixed to big} is not an equivalence (and not even fully faithful unless
we restrict to abelian categories). 

\medskip

Namely, in the non-root of unity case, the category $\Rep_q(\cG)$ consists of modules that are locally finite for
the action of the entire quantum group, whereas $\Rep^{\on{mxd}}_q(\cG)$ is the quantum category $\CO$, i.e.,
we only impose local finiteness for the positive part.
\end{rem} 

\sssec{}

As we just saw, although the functor \eqref{e:Lus inv quant} is not an equivalence, it is the composition of
a (rather explicit) forgetful functor with an equivalence from another meaningful representation-theoretic category.
We have a similar situation for the functors \eqref{e:Lus inv Whit} and \eqref{e:Lus inv KL}.

\medskip

Namely, the functor \eqref{e:Lus inv KL} is the composition of the forgetful functor
$$\on{KL}_\kappa(\cG):=\wh{\cg}_{\check\kappa}\mod^{\cG\qqart} \to \wh{\cg}_{\check\kappa}\mod^{\check{I}} ,$$
and a functor
$$\wh{\cg}_{\check\kappa}\mod^{\check{I}} \to \Omega^{\on{KL,Lus}}_{\check\kappa}\on{-FactMod}(\on{KL}_{\check\kappa}(\cT))_{\on{untl}},$$
which is conjectured to be an equivalence. In the above formula $\check{I}\subset \cG\qqart$ is the Iwahori subgroup. The latter
conjectural equivalence is essentially equivalent to the main conjecture of \cite{Ga8}, see Conjecture 9.2.2 in {\it loc.cit.}

\medskip

Similarly, the functor \eqref{e:Lus inv Whit} is the composition of the pullback functor
$$\Whit_\kappa(G):=\Dmod_{\kappa}(\Gr_G)^{N\ppart,\chi}\to \Dmod_{\kappa}(\Fl_G)^{N\ppart,\chi}$$
and a functor 
$$\Dmod_{\kappa}(\Fl_G)^{N\ppart,\chi}\to \Omega^{\on{Whit,Lus}}_\kappa\on{-FactMod}(\Whit_\kappa(T))_{\on{untl}},$$
which is conjectured to be an equivalence. In the above formula $\Fl_G=G\ppart/I$ is the affine flag space. 
The latter conjectural equivalence should be provable by methods close to those developed in the present work. 

\ssec{The present work: the ``small" FLE}

Having explained the idea of the proof of the FLE via Jacquet functors, we will now
modify the source categories, and explain what it is that we actually prove in this work.

\sssec{}

Just as above, we will first place ourselves in the context of quantum groups. From now on we will
assume that $q$ takes values in the group of roots of unity. 

\medskip 

In this case, following Lusztig, to the datum of $(\cG,q)$ one attaches another reductive group $H$, 
see \secref{sss:roots in dual} for the detailed definition. Here we will just say that weight lattice of $H$,
denoted $\Lambda_H$ is a sublattice of $\Lambda$ and consists of the kernel of the symmetric
bilinear form $b$ associated to $q$
$$b(\lambda_1,\lambda_2)=q(\lambda_1+\lambda_2)-q(\lambda_1)-q(\lambda_2),$$
where we use the additive notation for the abelian group $\sfe^\times$. 

\medskip

In addition, following Lusztig, we have the \emph{quantum Frobenius} map, which we will interpret
as a monoidal functor
\begin{equation} \label{e:quantum Frob}
\on{Frob}_q^*:\Rep(H)\to \Rep_q(\cG).
\end{equation} 

We will use $\on{Frob}_q^*$ to regard $\Rep(H)$ as a monoidal category acting on $\Rep_q(\cG)$, 
$$V,M\mapsto \on{Frob}_q^*(V)\otimes M.$$

\medskip

Given this data, we can consider the \emph{graded Hecke category} of $\Rep_q(\cG)$ with respect to
$\Rep(H)$,
$$\bHecke(\Rep_q(\cG)):=\Rep(T_H)\underset{\Rep(H)}\otimes \Rep_q(\cG).$$ 

See \secref{ss:graded Hecke} for the discussion of the formalism of the formation of Hecke categories. 

\medskip

Recall the functor $\fJ^{\on{Quant}}_{!*}$ of \eqref{e:small quant grp cohomology}. It has the following structure: it intertwines
the actions of $\Rep(H)$ on $\Rep_q(\cG)$ and of $\Rep(T_H)$ on $\Rep_q(\cT)$ via the restriction functor
$\Rep(H)\to \Rep(T_H)$, where $T_H$ is the Cartan subgroup of $H$, and the latter action is given by the quantum 
Frobenius for $T$. 

\medskip

This formally implies that $\fJ^{\on{Quant}}_{!*}$ gives rise to a functor 
\begin{equation}  \label{e:Hecke small quant grp cohomology prel}
(\fJ^{\on{Quant}}_{!*})^{\bHecke}: \bHecke(\Rep_q(\cG))\to \Rep_q(\cT),
\end{equation} 
and further to a functor
\begin{equation}  \label{e:Hecke small quant grp cohomology}
(\fJ^{\on{Quant}}_{!*})^{\bHecke,\on{enh}}: \bHecke(\Rep_q(\cG))\to \Omega^{\on{small}}_q\mod_{\BE_2}(\Rep_q(\cT)),
\end{equation} 
where $\Omega^{\on{small}}_q$ is the $\BE_2$-algebra in $\Rep_q(\cT)$ obtained by applying $\fJ^{\on{Quant}}_{!*}$ to the unit.

\medskip

A key observation is that the functor \eqref{e:Hecke small quant grp cohomology} is an equivalence (up to renormalization,
which we will ignore in this Introduction). We will explain the mechanism of why this equivalence takes
place in a short while (see \secref{sss:Hecke small quant grp cohomology} below).

\sssec{}

We will can perform the same procedure for $\Whit$ and $\on{KL}$.  In the case of the latter, metaplectic geometric Satake
(see \secref{ss:metapl geom Sat}) defines an action of $\Rep(H)$ on $\on{KL}_{\check\kappa}(\cG)$. This action
matches the action of $\Rep(H)$ on $\Rep_q(\cG)$ via the Kazhdan-Lusztig equivalence. 

\medskip

Furthermore, one can define a functor
$$\fJ^{\on{KL}}_{!*}:\on{KL}_{\check\kappa}(\cG)\to \on{KL}_{\check\kappa}(\cT)$$
with properties mirroring those of $\fJ^{\on{Quant}}_{!*}$. In particular, we obtain:
\begin{equation}  \label{e:Hecke J KL !*}
(\fJ^{\on{KL}}_{!*})^{\bHecke,\on{enh}}: \bHecke(\on{KL}_{\check\kappa}(\cG))\to 
\Omega^{\on{KL,small}}_{\check\kappa}\on{-FactMod}(\on{KL}_{\check\kappa}(\cT))_{\on{untl}}.
\end{equation} 

If we use the Kazhdan-Lusztig equivalence \eqref{e:KL}, it would follow formally that the functor 
\eqref{e:Hecke J KL !*} is an equivalence. 

\medskip

Alternatively, we expect it to be possible to prove that \eqref{e:Hecke J KL !*} is an equivalence
directly. Juxtaposing this with the equivalence \eqref{e:Hecke small quant grp cohomology}, with some extra work, this would provide 
an alternative construction of the Kazhdan-Lusztig equivalence \eqref{e:KL}. 

\sssec{}

We now come to the key point of the present work. Using metaplectic geometric Satake we define an action
of the same category $\Rep(H)$ also on $\Whit_\kappa(G)$. Hence, we can form the category
$\bHecke(\Whit_\kappa(G))$. 

\medskip

Furthermore, in Part V of this work, we construct a functor
\begin{equation}  \label{e:J Whit !*}
\fJ^{\Whit}_{!*}:\Whit_\kappa(G)\to \Whit_\kappa(T),
\end{equation} 
and we show that it also intertwines the $\Rep(H)$ and $\Rep(T_H)$-actions. From here we obtain the functor
$$(\fJ^{\Whit}_{!*})^{\bHecke}:\bHecke(\Whit_\kappa(G))\to \Whit_\kappa(T),$$
and further a functor
\begin{equation}  \label{e:Hecke J Whit !*}
(\fJ^{\Whit}_{!*})^{\bHecke,\on{enh}}:\bHecke(\Whit_\kappa(G))\to \Omega^{\on{Whit,small}}_{\kappa}\on{-FactMod}(\Whit_\kappa(T))_{\on{untl}}. 
\end{equation} 

\sssec{}

The main result of this work, \thmref{t:main} says that the functor \eqref{e:Hecke J Whit !*} is an equivalence. 

\medskip

Furthermore, another
of our key results, \thmref{t:ident vacuum}, essentially says that under the equivalence
$$\Whit_\kappa(T)\simeq \on{KL}_{\check\kappa}(\cT),$$
the factorization algebras $\Omega^{\on{Whit,small}}_{\kappa}$ and $\Omega^{\on{KL,small}}_{\check\kappa}$ correspond to one another.

\medskip

Hence, taking into account the equivalence \eqref{e:Hecke J KL !*}, we obtain that our \thmref{t:main}, 
combined with \thmref{t:ident vacuum}, imply the following equivalence
\begin{equation} \label{e:small FLE}
\bHecke(\Whit_\kappa(G)) \simeq \bHecke(\on{KL}_{\check\kappa}(\cG))
\end{equation}

We refer to the equivalence \eqref{e:small FLE} as the ``small FLE", for reasons that will be explained in 
\secref{sss:Hecke small quant grp cohomology} right below. 

\medskip

We remark that, on the one hand, \eqref{e:small FLE} is a consequence of \eqref{e:FLE}. On the other hand, it should not
be a far stretch to get the original \eqref{e:FLE} from \eqref{e:small FLE}, which we plan to do in the future. 

\sssec{}

Here is what we \emph{do not} prove: the two sides in \thmref{t:main} are naturally factorization categories, 
and one wants the equivalence to preserve this structure. 

\medskip

However, in main body of the text we do not give the definition of factorization categories, so we do not formulate
this extension of \thmref{t:main}. That said, such an extension (in our particular situation) is rather straightforward:

\medskip

In general, if a functor between factorization categories is an equivalence at a point, this \emph{does not at all}
guarantee that it is an equivalence \emph{as factorization categories}. What makes it possible to prove this in 
our situation is the fact that we can control how our functor (i.e., $(\fJ^{\Whit}_{!*})^{\bHecke,\on{enh}}$) commutes with
Verdier duality, see \thmref{t:duality pres}. 

\sssec{}  \label{sss:Hecke small quant grp cohomology}

Let us now explain the origin of the terminology ``small". 

\medskip

A basic fact established in \cite{ArG} is that we have an equivalence of categories
\begin{equation}  \label{e:Arkh}
\bHecke(\Rep_q(\cG))\simeq \overset{\bullet}\fu_q(\cG)\mod,
\end{equation} 
where the RHS is the (graded version of the) category of representations of the small 
quantum group. 

\medskip

With respect to this equivalence, the functor $(\fJ^{\on{Quant}}_{!*})^{\bHecke}$ of \eqref{e:Hecke small quant grp cohomology prel}
is simply the functor
$$\overset{\bullet}\fu_q(\cG)\mod \to \Rep_q(\CT), \quad M\mapsto \on{C}^\cdot(\fu_q(\cN),M).$$

Now, the resulting functor 
\begin{equation} \label{e:equivalence for small}
(\fJ^{\on{Quant}}_{!*})^{\bHecke,\on{enh}}:\overset{\bullet}\fu_q(\cG)\mod \to 
\Omega^{\on{small}}_q\mod_{\BE_2}(\Rep_q(\cT))
\end{equation} 
is indeed an equivalence for the following reason: 

\medskip

The category $\overset{\bullet}\fu_q(\cG)\mod$ is (the relative with respect to $\Rep_q(\cT)$) Drinfeld center
of the monoidal category $\fu_q(\cN)\mod(\Rep_q(\cT))$, see \secref{ss:rel Dr center}) for what this means.
Then the equivalence \eqref{e:equivalence for small} is a combination of the following two statements
(see \secref{ss:Koszul abstract} for details): 

\begin{itemize}

\item A Koszul duality equivalence (up to renormalization) between $\fu_q(\cN)\mod(\Rep_q(\cT))$ and 
$\Omega^{\on{small}}_q\mod$;

\item A general fact that for an $\BE_2$-algebra $A$, the Drinefeld center of the monoidal category $A\mod$
identifies with $A\mod_{\BE_2}$. 

\end{itemize} 

\begin{rem}
At the heart of the equivalence \eqref{e:equivalence for small} was the fact that $\overset{\bullet}\fu_q(\cG)$
is the Drinfeld double of $\fu_q(\cN)$ relative to $\Rep_q(\cT)$. This boils down to the fact that 
$\fu_q(\cN^-)$ is the component-wise dual of $\fu_q(\cN)$.

\medskip

A similar fact is responsible for the equivalence \eqref{e:Lus inv quant mixed}. There we have that 
$U^{\on{DK}}_q(\cN^-)$ is the component-wise due of $U^{\on{Lus}}_q(\cN)$.
\end{rem} 

\ssec{What do we actually prove?}

As was explained above, the goal of this work is to prove the equivalence
\begin{equation} \label{e:Hecke J Whit !* again}
(\fJ^{\Whit}_{!*})^{\bHecke,\on{enh}}:\bHecke(\Whit_\kappa(G))\to \Omega^{\on{Whit,small}}_{\kappa}\on{-FactMod}(\Whit_\kappa(T))_{\on{untl}}
\end{equation} 
of \eqref{e:Hecke J Whit !*}. However, at a glance, our \thmref{t:main} looks a little different. We will now explain the precise 
statement of what we actually prove. 

\sssec{}

Perhaps, the main difference is that we actually work in greater generality than the one explained above. Namely, the story of FLE
inherently lives in the world of D-modules (in particular, it was tied to the assumption that our ground field $k$ and the field 
of coefficients $\sfe$ are the same); this is needed in order for categories such as $\on{KL}_\kappa(\cG)$ to make sense. 

\medskip

However, the assertion that \eqref{e:Hecke J Whit !* again} is an equivalence is entirely geometric. 
I.e., it can be formulated within \emph{any} sheaf
theory from the list in \secref{sss:sheaf theory}. For example, we can work with $\ell$-adic sheaves over an arbitrary ground
field $k$ (so in this case $\sfe=\ol\BQ_\ell$). This is the context in which we prove our \thmref{t:main}.

\sssec{}

That said, we need to explain what replaces the role of the parameter $\kappa$. 

\medskip

In the context of a general sheaf theory,
we can no longer talk about \emph{twistings} (those only make sense for D-modules). However, we can talk about \emph{gerbes}
with respect to to the multiplicative group $\sfe^\times$ of the field $\sfe$ of coefficients. The assumption that
$\kappa$ was rational gets replaced by the assumption that we take gerbes with respect to the group $\sfe^{\times,\on{tors}}$
of torsion elements in $\sfe^\times$, i.e., the group of roots of unity. 

\medskip

Now, instead of the datum of $\kappa$, our input is a \emph{geometric metaplectic data} $\CG^G$, which is a 
\emph{factorization gerbe} on the factorization version of the affine Grassmannian. We refer the reader to \cite{GLys},
where this theory is developed. 

\medskip

Returning to the context of D-modules, one shows that a datum of $\kappa$ gives rise to a geometric metaplectic data
for the sheaf theory of D-modules. 

\sssec{}

The second point of difference between \eqref{e:Hecke J Whit !* again} and the statement of \thmref{t:main} involves
the fact that
the former talks about factorization modules for a factorization algebra on $\Gr_T$, and the latter talks about a 
factorization algebra on the configuration space instead. 

\medskip

However, this difference is immaterial (we chose the configuration space formulation because of its finite-dimensional
nature and the proximity to the language of \cite{BFS}). Indeed, as explained in \secref{ss:Conf vs Gr}, factorization
algebras on $\Gr_T$ satisfying a certain support condition (specifically, if they are supported on the connected components
of $\Gr_T$ corresponding to elements of $(\Lambda^{\on{neg}}-0)\subset \Lambda$) can be equivalently though of as 
factorization algebras on the configuration space.

\medskip

The factorization algebra on $\Gr_T$ that we deal with is the \emph{augmentation ideal} in 
$\Omega^{\on{Whit,small}}_{\kappa}$, so factorization modules for it are the same as \emph{unital} factorization 
modules $\Omega^{\on{Whit,small}}_{\kappa}$. 

\sssec{}  \label{sss:renorm Intro}

The final point of difference is that in order for the equivalence to hold, we need to apply a certain renormalization 
procedure to the category of factorization modules
$$\Omega_q^{\on{small}}\on{-FactMod} \rightsquigarrow \Omega_q^{\on{small}}\on{-FactMod}^{\on{ren}}.$$

Specifically, we will have enlarge the class of objects that we declare as compact, see \secref{ss:renorm fact}.

\medskip

This renormalization procedure follows the pattern of how one obtains $\IndCoh$ from $\QCoh$: it does not affect the heart 
(of the naturally defined t-structure); we have a functor
$$\Omega_q^{\on{small}}\on{-FactMod}^{\on{ren}}\to \Omega_q^{\on{small}}\on{-FactMod}$$
which is t-exact and induces an equivalence on the eventually coconnective parts. So the two categories 
differ only at the cohomological $-\infty$. 

\sssec{}

In Part IX of this work we prove that in the Betti situation (i.e., for the ground field being $\BC$ and for the sheaf
theory being constructible sheaves in classical topology with $\sfe$-coefficients for any $\sfe$), we have a canonical equivalence
\begin{equation} \label{e:BFS}
\overset{\bullet}\fu_q(\cG)\mod^{\on{baby-ren}}\simeq \Omega_q^{\on{small}}\on{-FactMod}
\end{equation}
and
\begin{equation} \label{e:BFS ren}
\overset{\bullet}\fu_q(\cG)\mod^{\on{ren}}\simeq \Omega_q^{\on{small}}\on{-FactMod}^{\on{ren}},
\end{equation}
where 
$$\overset{\bullet}\fu_q(\cG)\mod^{\on{baby-ren}} \text{ and } \overset{\bullet}\fu_q(\cG)\mod^{\on{ren}}$$
are two renormalizations of the (usual version of the) category $\overset{\bullet}\fu_q(\cG)\mod$. 

\medskip

The equivalence \eqref{e:BFS} had been previously established in (and was the subject of) the book \cite{BFS}.
The proof we give is just a remake in a modern language, where the key new tool we use is Lurie's equivalence
between $\BE_2$ algebras and factorization algebras. The proof is an expansion of the contents of \secref{sss:Hecke small quant grp cohomology}
above. 

\sssec{}

Thus, for $k=\BC=\sfe$, and using Riemann-Hilbert correspondence, we obtain the equivalences
\begin{equation} \label{e:Whit vs quant small comp}
\bHecke(\Whit_\kappa(G)) \simeq \Omega_q^{\on{small}}\on{-FactMod}^{\on{ren}} \simeq 
\overset{\bullet}\fu_q(\cG)\mod^{\on{ren}}.
\end{equation}

The composite equivalence
\begin{equation} \label{e:Whit vs quant small}
\bHecke(\Whit_\kappa(G)) \simeq \overset{\bullet}\fu_q(\cG)\mod^{\on{ren}}
\end{equation}
is a ``small" version of the equivalence \eqref{e:Whit vs quant}. 

\medskip

It is the equivalence \eqref{e:Whit vs quant small} that is responsible for the title of this work. 

\ssec{How do we prove it?}  \label{ss:how}

Modulo technical nuances that have to do with renormalization, our main result says that a functor denoted 
\begin{equation} \label{e:Phi intro}
\Phi^{\bHecke}_{\on{Fact}}:\bHecke(\Whit_q(G)) \to  \Omega_q^{\on{small}}\on{-FactMod},
\end{equation}
which is the configuration space version of the functor \eqref{e:Hecke J Whit !* again}, is an equivalence. 

\medskip

Let us explain the main ideas involved in the proof.

\sssec{}

Consider the following general paradigm: let 
$$\Phi:\CC_1\to \CC_2$$ 
be a functor between (unital) factorization categories. It induces a functor
$$\Phi^{\on{enh}}:\CC_1\to \Omega\on{-ModFact}(\CC_2)_{\on{untl}},$$
and we wish to know when the latter is an equivalence.
 
\medskip

Such an equivalence is not something we can generally expect. For example, it does hold for $\Phi$ being the functor
$$\fJ^{\on{Quant}}_{!*}: \overset{\bullet}\fu_q(\cG)\mod\to \Rep_q(\cT),$$
but it fails for $\Phi$ being the functor
$$\fJ^{\on{Quant}}_{!}: \Rep_q(\cG)\mod\to \Rep_q(\cT).$$
 
\medskip

So we will need to prove something very particular about the functor $\Phi^{\bHecke}_{\on{Fact}}$ in order
to know that it is an equivalence.

\sssec{}

Our proof of the equivalence \eqref{e:Phi intro} follows a rather standard pattern in representation theory.
(the challenge is, rather, to show that this pattern is realized in our situation):

\medskip

We will
define a family of objects in either category that we will call ``standard", indexed by the elements of $\Lambda$
$$\mu\mapsto \CM^{\mu,!}_{\Whit}\in \bHecke(\Whit_q(G)) \text{ and } \CM^{\mu,!}_{\Conf}\in 
\Omega_q^{\on{small}}\on{-FactMod}^{\on{ren}}.$$

These objects will be compact and ``almost" generate $\bHecke(\Whit_q(G))$ (resp., $\Omega_q^{\on{small}}\on{-FactMod}^{\on{ren}}$). 
Define the co-standard objects $\CM^{\mu,*}$ (in either context) by
\begin{equation} \label{e:orth intro}
\CHom(\CM^{\mu',!},\CM^{\mu,*})=\begin{cases}
&\sfe \text{ if } \mu'=\mu \\
&0 \text{ otherwise}.
\end{cases}
\end{equation}

A standard argument shows that in order to prove that the functor $\Phi^{\bHecke}_{\on{Fact}}$ is an equivalence, 
it is sufficient to prove that it sends
\begin{equation} \label{e:standard to standard Intro}
\CM^{\mu,!}_{\Whit}\mapsto \CM^{\mu,!}_{\Conf} \text{ and } 
\CM^{\mu,*}_{\Whit}\mapsto \CM^{\mu,*}_{\Conf}.
\end{equation}

So the essence of the proof is in defining the corresponding families of objects (on each side)
and proving \eqref{e:standard to standard Intro}.

\sssec{}

Our guiding principle is again the comparison with the quantum group, in this case, with the category 
$\overset{\bullet}\fu_q(\cG)\mod$. 

\medskip

Namely, we want that the objects $\CM^{\mu,!}$ correspond under the equivalences \eqref{e:Whit vs quant small comp}
to the \emph{baby Verma modules}
$$\CM^{\mu,!}_{\on{Quant}}\in \overset{\bullet}\fu_q(\cG)\mod,$$
i.e., the induced modules 
$$\on{Ind}^{\overset{\bullet}\fu_q(\cG)}_{\overset{\bullet}\fu_q(\cB)}(\sfe^\mu),$$
where $\sfe^\mu$ denotes the one-dimensional module over the quantum Borel with character $\mu$.

\medskip

Given this, the corresponding objects $\CM^{\mu,!}_{\Conf}$ and $\CM^{\mu,*}_{\Conf}$ are easy to guess:
they are given by !- and *- extensions, respectively, from the corresponding strata on the configuration space. 

\sssec{}

The situation with $\CM^{\mu,!}_{\Whit}$ and $\CM^{\mu,*}_{\Whit}$ is more interesting. If we were dealing with 
$\Whit_q(G)$ rather than with $\bHecke(\Whit_q(G))$, we would have a natural collection of standard objects,
denoted $W^{\lambda,!}_{\Whit}$, indexed by \emph{dominant} elements of $\Lambda^{\on{pos}}$. However, the sought-for
objects $\CM^{\mu,!}_{\Whit}\in \bHecke(\Whit_q(G))$ are \emph{not} obtained from the objects $W^{\lambda,!}_{\Whit}$
by applying the (obvious) functor
$$\ind_{\bHecke}:\Whit_q(G)\to \bHecke(\Whit_q(G)).$$

Indeed, drawing on the equivalence with quantum groups (i.e., \eqref{e:Whit vs quant}), the objects $W^{\lambda,!}_{\Whit}$ 
correspond to the \emph{Weyl} 
modules over the big quantum group, denoted $W^{\lambda,!}_{\on{Quant}}$, 
and the latter do \emph{not} restrict to the baby Verma modules over the small quantum group. 

\medskip

However, there exists an \emph{explicit colimit procedure} that allows to express the \emph{dual} baby Verma module
$\CM^{\mu,*}_{\on{Quant}}$ in terms of the \emph{dual} Weyl modules $W^{\lambda,*}_{\on{Quant}}$ and modules pulled back by the
quantum Frobenius \eqref{e:quantum Frob}. This procedure fits into the Drinfeld-Pl\"ucker paradigm, explained 
in \secref{s:DrPl}.   (How exactly it can be applied to the quantum groups situation is explained in \cite[Sect. 10.3]{Ga8}.) 

\medskip

By essentially mimicking the procedure by which one expresses $\CM^{\mu,*}_{\on{Quant}}$ via the $W^{\lambda,*}_{\on{Quant}}$'s,
or rigorously speaking applying the Drinfeld-Pl\"ucker formalism in the situation of $\Whit_q(G)$ (see \secref{s:baby Verma in Whit}),
we define the sought-for objects $\CM^{\mu,*}_{\Whit}$ in terms of the corresponding objects $W^{\lambda,*}_{\Whit}$. 

\sssec{}

The next task is to show that these objects satisfy
\begin{equation} \label{e:costandard to costandard Intro}
\Phi^{\bHecke}_{\on{Fact}}(\CM^{\mu,*}_{\Whit})\simeq \CM^{\mu,*}_{\Conf}.
\end{equation}

This is equivalent to showing that the !-fiber $(\Phi^{\bHecke}_{\on{Fact}}(\CM^{\mu,*}_{\Whit}))_{\mu'\cdot x}$
of $\Phi^{\bHecke}_{\on{Fact}}(\CM^{\mu,*}_{\Whit})$ at a point $\mu'\cdot x$ of the configuration space satisfies:
\begin{equation} \label{e:! at mu'}
(\Phi^{\bHecke}_{\on{Fact}}(\CM^{\mu,*}_{\Whit}))_{\mu'\cdot x}=\begin{cases}
&\sfe \text{ if } \mu'=\mu \\
&0 \text{ otherwise}.
\end{cases}
\end{equation}

This is a non-tautological calculation because it does not reduce to just a calculation of the cohomology of some sheaf: 
indeed we are dealing with the object $\CM^{\mu,*}_{\Whit}$ of $\bHecke(\Whit_q(G))$ and not with just a \emph{sheaf}, 
which would be an object of $\Whit_q(G)$. 

\medskip

Of course, in order to perform this calculation we need to know something about the functor $\Phi^{\bHecke}_{\on{Fact}}$.
This functor is obtained via the Hecke property from the functor
$$\Phi_{\on{Fact}}:\Whit_q(G))\to  \Omega_q^{\on{small}}\on{-FactMod},$$
and when we take the !-fiber at $\mu'\cdot x$ the corresponding functor is
$$\CF\mapsto \on{C}^\cdot(\Gr_{G,x},\CF\sotimes \on{IC}_x^{\mu'+\frac{\infty}{2},-}),$$
already mentioned in Remark \ref{r:J !* Whit}. 

\medskip

In the above formula, $\on{IC}_x^{\mu'+\frac{\infty}{2},-}$ is the \emph{metaplectic semi-infinite IC sheaf}
on the $N\ppart$ orbit 
$$S^{\mu'}=N\ppart\cdot t^{\mu'}\subset \Gr_{G,x}.$$

A significant part of this work (namely, the entire Part IV) is devoted to the construction of this semi-infinite
IC sheaf and the discussion of its properties (we should say that this Part closely follows \cite{Ga7}, 
where a non-twisted situation is considered).  

\medskip

On the one hand, we define the semi-infinite IC by the procedure of Goresky-MacPherson extension
inside the \emph{semi-infinite category on the Ran version of the affine Grassmannian}, and 
$\on{IC}_x^{\mu'+\frac{\infty}{2},-}$ is obtained as !-restriction to the fiber over $x\in \Ran_x$. 

\medskip

On the other hand, it turns out to be possible to describe $\on{IC}_x^{\mu'+\frac{\infty}{2},-}$
explicitly, and it turns out that this description can also be phrased in terms of the Drinfeld-Pl\"ucker formalism
(see \secref{ss:F and semiinf}). 

\medskip

This is what makes the calculation \eqref{e:! at mu'} feasible:
the Drinfeld-Pl\"ucker nature of both $\CM^{\mu,*}_{\Whit}$ and $\on{IC}_x^{\mu'+\frac{\infty}{2},-}$
so to say cancel each other out. See \secref{s:stalks} for how exactly this happens. 

\sssec{}

We are now left with the following task: we need to define the family of objects $\CM^{\mu,!}_{\Whit}$, so that
they satisfy \eqref{e:orth intro} and 
the following property holds:
\begin{equation} \label{e:* fiber St}
\Phi^{\bHecke}_{\on{Fact}}(\CM^{\mu,!}_{\Whit})\simeq \CM^{\mu,!}_{\Conf}.
\end{equation}

\medskip

To construct the objects $\CM^{\mu,!}_{\Whit}$ we once again resort to the equivalence with the quantum group
situation.

\medskip

We know that contragredient duality defines an equivalence
$$(\Rep_q(\cG)^c)^{\on{op}}\to  \Rep_{q^{-1}}(\cG)^c,$$
which induces a similar equivalence for the small quantum group
$$(\overset{\bullet}\fu_q(\cG)\mod^c)^{\on{op}}\to \overset{\bullet}\fu_{q^{-1}}(\cG)\mod^c.$$

The latter has the property that it sends $\CM^{\mu,*}_{\on{Quant}}$ to $\CM^{\mu,!}_{\on{Quant}}$.

\medskip

Hence, it is naturally to try to define a duality
\begin{equation} \label{e:verdier Whit intro}
(\Whit_q(G)^c)^{\on{op}}\to \Whit_{q^{-1}}(G)^c,
\end{equation}
which would then give rise to a duality
\begin{equation} \label{e:verdier Hecke Whit intro}
(\bHecke(\Whit_q(G))^c)^{\on{op}}\to \bHecke(\Whit_{q^{-1}}(G))^c,
\end{equation}
and we will define $\CM^{\mu,!}_{\Whit}$ as the image of $\CM^{\mu,*}_{\Whit}$
under the latter functor.

\sssec{}

We define the desired duality functor \eqref{e:verdier Whit intro} in \secref{ss:duality Whit}, but this
is far from automatic. 

\medskip

The difficulty is that the category $\Whit_q(G)$ is defined by imposing 
\emph{invariance} with respect to a group \emph{ind}-scheme, while the dual category would
naturally be \emph{coinvariants}. So we need to show that the categories of invariants
and coinvariants are equivalent to one another. This turns out to be a general phenomenon
in the Whittaker situation, as is shown in the striking paper \cite{Ras2}. For completeness
we supply a proof of this equivalence in our situation using global methods, see \thmref{t:inv and coinv}
and its proof in \secref{ss:loc to glob Whit}. 
 
\medskip

Thus, we obtain the duality functor in \eqref{e:verdier Hecke Whit intro} with desired properties,
and one can prove \eqref{e:orth intro} by essentially mimicking the quantum group situation.

\sssec{}

Finally, we need to prove \eqref{e:* fiber St}. At the first glance, this may appear as super hard: we need 
to prove an analog of \eqref{e:! at mu'} for the objects $\CM^{\mu,!}_{\Whit}$, but for \emph{*-fibers} instead of the
!-fibers, and we do not really know how to do that: our theory is well-adjusted to computing !-pullbacks
and *-direct images, but not *-pullbacks and !-direct images. 

\medskip

What saves the game is that we can show that the functor $\Phi_{\on{Fact}}$
(and then $\Phi^{\bHecke}_{\on{Fact}}$), composed with the forgetful 
functor from $\Omega_q^{\on{small}}\on{-FactMod}$ to just sheaves on the configuration space,
intertwines the duality \eqref{e:verdier Whit intro} on $\Phi_{\on{Fact}}$ with Verdier duality.
This is the assertion of our \thmref{t:duality pres} (and its Hecke extension \thmref{t:duality pres Hecke}). 

\medskip

The above results about commutation of the (various versions of the) functor $\Phi$ with Verdier
duality point to yet another layer in this work: 

\medskip

In order to prove this commutation, we give a 
\emph{global interpretation} of both the Whittaker category (in Sects \ref{ss:Whit glob}-\ref{ss:loc to glob Whit})
and of the functor $\Phi$ (in Part VII).  In other words, we realize $\Whit_q(G)$ 
as sheaves on (ind)-algebraic stacks (rather than general prestacks). Eventually this leads to the desired
commutation with Verdier duality because we show that the morphisms between algebraic stacks involved
in the construction are \emph{proper}. 

\ssec{Organization of the text} 

This work consists of nine Parts and an Appendix. We will now outline the contents of each Part,
in order to help the reader navigating this rather lengthy text. 

\sssec{}

Part I is preliminaries. It can be skipped on the first pass and returned to when necessary. 

\medskip

In \secref{s:Grass} we recall the definitions of the affine Grassmannian, the loop groups, etc.,
along with their \emph{factorizable versions}.

\medskip

In \secref{s:metap} we summarise, following \cite{GLys}, the basic tenets of the 
\emph{geometric metaplectic theory}. In particular, we explain what a geometric metaplectic
datum is (the ``q" parameter), the construction of the metaplectic dual $H$ (the recipient of
Lusztig's quantum Frobenius), and some relevant facts regarding the metaplectic version of
the geometric Satake functor. 

\medskip

In \secref{s:fact} we discuss the formalism of factorization algebras and modules over them.
In particular, we give (one of the possible) definitions of these objects in the context of $\infty$-categories.

\medskip

In \secref{s:conf} we discuss some basics of the geometry of the configuration space (of divisors on $X$
colored by elements of $\Lambda^{\on{neg}}-0$). In particular, we explain how this configuration space
relates to the Ran version of the affine Grassmannian $\Gr_{T,\Ran}$ of $T$. 

\medskip

In \secref{s:fact conf} we talk about about factorization algebras and modules over them 
\emph{on the configuration space}. In particular, we explain that if a factorization algebra
is \emph{perverse}, then the corresponding category of its modules has a t-structure and 
is a highest weight category. 

\sssec{}

Part II is devoted to the study of the metaplectic Whittaker category of the affine Grassmannian. 
The material here is largely parallel to one in \cite{Ga9}, where the non-twisted situation is considered.

\medskip

In \secref{s:Whit} we define the metaplectic Whittaker category $\Whit_{q,x}(G)$ as $(N\ppart,\chi)$-invariants
in the category of (metaplectically twisted) sheaves on the affine Grassmannian, and study its basic properties,
such as the t-structure, standard and costandard objects, etc.

\medskip

In \secref{s:dual Whit} we give a dual definition of the Whittaker category, denoted $\Whit_{q,x}(G)_{\on{co}}$,
equal to $(N\ppart,\chi)$-coinvariants rather than invariants. We state a (non-trivial) theorem to the effect that some 
(non-tautological) functor 
$$\on{Ps-Id}:\Whit_{q,x}(G)_{\on{co}}\to \Whit_{q,x}(G)$$
is an equivalence. This gives rise to a duality identification
\begin{equation} \label{e:Verd Whit}
\Whit_{q,x}(G)^\vee \simeq \Whit_{q^{-1},x}(G). 
\end{equation} 
We also introduce a global version of the Whittaker category,
denoted $\Whit_{q,\on{glob}}(G)$ (using a projective curve $X$ and the stack $\BunNbox$, which is version of
\emph{Drinfeld's compactification}), and we state a theorem that says that the global version is actually equivalent to the
local ones, i.e., $\Whit_{q,x}(G)$ and $\Whit_{q,x}(G)_{\on{co}}$. In particular, the duality identification corresponds to
the usual Verdier duality on $\BunNbox$. 

\medskip

In \secref{s:proof local to global} we prove the local-to-global equivalence for the metaplectic Whittaker category
$$\Whit_{q,\on{glob}}(G)\to \Whit_{q,x}(G).$$

In the process of doing so, we introduce the Ran version of the Whittaker category $\Whit_{q,\Ran}(G)$,
and a \emph{factorization algebra object} 
$$\on{Vac}_{\Whit,\Ran}\in \Whit_{q,\Ran}(G).$$

We also define a functor, crucially used in the sequel:
\begin{equation} \label{e:sprd Intro} 
\on{sprd}_{\Ran_x}:\Whit_{q,x}(G)\to \Whit_{q,\Ran_x}(G),
\end{equation} 
which, so to say, inserts $\on{Vac}_{\Whit,\Ran}$ on points of $X$ other than the marked point $x$. 
The functor \eqref{e:sprd Intro} can be further upgraded to a functor 
$$\Whit_{q,x}(G)\to \on{Vac}_{\Whit,\Ran}\on{-FactMod}(\Whit_{q,\Ran_x}(G)).$$

\sssec{}

In Part III we study the Hecke action of $\Rep(H)$ on $\Whit_{q,x}(G)$.

\medskip

In \secref{s:Hecke Whit} we define this Hecke action and prove two crucial results. One says that
this action is t-exact (in the t-structure on $\Whit_{q,x}(G)$ introduced in \secref{s:Whit}). The other says
that we start with an irreducible object $W^{\lambda,!*}_{\Whit}\in (\Whit_{q,x}(G))^\heartsuit$ corresponding 
to the orbit $S^{\lambda,-}\subset \Gr_G$ for $\lambda$ a \emph{restricted coweight} (see \secref{ss:restr} for
what this means), for any irreducible $V^{\gamma}\in \Rep(H)$, the result of the
convolution
$$W^{\lambda,!*}_{\Whit}\star V^\gamma$$ 
is the irreducible object $W^{\lambda+\gamma,!*}_{\Whit}\in (\Whit_{q,x}(G))^\heartsuit$. 

\medskip

This is a direct counterpart 
of Steinberg's theorem in the context of quantum groups: for an irreducible object 
$W^{\lambda,!*}_{\on{Quant}}\in (\Rep_q(\cG))^\heartsuit$, the tensor product
$$W^{\lambda,!*}_{\on{Quant}}\otimes \on{Frob}^*_q(V^\gamma)$$
equals $W^{\lambda+\gamma,!*}_{\on{Quant}}$ for $\lambda$ restricted and any $\gamma$. 

\medskip

In \secref{s:Hecke} we discuss the general formalism of the formation of the category of Hecke eigen-objects
$\bHecke(\bC)$ for a category $\bC$ equipped with an action of $\Rep(H)$. In particular, we introduce the
notion of such an action being \emph{accessible} (with respect to a given t-structure on $\bC$)
and show that in this case the resulting t-structure on $\bHecke(\bC)$ is Artinian.

\medskip

In \secref{s:Hecke Whit categ} we apply the formalism of \secref{s:Hecke} to $\bC=\Whit_{q,x}(G)$. The resulting
category $\bHecke(\Whit_{q,x}(G))$ is the LHS in our main theorem. We discuss some basic properties of this
category. In particular, we describe the irreducible objects of $( \bHecke(\Whit_{q,x}(G)))^\heartsuit$. 

\sssec{}

In Part IV we are concerned with the construction and properties of the metaplectic semi-infinite
IC sheaf, denoted $\ICs_{q,\Ran}$. 

\medskip

In \secref{s:semiinf categ} we discuss the metaplectic semi-infinite category $\SI_{q,\Ran}(G)$, which is a full subcategory inside
$\Shv_{\CG^G}(\Gr_{G,\Ran})$, obtained by imposing the equivariance condition with respect to the loop group of $N$. 
We restrict our attention to the ``non-negative part" of $\SI_{q,\Ran}(G)$, denoted $\SI_{q,\Ran}(G)^{\leq 0}$. We define a stratification
of $\SI_{q,\Ran}(G)^{\leq 0}$ indexed by elements of $\Lambda^{\on{neg}}$. We also introduce a full subcategory
$$\SI_{q,\Ran}(G)^{\leq 0}_{\on{untl}}\subset \SI_{q,\Ran}(G)^{\leq 0}$$
that consists of \emph{unital} objects. We show that for every $\lambda\in \Lambda^{\on{neg}}$ the corresponding subcategory 
$$\SI_{q,\Ran}(G)^{= \lambda}_{\on{untl}}\subset \SI_{q,\Ran}(G)^{= \lambda}$$
is equivalent to the category of (gerbe-twisted) sheaves on the corresponding connected component $\Conf^\lambda$ of
the configuration space.

\medskip

In \secref{s:ICs} we introduce a t-structure on $\SI_{q,\Ran}(G)^{\leq 0}$ and define $\ICs_{q,\Ran}$ as the minimal extension
of the dualizing sheaf on $S^0_\Ran\subset \Gr_{G,\Ran}$. We describe the !-restriction of $\ICs_{q,\Ran}$ to 
$$\Gr_{G,x}=\{x\}\underset{\Ran}\times \Gr_{G,\Ran}.$$
We also discuss the \emph{factorization structure} on $\ICs_{q,\Ran}$, which makes it into a \emph{factorization algebra}
on $\Gr_{G,\Ran}$. Finally, we discuss the relationship between $\ICs_{q,\Ran}$ and the (gerbe-twisted) IC sheaf
on Drinfeld's compactification $\BunNbom$. 

\medskip

In \secref{s:Hecke ICs} we establish the (crucial for the sequel) Hecke property of $\ICs_{q,\Ran}$. 

\sssec{}

In Part V we define (various forms of) the Jacquet functor that maps $\Whit_{q,x}(G)$ to the coresponding category for $T$.

\medskip

In \secref{s:Jacquet} we first consider the functor
$$\fJ_{!*,\Ran}:\Shv_{\CG^G}(\Gr_{G,\Ran})\to \Shv_{\CG^T}(\Gr_{T,\Ran}),$$
defined using the diagram
$$
\CD
\Gr_{B^-,\Ran}  @>>>  \Gr_{G,\Ran}  \\
@VVV  \\
\Gr_{T,\Ran},
\endCD
$$
and using $\ICsm_{q^{-1},\Ran}$ as a kernel.

\medskip

We precompose this functor with the functor
$\on{sprd}_{\Ran_x}$ of \eqref{e:sprd Intro} and obtain a functor
$$\fJ_{!*,\on{sprd}}:\Whit_{q,x}(G)\to \Shv_{\CG^T}(\Gr_{T,\Ran_x}).$$

We define the factorization algebra $\Omega_q^{\Whit_{!*}}$
on $\Gr_{T,\Ran}$ as $\fJ_{!*,\Ran}(\on{Vac}_{\Whit,\Ran})$, and we upgrade the
functor $\fJ_{!*,\on{sprd}}$ to a functor
$$\fJ_{!*,\on{Fact}}:\Whit_{q,x}(G)\to \Omega_q^{\Whit_{!*}}\on{-FactMod}\left(\Shv_{\CG^T}(\Gr_{T,\Ran_x})\right).$$

\medskip

In \secref{s:Hecke Jacquet} we establish the Hecke property of the functor $\fJ_{!*,\on{Fact}}$ and extend it to the functor
$$\fJ^{\bHecke}_{!*,\on{Fact}}:\bHecke(\Whit_{q,x}(G))\to \Omega_q^{\Whit_{!*}}\on{-FactMod}\left(\Shv_{\CG^T}(\Gr_{T,\Ran_x})\right).$$

\sssec{}

In Part VI we interpret the functor $\fJ^{\bHecke}_{!*,\on{Fact}}$ via the factorization space.

\medskip 

In \secref{s:Omega small} we define a particular (in fact, the most basic) factorization algebra $\Omega^{\on{small}}_q$
in $\Shv_{\CG^\Lambda}(\Conf)$ (here $\CG^\Lambda$ is a factorization gerbe on $\Conf$ constructed from the geometric
metaplectic data $\CG^G$). Namely, $\Omega^{\on{small}}_q$, viewed as a sheaf, is \emph{perverse}, and is the Goresky-MacPherson extension 
of its restriction to the open locus $\overset{\circ}\Conf\subset \Conf$ that consists of multiplicity-free divisors, where, in its turn,
$\Omega^{\on{small}}_q|_{\overset{\circ}\Conf}$ is the \emph{sign local system} (the latter makes sense since the gerbe $\CG^\Lambda$
has the property that it is canonically trivial when restricted to $\overset{\circ}\Conf$). 

\medskip

In \secref{s:Jacquet to conf} we use the diagrams
$$
\CD
(\Gr_{T,\Ran})^{\on{neg}}  @>>>  \Gr_{T,\Ran} \\
@VVV   \\
\Conf
\endCD
$$
and 
$$
\CD
((\Gr_{T,\Ran_x})^{\on{neg}})_{\infty\cdot x}  @>>>  \Gr_{T,\Ran_x} \\
@VVV   \\
\Conf_{\infty\cdot x}
\endCD
$$
to transfer our constructions from (gerbe-twisted) sheaves on $\Gr_{T,\Ran}$ (resp., $\Gr_{T,\Ran_x}$)
to those on $\Conf$ (resp., $\Conf_{\infty\cdot x}$).

\medskip

We prove the theorem that the factorization algebra on $\Conf$, corresponding under this procedure to 
$\Omega_q^{\Whit_{!*}}\in \Shv_{\CG^T}(\Gr_{T,\Ran})$, identifies with the factorization algebra
$\Omega_q^{\on{small}}\in \Shv_{\CG^\Lambda}(\Conf)$ introduced above. 

\medskip

In particular, the functor $\fJ^{\bHecke}_{!*,\on{Fact}}$ gives rise to a functor
$$\Phi^{\bHecke}_{\on{Fact}}:\bHecke(\Whit_{q,x}(G))\to \Omega_q^{\on{small}}\on{-FactMod}.$$

\medskip

Let $\Phi^{\bHecke}$ be the functor
$$\bHecke(\Whit_{q,x}(G))\to \Shv_{\CG^\Lambda}(\Conf_{\infty\cdot x}),$$
obtained from $\Phi^{\bHecke}_{\on{Fact}}$ by forgetting the factorization 
$\Omega_q^{\on{small}}$-module structure. Let $\Phi$ be the pre-composition of $\Phi^{\bHecke}$
with the tautological functor
$$\ind_{\bHecke}:\Whit_{q,x}(G)\to \bHecke(\Whit_{q,x}(G)).$$

\medskip

We state a key result, \thmref{t:duality pres}, which says that the functor $\Phi$ commutes with Verdier duality. We use this
fact to show that the functor $\Phi^{\bHecke}_{\on{Fact}}$ is t-exact and sends irreducible objects in 
$(\bHecke(\Whit_{q,x}(G)))^\heartsuit$ to irreducible objects in $(\Omega_q^{\on{small}}\on{-FactMod})^\heartsuit$. 

\medskip

In \secref{s:statement} we define the renormalized version $\Omega_q^{\on{small}}\on{-FactMod}^{\on{ren}}$ 
of the category $\Omega_q^{\on{small}}\on{-FactMod}$ (see \secref{sss:renorm Intro} above), and we state
our main result, \thmref{t:main}, which says that the resulting functor
$$\Phi^{\bHecke,\on{ren}}_{\on{Fact}}:\bHecke(\Whit_{q,x}(G))\to 
\Omega_q^{\on{small}}\on{-FactMod}^{\on{ren}}$$
is an equivalence. In the same section we outline the strategy of the proof of \thmref{t:main}, which was already
mentioned in \secref{ss:how} above. 

\sssec{}

In Part VII we establish the commutation of $\Phi$ with Verdier duality, i.e., \thmref{t:duality pres}. 

\medskip

In \secref{s:Zastava} we give a global interpretation of the functor $\Phi$, in which the geometric object
known as the Zastava space plays a prominent role. 

\medskip

In \secref{s:proof of Verdier upstairs} we prove \thmref{t:duality pres} (or rather, its more precise version, which
is local on the Zastava space). The key idea is a certain ULA property of the \emph{global} metaplectic IC
sheaf on $\ol\Bun_{B^-}$. 

\medskip

In \secref{s:Hecke Zast} we prove a Hecke extension of \thmref{t:duality pres}, which says that the functor 
$\Phi^{\bHecke}$ commutes with Verdier duality.

\sssec{}

Part VIII is devoted to the realization of the outline of the proof of \thmref{t:main} indicated at the end of 
\secref{s:statement}. I.e., we need to construct the standard and costandard objects in $\bHecke(\Whit_{q,x}(G))$
and prove their properties. 

\medskip

In \secref{s:DrPl} we discuss the general framework of the Drinfeld-Pl\"ucker formalism (this formalism was 
suggested by S.~Raskin). We show that the restriction of the metaplectic semin-infinite IC sheaf to
the fiber over $x\in \Ran_x$ can be constructed via this formalism. 

\medskip

In \secref{s:Fl} we make a (necessary) digression and discuss the construction of the (dual) baby Verma object
in the category of \emph{Iwahori}-equivariant sheaves on the affine Grassmannian. We remark that this object
has appeared prominently in the papers \cite{ABBGM}, \cite{FG2}, \cite{FG3}, \cite{Ga8}. 

\medskip

In \secref{s:baby Verma in Whit} we construct the sought-for objects $\CM^{\mu,*}_{\Whit}$ (using the
(dual) baby Verma object in $\Shv_{\CG^G}(\Gr_G)^I$ constructed in \secref{s:Fl}). We then construct the 
objects $\CM^{\mu,!}_{\Whit}$ by applying duality. Finally, we verify the orthogonality property \eqref{e:orth intro}.  

\medskip

In \secref{s:stalks} we prove \eqref{e:standard to standard Intro}. This is where everything comes together in this
work. The isomorphism \eqref{e:costandard to costandard Intro} uses the full force of the Drinfeld-Pl\"ucker 
formalism. The isomorphism \eqref{e:* fiber St} is proved using the commutation of the functor 
$\Phi^{\bHecke}$ with Verdier duality.

\sssec{}

Part IX is logically disjoint from this rest of this work. Here we prove that in the Betti context we have a canonical
equivalence 
$$\overset{\bullet}\fu_q(\cG)\mod^{\on{baby-ren}}\simeq \Omega_q^{\on{small}}\on{-FactMod}$$
of \eqref{e:BFS}. 

\medskip

In \secref{s:quantum} we introduce the category $\overset{\bullet}\fu_q(\cG)\mod$ from the perspective 
of \emph{higher algebra} (i.e., with using explicit formulas as little as possible).

\medskip

In \secref{s:renorm} we introduce the two renormalizations of $\overset{\bullet}\fu_q(\cG)\mod$, denoted
$$\overset{\bullet}\fu_q(\cG)^{\on{ren}}\mod \text{ and } \overset{\bullet}\fu_q(\cG)\mod^{\on{baby-ren}}.$$

The former is compactly generated by the irreducible objects in $(\overset{\bullet}\fu_q(\cG)\mod)^\heartsuit$,
and the latter by the \emph{baby Verma modules} (hence the notation). 

\medskip

In \secref{s:BFS} we prove the equivalence \eqref{e:BFS}. The proof consists of three steps: (i) Koszul duality;
(ii) interpretation of the category of $\BE_2$-modules over an $\BE_2$-algebra as the Drinfeld center of
the monoidal category of left modules; (iii) equivalence between $\BE_2$-modules over an $\BE_2$-algebra
and factorization modules over the corresponding factorization algebra. 

\sssec{}

In the Appendix we introduce a device that we call the Kirillov model that allows to talk about Whittaker categories
when the Artin-Schreier sheaf does not exist, for example in the Betti setting (i.e., constructible sheaves in the
classical topology for schemes over $\BC$). 

\ssec{Conventions}

\sssec{Algebraic geometry}

We will work over an algebraically closed field $k$ (of arbitrary characteristic). Our algebro-geometric 
objects will be schemes, and more generally pre-stacks over $k$, i.e., (accessible) functors
$$(\affSch)^{\on{op}}\to \on{Groupoids}.$$

\medskip

Most of our prestacks will be \emph{locally of finite type}, which means that they are left-Kan-extended
from the full subcategory
$$(\affSch_{\on{ft}})^{\on{op}}\subset (\affSch)^{\on{op}},$$
consisting of affine schemes of finite type. See, e.g., \cite[Sect. 0.5.1]{Ga9} for details. 

\medskip

For the purposes of this work, we will \emph{not} need derived alegebraic geometry. Also, for our purposes
it is sufficient to consider classical (and not higher) groupoids.

\medskip

We denote $\on{pt}:=\Spec(k)$. 

\medskip

We let $X$ be a smooth connected curve over $k$ with a marked point $x\in X$. In some places we will
need $X$ to be complete; we will explicitly say so when this is the case. We let $\omega$ denote the 
canonical line bundle on $X$. 

\sssec{Reductive groups}

We let $G$ be a reductive group over $k$, with a chosen Borel and Cartan subgroups
$$T\subset B\subset G.$$

We let $\Lambda$ denote the cocharacter lattice of $T$. We let $\Lambda^+\subset \Lambda$ be the
sub-monoid of dominant coweights. Let $I$ denote the set of vertices of the Dynkin diagram of $G$.
For each $i\in I$ we let $\alpha_i\in \Lambda$ denote the corresponding coroot. We let
$\Lambda^{\on{pos}}\subset \Lambda$ be the submonoid equal to the positive integral span of 
the elements $\alpha_i$. 

\sssec{DG categories}

We let $\sfe$ be an algebraically closed field of characteristic $0$, which will serve as our field of coefficients. 
Our object of study is various $\sfe$-linear DG categories. We refer the reader to \cite[Chapter 1, Sect. 10]{GR1}
for a detailed exposition of the theory of DG categories. 

\medskip

Unless specified otherwise, our DG categories will be assumed cocomplete (i.e., closed under infinite direct sums).
By default, we will only consider functors between DG categories that commute with infinite direct sums 
(we call such functors \emph{continuous}). 

\medskip

We let $\Vect$ denote the DG category of chain complexes of vector spaces. 

\medskip

For $\bc_1,\bc_2\in \bC$, we denote by 
$$\CHom_\bC(\bc_1,\bc_2)\in \Vect$$ 
their ``internal Hom", i.e.,
$$\Hom_{\bC}(\bc_1,\bc_2)=H^0(\CHom_\bC(\bc_1,\bc_2)).$$

\sssec{}

If $\bC$ is a DG category equipped with a t-structure, we let $(\bC)^{\leq 0}$ (resp., $(\bC)^{\geq 0}$)
denote the full subcategory consisting of \emph{connective} (resp., \emph{coconnective}) objects.
We let $(\bC)^\heartsuit=(\bC)^{\leq 0}\cap (\bC)^{\geq 0}$ denote the heart of the t-structure. 

\sssec{}  \label{sss:small categ}

Recall that if $\bC$ is a (complete) DG category, one can talk about the (non-cocomplete) category
of its compact objects, denoted $\bC^c$. Vice versa, starting with a non-cocomplete category $\bC_0$,
one can form its ind-completion, denoted $\on{IndCompl}(\bC_0)$, which is universal among 
non-cocomplete categories receiving an exact functor from $\bC_0$. 

\medskip

Recall that a DG category is said to be compactly generated if the tautological functor
$$\on{IndCompl}(\bC^c)\to \bC$$
is an equivalence. 

\medskip

For $\bC_0$ as above, we always have $\bC_0\subset \on{IndCompl}(\bC_0)^c$; moreover, every object in
$\on{IndCompl}(\bC_0)^c$ is a retract of one in $\bC_0$. 

\sssec{Limits of DG categories}   \label{sss:limits and colimits}

One important think about DG categories is that they form an $\infty$-category. So one needs the full
force of higher category theory when forming limits and colimits of DG categories. 

\medskip

Here is one paradigm that appears often in representation theory. Let 
$$J\to \on{DGCat}, \quad j\mapsto \bC_j$$ 
be a functor, where $J$ is some index category. Suppose that for every arrow $(j_1\to j_2)\in J$, the corresponding functor
$\bC_{j_1}\to \bC_{j_2}$ admits a continuous right adjoint. By passing to the right adjoints we obtain another functor
$$J^{\on{op}}\to \on{DGCat}.$$

\medskip

Then for every $j_0\in J$, the tautological functor
$$\on{ins}_{j_0}:\bC_{j_0}\to \underset{j\in J}{\on{colim}}\, \bC_j$$
admits a continuous right adjoint. Furthermore, the resulting functor 
$$\underset{j\in J}{\on{colim}}\, \bC_j \to \underset{j\in I^{\on{op}}}{\on{colim}}\, \bC_j,$$
formed by these right adjoints, is an equivalence. 

\sssec{}

The $\infty$-category $\on{DGCat}$ of DG categories carries a symmetric monoidal structure,
Lurie's tensor product
$$\bC_1,\bC_2\mapsto \bC_1\otimes \bC_2.$$

Hence, one can talk about dualizable DG categories. It is known that compactly generated categories
are dualizable. Moreover, for $\bC_0$ as in \secref{sss:small categ}, we have
$$\on{IndCompl}(\bC_0)^\vee \simeq \on{IndCompl}((\bC_0)^{\on{op}}).$$

\medskip

In particular, if $\bC$ is compactly generated, we have a canonical equivalence
$$(\bC^c)^{\on{op}}\simeq (\bC^\vee)^c.$$

\sssec{Sheaf theories}  \label{sss:sheaf theory}

By a sheaf theory we will mean a right-lax symmetric monoidal functor
\begin{equation} \label{e:sheaf theory}
(\affSch_{\on{ft}})^{\on{op}}\to \on{DGCat}, \quad S\mapsto \Shv(S), \quad (S_1\overset{f}\to S_2)\mapsto 
(\Shv(S_2)\overset{f^!}\to \Shv(S_1)),
\end{equation} 
\begin{equation} \label{e:ten prod shv}
\Shv(S_1)\otimes \Shv(S_2)\to \Shv(S_1\otimes S_2).
\end{equation} 

Rather than axiomatizing the situation, we will list the examples of sheaf theories that we have in mind:

\begin{itemize}

\item For $k$ being of characteristic zero we can take $\Shv(S)=\Dmod(S)$. In this case the field $\sfe$
of coefficients equals $k$. We refer to this example as the D-module context.

\item For $k$ as above, we can take $\Shv(S)$ to be the ind-completion of the full subcategory of
$\Dmod(S)$ consisting of holonomic (resp., regular holonomic) D-modules. We refer to this example as
holonomic (resp., regular holonomic) context. 

\item For $k$ being $\BC$, we can take $\Shv(S)$ to be the ind-completion of the
category of $\sfe$-constructible sheaves in the classical topology, for any $\sfe$.
We refer to this example as the \emph{Betti context}.  

\item For any $k$, we can take $\Shv(S)$ to be the ind-completion of the
category of constructible $\ol\BQ_\ell$-adic sheaves on $S$. In this case the field
$\sfe$ of coefficients is $\ol\BQ_\ell$. We refer to this example as the \emph{$\ell$-adic context}. 

\end{itemize}

Note that in all of the above examples, the functor \eqref{e:ten prod shv} is fully faithful,
and in the first example, it is actually an equivalence. 

\medskip

We will refer to the last three examples as the \emph{constructible context}.

\sssec{Sheaves on prestacks}

Given a sheaf theory, we apply a procedure of \emph{right Kan extension} to extend it to a functor
\begin{equation} \label{e:shvs on prestacks}
\Shv:(\on{PreStk}_{\on{lft}})^{\on{op}}\to \on{DGCat}.
\end{equation} 

\medskip

Explicitly, for a prestack $\CY$
$$\Shv(\CY)=\underset{S\to \CY}{\on{lim}}\, \Shv(S),$$
where the limit is taken over the category (opposite to that) of affine schemes of finite type mapping to $\CY$.
Note that the above limit is formed within $\on{DGCat}$, so it is a higher categorical procedure. In particular,
are using the fact that \eqref{e:sheaf theory} is a functor of $\infty$-categories. 

\sssec{Functors defined on sheaves}

The functor \eqref{e:shvs on prestacks} has in fact more functoriality: it actually extends to a functor
out of the category of correspondences on prestacks, where we allow to take direct images
along ind-schematic maps, see \cite[Chapter 5, Sect. 2]{GR1} and \cite[Chapter 3, Sect. 5]{GR2}.

\medskip

In the constructible context, for any ind-schematic morphism $f:\CY_1\to \CY_2$, the functor
$f^!$ has a left adjoint, denoted $f_!$. In the context of D-modules, this left adjoint is only partially defined.

\medskip

Similarly, for a map of schemes $f:Y_1 \to Y_2$, the functor $f_*:\Shv(Y_1)\to \Shv(Y_2)$ has a partially
defined left adjoint $f^*$, which is always defined in the constructible context. 

\medskip

Finally, Verdier duality for a ind-scheme $Y$ is an equivalence
$$\BD^{\on{Verdier}}(\Shv(Y)^c)^{\on{op}}\to \Shv(Y)^c,$$
uniquely characterized by the requirement that
$$\CHom(\BD^{\on{Verdier}}(\CF_1),\CF_2)=\on{C}^\cdot(Y,\CF_1\sotimes \CF_2),$$
where $\on{C}^\cdot(Y,-)$ is the functor of cochains, i.e., direct image along $Y\to \on{pt}$. 

\sssec{Gerbes}

Let $\CY$ be a prestack and let $\CG$ be a gerbe on it with respect to the group
$\sfe^{\times,\on{tors}}$ of torsion elements in $\sfe^\times$ of orders prime to $\on{char}(p)$,
see \cite[Sect. 1.3]{GLys}. 

\medskip

The data of a gerbe $\CG$ allows to twist the category $\Shv(\CY)$ and obtain a new category $\Shv_\CG(\CY)$,
see \cite[Sect. 1.7.4]{GLys}.  Moreover, the functor \eqref{e:shvs on prestacks} extends to a functor
$$(\on{PreStk}+\on{Grb})^{\on{op}}\to \on{DGCat},$$
where the source category consists now of pairs 
$$(\CY\in \on{PreStk}_{\on{lft}},\CG\in \on{Grb}(\CY)),$$
and where the morphisms $(\CY_1,\CG_1)\to (\CY_2,\CG_2)$ are maps of prestacks $f:\CY_1\to \CY_2$ equipped with
an identification of the gerbes $f^*(\CG_2)\simeq \CG_1$.

\ssec{Acknowledgements}

\sssec{}

The first-named author is enormously indebted to J.~Lurie: this entire project originated from his original guess as to
the existence of the equivalence \eqref{e:Whit vs quant}. 

\medskip

We are very grateful to S.~Raskin for suggesting the idea Drinfeld-Pl\"ucker formalism: this made the proof
of \eqref{e:costandard to costandard Intro} short and conceptual (it was a mess in the previous version).

\medskip

Special thanks are due to M.~Finkelberg: it is from him that we learned the idea that one can describe interesting 
representation-theoretic categories as factorization modules (the book \cite{BFS} was the first example of this). 

\medskip

The idea of quantum Langlands theory originated in the (mostly unpublished) work of B.~Feigin, E.~Frenkel and A.~Stoyanovsky 
in the mid-1990's. Their insight was crucial in informing our formulation of the FLE. 

\medskip

This works inhabits the area known as the geometric Langlands theory, which was put forth in its modern form
by A.~Beilinson and V.~Drinfeld. The influence of their ideas over this work is all-pervasive. 

\sssec{}

The research of the first-named author was supported by NSF grant DMS-1063470. A large part of this work
was written when he was visiting IHES in summer 2018 and MSRI in spring 2019.  

\newpage 

\centerline{\bf Part I: Background and preliminaries} 

\bigskip

In order not to interrupt the flow of exposition in the main body of this work, in this Part 
we have collected several pieces of mathematical background that we will use. 

\medskip

These
include the definitions of the main geometric objects (such as the affine Grassmannian and
its various relatives), main tenets of the \emph{geometric metaplectic theory}; factorization
algebras and modules, and some discussion of configuration spaces (a.k.a. colored divisors). 

\section{The geometric objects}  \label{s:Grass}

In this section we will recall the definition of the key geometric objects that we will work with in this paper:
the affine Grassmannian, the loop group, the Hecke stack, etc.

\medskip

An important feature of these objects is that they have a \emph{factorizable nature}; we will take this 
perspective when introducing them. 

\ssec{The Ran space}  \label{ss:Ran}

The Ran space is a gadget that allows us to talk about factorization. In this subsection we recall 
its definition and key structures. 

\sssec{}  \label{sss:Ran}

Recall that the Ran space $\Ran$ of $X$ assigns to a test scheme $S$ the set of finite non-empty subsets
$\CI\subset \Maps(S,X)$.

\medskip

For an element $i\in \CI$ we will denote by $\Gamma_i\subset S\times X$ the graph of the corresponding map. 
Let $\Gamma_\CI$ denote the (set-theoretic) union of $\Gamma_i$ over $i\in \CI$. 

\sssec{}

A basic structure that exists on the Ran space is the structure of commutative semi-group: for a finite set $J$ we have the map
\begin{equation} \label{e:union map}
\Ran^J\to \Ran, \quad (\CI_j\subset \Maps(S,X),\, j\in J)\,\, \mapsto\,\, (\underset{j\in J}\cup\, \CI_j\subset \Maps(S,X)).
\end{equation}

\sssec{}

For a finite set $J$, let
$$(\Ran^J)_{\on{disj}}\subset \Ran^J$$
be the open subfunctor corresonding to the following condition: we allow those $J$-subsets
$$\CI_j\subset \Maps(S,X), \, j\in J$$
such that for every $j_1\neq j_2$
$$\Gamma_{\CI_{j_1}}\cap \Gamma_{\CI_{j_2}}=\emptyset.$$

\medskip

An important feature is that the map 
\begin{equation} \label{e:union map disj}
(\Ran^J)_{\on{disj}}\to \Ran,
\end{equation}
induced by \eqref{e:union map}, is \'etale. 

\sssec{}  \label{sss:fact space naive}

By a factorization space over $\Ran$ we will mean a prestack $Z_\Ran \to \Ran$ equipped with a system of identifications
(factorization isomorphisms)
\begin{equation} \label{e:factorization space}
Z_\Ran\underset{\Ran}\times (\Ran^J)_{\on{disj}}\simeq Z_\Ran^J\underset{\Ran^J}\times (\Ran^J)_{\on{disj}},
\end{equation} 
for every finite set $J$, that are compatible in a natural sense (see \cite[Sect. 2.2.1]{GLys} or \secref{sss:fact spc} below). 

\ssec{Examples of factorization spaces}

Having introduced the Ran space, we will now move one step closer to the definition of the geometric objects 
we will be working with. 

\medskip

These objects are obtained by forming \emph{loop spaces}, which have to do with 
mapping the \emph{multi-disc} parameterized by points of $\Ran$ to a given target. 

\sssec{}

Let $\CI\subset \Hom(S,X)$ be as in \secref{sss:Ran}. We let $\wh\cD_\CI$
the formal scheme equal to the completion of $S\times X$ along $\Gamma_\CI$.

\medskip

Let $\cD_\CI$ denote the affine scheme obtained from $\wh\cD_\CI$ (i.e., the universal 
recipient of a map from $\cD_\CI$ among affine schemes, see \cite[Sect. 7.1.2]{GLys}).

\medskip

Note that $\Gamma_J$ is naturally a closed subset of $\cD_\CI$. Denote
$$\ocD_\CI:=\cD_\CI-\Gamma_\CI.$$ 

It is easy to see that $\ocD_\CI$ is also affine. 

\sssec{}

Note that for
$$(\CI_j,j\in J)\in (\Ran^J)_{\on{disj}}, \quad \CI:=\underset{j}\cup\, \CI_j$$
we have
$$\wh\cD_\CI\simeq \underset{j}\sqcup\, \wh\cD_{\CI_j}.$$

\medskip 

From here,
\begin{equation} \label{e:factorization disk}
\cD_\CI\simeq \underset{j}\sqcup\, \cD_{\CI_j} \text{ and } 
\ocD_\CI\simeq \underset{j}\sqcup\, \ocD_{\CI_j}.
\end{equation}

\sssec{}

For a prestack $Y$, define the prestack
$$\fL^+(Y)_{\Ran}\to \Ran$$
by assigning to $\CI$ as above the space of maps
$$\wh\cD_\CI\to Y.$$

When $Y$ is affine scheme, maps as above are the same as maps 
$$\cD_\CI\to Y,$$
by the universal property of $\cD_\CI$. 

\sssec{}

Define 
$$\fL(Y)_{\Ran}\to \Ran$$
by assigning to $\CI$ the space of maps $\ocD_\CI\to Y$.

\medskip

These spaces have natural factorization structures due to the identifications
\eqref{e:factorization disk}.

\sssec{}  \label{sss:loop rel}

The above definitions have a variant when instead of an affine scheme $Y$, we have an 
affine morphism $Y_X\to X$. In this case we will be looking at maps
$$\cD_\CI\to Y_X \text{ and } \ocD_\CI\to Y_X,$$
respectively, \emph{over} $X$. 

\medskip

The assignments 
$$Y_X\rightsquigarrow \fL^+(X)_\Ran \text{ and } Y_X\rightsquigarrow \fL(Y_X)_\Ran$$
are functorial in $Y_X$.

\medskip

In particular, we obtain that a section $X\to Y_X$ of $Y_X\to X$ gives rise to a map
of factorization spaces
$$\Ran\to \fL^+(X)_\Ran.$$ 

\ssec{The affine Grassmannian and other animals}  \label{ss:Gr}

We will now specialize to the case when the target space $Y$ is an algebraic group $G$, and thus 
define the actual geometric objects of interest. 

\sssec{}

Note that, by functoriality, when the target space is $G$, the prestacks $\fL^+(G)_\Ran$ and $\fL(G)_\Ran$ acquire
a structure of group-spaces over $\Ran$.

\sssec{}  \label{sss:G bndls on disk}

Consider the quotient $\fL^+(G)_{\Ran}\backslash \Ran$ (understood in the \'etale sense).
It attaches to $\CI$ the datum of a $G$-bundle on $\cD_\CI$. It is easy to see, however, that restriction 
defines an equivalence from $G$-bundles on $\cD_\CI$ to those on $\wh\cD_\CI$. 

\medskip

Note that we can think of $\fL^+(G)\backslash \Ran$ also as $\fL^+(G\backslash \on{pt})$. 

\medskip

In particular, from \secref{sss:loop rel}, we obtain that a $G$-bundle on $X$ defines a map of factorization spaces
$$\Ran\to \fL^+(G)_{\Ran}\backslash \Ran.$$

\sssec{}  \label{sss:aff Gr}

The (Ran version of the) affine Grassmannian $\Gr_{G,\Ran}$ can be defined 
as the \'etale sheafification of the quotient 
$$\fL(G)_{\Ran}/\fL^+(G)_{\Ran}.$$

Equivalently, 
$$\Gr_{G,\Ran}\to \Ran$$
assigns to $\CI$ as above the data of pairs $(\CP_G,\alpha)$, where:

\begin{itemize}

\item $\CP^G$ is a $G$-bundle $\CP^G$ on $\cD_\CI$ (equivalently, on $\wh\cD_\CI$);

\item $\alpha$ is a trivialization of the restriction to $\CP_G$ to $\ocD_\CI$.

\end{itemize} 

\sssec{}  \label{sss:BL}

A theorem of Beauville and Laszlo says that the affine Grassmannian can be 
defined using the curve $X$ instead of the disc. Namely, we can take pairs $(\CP_G,\alpha)$, where:

\begin{itemize}

\item $\CP_G$ is a $G$-bundle $\CP^G$ on $S\times X$;

\item $\alpha$ is a trivialization of the restriction to $\CP_G$ to $S\times X-\Gamma_\CI$.

\end{itemize} 

Restriction defines a map from the data above to that in \secref{sss:aff Gr}. The Beauville-Laszlo
theorem says that this map is a bijection.

\sssec{}

Set
$$\on{Hecke}^{\on{loc}}_{G,\Ran}:=\fL^+(G)_{\Ran}\backslash \Gr_{G,\Ran};$$
where the quotient is understood in the \'etale sense. 

\medskip

The functor of points of $\on{Hecke}^{\on{loc}}_{G,\Ran}$ consists of triples $(\CP'_G,\CP_G,\alpha)$, where:

\begin{itemize}

\item $\CP'_G$ is a $G$-bundle $\CP^G$ on $\cD_\CI$;

\item $\CP_G$ is a $G$-bundle $\CP^G$ on $\cD_\CI$;

\item $\alpha$ is an isomorphism $\CP'_G|_{\ocD_\CI}\simeq \CP_G|_{\ocD_\CI}$. 

\end{itemize} 

\sssec{}

In what follows we will denote by $\hl_G,\hr_G$ the two maps 
$$\on{Hecke}^{\on{loc}}_{G,\Ran}\rightrightarrows \fL^+(G)_{\Ran}\backslash \Ran$$
equal to
$$\on{Hecke}^{\on{loc}}_{G,\Ran}=\fL^+(G)_{\Ran}\backslash \Gr_{G,\Ran}\to \fL^+(G)_{\Ran}\backslash \Ran$$
\begin{multline*} 
\on{Hecke}^{\on{loc}}_{G,\Ran}=\fL^+(G)_{\Ran}\backslash \Gr_{G,\Ran}\simeq 
\fL^+(G)_{\Ran}\backslash \fL(G)_{\Ran}/\fL^+(G)_{\Ran}\to \\
\to \fL^+(G)_{\Ran}\backslash \fL(G)_{\Ran}/\fL(G)_{\Ran}\simeq
\fL^+(G)_{\Ran}\backslash \Ran,
\end{multline*} 
respectively. 

\medskip

In other words, the map $\hl_G$ remembers the data of $\CP'_G$ and the map $\hr_G$ remembers the data of $\CP_G$. 

\sssec{}  \label{sss:inv on Hecke}

The prestack $\on{Hecke}^{\on{loc}}_{G,\Ran}$ carries an involution (swapping the roles of $\CP_G$ and $\CP'_G$) denoted
$\inv^G$, which interchanges $\hl_G$ and $\hr_G$. 

\ssec{Twist by the canonical bundle}

In this subsection we will take $G$ to be reductive, with a chosen Borel subgroup $B$ and a Cartan subgroup $T\subset B$. 

\medskip

In this subsection we will introduce twisted versions of $\Gr_G$, $\fL(G)$, etc., that have to do with the canonical line bundle
on $X$. 

\medskip

The necessity for such a twist can be explained succinctly as follows: it makes the additive character on the loop
group $\fL(N)$ canonical. 

\sssec{}

For what follows we will choose a square root $\omega^{\otimes \frac{1}{2}}$ of the canonical line bundle $\omega$ on $X$.
Let $\omega^\rho$ denote the $T$-bundle on $X$ induced from $\omega^{\otimes \frac{1}{2}}$
by means of the homomorphism
$$2\rho:\BG_m\to T.$$ 

\medskip

We will denote by the same character $\omega^\rho$ the induced $B$-bundle and $G$-bundle via
$$T\hookrightarrow B\hookrightarrow G.$$

\sssec{}  \label{sss:omega rho as twist}

By \secref{sss:G bndls on disk}, the datum of $\omega^\rho$ defines a map of factorization spaces
$$\Ran\to \fL^+(T)_\Ran\backslash \Ran,$$
i.e., a $\fL^+(T)_\Ran$-torsor over $\Ran$, denoted $\omega^\rho_\Ran$, compatible with factorization. 

\medskip

Using the (adjoint) action of $\fL^+(T)_\Ran$ on the objects introduced in \secref{ss:Gr}, we obtain their twisted versions, to be denoted
$$\fL^+(G)_\Ran^{\omega^\rho}, \,\, \fL(G)_\Ran^{\omega^\rho},\,\, \Gr_{G,\Ran}^{\omega^\rho}.$$

\sssec{}   \label{sss:twisted group scheme}

Note that we can identify $\fL^+(G)_\Ran^{\omega^\rho}$ (resp., $\fL(G)_\Ran^{\omega^\rho}$) with
$\fL^+(G^{\omega^\rho})_\Ran$ (resp., $\fL(G^{\omega^\rho})_\Ran$), where $G^{\omega^\rho}$ is the group-scheme
over $X$ obtained by twisting $G$ by means of $\omega^\rho$.

\sssec{}   \label{sss:Gr G rho}

The twisted version $\Gr_{G,\Ran}^{\omega^\rho}$ of the affine Grassmannian can be explicitly described as follows: it assigns
to $\CJ$ the set of pairs $(\CP_G,\alpha)$, where:

\begin{itemize}

\item $\CP^G$ is a $G$-bundle $\CP^G$ on $\cD_\CI$ (equivalently, on $\wh\cD_\CI$);

\item $\alpha$ is an identification of the restriction to $\CP_G$ to $\ocD_\CI$ with that of $\omega^\rho$. 

\end{itemize} 

\begin{rem}  \label{r:no twist Hecke}

Note that the adjoint action of $\fL^+(G)_{\on{Ran}}$ on $\on{Hecke}^{\on{loc}}_{G,\Ran}$ is 
canonically trivialized. Hence, the twist
$$(\on{Hecke}^{\on{loc}}_{G,\Ran})^{\omega^\rho}$$
identifies canonically with the non-twisted version $\on{Hecke}^{\on{loc}}_{G,\Ran}$.

\end{rem}

\begin{rem} \label{r:no twist Gr T}

Note that in the case when $G=T$, the operation of tensoring a given $T$-torsor with the $T$-torsor $\omega^\rho$, 
gives rise to an identification
$$\Gr_{T,\Ran}^{\omega^\rho}\simeq \Gr_{T,\Ran}.$$

However, we will avoid using it, as it is \emph{incompatible} with gerbes considered 
in the next section, see Remark \ref{r:gerbe on twist}. 

\end{rem}

\sssec{}  \label{sss:twisted group N}

Since $N\subset B\subset G$ are normalized by $T$, we can form the twisted forms of the corresponding loop groups
$$\fL^+(B)_\Ran^{\omega^\rho}, \,\, \fL(B)_\Ran^{\omega^\rho}$$
and
$$\fL^+(N)_\Ran^{\omega^\rho}, \,\, \fL(N)_\Ran^{\omega^\rho}.$$

\medskip

Note, however, that since the adjoint action of $T$ on itself is trivial, we have
$$\fL^+(T)_\Ran^{\omega^\rho}=\fL^+(T)_\Ran \text{ and } \fL(T)_\Ran^{\omega^\rho}=\fL(T)_\Ran.$$

\sssec{}  \label{sss:twisted N}

Consider the special case when $G=PGL_2$, in which case $N\simeq \BG_a$, and the adjoint action of $T=\BG_m$
is given by dilations. We have:
$$\fL(N)^{\omega^\rho}_\Ran\simeq \fL(\BG_a)^{\omega}_\Ran\simeq \fL(\on{Tot}(\omega))_\Ran,$$
where $\on{Tot}(\omega)$ as the total space of $\omega$, viewed as a group-scheme over $X$.

\medskip

Note that have a canonically defined homomorphism
$$\fL(\BG_a)^{\omega}_\Ran\to \BG_a\times \Ran,$$
given by taking the residue. 

\ssec{Marked point version}

Factorization structures in this paper will appear as a device to introduce representation-theoretic categories
of interest. However, the actual representation theory will occur at (or around) a given point $x$ of the curve $X$.

\medskip

In this subsection we will introduce versions of the factorization spaces considered above but \emph{with the marked point $x$}.

\medskip

This will eventually allow us to talk about \emph{factorization modules} (at $x$) for a given \emph{factorization algebra}. 

\sssec{}

The marked point version of the Ran space, denoted $\Ran_x$ is defined as follows. 

\medskip

For a test scheme $S$, the space 
$\Maps(S,\Ran_x)$ is the set of finite subsets $\CI\subset \Maps(S,X)$ with a distinguished element $i_x$, which 
corresponds to the constant map 
\begin{equation} \label{e:constant x}
S\to \on{pt}\overset{x}\to X.
\end{equation}

\medskip

We have the natural forgetful map
$$\Ran_x\to \Ran$$
as well as a map
$$\Ran\to \Ran_x,$$
obtained by adding the element $i_x$. 

\medskip

The semi-group $\Ran$ acts on $\Ran_x$ by the operation of union of finite sets, i.e., for a finite set $J$ we have a map
\begin{equation} \label{e:union map ptd}
\Ran^J \times \Ran_x\to \Ran_x
\end{equation}

\sssec{}

For a finite set $J$, let
$$(\Ran^J\times \Ran_x)_{\on{disj}}\subset \Ran^J\times \Ran_x$$
be equal to the preimage of 
$$(\Ran^{J\sqcup *})_{\on{disj}}\subset \Ran^{J\sqcup *}$$
under the forgetful map
$$\Ran^J\times \Ran_x\to \Ran^J\times \Ran =\Ran^{J\sqcup *}.$$

\medskip

The map
\begin{equation} \label{e:union map ptd disj}
(\Ran^J\times \Ran_x)_{\on{disj}}\to \Ran_x,
\end{equation}
induced by \eqref{e:union map ptd}, is \'etale. 

\sssec{}

Let $Z_\Ran$ be a factorization space over $\Ran$. By a \emph{factorization module space} with respect to $Z_\Ran$ we will mean a prestack
$$Z_{\Ran_x}\to \Ran_x,$$
equipped with a system of identifications (factorization isomorphisms)
\begin{equation}   \label{e:factorization space mod}
Z_{\Ran_x}\underset{\Ran_x}\times (\Ran^J\times \Ran_x)_{\on{disj}}\simeq 
(Z_{\Ran_x}\times Z_\Ran^J)\underset{\Ran_x\times  \Ran^J}\times (\Ran^J\times \Ran_x)_{\on{disj}},
\end{equation} 
that are compatible in the natural sense.

\medskip

In what follows we will denote by $Z_x$ the fiber of $Z_{\Ran_x}$ over $\{x\}\in \Ran_x$. 

\sssec{An example}  \label{sss:ex fact mod}

Let $Z_\Ran$ be a factorization space over $\Ran$. We can produce from it a factorization module space by setting
$$Z_{\Ran_x}:=\Ran_x\underset{\Ran}\times Z_\Ran.$$

\sssec{}  \label{sss:Gr with marked}

In this way, from the factorization spaces discussed in \secref{ss:Gr} we obtain their marked point versions, denoted
$$\fL^+(G)_{\Ran_x}, \,\, \fL(G)_{\Ran_x},\,\, \Gr_{G,\Ran_x}.$$
along with their twists
$$\fL^+(G)_{\Ran_x}^{\omega^\rho}, \,\, \fL(G)_{\Ran_x}^{\omega^\rho},\,\, \Gr_{G,\Ran_x}^{\omega^\rho},$$
etc.

\begin{rem}
We should point out that there are many more examples of factorization module spaces that \emph{do not}
arise by the construction of \secref{sss:ex fact mod}.

\medskip

For example, if we take the factorization space $\Gr_{G,\Ran}$, we can create a factorization module space
$\Fl_{G,\Ran_x}$ by assigning to $\CI$ the set of triples $(\CP_G,\alpha,\beta)$, where $(\CP_G,\alpha)$ are
as in the definition of $\Gr_{G,\Ran_x}$, and $\beta$ is the datum of reduction of the restriction 
of $\CP_G$ at $S\times x$ to $B$.
\end{rem} 

\sssec{}

Let $\fL^+(G)_{\Ran_x,\infty\cdot x}$ denote the following factorizatiom module space with respect to
$\fL(G)_{\Ran}$. It assigns to $\CI$ the datum of a map
$$(\cD_\CI-S\times x)\to G.$$

\medskip

Note that the fiber $\fL^+(G)_{x,\infty\cdot x}$ of $\fL^+(G)_{\Ran_x,\infty\cdot x}$ over $x\in \Ran_x$
identifies with $\fL(G)_x$. 

\medskip

Note also that restriction to
$$\ocD_x\hookrightarrow (\cD_\CI-S\times x)$$
(where $\ocD_x$ is the pictured disc corresponding to the constant map \eqref{e:constant x})
defines a map 
\begin{equation} \label{e:ev at x loops}
\fL^+(G)_{x,\infty\cdot x}\to \fL(G)_x. 
\end{equation} 

\medskip

The inclusion
$$(\cD_\CI-S\times x)\hookrightarrow \ocD_\CI$$
defines a closed embedding 
$$\fL^+(G)_{\Ran_x,\infty\cdot x}\to \fL(G)_{\Ran_x}.$$

\sssec{} \label{sss:restr Hecke}

Consider the double quotient 
\begin{equation} \label{e:restr Hecke}
\fL^+(G)\backslash \fL^+(G)_{x,\infty\cdot x}/\fL^+(G).
\end{equation} 

On the other hand, we can view it as a closed sunfunctor in $\on{Hecke}^{\on{loc}}_{G,\Ran_x}$,
and as such, as a groupoid acting on $\fL^+(G)_{\Ran_x}\backslash \Ran_x$. 

\medskip

On the other hand, the map \eqref{e:ev at x loops} gives rise to the following commutative diagram in which both squares are Cartesian:
$$
\CD
\fL^+(G)_{\Ran_x}\backslash \Ran_x @<{\hl}<< \fL^+(G)_{\Ran_x,\infty\cdot x} @>{\hr}>> \fL^+(G)_{\Ran_x}\backslash \Ran_x \\
@VVV   @VVV   @VVV   \\
\fL^+(G)_x\backslash \on{pt}  @<{\hl}<<  \on{Hecke}^{\on{loc}}_{G,x} @>{\hr}>> \fL^+(G)_x\backslash \on{pt},
\endCD
$$
where the maps $\fL^+(G)_{\Ran_x}\backslash \Ran_x\to \fL^+(G)_x\backslash \on{pt}$ are given by restricting bundles along
$$\cD_x\hookrightarrow \cD_\CI.$$

So we can view \eqref{e:restr Hecke} as incarnating the lift of the action of $\on{Hecke}^{\on{loc}}_{G,x}$ on 
$\fL^+(G)_x\backslash \on{pt}$ to that on $\fL^+(G)_{\Ran_x}\backslash \Ran_x$. 

\ssec{Technical detour: unital structures}  \label{ss:unital}

The discussion of \emph{unitality} for factorization structures will be largely suppressed in this work.
However, for technical purposes that have to do with the construction of the semi-infinite IC sheaf, 
we will need to include a brief discussion.

\medskip

The idea of ``unitality" in the context of spaces over $\Ran$ is to account for the following additional structure:
we can talk not only about \emph{equality} of two finite subsets of $\Hom(S,X)$, but also about \emph{containment}
of one subset in another. 

\sssec{}

Let 
$$(\Ran\times \Ran)^{\subset} \subset \Ran\times \Ran$$
be the following subfunctor:

\medskip

A point $(\CI,\CI')\in \Hom(S,\Ran)$ belongs to $(\Ran\times \Ran)^{\subset}$ if $\Gamma_\CI$ is
set-theoretically contained in $\Gamma_{\CI'}$. 

\begin{rem}
Note that the substack $(\Ran\times \Ran)^{\subset}$ as defined above is larger than
the substack denoted by the same symbol in \cite[4.1.1]{Ga7}. However, this is largely
immaterial for our purposes: the above inclusion induces an equivalence
on categories of sheaves as it is surjective in the topology generated by finite 
surjective maps.
\end{rem}

\sssec{}

Note that the diagonal map $\Delta_\Ran:\Ran\to \Ran\times \Ran$ factors through $(\Ran\times \Ran)^{\subset}$. 

\medskip

Let $\varphi_{\on{small}}$ and $\varphi_{\on{big}}$ be the two maps
$$(\Ran\times \Ran)^{\subset}\to \Ran$$
that remember $\CI$ and $\CI'$, respectively.

\medskip

The following is established in \cite[Lemma 4.1.2]{Ga7}:

\begin{lem}  \label{l:UHC unital}  \hfill

\smallskip

\noindent{\em(a)}
For any prestack $\CY\to \Ran$, pullback with respect to the base-changed map
$$\on{id}_\CY\times \varphi_{\on{small}}: 
\CY\underset{\Ran,\varphi_{\on{small}}}\times (\Ran\times \Ran)^{\subset}\to \CY$$
induces an equivalence on categories of \'etale gerbes.

\smallskip

\noindent{\em(b)}
For any gerbe $\CG$ on $\CY$, the functor of !-pullback along $\on{id}_\CY\times \varphi_{\on{small}}$
induces a fully faithful functor
$$\Shv_{\CG}(\CY) \to \Shv_\CG(\CY\underset{\Ran,\varphi_{\on{small}}}\times (\Ran\times \Ran)^{\subset}).$$
\end{lem} 

\sssec{}

Let $Z_\Ran$ be a prestack over $\Ran$. By a unital structure on $Z_{\Ran}$ we will mean a map
$$\varphi_{\on{big}}:
Z_{\Ran}\underset{\Ran,\varphi_{\on{small}}}\times (\Ran\times \Ran)^{\subset})\to Z_{\Ran},$$
which makes the following diagram commute
$$
\CD
Z_{\Ran}\underset{\Ran,\varphi_{\on{small}}}\times (\Ran\times \Ran)^{\subset}) @>{\varphi_{\on{big}}}>>   Z_{\Ran}  \\
@VVV   @VVV   \\
(\Ran\times \Ran)^{\subset}   @>{\varphi_{\on{big}}}>> \Ran.
\endCD
$$
We also require that $\varphi_{\on{big}}$ be associative in a natural sense, and that the composite
$$Z_\Ran \overset{\on{id}\times \Delta_\Ran}\longrightarrow 
Z_{\Ran}\underset{\Ran,\varphi_{\on{small}}}\times (\Ran\times \Ran)^{\subset}) \overset{\varphi_{\on{big}}}\longrightarrow Z_{\Ran}$$
be the identity map. 

\sssec{} \label{sss:unital gerbe}

From \lemref{l:UHC unital}(a) we obtain:

\begin{cor}  \label{c:gerbe unital}
For a gerbe $\CG$ on $\Gr_{G,\Ran}$ we have a canonical isomorphism
$$\varphi_{\on{big}}^*(\CG)\simeq \varphi_{\on{small}}^*(\CG),$$
uniquely characterized by the requirement that the composite
$$\CG\simeq \Delta_{\Ran}^*\circ \varphi_{\on{big}}^*(\CG)\simeq \Delta_{\Ran}^*\circ \varphi_{\on{small}}^*(\CG)\simeq \CG$$
is the identity map.
\end{cor} 

\sssec{}

If $Z_\Ran$ has a factorization structure over $\Ran$, then 
$Z_{\Ran}\underset{\Ran,\varphi_{\on{small}}}\times (\Ran\times \Ran)^{\subset}$, 
viewed as mapping to $\Ran$ via
$$Z_{\Ran}\underset{\Ran,\varphi_{\on{small}}}\times (\Ran\times \Ran)^{\subset}\to 
(\Ran\times \Ran)^{\subset} \overset{\varphi_{\on{big}}}\longrightarrow \Ran$$
also has a natural factorization structure. 

\medskip

This, we can talk about a unital structure being compatible with factorization. 

\sssec{}   \label{sss:unitality Gr} 

Note that $\Gr_{G,\Ran}$ (or its variant $\Gr^{\omega^\rho}_{G,\Ran}$) provides an example of a factorization space 
with a unital structure. Indeed the map $\varphi_{\on{big}}$ is defined as follows:

\medskip

Let us interpret the affine Grassmannian as in \secref{sss:BL}. Then $\varphi_{\on{big}}$
sends $(\CI,\CI',\CP_G,\alpha)$ to $(\CI',\CP_G,\alpha')$,
where $\alpha'$ is the restriction of $\alpha$ along
$$(S\times X-\Gamma_\CI)\hookrightarrow (S\times X-\Gamma_{\CI'}).$$

\section{Geometric metaplectic data}  \label{s:metap}

The object of study of this work is \emph{metaplectically twisted} sheaves on geometries attached to
the loop group $\fL(G)$. In this section we explain what data goes into defining such a twist. 

\medskip

We will also discuss the phenomenon of \emph{metaplectic Langlands duality}. 

\ssec{Definition of the geometric metaplectic data}

In this subsection we recall, following \cite{GLys}, the definition of geometric metaplectic data. 

\sssec{}

Given a factorization space over $\Ran$, it makes sense to talk about $\sfe^{\times,\on{tors}}$-gerbes over it compatible with factorization
(see \cite[Sect. 2.2.4]{GLys} or \secref{sss:fact grb} below).

\medskip

By a \emph{geometric metaplectic data} for $G$ over $X$ we will mean a factorization gerbe over $\Gr_{G,\Ran}$.
In what follows we will denote such a gerbe by $\CG^G$. 

\sssec{}

Consider now the group-objects
$$\fL^+(G)_\Ran \hookrightarrow \fL(G)_\Ran.$$

We can talk about \emph{multiplicative} factorization gerbes on $\fL(G)_\Ran$ (resp., $\fL^+(G)_\Ran$),
i.e., factorization gerbes compatible with the group structure. According to \cite[Proposition 7.3.5]{GLys}, the map

\medskip

\noindent from the space of \hfill 

\smallskip

\noindent--\emph{Multiplicative factorization gerbes on $\fL(G)_\Ran$ equipped with a trivialization 
of the restriction to $\fL^+(G)_\Ran$ (as a multiplicative factorization gerbe)} \hfill 

\smallskip

\noindent to the space of \hfill 

\smallskip

\noindent--\emph{Factorization gerbes on $\Gr_{G,\Ran}$}, \hfill  

\smallskip

\noindent given by descent along $\fL(G)_{\Ran}\to \Gr_{G,\Ran}$, is an equivalence. 

\medskip

Thus, given a geometric metaplectic data, we obtain a multiplicative factorization gerbe on $\fL(G)_\Ran$, also denoted
$\CG^G$, equipped with a trivialization of its restriction to $\fL^+(G)_\Ran$. 

\medskip

By construction, the gerbe $\CG^G$ on $\Gr_{G,\Ran}$
is twisted-equivariant with respect to the action of $\fL(G)_\Ran$ on $\Gr_{G,\Ran}$ against the multiplicative gerbe $\CG^G$ on 
$\fL(G)_\Ran$. 

\sssec{}  \label{sss:gerbe on Hecke}

Since the gerbe $\CG^G$ on $\Gr_{G,\Ran}$ is  $\fL^+(G)_\Ran$-equivariant, it descends to a factorization gerbe on 
$\on{Hecke}^{\on{loc}}_{G,\Ran}$ that we will denote by $\CG^{G,G,\on{ratio}}$. 

\medskip

Note that the involution on $\on{Hecke}^{\on{loc}}_{G,\Ran}$ from \secref{sss:inv on Hecke} turns 
$\CG^{G,G,\on{ratio}}$ to $(\CG^{G,G,\on{ratio}})^{-1}$.

\sssec{}   \label{sss:gerbe on twisted}

Recall the twisted version of $\Gr^{\omega^\rho}_{G,\Ran}$ of $\Gr_{G,\Ran}$ from \secref{sss:Gr G rho}. We have a projection
$$\Gr^{\omega^\rho}_{G,\Ran}\to \on{Hecke}^{\on{loc}}_{G,\Ran}.$$

Pulling back $\CG^{G,G,\on{ratio}}$ with respect to this projection gives rise to a factorization gerbe on $\Gr^{\omega^\rho}_{G,\Ran}$.
We will denote it by the same character $\CG^G$.

\begin{rem} \label{r:gerbe on twist}
There should be no danger of confusing $\CG^G$ on $\Gr_{G,\Ran}$ and $\CG^G$ on $\Gr^{\omega^\rho}_{G,\Ran}$,
as they live on different spaces. See, however, Remark \ref{r:no twist Gr T}.
\end{rem}

\sssec{}

Assume for a moment that $G=T$ is a torus. Here is explicit description of the fibers of the gerbe 
$\CG^T$ on $\Gr^{\omega^\rho}_{T,\Ran}$:
 
\medskip

For a point $x\in X$ and the point $t^\lambda\in \Gr^{\omega^\rho}_{T,x}$, the fiber $\CG^T|_{t^\lambda}$
identifies with
\begin{equation} \label{e:expl shifted gerbe vac}
\CG^T_{\lambda\cdot x}\otimes  (\omega_x^{\otimes\frac{1}{2}})^{b(\lambda,2\rho)},
\end{equation} 
where:

\begin{itemize}

\item $\CG^T_{\lambda\cdot x}$ is the fiber of the gerbe $\CG^T$ on $\Gr_{T,x}$ 
at the point $t^\lambda\in \Gr_{T,x}$;

\item $\omega^{\otimes \frac{1}{2}}_x$ is the fiber of $\omega^{\otimes \frac{1}{2}}$ ar $x\in X$;

\smallskip

\item $b:\Lambda\times \Lambda\to \sfe^\times(-1)$ is the symmetric bilinear form associated to $\CG^T$, see \secref{sss:sharp lattice} below.

\end{itemize}

\medskip

In other words, the passage $\Gr_{T,x}\rightsquigarrow \Gr^{\omega^\rho}_{T,x}$ results in the additional factor isomorphic to
$(\omega_x^{\otimes\frac{1}{2}})^{b(\lambda,2\rho)}$. 

\sssec{}  \label{sss:gerbe on twisted grp}

The gerbe $\CG^G$ on $\Gr^{\omega^\rho}_{G,\Ran}$ can be also seen as obtained from the gerbe 
$\CG^G$ on $\Gr_{G,\Ran}$ by applying the twisting construction using the $\fL^+(T)_{\Ran}$-action 
by means of the $\fL^+(T)_{\Ran}$-torsor $\omega^\rho_\Ran$, see \secref{sss:omega rho as twist}. 

\medskip

This twisting construction produces also a multiplicative factorization gerbe (still denoted $\CG^G$) on
$\fL(G)^{\omega^\rho}_\Ran$, equipped with the trivialization of its restriction to $\fL^+(G)^{\omega^\rho}_\Ran$.

\medskip 

The gerbe $\CG^G$ on $\Gr^{\omega^\rho}_{G,\Ran}$ is twisted-equivariant with respect to the action of 
$\fL(G)^{\omega^\rho}_\Ran$ on $\Gr^{\omega^\rho}_{G,\Ran}$ against the multiplicative gerbe $\CG^G$ on 
$\fL(G)^{\omega^\rho}_\Ran$. 

\sssec{}  \label{sss:gerbe global} 

Assume for a moment that $X$ is complete. I this case we can consider the algebraic stack $\Bun_G$ classifying
$G$-bundles on $X$.

\medskip

Consider the projection 
$$\Gr_{G,\Ran}\to \Bun_G.$$

According to \cite[Sect. 2.3.5]{GLys}, \emph{any} gerbe on $\Gr_{G,\Ran}$ uniquely descends to a gerbe on
$\Bun_G$.  We will denote by the same character $\CG^G$ the resulting gerbe on $\Bun_G$.

\begin{rem}
Note that we also have a map
$$\Gr^{\omega^\rho}_{G,\Ran}\to \Bun_G.$$

The pullback of $\CG^G$ along this map differs from the gerbe we denoted $\CG^G$ on $\Gr^{\omega^\rho}_{G,\Ran}$ by
tensoring by $\CG^G|_{\omega^\rho}$.
\end{rem}

\ssec{The case of tori}

In this subsection we will take $G=T$ to be a torus. We will analyze some explicit combinatorial and
geometric objects attached to a geometric metaplectic data for $T$. 

\sssec{}  \label{sss:neutral case}

According to \cite[Sect. 4.2]{GLys}, to a geometric metaplectic data $\CG^T$ for $T$ one attaches 
a quadratic form $q$ on $\Lambda$ with values in $\sfe^{\times,\on{tors}}(-1)$. 

\medskip

According to \cite[Sect. 4.2.10]{GLys}, the space of geometric metaplectic data for which $q$ is trivial
is canonically isomorphic to the space of gerbes on $X$ with respect to the group
$$\Hom(\Lambda,\sfe^{\times,\on{tors}})\simeq \cT(\sfe)^{\on{tors}},$$
where $\cT$ is the Langlands dual torus of $T$, thought of as an algebraic group over $\sfe$. 

\sssec{}  \label{sss:sharp lattice}

Let 
$$b:\Lambda\otimes \Lambda\to \sfe^{\times,\on{tors}}(-1)$$
denote the symmetric bilinear form associated with $q$.

\medskip

Let $\Lambda^\sharp\subset \Lambda$ denote the kernel of $b$. Let 
$$T^\sharp\to T$$
be an isogenous torus whose lattice of coweights equals $\Lambda^\sharp$. 
Let $T_H$ denote the Langlands dual torus of $T^\sharp$, thought of as an algebraic group over $\sfe$. 

\sssec{} \label{sss:epsilon}

The restriction of $q$ to $\Lambda^\sharp$ is a \emph{linear} map
$$\Lambda^\sharp\to \pm 1\subset \sfe^{\times,\on{tors}}(-1).$$

We can view it as an element of order $2$, denoted $\epsilon$, in $T_H$. 

\sssec{}

Let $\CG^{T^\sharp}$ be the geometric metaplectic data for $T^\sharp$ obtained from $\CG^T$ by pulling back along
$$\Gr_{T^\sharp,\Ran}\to \Gr_{T,\Ran}.$$

Note that since $T^\sharp$ is commutative, the factorization space $\Gr_{T^\sharp,\Ran}$ carries a group structure
over $\Ran$. Hence, we can talk about (factorization) gerbes on $\Gr_{T^\sharp,\Ran}$ equipped with a  
multiplicative structure. 

\medskip

The following is established in \cite[Proposition 4.3.2 and Sect. 4.5]{GLys}:

\begin{prop} \hfill  \label{p:mult gerbe torus}

\smallskip

\noindent{\em(a)} The factorization gerbe $\CG^{T^\sharp}$ on $\Gr_{T^\sharp,\Ran}$ carries a uniquely defined
multiplicative structure. 

\smallskip

\noindent{\em(b)} To $\CG^{T^\sharp}$ one can canonically attach a geometric metaplectic data $\CG^{T^\sharp}_0$ for $T^\sharp$
with a vanishing quadratic form, such that $\CG^{T^\sharp}$ and $\CG^{T^\sharp}_0$ are isomorphic as multiplicative structure
(without the factorization structure).

\end{prop}

\begin{rem}
The discrepancy between the factorization structures on $\CG^{T^\sharp}$ and $\CG^{T^\sharp}_0$ is controlled by
the element $\epsilon$ of \secref{sss:epsilon}, see \cite[Sects. 4.5 and 4.6]{GLys}. 
\end{rem}

\sssec{}  \label{sss:gerbe for torus}

Let $\CG^T$ and $\CG^{T^\sharp}_0$ be as in \propref{p:mult gerbe torus}. According to \secref{sss:neutral case}, to 
$\CG^{T^\sharp}_0$ we can canonically attach a $T_H(\sfe)^{\on{tors}}$-gerbe on $X$, denoted $\CG_{T_H}$. 

%

\ssec{The metaplectic dual datum}

Let $\CG^G$ be a geometric metaplectic data. Following \cite[Sect. 6.3]{GLys}, we will now attach to it a 
\emph{metaplectic Langlands dual datum}, which is triple $(H,\CG_H,\epsilon)$, as explained below.

\sssec{} \label{sss:from G to T gerbes}

Consider the diagram 
$$
\CD
\Gr_{B,\Ran}  @>>> \Gr_{G,\Ran} \\
@VVV  \\
\Gr_{T,\Ran}. 
\endCD
$$

It is easy to see that any (factorization) gerbe on $\Gr_{B,\Ran}$ comes as pullback from a uniquely defined 
(factorization) gerbe on $\Gr_{T,\Ran}$. Thus, restricting $\CG^G$ to $\Gr_{B,\Ran}$, we obtain a geometric metaplectic
data $\CG^T$ for $T$. 

\begin{rem}
Our convention here is different from one in \cite[Sect. 5]{GLys} by a certain sign gerbe, which will be irrelevant for 
the purposes of this work.
\end{rem}

\sssec{}    \label{sss:q alpha}

Consider the resulting quadratic form $q$ on $\Lambda$, see \secref{sss:neutral case}. One shows that $q$ is
Weyl group invariant and \emph{restricted} (see \cite[Sect. 3.2.2]{GLys} or \secref{sss:b'} for what this means).

\medskip

Let $T_H$ be the torus associated to $\CG^T$, see \secref{sss:sharp lattice}.
The first ingredient in the triple $(H,\CG_{Z_H},\epsilon)$, namely $H$,
is a reductive group over the field of coefficients $\sfe$ with maximal torus $T_H$. We will now specify its root datum.

\medskip

As was just mentioned, the weight lattice of $H$ equals $\Lambda^\sharp$; we will sometimes denote it also by 
$\Lambda_H$. In particular, we have an inclusion
$$\Lambda_H\subset \Lambda,$$
which is a rational isomorphism. 

\medskip

The set of roots (resp., positive roots, simple roots) of $H$ is in bijection with those of $G$. For a coroot $\alpha$ of $G$
let $q_\alpha\in \sfe^\times(-1)$ denote the element $q(\alpha)$. Let $\ell_\alpha$ denote the order of $q_\alpha$.

\begin{defn}  \label{d:non-deg q}
We will say that a geometric metaplectic datum is \emph{non-degenerate} of $\ell_\alpha\neq 1$ for all $\alpha$.
\end{defn}

If $\alpha$ is a simple root $\alpha_i$ for $i\in I$ we will simply write $q_i$ instead of $q_{\alpha_i}$. By $W$-invariance, the above
non-degeneracy condition is equivalent to the non-triviality of all $q_i$. 

\sssec{}  \label{sss:roots in dual}

Set
$$\alpha_H=\ell_\alpha\cdot \alpha\in \Lambda \text{ and } \check\alpha_H=\frac{\check\alpha}{\ell_\alpha}\in \BQ\underset{\BZ}\otimes 
\check\Lambda.$$

According to \cite[Sect. 6.1]{GLys}, we have:
$$\alpha_H\in \Lambda_H \text{ and } \check\alpha_H\in \check\Lambda_H,$$
and the quadruple
$$(\Lambda_H,\check\Lambda_H,\{\alpha_H\},\{\check\alpha_H\})$$
forms a root datum of a reductive group.  This is the sought-for group $H$. 

\sssec{}

The second component in the triple $(H,\CG_{Z_H},\epsilon)$ is $\CG_{Z_H}$. This is a gerbe on $X$ with respect to
$Z_H(\sfe)^{\on{tors}}$, where $Z_H$ is the center of $H$:

\medskip

One shows (see \cite[Sect. 6.2]{GLys}) that the $T_H(\sfe)^{\on{tors}}$-gerbe $\CG_{T_H}$ of \secref{sss:gerbe for torus} is induced from a
canonically defined $Z_H(\sfe)^{\on{tors}}$-gerbe $\CG_{Z_H}$ along the inclusion $Z_H\hookrightarrow T_H$. 

\sssec{}  \label{sss:trivialize gerbe at point}

In the rest of this paper we will choose a trivialization of the fiber $\CG_{Z_H,x}$ of $\CG_{Z_H}$ at the chosen point $x\in X$.
This choice is made in order to unburden the notation. 

\medskip

The trivialization of $\CG_{Z_H,x}$ induces a trivialization of the $T_H(\sfe)^{\on{tors}}$-gerbe $\CG_{T_H,x}:=\CG_{T_H}|_x$. 

\sssec{}

The last component in the triple $(H,\CG_{Z_H},\epsilon)$ is the element $\epsilon\in T_H$ from \secref{sss:epsilon}. 
One shows that this element actually belongs to $Z_H\subset T_H$. 

\medskip

That said, the above element $\epsilon$ will not play any role in the present work. 

\ssec{Metaplectic geometric Satake}  \label{ss:metapl geom Sat}

Recall that the (usual) geometric Satake, in its weak form, is a monoidal functor
$$\Rep(\cG)\to \Sph_x(G):=\Shv(\Gr_{G,x})^{\fL^+(G)_x}.$$

In this subsection we will recall, following \cite[Sect. 9]{GLys}, its metaplectic counterpart.

\sssec{}

We fix a point $x\in X$. We define the metaplectic spherical category, denoted $\Sph_{q,x}(G)$ to be 
$$\Shv_{\CG^G}(\Gr_{G,x})^{\fL^+(G)_x}.$$

We regard it as equipped with a monoidal structure, given by convolution.

\sssec{}

The category $\Shv_{\CG^G}(\Gr_{G,x})^{\fL^+(G)_x}$ carries the (perverse) t-structure, and the convolution
functor
$$\Sph_{q,x}(G)\otimes \Sph_{q,x}(G)\to \Sph_{q,x}(G)$$
is t-exact. 

\sssec{}

Let $\Rep(H)$ be the DG category of representations of $H$, defined, e.g., as the category of quasi-coherent sheaves
on the algebraic stack (over $\sfe$) $BH$.

\medskip

According to \cite[Sect. 9.2]{GLys}, there exists a canonically defined monoidal functor
$$\Sat_{q,G}:\Rep(H)\to \Sph_{q,x}(G)$$
(here we are using the trivialization of the gerbe $\CG_{Z_H,x}$, see \secref{sss:trivialize gerbe at point}). 

\medskip

In what follows we will discuss some of properties of the functor $\Sat_{q,G}$ that we will use in the future.

\sssec{}

Let us first take $G=T$ to be a torus. Consider the corresponding torus $T^\sharp$. Since the $\fL^+(T^\sharp)_x$-action
on $\Gr_{T^\sharp,x}$ is trivial, we obtain a canonically defined functor
\begin{equation} \label{e:lift equiv}
\Shv(\Gr_{T^\sharp,x})\to \Shv(\Gr_{T^\sharp,x})^{\fL^+(T^\sharp)_x}.
\end{equation}

The functor $\Sat_{q,T}$ is the composition
\begin{multline*} 
\Rep(T_H)\simeq \Shv(\Gr_{T^\sharp,x})\overset{\text{\eqref{e:lift equiv}}}\longrightarrow \Shv(\Gr_{T^\sharp,x})^{\fL^+(T^\sharp)_x} \simeq \\
\simeq \Shv_{\CG^{T^\sharp}}(\Gr_{T^\sharp,x})^{\fL^+(T^\sharp)_x} \to 
\Shv_{\CG^T}(\Gr_{T,x})^{\fL^+(T)_x}=:\Sph_{q,x}(T),
\end{multline*} 
where:

\begin{itemize}

\item The equivalence $\Shv(\Gr_{T^\sharp,x})^{\fL^+(T^\sharp)_x} \simeq \Shv_{\CG^{T^\sharp}}(\Gr_{T^\sharp,x})^{\fL^+(T^\sharp)_x}$
is given by the trivialization of $\CG_{T_H,x}$;

\item The functor $\Shv_{\CG^{T^\sharp}}(\Gr_{T^\sharp,x})^{\fL^+(T^\sharp)_x} \to 
\Shv_{\CG^T}(\Gr_{T,x})^{\fL^+(T)_x}$ is given by direct image along 
$\fL^+(T^\sharp)_x\backslash \Gr_{T^\sharp,x}\to \fL^+(T)_x\backslash\Gr_{T,x}$. 

\end{itemize}

\sssec{}  \label{sss:descr of weight spaces}

Let now $G$ be general. 

\medskip

Due to the trivialization of $\CG_{T_H,x}$, the gerbe $\CG^G$ is trivialized when restricted to the orbits
$$S^\gamma=\fL(N)_x\cdot t^\gamma \subset \Gr_{G,x}, \quad \gamma\in \Lambda^\sharp.$$

\medskip

Hence, for 
$\CF\in \Shv_{\CG^G}(\Gr_G)$ it makes sense to consider 
$$\Gamma(S^\gamma, \CF|_{S^\gamma})\in \Vect.$$

The compatibility of metaplectic geometric Satake with Jacquet functors (see \cite[Sect. 9.4.3]{GLys}) implies: 
\begin{equation} \label{e:weight component}
\Gamma(S^\gamma,\Sat_{q,G}(V))[-\langle \gamma,2\check\rho\rangle]\simeq (\Res^H_{T_H}(V))(\gamma), \quad V\in \Rep(H),
\end{equation}
where $(\gamma)$ means weight component $\gamma$. 

\sssec{}

The trivialization of $\CG^G|_{S^\gamma}$ in particular implies that the fiber of $\CG^G$ at $t^\gamma$ is trivialized. 
Hence, the object
$$\IC_{q,\ol\Gr^\gamma_G}\in \Sph_{q,x}(G)$$
is well-defined. 

\medskip

It follows from \eqref{e:weight component} that we have a \emph{canonical} identification 
\begin{equation} \label{e:descr h.w.}
\Sat_{q,G}(V^\gamma)\simeq \IC_{q,\ol\Gr^\gamma_G},
\end{equation}
where $V^\gamma$ is the irreducible representation $H$ of highest weight $\gamma$ 
\emph{with a trivialized highest weight line}. 

\medskip

In particular, the functor $\Sat_{q,G}$ is t-exact. 

\begin{rem}
The functor $\Sat_{q,G}$ is \emph{not} an equivalence, but it \emph{does} induce an equivalence of monoidal
abelian categories
$$(\Rep(H))^\heartsuit\to (\Sph_{q,x}(G))^\heartsuit.$$
\end{rem}

\begin{rem} \label{r:lowest weight line}
Note that the isomorphisms \eqref{e:weight component} and \eqref{e:descr h.w.} give rise to canonical trvializations of the
\emph{lowest weight lines} in each $V^\gamma$:
$$\sfe \simeq  \Gamma(S^{w_0(\gamma)},\IC_{q,\ol\Gr^\gamma_G})[\langle \gamma,2\check\rho\rangle]\simeq
V^\gamma(w_0(\gamma)).$$

This system of trivializations corresponds to a canonically defined representative 
$$\sw_0\in \on{Norm}_H(T_W)$$
of the longest element of the Weyl group $w_0\in W_H$. The element $\sw_0$ is characterized
by the property that it makes the diagrams
$$
\CD
\sfe  @>{\simeq}>>  V^\gamma(\gamma)   \\
@V{\on{id}}VV   @VV{\sw_0}V   \\
\sfe  @>{\simeq}>>  V^\gamma(w_0(\gamma)) 
\endCD
$$
commute. 

\end{rem} 

\ssec{Metaplectic geometric Satake and Verdier duality}  \label{ss:Satake and duality}

One of the crucial steps in the proof of our main theorem depends on a Verdier duality
manipulation. In order to do this we will need to study how metaplectic geometric Satake
interacts with Verdier duality, and this is the subject of the present subsection. 

\sssec{}   \label{sss:inversion}

We note that inversion on $\fL(G)_x$ (or, which is equivalent, the involution on $\on{Hecke}^{\on{loc}}_{G,\Ran}$
from \secref{sss:inv on Hecke}), defines an equivalence 
$$\on{inv}^G:\Sph_{q,x}(G) \to \Sph_{q^{-1},x}(G),$$
which \emph{reverses} the monoidal structures. 

\medskip

Consider the Verdier duality functor
$$\BD^{\on{Verdier}}:(\Sph_{q,x}(G)^c)^{\on{op}}\to \Sph_{q^{-1},x}(G)^c.$$

\medskip
 
It follows from the definitions that the composite 
$$\BD^{\on{Verdier}}\circ \on{inv}^G: (\Sph_{q,x}(G)^c)^{\on{op}}\to \Sph_{q,x}(G)^c$$
is the functor of \emph{monoidal dualization} on $\Sph_{q,x}(G)^c$.

\begin{rem}
Note that the functor $\BD^{\on{Verdier}}\circ \on{inv}^G$ sends $\IC_{q,\ol\Gr^\gamma_G}$ to
$\IC_{q^{-1},\ol\Gr^{-w_0(\gamma)}_G}$. Combined with \eqref{e:descr h.w.}, this implies a \emph{canonical}
identification
\begin{equation} \label{e:duality on RepH}
(V^\gamma)^*\simeq V^{-w_0(\gamma)}.
\end{equation} 
In particular, \eqref{e:duality on RepH} implies that each $V^\gamma$ has a canonically trivialized
lowest weight line. However, it is easy to see that this is the same trivialization as the one specified in 
Remark \ref{r:lowest weight line}. 
\end{rem} 

\sssec{}

It follows from the constructions that if $(H,\CG_{Z_H},\epsilon)$ is the metaplectic dual datum for $\CG^G$,
then the one corresponding to $(\CG^G)^{-1}$ is given by $(H,(\CG_{Z_H})^{-1},\epsilon)$.

\medskip

In particular, a trivialization of the gerbe $\CG_{Z_H,x}$ (see \secref{sss:trivialize gerbe at point})
induces a trivialization of the corresponding gerbe arising from $\CG^{-1}$. 

\medskip

In particular, we obtain a geometric Satake functor 
$$\Sat_{q^{-1},G}:\Rep(H)\to \Sph_{q^{-1},x}(G).$$

\sssec{}    \label{sss:Cartan inv}

We normalize the Cartan involution $\tau^H$ on a reductive group $H$ with chosen Cartan and Borel subgroups $T_H\subset B_H$
so that it acts as inversion on $T_H$ and swaps $B_H$ and $B^-_H$.
We have a commutative diagram
\begin{equation} \label{e:Cartan H and T}
\CD
T_H  @>>>  H  \\
@V{\tau^{T_H}}VV   @VV{\tau^H}V   \\
T_H  @>>>  H,
\endCD
\end{equation} 
where, according to the above conventions, $\tau^{T_H}$ is inversion on $T_H$. 

\medskip

We will denote by the same symbol $\tau^H$ the corresponding involution on $\Rep(H)$. We have the
corresponding commutative diagram
\begin{equation} \label{e:two taus}
\CD
\Rep(H) @>{\on{Res}^H_{T_H}}>> \Rep(T_H)  \\
@V{\tau^H}VV   @VV{\tau^{T_H}}V    \\
\Rep(H)  @>{\on{Res}^H_{T_H}}>> \Rep(T_H).
\endCD
\end{equation} 

\medskip

Note that we have a \emph{canonical} identification
\begin{equation} \label{e:cartan and dual}
\tau^H(V^\gamma)\simeq (V^\gamma)^*.
\end{equation} 

Indeed, both representations are irreducible and have trivialized \emph{lowest} weight lines. 

\sssec{}  \label{sss:Satake and inversion}

Combining \eqref{e:cartan and dual} with \eqref{e:descr h.w.}, we obtain that the following diagram of monoidal functors 
canonically commutes
\begin{equation} \label{e:metapl Satake and duality}
\CD
(\Rep(H)^c)^{\on{op}}  @>{(\Sat_{q,G})^{\on{op}}}>>  (\Sph_{q,x}(G)^c)^{\on{op}}  \\
@V{\tau^H\circ \BD^{\on{lin}}}VV    @VV{\BD^{\on{Verdier}}}V   \\
\Rep(H)^c @>{\Sat_{q,G}}>>  (\Sph_{q^{-1},x}(G))^c,
\endCD
\end{equation}
where $\BD^{\on{lin}}$ is the (usual) dualization functor  
$$(\Rep(H)^c)^{\on{op}}\to \Rep(H)^c.$$ 

\medskip

Juxtaposing \eqref{e:metapl Satake and duality} with \secref{sss:inversion}, we obtain 
the following commutative diagram of monoidal functors: 
\begin{equation} \label{e:Cartan and inv}
\CD
\Rep(H)  @>{\Sat_{q,G}}>> \Sph_{q,x}(G) \\
@V{\tau^H}VV  @VV{\on{inv}^G}V  \\
\Rep(H)  @>{\Sat_{q,G}}>> \Sph_{q^{-1},x}(G). 
\endCD
\end{equation}
 
\section{Factorization algebras and modules} \label{s:fact}

Our main theorem compares the twisted Whittaker category on the affine Grassmannian with the
category of \emph{factorization modules} over a cartain \emph{factorization algebra}. 

\medskip

In this section we will recall the definition of these objects in the context of factorization spaces over the
Ran space. 

\ssec{Factorization algebras}  

In this subsection we will recall the definition of factorization algebras (on factorization spaces over the Ran space). 

\sssec{} \label{sss:fact alg geom}

Let $Z_\Ran\to \Ran$ be a factorization space over $\Ran$, and let $\CG$ be a factorization gerbe on $Z_\Ran$.
By a \emph{factorization algebra} in $\Shv_{\CG}(Z_\Ran)$ we will mean an object
$\CA\in \Shv_{\CG}(Z_\Ran)$ equipped with a \emph{homotopy compatible} system 
of isomorphisms
$$\CA|_{Z_\Ran\underset{\Ran}\times (\Ran^J)_{\on{disj}}}\simeq
\CA^{\boxtimes J}|_{Z_\Ran^J\underset{\Ran^J}\times (\Ran^J)_{\on{disj}}},$$
where the two spaces are identified by \eqref{e:factorization space}.

\medskip

The expression ``homotopy compatible" can be formalized is several different (but equivalent) ways. Below we discuss
one of the possibilities (which is very close to one from \cite[Sect. Sect. 6]{Ras1}). 
We start with spelling out the details in the definition of the notion of \emph{factorization space}. 

\sssec{}  \label{sss:fact spc} 

First, consider the assignment
\begin{equation} \label{e:Ran disj}
J\rightsquigarrow (\Ran^J)_{\on{disj}}
\end{equation}
as a functor 
$$\on{fSet}^{\on{surj}}\to \on{PreStk},$$
where $\on{fSet}^{\on{surj}}$ is the category of finite non-empty sets and surjective maps. 

\medskip

The operation of disjoint union makes
$\on{fSet}^{\on{surj}}$ into a symmetric monoidal category. The functor \eqref{e:Ran disj} has a natural op-lax 
symmetric monoidal structure, which means that we have the natural maps
$$(\Ran^{J_1\sqcup J_2})_{\on{disj}}\to (\Ran^{J_1})_{\on{disj}}\times (\Ran^{J_2})_{\on{disj}},$$
etc. 

%
%

\medskip

A factorization space over $\Ran$ is an op-lax symmetric monoidal functor 
\begin{equation} \label{e:fact space again}
\on{fSet}^{\on{surj}}\to \on{PreStk}, \quad J \mapsto Z_J,
\end{equation} 
equipped with a natural transformation to the functor \eqref{e:Ran disj}, 
such that the following two requirements hold:

\begin{itemize} 

\item The map 
$$Z_J\to (\Ran^J)_{\on{disj}}\underset{\Ran}\times Z_*,$$
induced by the map $I\to *$, is an isomorphism.

\item The map
$$Z_J \to (\Ran^J)_{\on{disj}}\underset{\Ran^J}\times (Z_*)^J,$$
induced by the op-lax symmetric monoidal structure, is an isomorphism 

\end{itemize}

\medskip

The relation of this definition to the naive one in \secref{sss:fact space naive} is that
$$Z_*:=Z_\Ran,\,\, Z_I:=(\Ran^J)_{\on{disj}}\underset{\Ran^J}\times (Z_\Ran)^J.$$

\sssec{}  \label{sss:fact grb} 

Replacing the symmetric monoidal category $\on{PreStk}$ by that of $\on{PreStk+Grb}$ consisting 
pairs $(\CY,\CG)$, where $\CY$ is a prestack and $\CG$ is a gerbe on $\CY$, we obtain the notion of factorization
gerbe over a factorization space. 

\sssec{}  \label{sss:fact alg} 

Let 
\begin{equation} \label{e:fact spc+grb}
I\mapsto (Z_I,\CG_I)
\end{equation}
be a factorization gerbe on a factorization space.

\medskip

Composing with the (lax symmetric monoidal) functor
\begin{equation} \label{e:twisted sheaves}
\Shv:(\on{PreStk+Grb})^{\on{op}}\to \inftyCat, \quad (\CY,\CG)\mapsto \Shv_\CG(\CY),\quad
(\CY_0\overset{f}\to \CY_1)\mapsto f^!,
\end{equation}
we obtain a lax symmetric monoidal functor
\begin{equation} \label{e:factorization category}
(\on{fSet}^{\on{surj}})^{\on{op}}\to \inftyCat, \quad I\mapsto \Shv_{\CG_I}(Z_I).
\end{equation}

We can view \eqref{e:factorization category} as a Cartesian fibration 
\begin{equation} \label{e:Cart fibr}
\Shv_\CG(Z_{\on{fSet}^{\on{surj}}})\to \on{fSet}^{\on{surj}},
\end{equation}
where $\Shv_\CG(Z_{\on{fSet}^{\on{surj}}})$ is equipped with a symmetric monoidal structure 
and \eqref{e:Cart fibr} is a symmetric monoidal functor. 

\medskip

A factorization algebra in $\Shv_{\CG}(Z_\Ran)$ is by definition a symmetric monoidal section of
\eqref{e:Cart fibr}, which is Cartesian as a section of $\infty$-categories (i.e., sends arrows in 
$\on{fSet}^{\on{surj}}$ to arrows in $\Shv_\CG(Z_{\on{fSet}^{\on{surj}}})$ that are Cartesian with
respect to \eqref{e:Cart fibr}). 

\begin{rem}
Above we gave a definition of factorization algebras \emph{in} $\Shv_{\CG}(Z_\Ran)$, where $(Z_\Ran,\CG)$
has a factorization structure. However, for many purposes it is convenient to give a more general definition--that
of factorization algebra inside a general \emph{factorization category}. The latter will not appear explicitly in this work. 
\end{rem}

\sssec{Example}

Take $Z_\Ran=\Ran$, with its tautological structure of factorization space. Then 
$$\omega_\Ran\in \Shv(\Ran)$$
acquires a structure of factorization algebra.

\ssec{Functoriality of factorization algebras}

In this subsection we will study functoriality properties of factorization algebras with respect to maps
of factorization spaces. 

\sssec{}  \label{sss:fact funct geom pullback}

Let $f:Z^1_{\Ran}\to Z^2_{\Ran}$ be a map of factorization spaces. Let $\CG^2$ be a factorization gerbe on $Z^2_\Ran$,
and let $\CG^1$ be its pullback to $Z^1_\Ran$, equipped with its natural factorization structure.

\medskip

It is clear that the pullback functor
$$f^!:\Shv_{\CG^2}(Z^2_\Ran) \to \Shv_{\CG^1}(Z^1_\Ran)$$
induces a functor on the corresponding categories of factorization algebras
$$f^!:\on{FactAlg}(\Shv_{\CG^2}(Z^2_\Ran))\to \on{FactAlg}(\Shv_{\CG^1}(Z^1_\Ran)).$$

\sssec{}  \label{sss:dualizing as an example of FA}

A basic example of this situation is when $Z^1_\Ran=Z_\Ran$ is an arbitrary factorization space
and $Z^2_\Ran=\Ran$. Taking the factorization algebra $\omega_\Ran\in \Shv(\Ran)$, we obtain 
that 
$$\omega_{Z_\Ran}\in \Shv(Z_\Ran)$$
has a natural structure of factorization algebra. 

\sssec{}  \label{sss:fact funct geom pushforward}

Let $f:Z^1_{\Ran}\to Z^2_{\Ran}$ be as before, but let us assume that $f$ is \emph{ind-schematic}.
Then the pushforward functor 
$$f_*:\Shv_{\CG^1}(Z^1_\Ran) \to \Shv_{\CG^2}(Z^2_\Ran)$$
induces a functor on the corresponding categories of factorization algebras
$$f_*:\on{FactAlg}(\Shv_{\CG^1}(Z^1_\Ran))\to \on{FactAlg}(\Shv_{\CG^2}(Z^2_\Ran)).$$

\sssec{}   \label{sss:unit as an example of FA}

Let $Z^2_\Ran=Z_\Ran$ is an arbitrary factorization space
and $Z^1_\Ran=\Ran$. Let 
$$\on{unit}:\Ran\to Z_\Ran$$ be a section of the tautological projection; assume that 
it is schematic as a morphism of prestacks.

\medskip

We obtain that
$$\on{unit}_*(\omega_\Ran)\in \Shv(Z_\Ran)$$
has a natural structure of factorization algebra. 

\ssec{Factorization modules}  \label{ss:fact modules}

We now come to a definition crucial for this work: that of factorization module over a given 
factorization algebra. 

\sssec{}  \label{sss:fact mod geom}

Let $Z_\Ran$ be a facorization space over $\Ran$, and let $Z_{\Ran_x}\to \Ran_x$ be a factorization module
space. Let $\CG$ be a factorization gerbe on $Z_\Ran$. Assume being given a gerbe $\CG$ over
$Z_{\Ran_x}$, equipped with a factorization structure with respect to the gerbe $\CG$ on $Z_\Ran$.

\medskip

Let $\CA$ be a factorization algebra in $\Shv_\CG(Z_\Ran)$. By a \emph{factorization module}
module in $\Shv_\CG(Z_{\Ran_x})$ with respect to $\CA$ we will mean an object $\CF\in \Shv_\CG(Z_{\Ran_x})$
equipped with a homotopy compatible system of isomorphisms 
$$\CF|_{Z_{\Ran_x}\underset{\Ran_x}\times (\Ran^J\times \Ran_x)_{\on{disj}}}\simeq
(\CF\boxtimes \CA^{\boxtimes J})|_{(\Ran_x\times Z_\Ran^J)\underset{\Ran^J\times \Ran_x}\times ((\Ran^J\times \Ran_x)_{\on{disj}})},$$
where two spaces are identified by \eqref{e:factorization space mod}.

\medskip

Below we give one of the possible formulations of the expression ``homotopy coherence" in this context.

\sssec{}

Let $\on{fSet}_*^{\on{surj}}$ be the category of pointed finite sets and surjective maps. We will
view it as a module category over the monoidal category $\on{fSet}^{\on{surj}}$ under the
operation of disjoint union. 

\medskip

We consider the functor
\begin{equation} \label{e:Ran disj ptd}
\on{fSet}_*^{\on{surj}}\to \on{PreStk},\quad J\mapsto (\Ran^J_*)_{\on{disj}}:=(\Ran^J)_{\on{disj}}\underset{\Ran}\times \Ran_x,
\end{equation} 
where the map $(\Ran^J)_{\on{disj}}\to \Ran$ corresponds to the element $*\in J$.

\medskip

When we regard $\on{PreStk}$ as a module category over itself, the functor \eqref{e:Ran disj ptd} has a structure of op-lax 
compatibility with actions, with respect to the op-lax monoidal functor \eqref{e:Ran disj} and the above module structure on 
$\on{fSet}_*^{\on{surj}}$ over $\on{fSet}^{\on{surj}}$. 

\sssec{}

Given $Z_\Ran$, thought of as a functor \eqref{e:fact space again}, a factorization module space over $Z_\Ran$
is a functor
$$\on{fSet}_*^{\on{surj}}\to \on{PreStk},\quad J\mapsto \wt{Z}_J,$$
equipped with a functor of op-lax compatibility with actions, and a natural transformation to \eqref{e:Ran disj ptd},
such that the following requirements hold:

\begin{itemize}

\item The map 
$$\wt{Z}_J\to (\Ran^J_*)_{\on{disj}}\underset{\Ran_x} \times \wt{Z}_*,$$
induced by $J\to *$, is an isomorphism. 

\item The map
$$\wt{Z}_J \to (\Ran^J_*)_{\on{disj}}\underset{\Ran^{J-*}\times \Ran_x}\times (Z_*^{J-*}\times \wt{Z}_*),$$
induced by the structure of op-lax compatibility with actions, is an isomorphism.  

\end{itemize} 

The relation of this definition to the naive one is that
$$\wt{Z}_*:=Z_{\Ran_x}.$$

\sssec{}

Given a factorization gerbe on $Z_\Ran$, we define a factorization structure on a gerbe on $Z_{\Ran_x}$, 
following the recipe of \secref{sss:fact grb}. 

\sssec{}

Given a factorization algebra $\CA\in \Shv_{\CG}(Z_\Ran)$, we define the notion of factorization module for it
in $\Shv_{\CG}(Z_{\Ran_x})$, along the lines of \secref{sss:fact alg}: 

\medskip

Namely, composing with the functor \eqref{e:twisted sheaves}, from $Z_{\Ran_x}$ we create a functor
$$(\on{fSet}_*^{\on{surj}})^{\on{op}}\to \inftyCat, \quad J\mapsto \Shv_{\CG^G}(\wt{Z}_J),$$
which we turn into Cartesian fibration
\begin{equation} \label{e:Cart fibr ptd}
\Shv_\CG(Z_{\on{fSet}_*^{\on{surj}}})\to \on{fSet}^{\on{surj}}_*, 
\end{equation}
so that the category $\Shv_\CG(Z_{\on{fSet}_*^{\on{surj}}})$ is equipped with a monoidal action of 
$\Shv_\CG(Z_{\on{fSet}^{\on{surj}}})$, and the functor \eqref{e:Cart fibr ptd} is compatible with
the actions with respect to the (symmetric) monoidal functor \eqref{e:Cart fibr}. 

\medskip

When we view $\CA$ as a (symmetric) monoidal section of \eqref{e:Cart fibr}, a factorization module for $\CA$ in 
$\Shv_{\CG}(Z_{\Ran_x})$ is a Cartesian section of \eqref{e:Cart fibr ptd}, compatible with the actions. 

\sssec{}

We denote the category of factorization $\CA$-modules in $\Shv_{\CG}(Z_{\Ran_x})$ by
$\CA\on{-FactMod}(\Shv_{\CG}(Z_{\Ran_x}))$, or simply $\CA\on{-FactMod}$ when no confusion is
likely to occur. We let $\oblv_{\on{Fact}}$ denote the forgetful
functor  
$$\CA\on{-FactMod}\to \Shv_{\CG}(Z_{\Ran_x}).$$

\begin{rem}  \label{r:remove point}
Let $X':=X-x$, and let $\Ran'$ denote the Ran space of $X$. Note that in the definition of a factorization 
module space for a given factorization space $Z_\Ran$, only
$$Z_{\Ran'}:=\Ran'\underset{\Ran}\times Z_\Ran$$
plays a role. Indeed, for $J\in \on{fSet}^{\on{surj}}_*$ and $J'\in \on{fSet}^{\on{surj}}$, we have
$$(\Ran^{J'\sqcup J}_*)_{\on{disj}}\simeq \Ran'{}^{J'}\underset{\Ran^{J'}}\times (\Ran^{J'\sqcup J}_*)_{\on{disj}}.$$

The same remark applies to factorization gerbes and factorization algebras. 
\end{rem}

\sssec{}

The first (non-zero) example of a factorization module is the so-called vacuum module: take
$$Z_{\Ran_x}:=\Ran_x\underset{\Ran}\times Z_\Ran$$
and let $\CF\in \Shv_{\CG}(Z_{\Ran_x})$ be the pulllback of $\CA$ itself under the forgetful map
$$Z_{\Ran_x}\to \Ran.$$

This incarnates the principle that ``a commutative algebra is naturally a left module over itself". 

\sssec{Modules for the unit} \label{sss:omega modules}

We will now describe a particular (albeit tautological) example of construction of factorization modules. 
This construction will play an important role in the sequel, as it can \emph{generate} other constructions
using functoriality (see \secref{ss:fact modules funct} below).

\medskip

Let $Z$ be an arbitary prestack with a gerbe $\CG$ on it. We can regard 
$$\Ran_x\times Z,$$
equipped with its tautological projection to $\Ran_x$ as a factorization module space with respect to the 
factorization space equal to $\Ran$ itself. 

\medskip

The pullback $\CG|_{\Ran_x\times \CG}$ has a natural factorization structure with respect to the
(necessarily) trivial factorization gerbe on $\Ran$. 

\medskip

Then the pullback functor
$$\Shv_{\CG}(Z)\to \Shv_{\CG}(\Ran_x\times Z)$$
naturally lifts to a functor
$$\Shv_{\CG}(Z)\to \omega_\Ran\on{-FactMod}(\Shv_{\CG}(\Ran_x\times Z)),$$
where we regard $\omega_\Ran$ as a factorization algebra in $\Shv(\Ran)$, see \secref{sss:dualizing as an example of FA}. 

\sssec{}  \label{sss:fact mod expl}

Here is an example of a situation where we can describe the category of factorization modules explicitly.

\medskip

Let $Z_{\Ran'}$ be a factorization space over $\Ran'$ (see Remark \ref{r:remove point}),
equipped with a factorization gerbe $\CG'$. Let $Z_x$ be an arbitrary prestack with a gerbe $\CG_x$. 

\medskip

We define $Z_{\Ran_x}$ as follows: for an affine test scheme $S$ and $*\in \CI\subset \Hom(S,X)$,
denote $\CI':=\CI-*$, and set
$$S\underset{\Ran_x}\times Z_{\Ran_x}:=(S\underset{\CI',\Ran}\times Z_{\Ran'})\times Z_x.$$

\medskip

We have the projections 
$$Z_x \leftarrow Z_{\Ran_x}\to Z_{\Ran'},$$ 
and we define the gerbe $\CG$ on $Z_{\Ran_x}$ as the tensor product of the pullbacks of $\CG_x$ and 
$\CG'$, respectively. The gerbe $\CG$ acquires a natural structure of factorization with respect to $\CG'$
(see Remark \ref{r:remove point}). 

\medskip

Let $\CA'$ be a factorization algebra in $\Shv_{\CG'}(Z_{\Ran'})$. By unwinding the definitions, we obtain
that the functor of restriction along $Z_x\to Z_{\Ran_x}$ defines an equivalence
$$\CA'\on{-FactMod}(\Shv_{\CG'}(Z_{\Ran_x}))\to \Shv_{\CG_x}(Z_x)$$
is an equivalence.

\ssec{Functoriality properties of factorization modules}  \label{ss:fact modules funct}

We will now study the functoriality of the category of factorization modules under the change of
factorization (module) space. 

\sssec{}   \label{sss:fact funct mod geom}

Let
\begin{equation} \label{e:corr fact}
\CD
Z^{1,2}_\Ran @>{f}>> Z^1_\Ran \\
@V{g}VV  \\
Z^2_\Ran
\endCD
\end{equation}
be a \emph{correspondence} between factorization spaces, where the morphism $g$ is ind-schematic. Let $\CG^1$ and $\CG^2$
be factorization gerbes on $Z^1$ and $Z^2$, respectively, equipped with an isomorphism $\CG^1|_{Z^{1,2}}\simeq \CG^2|_{Z^{1,2}}$
as factorization gerbes. 

\medskip

Then according to Sects. \ref{sss:fact funct geom pullback} and  \ref{sss:fact funct geom pushforward}, given 
a factorization algebra $\CA^1\in \Shv_{\CG^1}(Z^1_\Ran)$, the object
$$\CA^2:=g_*\circ f^!(\CA^1)\in \Shv_{\CG^2}(Z^2_\Ran)$$
acquires a structure of factorization algebra in $\Shv_{\CG^2}(Z^2_\Ran)$. 

\medskip

Let now
\begin{equation} \label{e:corr fact mod}
\CD
Z^{1,2}_{\Ran_x} @>{f}>> Z^1_{\Ran_x} \\
@V{g}VV  \\
Z^2_{\Ran_x}
\endCD
\end{equation}
be a diagram of factorization module spaces for the factorization spaces appearing in \eqref{e:corr fact}. We obtain that
the functor
\begin{equation} \label{e:corr functor}
g_*\circ f^!:\Shv_{\CG^1}(Z^1_{\Ran_x})\to \Shv_{\CG^2}(Z^2_{\Ran_x})
\end{equation}
gives rise to a functor
\begin{equation} \label{e:corr functor fact} 
g_*\circ f^!:\CA^1\on{-FactMod}(\Shv_{\CG^1}(Z^1_{\Ran_x}))\to \CA^2\on{-FactMod}(\Shv_{\CG^2}(Z^2_{\Ran_x})).
\end{equation}

\sssec{}   \label{sss:fact funct mod geom bis}

The functors $g_*\circ f^!$ inherit the usual properties of functors defined by correspondences. For
example, if $g$ is ind-proper, then $g^!$ is the right adjoint of $g_*=:g_!$. Similarly, if $f$ is \'etale, then
$f^!=:f^*$ is the left adjoint of $f_*$.

\sssec{}  \label{sss:fact funct mod geom et}

In addition, the fact that the map \eqref{e:union map ptd disj} is \'etale has the following consequences: 

\medskip

Suppose that $f$ is \'etale and $g$ is proper, and suppose that 
$\CF\in \CA^2\on{-FactMod}(\Shv_{\CG^2}(Z^2_{\Ran_x}))$ is such that the (partially defined) left adjoint
$$f_!\circ g^*:\Shv_{\CG^2}(Z^2_{\Ran_x})\to \Shv_{\CG^1}(Z^1_{\Ran_x})$$
of \eqref{e:corr functor} is defined on $\oblv_{\on{Fact}}(\CF)$.  

\medskip

Then the (partially defined) left adjoint $f_!\circ g^*$ of \eqref{e:corr functor fact} is defined on $\CF$, and we have
$$\oblv_{\on{Fact}}\circ (f_!\circ g^*)\simeq (f_!\circ g^*)\circ \oblv_{\on{Fact}}.$$

\section{Configuration spaces}  \label{s:conf}

In the previous section we discussed factorization spaces and factorization algebras (and modules over them). 
However, factorization spaces over $\Ran$ are prestacks that are not even ind-schemes (such as Ran itself), 
and sheaves on them may be unwieldy. 

\medskip

In this section we will introduce another paradigm for factorization: the underlying geometry will be the 
(pointed) configuration space, which has the advantage of being a scheme (resp., ind-scheme). 

\medskip

We will see also that the affine Grassmannian for $\Gr_T$ has a closed subfunctor essentially
isomorphic to $\Conf$. This will allow us to transfer the information between the two contexts. 

\ssec{Configuration space as the spaces of colored divisors}

In this subsection we introduce the configuration space. 

\sssec{}  \label{sss:conf}

Let $\Conf$ be the scheme that classifies the data of $(\Lambda^{\on{neg}}-0)$-valued divisors on $X$, i.e., expressions of the form
\begin{equation} \label{e:point of conf vac}
D=\underset{k}\Sigma\, \lambda_k\cdot x_k,
\end{equation} 
where:

\begin{itemize}

\item The index $k$ runs over some finite set; 

\smallskip

\item The points $x_k\in X$ are pairwise distinct; 

\smallskip

\item All $\lambda_k$ are in $\Lambda^{\on{neg}}-0$. 

\end{itemize} 

\sssec{}

We have:
$$\Conf=
\underset{\lambda\in \Lambda^{\on{neg}}-0}\bigsqcup\, \Conf^\lambda,$$
where where $\lambda$ is the total degree (i.e., for a point \eqref{e:point of conf vac} its total degree is $\underset{k}\Sigma\, \lambda_k$). 

\medskip

Each $\Conf^\lambda$ is isomorphic to $X^\lambda$, where for 
$$\lambda=\underset{i}\Sigma\, n_i\cdot (-\alpha_i),\,\,\alpha_i \text{ are the simple coroots},\,\, n_i\in \BZ^{\geq 0}$$
we have
$$X^\lambda=\underset{i}\prod\, X^{(n_i)}.$$

\begin{rem}  \label{r:config as divisors}
Note that if $G$ is semi-simple and simply connected, then $\Conf$ can also be interpreted as the moduli
space of non-zero homomorphisms from the monoid $\cLambda^+$ to the scheme of effective divisors
$$\on{Div}^{\on{eff}}(X)\simeq \underset{n\geq 0}\sqcup\, X^{(n)}.$$
\end{rem}

\sssec{}

Let $$\overset{\circ}\Conf\subset \Conf$$
be the open subscheme corresponding to the condition that 
in \eqref{e:point of conf vac} every $\lambda_k$ is a negative simple coroot. 

\medskip

We have
$$\overset{\circ}\Conf=\underset{\lambda\in \Lambda^{\on{neg}}}\bigsqcup\, 
\overset{\circ}\Conf{}^\lambda,$$
where each 
$\overset{\circ}\Conf{}^\lambda$ is isomorphic to the open subscheme
$$\overset{\circ}{X}{}^\lambda \subset X^\lambda,$$
obtained by removing the diagonal divisor. 

\sssec{}

The scheme $\Conf$ has a natural structure of commutative semigroup: for a finite non-empty set $I$
we have the map 
\begin{equation} \label{e:add divisors}
\Conf^I\to \Conf,
\end{equation}
given by the addition of operation on $(\Lambda^{\on{neg}}-0)$-valued divisors. 
 
\sssec{}

We will denote by
$$(\Conf^I)_{\on{disj}}\subset \Conf^I$$
the open subscheme given by the following condition:

\medskip

The corresponding configurations $\Sigma\, \lambda^i_k\cdot x^i_k$ 
must have disjoint support, i.e., $x^i_k\neq x^{i'}_{k'}$ for all $k,k'$ 
for every pair of indices $i\neq i'$.

\medskip

Note that the map \eqref{e:add divisors}, restricted to $(\Conf^I)_{\on{disj}}$, is \'etale. 

\ssec{Configurations with a marked point}

In this subsection we introduce a version of $\Conf$, where at a marked point $x$, we allow the value
of our divisor to be any element of $\Lambda$. The resulting space $\Conf_{\infty\cdot x}$ will no longer 
be a scheme, but it will be an ind-scheme. 

\sssec{}

Fix a point $x\in X$. Let $\Conf_{\infty\cdot x}$ denote the ind-scheme classifying the data of $\Lambda$-colored
divisors on $X$ of the form
\begin{equation} \label{e:point of conf}
D=\lambda_x\cdot x+\underset{k}\Sigma\, \lambda_k\cdot x_k,
\end{equation} 
where:

\begin{itemize}

\item The index $k$ runs over some finite set; 

\smallskip

\item The points $x_k\in X$ are pair-wise distinct as well as distinct from $x$; 

\smallskip

\item $\lambda_k\in \Lambda^{\on{neg}}-0$ and $\lambda_x$ is an arbitrary element of $\Lambda$. 

\end{itemize} 

\sssec{}

One can explicitly write down $\Conf_{\infty\cdot x}$ as follows. It equals the colimit
\begin{equation} \label{e:config as colim}
\Conf_{\infty\cdot x}=
\underset{\mu\in \Lambda}{\underset{\longrightarrow}{\on{colim}}}\, 
\Conf_{\leq \mu\cdot x},
\end{equation} 
where $\Conf_{\leq \mu\cdot x}$ is the space of those configurations 
\eqref{e:point of conf} for which $\lambda_x\leq \mu$ in the standard order relation
(i.e., $\mu-\lambda_x\in \Lambda^{\on{pos}}$). 

\medskip

Each $\Conf_{\leq \mu\cdot x}$ is a scheme. Explicitly, it is the disjoint union 
$$\Conf_{\leq \mu\cdot x}=
\underset{\lambda\in \mu+\Lambda^{\on{neg}}}\bigsqcup\, (\Conf_{\leq \mu\cdot x})^\lambda,$$
where $\lambda$ is the total degree (i.e., for a point \eqref{e:point of conf} its total degree is $\lambda_x+\underset{k}\Sigma\, \lambda_k$). 

\medskip

For every fixed $\lambda$, we have
\begin{equation} \label{e:config as sym}
(\Conf_{\leq \mu\cdot x})^\lambda\simeq X^{\lambda-\mu}.
\end{equation} 

In terms of the identifications \eqref{e:config as sym}, the transition maps
$$X^{\lambda-\mu_1}\to X^{\lambda-\mu_2}$$
in forming the colimit \eqref{e:config as colim} are given by adding the divisor $(\mu_1-\mu_2)\cdot x$. 

\sssec{}  \label{sss:config inside marked}

Note that $\Conf$ can be also thought of as a closed subscheme of $\Conf_{\infty\cdot x}$.
Namely, it identifies with $\Conf_{x,\leq 0}$, with the connected component
$$(\Conf_{x,\leq 0})^0\simeq \on{pt}$$
removed. 

\sssec{}

The indscheme $\Conf_{\infty\cdot x}$ has a natural structure of module over $\Conf$. 

\medskip

For a finite set $I$, we denote by
$$(\Conf^I\times \Conf_{\infty\cdot x})_{\on{disj}}\subset 
\Conf^I\times \Conf_{\infty\cdot x}$$
the open ind-subscheme given by the following condition:

\medskip

The corresponding two configurations $\Sigma\, \lambda^i_k\cdot x^i_k$ and $\lambda_x\cdot x+ \Sigma\, \mu_j\cdot y_j$ 
must have disjoint support, i.e., $x^i_k\neq x^{i'}_{k'}$ for all $k,k'$ every pair of indices $i\neq i'$ \emph{and} $y_j\neq x^i_k\neq x$ for all 
$i,j,k$. 

\medskip

Note that the action map
\begin{equation} \label{e:add disj}
(\Conf^I\times \Conf_{\infty\cdot x})_{\on{disj}}\to 
\Conf_{\infty\cdot x}
\end{equation} 
is \'etale. 

\ssec{Sheaves on configuration spaces}

In this subsection we introduce factorization gerbes and the corresponding categories of sheaves 
on configuration spaces. 

\sssec{}

Note that for a gerbe $\CG^\Lambda$ on $\Conf$, one can talk about a \emph{factorization structure} on it.
This means a system of isomorphisms
$$\CG^\Lambda|_{(\Conf^I)_{\on{disj}}}\simeq (\CG^\Lambda)^{\boxtimes I}|_{(\Conf^I)_{\on{disj}}}$$
for every finite set $I$ that are compatible in the evident sense.

\medskip

Given a factorization gerbe $\CG^\Lambda$ on $\Conf$, one can talk about a factorization structure on a gerbe $\CG^\Lambda$ on 
$\Conf_{\infty\cdot x}$. By definition, this means a compatible system of isomorphisms
\begin{equation} \label{e:fact gerbe conf mod}
\CG^\Lambda|_{(\Conf^I\times \Conf_{\infty\cdot x})_{\on{disj}}}\simeq 
(\CG^\Lambda)^{\boxtimes I}\boxtimes \CG^\Lambda|_{(\Conf^I\times \Conf_{\infty\cdot x})_{\on{disj}}}.
\end{equation} 

\sssec{}

For the duration of this section we fix such a pair of factorization gerbes $\Lambda^\Lambda$ on $\Conf$ 
and $\Conf_{\infty\cdot x}$.

\medskip

We will consider the corresponding categories of sheaves
$$\Shv_{\CG^\Lambda}(\Conf) \text{ and } \Shv_{\CG^\Lambda}(\Conf_{\infty\cdot x}).$$

\sssec{}  \label{sss:locally compact}

Being a category of sheaves on a scheme (resp., ind-scheme), the category 
$\Shv_{\CG^\Lambda}(\Conf)$ (resp., $\Shv_{\CG^\Lambda}(\Conf_{\infty\cdot x})$)
is compactly generated.

\medskip

For any ind-scheme (ind-algebraic stack) $\CY$ with a gerbe $\CG$ on it, let
$$\Shv_\CG(\CY)^{\on{loc.c}}\subset \Shv_\CG(\CY)$$
denote the full subcategory consisting of objects $\CF$ that satisfy the following:

\begin{itemize}

\item The support of $\CF$ is a \emph{scheme} (resp., \emph{algebraic stack}), to be denoted $\CY'$;

\item The restriction of $\CF$ to every quasi-compact open subscheme (resp., substack) $\overset{\circ}\CY\subset \CY'$
belongs to $\Shv_\CG(\overset{\circ}\CY)$.

\end{itemize}

\medskip

It is clear that $$\Shv_\CG(\CY)^c\subset \Shv_\CG(\CY)^{\on{loc.c}}.$$

Another feature of this subcategory is that we have a well-defined Verdier duality equivalence
$$\BD^{\on{Verdier}}:(\Shv_\CG(\CY)^{\on{loc.c}})^{\on{op}}\to \Shv_{\CG^{-1}}(\CY)^{\on{loc.c}}.$$

\medskip

Consider the corresponding full subcategory
$$\Shv_{\CG^\Lambda}(\Conf)^{\on{loc.c}}\subset \Shv_{\CG^\Lambda}(\Conf).$$
Note that it consists objects that are compact when restricted to every connected component 
$\Conf^\lambda$ of $\Conf$. Similarly, consider the full subcategory
$$\Shv_{\CG^\Lambda}(\Conf_{\infty\cdot x})^{\on{loc.c}}\subset \Shv_{\CG^\Lambda}(\Conf_{\infty\cdot x}).$$

\medskip

Verdier duality defines equivalences
$$\BD^{\on{Verdier}}:(\Shv_{\CG^\Lambda}(\Conf)^{\on{loc.c}})^{\on{op}}\to \Shv_{(\CG^\Lambda)^{-1}}(\Conf)^{\on{loc.c}}$$
and 
$$\BD^{\on{Verdier}}:(\Shv_{\CG^\Lambda}(\Conf_{\infty\cdot x})^{\on{loc.c}})^{\on{op}}\to \Shv_{(\CG^\Lambda)^{-1}}(\Conf_{\infty\cdot x})^{\on{loc.c}}.$$

\ssec{Translation action on colored divisors}    \label{ss:Hecke lattice}
 
In this subsection we introduce a piece of structure, crucial for the rest of the paper: on action
of a sublattice on $\Conf_{\infty\cdot x}$ by adding divisors supported at the point $x$. 

\sssec{}

Let $\Lambda^\sharp\subset \Lambda$ be a sublattice. We will consider $\Lambda^\sharp$ as a discrete scheme
acting on $\Conf_{\infty\cdot x}$ by adding the corresponding divisor at $x$:
$$\gamma,D\mapsto \on{Tr}^\gamma(D):=D+\gamma\cdot x.$$

\medskip

We will make the following assumption: the gerbe $\CG^\Lambda$ on $\Conf_{\infty\cdot x}$ is \emph{equivariant} with respect
to this action, in a way compatible with the factorization structure with respect to the given gerbe $\CG^\Lambda$ on $\Conf$.

\medskip

This means that we are given a compatible system of identifications
\begin{equation}  \label{e:tr gamma}
\on{Tr}^\gamma(\CG^\Lambda)\simeq \CG^\Lambda
\end{equation} 
that make the following diagrams commute:
$$
\CD
\on{Tr}^\gamma(\CG^\Lambda)|_{(\Conf^I\times \Conf_{\infty\cdot x})_{\on{disj}}}  @>>>  
(\CG^\Lambda)^{\boxtimes I}\boxtimes \on{Tr}^\gamma(\CG^\Lambda)|_{(\Conf^I\times \Conf_{\infty\cdot x})_{\on{disj}}}  \\
@VVV   @VVV  \\
\CG^\Lambda|_{(\Conf^I\times \Conf_{\infty\cdot x})_{\on{disj}}}  @>>>  
(\CG^\Lambda)^{\boxtimes I}\boxtimes \CG^\Lambda|_{(\Conf^I\times \Conf_{\infty\cdot x})_{\on{disj}}},
\endCD
$$
where the top horizontal isomorphism is induced by \eqref{e:fact gerbe conf mod} via the commutative diagram
$$
\CD
(\Conf^I\times \Conf_{\infty\cdot x})_{\on{disj}} @>{\on{id}\times \on{Tr}^\gamma}>> (\Conf^I\times \Conf_{\infty\cdot x})_{\on{disj}}  \\
@V{\text{\eqref{e:add disj}}}VV   @VV{\text{\eqref{e:add disj}}}V  \\
\Conf_{\infty\cdot x}   @>{\on{Tr}^\gamma}>>   \Conf_{\infty\cdot x}. 
\endCD
$$

\sssec{}   \label{sss:translation functors}

The maps \eqref{e:tr gamma} induce functors
$$\on{Tr}^\gamma: \Shv_{\CG^\Lambda}(\Conf_{\infty\cdot x})\to \Shv_{\CG^\Lambda}(\Conf_{\infty\cdot x}).$$

We will regard this collection of functors as an action of the (symmetric) monoidal category
$\Rep(T_H)$ on $\Shv_{\CG^\Lambda}(\Conf_{\infty\cdot x})$, where $T_H$ is the torus whose lattice of characters is $\Lambda^\sharp$.

\medskip

By definition, for $\gamma\in \Lambda^\sharp$, the object $\sfe^\gamma\in \Rep(T_H)$ acts as $\on{Tr}^\gamma$. 

\ssec{Isogenies}  \label{ss:isogenies}

In this subsection we will study the behavior of the category $\Shv_{\CG^\Lambda}(\Conf_{\infty\cdot x})$ under the
change of the lattice $\Lambda$ by an isogenous one. This will be handy in the future, as it will be convenient for us
to replace the given group $G$ by another one in the same isogeny class. 

\sssec{}  \label{sss:isogenies conf}

Let is be given a short exact sequence of lattices 
\begin{equation} \label{e:isog lattice}
0\to \Lambda\to \wt\Lambda\to \Lambda_0\to 0.
\end{equation} 
We let $\wt\Lambda^{\on{neg}}\subset \wt\Lambda$ be the image of $\Lambda^{\on{neg}}$
under the above map.

\medskip

Denote by $\wt\Conf_{\infty\cdot x}$ the corresponding configuration space with a marked point. 
The map $\Lambda\to \wt\Lambda$ defines a closed embedding
\begin{equation} \label{e:isog conf}
i:\Conf_{\infty\cdot x}\to \wt\Conf_{\infty\cdot x}.
\end{equation} 

\sssec{}

Let $\CG^{\wt\Lambda}$ be a gerbe on $\wt\Conf_{\infty\cdot x}$ equipped with a factorization structure 
with respect to the given factorization gerbe $\CG^\Lambda$ on $\Conf$. 

\medskip

Let us assume being given an identification
$$\CG^{\wt\Lambda}|_{\Conf_{\infty\cdot x}}\simeq \CG^\Lambda$$
as gerbes on $\Conf_{\infty\cdot x}$ equipped with a factorization structure 
with respect to the factorization gerbe $\CG^\Lambda$ on $\Conf$. 

\medskip

Then the map $i$ of \eqref{e:isog conf} gives rise to a functor
\begin{equation} \label{e:dir image isog}
i_*:\Shv_{\CG^\Lambda}(\Conf_{\infty\cdot x})\to \Shv_{\CG^{\wt\Lambda}}(\wt\Conf_{\infty\cdot x}).
\end{equation} 

\sssec{}  

Let us now be in the situation of \secref{ss:Hecke lattice} for $\Lambda$ and $\wt\Lambda$, and assume that that we
have a commutative diagram
\begin{equation} \label{e:Lambda sharp lattices}
\CD
\wt\Lambda^\sharp  @>>>  \wt\Lambda \\
@AAA   @AAA   \\
\Lambda^\sharp @>>> \Lambda. 
\endCD
\end{equation}

\medskip

We obtain an action of $\Rep(T_H)$ on $\Shv_{\CG^\Lambda}(\Conf_{\infty\cdot x})$ and an action of 
$\Rep(T_{\wt{H}})$ on $\Shv_{\CG^{\wt\Lambda}}(\wt\Conf_{\infty\cdot x})$, which are intertwined 
by the functor $i_*$ of \eqref{e:dir image isog} and the restriction functor
$$\Rep(T_H)\to \Rep(T_{\wt{H}}).$$

\medskip

In particular, we obtain a functor 
\begin{equation} \label{e:enlarge lattice}
\Rep(T_{\wt{H}})\underset{\Rep(T_H)}\otimes \Shv_{\CG^\Lambda}(\Conf_{\infty\cdot x})\to
\Shv_{\CG^{\wt\Lambda}}(\wt\Conf_{\infty\cdot x}).
\end{equation}

\sssec{}  \label{sss:assump isogen lattice}

Assume now that \eqref{e:Lambda sharp lattices} is a \emph{push-out} diagram.

\medskip

The following results easily from the definitions:

\begin{lem}\label{l:enlarge lattice}
Under the above assumptions, the functor \eqref{e:enlarge lattice} is an equivalence.
\end{lem}

\sssec{}

Note that the assumption in \secref{sss:assump isogen lattice} implies that we have 
a short exact sequence of lattices
\begin{equation} \label{e:SES sharp}
0\to \Lambda^\sharp\to \wt\Lambda^\sharp \to \Lambda_0\to 0.
\end{equation}

A choice of a splitting of \eqref{e:SES sharp} defines an equivalence 
$$\Rep(T_H)\otimes \Rep(T_0)\simeq \Rep(T_{\wt{H}}),$$
as $\Rep(T_H)$-module categories, where $T_0$ is a torus with weight lattice $\Lambda_0$. 

\medskip

Combining with \lemref{l:enlarge lattice}, we obtain:

\begin{cor}\label{c:enlarge lattice}
A choice of a splitting of \eqref{e:SES sharp} defines an equivalence
$$\Shv_{\CG^\Lambda}(\Conf_{\infty\cdot x})\otimes \Rep(T_0)\to 
\Shv_{\CG^{\wt\Lambda}}(\wt\Conf_{\infty\cdot x}).$$
\end{cor} 

\ssec{Configuration space via the affine Grassmannian}     \label{ss:Conf Gr}

Consider the Ran Grassmannian $\Gr^{\omega^\rho}_{T,\Ran}$ of $T$. We will now describe certain closed subfunctors
$$(\Gr^{\omega^\rho}_{T,\Ran})^{\on{neg}}\subset (\Gr^{\omega^\rho}_{T,\Ran})^{\on{non-pos}}\subset \Gr^{\omega^\rho}_{T,\Ran}.$$

It will turn out that the prestack $(\Gr^{\omega^\rho}_{T,\Ran})^{\on{neg}}$ is essentially equivalent to
the configuration space $\Conf$. 

\sssec{}  \label{sss:neg part Gr}

Consider the simply connected cover $G_{\on{sc}}$ of (the derived group of) $G$. Let $T_{\on{sc}}$ denote the 
Cartan group of $G_{\on{sc}}$. Note that the map 
\begin{equation} \label{e:Gr sc to T}
\Gr^{\omega^\rho}_{T_{\on{sc}},\Ran}\to \Gr^{\omega^\rho}_{T,\Ran}
\end{equation}
is a closed embedding. 

\medskip

For the definition of $(\Gr^{\omega^\rho}_{T,\Ran})^{\on{neg}}$ and $(\Gr^{\omega^\rho}_{T,\Ran})^{\on{non-pos}}$ we stipulate that the equal 
equals the images of 
$$(\Gr^{\omega^\rho}_{T_{\on{sc}},\Ran})^{\on{neg}}\subset (\Gr^{\omega^\rho}_{T_{\on{sc}},\Ran})^{\on{non-pos}}\subset
\Gr^{\omega^\rho}_{T_{\on{sc}},\Ran}$$ along \eqref{e:Gr sc to T}, so for the definition of $(\Gr^{\omega^\rho}_{T,\Ran})^{\on{neg}}$ and 
$(\Gr^{\omega^\rho}_{T,\Ran})^{\on{non-pos}}$
we will assume that $G=G_{\on{sc}}$.  

\sssec{}

An $S$-point $(J,\CP_T,\alpha)$ of $\Gr^{\omega^\rho}_{T,\Ran}$ belongs to $(\Gr^{\omega^\rho}_{T,\Ran})^{\on{non-pos}}$ if the following 
condition hold:

\begin{itemize}

\item{\it Regularity}: for every \emph{dominant} weight $\check\lambda\in \cLambda^+$, the meromorphic map of line bundles on $S\times X$
(resp., $\cD_J$) 
$$\clambda(\CP_T)\to \clambda(\CP^0_T),$$ 
induced by $\alpha$, is regular.

\end{itemize} 

An $S$-point $(\CJ,\CP_T,\alpha)$ as above $(\Gr^{\omega^\rho}_{T,\Ran})^{\on{neg}}$ if moreover the following holds:

\begin{itemize}

\item{\it Non-redundancy}: 
for every point $s\in S$ and every element $j\in \CJ$ there exists at least one $\check\lambda\in \cLambda^+$, for which 
the above map of line bundles has a zero at the point of $X$ corresponding to $s \to S\overset{j}\to X$.

\end{itemize} 

Note that $(\Gr^{\omega^\rho}_{T,\Ran})^{\on{neg}}$ and $(\Gr^{\omega^\rho}_{T,\Ran})^{\on{non-pos}}$ have a natural structure of 
factorization spaces over $\Ran$. 

\sssec{}

Evaluation on fundamental weights defines a map of prestacks
\begin{equation} \label{e:Ran to config}
(\Gr^{\omega^\rho}_{T,\Ran})^{\on{neg}}\to \Conf
\end{equation} 
(see Remark \ref{r:config as divisors}). 

\medskip

The following is obtained as \cite[Lemma 8.1.4]{Ga4}:

\begin{lem}  \label{l:Ran to config}
The map \eqref{e:Ran to config} induces an isomorphism \emph{of the sheafifications} in the topology generated by
finite surjective maps. 
\end{lem} 

\sssec{}   \label{sss:gerbe Lambda}

As a corollary, we obtain that the map \eqref{e:Ran to config} induces an isomorphisms on spaces of gerbes.
In particular, we obtain that for a geometric metaplectic data for $T$, the factorization gerbe
$\CG^T|_{(\Gr^{\omega^\rho}_{T,\Ran})^{\on{neg}}}$ on $(\Gr^{\omega^\rho}_{T,\Ran})^{\on{neg}}$ is the pullback of a uniquely defined
factorization gerbe, deboted $\CG^\Lambda$ on $\Conf$.

\medskip

Furthermore, \lemref{l:Ran to config} implies that pullback defines an equivalence
\begin{equation}  \label{e:Ran to config sheaves}
\Shv_{\CG^\Lambda}(\Conf)\to \Shv_{\CG^T}((\Gr^{\omega^\rho}_{T,\Ran})^{\on{neg}}).
\end{equation} 

\sssec{}    \label{sss:neg part Gr ptd}

We define the closed subfunctors 
$$(\Gr^{\omega^\rho}_{T,\Ran_x})^{\on{neg}}_{\infty\cdot x}\subset (\Gr^{\omega^\rho}_{T,\Ran_x})^{\on{non-pos}}_{\infty\cdot x}\subset \Gr^{\omega^\rho}_{T,\Ran_x}$$
as follows: 

\medskip

A point $(J,\CP_T,\alpha)$ of $\Gr^{\omega^\rho}_{T,\Ran_x}$ belongs to $(\Gr^{\omega^\rho}_{T,\Ran_x})^{\on{neg}}_{\infty\cdot x}$ 
if there exists another $T$-bundle
$\CP'_T$ on $S\times X$ and an isomorphism $\CP_T|_{S\times (X-x)}\simeq \CP'_T|_{S\times (X-x)}$, such that the resulting
point $(J,\CP'_T,\alpha')$ of $\Gr^{\omega^\rho}_{T,\Ran_x}$ belongs to
$$\Ran_x\underset{\Ran}\times (\Gr^{\omega^\rho}_{T,\Ran})^{\on{neg}}\subset \Ran_x\underset{\Ran}\times \Gr^{\omega^\rho}_{T,\Ran}.$$

Replacing $(\Gr^{\omega^\rho}_{T,\Ran})^{\on{neg}}$ by $(\Gr^{\omega^\rho}_{T,\Ran})^{\on{non-pos}}$ we obtain the definition of 
$(\Gr^{\omega^\rho}_{T,\Ran_x})^{\on{non-pos}}_{\infty\cdot x}$. 

\sssec{}

As in \eqref{e:Ran to config} we have a canonically defined map
\begin{equation}  \label{e:Ran to config marked}
(\Gr^{\omega^\rho}_{T,\Ran_x})^{\on{neg}}_{\infty\cdot x}\to   \Conf_{\infty\cdot x},
\end{equation} 
and a counterpart of \lemref{l:Ran to config} holds.

\medskip

Hence, stating with a geometric metaplectic data for $T$, we obtain that the corresponding gerbe
$$\CG^T|_{(\Gr^{\omega^\rho}_{T,\Ran_x})^{\on{neg}}_{\infty\cdot x}},$$
viewed as a equipped with a factorization structure with respect to $\CG^T|_{(\Gr^{\omega^\rho}_{T,\Ran})^{\on{neg}}}$, 
is the pullback of a uniquely defined gerbe $\CG^\Lambda$ on $\Conf_{\infty\cdot x}$ equipped 
with a factorization structure with respect to the facorization gerbe $\CG^\Lambda$ on $\Conf$. 

\medskip

Furthermore, pullback with respect to \eqref{e:Ran to config marked} defines an equivalence
\begin{equation}  \label{e:Ran to config sheaves ptd}
\Shv_{\CG^\Lambda}(\Conf_{\infty\cdot x})\to \Shv_{\CG^T}((\Gr^{\omega^\rho}_{T,\Ran_x})^{\on{neg}}_{\infty\cdot x}).
\end{equation} 

\section{Factorization algebras anf modules on configuration spaces}   \label{s:fact conf}

In this section we will finally define factorization algebras and modules over them on the configuration spaces.

\medskip

We will see that they model a certain subcategory of factorization algebras (resp., modules over them) on $\Gr_T$.

\ssec{Factorization algebras on configuration spaces}  \label{ss:fact on config}

In this subsection we define the notion of factorization algebra on $\Conf$. 

\sssec{}

Let $\CG^\Lambda$ be a factorization gerbe on $\Conf$. A factorization algebra in $\Shv_{\CG^\Lambda}(\Conf)$
is an object $\CA\in \Shv_{\CG^\Lambda}(\Conf)$ equipped with a \emph{homotopy-compatible} system of identifications
\begin{equation} \label{e:fact conf}
\CA|_{(\Conf^I)_{\on{disj}}}\simeq (\CA)^{\boxtimes I}|_{(\Conf^I)_{\on{disj}}}
\end{equation} 
for finite non-empty sets $I$.

\medskip

Below we explain one of the possible ways to formalize the phrase ``homotopy-coherent" in this context.
We will follow the the same idea as in Sects. \ref{sss:fact spc}-\ref{sss:fact alg}. 

\sssec{}

The assignment
$$I\mapsto (\Conf^I)_{\on{disj}}$$
has a structure of op-lax symmetric monoidal functor
$$\on{fSet}^{\on{surj}}\to \Sch.$$

A factorization gerbe on $\Conf$ is a lift of the above functor to a functor with values 
in the symmetric monoidal category 
$$\on{Sch+Grb}.$$

\medskip

Composing with \eqref{e:twisted sheaves}, we obtain that the assignment
$$I\mapsto \Shv_{(\CG^\Lambda)^{\boxtimes I}}((\Conf^I)_{\on{disj}})$$
has a structure of lax monoidal functor
$$(\on{fSet}^{\on{surj}})^{\on{op}}\to \inftyCat.$$

\medskip

We interpret this functor as a Cartesian fibration
\begin{equation} \label{e:Cart fibr conf}
\Shv_\CG((\CG^\Lambda)^{\on{fSet}^{\on{surj}}})\to \on{fSet}^{\on{surj}}
\end{equation} 
of symmetric monoidal categories.

\medskip

A factorization algebra in $\Shv_{\CG^\Lambda}(\Conf)$ is a symmetric monoidal Cartesian 
section of \eqref{e:Cart fibr conf}. 

\sssec{}

Proceeding as in \secref{ss:fact modules}, given a factorization algebra $\CA\in \Shv_{\CG^\Lambda}(\Conf)$,
we define the categories of factorization modules with respect to it in $\Shv_{\CG^\Lambda}(\Conf_{\infty\cdot x})$.
We denote this category
$$\CA\on{-FactMod}(\Shv_{\CG^\Lambda}(\Conf_{\infty\cdot x})),$$
or simply $\CA\on{-FactMod}$ if no confusion is likely to occur. 

\medskip

We have a tautological conservative forgetful functor
$$\oblv_{\on{Fact}}: \CA\on{-FactMod}\to \Shv_{\CG^\Lambda}(\Conf_{\infty\cdot x}).$$

\ssec{Change of lattice and isogenies}

In this subsection we will remark that the material of Sects. \ref{ss:Hecke lattice} and \ref{ss:isogenies}
carries over to categories of factorization modules. 

\sssec{} \label{sss:action of lattice 1}

Let us be in the situation of \secref{ss:Hecke lattice}. As in \secref{sss:fact funct mod geom}, it 
follows that the action of
$\Rep(T_H)$ on $\Shv_{\CG^\Lambda}(\Conf_{\infty\cdot x})$ gives rise to one on $\CA\on{-FactMod}$.

\medskip

For $\gamma\in \Lambda^\sharp$, we will denote by the same symbol $\on{Tr}^\gamma$ the corresponding
translation endo-functor of $\CA\on{-FactMod}$.

\sssec{}  \label{sss:isogenies conf fact} 

Let us now be in the situation of \secref{ss:isogenies}. As in \secref{sss:fact funct mod geom}, it follows that the functor
$i_*$ of \secref{e:dir image isog} induces a functor
$$i_*:\CA\on{-FactMod}(\Shv_{\CG^\Lambda}(\Conf_{\infty\cdot x}))\to
\CA\on{-FactMod}(\Shv_{\CG^{\wt\Lambda}}(\wt\Conf_{\infty\cdot x}))$$
that intertwines the actions of $\Rep(T_H)$ and $\Rep(T_{\wt{H}})$, respectively. 

\medskip

In particular, we obtain a functor
\begin{equation} \label{e:enlarge lattice fact}
\Rep(T_{\wt{H}})\underset{\Rep(T_H)}\otimes \CA\on{-FactMod}(\Shv_{\CG^\Lambda}(\Conf_{\infty\cdot x}))\to
\CA\on{-FactMod}(\Shv_{\CG^{\wt\Lambda}}(\wt\Conf_{\infty\cdot x})). 
\end{equation}

We claim:

\begin{prop} \label{p:enlarge lattice fact}
Under the assumption of \secref{sss:assump isogen lattice}, the functor \eqref{e:enlarge lattice fact} 
is an equivalence.
\end{prop}

\begin{proof}

It follows as in \secref{sss:fact funct mod geom} that the functor \secref{e:enlarge lattice} and its right adjoint 
induce an adjoint pair of functors
$$\Rep(T_{\wt{H}})\underset{\Rep(T_H)}\otimes \CA\on{-FactMod}(\Shv_{\CG^\Lambda}(\Conf_{\infty\cdot x}))\rightleftarrows
\CA\on{-FactMod}(\Shv_{\CG^{\wt\Lambda}}(\wt\Conf_{\infty\cdot x})).$$

We wish to show that, under the assumption of \secref{sss:assump isogen lattice}, these functors are mutually
inverse. I.e., we need to show that the unit and the counit of this adjunction are isomorphisms. For the latter,
it is sufficient to show that the natural transformations become isomorphisms after applying the (conservative) forgeful
functors
$$\oblv_{\on{Fact}}:
\Rep(T_{\wt{H}})\underset{\Rep(T_H)}\otimes \CA\on{-FactMod}(\Shv_{\CG^\Lambda}(\Conf_{\infty\cdot x}))\to
\Rep(T_{\wt{H}})\underset{\Rep(T_H)}\otimes \Shv_{\CG^\Lambda}(\Conf_{\infty\cdot x})$$
and
$$\oblv_{\on{Fact}}:\CA\on{-FactMod}(\Shv_{\CG^{\wt\Lambda}}(\wt\Conf_{\infty\cdot x}))\to
\Shv_{\CG^{\wt\Lambda}}(\wt\Conf_{\infty\cdot x}),$$
respectively.

\medskip

Now fact that the resulting natural transformations are isomorphisms follows from \lemref{l:enlarge lattice}. 

\end{proof} 

\begin{cor}\label{c:enlarge lattice fact}
A choice of a splitting of \eqref{e:SES sharp} defines an equivalence
$$\CA\on{-FactMod}(\Shv_{\CG^\Lambda}(\Conf_{\infty\cdot x}))\otimes \Rep(T_0)\to 
\CA\on{-FactMod}(\Shv_{\CG^{\wt\Lambda}}(\wt\Conf_{\infty\cdot x})).$$
\end{cor} 

\ssec{Structure of the category of factorization modules on the configuration space}  \label{ss:mod over conf}

In this subsection we fix a factorization algebra $\CA\in \Shv_{\CG^\Lambda}(\Conf)$. 

\sssec{}

For $\mu\in \Lambda$ we let $\CA\on{-FactMod}_{\leq \mu}$ be the category of factorization 
$\CA$-modules in the category
$\Shv_{\CG^\Lambda}(\Conf_{\leq \mu\cdot x})$, or which is the same,
the preimage of
$$\Shv_{\CG^\Lambda}(\Conf_{\leq \mu\cdot x})\subset 
\Shv_{\CG^\Lambda}(\Conf_{\infty\cdot x})$$
under the functor $\oblv_{\on{Fact}}$.

\medskip

Let $\iota_{\mu}$ denote the closed embedding 
$$\Conf_{\leq \mu\cdot x}\hookrightarrow \Conf.$$

\medskip

As in \secref{sss:fact funct mod geom bis}, the adjoint pair
$$(\iota_{\mu})_!:\Shv_{\CG^\Lambda}(\Conf_{\leq \mu\cdot x})\rightleftarrows
\Shv_{\CG^\Lambda}(\Conf_{\infty\cdot x}):(\iota_{\mu})^!$$
induces a pair of adjoint functors
$$(\iota_{\mu})_!:\CA\on{-FactMod}_{\leq \mu}\rightleftarrows \CA\on{-FactMod}:(\iota_{\mu})^!$$
both of which commute with the forgetful functor $\oblv_{\on{Fact}}$. 

\medskip

Since the unit of the adjunction
$$\on{Id}\to (\iota_{\mu})^!\circ (\iota_{\mu})_!$$
is an isomorphism on $\Shv_{\CG^\Lambda}(\Conf_{\leq \mu\cdot x})$, the conservativity of 
$\oblv_{\on{Fact}}$ implies that it is also an isomorphism on $\CA\on{-FactMod}_{\leq \mu}$.

\medskip

Hence, the functor
\begin{equation} \label{e:iota on mod}
(\iota_{\mu})_!:\CA\on{-FactMod}_{\leq \mu}\to \CA\on{-FactMod}
\end{equation} 
is fully faithful. 

\medskip

We will often identify $\CA\on{-FactMod}_{\leq \mu}$ with its essential image in $\CA\on{-FactMod}$.
We have
$$\CA\on{-FactMod}_{\leq \mu_1}\subset \CA\on{-FactMod}_{\leq \mu_2} \text{ for } \mu_1\leq \mu_2.$$

\sssec{}

The presentation \eqref{e:config as colim} implies that the functors $(\iota_{\mu})_!$ define an equivalence
$$\underset{\mu\in \Lambda}{\underset{\longrightarrow}{\on{colim}}}\, 
\Shv_{\CG^\Lambda}(\Conf_{\leq \mu\cdot x})\to 
\Shv_{\CG^\Lambda}(\Conf_{\infty\cdot x}),$$
(see \secref{sss:limits and colimits} for the general paradigm). 

\medskip

In particular, the map
\begin{equation} \label{e:colim shv}
\underset{\mu\in \Lambda}{\underset{\longrightarrow}{\on{colim}}}\,  (\iota_{\mu})_!\circ (\iota_{\mu})^!\to \on{Id}
\end{equation}
is an isomorphism on $\Shv_{\CG^\Lambda}(\Conf_{\infty\cdot x})$. By the conservativity 
of $\oblv_{\on{Fact}}$, we obtain that \eqref{e:colim shv} is an isomorphism also in $\CA\on{-FactMod}$.
Hence, we obtain that the functors \eqref{e:iota on mod} also define an equivalence
$$\underset{\mu\in \Lambda}{\underset{\longrightarrow}{\on{colim}}}\, 
\CA\on{-FactMod}_{\leq \mu} \to \CA\on{-FactMod}.$$

\sssec{}

Let
$$\Conf_{=\mu\cdot x}\subset \Conf_{\leq \mu\cdot x}$$
be the open subscheme consisting of points \eqref{e:point of conf} for which $\lambda_x=\mu$. 
Note that the above open embedding, to be denoted by $\jmath_\mu$, is affine. 

\medskip

We can consider the corresponding category $\CA\on{-FactMod}_{=\mu}$, along with the pair of adjoint functors
\begin{equation} \label{e:on open}
(\jmath_\mu)^*:\CA\on{-FactMod}_{\leq \mu}\rightleftarrows \CA\on{-FactMod}_{=\mu}: (\jmath_\mu)_*,
\end{equation} 
commuting with the forgetful functor $\oblv_{\on{Fact}}$ and with $(\jmath_\mu)_*$ being fully faithful. 

\medskip

Note also that the essential image of $(\jmath_\mu)_*$ 
is the right orthogonal to the full subcategory of $\CA\on{-FactMod}_{\leq \mu}$ generated
by $\CA\on{-FactMod}_{\leq \mu'}$ for $\mu'<\mu$. 

\sssec{}

We have the following assertion:

\begin{lem} \label{l:on open}
The functor of stalk at the (unique) point 
$$\mu \cdot x\in (\Conf_{=\mu\cdot x})^\mu\subset \Conf_{=\mu\cdot x}$$
defines a t-exact equivalence from $\CA\on{-FactMod}_{=\mu}$ to the category $\Vect_{\CG^\Lambda_{\mu\cdot x}}$
(i.e., the category of vector spaces twisted by the fiber of the gerbe $\CG^\Lambda$ at the point $\mu\cdot x$). 
\end{lem} 

\begin{proof}
Follows, as in \secref{sss:fact mod expl}, from the fact that the map \eqref{e:add disj} defines an isomorphism from
$$\left(\Conf\times \Conf_{\infty\cdot x}\right)_{\on{disj}}\cap \Conf\times \{\mu\cdot x\}$$
to $\Conf_{=\mu\cdot x}$.
\end{proof} 

\sssec{}  \label{sss:fact hol}

Assume now that $\CA\in \Shv_{\CG^\Lambda}(\Conf)$ is \emph{holonomic} if our theory is that of D-modules
(the condition is vacuous for other choices of sheaf theory). 

\medskip

Then \lemref{l:on open} implies that the functor
$$(\jmath_\mu)_!:\Shv_{\CG^\Lambda}(\Conf_{=\mu\cdot x})\to 
\Shv_{\CG^\Lambda}(\Conf_{\leq \mu\cdot x}),$$
left adjoint to $(\jmath_\mu)^!=(\jmath_\mu)^*$, is well-defined on the essential image on the functor $\oblv_{\on{Fact}}$.

\medskip

Hence, as in \secref{sss:fact funct mod geom et}, 
we obtain that the functor $(\jmath_\mu)^!=(\jmath_\mu)^*$ in \eqref{e:on open} also admits a \emph{left adjoint}, to be denoted
$(\jmath_\mu)_!$, which commutes with the forgetful functor $\oblv_{\on{Fact}}$.

%

\medskip

The existence of $(\jmath_\mu)_!$ and its commutation with $\oblv_{\on{Fact}}$ implies that the functor
$(\iota_{\mu})_!=:(\iota_{\mu})_*$ of \eqref{e:iota on mod} admits also a left adjoint, to be denoted $(\iota_{\mu})^*$, which also
commutes with $\oblv_{\on{Fact}}$. 

\sssec{}  \label{sss:standards and costandards}

Choose a trivialization of the gerbe $\CG^\Lambda_{\mu\cdot x}$, thereby identifying $\CA\on{-FactMod}_{=\mu}$ with $\Vect$.

\medskip

Let $\CM^{\mu,*}_{\Conf}$ (resp., $\CM^{\mu,!}_{\Conf}$) denote the object of $\CA\on{-FactMod}$ equal to 
$(\iota_\mu)_*\circ (\jmath_\mu)_*$ (resp., $(\iota_\mu)_!\circ (\jmath_\mu)_!$) applied to 
$$\sfe\in \Vect\simeq \CA\on{-FactMod}_{=\mu}.$$

\medskip

We will call $\CM^{\mu,*}_{\Conf}$ (resp., $\CM^{\mu,!}_{\Conf}$) the \emph{co-standard} (resp., \emph{standard}) object. 

\sssec{}

Note that the objects $\CM^{\mu,!}_{\Conf}$ form a set of compact generators of the category $\CA\on{-FactMod}$. Indeed,
the functor $\CHom_{\CA\on{-FactMod}}(\CM^{\mu,!}_{\Conf},-)$ identifies with the functor
$$\CA\on{-FactMod}\overset{\iota_\mu^!}\longrightarrow \CA\on{-FactMod}_{\leq \mu}  \overset{\jmath_\mu^*}\longrightarrow 
\CA\on{-FactMod}_{=\mu}\simeq \Vect_{\CG^\Lambda_{\mu\cdot x}}\simeq \Vect.$$

Note that above functor identifies with
\begin{equation} \label{e:mu fiber}
\CA\on{-FactMod}\overset{\oblv_{\on{Fact}}}\longrightarrow \Shv_{\CG^\Lambda}(\Conf_{\infty\cdot x})\to 
\Vect_{\CG^\Lambda_{\mu\cdot x}}\simeq \Vect,
\end{equation}
where the middle arrow is the functor of !-fiber at $\mu\cdot x\in \Conf_{\infty\cdot x}$. 

\medskip

In particular, we obtain:

\begin{equation} \label{e:Ext standard fact}
\CHom_{\CA\on{-FactMod}}(\CM^{\mu',!}_{\Conf},\CM^{\mu,*}_{\Conf})=
\begin{cases}
&\sfe \text{ if } \mu'=\mu \\
&0 \text{ otherwise}.
\end{cases}
\end{equation}

\medskip

Note also that the objects $\CM^{\mu,*}_{\Conf}$ \emph{co-generate} $\CA\on{-FactMod}$.

\sssec{} \label{sss:action of lattice 2}

Let us be in the situation of \secref{ss:isogenies}. It is clear that for $\gamma\in \Lambda^\sharp$, we have
$$\on{Tr}^\gamma(\CM^{\mu,!}_{\Conf})\simeq \CM^{\mu+\gamma,!}_{\Conf} \text{ and }
\on{Tr}^\gamma(\CM^{\mu,*}_{\Conf})\simeq \CM^{\mu+\gamma,*}_{\Conf}.$$

\sssec{}

Assume now that $\CA$, viewed as an object of $\Shv_{\CG^\Lambda}(\Conf)$, belongs to 
$\Shv_{\CG^\Lambda}(\Conf)^{\on{loc.c}}$.

\medskip

Let
$$\CA\on{-FactMod}^{\on{loc.c}}\subset \Shv_{\CG^\Lambda}(\Conf_{\infty\cdot x})$$
be the full subcategory consisting of objects whose image under $\oblv_{\on{Fact}}$ belongs to the subcategory 
$\Shv_{\CG^\Lambda}(\Conf_{\infty\cdot x})^{\on{loc.c}}\subset \Shv_{\CG^\Lambda}(\Conf_{\infty\cdot x})$.

\medskip

Evidently, the objects $\CM^{\mu,!}_{\Conf}$ and $\CM^{\mu,*}_{\Conf}$ belong to $\CA\on{-FactMod}^{\on{loc.c}}$. 

\ssec{The t-structure on factorization modules} \label{ss:t on fact}

We retain the assumptions of \secref{sss:fact hol}. Let assume, in addition, that $\CA$, when viewed as an object
of $\Shv_{\CG^\Lambda}(\Conf)$, belongs to $(\Shv_{\CG^\Lambda}(\Conf))^\heartsuit$.

\medskip

We will show that in this case, the category $\CA\on{-FactMod}$ has a well-behaved t-structure and the abelian category
$(\CA\on{-FactMod})^\heartsuit$ is a highest weight category. 

\sssec{}

We define a t-structure on the category $\CA\on{-FactMod}$ by declaring an object $\CF$ to be coconnective if
$$\Hom_{\CA\on{-FactMod}}(\CM^{\mu,!}_{\Conf}[k],\CF)=0$$
for all $\mu\in \Lambda$ and $k>0$.

\medskip

By factorization and the assumption on $\CA$, we obtain that this condition is equivalent to
$$\oblv_{\on{Fact}}(\CF)\in (\Shv_{\CG^\Lambda}(\Conf_{\infty\cdot x}))^{\geq 0}.$$

We claim:

\begin{prop}  \label{p:properties of t fact} \hfill

\smallskip

\noindent{\em(a)} The functor $\oblv_{\on{Fact}}$ is t-exact.

\smallskip

\noindent{\em(b)} The objects $\CM^{\mu,!}_{\Conf},\CM^{\mu,*}_{\Conf}$ belong to $(\CA\on{-FactMod})^\heartsuit$.

\end{prop}

\begin{proof}

The fact that the functor $\oblv_{\on{Fact}}$ is left t-exact has been noted above. Since $(\CA\on{-FactMod})^{\leq 0}$
is generated under colimits by the objects $\CM^{\mu,!}_{\Conf}$, in order to show that $\oblv_{\on{Fact}}$ is right t-exact,
it suffices to show that $\oblv_{\on{Fact}}(\CM^{\mu,!}_{\Conf})\in (\Shv_{\CG^\Lambda}(\Conf_{\infty\cdot x}))^{\geq 0}$. However,
this follows from the fact that the open embedding $\jmath_\mu$ is affine. 

\medskip

Thus point (a) of the proposition has been proved. In order to prove point (b), 
since the functor $\oblv_{\on{Fact}}$ is conservative, it suffices to show that
$$\oblv_{\on{Fact}}(\CM^{\mu,!}_{\Conf}),\,\oblv_{\on{Fact}}(\CM^{\mu,*}_{\Conf})\in 
(\Shv_{\CG^\Lambda}(\Conf_{\infty\cdot x}))^\heartsuit.$$

This fact for $\oblv_{\on{Fact}}(\CM^{\mu,!}_{\Conf})\in 
(\Shv_{\CG^\Lambda}(\Conf_{\infty\cdot x}))^\heartsuit$ has been noted above. The corresponding fact for 
$\oblv_{\on{Fact}}(\CM^{\mu,*}_{\Conf})$ also follows from the fact that the open embedding $\jmath_\mu$ is affine. 

\end{proof}

\sssec{}

By \cite[Theorem 1.3.3.2]{Lur}, we have a canonically defined t-exact functor 
\begin{equation} \label{e:D of heart}
D^+\left((\CA\on{-FactMod})^\heartsuit\right)\to \CA\on{-FactMod}.
\end{equation}

We claim:

\begin{prop} \label{p:D of heart}
The functor \eqref{e:D of heart} is an equivalence onto the eventually coconnective part.
\end{prop} 

Before we give a proof, let us note the following: 

\medskip

The functor \eqref{e:D of heart} induces a \emph{bijection}
$$\on{Ext}^i_{(\CA\on{-FactMod})^\heartsuit}(\CM^{\mu',!}_{\Conf},\CM^{\mu,*}_{\Conf})\to
H^i\left(\CHom_{\CA\on{-FactMod}}(\CM^{\mu',!}_{\Conf},\CM^{\mu,*}_{\Conf})\right)$$
for $i=0,1$ and 
and an injection
$$\on{Ext}^2_{(\CA\on{-FactMod})^\heartsuit}(\CM^{\mu',!}_{\Conf},\CM^{\mu,*}_{\Conf})\to
H^2\left(\CHom_{\CA\on{-FactMod}}(\CM^{\mu',!}_{\Conf},\CM^{\mu,*}_{\Conf})\right).$$

\medskip

In particular, it follows from \eqref{e:Ext standard fact} that
$$\on{Ext}^2_{(\CA\on{-FactMod})^\heartsuit}(\CM^{\mu',!}_{\Conf},\CM^{\mu,*}_{\Conf})=0.$$

\medskip

Hence, we obtain that $(\CA\on{-FactMod})^\heartsuit$ has a structure of \emph{highest weight category}. 

\begin{proof}[Proof of \propref{p:D of heart}]

It is enough to show that the functor \eqref{p:D of heart} is fully faithful. Since the objects $\CM^{\mu,!}_{\Conf}$
(resp., $\CM^{\mu,*}_{\Conf}$) generate (resp., co-generate) $\CA\on{-FactMod}$, it is sufficient to show that the map
$$\on{Ext}^i_{(\CA\on{-FactMod})^\heartsuit}(\CM^{\mu',!}_{\Conf},\CM^{\mu,*}_{\Conf})\to
H^i\left(\CHom_{\CA\on{-FactMod}}(\CM^{\mu',!}_{\Conf},\CM^{\mu,*}_{\Conf})\right)$$
is an isomorphism for all $i$. 

\medskip

However, this follows from \eqref{e:Ext standard fact} and the fact that in a highest weight category
$$\on{Ext}^i_{(\CA\on{-FactMod})^\heartsuit}(\CM^{\mu',!}_{\Conf},\CM^{\mu,*}_{\Conf})=0, \text{ for } i\geq 1.$$

\end{proof}

\sssec{}

Consider the canonical map
$$\CM^{\mu,!}_{\Conf}\to \CM^{\mu,*}_{\Conf}.$$

Let $\CM^{\mu,!*}_{\Conf}$ denote its image, viewed as an object in $(\CA\on{-FactMod})^\heartsuit$. 

\medskip

We obtain that the 
objects $\CM^{\mu,!*}_{\Conf}$ are the irreducibles in the abelian category $(\CA\on{-FactMod})^\heartsuit$.

\sssec{} \label{sss:action of lattice 3}

Let us be in the situation of \secref{ss:isogenies}. It is clear that for $\gamma\in \Lambda^\sharp$, we have
$$\on{Tr}^\gamma(\CM^{\mu,!*}_{\Conf})\simeq \CM^{\mu+\gamma,!*}_{\Conf}.$$

\ssec{Comparison with the affine Grassmannian}    \label{ss:Conf vs Gr}

Recall the setting of \secref{ss:Conf Gr}. In this subsection we will use it to compare factorization algebras
in $\Shv_{\CG^T}(\Gr^{\omega^\rho}_{T,\Ran})$ and those in $\Shv{\CG^\Lambda}(\Conf)$, as well as modules
over them.   

\sssec{}   \label{sss:fact Gr T vs Conf} 

First off, the equivalence \eqref{e:Ran to config sheaves} implies that pullback along \eqref{e:Ran to config}
defines an equivalence
$$\on{FactAlg}(\Shv_{\CG^\Lambda}(\Conf))\simeq \on{FactAlg}(\Shv_{\CG^T}((\Gr^{\omega^\rho}_{T,\Ran})^{\on{neg}})), \quad
\CA_{\Conf}\mapsto \CA_{\Gr_T}.$$


\sssec{}

Let $\CA_{\Gr_T}$ be a factorization algebra in $\Shv_{\CG^T}(\Gr^{\omega^\rho}_{T,\Ran})$, and let
$\CA_{\Conf}$ be the corresponding factorization algebra in $\Shv_{\CG^\Lambda}(\Conf)$. 

\medskip

We obtain that pullback with respect to  \eqref{e:Ran to config marked} defines an equivalence
\begin{equation} \label{e:compare fact mod}
\CA_{\Conf}\on{-FactMod}(\Shv_{\CG^\Lambda}(\Conf_{\infty\cdot x}))\to 
\CA_{\Gr_T}\on{-FactMod}(\Shv_{\CG^T}((\Gr^{\omega^\rho}_{T,\Ran_x})^{\on{neg}}_{\infty\cdot x})).
\end{equation} 

\begin{rem}
Let $\CA$ be a factorization algebra in $\Shv_{\CG^T}(\Gr^{\omega^\rho}_{T,\Ran})$ supported on
$(\Gr^{\omega^\rho}_{T,\Ran})^{\on{neg}}$ (resp., $(\Gr^{\omega^\rho}_{T,\Ran})^{\on{non-pos}}$).

\medskip

It follows automatically from factorization, that any
object in $\CA_{\Gr_T}\on{-FactMod}(\Shv_{\CG^T}(\Gr^{\omega^\rho}_{T,\Ran_x}))$ is supported on
$(\Gr^{\omega^\rho}_{T,\Ran_x})^{\on{neg}}_{\infty\cdot x}$ (resp., $(\Gr^{\omega^\rho}_{T,\Ran_x})^{\on{non-pos}}_{\infty\cdot x}$). 
\end{rem} 

\sssec{Hecke action}

Let $\Lambda^\sharp$ be as in \secref{sss:sharp lattice}. Recall that we have an action of
$\Rep(T_H)$ on $\Shv_{\CG^T}(\Gr^{\omega^\rho}_{T,\Ran_x})$. It is easy to see that this action preserves
the subcategory
$$\Shv_{\CG^T}((\Gr^{\omega^\rho}_{T,\Ran_x})^{\on{neg}}_{\infty\cdot x}))\subset \Shv_{\CG^T}(\Gr^{\omega^\rho}_{T,\Ran_x}).$$

Let $\CA_{\Gr_T}$ be a factorization algebra in $\Shv_{\CG^T}(\Gr^{\omega^\rho}_{T,\Ran})$ supported on
$(\Gr^{\omega^\rho}_{T,\Ran})^{\on{neg}}$. By \secref{sss:fact funct mod geom}, the above action gives rise to an action 
of $\Rep(T_H)$ on $\CA_{\Gr_T}\on{-FactMod}(\Shv_{\CG^T}((\Gr^{\omega^\rho}_{T,\Ran_x})^{\on{neg}}_{\infty\cdot x}))$.

\medskip

It follows from the constructions that the equivalence \eqref{e:compare fact mod} intertwines this action 
and the action of $\Rep(T_H)$ on $\CA_{\Conf}\on{-FactMod}(\Shv_{\CG^\Lambda}(\Conf_{\infty\cdot x}))$
from \secref{sss:action of lattice 1}. 

\newpage 

\centerline{\bf Part II: The metaplectic Whittaker category of the affine Grassmannian} 

\bigskip

The goal of this work is to establish an equivalence of two categories: the (Hecke version of the)
metaplectic Whittaker category of the affine Grassmannian and a certain category of factorization modules.
In this Part we initiate the study of the former of these categories.

\section{The metaplectic Whittaker category}  \label{s:Whit}

In this section we introduce the metaplectic Whittaker category
of the affine Grassmannian, and study its basic properties.  

\ssec{Definition of the metaplectic Whittaker category}  \label{ss:def Whit}

In this subsection we introduce the metaplectic Whittaker category 
of the affine Grassmannian, denoted $\Whit_{q,x}(G)$. 
The definition involves infinite-dimensional algebro-geometric objects,
and we will rewrite it as a limit of categories of finite-dimensional nature. 

\sssec{}

We start with a geometric metaplectic data for the group $G$, i.e., a factorization gerbe $\CG^G$ on the affine
Grassmannian $\Gr_{G,\Ran}$.  

\medskip

We consider the twisted version $\Gr_{G,\Ran}^{\omega^\rho}$, see \secref{sss:Gr G rho}, and its fiber $\Gr^{\omega^\rho}_{G,x}$
over $\{x\}\in \Ran$. We denote by the same symbol $\CG^G$ the resulting gerbe on $\Gr^{\omega^\rho}_{G,x}$. 


\medskip

The indscheme $\Gr^{\omega^\rho}_{G,x}$ is acted on by the $\omega^\rho$-twist of the loop group at $x$, denoted 
$\fL(G)^{\omega^\rho}_x$. This is the group of automorphisms of the $G$-bundle $\omega^\rho$ on the formal punctured
disc around $x$. 

\medskip

Recall (see \secref{sss:gerbe on twisted grp}) that $\fL(G)^{\omega^\rho}_x$ carries a canonically defined multiplicative gerbe 
(also denoted $\CG^G$), so that the gerbe $\CG^G$ on $\Gr^{\omega^\rho}_{G,x}$ is twisted-equivariant against the gerbe 
$\CG^G$ on $\fL(G)^{\omega^\rho}_x$ with respect to the above action.

\sssec{}

Consider now the subgroup $\fL(N)^{\omega^\rho}_x\subset \fL(G)^{\omega^\rho}_x$. Due to the fact that $\fL(N)^{\omega^\rho}_x$
is ind-pro-unipotent, any gerbe over $\fL(N)^{\omega^\rho}_x$
is trivial (and the trivialization is uniquely fixed by its value at the origin). In particular, any multiplicative
gerbe on  $\fL(N)^{\omega^\rho}_x$ admits a canonical trivialization compatible with the multiplicative structure. 

\medskip

Thus, the gerbe $\CG^G$ on $\Gr^{\omega^\rho}_{G,x}$ is \emph{equivariant} with respect to $\fL(N)^{\omega^\rho}_x$. 

\sssec{}

Note that the group 
$$\fL(N)^{\omega^\rho}_x/[\fL(N)^{\omega^\rho}_x,\fL(N)^{\omega^\rho}_x]$$
identifies with $(\fL(\BG_a)^{\omega}_x)^I$, where $\fL(\BG_a)^{\omega}_x$ is the group indscheme of meromorphic differentials on the formal
punctured disc around $x$, and $I$ is the set of vertices of the Dynkin diagram of $G$. 

\medskip

The residue map defines a homomorphism $\fL(\BG_a)^{\omega}_x\to \BG_a$, see \secref{sss:twisted N}, 
Let $\chi_N$ denote the pullback of the Artin-Schreier sheaf
$\chi$ under the map
$$\fL(N)^{\omega^\rho}_x\to \fL(N)^{\omega^\rho}_x/[\fL(N)^{\omega^\rho}_x,\fL(N)^{\omega^\rho}_x]\simeq 
(\fL(\BG_a)^{\omega}_x)^I \to (\BG_a)^I\overset{\on{sum}}\longrightarrow \BG_a.$$

\medskip

We define $\Whit_{q,x}(G)$ to be the category of $(\fL(N)^{\omega^\rho}_x,\chi_N)$-equivariant objects in 
$\Shv_{\CG^G}(\Gr^{\omega^\rho}_{G,x})$, i.e.,
$$ \Whit_{q,x}(G):=\left(\Shv_{\CG^G}(\Gr^{\omega^\rho}_{G,x})\right)^{\fL(N)^{\omega^\rho}_x,\chi_N}.$$

\sssec{}  \label{sss:Nk}

Let us rewrite this definition in more detail. In particular, we will see that the forgetful functor
\begin{equation} \label{e:forget equivariance}
\Whit_{q,x}(G)\to \Shv_{\CG^G}(\Gr^{\omega^\rho}_{G,x})
\end{equation}
is fully faithful. 

\medskip

First, let us write $\fL(N)^{\omega^\rho}_x$ as a union of its group subschemes $N_k$, $k=1,2,...$.  By definition, we have 
\begin{equation} \label{e:bigger sugr}
\left(\Shv_{\CG^G}(\Gr^{\omega^\rho}_{G,x})\right)^{\fL(N)^{\omega^\rho}_x,\chi_N}\simeq \underset{k}{\on{lim}}\,
\left(\Shv_{\CG^G}(\Gr^{\omega^\rho}_{G,x})\right)^{N_k,\chi_N},
\end{equation} 
where the limit is taken with respect to the forgetful functors 
$$\left(\Shv_{\CG^G}(\Gr^{\omega^\rho}_{G,x})\right)^{N_{k'},\chi_N}\to 
\left(\Shv_{\CG^G}(\Gr^{\omega^\rho}_{G,x})\right)^{N_k,\chi_N}, \quad k'\geq k.$$

\medskip

For the fully faithfulness of \eqref{e:forget equivariance} it would suffice to see that each of the forgetful functors
\begin{equation} \label{e:forget equivariance k}
\left(\Shv_{\CG^G}(\Gr^{\omega^\rho}_{G,x})\right)^{N_k,\chi_N}\to \Shv_{\CG^G}(\Gr^{\omega^\rho}_{G,x})
\end{equation}
is fully faithful, and so the limit \eqref{e:bigger sugr} amounts to the intersection of the corresponding nested family of the subcategories 
$\left(\Shv_{\CG^G}(\Gr^{\omega^\rho}_{G,x})\right)^{N_k,\chi_N}$. 

\medskip

Thus, from now on we fix a particular index $k$. 

\sssec{}

Next, we write $\Gr^{\omega^\rho}_{G,x}$ as a union of an increasing family of its closed subschemes
$$\Gr^{\omega^\rho}_{G,x}\simeq \underset{j}\bigcup\, Y_j.$$

With no restriction of generality, we can assume that all $Y_j$ are $N_k$-invariant. We have
\begin{equation} \label{e:category as limit}
\Shv_{\CG^G}(\Gr^{\omega^\rho}_{G,x}) \simeq \underset{j}{\on{lim}}\, \Shv_{\CG^G}(Y_j),
\end{equation}
where the limit is taken with respect to the !-restriction functors
$$\Shv_{\CG^G}(Y_{j'})\to \Shv_{\CG^G}(Y_j), \quad j'\geq j.$$

\medskip

Thus, from \eqref{e:category as limit} we obtain
\begin{equation} \label{e:category as limit equiv}
\left(\Shv_{\CG^G}(\Gr^{\omega^\rho}_{G,x})\right)^{N_k,\chi_N}\simeq 
\underset{j}{\on{lim}}\, \left(\Shv_{\CG^G}(Y_j)\right)^{N_k,\chi_N}.
\end{equation}

\medskip

Thus, for the fully faithfulness of \eqref{e:forget equivariance k}, it would suffice to show that each of the functors 
\begin{equation} \label{e:forget equivariance k j}
\left(\Shv_{\CG^G}(Y_j)\right)^{N_k,\chi_N}\to \Shv_{\CG^G}(Y_j)
\end{equation} 
is fully faithful.

\medskip

Thus, from now on, the index $j$ will also be fixed. 

\sssec{}

The group indscheme $N_k$ can be written as a limit of finite-dimensional unipotent groups $N_{k,l}$, $l=1,2,...$.
With no restriction of generality, we can assume that the action of $N_k$ on $Y_j$ factors through $N_{k,l}$ for every $l$.

\medskip

By definition, we have
$$\left(\Shv_{\CG^G}(Y_j)\right)^{N_k,\chi_N}=\left(\Shv_{\CG^G}(Y_j)\right)^{N_{k,l},\chi_N}$$
for \emph{any} such $l$, where we notice that the forgetful functors
$$\left(\Shv_{\CG^G}(Y_j)\right)^{N_{k,l},\chi_N}\to \left(\Shv_{\CG^G}(Y_j)\right)^{N_{k,l'},\chi_N}, \quad l'\geq l$$
are equivalences, because
$$\on{ker}(N_{k,l'}\to N_{k,l})$$
are unipotent. Since the groups $N_{k,l}$ are themselves unipotent, the forgetful functors
\begin{equation} \label{e:forget equivariance k j l}
\left(\Shv_{\CG^G}(Y_j)\right)^{N_{k,l},\chi_N}\to \Shv_{\CG^G}(Y_j)
\end{equation} 
are fully faithful, and hence so is \eqref{e:forget equivariance k j}. 

\sssec{}

We note that the forgetful functors \eqref{e:forget equivariance k j l} admit right adjoints, denoted $\on{Av}_*^{N_{k,l},\chi_N}$. 
Denote the resulting forgetful functor to \eqref{e:forget equivariance k j} by $\on{Av}_*^{N_k,\chi_N}$. 

\medskip

These right adjoints make each of the diagrams 
$$
\CD
\Shv_{\CG^G}(Y_j)  @>{\on{Av}_*^{N_k,\chi_N}}>>  \left(\Shv_{\CG^G}(Y_j)\right)^{N_k,\chi_N} \\
@AAA   @AAA   \\
\Shv_{\CG^G}(Y_{j'})  @>{\on{Av}_*^{N_k,\chi_N}}>>  \left(\Shv_{\CG^G}(Y_{j'})\right)^{N_k,\chi_N} 
\endCD
$$
with $j'\geq j$ commutative, e.g., because the corresponding diagram of left adjoints 
$$
\CD
\Shv_{\CG^G}(Y_j)  @<<<  \left(\Shv_{\CG^G}(Y_j)\right)^{N_k,\chi_N} \\
@VVV   @VVV   \\
\Shv_{\CG^G}(Y_{j'})  @<<<  \left(\Shv_{\CG^G}(Y_{j'})\right)^{N_k,\chi_N} 
\endCD
$$
is commutative (here the vertical arrows are given by direct image with respect to the closed embedding $Y_j\hookrightarrow Y_{j'}$). 

\medskip

Thus, the above right adjoints combine to a right adjoint, also denoted $\on{Av}_*^{N_k,\chi_N}$ to \eqref{e:forget equivariance k}.

\sssec{}

Using the fully faithful embedding \eqref{e:forget equivariance k}, we can view $\on{Av}_*^{N_k,\chi_N}$ as an endo-functor of 
$$\Shv_{\CG^G}(\Gr^{\omega^\rho}_{G,x}).$$

Being a composition of a right adjoint followed by a left adjoint, it has a natural structure of a comonad; moreover the co-multiplication
map
$$\on{Av}_*^{N_k,\chi_N}\to \on{Av}_*^{N_k,\chi_N}\circ \on{Av}_*^{N_k,\chi_N}$$
is an isomorphism. 

\medskip

The subcategory 
$$\left(\Shv_{\CG^G}(\Gr^{\omega^\rho}_{G,x})\right)^{N_k,\chi_N}\subset \Shv_{\CG^G}(\Gr^{\omega^\rho}_{G,x})$$
consists of those objects $\CF$, for which the counit map
\begin{equation} \label{e:map to Av}
\on{Av}_*^{N_k,\chi_N}(\CF)\to \CF
\end{equation} 
is an isomorphism. The subcategory
\begin{equation} \label{e:forget equivariance all}
\Whit_{q,x}(G)\subset \Shv_{\CG^G}(\Gr^{\omega^\rho}_{G,x})
\end{equation} 
consists of those objects $\CF$, for which the maps \eqref{e:map to Av} are isomorphisms for all $k$. 

\sssec{}

The inclusion \eqref{e:forget equivariance all} admits a \emph{discontinuous} right adjoint given by
$$\on{Av}_*^{\fL(N)^{\omega^\rho}_x,\chi_N}:=\underset{k}{\on{lim}}\, \on{Av}_*^{N_k,\chi_N}.$$

\ssec{Structure of the metaplectic Whittaker category}

In this subsection we will study the basic structural properties of $\Whit_{q,x}(G)$, namely, its
stratification indexed by the elements of $\Lambda^+$, and the corresponding standard and 
costandard objects. 

\sssec{}

In addition to right adjoint given by $\on{Av}_*^{N_k,\chi_N}$, the functor \eqref{e:forget equivariance} admits a \emph{partially defined} left adjoint,
denoted $\on{Av}_!^{N_k,\chi_N}$, given by !-averaging. This partially defined adjoint is always defined in the context of $\ell$-adic sheaves. 
In the context of D-modules, it is defined on holonomic objects. 

\medskip

If $\CF\in  \Shv_{\CG^G}(\Gr^{\omega^\rho}_{G,x})$ is an object on which all $\on{Av}_!^{N_k,\chi_N}$ are defined, then
the partially left adjoint to the inclusion \eqref{e:forget equivariance all} is defined on $\CF$ and is given by
$$\on{Av}_!^{\fL(N)^{\omega^\rho}_x,\chi_N}(\CF):=\underset{k}{\on{colim}}\, \on{Av}_!^{N_k,\chi_N}(\CF).$$

\sssec{}

Pick a uniformizer $t$ at $x$ and for $\lambda\in \Lambda$ let $t^\lambda$ denote the corresponding point of $\Gr^{\omega^\rho}_{G,x}$.
Choose a trivialization of the fiber of $\CG^G$ at $t^\lambda$. Let $\delta_{t^\lambda,\Gr}$ denote the resulting point of
$\Shv_{\CG^G}(\Gr^{\omega^\rho}_{G,x})$.

\medskip

Denote 
$$W^{\lambda,!}:=\on{Av}_!^{\fL(N)^{\omega^\rho}_x,\chi_N}(\delta_{t^\lambda,\Gr})[-\langle \lambda,2\check\rho\rangle ]\in \Whit_{q,x}(G).$$ 

By construction, the objects $W^{\lambda,!}$ are compact in $\Whit_{q,x}(G)$. We will shortly see that these objects generate
$\Whit_{q,x}(G)$ (and they vanish unless $\lambda$ is dominant).

\sssec{}

For $\lambda\in \Lambda$, let $S^\lambda$ be the corresponding $\fL(N)^{\omega^\rho}_x$-orbit on $\Gr^{\omega^\rho}_{G,x}$, i.e.,
\begin{equation} \label{e:semiinf orbits}
S^\lambda=\fL(N)^{\omega^\rho}_x\cdot t^\lambda\cdot \fL^+(G)^{\omega^\rho}_x/\fL^+(G)^{\omega^\rho}_x,
\end{equation}
and let $\ol{S}{}^\lambda$
denote its closure. 
Denote by $\ol\bi{}_\lambda$ the closed embedding
$$\ol{S}{}^\lambda\to \Gr^{\omega^\rho}_{G,x},$$
and by $\bi_\lambda$ the locally closed embedding
$$S^\lambda\to \Gr^{\omega^\rho}_{G,x}.$$

\sssec{}  

Recall the context of \secref{sss:limits and colimits}.  In particular, if $Y$ is a prestack equipped with a gerbe $\CG$ and written as 
$$Y=\underset{\alpha}{\on{lim}}\, Y_\alpha,\quad (\alpha\to \alpha')\mapsto Y_{\alpha}\overset{f_{\alpha,\alpha'}}\longrightarrow Y_{\alpha'}$$
so that 
$$\Shv_\CG(Y)\simeq \underset{\alpha}{\on{lim}}\, \Shv_\CG(Y_\alpha),$$
where the transition functors are $f_{\alpha,\alpha'}^!$, and if the functors $(f_{\alpha,\alpha'})_!$ are well-defined, then we also have
$$\Shv_\CG(Y)\simeq \underset{\alpha}{\on{colim}}\, \Shv_\CG(Y_\alpha),$$
where the transition functors are now $(f_{\alpha,\alpha'})_!$. 

\sssec{}

We have
$$\Gr^{\omega^\rho}_{G,x}\simeq \underset{\lambda\in \Lambda}{\on{colim}}\, \ol{S}{}^\lambda,$$
where $\Lambda$ is regarded as a poset with the standard order relation. 

\medskip

Thus, we obtain
$$\Shv_{\CG^G}(\Gr^{\omega^\rho}_{G,x})\simeq \underset{\lambda\in \Lambda}{\on{lim}}\, \Shv_{\CG^G}(\ol{S}{}^\lambda)$$
with respect to the pullback functors, \emph{and also}
$$\Shv_{\CG^G}(\Gr^{\omega^\rho}_{G,x})\simeq \underset{\lambda\in \Lambda}{\on{colim}}\, \Shv_{\CG^G}(\ol{S}{}^\lambda),$$
with respect to the pushforward functors.

\medskip

For each pair of indices $\lambda\leq \lambda'$, we have the following commutative diagrams
$$
\CD
\Shv_{\CG^G}(\ol{S}{}^{\lambda'}) @<<<  (\Shv_{\CG^G}(\ol{S}{}^{\lambda'}))^{\fL(N)^{\omega^\rho}_x,\chi_N}   \\
@VVV   @VVV  \\
\Shv_{\CG^G}(\ol{S}{}^\lambda) @<<<  (\Shv_{\CG^G}(\ol{S}{}^\lambda))^{\fL(N)^{\omega^\rho}_x,\chi_N} 
\endCD
$$
(where the left vertical arrows are given by pullback)
and 
$$
\CD
\Shv_{\CG^G}(\ol{S}{}^{\lambda'}) @<<<  (\Shv_{\CG^G}(\ol{S}{}^{\lambda'}))^{\fL(N)^{\omega^\rho}_x,\chi_N}   \\
@AAA   @AAA  \\
\Shv_{\CG^G}(\ol{S}{}^\lambda) @<<<  (\Shv_{\CG^G}(\ol{S}{}^\lambda))^{\fL(N)^{\omega^\rho}_x,\chi_N} 
\endCD
$$
(where the left vertical arrows are given by pushforward).

\medskip

From here we obtain the presentations
$$\Whit_{q,x}(G)\simeq \underset{\lambda\in \Lambda}{\on{lim}}\,    (\Shv_{\CG^G}(\ol{S}{}^\lambda))^{\fL(N)^{\omega^\rho}_x,\chi_N},$$
with respect to pullbacks,
and  
$$\Whit_{q,x}(G)\simeq \underset{\lambda\in \Lambda}{\on{colim}}\,    (\Shv_{\CG^G}(\ol{S}{}^\lambda))^{\fL(N)^{\omega^\rho}_x,\chi_N},$$
with respect to pushforwards.

\medskip

Moreover, since each of the pushforward functors
$$(\Shv_{\CG^G}(\ol{S}{}^\lambda))^{\fL(N)^{\omega^\rho}_x,\chi_N} \to (\Shv_{\CG^G}(\ol{S}{}^{\lambda'}))^{\fL(N)^{\omega^\rho}_x,\chi_N}$$
is fully faithful, so are the resulting pushforward functors
$$(\ol\bi{}_\lambda)_!: (\Shv_{\CG^G}(\ol{S}{}^\lambda))^{\fL(N)^{\omega^\rho}_x,\chi_N} \to \Whit_{q,x}(G).$$

We denote the essential image of the latter functor by $\Whit_{q,x}(G)_{\leq \lambda}$. By construction, an object of 
$\Whit_{q,x}(G)$ belongs to $\Whit_{q,x}(G)_{\leq \lambda}$ if it is supported on
$\ol{S}{}^\lambda$. We have:
$$\lambda\leq \lambda'\, \Rightarrow\, \Whit_{q,x}(G)_{\leq \lambda}\subset \Whit_{q,x}(G)_{\leq \lambda'},$$
and 
$$W^{\lambda,!}\in \Whit_{q,x}(G)_{\leq \lambda}.$$  

\sssec{}

Let $\bj_\lambda$ denote the open embedding 
$$S^\lambda \hookrightarrow \ol{S}{}^\lambda.$$

\medskip

The adjoint pair 
\begin{equation} \label{e:open adj Gr}
(\bj_\lambda)^*:\Shv_{\CG^G}(\ol{S}{}^\lambda) \rightleftarrows
\Shv_{\CG^G}(S^\lambda):(\bj_\lambda)_*
\end{equation} 
gives rise to an adjunction
\begin{multline} \label{e:open adj Whit}
(\bj_\lambda)^*:\Whit_{q,x}(G)_{\leq\lambda}=\\
=(\Shv_{\CG^G}(\ol{S}{}^\lambda))^{\fL(N)^{\omega^\rho}_x,\chi_N} \rightleftarrows
(\Shv_{\CG^G}(S^\lambda))^{\fL(N)^{\omega^\rho}_x,\chi_N}=:
\Whit_{q,x}(G)_{=\lambda}:(\bj_\lambda)_*
\end{multline} 
such that make both circuits in the diagram
$$
\xymatrix{
\Whit_{q,x}(G)_{\leq\lambda} \ar[rr]<2pt>^{\bj_\lambda^*} \ar[d] &&
\Whit_{q,x}(G)_{=\lambda}
\ar[ll]<2pt>^{(\bj_\lambda)_*} \ar[d] \\
\Shv_{\CG^G}(\ol{S}{}^\lambda)
\ar[rr]<2pt>^{\bj_\lambda^*} &&
\Shv_{\CG^G}(S^\lambda)
\ar[ll]<2pt>^{(\bj_\lambda)_*}}
$$
commute.

\medskip

In particular, since the co-unit of the adjunction
$$\bj_\lambda^*\circ (\bj_\lambda)_*\to \on{Id}$$
is an isomorphism for \eqref{e:open adj Gr}, it is such also for \eqref{e:open adj Whit}. In particular, we obtain that the functor
$$(\bj_\lambda)_*:\Whit_{q,x}(G)_{=\lambda}\to \Whit_{q,x}(G)_{\leq \lambda}$$
is fully faithful. 

\medskip

The essential image of $\Whit_{q,x}(G)_{=\lambda}$ in $\Whit_{q,x}(G)_{\leq \lambda}$ is the \emph{right orthogonal} of the full subcategory 
$$\Whit_{q,x}(G)_{<\lambda}\subset \Whit_{q,x}(G)_{\leq \lambda}$$
generated
by the essential images of $\Whit_{q,x}(G)_{\leq \lambda'}$ with $\lambda'<\lambda$. 

\sssec{}

Set
$$\overset{\circ}W{}^\lambda:=\on{Av}_!^{\fL(N)^{\omega^\rho}_x,\chi_N}(\delta_{t^\lambda,\Gr})[-\langle \lambda,2\check\rho\rangle ]\in \Whit_{q,x}(G)_{=\lambda}.$$ 

Clearly,
$$W^{\lambda,!}\simeq (\bi_\lambda)_!(\overset{\circ}W{}^\lambda)\simeq (\ol\bi_\lambda)_*\circ (\bj_\lambda)_!(\overset{\circ}W{}^\lambda),$$
where $(\bi_\lambda)_!$ (resp., $(\bj_\lambda)_!$) denotes the (partially defined) left adjoint of $\bi_\lambda^!$ (resp., $\bj_\lambda^!$). 

\medskip

We denote by 
$$W^{\lambda,*}\in \Whit_{q,x}(G)_{\leq \lambda}$$ 
the object
$(\bi_\lambda)_*(\overset{\circ}W{}^\lambda)$. 

\medskip

Almost by definition, we have:
\begin{equation} \label{e:standards Whit}
\Maps_{\Whit_{q,x}(G)}(W^{\lambda,!},W^{\lambda',*})=
\begin{cases}
&\sfe \text{ if } \lambda= \lambda' \\
&0 \text{ otherwise}.\\
\end{cases}
\end{equation}

\sssec{}

We claim:

\begin{prop} \label{p:Whittaker strata}  \hfill

\smallskip

\noindent{\em(a)} The category $\Whit_{q,x}(G)_{=\lambda}$ is zero unless $\lambda$ is dominant.

\smallskip

\noindent{\em(b)} For $\lambda$ dominant, the category $\Whit_{q,x}(G)_{=\lambda}$ is non-canonically equivalent to 
$\Vect$, with the generator given by $\overset{\circ}W{}^\lambda$. 
\end{prop}

From here we easily obtain: 

\begin{cor}  \label{c:Whittaker strata}   \hfill

\smallskip

\noindent{\em(a)} 
The objects $W^{\lambda,!}$ and $W^{\lambda,*}$ with $\lambda$ dominant generate $\Whit_{q,x}(G)$. 

\smallskip

\noindent{\em(b)} 
The canonical map $W^{\lambda,!}\to W^{\lambda,*}$ is an isomorphism for any $\lambda$ that is minimal in $\Lambda^+$ with respect 
to the standard order relation (in particular, for $\lambda=0$). 

\smallskip

\noindent{\em(c)} The category $\Whit_{q,x}(G)_{\leq \lambda}$ is zero unless $\lambda$ is dominant.

\end{cor} 

\begin{proof}[Proof of \propref{p:Whittaker strata}]

Consider the functor 
\begin{equation} \label{e:eval Whit at pt}
\Whit_{q,x}(G)_{=\lambda}\to \Vect,
\end{equation} 
equal to the composition of the forgetful functor
$$\Whit_{q,x}(G)_{=\lambda}\to \Shv_{\CG^G}(S^\lambda),$$
followed by the functor of taking the !-fiber at $t^\lambda$. 

\medskip

We claim that the functor \eqref{e:eval Whit at pt} is conservative. Indeed, if for some $\CF\in 
(\Shv_{\CG^G}(S^\lambda))^{\fL(N)^{\omega^\rho}_x,\chi_N}$ its fiber at $t^\lambda$ is zero,
then the restriction of $\CF$ to the $N_k$-orbit of $t^\lambda$ is zero for all $k$. However,
$$S^\lambda=\underset{k}\cup\, N_k\cdot t^\lambda,$$
and hence $\CF=0$. 

\medskip

The functor \eqref{e:eval Whit at pt} admits a left adjoint
that sends the generator $\sfe\in \Vect$ to
$$\on{Av}_!^{\fL(N)^{\omega^\rho}_x,\chi_N}(\delta_{t^\lambda,\Gr})\simeq \underset{k}{\on{colim}}\, \on{Av}_!^{N_k,\chi_N}(\delta_{t^\lambda,\Gr}).$$

\medskip

We claim that the resulting monad on $\Vect$ is the identity if $\lambda$ is dominant and zero otherwise.

\medskip

Indeed, note that $\on{Av}_!^{N_k,\chi_N}(\delta_{t^\lambda,\Gr})$ identifies with $\chi_N|_{N_k\cdot t^\lambda}$ if $\chi_N$, restricted
to $\on{Stab}_{N_k}(t^\lambda)$, is trivial, and zero otherwise. 
Hence, its !-fiber at $t^\lambda$ of $\on{Av}_!^{N_k,\chi_N}(\delta_{t^\lambda,\Gr})$ is $\sfe$ if $\chi_N$ restricted
to $\on{Stab}_{N_k}(t^\lambda)$ is trivial, and zero otherwise. Now, the dominance condition on $\lambda$ is equivalent to the fact that 
$\chi_N$ restricted to $\on{Stab}_{\fL(N)^{\omega^\rho}_x}(t^\lambda)$ is trivial.

\end{proof}

\ssec{The t-structure}  \label{ss:t on Whit}

In this subsection we will show how to endow $\Whit_{q,x}(G)$ with a t-structure. We note, however, that
this t-structure ``has nothing to do" with the t-structure on the ambient category $\Shv_{\CG^G}(\Gr^{\omega^\rho}_{G,x})$.

\sssec{}

We introduce a t-structure on $\Whit_{q,x}(G)$ by declaring that an object $\CF$ is coconnective if
$$\Hom_{\Whit_{q,x}(G)}(W^{\lambda,!}[k],\CF)=0$$ for all $\lambda$ and $k>0$. 

\medskip

In \secref{sss:proof t on Whit} we will prove: 

\begin{prop} \label{p:t on Whit}  \hfill 

\smallskip

\noindent{\em(a)} The subcategories $\Whit_{q,x}(G)_{\leq \lambda}\subset \Whit_{q,x}(G)$ are compatible with the t-structure. 

\smallskip

\noindent{\em(b)} The objects $W^{\lambda,!}$ and $W^{\lambda',*}$ belong to the heart of the t-structure and 
are of finite length.

\end{prop}

\begin{rem}
It follows formally from \propref{p:t on Whit}(b) (see \propref{p:D of heart}) that the functor
$$D^+\left(\Whit_{q,x}(G)\right)\to \Whit_{q,x}(G)$$
is an equivalence onto the eventually connective part.
\end{rem}

\sssec{}

Let $W^{\lambda,!*}$ denote the image of the canonical map $W^{\lambda,!}\to W^{\lambda,*}$.

\begin{cor} \label{c:t on Whit}  \hfill 

\smallskip

\noindent{\em(a)} The irreducibles in $(\Whit_{q,x}(G))^\heartsuit$ are the objects $W^{\lambda,!*}$.

\smallskip

\noindent{\em(b)} The objects $W^{\lambda,!*}$ are compact in $\Whit_{q,x}(G)$ and they generate $\Whit_{q,x}(G)$. 
\end{cor}

\begin{proof}[Proof of \corref{c:t on Whit}]

Let $L$ be an object in $(\Whit_{q,x}(G))^\heartsuit$. By the definition of the t-structure, it admits a non-zero map
$W^{\lambda,!}\to L$ for some $\lambda\in \Lambda^+$. If $L$ is irreducible, then the above map is surjective. 
In particular, $L\in \Whit_{q,x}(G)_{\leq \lambda}$, and $(\bj_\lambda)^*\circ (\ol\bi{}_\lambda)^!(L)$ admits a non-zero map to
$\bj_\lambda^*(W^{\lambda,!})$. Hence, $L$ admits a non-zero map to $W^{\lambda,*}$. It follows from 
\corref{c:Whittaker strata} that $L$ equals $W^{\lambda,!*}$. 

\medskip

Vice versa, let $L$ be an irreducile submodule of $W^{\lambda,!*}$. By the above, it is of the form $W^{\lambda',!*}$
for some $\lambda'$. In particular, we obtain a non-zero map
$$W^{\lambda',!}\twoheadrightarrow L\hookrightarrow W^{\lambda,!*}\hookrightarrow W^{\lambda,!},$$
which by \corref{c:Whittaker strata} implies that $\lambda'=\lambda$. 

\medskip

To prove that $W^{\lambda,!*}$ are compact in $\Whit_{q,x}(G)$ we argue by induction on $\lambda\in \Lambda^+$ with respect to the standard
order relation. The base of the induction is provided by \corref{c:Whittaker strata}(b). 

\medskip

Suppose the statement is true for $\lambda'<\lambda$. It is enough to show that
$\on{ker}(W^{\lambda,!}\to W^{\lambda,!*})$ is compact. We note, however, that the above object belongs to $\Whit_{q,x}(G)_{< \lambda}$ and is of finite length
(by \propref{p:t on Whit}(b)), and thus is a finite
extension of objects $W^{\lambda',!*}$ for $\lambda'<\lambda$. This implies the assertion.

\end{proof}

\begin{cor}  \label{c:cost compact in Whit}
The objects $W^{\lambda,*}$ are compact in $\Whit_{q,x}(G)$.
\end{cor}  

\begin{proof}
Follows by combining \propref{p:t on Whit}(b) and \corref{c:t on Whit}(b).
\end{proof} 

\begin{rem}
Note that objects lying in the image of the forgetful functor
$$\Whit_{q,x}(G)\to \Shv_{\CG^G}(\Gr^{\omega^\rho}_{G,x})$$
is \emph{infinitely connective}, i.e., it sends all objects of $\Whit_{q,x}(G)$ to objects of 
$\Shv_{\CG^G}(\Gr^{\omega^\rho}_{G,x})$ that have zero cohomologies with respect to 
the natural t-structure on that category.

\medskip

To show this, it suffices to show that the generators of $\Whit_{q,x}(G)$, i.e., the
objects $W^{\lambda,!}$, map to infinitely connective objects in. This follows from the fact that  $\on{Av}_!^{N_k,\chi_N}(\delta_{t^\lambda,\Gr})$
lives in cohomological degrees $\leq -\dim(N_k\cdot t^\lambda)$.

\end{rem}

\sssec{A digression: properties of t-structures} \label{sss:properties of t}

We shall say that a t-structure on a compactly generated category $\bC$ is \emph{compactly generated} if:

\begin{itemize}

\item $(\bC)^{\leq 0}$ is generated under colimits by $\bC^c\cap (\bC)^{\leq 0}$.

\end{itemize}

\medskip

We shall say that a t-structure on a compactly generated category $\bC$ is \emph{coherent} if, moreover, 

\begin{itemize}

\item Compact objects in $\bC$ are cohomologically bounded;

\item The subcategory $\bC^c\subset \bC$ is preserved by the truncation functors.

\end{itemize}

\medskip

We shall say that a t-structure is \emph{Noetherian} if, in addition:

\begin{itemize}

\item 
The subcategory 
$\bC^c\cap (\bC)^\heartsuit\subset (\bC)^\heartsuit$
is stable under subquotients (in particular is abelian). 

\end{itemize}

\medskip

Finally, we will that a t-structure is \emph{Artinian} if, moreover:

\begin{itemize}

\item Each object of $\bC^c\cap (\bC)^\heartsuit\subset (\bC)^\heartsuit$ has finite length. 

\end{itemize}

It is easy to see that a t-structure on $\bC$ is Artinian if and only if 
irreducible objects in $(\bC)^\heartsuit$ are compact and they generate $\bC$. 

\sssec{}

With the above definitions, we obtain that \corref{c:t on Whit} implies the following: 

\begin{cor} \label{c:Whit Art}
The t-structure on $\Whit_{q,x}(G)$ is Artinian. 
\end{cor} 

\section{The dual and the global definitions of the metaplectic Whittaker category}  \label{s:dual Whit}

The goal of this section is two-fold: we will show that the metaplectic Whittaker category is essentially 
self-dual (as a DG category), and that one can define it alternatively using global geometry. 

\medskip

Both these
facts will be used in the proof of the main theorem, which compares $\Whit_{q,x}(G)$ with a certain category
of factorization modules. 
 
\ssec{The definition as \emph{coinvariants}}

In the previous section we defined $\Whit_{q,x}(G)$ as the category of $(\fL(N)^{\omega^\rho}_x,\chi_N)$-invariant objects in 
$\Shv_{\CG^G}(\Gr^{\omega^\rho}_{G,x})$. We will now introduce another category, denoted
$\Whit_{q,x}(G)_{\on{co}}$, by taking $\fL(N)^{\omega^\rho}_x$-coinvariant objects in 
$\Shv_{\CG^G}(\Gr^{\omega^\rho}_{G,x})$.

\sssec{}

We define $\Whit_{q,x}(G)_{\on{co}}$ to be the quotient DG category of $\Shv_{\CG^G}(\Gr^{\omega^\rho}_{G,x})$ by the 
full subcategory generated by objects 
$$\on{Fib}(\on{Av}_*^{N_k,\chi_N}(\CF)\to \CF), \quad \forall\, \CF\in \Shv_{\CG^G}(\Gr^{\omega^\rho}_{G,x}), \forall k.$$

\medskip

I.e., for a test DG category $\bC$, the datum of a functor $\Whit_{q,x}(G)_{\on{co}}\to \bC$ is equivalent to the datum of a functor 
$\Shv_{\CG^G}(\Gr^{\omega^\rho}_{G,x})\to \bC$ that maps all morphisms \eqref{e:map to Av} to isomorphisms.

\sssec{}

We can give a similar definition for every fixed subgroup $N_k$; denote the resulting category of coinvariants by 
$$\left(\Shv_{\CG^G}(\Gr^{\omega^\rho}_{G,x})\right)_{N_k,\chi_N}.$$

\medskip

Note that since the co-monad $\on{Av}_*^{N_k,\chi_N}$ is idempotent, the averaging functor
$$\on{Av}_*^{N_k,\chi_N}:\Shv_{\CG^G}(\Gr^{\omega^\rho}_{G,x})\to 
\left(\Shv_{\CG^G}(\Gr^{\omega^\rho}_{G,x})\right)^{N_k,\chi_N}$$
defines a functor
\begin{equation} \label{e:from coinv to inv fixed k}
\left(\Shv_{\CG^G}(\Gr^{\omega^\rho}_{G,x})\right)_{N_k,\chi_N}\to
\left(\Shv_{\CG^G}(\Gr^{\omega^\rho}_{G,x})\right)^{N_k\chi_N}.
\end{equation} 

\begin{prop} \label{p:from coinv to inv fixed k}
The functor \eqref{e:from coinv to inv fixed k} is an equivalence.
\end{prop}

\begin{proof} 
This is a formal assertion, valid for any idempotent co-monad acting on a DG category.
\end{proof}

\sssec{}

Consider now the colmit DG category 
$$\underset{k}{\on{colim}}\, \left(\Shv_{\CG^G}(\Gr^{\omega^\rho}_{G,x})\right)^{N_k,\chi_N},$$
where for $k\leq k'$ the transition functor
$$\left(\Shv_{\CG^G}(\Gr^{\omega^\rho}_{G,x})\right)^{N_k,\chi_N}\to 
\left(\Shv_{\CG^G}(\Gr^{\omega^\rho}_{G,x})\right)^{N_{k'},\chi_N}$$
is given by $\on{Av}_*^{N_{k'},\chi_N}$.  (Compare this with the limit in \eqref{e:bigger sugr}, which was taken
with respect to the transition functors being the forgetful functors.) 

\medskip

From \propref{p:from coinv to inv fixed k} we obtain that we have a canonical equivalence
\begin{equation} \label{e:coin as colimit}
\underset{k}{\on{colim}}\, \left(\Shv_{\CG^G}(\Gr^{\omega^\rho}_{G,x})\right)^{N_k,\chi_N}\simeq
\Whit_{q,x}(G)_{\on{co}}.
\end{equation} 

\sssec{}

For a pair of indices $k\leq k'$, let $\ell_{k,k'}$ denote the *-fiber of the dualizing sheaf on $N_{k'}/N_k$ at the origin
(ignoring the Galois action, this is just $\sfe[2\dim(N_{k'}/N_k)]$). 

\medskip

In addition to the tautological natural transformation
$$\on{Av}_*^{N_{k'},\chi_N}\to\on{Av}_*^{N_k,\chi_N}.$$
We have a natural transformation
\begin{equation}  \label{e:more ev}
\on{Av}_*^{N_k,\chi_N}\to \ell_{k,k'}\otimes \on{Av}_*^{N_{k'},\chi_N}.
\end{equation}

\medskip

With no restriction of generality we can assume that our set of indices has a smallest element $k=0$ such that $N_0=\fL^+(N)^{\omega^\rho}_x$.
Consider the functor 
$$\Shv_{\CG^G}(\Gr^{\omega^\rho}_{G,x})\to \Shv_{\CG^G}(\Gr^{\omega^\rho}_{G,x})$$
defined as
\begin{equation} \label{e:pre T}
\underset{k}{\on{colim}}\, \ell_{0,k}\otimes \on{Av}_*^{N_k,\chi_N},
\end{equation}
where the transition maps are given by \eqref{e:more ev} and the isomorphisms 
$$\ell_{0,k}\circ \ell_{k,k'}\simeq \ell_{0,k'}.$$

\medskip

It is clear that the image of \eqref{e:pre T} belongs to $\left(\Shv_{\CG^G}(\Gr^{\omega^\rho}_{G,x})\right)^{N_k,\chi_N}$
for every $k$. 
Moreover, it maps all the morphisms of the form \eqref{e:map to Av} to isomorphisms. Hence, \eqref{e:pre T} gives rise to a functor

\begin{equation} \label{e:functor T}
\on{Ps-Id}:\Whit_{q,x}(G)_{\on{co}}\to \Whit_{q}(\Gr^{\omega^\rho}_{G,x})
\end{equation}

\sssec{}

For $\lambda\in \Lambda$, let $W^{\lambda,*}_{\on{co}}\in \Whit_{q,x}(G)_{\on{co}}$ denote the image of 
$\delta_{t^\lambda,\Gr}[\langle \lambda,2\check\rho\rangle ]\in \Shv_{\CG^G}(\Gr^{\omega^\rho}_{G,x})$ 
under the projection 
$$\Shv_{\CG^G}(\Gr^{\omega^\rho}_{G,x})\to \Whit_{q,x}(G)_{\on{co}}.$$

\medskip

It follows from the definitions that for $\lambda$ dominant
\begin{equation} \label{e:costandrard to costandard}
\on{Ps-Id}(W^{\lambda,*}_{\on{co}})\simeq W^{\lambda,*};
\end{equation}
(the shift by twice $\langle \lambda,2\check\rho\rangle $ appears since this integer equals 
$\dim(\fL^+(N)^{\omega^\rho}/\on{Stab}_{\fL(N)^{\omega^\rho}}(t^\lambda))$). 

\medskip

Also, as in the proof of \propref{p:Whittaker strata}, it is easy to see that if $\lambda$ is non-dominant,
then $W^{\lambda,*}_{\on{co}}=0$. 

\sssec{}

In \secref{sss:proof of inv and coinv} we will prove:

\begin{thm}  \label{t:inv and coinv}
The functor \eqref{e:functor T} is an equivalence.
\end{thm}

\begin{cor}
The category $\Whit_{q,x}(G)_{\on{co}}$ is compactly generated.
\end{cor}

\ssec{Duality for the Whittaker category}  \label{ss:duality Whit}

In this subsection we will show that the metaplectic Whittaker category is (essentially) self-dual
as a DG category. This is not tautological, as the definition of $\Whit_{q,x}(G)$ involved taking invariants with respect to a group
ind-scheme, and this operation is in general not self-dual. 

\sssec{}

Being compactly generated, the category $\Whit_{q,x}(G)_{\on{co}}$ is dualizable. We will now construct a canonical 
equivalence\footnote{Up to replacing the Artin-Schreier sheaf by its inverse.}
\begin{equation} \label{e:duality for Whit}
(\Whit_{q,x}(G)_{\on{co}})^\vee \simeq \Whit_{q^{-1},x}(G),
\end{equation}
where $q^{-1}$ indicates that we are taking the inverse geometric metaplectic data. 

\sssec{}

In order to construct \eqref{e:duality for Whit}, we need to identify the categories $\Whit_{q^{-1},x}(G)$ and 
$$\on{Funct}(\Whit_{q,x}(G)_{\on{co}},\Vect),$$
where the latter is, by definition, the full subcategory of
$$\on{Funct}(\Shv_{\CG^G}(\Gr^{\omega^\rho}_{G,x}),\Vect)$$
that consists of those functors that map all morphisms \eqref{e:map to Av} to isomorphisms.

\medskip

Verdier duality defines an equivalence
$$(\Shv_{\CG^G}(\Gr^{\omega^\rho}_{G,x}))^\vee\simeq \Shv_{(\CG^G)^{-1}}(\Gr^{\omega^\rho}_{G,x}),$$
i.e., an equivalence
\begin{equation} \label{e:Verdier duality on Gr}
\on{Funct}(\Shv_{\CG^G}(\Gr^{\omega^\rho}_{G,x}),\Vect) \simeq 
\Shv_{(\CG^G)^{-1}}(\Gr^{\omega^\rho}_{G,x}).
\end{equation}

Under this equivalence, precomposition with $\on{Av}_*^{N_k,\chi_N}$ on the left-hand side of \eqref{e:Verdier duality on Gr} goes over to the functor 
$\on{Av}_*^{N_k,-\chi_N}$ on the right-hand side. Thus defines the sought-for equivalence \eqref{e:duality for Whit}.

\sssec{}

Combining with the equivalence \eqref{e:functor T}, we thus obtain an equivalence
\begin{equation} \label{e:self duality for Whit}
(\Whit_{q,x}(G))^\vee \simeq \Whit_{q^{-1},x}(G).
\end{equation}

\medskip

In particular, we obtain an equivalence
\begin{equation} \label{e:self duality for Whit comp}
\left((\Whit_{q,x}(G))^c\right)^{\on{op}}\to (\Whit_{q^{-1},x}(G))^c
\end{equation} 
that we denote by $\BD^{\on{Verdier}}$. 

\medskip

We note that by construction, the equivalences \eqref{e:self duality for Whit} and \eqref{e:self duality for Whit comp} are
\emph{involutive}.

\sssec{}

We claim:

\begin{prop}  \label{p:dual of standard}  
$\BD^{\on{Verdier}}(W^{\lambda,!})\simeq W^{\lambda,*}$.
\end{prop}

By the involutivity of $\BD^{\on{Verdier}}$ we then obtain:

\begin{cor} \label{c:duality irred}
We have $\BD^{\on{Verdier}}(W^{\lambda,*})\simeq W^{\lambda,!}$ and $\BD^{\on{Verdier}}(W^{\lambda,!*})\simeq W^{\lambda,!*}$.
\end{cor} 

\begin{cor} \label{c:duality and t}
A compact object $\CF\in \Whit_{q,x}(G)$ is connective with respect to the t-structure if and only if 
$\BD^{\on{Verdier}}(\CF)\in \Whit_{q^{-1},x}(G)$
is coconnective.
\end{cor} 

\begin{proof}[Proof of \propref{p:dual of standard}]

Taking into account \eqref{e:costandrard to costandard},
we need to show that the functor $\Whit_{q,x}(G)\to \Vect$, given by
\begin{equation} \label{e:Hom from standard}
\CF\mapsto \CMaps_{\Whit_{q,x}(G)}(W^{\lambda,!},\CF),
\end{equation}
identifies with the functor
\begin{equation} \label{e:pairing with costandard}
\CF\mapsto \langle W^{\lambda,*}_{\on{co}},\CF\rangle,
\end{equation}
where $\langle -,-\rangle$ denotes the canonical pairing 
$$\Whit_{q,x}(G) \otimes \Whit_{q^{-1},\on{co}}(G)\to \Vect.$$

\medskip

The functor \eqref{e:pairing with costandard} is by definition
$$\CF\mapsto  \langle \CF,\delta_{t^\lambda,\Gr}\rangle[\langle \lambda,2\check\rho\rangle ],$$
where $\langle -,-\rangle$ now denotes the canonical pairing 
$$\Shv_{\CG^G}(\Gr^{\omega^\rho}_{G,x})\otimes \Shv_{(\CG^G)^{-1}}(\Gr^{\omega^\rho}_{G,x})\to \Vect.$$

However, again by definition, $\langle \CF,\delta_{t^\lambda,\Gr}\rangle$ is given by taking the !-fiber of $\CF$ at $t^\lambda$. Now, by the definition of 
$W^{\lambda,!}$, the expression in \eqref{e:Hom from standard} is also given by taking the !-fiber of $\CF$ at $t^\lambda$, shifted by 
$[\langle \lambda,2\check\rho\rangle ]$.

\end{proof} 

\ssec{The global definition}  \label{ss:Whit glob}

In this subsection we will explore a different way to define the metaplectic Whittaker category, where instead of the
affine Grassmannian we will use a ``global" algebro-geometric object. The advantage of this approach is that it provides
a finite-dimensional model for $\Whit_{q,x}(G)$. 

\sssec{}

In this subsection we take $X$ to be complete. Let $\BunNbox$ be the version of Drinfeld's compactification
introduced in \cite[Sect. 4.1]{Ga9}. Namely, $\BunNbox$ classifies the data of a $G$-bundle $\CP_G$ on $X$ and 
that of injective maps of coherent sheaves
\begin{equation}  \label{e:Plucker maps}
\kappa^{\clambda}:(\omega^{\frac{1}{2}})^{\langle \clambda,2\rho\rangle}\to \CV^{\clambda}_{\CP_G}(\infty \cdot x), \quad \clambda\in \cLambda^+
\end{equation}
(here $\CV^{\clambda}$ denotes the Weyl module of highest weight $\clambda$), such that the maps $\kappa^{\clambda}$ satisfy
the Pl\"ucker relations, i.e., they define a reduction of $\CP_G$ to $B$ at the generic point of $X$. 

\begin{rem}
When the derived group of $G$ is not simply connected, in addition to the Pl\"ucker relations one imposes another closed condition,
restricting the possible defect of the maps \eqref{e:Plucker maps}, see \cite[Sect. 7]{Sch}. However, 
for the purposes of defining the global Whittaker category, the difference is material, as the objects satisfying the 
Whittaker condition will be automatically supported on the closed substack in question.
\end{rem} 

\sssec{}

For $\lambda\in \Lambda$, let
$$(\ol\Bun^{\omega^{\rho}}_N)_{\leq\lambda\cdot x}\overset{\ol\bi{}_\lambda}\hookrightarrow \BunNbox$$
be the closed substack where we require that for every $\clambda\in \cLambda^+$, the corresponding map \eqref{e:Plucker maps}
has a pole of order $\leq \langle \clambda,\lambda \rangle$. 

\medskip

We denote by 
$$(\ol\Bun^{\omega^{\rho}}_N)_{=\lambda\cdot x}\overset{\bj_\lambda}\hookrightarrow (\ol\Bun^{\omega^{\rho}}_N)_{\leq\lambda\cdot x}$$
the open substack, where we require that for for every $\clambda\in \cLambda^+$, the corresponding map \eqref{e:Plucker maps}
has a pole of order equal $\langle \clambda,\lambda \rangle$, \emph{and is non-vanishing} at other points of $X$. 

\medskip

Set
$$\bi_\lambda=\ol\bi_\lambda\circ \bj_\lambda.$$

\medskip

We note that the strata $\ol\Bun^{\omega^{\rho}}_{N,=\lambda'\cdot x}$ with $\lambda'\leq \lambda$ \emph{do not} cover all of 
$(\ol\Bun^{\omega^{\rho}}_N)_{\leq\lambda\cdot x}$.
Namely, they miss all the points for which the maps \eqref{e:Plucker maps} have zeroes on $X-x$. 

\medskip

Note that the stack $\ol\Bun^{\omega^{\rho}}_{N,=0\cdot x}$ identifies with 
$$\Bun_N^{\omega^\rho}:=\Bun_B\underset{\Bun_T}\times \{\omega^\rho\}.$$

\sssec{}  \label{sss:global gerbe G}

According to \secref{sss:gerbe global}, the geometric metaplectic data $\CG^G$ descends to a gerbe,
also denoted $\CG^G$, on $\Bun_G$.  We will denote by the same symbol 
$\CG^G$ the gerbe on $\BunNbox$ equal to the ratio of the pullback of $\CG^G$ along the forgetful map
$$\BunNbox\to \Bun_G$$
and the fiber of $\CP^G$ at the point $\omega^\rho\in \Bun_G$. 

\medskip

Note that wit this definition, the restriction of $\CG^G$ to the locally closed substack
$$\Bun_N^{\omega^\rho}\hookrightarrow \BunNbox$$
is canonically trivialized. 

\sssec{}

Inside $\Shv_{\CG^G}(\BunNbox)$ one singles out a full subcategory, denoted $\Whit_{q,\on{glob}}(G)$, by imposing the
condition of equivariance with respect to a certain unipotent groupoid. We refer the reader to \cite[Sects. 4.4-4.7]{Ga9},
where this equivariance condition is written out in detail. We note that in Remark \ref{r:glob Whit condition} below we will give 
another (but of course equivalent) way to characterize this subcategory. 

\medskip

The embedding
$$\Whit_{q,\on{glob}}(G)\hookrightarrow \Shv_{\CG^G}(\BunNbox)$$
is compatible with the (perverse) t-structure on $\Shv_{\CG^G}(\BunNbox)$, 
and admits a continuous right adjoint, denoted $\on{Av}^{N_{\on{glob}},\chi_N}_*$, see \cite[Corollary 4.7.4]{Ga9}.

\sssec{}

Replacing $\BunNbox$ by $(\ol\Bun^{\omega^{\rho}}_N)_{\leq\lambda\cdot x}$ or $(\ol\Bun^{\omega^{\rho}}_N)_{=\lambda\cdot x}$, one has
the similarly defined full subcategories 
$$\Whit_{q,\on{glob}}(G)_{\leq \lambda}\subset \Shv_{\CG^G}((\ol\Bun^{\omega^{\rho}}_N)_{\leq\lambda\cdot x}) \text{ and } 
\Whit_{q,\on{glob}}(G)_{=\lambda}\subset \Shv_{\CG^G}((\ol\Bun^{\omega^{\rho}}_N)_{=\lambda\cdot x}).$$

\medskip

The adjunctions
$$(\ol\bi{}_\lambda)_!:\Shv_{\CG^G}((\ol\Bun^{\omega^{\rho}}_N)_{\leq\lambda\cdot x}) \rightleftarrows \Shv_{\CG^G}(\BunNbox):\ol\bi{}_\lambda^!$$
and 
$$\bj_\lambda^*:\Shv_{\CG^G}((\ol\Bun^{\omega^{\rho}}_N)_{\leq\lambda\cdot x}) \rightleftarrows \Shv_{\CG^G}((\ol\Bun^{\omega^{\rho}}_N)_{=\lambda\cdot x}):(\bj_\lambda)_*$$
give rise to commutative diagrams
$$
\xymatrix{
\Whit_{q,\on{glob}}(G)_{\leq\lambda} \ar[rr]<2pt>^{(\ol\bi{}_\lambda)_!} \ar[d] &&
\Whit_{q,\on{glob}}(G)_{=\lambda}
\ar[ll]<2pt>^{\ol\bi{}_\lambda^!} \ar[d] \\
\Shv_{\CG^G}((\ol\Bun^{\omega^{\rho}}_N)_{\leq\lambda\cdot x})
\ar[rr]<2pt>^{(\ol\bi{}_\lambda)_!} &&
 \Shv_{\CG^G}(\BunNbox)
\ar[ll]<2pt>^{\ol\bi{}_\lambda^!}}
$$
and
$$
\xymatrix{
\Whit_{q,\on{glob}}(G)_{\leq\lambda} \ar[rr]<2pt>^{\bj_\lambda^*} \ar[d] &&
\Whit_{q,\on{glob}}(G)_{=\lambda}
\ar[ll]<2pt>^{(\bj_\lambda)_*} \ar[d] \\
\Shv_{\CG^G}((\ol\Bun^{\omega^{\rho}}_N)_{\leq\lambda\cdot x})
\ar[rr]<2pt>^{\bj_\lambda^*} &&
\Shv_{\CG^G}((\ol\Bun^{\omega^{\rho}}_N)_{=\lambda\cdot x}).
\ar[ll]<2pt>^{(\bj_\lambda)_*}}
$$

Furthermore, the diagrams
$$
\xymatrix{
\Whit_{q,\on{glob}}(G)_{\leq\lambda} \ar[rr]<2pt>^{(\ol\bi{}_\lambda)_!}  &&
\Whit_{q,\on{glob}}(G)_{=\lambda}
\ar[ll]<2pt>^{\ol\bi{}_\lambda^!} \\
\Shv_{\CG^G}((\ol\Bun^{\omega^{\rho}}_N)_{\leq\lambda\cdot x})
\ar[rr]<2pt>^{(\ol\bi{}_\lambda)_!} \ar[u]^{\on{Av}^{N_{\on{glob}},\chi_N}_*} &&
\Shv_{\CG^G}(\BunNbox)
\ar[ll]<2pt>^{\ol\bi{}_\lambda^!} \ar[u]_{\on{Av}^{N_{\on{glob}},\chi_N}_*}}
$$
and
$$
\xymatrix{
\Whit_{q,\on{glob}}(G)_{\leq\lambda} \ar[rr]<2pt>^{\bj_\lambda^*}  &&
\Whit_{q,\on{glob}}(G)_{=\lambda}
\ar[ll]<2pt>^{(\bj_\lambda)_*}  \\
\Shv_{\CG^G}((\ol\Bun^{\omega^{\rho}}_N)_{\leq\lambda\cdot x})
\ar[rr]<2pt>^{\bj_\lambda^*} \ar[u]^{\on{Av}^{N_{\on{glob}},\chi_N}_*} &&
\Shv_{\CG^G}((\ol\Bun^{\omega^{\rho}}_N)_{=\lambda\cdot x}).
\ar[ll]<2pt>^{(\bj_\lambda)_*}  \ar[u]_{\on{Av}^{N_{\on{glob}},\chi_N}_*}}
$$
commute as well. 

\medskip

The partially defined functor 
$$(\bj_\lambda)_!:\Shv_{\CG^G}((\ol\Bun^{\omega^{\rho}}_N)_{=\lambda\cdot x})\to \Shv_{\CG^G}((\ol\Bun^{\omega^{\rho}}_N)_{\leq\lambda\cdot x}),$$
left adjoint to $\bj_\lambda^*$ is defined on $\Whit_{q,\on{glob}}(G)_{=\lambda}$, and makes the diagram
$$
\CD
\Whit_{q,\on{glob}}(G)_{=\lambda}  @>{(\bj_\lambda)_!}>>  \Whit_{q,\on{glob}}(G)_{\leq\lambda}  \\
@VVV   @VVV    \\
\Shv_{\CG^G}((\ol\Bun^{\omega^{\rho}}_N)_{=\lambda\cdot x})  @>{(\bj_\lambda)_!}>> \Shv_{\CG^G}((\ol\Bun^{\omega^{\rho}}_N)_{\leq\lambda\cdot x})
\endCD
$$
commute. 

\medskip

Finally, we have:

\begin{lem} \label{l:properties global} \hfill

\smallskip

\noindent{\em(a)} The category $\Whit_{q,\on{glob}}(G)_{=\lambda}$ is zero unless $\lambda$ is dominant. 

\smallskip

\noindent{\em(b)} For $\lambda$ dominant, the category $\Whit_{q,\on{glob}}(G)_{=\lambda}$ is 
non-canonically equivalent to $\Vect$; its generator is locally constant (as an object of $\Shv_{\CG^G}((\ol\Bun^{\omega^{\rho}}_N)_{=\lambda\cdot x})$); 

\smallskip

\noindent{\em(c)} Every object of $\Whit_{q,\on{glob}}(G)_{\leq\lambda}$, whose restriction to $(\ol\Bun^{\omega^{\rho}}_N)_{=\lambda\cdot x}$ vanishes, is supported 
on $(\ol\Bun^{\omega^{\rho}}_N)_{<\lambda\cdot x}:=\underset{\lambda'<\lambda}\bigcup\, \ol\Bun^{\omega^{\rho}}_{N,\leq \lambda'\cdot x}$.

\end{lem} 

\begin{cor} \label{c:properties global}
For every $\lambda\in \Lambda^+$ there exists a \emph{quasi-compact} open substack $U\subset (\ol\Bun^{\omega^{\rho}}_N)_{\leq\lambda\cdot x}$, such that every
object of $\Whit_{q,\on{glob}}(G)_{\leq\lambda}$ is a \emph{clean extension} of its restriction to $U$.
\end{cor}

\sssec{} \label{sss:irred global}

For $\lambda$ dominant, pick a generator of $\Whit_{q,\on{glob}}(G)_{=\lambda}$ that is perverse on 
$(\ol\Bun^{\omega^{\rho}}_N)_{=\lambda\cdot x}$, and let $W^{\lambda,!}_{\on{glob}}\in \Whit_{q,\on{glob}}(G)$ 
(resp., $W^{\lambda,*}_{\on{glob}}\in \Whit_{q,\on{glob}}(G)$)
be obtained by applying to it the functor $(\bi_\lambda)_!:=(\ol\bi{}_\lambda)_!\circ (\bj_\lambda)_!$ 
(resp., $(\bi_\lambda)_*:=(\ol\bi{}_\lambda)_!\circ (\bj_\lambda)_*$). 

\medskip

Let also 
$W^{\lambda,!*}_{\on{glob}}$ be the Goresky-MacPherson extension of the above object in $\Whit_{q,\on{glob}}(G)_{=\lambda}$. The objects
$W^{\lambda,!*}_{\on{glob}}$ are the irreducibles in $(W^{\lambda,!}_{\on{glob}})^\heartsuit$. 

\medskip

It follows from \lemref{l:properties global} that the objects $W^{\lambda,!}_{\on{glob}}$ generate $\Whit_{q,\on{glob}}(G)$. Since the open embedding
$(\bj_\lambda)_*$ is affine (see \cite[Sect. Prop. 3.3.1]{FGV}), we have
$$W^{\lambda,!}_{\on{glob}},W^{\lambda,*}_{\on{glob}}\in (\Whit_{q,\on{glob}}(G))^\heartsuit.$$

They are of finite length and compact as objects of $\Shv_{\CG^G}(\BunNbox)$ by 
\corref{l:properties global}. 

\medskip

Furthermore, an object $\CF\in \Whit_{q,\on{glob}}(G)$ is coconnective if and only if
$$\Hom_{\Whit_{q,\on{glob}}(G)}(W^{\lambda,!}_{\on{glob}}[k],\CF)=0$$
for all $\lambda\in \Lambda^+$ and $k>0$. 

\sssec{}

It follows from \corref{l:properties global} that Verdier duality for $\Shv_{\CG^G}(\BunNbox)^{\on{loc.c}}$ 
(see \secref{sss:locally compact}) defines an equivalence\footnote{Up to replacing the 
Artin-Schreier sheaf by its inverse.}
\begin{equation} \label{e:self duality for glob Whit}
(\Whit_{q,\on{glob}}(G))^\vee\simeq \Whit_{q^{-1},\on{glob}}(G).
\end{equation} 

Denote the resulting equivalence 
\begin{equation} \label{e:self duality for glob Whit comp}
\left((\Whit_{q,\on{glob}}(G))^c\right)^{\on{op}}\to (\Whit_{q^{-1},\on{glob}}(G))^c
\end{equation} 
by $\BD^{\on{Verdier}}$. 

\medskip

We have
$$\BD^{\on{Verdier}}(W^{\lambda,!}_{\on{glob}})\simeq W^{\lambda,*}_{\on{glob}} \text{ and }
\BD^{\on{Verdier}}(W^{\lambda,*}_{\on{glob}})\simeq W^{\lambda,!}_{\on{glob}}.$$

\ssec{The local vs global equivalence}  \label{ss:loc to glob Whit}

In this subsection we will state a theorem to the effect that the global Whittaker category 
$\Whit_{q,\on{glob}}(G)$ is equivalent to the local Whittaker category $\Whit_{q,x}(G)$
defined earlier. 

\sssec{}

We have a natural projection
$$\pi_x:\Gr^{\omega^\rho}_{G,x}\to \BunNbox.$$

Note that, according to our conventions, the gerbe on $\BunNbox$ 
that we denoted $\CG^G$ pulls back to the gerbe $\CG^G$ on $\Gr^{\omega^\rho}_{G,x}$.
Consider the corresponding pullback functor
$$\pi_x^!:\Shv_{\CG^G}(\BunNbox)\to \Shv_{\CG^G}(\Gr^{\omega^\rho}_{G,x}).$$

\medskip

According to \cite[Theorem 5.1.4(a)]{Ga9}, the functor 
$\pi_x^!$ sends $\Whit_{q,\on{glob}}(G)$ to
$$\Whit_{q,x}(G)\subset \Shv_{\CG^G}(\Gr^{\omega^\rho}_{G,x}).$$

\medskip

In \secref{ss:proof local to global} we will prove:

\begin{thm}   \label{t:local to global}
The resulting functor 
$$\pi_x^!:\Whit_{q,\on{glob}}(G)\to \Whit_{q,x}(G)$$
is an equivalence. Moreover, after applying the cohomological shift by
$$d_g:=\dim(\Bun_N^{\omega^\rho})=(g-1)(d-\langle 2\check\rho,2\rho\rangle),\quad d=\dim(\fn),$$
it is t-exact and sends standards (resp., costandards) to standards  (resp., costandards).
\end{thm} 

\sssec{}    \label{sss:proof t on Whit}

Some remarks are in order. 

\medskip

First off, it is easy to see from the definitions that $\pi_x^!$, shifted cohomologically by 
$d_g$ sends 
$W^{\mu,*}_{\on{glob}}$ to $W^{\mu,*}$. Since the latter objects
generate $\Whit_{q,x}(G)$, in order to prove that $\pi_x^!$ is an equivalence, it suffices to show that it is fully faithful.
The proof of fully faithfulness will be given in \secref{sss:proof local to global}.

\medskip

Second, we have a tautological map
\begin{equation} \label{e:standard to standard}
W^{\mu,!}\to \pi_x^!(W^{\mu,!}_{\on{glob}})[d_g]
\end{equation} 

If we assume that $\pi_x^!$ is fully faithful, we obtain that the map \eqref{e:standard to standard} induces an isomorphism
on maps into any $W^{\mu',*}$. Since the latter objects \emph{co-generate} $\Whit_{q,x}(G)$, we obtain that \eqref{e:standard to standard}
is an isomorphism. 

\medskip

Since the t-structures on both $\Whit_{q,x}(G)$ and $W^{\mu,*}_{\on{glob}}$ are characterized in terms of the objects $W^{\mu,!}$ 
and $W^{\mu,!}_{\on{glob}}$, respectively, we obtain that the functor 
$$\pi_x^![d_g]$$
is t-exact.

\medskip

This implies the assertion of \propref{p:t on Whit}. 

\sssec{}

Since the morphism $\pi_x$ is ind-schematic, we have a well-defined functor
$$(\pi_x)_*:\Shv_{\CG^G}(\Gr^{\omega^\rho}_{G,x})\to \Shv_{\CG^G}(\BunNbox).$$

Consider the composite
$$\on{Av}^{N_{\on{glob}},\chi_N}_*\circ (\pi_x)_*: \Shv_{\CG^G}(\Gr^{\omega^\rho}_{G,x})\to \Whit_{q,\on{glob}}(G).$$

It is not difficult to show (see \cite[Lemma 5.3.3]{Ga9}) that the above functor factors through a functor
\begin{equation}  \label{e:from coinv to glob}
\Whit_{q,x}(G)_{\on{co}}\to \Whit_{q,\on{glob}}(G);
\end{equation}
moreover, the latter is the functor dual to
$$\pi_x^!:\Whit_{q^{-1},\on{glob}}(G)\to \Whit_{q^{-1},x}(G)$$
in terms of the identifications
$$\Whit_{q,\on{glob}}(G)^\vee \simeq \Whit_{q^{-1},\on{glob}}(G) \text{ and }
\Whit_{q,x}(G)^\vee \simeq \Whit_{q^{-1},\on{co}}(G).$$

\medskip

By a slight abuse of notation, we will denote the functor appearing in \eqref{e:from coinv to glob} by the same 
character $\on{Av}^{N_{\on{glob}},\chi_N}_*\circ (\pi_x)_*$. From \thmref{t:local to global} we obtain:

\begin{cor}  \label{c:local to global}
The functor \eqref{e:from coinv to glob} is an equivalence of categories.
\end{cor}

\sssec{}  \label{sss:proof of inv and coinv}

Consider now the composite functor
\begin{equation} \label{e:composite via global}
\Whit_{q,x}(G)_{\on{co}} \overset{\on{Av}^{N_{\on{glob}},\chi_N}_*\circ (\pi_x)_*}\longrightarrow \Whit_{q,\on{glob}}(G) \overset{\pi_x^!}\longrightarrow \Whit_{q,x}(G).
\end{equation} 

It follows from the construction (see \cite[Corollary 5.4.5]{Ga9}) that the functor \eqref{e:composite via global} identifies with the functor 
$$\on{Ps-Id}\otimes \ell_g,$$
where $\on{Ps-Id}$ is as in \eqref{e:functor T},
where $\ell_g$ is line equal to the !-fiber of the constant sheaf on $\Bun_N^{\omega^\rho}$ (at any $k$-point).

\medskip

Thus, we obtain that \thmref{t:local to global} (combined with \corref{c:local to global}) implies \thmref{t:inv and coinv}.

\begin{rem}  \label{r:local vs global duality}
It follows by unwinding the constructions that the equivalence of \thmref{t:local to global} intertwines the duality equivalences
$$(\Whit_{q,\on{glob}}(G))^\vee\simeq \Whit_{q^{-1},\on{glob}}$$
of \eqref{e:self duality for glob Whit}
and 
$$(\Whit_{q,x}(G))^\vee \simeq \Whit_{q^{-1},x}(G)$$
of \eqref{e:self duality for Whit},
up to tensoring by $\ell_g$. 
\end{rem}

\section{The Ran version and proof of \thmref{t:local to global}}  \label{s:proof local to global}

This section is ostensibly devoted to the proof of \thmref{t:local to global}. However, in the 
process, we will introduce another player--the Ran version of the Whittaker category. 

\medskip

It will play a crucial role in the sequel as it will provide one of the ingredients for the construction
of the \emph{factorization enhancement} of the Jacquet functor. 

\ssec{The Ran version of the semi-infinite orbit}  \label{ss:semi-inf orbit Ran}

In this subsection we will state a general fully-faithfulness result that allows to compare categories of sheaves on $\BunNbox$
and (various versions of) the affine Grassmannian. 

\sssec{}  \label{sss:tw grass}

Consider the Ran Grassmannian $\Gr^{\omega^\rho}_{G,\Ran}$ and its version $\Gr^{\omega^\rho}_{G,\Ran_x}$,
see \secref{sss:Gr with marked}. We let
$$(\ol{S}{}^0_{\Ran_x})_{\infty \cdot x}\subset \Gr^{\omega^\rho}_{G,\Ran_x}$$
be the closed subfunctor given by the following condition:

\medskip

A point $(\CI,\CP_G,\alpha)$ belongs to $(\ol{S}{}^0_{\Ran_x})_{\infty \cdot x}$ if for every dominant weight $\clambda$, the composite 
meromorphic map
\begin{equation} \label{e:Plucker again}
(\omega^{\frac{1}{2}})^{\langle \clambda,2\rho\rangle}\to \CV^{\clambda}_{\CP'_G}\overset{\alpha}\longrightarrow
\CV^{\clambda}_{\CP_G}(\infty \cdot x)
\end{equation} 
is \emph{regular on $X-x$}, where:

\begin{itemize}

\item $\CP'_G$ denotes the $G$-bundle induced from the $T$-bundle $\omega^\rho$;

\item The map $(\omega^{\frac{1}{2}})^{\langle \clambda,2\rho\rangle}\to \CV^{\clambda}_{\CP'_G}$ corresponds to the
highest weight vector in $\CV^\lambda$.

\end{itemize}

\medskip

Note that we have a Cartesian square
$$
\CD
\Gr^{\omega^\rho}_{G,x}  @>>>  (\ol{S}{}^0_{\Ran_x})_{\infty \cdot x}  \\
@VVV   @VVV   \\
\on{pt}  @>>>  \Ran_x,
\endCD
$$
where $\on{pt}\to \Ran_x$ corresponds to the point $\{x\}$.

\sssec{}

Let us denote by $\pi_{\Ran_x}$ the natural forgetful map
$$(\ol{S}{}^0_{\Ran_x})_{\infty \cdot x}\to \BunNbox.$$

\medskip

Note the composite
$$\Gr^{\omega^\rho}_{G,x} \hookrightarrow (\ol{S}{}^0_{\Ran_x})_{\infty \cdot x}  \overset{\pi_{\Ran_x}}\longrightarrow \BunNbox$$
is the map that we have earlier denoted by $\pi_x$.

\sssec{}

Note that the pullback of the gerbe $\CG^G$ on $\BunNbox$ identifies with the gerbe $\CG^G$ on $(\ol{S}{}^0_{\Ran_x})_{\infty \cdot x}$. 
Hence, we have a well-defined functor
\begin{equation} \label{e:glob to loc gen}
\pi_{\Ran_x}^!:\Shv_{\CG^G}(\BunNbox)\to \Shv_{\CG^G}((\ol{S}{}^0_{\Ran_x})_{\infty \cdot x}).
\end{equation} 

We claim:

\begin{thm}  \label{t:N contr}
The functor \eqref{e:glob to loc gen} is fully faithful.
\end{thm}

We omit the proof of this theorem as it repeats verbatim the proof of  \cite[Theorem 3.4.4]{Ga7}. 

\ssec{The Ran version of the metaplectic Whittaker category}  \label{ss:Whit Ran}

In this subsection we will introduce another key player--the Ran version of the Whittaker category.
It will play a technical role in the proof of \thmref{t:local to global}, and also a central role in the construction
of the functor in the main theorem. 

\sssec{}

Recall the group ind-schemes over $\Ran$ denoted 
$$\fL(N)^{\omega^\rho}_\Ran\subset \fL(G)^{\omega^\rho}_\Ran,$$
see Sects. \ref{sss:twisted group scheme}-\ref{sss:twisted group N}. 
Note that as in the case 
of $\fL(N)^{\omega^\rho}_x$ we have a canonically defined homomorphism $\fL(N)^{\omega^\rho}_{\Ran} \to \BA^1$.
We denote by the same character $\chi_N$ the pullback of the Artin-Schreier sheaf $\chi$ to $\fL(N)^{\omega^\rho}_{\Ran}$.

\medskip

As in the case of $\fL(N)^{\omega^\rho}_x$, the restriction of the multiplicative gerbe $\CG^G$ along
$$\fL(N)^{\omega^\rho}_{\Ran}\to \fL(G)^{\omega^\rho}_{\Ran}$$
admits a unique trivialization, normalized by the requirement that it is the tautological one on the unit section.

\sssec{}

Let 
$\fL(N)^{\omega^\rho}_{\Ran_x}$ denote the pullback of $\fL(N)^{\omega^\rho}_{\Ran}$ along the map $\Ran_x\to \Ran$.
Note that $\fL(N)^{\omega^\rho}_{\Ran_x}$ acts on $\Gr^{\omega^\rho}_{G,\Ran_x}$, preserving $(\ol{S}{}^0_{\Ran_x})_{\infty \cdot x}$.
Hence, it makes sense to talk about the categories
$$\Whit_{q,\Ran_x}(G):=\left(\Shv_{\CG^G}(\Gr^{\omega^\rho}_{G,\Ran_x})\right)^{\fL(N)^{\omega^\rho}_{\Ran_x},\chi_N}$$
and
$$\Whit_{q,\Ran_x}(G)^{\leq 0}_{\infty\cdot x}:=
\left(\Shv_{\CG^G}((\ol{S}{}^0_{\Ran_x})_{\infty \cdot x})\right)^{\fL(N)^{\omega^\rho}_{\Ran_x},\chi_N},$$
the latter being a full subcategory of the former consisting of objects that are supported on 
$(\ol{S}{}^0_{\Ran_x})_{\infty \cdot x}\subset \Gr^{\omega^\rho}_{G,\Ran_x}$. 

\sssec{}  

One shows (see \cite[Corollary 6.2.2]{Ga9}): 

\begin{prop}
The pullback functor $\pi_{\Ran_x}^!$ sends 
$$\Whit_{q,\on{glob}}(G)\subset \Shv_{\CG^G}(\BunNbox)$$
to
$$\Whit_{q,\Ran_x}(G)^{\leq 0}_{\infty\cdot x} \subset  \Shv_{\CG^G}((\ol{S}{}^0_{\Ran_x})_{\infty \cdot x}).$$
\end{prop} 

\begin{rem} \label{r:glob Whit condition} 
One can show (see \cite[Theorem 5.1.4(b)]{Ga9}) that 
$$
\CD
\Whit_{q,\Ran_x}(G)^{\leq 0}_{\infty\cdot x}   @>>>  \Shv_{\CG^G}((\ol{S}{}^0_{\Ran_x})_{\infty \cdot x}) \\
@AAA   @AAA  \\
\Whit_{q,\on{glob}}(G)   @>>> \Shv_{\CG^G}(\BunNbox)
\endCD
$$
is a pullback diagram, i.e., full subcategory $\Whit_{q,\on{glob}}(G)\subset \Shv_{\CG^G}(\BunNbox)$ consists exactly of those
objects that satisfy the Whittaker equivariance condition when pulled back to $(\ol{S}{}^0_{\Ran_x})_{\infty \cdot x}$. 
\end{rem} 

\sssec{}

Note now that there is a tautological map
$$\on{unit}:\Ran_x\times \Gr^{\omega^\rho}_{G,x}\to \Gr^{\omega^\rho}_{G,\Ran_x},$$
whose image belongs to $(\ol{S}{}^0_{\Ran_x})_{\infty \cdot x}$: namely, a $G$-bundle trivialized away from $x$ can be thought of as
trivialized away from a finite set of points containing $\{x\}$. 

\medskip

We have the following crucial result, whose proof repeats verbatim the proof of \cite[Theorem 6.2.5]{Ga9}.

\begin{thm} \label{t:restr to unit}  
The functor
$$\on{unit}^!: \Shv_{\CG^G}((\ol{S}{}^0_{\Ran_x})_{\infty \cdot x})\to \Shv_{\CG^G}(\Ran_x\times \Gr^{\omega^\rho}_{G,x})$$
defines an equivalence from $\Whit_{q,\Ran_x}(G)^{\leq 0}_{\infty\cdot x}$ to
$$\left(\Shv_{\CG^G}(\Ran_x\times \Gr^{\omega^\rho}_{G,x})\right)^{\fL(N)^{\omega^\rho}_x,\chi_N}.$$
\end{thm}

We note that in the statement of \thmref{t:restr to unit}, the target category is a version of $\Whit_{q,x}(G)$, where instead of
$\Gr_{G,x}^{\omega^\rho}$ we take $\Ran_x\times \Gr_{G,x}^{\omega^\rho}$, with $\fL(N)^{\omega_\rho}_x$ acting
trivially on $\Ran_x$. 

\sssec{}  \label{sss:spread functor}

In what follows we will denote the functor 
\begin{multline*} 
\Whit_{q,x}(G)=\left(\Shv_{\CG^G}(\Gr^{\omega^\rho}_{G,x})\right)^{\fL(N)^{\omega^\rho}_x,\chi_N}\to
\left(\Shv_{\CG^G}(\Ran_x\times \Gr^{\omega^\rho}_{G,x})\right)^{\fL(N)^{\omega^\rho}_x,\chi_N} \to  \\
\to \left(\Shv_{\CG^G}((\ol{S}{}^0_{\Ran_x})_{\infty \cdot x})\right)^{\fL(N)^{\omega^\rho}_{\Ran_x},\chi_N}
\hookrightarrow \Shv_{\CG^G}((\ol{S}{}^0_{\Ran_x})_{\infty \cdot x})
\hookrightarrow \Shv_{\CG^G}(\Gr^{\omega^\rho}_{G,\Ran_x})
\end{multline*}
by $\on{sprd}_{\Ran_x}$, where in the above formula, the second arrow is !-pullback along the projection 
$$\Ran_x\times \Gr^{\omega^\rho}_{G,x}\to \Gr^{\omega^\rho}_{G,x},$$
and the third arrow is the equivalence inverse to one in \thmref{t:restr to unit}. 

\begin{rem}
The functor $\on{sprd}_{\Ran_x}$ encodes the \emph{unital structure} on $\Whit_{q,x}(G)$, viewed as a factorization category.
\end{rem} 

\ssec{Proof of \thmref{t:local to global}}   \label{ss:proof local to global}

In this subsection we will combine Theorems \ref{t:N contr} and \ref{t:restr to unit} to prove \thmref{t:local to global}. 

\sssec{}  \label{sss:proof local to global}

According to \secref{sss:proof t on Whit}, it suffices to show that the functor
$$(\pi_x)^!:\Whit_{q,\on{glob}}(G)\to  \Whit_{q,x}(G)$$ is fully faithful. We factor the above functor as a composite
$$\Whit_{q,\on{glob}}(G)  \overset{\pi^!_{\Ran_x}}\longrightarrow \Whit_{q,\Ran_x}(G)^{\leq 0}_{\infty\cdot x}  
\overset{\on{unit} ^!} \longrightarrow \left(\Shv_{\CG^G}(\Ran_x\times \Gr^{\omega^\rho}_{G,x})\right)^{\fL(N)^{\omega^\rho}_x,\chi_N} \to
\Whit_{q,x}(G),$$
where the last arrow is restriction along 
$$\{x\}\times \Gr_{G,x}^{\omega^\rho}\to \Ran_x\times \Gr_{G,x}^{\omega^\rho}.$$

\sssec{}

According to \thmref{t:N contr}, the first arrow is fully faithful and, according to \thmref{t:restr to unit}, 
the second arrow is an equivalence. Hence, the functor of
pullback along $\pi_{\Ran_x}\circ \on{unit} $ is fully faithful when restricted to $\Whit_{q,\on{glob}}(G)$. 

\sssec{}

Note, however, that the map $\pi_{\Ran_x}\circ \on{unit} $ factors as 
$$\Ran_x\times \Gr^{\omega^\rho}_{G,x} \to \Gr^{\omega^\rho}_{G,x} \overset{\pi_x}\longrightarrow \BunNbox.$$ 

Hence, $\pi_x^!$, restricted to $\Whit_{q,\on{glob}}(G)$ is a retract of a fully faithful functor, and therefore is itself fully faithful.

\qed[\thmref{t:local to global}]

\ssec{The non-marked case}  \label{ss:non-marked Whit}

For future reference, we will discuss variants of the constructions in Sects. \ref{ss:semi-inf orbit Ran}-\ref{ss:Whit Ran} \emph{without} 
the marked point $x$. 

\sssec{}   \label{sss:non-marked S Ran}

We define the closed subfunctor
$$\ol{S}^0_\Ran \subset \Gr_{G,\Ran}$$
as follows:

\medskip

A point $(\CI,\CP_G,\alpha)$ belongs to $\ol{S}^0_{\Ran}$ if for every dominant weight $\clambda$, the corresponding 
meromorphic map
\begin{equation} \label{e:Plucker again again}
(\omega^{\frac{1}{2}})^{\langle \clambda,2\rho\rangle}\to \CV^{\clambda}_{\CP'_G}\overset{\alpha}\longrightarrow
\CV^{\clambda}_{\CP_G}
\end{equation} 
is regular on all of $X$ (we retain the notations used in \eqref{e:Plucker again again}). 

\sssec{}

We define the categories 
$$\Whit_{q,\Ran}(G):=\Shv_{\CG^G}(\Gr_{G,\Ran})^{\fL(N)^{\omega^\rho}_{\Ran},\chi_N}$$
and 
$$\Whit_{q,\Ran}(G)^{\leq 0}:=\Shv_{\CG^G}(\ol{S}^0_\Ran)^{\fL(N)^{\omega^\rho}_{\Ran},\chi_N}.$$

\sssec{}

We have the tautological section 
$$\on{unit}:\Ran\to \ol{S}^0_\Ran,$$
and a counterpart of \thmref{t:restr to unit} says that the functor
$$\on{unit}^!:\Shv_{\CG^G}(\ol{S}^0_\Ran)\to \Shv(\Ran)$$
induces an equivalence from
\begin{equation} \label{e:Ran vac equiv}
\Whit_{q,\Ran}(G)^{\leq 0}\subset \Shv_{\CG^G}(\ol{S}^0_\Ran)
\end{equation}
to $\Shv(\Ran)$. 

\sssec{}  \label{sss:vacuum}

We let 
$$\on{Vac}_{\Whit,\Ran}\in \Whit_{q,\Ran}(G)^{\leq 0}$$
denote the object equal to the image of $\omega_{\Ran}\in \Shv(\Ran)$ under the equivalence inverse to
the above equivalence
$$\Whit_{q,\Ran}(G)^{\leq 0}\to \Whit_{q,\Ran}(G)^{\leq 0}.$$

\medskip

Sometimes, by a slight abuse of notation, we will regard $\on{Vac}_{\Whit,\Ran}$ as an object just of 
$\Shv_{\CG^G}(\ol{S}^0_\Ran)$ or $\Shv_{\CG^G}(\Gr^{\omega^\rho}_{G,\Ran})$. 

\medskip

Note that the restriction 
$$\on{Vac}_{\Whit,x}:=\on{Vac}_{\Whit,\Ran}|_{\Gr^{\omega^\rho}_{G,x}}\in \Whit_{q,x}(G)$$
identifies with the object
$$W^{0,*}\simeq W^{0,!*}\simeq W^{0,!}.$$

\sssec{}

The following results from the equivalences \eqref{e:Ran vac equiv} and \thmref{t:restr to unit}:

\begin{thm} \label{t:fact on Vac Whit}  \hfill 

\smallskip

\noindent{\em(a)} The object $\on{Vac}_{\Whit,\Ran}$ has a structure of factorization algebra in 
$\Shv_{\CG^G}(\Gr^{\omega^\rho}_{G,\Ran})$, uniquely characterized by the requirement that 
the induced factorization algebra structure on
$$\on{unit}^!(\on{Vac}_{\Whit,\Ran})\in \Shv(\Ran)$$
corresponds to the tautological one on $\omega_\Ran$ (see \secref{sss:dualizing as an example of FA})
with respect to the identification 
$$\on{unit}^!(\on{Vac}_{\Whit,\Ran})\simeq \omega_\Ran.$$

\smallskip

\noindent{\em(b)} The functor
$$\on{sprd}_{\Ran_x}: \Whit_{q,x}(G)\to \Shv_{\CG^G}(\Gr^{\omega^\rho}_{G,\Ran_x})$$
lifts to a functor
$$\on{sprd}_{\on{Fact}}: \Whit_{q,x}(G)\to \on{Vac}_{\Whit,\Ran}\on{-FactMod}(\Shv_{\CG^G}(\Gr^{\omega^\rho}_{G,\Ran_x})),$$
uniquely characterized by the requirement that the composite functor
$$\on{unit}^!\circ \on{sprd}_{\on{Fact}}:  \Whit_{q,x}(G)\to 
\omega_{\Ran}\on{-FactMod}(\Shv_{\CG^G}(\Ran_x\times \Gr^{\omega^\rho}_{G,x}))$$
identifies with the composite
$$\Whit_{q,x}(G)\hookrightarrow \Shv_{\CG^G}(\Gr^{\omega^\rho}_{G,x}))\to 
\omega_{\Ran}\on{-FactMod}(\Shv_{\CG^G}(\Ran_x\times \Gr^{\omega^\rho}_{G,x})),$$
where the second arrow is the functor of \secref{sss:omega modules}. 
\end{thm} 

\newpage 

\centerline{\bf Part III: Hecke action and Hecke eigen-objects}

\bigskip 

The main theorem of this paper describes not the category $\Whit_{q,x}(G)$ itself, but rather its \emph{de-equivariantization}
with respect to the Hecke action. In this Part we will introduce and study the resulting category, denoted $\bHecke(\Whit_{q,x}(G))$. 

\section{Hecke action on the metaplectic Whittaker category}  \label{s:Hecke Whit}

In this section we study the Hecke action of the category of representations of the group $H$ (see \secref{sss:roots in dual}) 
on $\Whit_{q,x}(G)$. This is a structure needed to define $\bHecke(\Whit_{q,x}(G))$. 

\ssec{Definition of the Hecke action on the Whittaker category}  \label{ss:Hecke action on Whit}

In this subsection we define the Hecke action of $\Rep(H)$ on $\Whit_{q,x}(G)$. 

\sssec{}

Note that $$\Sph_q(G):=\left(\Shv_{\CG^G}(\Gr_{G,x})\right)^{\fL^+(G)_x}$$ identifies 
\emph{canonically} with
$$\left(\Shv_{\CG^G}(\Gr^{\omega^\rho}_{G,x})\right)^{\fL^+(G)^{\omega^\rho}_x}.$$

\medskip

Therefore, we obtain a (right) action of 
$\Sph_q(G)$ on $\Shv_{\CG^G}(\Gr^{\omega^\rho}_{G,x})$ by convolutions (on the right), which preserves the action of $\fL(G)^{\omega^\rho}$
on $\Shv_{\CG^G}(\Gr^{\omega^\rho}_{G,x})$ by left translations. 

\bigskip

\noindent{\bf Convention:} {\it For the duration of this section, in order to unburden the notation, we will omit the superscript $\omega^\rho$
and the subscript $x$, so $\Gr^{\omega^\rho}_{G,x}$ will be denoted simply by $\Gr_G$.}

\sssec{}

The action of $\Sph_q(G)$ on $\Shv_{\CG^G}(\Gr_G)$ commutes with the functors 
$\on{Av}_*^{N_k,\chi_N}$. Thus, we obtain that the action of $\Sph_q(G)$ on $\Shv_{\CG^G}(\Gr_G)$
preserves the full subcategory 
$$\Whit_q(G)\subset \Shv_{\CG^G}(\Gr_G).$$

Thus, we obtain a monoidal right action of $\Sph_q(G)$ on $\Whit_q(G)$.

\sssec{}

Combining with metaplectic geometric Satake (see \secref{sss:trivialize gerbe at point}), 
we obtain a monoidal action of $\Rep(H)$ on $\Whit_q(G)$. We refer to it as the \emph{Hecke action},
and will denote it by
$$\CF,V\mapsto \CF\star \Sat_{q,G}(V).$$



\ssec{Hecke action and duality}   \label{ss:Hecke and duality}

In this subsection we will study the interaction of the Hecke action and self-duality of $\Whit_q(G)$.

\sssec{}

Since the action of $\Sph_q(G)$ on on $\Shv_{\CG^G}(\Gr_G)$ commutes with the functors 
$\on{Av}_*^{N_k,\chi_N}$, we obtain a canonically defined action of $\Sph_q(G)$ also on the
category $\Whit_q(G)_{\on{co}}$.

\sssec{}

By construction, the functor 
$$\on{Ps-Id}:\Whit_q(G)_{\on{co}}\to \Whit_q(G)$$
intertwines the actions of $\Sph_q(G)$ on both sides.

\medskip

In particular, we obtain that $\Whit_q(G)_{\on{co}}$ is a $\Rep(H)$-module category, and the functor 
$\on{Ps-Id}$ is a map of such. 

\sssec{}

Since the convolution action of $\Sph_q(G)$ on $\Shv_{\CG^G}(\Gr_G)$ is given by a \emph{proper pushforward}, we have a commutative diagram
of actions
$$
\CD
(\Sph_q(G)^c)^{\on{op}} \otimes ((\Whit_q(G))^c)^{\on{op}}  @>>>   ((\Whit_q(G))^c)^{\on{op}}   \\
@V{\BD^{\on{Verdier}}\otimes \BD^{\on{Verdier}}}VV   @VV{\BD^{\on{Verdier}}}V   \\
\Sph_{q^{-1}}(G)^c \otimes (\Whit_{q^{-1},\on{co}}(G))^c  @>>>  (\Whit_{q^{-1},\on{co}}(G))^c,
\endCD
$$
where the equivalence 
$$\BD^{\on{Verdier}}:((\Whit_q(G))^c)^{\on{op}}\simeq (\Whit_{q^{-1},\on{co}}(G))^c$$
is that induced by \eqref{e:duality for Whit}. 

\sssec{}  \label{sss:Hecke action and duality}

Combining with \secref{ss:Satake and duality}, we obtain a commutative diagram of actions
\begin{equation} \label{e:metapl Satake and duality bis}
\CD
(\Rep(H)^c)^{\on{op}}  \otimes ((\Whit_q(G))^c)^{\on{op}}  @>>>   ((\Whit_q(G))^c)^{\on{op}}   \\
@V{(\tau^H\circ \BD^{\on{lin}})\otimes \BD^{\on{Verdier}}}VV   @VVV   \\
\Rep(H)^c \otimes (\Whit_{q^{-1}}(G))^c  @>>>  (\Whit_{q^{-1}}(G))^c. 
\endCD
\end{equation}
where we remind that $\tau^H$ denoted the Cartan involution on $H$, see \secref{sss:Cartan inv}. 
In the above diagram, we denoted by 
$$\BD^{\on{Verdier}}:((\Whit_q(G))^c)^{\on{op}}\simeq (\Whit_{q^{-1}}(G))^c$$
the identification of \eqref{e:self duality for Whit comp}.

\ssec{Hecke action and the t-structure}

In this subsection we will study how the Hecke action interacts with the t-structure on $\Whit_q(G)$,
which was introduced in \secref{ss:t on Whit}. 

\sssec{}

We claim:

\begin{prop} \label{p:conv is exact}
For $V\in \Rep(H)^\heartsuit$, the corresponding Hecke functor
\begin{equation} \label{e:Hecke V}
\CF\mapsto \CF\star \Sat_{q,G}(V)
\end{equation}
is t-exact.
\end{prop} 

Before we prove this proposition, having future needs in mind, we will perform a certain elementary 
but crucial calculation.

\sssec{Stalks of the convolution, 1-st approximation} \label{sss:1st approx}

In this subsection we will give an explicit expression for the cohomology of the !-fiber at $t^\mu$ of 
$W^{\lambda,*}\star \Sat_{q,G}(V)$ for $V\in \Rep(H)$. 

\medskip

First off, we note that the fiber in question vanishes unless $\mu\in \Lambda^+$ (by \propref{p:Whittaker strata}(a)),
so henceforth we will assume that $\mu$ is dominant.  

\medskip

Consider the ind-subscheme 
$$S^{\mu-\lambda}=\fL(N)\cdot t^{\mu-\lambda}\subset \Gr_G.$$

\medskip

Let $\chi^\lambda_N$ be the character sheaf on $\fL(N)$ obtained from $\chi_N$ by pullback with respect to the automorphism
$\on{Ad}_{t^\lambda}$. Since $\mu$ was assumed dominant, $\chi^\lambda_N$ descends to a well-defined object of
$\Shv(S^{\mu-\lambda})$, which we denote by the same character $\chi^\lambda_N$.

\medskip

Since $\fL(N)$ is ind-pro-unipotent, the restriction of $\CG^G$ to $\fL(N)$ is canonically trivialized, and the restriction of
$\CG^G$ to $S^{\mu-\lambda}$ admits a unique (up to a non-canonical isomorphism) $\fL(N)$-equivariant trivialization. 
Due to this trivialization, we can regard $\Sat_{q,G}(V)|_{S^{\mu-\lambda}}$ as an object of the
non-twisted category $\Shv(S^{\mu-\lambda})$. 

\medskip

By unwinding the definitions, we obtain:
\begin{equation} \label{e:!-stalk of conv}
(W^{\lambda,*}\star \Sat_{q,G}(V))_{t^\mu}\simeq H(S^{\mu-\lambda},\Sat_{q,G}(V)|_{S^{\mu-\lambda}}
\otimes \chi^\lambda_N)[\langle \lambda, 2\check\rho\rangle ].
\end{equation} 

\sssec{Stalks of the convolution, 2nd approximation}  \label{sss:2nd approx}

We now claim that the expression in \eqref{e:!-stalk of conv} lives in the cohomological degrees $\geq -\langle \mu,2\check\rho\rangle $. I.e., we claim 
that 
\begin{equation} \label{e:cohomology of restr}
H(S^{\mu-\lambda},\Sat_{q,G}(V)|_{S^{\mu-\lambda}}\otimes \chi^\lambda_N)
\end{equation} 
lives in cohomological degrees $\geq \langle \lambda-\mu, 2\check\rho\rangle$. 

\medskip

We stratify $S^{\mu-\lambda}$ by the intersections $S^{\mu-\lambda}\cap \Gr_G^\nu$ with $\nu\in \Lambda^+$. So, it sufficient to show
that each 
\begin{equation} \label{e:cohomology of intersect}
H(S^{\mu-\lambda}\cap \Gr^\nu_G,\Sat_{q,G}(V)|_{S^{\mu-\lambda}\cap \Gr^\nu_G}\otimes \chi^\lambda_N)
\end{equation} 
lives in cohomological degrees $\geq \langle \lambda-\mu, 2\check\rho\rangle$. 

\medskip

By perversity, $\Sat_{q,G}(V)|_{\Gr^\nu_G}$ lives in non-negative \emph{perverse} cohomological degrees, and it is lisse due to 
$\fL^+(G)$-equivariance. Hence, it is the Verdier dual of an object that lives in the \emph{usual}
cohomological degrees $\leq -\langle \nu,2\check\rho\rangle$ (we recall that $\langle \nu,2\check\rho\rangle=\dim(\Gr^\nu_G)$). 

\medskip

Therefore,
$\Sat_{q,G}(V)|_{S^{\mu-\lambda}\cap \Gr^\nu_G}$ is the Verdier dual of an object that lives in the \emph{usual}
cohomological degrees $\leq -\langle \nu,2\check\rho\rangle$. 

\medskip

Now the required cohomological estimate follows from the fact that
$$\dim(S^{\mu-\lambda}\cap \Gr^\nu_G)\leq \langle \nu+\mu-\lambda,\check\rho\rangle.$$ 

\sssec{Proof of \propref{p:conv is exact}}

It suffices to consider the case when $V$ is finite-dimensional. Note that both the left and right adjoints of the functor
\eqref{e:Hecke V} identify with 
$$\CF\mapsto \CF\star \Sat_{q,G}(V^*),$$
where $V^*$ is the dual representation of $V$. So, it suffices to show that \eqref{e:Hecke V} is left t-exact. 

\medskip

By the definition of the t-structure on $\Whit_q(G)$, its subcategory of connective objects is generated 
under colimits by the objects $W^{\mu,!}$. Taking into account \eqref{e:standards Whit}, we obtain that
the subcategory of coconnective objects in $\Whit_q(G)$ is \emph{co-generated} under limits by the
objects $W^{\lambda,*}$. 

\medskip

Thus, it suffices to show that for every $\lambda$ and $\mu$, the object 
$$\CHom_{\Whit_q(G)}(W^{\mu,!},W^{\lambda,*}\star \Sat_{q,G}(V))\in \Vect$$ lives
in cohomological degrees $\geq 0$. In other words, we have to show that that 
the !-fiber at $t^\mu$ of $W^{\lambda,*}\star \Sat_{q,G}(V)$ lives in the cohomological
degrees $\geq -\langle \mu,2\check\rho\rangle $. However, the latter has been established in \secref{sss:2nd approx}.

\qed[\propref{p:conv is exact}]

\ssec{Restricted coweights}  \label{ss:restr}

In this subsection we will make the analysis of the action of $\Rep(H)$ on $\Whit_q(G)$ even more explicit. 

\medskip

Namely,
we will show that for certain coweights $\lambda$ (called \emph{restricted}), the image of the corresponding $W^{\lambda,!*}$
under $-\star \Sat_{q,G}(V)$ for $V\in \on{Irrep}(H)$ stays irreducible. 

\medskip

This is a counterpart of Steinberg's theorem in
the context of quantum groups.

\sssec{}  \label{sss:restricted coweights}

We shall say that a coweight $\mu\in \Lambda^+$ is \emph{restricted} if for every vertex $i$ of the Dynkin diagram
we have
$$\langle \mu,\check\alpha \rangle < \on{ord}(q_i),$$
where $q_i$ is as in \secref{sss:q alpha}. 

\medskip

First, we note: 

\begin{lem}  \label{l:many restr}
Assume that the derived group of $H$ is simply connected. Then any element $\lambda\in \Lambda^+$
can be written as $\mu+\gamma$ with $\mu\in \Lambda^+$ restricted and $\gamma\in (\Lambda^\sharp)^+$.
\end{lem} 

\sssec{}

We are now ready to state the main result of this section: 

\begin{thm} \label{t:restr convolution}
Suppose that $\mu\in \Lambda^+$ is restricted. Then for an irreducible $V^\gamma\in \Rep(\cG)$ with highest weight
$\gamma\in (\Lambda^\sharp)^+$, we have
$$W^{\mu,!*}\star \Sat_{q,G}(V^\gamma) \simeq W^{\mu+\gamma,!*}.$$
\end{thm} 

This theorem was proved in \cite[Sect. 7]{Lys}. We include the proof for completeness.

\sssec{Proof of \thmref{t:restr convolution}, Step 0}

It is easy to see that $W^{\mu,!*}\star \Sat_{q,G}(V^\gamma)$ is supported on $\ol{S}^{\mu+\gamma}$ and 
$$W^{\mu,!*}\star \Sat_{q,G}(V^\gamma)|_{S^{\mu+\gamma}}\simeq W^{\mu+\gamma,!*}|_{S^{\mu+\gamma}}.$$

In particular, we have the maps
\begin{equation} \label{e:map to an from conv}
W^{\mu+\gamma,!}\to W^{\mu,!*}\star \Sat_{q,G}(V^\gamma) \text{ and } 
W^{\mu,!*}\star \Sat_{q,G}(V^\gamma)\to W^{\mu+\gamma,*},
\end{equation}
whose composition (is a non-zero scalar multiple of) the canonical map $W^{\mu+\gamma,!}\to W^{\mu+\gamma,*}$.

\medskip

Thus, we have to show that the maps in \eqref{e:map to an from conv} are surjective and injective, respectively. 
We will show the former, as the latter would follow by duality (see \secref{ss:Hecke and duality}). 

\medskip

The surjectivity of the map $W^{\mu+\gamma,!}\to W^{\mu,!*}\star \Sat_{q,G}(V^\gamma)$ is equivalent 
to the assertion that there are no non-zero maps
$$W^{\mu,!*}\star \Sat_{q,G}(V^\gamma)\to W^{\lambda,!*}, \quad \lambda\neq \mu+\gamma.$$
 
Using the t-exactness of the convolution, it suffices to show that if there exists a nonzero Hom
$$W^{\mu,!}\star \Sat_{q,G}(V^\gamma)\to W^{\lambda,*}.$$
then $\lambda=\mu+\gamma$. 

\sssec{Proof of \thmref{t:restr convolution}, Step 1}

By adjunction, we have to show that if there exists a nonzero Hom
$$W^{\mu,!}\to W^{\lambda,*}\star \Sat_{q,G}(V^\gamma),$$
then $\lambda-\mu=-w_0(\gamma)$.
 
\medskip

I.e., we have to show that \emph{if} the expression 
$$(W^{\lambda,*}\star \Sat_{q,G}(V^\gamma))_{t^\mu}$$ has
cohomology in degree $-\langle \mu,2\check\rho\rangle$, \emph{then} $\mu-\lambda=w_0(\gamma)$. 

\medskip

By \secref{sss:2nd approx}, we need to analyze when 
\begin{equation} \label{e:cohomology of intersect again}
H^{\langle \lambda-\mu, 2\check\rho\rangle}(S^{\mu-\lambda}\cap 
\Gr^\gamma_G,\Sat_{q,G}(V^\gamma)|_{S^{\mu-\lambda}\cap \Gr^\gamma_G}\otimes \chi^\lambda_N)
\end{equation} 
is non-zero (we note that the strata with $\nu\neq \gamma$ do not
contribute to this cohomology as $\Sat_{q,G}(V^\gamma)|_{\Gr^\nu}$ will sit in strictly positive perverse cohomological degrees). 

\medskip

Note also that the condition that $\mu-\lambda=w_0(\gamma)$ implies that 
$w_0(\gamma)$ is the smallest element of $\Lambda$, for which the intersection 
$S^{\mu-\lambda}\cap \Gr^\gamma_G$ is non-empty, and in the latter case it consists of one point. 

\sssec{Stalks of the convolution, bottom cohomology}

Note that for $\gamma\in \Lambda^\sharp$, the gerbe $\CG^G|_{\Gr^\gamma_G}$
admits a unique (up to a non-canonical isomorphism) $\fL^+(G)$-equivariant
trivialization. 

\medskip

Thus, we obtain that over the intersection
$$S^{\mu-\lambda}\cap \Gr^\gamma_G,$$
the gerbe $\CG^G$ admits two \emph{different} trivializations. Hence, their ratio is given by a local system that we
temporarily denote by $\Psi_q$. 

\medskip

Thus, the expression in \eqref{e:cohomology of intersect again} is non-zero if and only if the local system 
$$\chi_N^\lambda \otimes \Psi_q$$ 
on $S^{\mu-\lambda}\cap \Gr^\gamma_G$ 
is \emph{trivial} on some irreducible component of  of (the top) dimension
$\langle \mu-\lambda+\gamma,\check\rho\rangle$. 

\sssec{Proof of \thmref{t:restr convolution}, Step 2: Reduction to an intersection of semi-infinite orbits}  \label{sss:Psiq}

Let $S^{-,w_0(\gamma)}$ denote the $\fL(N^-)$-orbit of $t^{w_0(\gamma)}$. It is known (see, e.g., \cite[Sect. 6]{BFGM}) that the inclusions
$$S^{\mu-\lambda}\cap \Gr^\gamma_G \hookleftarrow S^{\mu-\lambda}\cap \fL^+(N^-)\cdot t^{w_0(\gamma)} \hookrightarrow 
S^{\mu-\lambda}\cap S^{-,w_0(\gamma)}$$
induce bijections on the sets of irreducible components of (the top) dimension
$\langle \mu-\lambda+\gamma,\check\rho\rangle$. 

\medskip

The restriction of $\CG^G$ to $S^{-,w_0(\gamma)}$ also acquires a non-canonical trivialization. Hence, the discrepancy between
the two trivializations over $S^{\mu-\lambda}\cap S^{-,w_0(\gamma)}$ is given by a local system that we also temporarily
denote by $\Psi_q$. Its further restriction to
$$S^{\mu-\lambda}\cap \fL^+(N^-)\cdot t^{w_0(\gamma)}$$ identifies (non-canonically) with the restriction of the local system
on $S^{\mu-\lambda}\cap \Gr^\gamma_G$ that we had earlier denoted by $\Psi_q$.

\medskip

Thus, it suffices to show that if $\mu$ is restricted and 
$\mu-\lambda\neq w_0(\gamma)$, then the resulting local system
$$\chi_N^\lambda \otimes \Psi_q$$ 
on $S^{\mu-\lambda}\cap S^{-,w_0(\gamma)}$ 
is \emph{non-trivial} on every irreducible component of (the top) dimension
$\langle \mu-\lambda+\gamma,\check\rho\rangle$. 
 
\medskip

Translating by $t^\lambda$, we obtain that the required statement follows from the next result, proved in \cite[Sect. 6]{Lys}: 

\begin{thm} \label{t:restr}
For a restricted dominant coweight $\mu$ and any $\nu\neq \mu$, the local system
$$\chi_N^0 \otimes \Psi_q$$ 
on $S^\mu\cap S^{-,\nu}$ 
is non-trivial on every irreducible component of (the top) dimension $\langle \mu-\nu,\check\rho\rangle$.
\end{thm} 

\qed[\thmref{t:restr convolution}]

\ssec{Proof of \thmref{t:restr}}  

For the sake of completeness, we will now reproduce a sketch of the proof of 
\thmref{t:restr}. 

\medskip

We note that, in addition to the proof of \thmref{t:restr convolution}, we will use
\thmref{t:restr} one more time, for the analysis of the Hecke action on $\Shv(\Gr_G)^I$, where
$I$ is the Iwahori subgroup of $\fL^+(G)$. 

\sssec{}

Consider the action of the torus $T\subset \fL^+(T)$ on $\Gr_G$. Since it stabilizes the points $t^\mu$
and normalizes $\fL(N)$, it acts on each $S^\mu$ and $S^{-,\nu}$. By \cite[Sect. 7.4.2]{GLys}, the 
$\fL(N)$-equivariant trivialization of $\CG^G|_{S^\eta}$ is $T$-twisted equivariant against the Kummer
local system on $T$ corresponding to the character 
$$b(\mu,-):\Lambda\to \sfe^*(-1).$$

\medskip

Similarly, the 
$\fL(N^-)$-equivariant trivialization of $\CG^G|_{S^{-,\nu}}$ is $T$-twisted equivariant against the Kummer
local system on $T$ corresponding to the character 
$$b(\nu,-):\Lambda\to \sfe^*(-1).$$

\medskip

Hence, the local system $\Psi_q$ on $S^\mu\cap S^{-,\nu}$ is $T$-twisted equivariant against the Kummer
local system on $T$ corresponding to the character 
$$b(\mu-\nu,-):\Lambda\to \sfe^*(-1).$$

\sssec{}

Note now that the local system $\chi_N^0$ on $S^\mu$ is the pullback of the Artin-Schreier sheaf along a map
$S^\mu \to \BG_a$ that is $\BG_m$-equivariant, where $\BG_m$ acts on $S^\mu$ via the cocharacter $\rho$. 

\medskip

In particular, the push-forward of $\Psi_q$ along the resulting map $S^\mu\cap S^{-,\nu}\to \BG_a$
is twisted $\BG_m$-equivariant against the Kummer local system on $\BG_m$ corresponding to
$$b(\mu-\nu,\rho)\in  \sfe^*(-1).$$

\medskip

In particular, we obtain that the local system $\chi_N^0 \otimes \Psi_q$ can be trivial on a given
irreducible component of $S^\mu\cap S^{-,\nu}$  of (the top) dimension $\langle \mu-\nu,\check\rho\rangle$ only if both 
$\chi_N^0$ and $\Psi_q$ are trivial on that component.

\medskip

In particular, this can only happen if $\mu-\nu\in \Lambda^\sharp$. 

\sssec{}

We now recall that the union of sets of irreducible components of $S^\lambda\cap S^{-,0}$ over $\lambda\in \Lambda$
has a structure of crystal. 

\medskip

For a vertex $i$ of the Dynkin diagram, let $\phi_i$ be the corresponding function (measuring the power of the 
lowering operator needed to kill the given element). On the one hand, it is shown in \cite[Sect. 6]{Lys} 

\medskip

On the one hand, it is shown in \cite[Prop. 6.1.7]{Lys} that if $K$ is an irreducible component on which $\Phi_q$ is \emph{trivial},
we have $$\phi_i(K)\in \BZ^{\geq 0}\cdot \on{ord}(q_i).$$
 
\begin{rem}
In fact, it follows from \cite[Sect. 6]{Lys} that the set of such irreducible components 
also has a structure of crystal, where instead of $e_i$ and $f_i$ operators we take their
$\on{ord}(q_i)$-powers. 
\end{rem}

\sssec{}

On the other hand, it is known (see \cite[Sect. 7.3]{FGV}) that under the bijection
$$S^\mu\cap S^{-,\nu}\simeq S^{\mu-\nu} \cap S^{-,0}$$
given by the action of $t^{-\nu}$, the set of irreducible components of $S^\mu\cap S^{-,\nu}$ on which
$\chi^0_N$ is trivial corresponds to the subset of irreducible components $K$ of $S^{\mu-\nu} \cap S^{-,0}$
for which
$$\phi_i(K) \leq \check\alpha_i(\mu),\,\,\forall\,\, i.$$

\sssec{}

Combining, we obtain that for an irreducible component of $S^\mu\cap S^{-,\nu}$, denoted $K$,
on which both $\Psi_q$ and $\chi^0_N$ are trivial, we have:
$$\phi_i(K)\in \BZ^{\geq 0}\cdot \on{ord}(q_i) \text{ and } \phi_i(K) \leq \check\alpha_i(\mu)<  \on{ord}(q_i),$$
where the latter inequality is due to the fact that $\mu$ is restricted. 

\medskip

Hence, we obtain that $\phi_i(K)=0$ for all $i$, which forces $\mu-\nu=0$. 

\qed[\thmref{t:restr}]

\section{Hecke eigen-objects}  \label{s:Hecke}

In this section we will study the general paradigm of forming Hecke categories: given a category $\bC$ with an action
of $\Rep(H)$, we will define a new category $\bHecke(\bC)$ and study its properties. 

\ssec{Tensor products over $\Rep(H)$: a reminder}   \label{ss:ten prod Rep(H)}

Recall that if $\bC$ and $\bD$ are DG categories that are right and left modules, respectively, over a monoidal
DG category $\bA$, we can form the tensor product
$$\bC\underset{\bA}\otimes \bD,$$
which is another DG category.

\medskip

In this subsection we discuss some general features of this operation, when $\bA$ is the category $\Rep(H)$ of representations
of an algebraic group $H$. 

\sssec{}

Let $H$ be an algebraic group, and $\bC$ and $\bD$ categories with an action of the monoidal category $\Rep(H)$.

\medskip

Consider the tensor product 
\begin{equation} \label{e:ten prod over Rep}
\bC\underset{\Rep(H)}\otimes \bD.
\end{equation}

By definition, the category \eqref{e:ten prod over Rep} comes equipped with a functor
$$\Phi_{\on{univ}}:\bC\otimes \bD \to \bC\underset{\Rep(H)}\otimes \bD,$$
universal among functors
$$\Phi:\bC\otimes \bD \to \bE$$
equipped with functorial isomorphisms
$$\Phi(\bc\star V,\bd)\overset{\alpha_V}\simeq \Phi(\bc, V\star \bd), \quad \bc\in \bC,\bd\in \bD,V\in \Rep(H),$$
compatible with associativity in the sense that for $V_1,V_2\in \Rep(H)$, the diagrams
$$
\CD
\Phi(\bc\star (V_1\otimes V_2), \bd)  @>{\alpha_{V_1\otimes V_2}}>>   \Phi(\bc,(V_1\otimes V_2)\star \bd)  \\
@V{\sim}VV    @VV{\sim}V  \\
\Phi((\bc\star V_1)\star V_2,\bd)   & & \Phi(\bc,V_1\star (V_2\star \bd))   \\
@V{\alpha_{V_2}}VV     @AA{\alpha_{V_1}}A  \\
\Phi(\bc\star V_1,V_2\star \bd)  @>{\on{id}}>>  \Phi(\bc\star V_1,V_2\star \bd) 
\endCD
$$
should commute, along with a homotopy-coherent system of compatibilties for multi-fold tensor products. 

\sssec{}  \label{sss:coten}

According to \cite[Chapter1, Proposition 9.4.8]{GR1}, the functor $\Phi_{\on{univ}}$ admits a (continuous) conservative right adjoint, denoted
$$\Psi_{\on{univ}}: \bC\underset{\Rep(H)}\otimes \bD\to \bC\otimes \bD$$
that realizes $\bC\underset{\Rep(H)}\otimes \bD$ as the category consisting of objects $\be\in  \bC\otimes \bD$,
equipped with a system of isomorphisms
$$((\on{Id}\star V)\otimes \on{Id})(\be) \overset{\beta_V}\simeq (\on{Id}\otimes (V\star \on{Id})) (\be) ,\quad V\in \Rep(H)$$
compatible with associativity in the sense that for $V_1,V_2\in \Rep(H)$, the diagrams 
$$
\CD
((\on{Id}\star  (V_1\otimes V_2))\otimes \on{Id})(\be)  @>{\beta_{V_1\otimes V_2}}>>  (\on{Id}\otimes ((V_1\otimes V_2)\star \on{Id})) (\be)  \\
@V{\sim}VV   @VV{\sim}V    \\
((\on{Id}\star  V_2)\otimes \on{Id}) \circ ((\on{Id}\star  V_1)\otimes \on{Id})(\be)  & &  (\on{Id}\otimes (V_1\star \on{Id})) \circ (\on{Id}\otimes (V_2\star \on{Id})) (\be)  \\
@V{\beta_{V_1}}VV   @AA{\beta_{V_2}}A   \\
((\on{Id}\star  V_2)\otimes \on{Id}) \circ  (\on{Id}\otimes (V_1\star \on{Id})) (\be)   
@>{\sim}>> (\on{Id}\otimes (V_1\star \on{Id}))\circ ((\on{Id}\star  V_2)\otimes \on{Id}) \ (\be) 
\endCD
$$
should commute, along with a homotopy-coherent system of compatibilties for multi-fold tensor products. 

\sssec{}  \label{sss:ind Hecke}

Thus, we obtain that $\bC\otimes \bD$ and $\bC\underset{\Rep(H)}\otimes \bD$ are related by a pair of (continuous) adjoint functors 
\begin{equation} \label{e:adj funct ten prod}
\Phi_{\on{univ}}:\bC\otimes \bD \rightleftarrows \bC\underset{\Rep(H)}\otimes \bD:\Psi_{\on{univ}}
\end{equation}
with the right adjoint being conservative. 

\medskip

Hence, by the Barr-Beck-Lurie theorem, the category $\bC\underset{\Rep(H)}\otimes \bD$ identifies with the category
$$\Reg(H)\mod(\bC\otimes \bD),$$
where $\Reg(H)$ is the monad on $\bC\otimes \bD$, given by the action of the associative algebra object 
$$\Reg(H)\in \Rep(H)\otimes \Rep(H),$$
the ``regular representation" of $H$. 

\sssec{}

Assume now that $\bC$ and $\bD$ are compactly generated. In this case 
$\bC\underset{\Rep(H)}\otimes \bD$ is also compactly generated,
with the set of compact generators provided by
$$\Phi_{\on{univ}}(\bc\otimes\bd),\quad \bc\in \bC^c,\,\,\bd\in \bD^c.$$

\sssec{}  \label{sss:duality and ten}

Consider $\bC^\vee$ and $\bD^\vee$ equipped with the natural $\Rep(H)$-actions, and consider the corresponding functors
\begin{equation} \label{e:adj funct ten prod dual}
\Phi_{\on{univ}}:\bD^\vee \otimes \bC^\vee \rightleftarrows \bD^\vee \underset{\Rep(H)}\otimes \bC^\vee:\Psi_{\on{univ}}.
\end{equation}

It follows from \cite[Prop. 9.4.8]{GR1} that we have a canonical identification
$$\bD^\vee \underset{\Rep(H)}\otimes \bC^\vee \simeq (\bC\underset{\Rep(H)}\otimes \bD)^\vee$$
so that the functors in \eqref{e:adj funct ten prod dual} identify with the duals of those in \eqref{e:adj funct ten prod}.

\medskip

In particular, if we denote by $x\mapsto x^\vee$ the corresponding equivalences
$$(\bC^c)^{\on{op}}\to (\bC^\vee)^c, \quad (\bD^c)^{\on{op}}\to (\bD^\vee)^c \text{ and }
((\bC\underset{\Rep(H)}\otimes \bD)^c)^{\on{op}}\simeq (\bD^\vee \underset{\Rep(H)}\otimes \bC^\vee)^c,$$
we have
$$(\Phi_{\on{univ}}(\bc\otimes \bd))^\vee\simeq \Phi_{\on{univ}}(\bc^\vee\otimes \bd^\vee).$$

\begin{rem}
The above discussion applies to $\Rep(H)$ replaced by any \emph{rigid} symmetric monoidal category $\bA$. The role of the regular representation 
is played by the image of the unit object under the functor
$$\bA\to \bA\otimes \bA,$$
right adjoint adjoint to the tensor product functor $\bA\otimes \bA\to \bA$.
\end{rem} 

\sssec{}  \label{sss:t on Hecke}

Assume now that $\bC$ and $\bD$ are each equipped with a t-structure, such that the connective subcategories are generated
by compact objects. 

\medskip

Assume also that the action functors
$$\bC\otimes \Rep(H) \to \bC, \quad \Rep(H)\otimes \bD\to \bD$$
are t-exact. 

\medskip

Then $\bC\underset{\Rep(H)}\otimes \bD$ also acquires a t-structure, with both functors
$$\Phi_{\on{univ}}: \bC\otimes \bD \rightleftarrows \bC\underset{\Rep(H)}\otimes \bD:\Psi_{\on{univ}}$$ 
being t-exact. 

\ssec{The paradigm of Hecke eigen-objects: a reminder}  \label{ss:Hecke}

In this section we introduce a key definition: that of the category of Hecke eigen-objects arising
from a category equipped with an action of $\Rep(H)$. 

\sssec{}

We apply the discussion in \secref{ss:ten prod Rep(H)} to the case when $\bD=\Vect$ with the action of $\Rep(H)$ on $\bD$
given by the forgetful functor 
$$\Res^H:\Rep(H)\to \Vect,$$
i.e., the functor that sends a representation $V$ to its underlying vector space $\ul{V}$. 

\medskip

We will refer to the resulting category
$$\bC\underset{\Rep(H)}\otimes \Vect$$
as the category of Hecke eigen-objects in $\bC$, and denote it also by $\on{Hecke}(\bC)$.

\sssec{}

By definition, the functor
$$\Phi_{\on{univ}}:\bC\to \on{Hecke}(\bC)$$
is universal among functors $\Phi:\bD\to \bE$ equipped with a system of isomorphisms
$$\Phi(\bc \star V) \overset{\alpha_V}\simeq \ul{V}\otimes \Phi(\bc),\quad V\in \Rep(H),\,\,\bc\in \bC,$$
compatible with tensor products of representations. 

\sssec{}

The monad on $\bC$, corresponding to the pair of adjoint functors
$$\bC\rightleftarrows \on{Hecke}(\bC)$$
is given by the action of 
$$(\on{Id}\otimes \Res^H)(\Reg(H))\in \Rep(H),$$
i.e., we think of $\Reg(H)$ as a representation of one copy of $H$, rather than $H\times H$. 

\medskip

In what follows we will use the notation
$$\ind_{\on{Hecke}}:=\Phi_{\on{univ}} \text{ and } \oblv_{\on{Hecke}}:=\Psi_{\on{univ}}.$$

\sssec{}

By \secref{sss:coten}, we can think of $\on{Hecke}(\bC)$ as the category of objects $\bc\in \bC$, equipped with a
a system of isomorphisms 
\begin{equation} \label{e:Hecke eigen}
\bc\star V \overset{\beta_V}\simeq \ul{V}\otimes \bc,
\end{equation}
compatible with tensor products of representations. 

\medskip

The latter interpretation is the source of the name ``Hecke eigen-objects". 

\sssec{}

The category 
$$\on{Hecke}(\bC)=\bC\underset{\Rep(H)}\otimes \Vect$$
has a natural structure of category acted on by $H$.

\medskip

Explicitly, when we think of objects $\on{Hecke}(\bC)$ as in \eqref{e:Hecke eigen}, the action action of a point $h\in H$ 
on such an object is given by modifying the isomorphisms $\beta_V$ via the action of $h$ on $\ul{V}$.

\sssec{}  \label{sss:de-eq}

For a category $\wt\bC$ equipped with an action of $H$, let $\wt\bC^H$ denote the corresponding category of $H$-equivariant 
objects. By \cite[Theorem 2.2.2]{Ga5}, the category $\wt\bC^H$ is equipped with a natural action of $\Rep(H)$, and the assignments
$$\bC\mapsto \on{Hecke}(\bC) \text{ and } \wt\bC\mapsto (\wt\bC)^H$$
define mutually inverse equivalences between the $(\infty,2)$-categories
$$\Rep(H)\mmod \text{ and } H\mmod.$$

\medskip

In particular, the category $\bC$ can be recovered from $\on{Hecke}(\bC)$ as the category of $H$-equivariant objects, i.e.,
$$\bC\simeq (\on{Hecke}(\bC))^H.$$

The latter point of view allows us to think of $\on{Hecke}(\bC)$ as a ``de-equivariantization" of $\bC$.

\sssec{}

In particular, we can think of the functor 
$$\ind_{\on{Hecke}}:\bC\to \on{Hecke}(\bC)$$
as
$$\Res^H: (\on{Hecke}(\bC))^H\to  \on{Hecke}(\bC),$$
and of 
$$\oblv_{\on{Hecke}}:\on{Hecke}(\bC)\to \bC$$
as 
$$\coInd^H: \on{Hecke}(\bC)\to (\on{Hecke}(\bC))^H.$$

Here, for a category $\wt\bC$ acted on by $H$, we denote by
$$\Res^H:\wt\bC^H\rightleftarrows \wt\bC:\coInd^H$$
the corresponding adjoint pair of functors. 

\sssec{Example}   \label{sss:de-equiv}

Let $H$ be a torus with the weight lattice $\Lambda_H$. Then the datum of a category equipped with an action of
$\Rep(H)$ is the same as that of a category equipped with an action of $\Lambda_H$.

\medskip

Let $A$ be a $\Lambda_H$-graded algebra. Then the we can take as $\bC$ the category $\overset{\bullet}{A}\mod$ of
$\Lambda_H$-graded $A$-modules. The action of $\Lambda_H$ on $\overset{\bullet}{A}\mod$ is given by shifting 
the grading.

\medskip

The corresponding category $\on{Hecke}(\bC)$ can be identified with the
category $A\mod$. The action of $H$ on $A$ defines an $H$-action on $A\mod$, and the corresponding equivariant
category $A\mod^H$ can be identified with $\overset{\bullet}{A}\mod$.

\medskip

The functor $\ind_{\on{Hecke}}\simeq \Res^H$ is the forgetful functor
$$\overset{\bullet}{A}\mod\to A\mod,$$
and $\oblv_{\on{Hecke}}\simeq \coInd^H$ is its right adjoint, given by averaging along $\Lambda_H$.

\ssec{Graded Hecke eigen-objects}  \label{ss:graded Hecke}

From now on, until the end of this section, we let $H$ be a reductive group and $T_H\subset H$ be its Cartan subgroup.

\medskip

In this subsection we will discuss a version of the construction of \secref{ss:Hecke}, where instead
of the ``absolute" de-equivariantization, we perform a partial one, relative to $T_H$. 

\sssec{}

We will apply the framework of \secref{ss:ten prod Rep(H)} to the case when $\bD=\Rep(T_H)$, where $T_H$ is a torus mapping to $H$.
We denote the corresponding category
$$\bC\underset{\Rep(H)}\otimes \Rep(T_H)$$ 
by $\bHecke(\bC)$.

\medskip

We denote the resulting pair of adjoint functors by 
$$\ind_{\bHecke}:\bC\otimes \Rep(T_H)\rightleftarrows \bHecke(\bC):\oblv_{\bHecke}.$$ 

\medskip

The corresponding monad on $\bC\otimes \Rep(T_H)$ identifies with the action of
$$(\on{Id}\otimes \Res^H_{T_H})(\Reg(H))\in \Rep(H)\otimes \Rep(T_H).$$
 
\sssec{}

By construction, $\bHecke(\bC)$ is acted on by $\Rep(T_H)$. If we take its Hecke category with respect to $T_H$,
we recover $\on{Hecke}(\bC)$. In particular, we obtain that we can recover $\bHecke(\bC)$ as
$$\bHecke(\bC)\simeq (\on{Hecke}(\bC))^{T_H}.$$

In particular, we have a pair of adjoint functors
$$\Res^{T_H}:\bHecke(\bC) \rightleftarrows \Hecke(\bC):\coInd^{T_H}.$$

The composite
$$\oblv_{\Hecke}\circ \Res^{T_H}:\bHecke(\bC) \to \bC$$
is the forgetful functor 
$$\bHecke(\bC) \overset{\oblv_{\bHecke}}\longrightarrow \bC\otimes \Rep(T_H)
\overset{\on{Id}\otimes \Res^{T_H}}\longrightarrow \bC.$$

\sssec{}  \label{sss:graded Hecke univ}

By definition, the functor $\ind_{\bHecke}$ is a universal recipient among functors
$$\Phi:\bC\to \bE,$$
where $\bE$ is a $\Rep(T_H)$-module category, equipped with a system of identifications
$$\Phi(\bc\star V)\overset{\alpha_V}\simeq \Phi(\bc)\star \Res^H_{T_H}(V),$$
compatible with the with tensor products of representations. 

\sssec{}

By \secref{sss:coten}, we can think of $\bHecke(\bC)$ as the category of objects $\overset\bullet\bc\in \Rep(T_H)\otimes \bC$, 
equipped with a
a system of isomorphisms 
\begin{equation} \label{e:Hecke eigen grd}
\overset\bullet\bc\star V \overset{\beta_V}\simeq \Res^H_{T_H}(V)\otimes \overset\bullet\bc
\end{equation}
(where $\Rep(T_H)$ acts on $\Rep(T_H)\otimes \bC$ via the 1st factor), 
compatible with tensor products of representations. 

\medskip

For this reason we will refer to $\bHecke(\bC)$ as graded (with respect to lattice of weights of $T_H$) Hecke eigen-objets of $\bC$. 

\ssec{The relative Hecke category}  \label{ss:rel Hecke}

For future use we will introduce yet another variant of the Hecke category. 

\sssec{}

Let $\bC$ be acted by $\Rep(H)$ (on the right) and by $\Rep(T_H)$ (on the left) so that these two actions commute.
In other words, we have an action of $\Rep(H)\otimes \Rep(T_H)$ on $\bC$.

\medskip

We introduce the category $\bHecke_{\on{rel}}(\bC)$ as
$$\bHecke_{\on{rel}}(\bC):=\bC\underset{\Rep(H)\otimes \Rep(T_H)}\otimes \Rep(T_H),$$
where the functors
$$\Rep(H) \to \Rep(T_H) \leftarrow \Rep(T_H)$$
are restriction and identity, respectively. 

\sssec{}

By \secref{sss:coten}, we can think of $\bHecke_{\on{rel}}(\bC)$ as the category of objects $\bc\in \bC$, equipped 
with a system of isomorphisms
\begin{equation} \label{e:rel Hecke}
\bc\underset{H}\star V\simeq \Res^H_{T_H}(V)\star \bc,
\end{equation} 
compatible with tensor products of representations. 

\sssec{}

We will denote by $\ind_{\bHecke_{\on{rel}}}$ the functor
$$\bC \overset{\on{Id}\otimes \on{unit}}\longrightarrow \bC\otimes \Rep(T_H)\overset{\Phi_{\on{univ}}}\longrightarrow
\bC\underset{\Rep(H)\otimes \Rep(T_H)}\otimes \Rep(T_H)=:\bHecke_{\on{rel}}(\bC),$$
and by 
$$\oblv_{\bHecke_{\on{rel}}}:\bHecke_{\on{rel}}(\bC)\to \bC.$$
its right adjoint. 

\medskip

When we think of $\bHecke_{\on{rel}}(\bC)$ as objects $\bc\in \bC$ equipped with a system of isomorphisms 
\eqref{e:rel Hecke}, the functor $\oblv_{\bHecke_{\on{rel}}}$ remembers the data of $\bc$. 

\ssec{Duality for the Hecke category}  \label{ss:duality on Hecke gen} 

In this short subsection we will study how the formation of the Hecke category interacts with duality
on categories acted on by $\Rep(H)$. 

\sssec{}

Let $\bC$ be again a compactly generated category, acted on by $\Rep(H)$. Let us consider $\bC^\vee$
as acted on by $\Rep(H)$ by the formula
$$\bc^\vee \star V=(\bc\star \tau^H(V^*))^\vee, \quad \bc\in \bC^c,V\in \Rep(H),$$
where $\tau^H$ is the Cartan involution on $H$, see \secref{sss:Cartan inv}. 

\medskip

According to \secref{sss:duality and ten}, we have a canonical identification
$$\Hecke(\bC)^\vee\simeq \Hecke(\bC^\vee),$$
so that the diagram
$$
\CD
(\bC^c)^{\on{op}}   @>{\ind_{\Hecke}}>> (\Hecke(\bC)^c)^{\on{op}} \\
@VVV  @VVV   \\
(\bC^\vee)^c  @>{\ind_{\Hecke}}>> \Hecke(\bC^\vee)^c
\endCD
$$
commutes.

\sssec{}  \label{sss:duality on Hecke graded}

Similarly, we define a duality
$$\bHecke(\bC)^\vee\simeq \bHecke(\bC^\vee)$$
by making the following diagram commute: 
$$
\CD
((\bC\otimes \Rep(T_H))^c)^{\on{op}}   @>{\ind_{\bHecke}}>> (\bHecke(\bC)^c)^{\on{op}} \\
@VVV  @VVV   \\
(\bC^\vee\otimes \Rep(T_H)^c  @>{\ind_{\bHecke}}>> \bHecke(\bC^\vee)^c,
\endCD
$$
where the left vertical arrow sends
$$(\bc \otimes V) \mapsto \bc^\vee \otimes \tau^{T_H}(V^*).$$

\ssec{Irreducible objects in the Hecke category}  \label{ss:irred Hecke}

In this subsection we will assume that $\bC$ is equipped with a t-structure satisfying the assumption of
\secref{sss:t on Hecke}. 

\medskip

Note that in this case, according to \secref{sss:t on Hecke}, the category $\on{Hecke}(\bC)$ acquires
a t-structure, in which both functors $\ind_{\on{Hecke}}$ and $\oblv_{\on{Hecke}}$ are t-exact. 

\medskip

Our goal in this subsection is to give an explicit description of the 
irreducible objects in $\Hecke(\bC)$. We will be able to do so under an additional assumption on the 
action of $\Rep(H)$ on $\bC$, namely, when this action is \emph{accessible}. 

\sssec{}  \label{sss:restricted}

We shall say that an irreducible $\bc\in \bC^\heartsuit$ is \emph{restricted} for the Hecke action if 
for every $V\in \on{Irrep}(H)$, the object $\bc\star V\in \bC^\heartsuit$ is irreducible.

\medskip

For example, \thmref{t:restr convolution} says that if $\lambda\in \Lambda$ is restricted, then the object
$W^{\lambda,!*}\in \Whit_{q,x}(G)$ is a restricted irreducible. 

\sssec{}

We are going to prove:
 
\begin{prop}  \label{p:irred Hecke}  
Let $\bc\in \bC^\heartsuit$ be restricted. Then $\ind_{\on{Hecke}}(\bc)\in \on{Hecke}(\bC)^\heartsuit$ is irreducible.
\end{prop}

\begin{proof}

Let $\bc'\in \on{Hecke}(\bC)^\heartsuit$ be equipped with a non-zero map $\bc'\to \ind_{\on{Hecke}}(\bc)$; 
let us show that this map is a surjection. 

\medskip

We have a surjection
$$\ind_{\on{Hecke}}\circ \oblv_{\on{Hecke}}(\bc')\to \bc',$$
so we can assume that $\bc'$ is of the form $\ind_{\on{Hecke}}(\bc_1)$
for some $\bc_1\in \bC^\heartsuit$. Hence, the map in question comes from a map in $\bC$
\begin{equation} \label{e:map into dir sum}
\bc_1\to \oblv_{\on{Hecke}}\circ \ind_{\on{Hecke}}(\bc)\simeq \underset{V\in \on{Irrep}(H)}\oplus\, (\bc\star V)\otimes \ul{V}^*.
\end{equation} 

Let $V\in \on{Irrep}(H)$ be such that the component 
$\bc_1\to (\bc\star V)\otimes \ul{V}^*$ of the map \eqref{e:map into dir sum} is non-zero. Replacing $\bc_1$ by the preimage of $(\bc\star V)\otimes \ul{V}^*$
under \eqref{e:map into dir sum}, and using the assumption on $\bc$, 
we can assume that $\bc_1$ is isomorphic to $\bc\star V$ and \eqref{e:map into dir sum} corresponds to an element $\xi\in \ul{V}^*$. 

\medskip 

Hence, the original map $\bc'\to \ind_{\on{Hecke}}(\bc)$ identifies with
$$\ind_{\on{Hecke}}(\bc\star V)\simeq \ul{V}\otimes \ind_{\on{Hecke}}(\bc)\overset{\xi\otimes \on{id}}\longrightarrow \ind_{\on{Hecke}}(\bc),$$
and hence is manifestly a surjection.

\end{proof} 

\sssec{}   \label{sss:accessible}

We shall say that the action of $\Rep(H)$ on $\bC$ is \emph{accessible} if every irreducible object of $\bC^\heartsuit$ 
is of the form $\bc\star V$ for $\bc\in \bC^\heartsuit$ restricted and $V\in \on{Irrep}(H)$. 

\medskip

For example, \lemref{l:many restr} says that if the derived group of $H$ is simply-connected, then the action of
$\Rep(H)$ on $\Whit_{q,x}(G)$ is accessible. 

\sssec{}

From  \propref{p:irred Hecke} we obtain: 

\begin{cor}    \label{c:irred Hecke bis}
Assume that the t-structure on $\bC$ is Artinian and that the action of $\Rep(H)$ is accessible. Then: 

\smallskip

\noindent{\em(a)} Every irreducible object of $\on{Hecke}(\bC)^\heartsuit$ is of the 
form $\ind_{\on{Hecke}}(\bc)$ for a restricted $\bc\in \bC^\heartsuit$.

\smallskip

\noindent{\em(b)} If for two such objects we have $\ind_{\on{Hecke}}(\bc_1)\simeq \ind_{\on{Hecke}}(\bc_2)$, then 
$\bc_1\simeq \bc_2\star V$ for a 1-dimensional representation $V$ of $H$.

\end{cor} 

\begin{proof}

Let $\bc'$ be an irreducible object of $\on{Hecke}(\bC)^\heartsuit$. There exists an object $\bc_1\in \bC^\heartsuit$ equipped with a non-zero
map $\ind_{\on{Hecke}}(\bc_1)\to \bc'$. By Artinianness, we can assume that $\bc_1$ is irreducible. Write
$\bc_1\simeq\bc\star V$ for $\bc$ restricted. 

\medskip

Thus, we obtain a non-zero map
$$\ind_{\on{Hecke}}(\bc \star V)\simeq \ul{V}\otimes \ind_{\on{Hecke}}(\bc)\to \bc',$$
from which we deduce the existence of a non-zero map
$$\ind_{\on{Hecke}}(\bc)\to \bc'.$$

However, by \propref{p:irred Hecke}, $\ind_{\on{Hecke}}(\bc)$ is irreducible, and hence $\ind_{\on{Hecke}}(\bc)\simeq \bc'$. This proves
point (a). 

\medskip

For point (b), let us be given a non-zero map
$$\ind_{\on{Hecke}}(\bc_1)\to  \ind_{\on{Hecke}}(\bc_2)$$
for $\bc_1,\bc_2$ as in \propref{p:irred Hecke}. Then we obtain a non-zero map in $\bC$
$$\bc_1\to \underset{V\in \on{Irrep}(H)}\oplus \, (\bc_2\star V)\otimes \ul{V}^*.$$

Hence, we obtain a non-zero map 
$$\bc_1\to \bc_2\star V$$
for some irreducible $V$. By the assumption on $\bc_2$, the latter map is an isomorphism. 

\medskip

Symmetrically, we have: $\bc_2\simeq \bc_1\star W$ for some $W\in \on{Irrep}(H)$. Hence,
$$\bc_1\simeq \bc_1\star (V\otimes W).$$

From here we obtain that $V$ and $W$ are 1-dimensional. 

\end{proof}

\begin{cor}    \label{c:irred Hecke}
Assume that the t-structure on $\bC$ is Artinian and that the action of $\Rep(H)$ is accessible. 
Then the  t-structure on $\on{Hecke}(\bC)^\heartsuit$ is Artinian.
\end{cor}

\begin{proof}
It suffices to show that for an irreducible $\bc_1\in \bC^\heartsuit$, the object $\ind_{\on{Hecke}}(\bc_1)\in \on{Hecke}(\bC)^\heartsuit$
has finite length. 

\medskip

Write $\bc_1=\bc\star V$ for $\bc$ as in \propref{p:irred Hecke}. Then 
$$\ind_{\on{Hecke}}(\bc_1)\simeq \ul{V}\otimes \ind_{\on{Hecke}}(\bc),$$
and the assertion follows.

\end{proof}

\ssec{Irreducible objects in the graded version}

In this subsection we retain the assumptions of \secref{ss:irred Hecke}. We will 
will adapt the results of {\it loc.cit.} to describe irreducibles in the graded Hecke category $\bHecke(\bC)$. 

\sssec{}

First off note that if $H$ is a torus, any irreducible object in $\bC$ is restricted for this action: 
indeed the irreducible objects of $\Rep(H)$ are 1-dimensional and act by invertible functors.

\medskip

In particular, an action of $\Rep(H)$ is automatically accessible. 

\sssec{}

Let us be in the context of \secref{ss:graded Hecke}. Applying \propref{p:irred Hecke} and 
Corollaries \ref{c:irred Hecke} and \ref{c:irred Hecke bis} to the torus $T_H$, we obtain:

\begin{cor}    \label{c:irred Hecke graded}  \hfill

\smallskip

\noindent{\em(a)} 
The forgetful functor 
$$\Res^{T_H}:\bHecke(\bC)^\heartsuit\to \on{Hecke}(\bC)^\heartsuit$$ 
sends irreducibles to irreducibles.

\smallskip

\noindent{\em(b)} Every irreducible of $\on{Hecke}(\bC)^\heartsuit$ is of the form $\Res^{T_H}(\bc)$ for some irreducible 
$\bc\in \bHecke(\bC)^\heartsuit$.

\smallskip

\noindent{\em(c)} If for two irreducibles $\bc_1,\bc_2\in \bHecke(\bC)^\heartsuit$, we have $\Res^{T_H}(\bc_1)\simeq \Res^{T_H}(\bc_2)$,
then $\bc_1$ and $\bc_2$ differ by a translation by an element of $\Lambda_H$.

\end{cor}

Combining with \propref{p:irred Hecke} and Corollaries \ref{c:irred Hecke} and \ref{c:irred Hecke bis}, we obtain: 

\begin{cor}    \label{c:irred Hecke graded bis}   \hfill 

\smallskip

\noindent{\em(a)} For every restricted irreducible $\bc\in \bC^\heartsuit$ and every $\gamma\in \Lambda_H$,
the object $\ind_{\bHecke}(\bc\otimes \sfe^\gamma)\in \bHecke(\bC)^\heartsuit$ is irreducible.

\smallskip

\noindent{\em(b)} Suppose that the t-structure on $\bC$ is Artinian and that the 
action of $\Rep(H)$ on $\bC$ is accessible. Then: 

\smallskip

\noindent{\em(i)} The t-structure on $\bHecke(\bC)^\heartsuit$ is Artinian;

\smallskip

\noindent{\em(ii)} Every irreducible object of $\bHecke(\bC)^\heartsuit$ is of the form $\ind_{\bHecke}(\bc\otimes \sfe^\gamma)$
for some $\bc\in \bC^\heartsuit$ as in \propref{p:irred Hecke} and $\gamma\in \Lambda_H$. 

\smallskip

\noindent{\em(iii)} Two irreducible objects $\ind_{\bHecke}(\bc_1\otimes \sfe^{\gamma_1})$ and $\ind_{\bHecke}(\bc_2\otimes \sfe^{\gamma_2})$
are isomorphic if and only if $\gamma_1-\gamma_2$ extends to a character (to be denoted $\gamma$) of $H$, and
$\bc_2\simeq \bc_1\star \sfe^\gamma$, where $\sfe^\gamma$ denotes the corresponding one-dimensional representation of $H$. 

\end{cor} 

\section{The category of Hecke eigen-objects in the Whittaker category}  \label{s:Hecke Whit categ}

In this section we will finally define and study the main character of this paper, the category of graded Hecke eigensheaves 
in the metaplectic Whittaker category. 

\ssec{Definition}

In this subsection we introduce the category of graded Hecke eigensheaves in the metaplectic Whittaker category,
$$\bHecke(\Whit_{q,x}(G)),$$
taken with respect to the $\Rep(H)$ action on $\Whit_{q,x}(G)$, which was defined in \secref{ss:Hecke action on Whit}. 

\medskip

We will also consider the non-graded version $\on{Hecke}(\Whit_{q,x}(G))$. 

\sssec{}

By \secref{sss:ind Hecke}, we have a pair of adjoint functors
$$\ind_{\bHecke}:\Whit_{q,x}(G)\otimes \Rep(T_H)\rightleftarrows \bHecke(\Whit_{q,x}(G)):\oblv_{\bHecke}.$$

The corresponding monad on $\Whit_{q,x}(G)$ is given by the action of the object
$$\Sat_{q,G}\otimes \Res^{H}_{T_H}(\Reg(\cG))\in \Sph_{q,x}(G)\otimes \Rep(T_H).$$

\medskip 

In particular, the category $\bHecke(\Whit_{q,x}(G))$ is compactly generated by the essential image of 
$(\Whit_{q,x}(G)\otimes \Rep(T_H))^c$ under the functor $\ind_{\bHecke}$. 

\sssec{}

By \secref{sss:coten}, we can rewrite $\bHecke(\Whit_{q,x}(G))$ as the category consisting of
$$\CF\in \Whit_{q,x}(G)\otimes \Rep(T_H),$$
equipped with a system of identifications 
$$\CF\star \Sat_{q,G}(V)\overset{\beta_V}\simeq \Res^{H}_{T_H}(V)\otimes \CF,$$
(where $\Rep(T_H)$ acts on $\Whit_{q,x}(G)\otimes \Rep(T_H)$ via the 2nd factor), 
compatible with tensor products of representations. 

\sssec{}  \label{sss:duality on Hecke}

Taking into account \secref{sss:Hecke action and duality}, the recipe of \secref{sss:duality on Hecke graded} defines an identification 
\begin{equation} \label{e:duality for Hecke Whit}
\bHecke(\Whit_{q,x}(G))^\vee \simeq \bHecke(\Whit_{q^{-1},x}(G)),
\end{equation}
i.e., an equivalence
\begin{equation} \label{e:duality for Hecke Whit comp}
\BD^{\on{Verdier}}:((\bHecke(\Whit_{q,x}(G)))^c)^{\on{op}} \simeq (\bHecke(\Whit_{q^{-1},x}(G)))^c,
\end{equation}
which makes the following diagram commute
$$
\CD
((\Whit_{q,x}(G)\otimes \Rep(T_H))^c)^{\on{op}}   @>{(\ind_{\bHecke})^{\on{op}}}>>    ((\bHecke(\Whit_{q,x}(G)))^c)^{\on{op}}   \\
@V{\BD^{\on{Verdier}}\otimes (\tau^{T_H}\circ \BD^{\on{lin}})}VV   @VV{\BD^{\on{Verdier}}}V   \\
(\Whit_{q^{-1},x}(G)\otimes \Rep(T_H))^c   @>{\ind_{\bHecke}}>>  (\bHecke(\Whit_{q^{-1},x}(G)))^c. 
\endCD
$$

\ssec{Behavior with respect to isogenies}  

As a convenient technical tool, we will study the behavior of the categores $\bHecke(\Whit_{q,x}(G))$ and $\Hecke(\Whit_{q,x}(G))$
under isogenies $G\to \wt{G}$. 

\sssec{}   \label{sss:isogen}

Let us be given a short exact sequence of reductive groups
$$1\to G\to \wt{G}\to T_0\to 1,$$
where $T_0$ is a torus. Consider the corresponding short exact sequence of tori
$$0\to T\to \wt{T}\to T_0\to 0$$
and lattices
$$0\to \Lambda\to \wt\Lambda\to \Lambda_0\to 0.$$

\medskip

Let us be given a geometric metaplectic data $\CG^{\wt{G}}$ for $\wt{G}$, whose restriction to $G$ gives $\CG^G$.
We obtain a map of lattices 
$$\Lambda^\sharp\to \wt\Lambda^\sharp$$
and a map of reductive groups
$$\wt{H}\to H.$$

\sssec{}  \label{sss:assump isogen}

We will say that the isogeny is \emph{strictly compatible with the geometric metaplectic data} if the diagram 
$$
\CD
\wt\Lambda^\sharp  @>>>  \wt\Lambda \\
@AAA   @AAA   \\
\Lambda^\sharp @>>> \Lambda 
\endCD
$$
is a push-out square, cf. \secref{sss:assump isogen lattice}. 

\medskip

In particular, in this case we obtain a short exact sequence of lattices
\begin{equation} \label{e:SES sharp H}
0\to \Lambda^\sharp\to \wt\Lambda^\sharp \to \Lambda_0\to 0,
\end{equation}
and an isogeny of metaplectic duals
$$1\to \cT_0\to \wt{H}\to H \to 1.$$

\sssec{}  \label{sss:Gr isogeny}

Note that the image of the resulting map 
$$\Gr^{\omega^\rho}_{G,x}\to \Gr^{\omega^\rho}_{\wt{G},x}$$
is a union of some connected components.  

\medskip

In particular, direct image defines a fully faithful functor
\begin{equation} \label{e:isogen Gr}
\Shv_{\CG^G}(\Gr^{\omega^\rho}_{G,x})\to \Shv_{\CG^{\wt{G}}}(\Gr^{\omega^\rho}_{\wt{G},x}),
\end{equation}
and in particular a fully faithful monoidal functor
$$\Sph_{q,x}(G)\to \Sph_{q,x}(\wt{G}).$$

We have a commutative diagram
$$
\CD
\Rep(H)  @>>>  \Rep(\wt{H}) \\
@V{\Sat_{q,G}}VV   @VV{\Sat_{q,\wt{G}}}V   \\
\Sph_{q,x}(G)  @>>>  \Sph_{q,x}(\wt{G}). 
\endCD
$$

\sssec{}

Consider the resulting functor
\begin{equation} \label{e:Hecke isogeny}
\on{Hecke}(\Shv_{\CG^G}(\Gr^{\omega^\rho}_{G,x}))\to \on{Hecke}(\Shv_{\CG^{\wt{G}}}(\Gr^{\omega^\rho}_{\wt{G},x})),
\end{equation} 
where $\on{Hecke}$ on the left-hand (resp., right-hand) side is taken with respect to action 
of $\Rep(H)$ (resp., $\Rep(\wt{H})$). 

\medskip

We have a commutative diagram 
$$
\CD
\Shv_{\CG^G}(\Gr^{\omega^\rho}_{G,x})  @>{\text{\eqref{e:isogen Gr}}}>>   \Shv_{\CG^{\wt{G}}}(\Gr^{\omega^\rho}_{\wt{G},x})   \\
@V{\ind_{\Hecke}}VV   @VV{\ind_{\Hecke}}V  \\
\on{Hecke}(\Shv_{\CG^G}(\Gr^{\omega^\rho}_{G,x})) @>{\text{\eqref{e:Hecke isogeny}}}>>  \on{Hecke}(\Shv_{\CG^{\wt{G}}}(\Gr^{\omega^\rho}_{\wt{G},x})). 
\endCD
$$

We claim:

\begin{prop}  \label{p:isogen Gr}  \hfill 

\smallskip

\noindent{\em(a)}
The functor \eqref{e:Hecke isogeny} is fully faithful.

\smallskip

\noindent{\em(b)}
If the isogeny is stricttly compatible with the geometric metaplectic data (see 
\secref{sss:assump isogen}), then \eqref{e:Hecke isogeny} is an equivalence. 

\end{prop}

\begin{proof}

To prove point (a), it suffices to show that for $\CF_0,\CF_1\in \Shv_{\CG^G}(\Gr^{\omega^\rho}_{G,x})$,
the map
\begin{multline*} 
\Maps_{\on{Hecke}(\Shv_{\CG^G}(\Gr^{\omega^\rho}_{G,x}))}(\ind_{\on{Hecke}}(\CF_0),\ind_{\on{Hecke}}(\CF_1)) \to \\
\to \Maps_{\on{Hecke}(\Shv_{\CG^{\wt{G}}}(\Gr^{\omega^\rho}_{\wt{G},x}))}(\ind_{\on{Hecke}}(\CF_0),\ind_{\on{Hecke}}(\CF_1))
\end{multline*} 
is an isomorphism, where in the left-hand side $\ind_{\on{Hecke}}$ denotes the functor
$$\Shv_{\CG^G}(\Gr^{\omega^\rho}_{G,x})\to \on{Hecke}(\Shv_{\CG^G}(\Gr^{\omega^\rho}_{G,x})),$$
and in the right-hand side, it denotes the functor
$$\Shv_{\CG^{\wt{G}}}(\Gr^{\omega^\rho}_{\wt{G},x})\to \on{Hecke}(\Shv_{\CG^{\wt{G}}}(\Gr^{\omega^\rho}_{\wt{G},x})).$$

\medskip

By adjunction, this is equivalent to showing that the map
\begin{multline*} 
\Maps_{\Shv_{\CG^G}(\Gr^{\omega^\rho}_{G,x})}\left(\CF_0,\underset{V\in \on{Irrep}(H)}\oplus\, \CF_1\star \Sat_{q,G}(V)\otimes \ul{V}\right) \to  \\
\Maps_{\Shv_{\CG^{\wt{G}}}(\Gr^{\omega^\rho}_{\wt{G},x})}\left(\CF_0,\underset{V\in \on{Irrep}(\wt{H})}\oplus\, 
\CF_1\star \Sat_{q,G}(V)\otimes \ul{V}\right)
\end{multline*} 
is an isomorphism. 

\medskip

The required isomorphism follows from the fact that for $\CF$ supported on $\Gr^{\omega^\rho}_{G,x}\subset \Gr^{\omega^\rho}_{\wt{G},x}$ and 
$V\in \on{Irrep}(\wt{H})$,
the object $\CF\star \Sat_{q,G}(V)\in \Shv_{\CG^{\wt{G}}}(\Gr^{\omega^\rho}_{\wt{G},x})$ is still supported on 
$\Gr^{\omega^\rho}_{G,x}$ if and only if 
$$V\in \on{Irrep}(H)\subset \on{Irrep}(\wt{H}).$$

\medskip

For point (b) we note that the condition in \secref{sss:assump isogen} implies that for every 
$0\neq \CF_1\in \Shv_{\CG^{\wt{G}}}(\Gr^{\omega^\rho}_{\wt{G},x})$ there exists $V\in \on{Irrep}(\wt{H})$ 
so that $\CF_1\star \Sat_{q,G}(V)$ is non-zero when restricted to 
$\Gr^{\omega^\rho}_{G,x}\subset \Gr^{\omega^\rho}_{\wt{G},x}$.

\medskip

To prove that \eqref{e:Hecke isogeny} is an equivalence, it suffices to show that for every 
$\CF'\in \on{Hecke}(\Shv_{\CG^{\wt{G}}}(\Gr^{\omega^\rho}_{\wt{G},x}))$ there exists $\CF\in \on{Hecke}(\Shv_{\CG^G}(\Gr^{\omega^\rho}_{G,x}))$ equipped
with a non-zero map 
$$\ind_{\on{Hecke}}(\CF)\to \CF'.$$

Choose $V\in \on{Irrep}(\wt{H})$ so that $\oblv_{\on{Hecke}}(\CF')\star \Sat_{q,G}(V)$ is non-zero when restricted to 
$\Gr^{\omega^\rho}_{G,x}$. Let $\CF$ be the resulting object of $\Shv_{\CG^G}(\Gr^{\omega^\rho}_{G,x})$.  By construction, we have a non-zero map
$$\CF\star \Sat_{q,G}(V^*)\to \oblv_{\on{Hecke}}(\CF'),$$
and hence a non-zero map
$$\ul{V}^*\otimes \ind_{\on{Hecke}}(\CF)\to \CF',$$
as desired. 

\end{proof} 
 
From \propref{p:isogen Gr}, we obtain:  

\begin{cor}  \label{c:isogen Gr} Assume that the isogeny is strictly compatible with the geometric metaplectic data. Then:

\smallskip

\noindent{\em(a)}
The functors \eqref{e:isogen Gr} and $\Rep(T_H)\to \Rep(T_{\wt{H}})$ induce an equivalence
of $\Rep(T_H)$-module categories: 
$$\Rep(T_{\wt{H}})\underset{\Rep(T_H)}\otimes 
\bHecke(\Shv_{\CG^G}(\Gr^{\omega^\rho}_{G,x})) \to \bHecke(\Shv_{\CG^{\wt{G}}}(\Gr^{\omega^\rho}_{\wt{G},x})),$$
where $\bHecke$ on the left-hand (resp., right-hand) side is taken with respect to
$\Rep(H)$ and $\Rep(T_H)$  (resp., $\Rep(\wt{H})$ and $\Rep(T_{\wt{H}})$).

\smallskip

\noindent{\em(b)} A choice of a splitting of \eqref{e:SES sharp H} defines
an equivalence
$$\Rep(\cT_0)\otimes \bHecke(\Shv_{\CG^G}(\Gr^{\omega^\rho}_{G,x})) \simeq \bHecke(\Shv_{\CG^{\wt{G}}}(\Gr^{\omega^\rho}_{\wt{G},x})).$$

\end{cor} 

\sssec{}

Imposing the Whittaker condition, from \eqref{e:isogen Gr}, we obtain a fully faithful functor 
\begin{equation} \label{e:isogen Whit}
\Whit_{q,x}(G)\to \Whit_q(\Gr^{\omega^\rho}_{\wt{G},x}).
\end{equation}

From \propref{p:isogen Gr} and \corref{c:isogen Gr} we obtain:

\begin{cor}  \label{c:isogen Whit}  Assume that the isogeny is strictly compatible with the geometric metaplectic data. Then:

\smallskip

\noindent{\em(a)} The functor 
$$\on{Hecke}(\Whit_{q,x}(G))\to \on{Hecke}(\Whit_q(\Gr^{\omega^\rho}_{\wt{G},x}))$$
is an equivalence.

\smallskip

\noindent{\em(b)} The functor 
$$\Rep(T_{\wt{H}})\underset{\Rep(T_H)}\otimes \bHecke(\Whit_{q,x}(G)) \to \bHecke(\Whit_q(\Gr^{\omega^\rho}_{\wt{G},x}))$$
is an equivalence.

\smallskip

\noindent{\em(c)}
A choice of a splitting of \eqref{e:SES sharp H} defines
an equivalence
$$\Rep(\cT_0) \otimes \bHecke(\Whit_{q,x}(G)) \simeq \bHecke(\Whit_q(\Gr^{\omega^\rho}_{\wt{G},x}))$$
is an equivalence.
\end{cor} 

\ssec{The t-structure and the description of irreducibles}

In this subsection we will study the behavior of the t-structure on the category $\bHecke(\Whit_q(\wt{G}))$. 
In particular, we will obtain an explicit description of irreducibles. 

\sssec{}

According to \secref{sss:t on Hecke} and \propref{p:conv is exact}, the category $\bHecke(\Whit_q(\wt{G}))$ acquires 
a t-structure in which both functors
$$\ind_{\bHecke}:\Whit_{q,x}(G)\otimes \Rep(T_H)\rightleftarrows \bHecke(\Whit_{q,x}(G)):\oblv_{\bHecke}$$
are t-exact. 

\sssec{}

We claim:

\begin{prop} \label{p:irred Hecke Whit}  \hfill

\smallskip

\noindent{\em(a)} The t-structure on $\bHecke(\Whit_{q,x}(G))$ is Artinian. 

\smallskip

\noindent{\em(b)} There is a natural bijection between irreducibles of $\bHecke(\Whit_{q,x}(G))^\heartsuit$ and elements
of $\Lambda$.

\end{prop}

In the course of the proof of \propref{p:irred Hecke Whit} we will give an explicit description of the irreducibles
of $\bHecke(\Whit_{q,x}(G))^\heartsuit$, which will also be useful later.

\sssec{Proof of \propref{p:irred Hecke Whit}, first case}  
 
Let us first assume that the pair $(G,\CG^G)$ is such that the derived group of $H$ is simply connected. 
In this case, by \lemref{l:many restr} and \thmref{t:restr}, the action of $\Rep(H)$ on $\Whit_{q,x}(G)$ is accessible. 

\medskip

Hence the assertion of the proposition follows from \corref{c:irred Hecke graded bis}(b). 

\sssec{Proof of \propref{p:irred Hecke Whit}, reduction step}   \label{sss:red to sc}

Suppose now that we are given an isogeny $$G\to \wt{G}$$ strictly compatible with the geometric metaplectic data. 
From \corref{c:isogen Whit}(b,c), we obtain that the assertion of \propref{p:irred Hecke Whit} for $(\wt{G},\CG^{\wt{G}})$
implies that for $(G,\CG^G)$.

\medskip

Hence, in order to prove \propref{p:irred Hecke Whit}, it suffices to show the following:

\begin{prop} \label{p:isogenies trick}
Given $(G,\CG^G)$, there exists a pair $(\wt{G},\CG^{\wt{G}})$ and an isogeny $G\to \wt{G}$, 
strictly compatible with the geometric metaplectic data, such the derived group of $\wt{H}$ is simply connected. 
\end{prop}

The proof of \propref{p:isogenies trick} is given in \secref{ss:proof isog} below. 

\qed[\propref{p:irred Hecke Whit}]

\sssec{}

For $\lambda\in \Lambda$, let $\CM^{\lambda,!*}_{\Whit}$ denote the corresponding irreducible in $\bHecke(\Whit_{q,x}(G))$.
By construction, for $\gamma\in \Lambda^\sharp$, we have:  
$$\CM^{\lambda,!*}_{\Whit}\otimes \sfe^\gamma\simeq \CM_{\Whit}^{\lambda+\gamma,!*},$$
where $\otimes$ stands for the action of $\Rep(T_H)$ on $\bHecke(\Whit_{q,x}(G))$. Moreover, for $\lambda$ restricted, we have
$$\CM^{\lambda,!*}_{\Whit}:=\ind_{\bHecke}(W^{\lambda,!*}).$$

\medskip

We claim:

\begin{cor} \label{c:isotypics in Weyl}  For $\lambda\in \Lambda^+$, the object 
$\ind_{\bHecke}(W^{\lambda,*})\in (\bHecke(\Whit_{q,x}(G)))^\heartsuit$
receives a non-zero map from the irreducible $\CM^{\lambda,!*}_{\Whit}$, and the
Jordan-Holder constituents of the quotient are of the form $\CM^{\lambda',!*}_{\Whit}$ for $\lambda'\leq \lambda$.
\end{cor}

\begin{proof}
It is enough to prove the assertion for $\ind_{\bHecke}(W^{\lambda,!*})$ instead of  $\ind_{\bHecke}(W^{\lambda,*})$. 
Again, we can assume that the derived group of $H$ is simply-connected. Write $\lambda=\lambda_1+\gamma$
with $\lambda_1$ restricted and $\gamma\in \Lambda^{\sharp,+}$. Then
$$\ind_{\bHecke}(W^{\lambda,!*})\simeq \CM^{\lambda_1,!*}_{\Whit}\otimes \Res^H_{T_H}(V^\gamma),$$
and the assertion follows.
\end{proof} 

\sssec{}

Recall the duality functor
$$\BD^{\on{Verdier}}:(\bHecke(\Whit_{q,x}(G))^c)^{\on{op}}\to \bHecke(\Whit_{q^{-1},x}(G))^c.$$

By \corref{c:duality irred} and the construction of the irreducibles $\CM^{\lambda,!*}_{\Whit}$, 
we have
$$\BD^{\on{Verdier}}(\CM^{\lambda,!*}_{\Whit})\simeq \CM^{\lambda,!*}_{\Whit}.$$

From here and \propref{p:irred Hecke Whit} we obtain:

\begin{cor}
An object $\CF\in \bHecke(\Whit_{q,x}(G))^c$ is connective/coconnective if and only if $\BD^{\on{Verdier}}(\CF)$
is coconnective/connective 
\end{cor}

\ssec{Proof of \propref{p:isogenies trick}}  \label{ss:proof isog}

The proof of \propref{p:isogenies trick} will amount to a manipulation with lattices and root data. 

\sssec{}  

We first choose an isogeny 
$$1\to \cT_0\to \wt{H} \twoheadrightarrow H\to 1$$ 
so that the derived group of $\wt{H}$ is simply-connected. Let $\wt\Lambda^\sharp$ denote the weight lattice of $\wt{H}$.

\medskip

We define the lattice
$\wt\Lambda$ to be the push-out
$$\wt\Lambda^\sharp \underset{\Lambda^\sharp}\sqcup\, \Lambda.$$

By construction, the map 
$$\wt\Lambda^\sharp\to \wt\Lambda$$
is a rational equivalence. 

\sssec{}

We now construct a root datum for which $\wt\Lambda$ is the coweight lattice. For a coroot $\alpha$ of $G$,
we let $\wt\alpha \in \wt\Lambda$ be the image of $\alpha$ under the natural 
embedding $\Lambda\to \wt\Lambda$.

\medskip

The corresponding root $\wt{\check\alpha}$ is constructed as follows:
$$\wt{\check\alpha}=\ell_\alpha\cdot \wt{\check\alpha}_H,$$
where $\wt{\check\alpha}_H$ is the corresponding coroot of $\wt{H}$,
and $\ell_\alpha$ is as in \secref{sss:q alpha}. 

\medskip
 
A priori, $\wt{\check\alpha}$ is defined as an element of
$$\wt\cLambda{}^\sharp\subset \wt\cLambda{}^\sharp\underset{\BZ}\otimes \BQ
\simeq \wt\cLambda\underset{\BZ}\otimes \BQ.$$
However, it is easy to see that it in fact belongs to $\wt\cLambda$.

\medskip

It follows from \secref{sss:roots in dual} that the elements 
$$\{\wt\alpha\in \wt\Lambda,\wt{\check\alpha}\in \wt\cLambda\}$$
indeed form a root system so that
$$\Lambda\to \wt\Lambda$$
is an isogeny.

\sssec{}

Let $\wt{G}$ be the corresponding reductive group over $k$. By construction, we have a short exact sequence
$$1\to G\to \wt{G}\to T_0\to 1,$$
where $T_0$ is the torus dual to $\cT_0$. 

\sssec{}

It remains to show that, on a Zariski neighborhood of the point $x\in X$, there exists a geometric metaplectic data
$\CG^{\wt{G}}$ for $\wt{G}$ such that the map $G\to \wt{G}$ is strictly compatible with the geometric metaplectic data. 

\medskip

Note, however, that by \cite[Sect. 3.3]{GLys}, on \emph{affine curves}, geometric metaplectic data are classified
by their associated quadratic forms. Hence, it remains to show that the quadratic form $q$ on $\Lambda$ can be extended
to an element 
$$\wt{q}\in \on{Quad}(\wt\Lambda,\sfe^\times(-1))^W_{\on{restr}}$$
(see \cite[Sect. 3.2.2]{GLys} or \secref{sss:b'} for the notation), 
in such a way that kernel of the associated symmetric bilinear form $\wt{b}$ equals $\wt\Lambda^\sharp\subset \wt\Lambda$. 

\sssec{}

Note that the restriction of the quadratic form $q$ to $\Lambda^\sharp$ is such that the associated
symmetric bilinear form vanishes. Hence $q^\sharp:=q|_{\Lambda^\sharp}$ is a linear map
$$\Lambda^\sharp \to \pm 1\subset \sfe^\times.$$

We extend the above map in an arbitrary way to a map 
$$\wt{q}^\sharp:\wt\Lambda^\sharp\to \pm 1.$$ 

We define a quadratic form $\wt{q}$ on $\wt\Lambda$ by the formula
$$\wt{q}(\lambda+\wt\lambda^\sharp)=q(\lambda)+\wt{q}^\sharp(\wt\lambda^\sharp), \quad
\lambda\in \Lambda,\,\wt\lambda^\sharp\in \wt\Lambda^\sharp.$$

It is easy to see that $\wt{q}$, constructed above, indeed belongs to
$$\on{Quad}(\wt\Lambda,\sfe^\times(-1))^W_{\on{restr}}\subset \on{Quad}(\wt\Lambda,\sfe^\times(-1)),$$
as required. 

\newpage 

\centerline{\bf Part IV: The metaplectic semi-infinite IC sheaf} 

\bigskip 

Having studied the Whittaker category on the affine Grassmannian and its de-equivariantization, i.e., 
the Hecke category, we move on to the next step: we want to relate it to the right-hand side of our main
theorem, which is the category of factorization modules over some factorization algebra on (an object
closely related to) the affine Grassmannian for the Cartan subgroup $T$. 

\medskip

The passage between $\Gr_G$
to $\Gr_T$ can justifiably be called a \emph{Jacquet functor} as it involves taking cohomology along
$\fL(N)$-orbits. However, there is a caveat: this is not just cohomology, but rather it is taken against
a non-trivial kernel, The kernel is metaplectic semi-infinite IC sheaf, denoted $\ICs_{q,\Ran}$, 
is the object of study in this Part. 

\bigskip

\section{The metaplectic semi-infinite category of the affine Grassmannian}   \label{s:semiinf categ}

The metaplectic semi-infinite IC sheaf is constructed by the usual procedure of \emph{intermediate extension}
inside a certain DG category equipped with a t-structure. 

\medskip

The goal of the present section is to introduce this DG category: this is the (unital version of) 
the metaplectic semi-infinite category on $\Gr_{G,\Ran}$, denoted $\SI_{q,\Ran}(G)^{\leq 0}_{\on{untl}}$. 

\ssec{The semi-infinite category}

In this subsection we will define the metaplectic semi-infinite category, first at a fixed point $x\in X$, denoted $\SI_{q,x}(G)$, 
and then its Ran version, denoted $\SI_{q,\Ran}(G)$. 

\sssec{}

We define the metaplectic semi-infinite category of the affine Grassmannian, denoted $\SI_{q,x}(G)$ as the full
subcategory in $\Shv_{\CG^G}(\Gr^{\omega^\rho}_{G,x})$ that consists of $\fL(N)^{\omega^\rho}_x$-equivariant 
objects, i.e.,
$$\SI_{q,x}(G):=\left(\Shv_{\CG^G}(\Gr^{\omega^\rho}_{G,x})\right)^{\fL(N)^{\omega^\rho}_x}.$$

The difference between $\Whit_{q,x}(G)$ and $\SI_{q,x}(G)$ is that instead of the non-degenerate character we use the trivial one.

\sssec{}

Much of the discussion pertaining to the definition of $\Whit_{q,x}(G)$ applies to $\SI_{q,x}(G)$. In  particular,
we have the full subcategories 
$$\SI_{q,x}(G)_{=\mu} \subset \SI_{q,x}(G)_{\leq \mu}\subset \SI_{q,x}(G)$$
and the corresponding adjoint functors. 

\medskip

However, instead of \propref{p:Whittaker strata}, we have the following assertion (with the same proof):

\begin{lem}  \label{l:semiinf strata}
The category $\SI_{q,x}(G)_{=\mu}$ is (non-canonically) equivalent to $\Vect$ for
\emph{any} $\mu\in \Lambda$, via the functor of !-fiber at the point $t^\mu\in \Gr^{\omega^\rho}_{G,x}$.
For $\mu=0$ this equivalence is canonical. 
\end{lem} 

\begin{rem}
We note, however, that although the standard objects (i.e., the !-extensions of the generators of each $\SI_{q,x}(G)_{=\mu}$) 
are compact, the corresponding co-standard objects (i.e., the *-extensions) are no longer such. This contrasts with the case of 
$\Whit_{q,x}(G)$, see \corref{c:cost compact in Whit}.
\end{rem} 

\sssec{}

Let $\SI_{q,\Ran}(G)$ be the Ran space version of the semi-infinite category, i.e.,
$$\SI_{q,\Ran}(G):=\left(\Shv_{\CG^G}(\Gr^{\omega^\rho}_{G,\Ran})\right)^{\fL(N)^{\omega^\rho}_{\Ran}}.$$

\sssec{}

Let $\ol{S}^0_{\Ran}\subset \Gr^{\omega^\rho}_{G,\Ran}$ be the corresponding closed subfunctor, see \secref{sss:non-marked S Ran}. 
We let
$$\SI_{q,\Ran}(G)^{\leq 0}\subset \SI_{q,\Ran}(G)$$
be the full subcategory that consists of objects supported on $\ol{S}^0_{\Ran}$.

\ssec{Stratifications}

In this subsection we introduce a stratification of $\ol{S}^0_{\Ran}$ by locally closed subfunctors, according to 
the order of degeneracy of the Drinfeld structure. This stratification will give rise to a stratification of the category 
$\SI_{q,\Ran}(G)^{\leq 0}$.

%

\sssec{}

Recall the space of colored divisors denoted $\Conf$, see \secref{sss:conf}. 
Let $\Gr^{\omega^\rho}_{G,\Conf}$ be the prestack over $\Conf$ that classifies triples $(D,\CP_G,\alpha)$,
where 

\begin{itemize}

\item $D=\underset{k}\Sigma\, \mu_k\cdot x_k$ is a point of $\Conf$;

\item $\CP_G$ is a $G$-bundle;

\item $\alpha$ is an identification of $\CP_G$ with $\omega^\rho$ away from $\{x_k\}$.

\end{itemize}

\medskip

In a similar way we define the group (ind)-schemes 
$$\fL^+(G)^{\omega^\rho}_{\Conf}\subset \fL(G)^{\omega^\rho}_{\Conf} \text{ and }
\fL^+(N)^{\omega^\rho}_{\Conf}\subset \fL(N)^{\omega^\rho}_{\Conf}$$
over $\Conf$. 

\sssec{}

Let $\ol{S}^{\Conf}_{\Conf}$ be the closed subfunctor of $\Gr^{\omega^\rho}_{G,\Conf}$ consisting of points
$(D,\CP_G,\alpha)$ as above, for which for every $\clambda\in \cLambda^+$, the composite map
\begin{equation} \label{e:S- Ran}
(\omega^{\frac{1}{2}})^{\langle \clambda,2\rho\rangle}
(\underset{k}\Sigma\, -\langle \clambda,\mu_k\rangle \cdot x_k)\to \CV^\clambda_{\CP'_G}\to \CV^\clambda_{\CP_G}
\end{equation}
which is a priori defined on $X-\{x_k\}$, extends to a regular map on all of $X$.

\medskip

Let 
$$S^{\Conf}_{\Conf}\overset{\bj_{\Conf}}\hookrightarrow \ol{S}{}^{\Conf}_{\Conf}$$
be the open subfunctor, where we require the composite map \eqref{e:S- Ran} to have no zeroes on $X$. 

\medskip

Let $p^{\Conf}$ (resp., $\ol{p}^{\Conf}$) denote the projection $S^{\Conf}_{\Conf}\to \Conf$
(resp., $\ol{S}^{\Conf}_{\Conf}\to \Conf$).

\sssec{}

We can also think of 
$$S^{\Conf}_{\Conf}\subset \Gr^{\omega^\rho}_{G,\Conf}$$ as follows: 

\medskip

The embedding $T\to G$ gives rise to a map $\Gr^{\omega^\rho}_{T,\Conf}\to \Gr^{\omega^\rho}_{G,\Conf}$. In addition,
the projection
$$\Gr^{\omega^\rho}_{T,\Conf}\to \Conf$$
has a canonical section. Composing, we obtain a section 
$$\Conf\to \Gr^{\omega^\rho}_{G,\Conf}.$$

\medskip

Then $S^{\Conf}_{\Conf}$ is the orbit of the group $\fL(N)^{\omega^\rho}_{\Conf}$ acting on the above section. 

\sssec{}

Let 
$$(\Conf\times \Ran)^{\subset}\subset \Ran\times \Conf$$
be the ind-closed subfunctor corresponding to the following condition:

\medskip  

An $S$-point of $$(\CI\subset \Hom(S,X);D\in \Hom(S,\Conf))$$ belongs to $(\Conf\times \Ran)^{\subset}$ if and 
only if the support of the divisor $D$ is \emph{set-theoretically} contained in the union of the graphs of the maps $i:S\to X, i\in \CJ$
(cf. \cite[Sect. 1.3.2]{Ga7}).

\medskip

Note that we have a canonical identification
\begin{equation} \label{e:Gr Conf incl}
\Gr^{\omega^\rho}_{G,\Ran}\underset{\Ran}\times (\Conf\times \Ran)^{\subset}\simeq
\Gr^{\omega^\rho}_{G,\Conf}\underset{\Conf}\times (\Conf\times \Ran)^{\subset}.
\end{equation}

\medskip

Let $\on{pr}_\Ran$ denote the projection $(\Conf\times \Ran)^{\subset}\to \Conf$. 

\sssec{}

Denote
$$\ol{S}^{\Conf}_{\Ran}:=(\Conf\times \Ran)^{\subset}\underset{\Conf}\times \ol{S}^{\Conf}_{\Conf}.$$

\medskip

Denote by $\ol{p}^{\Conf}_\Ran$ the projection 
$$\ol{S}^{\Conf}_{\Ran}\to (\Conf\times \Ran)^{\subset}.$$ 

\medskip

Note that the identification \eqref{e:Gr Conf incl} realizes $\ol{S}^{\Conf}_{\Ran}$ as a closed subfunctor in
$$(\Conf\times \Ran)^{\subset}\underset{\Ran}\times \ol{S}^0_{\Ran}.$$

Let $\ol\bi{}^{\Conf}_{\Ran}$ denote the composite map
$$\ol{S}^{\Conf}_\Ran\hookrightarrow (\Conf\times \Ran)^{\subset}\underset{\Ran}\times \ol{S}^0_{\Ran}\to
\ol{S}^0_{\Ran}.$$

Note that the map $\ol\bi{}^{\Conf}_{\Ran}$ is proper. 

\sssec{}

Denote
$$S^{\Conf}_{\Ran}:=(\Conf\times \Ran)^{\subset}\underset{\Conf}\times S^{\Conf}_{\Conf}.$$

Denote by $p^{\Conf}_\Ran$ the projection 
$$S^{\Conf}_{\Ran}\to (\Conf\times \Ran)^{\subset}.$$

Denote by $\bj^{\Conf}_{\Ran}$ the open embedding 
$$S^{\Conf}_{\Ran}\hookrightarrow \ol{S}^{\Conf}_{\Ran}.$$

\sssec{}

For $\lambda\in \Lambda^{\on{neg}}-0$, denote 
$$(\Conf^\lambda\times \Ran)^{\subset}:=(\Conf\times \Ran)^{\subset}\underset{\Conf}\times
\Conf^\lambda.$$

\medskip

Let $\on{pr}^\lambda_\Ran$ denote the restriction of the map $\on{pr}_\Ran$ to $(\Conf^\lambda\times \Ran)^{\subset}$. 
Denote also 
$$\ol{S}^\lambda_\Ran:=(\Conf^\lambda\times \Ran)^{\subset}\underset{(\Conf\times \Ran)^{\subset}}\times \ol{S}^{\Conf}_\Ran$$
and
$$S^\lambda_\Ran:=(\Conf^\lambda\times \Ran)^{\subset}\underset{(\Conf\times \Ran)^{\subset}}\times S^{\Conf}_\Ran.$$

Denote by $\bj^\lambda_{\Ran}$ the resulting map 
$$S^\lambda_\Ran\to \ol{S}^\lambda_\Ran$$
and by $\ol\bi{}^\lambda_{\Ran}$ the corresponding map
$$\ol{S}^\lambda_\Ran \to \ol{S}^{\Conf}_\Ran \overset{\ol\bi{}_{\Ran}}\longrightarrow \ol{S}^0_{\Ran}.$$
Denote
$$\bi^\lambda_\Ran:=\ol\bi{}^\lambda_\Ran\circ \bj^\lambda_{\Ran}.$$

\medskip

Denote by $\ol{p}^\lambda_\Ran$ (resp., $p^\lambda_\Ran$) the restriction of $\ol{p}_\Ran$ (resp., $p_\Ran$) to
$\ol{S}^\lambda_\Ran$ (resp., $S^\lambda_\Ran$).

\sssec{} \label{sss:S Conf}

We extend the above definitions to formally include the case of $\lambda=0$,  in which case we set
$$\Conf^0=\on{pt}.$$

\medskip

We let
$$S^0_{\Ran}\overset{\bj_{\Ran}}\hookrightarrow \ol{S}^0_{\Ran}$$
be the open subfunctor, where we require that the map \eqref{e:Plucker again again} have no zeroes. 

\medskip

We have: 
$$\ol\bi{}^0_\Ran=\on{id},\,\, \bi^0_\Ran=\bj^0_\Ran=\bj_\Ran,\,\,(\Conf^0\times \Ran)^{\subset}=\Ran,$$
$p^0_\Ran$ (resp., $\ol{p}^0_\Ran$) is the map $S^0_\Ran\to \Ran$ (resp., $\ol{S}^0_\Ran\to \Ran$), and  
$\on{pr}^0_\Ran$ is the projection $\Ran\to \on{pt}$. 

\sssec{}

The following results easily from the definitions:

\begin{lem}  \label{l:decomp S-}
The map
$$\bi^\lambda_{\Ran}:S^\lambda_\Ran \to \ol{S}^0_{\Ran}$$
is a locally closed embedding. Every field-valued point of $\ol{S}^0_{\Ran}$ belongs to the image
of exactly one such map.
\end{lem}

\ssec{Stratification of the category}

The strata $S^\lambda_\Ran$ of $\ol{S}^0_{\Ran}$ give rise to a \emph{recollement} pattern on $\SI_{q,\Ran}(G)$. 

\sssec{}

Let $\SI_{q,\Ran}(G)^{\leq \lambda}$ denote the full subcategory of $\Shv_{\CG^G}(\ol{S}^\lambda_\Ran)$
given by the condition of equivariance with respect to the pullback of $\fL(N)^{\omega^\rho}_{\Ran}$ to 
$(\Conf^\lambda\times \Ran)^{\subset}$. 

\medskip

Let $\SI_{q,\Ran}(G)^{=\lambda}$
be the corresponding full subcategory of 
$\Shv_{\CG^G}(S^\lambda_\Ran)$.

\sssec{}

The maps $\ol\bi{}^\lambda_{\Ran}$ and $\bj^\lambda_{\Ran}$ define pairs of mutually adjoint functors
$$(\ol\bi{}^\lambda_{\Ran})_!=(\ol\bi{}^\lambda_{\Ran})_*:
\SI_{q,\Ran}(G)^{\leq \lambda}\rightleftarrows \SI_{q,\Ran}(G)^{\leq 0}:(\ol\bi{}^\lambda_{\Ran})^!;$$
$$(\bj^\lambda_{\Ran})^*=(\bj^\lambda_{\Ran})^!:\SI_{q,\Ran}(G)^{\leq \lambda} 
\rightleftarrows \SI_{q,\Ran}(G)^{=\lambda}:(\bj^\lambda_{\Ran})_*.$$

\sssec{}

In addition, as in \cite[Corollary 1.4.5]{Ga7}, one shows that the partially defined left adjoint $(\bi{}^\lambda_{\Ran})^*$ of 
$$(\bi{}^\lambda_{\Ran})_*:=(\ol\bi{}^\lambda_{\Ran})_*\circ (\bj{}^\lambda_{\Ran})_*$$ is defined on 
$$\SI_{q,\Ran}(G)^{\leq 0}\subset \Shv_{\CG^G}(\ol{S}^0_{\Ran}),$$ 
giving rise to an adjoint pair 
$$(\ol\bi{}^\lambda_{\Ran})^*:\SI_{q,\Ran}(G)^{\leq 0} \rightleftarrows \SI_{q,\Ran}(G)^{\leq \lambda}:(\bi{}^\lambda_{\Ran})_*.$$

\sssec{}
Also, the partially defined left adjoint $(\bi{}^\lambda_{\Ran})_!$ of $(\bi{}^\lambda_{\Ran})^!$ is defined on  
$$\SI_{q,\Ran}(G)^{=\lambda}\subset \Shv_{\CG^G}(S^\lambda_{\Ran}),$$
giving rise to an adjoint pair 
$$(\bi^\lambda_{\Ran})_!:\SI_{q,\Ran}(G)^{=\lambda}\rightleftarrows \SI_{q,\Ran}(G)^{\leq 0}:(\bi{}^\lambda_{\Ran})^!.$$

\ssec{Description of the category on a stratum}

In this subsection we will describe explicitly the category $\SI_{q,\Ran}(G)^{=\lambda}$. It will turn out to be equivalent 
to the category of (gerbe-twisted) sheaves on $(\Conf^\lambda\times \Ran)^{\subset}$. 

\sssec{} 

We first observe:  

\begin{prop}  \label{p:gerbe on stratum}
The pullback of the gerbe $\CG^G$ along the map
\begin{equation} \label{e:1st map to GrG}
S^{\Conf}_\Ran\to \ol{S}^0_{\Ran}\to \Gr^{\omega^\rho}_{G,\Ran}
\end{equation} 
identifies canonically with the pullback of the gerbe $\CG^\Lambda$ on $\Conf$ 
of \secref{sss:gerbe Lambda} along the map 
\begin{equation} \label{e:2nd map to Conf}
S^{\Conf}_\Ran \overset{p_\Ran}\longrightarrow
(\Conf\times \Ran)^{\subset}  \overset{\on{pr}_\Ran}\longrightarrow \Conf.
\end{equation} 
\end{prop} 

The proof will be essentially a diagram chase, modulo the additional structure on the 
gerbe $\CG^G$ specified in \secref{sss:unital gerbe}. 

\begin{proof} 

By construction, the map \eqref{e:1st map to GrG}
factors as
$$S^{\Conf}_\Ran\to \Gr^{\omega^\rho}_{B,\Ran} \to \Gr^{\omega^\rho}_{G,\Ran}.$$

By definition, the correspondence between $\CG^G$ and $\CG^T$ is such that their
pullbacks to $\Gr^{\omega^\rho}_{B,\Ran}$ along the maps
$$\Gr^{\omega^\rho}_{G,\Ran} \leftarrow \Gr^{\omega^\rho}_{B,\Ran}\to \Gr^{\omega^\rho}_{T,\Ran}$$
are identified, see \secref{sss:from G to T gerbes}. 

\medskip

Hence, it suffices to show that pullback of $\CG^T$ along
\begin{equation} \label{e:1st map to GrT}
S^{\Conf}_\Ran\to \Gr^{\omega^\rho}_{B,\Ran} \to \Gr^{\omega^\rho}_{T,\Ran}
\end{equation} 
identifies with the pullback of $\CG^\Lambda$ along the map \eqref{e:2nd map to Conf}.

\medskip

In order to do so, we can replace $S^{\Conf}_\Ran$ by
$$(\Gr^{\omega^\rho}_{T,\Ran})^{\on{neg}}\underset{\Conf}\times S^{\Conf}_\Ran,$$
(see \secref{sss:neg part Gr} for the notation), which identifies with
$$S^{\Conf}_{\Conf}\underset{\Conf}\times 
(\Gr^{\omega^\rho}_{T,\Ran})^{\on{neg}}\underset{\Gr^{\omega^\rho}_{T,\Ran}}\times \Gr^{\omega^\rho}_{T,(\Ran\times \Ran)^{\subset}},$$
where the map $\Gr^{\omega^\rho}_{T,(\Ran\times \Ran)^{\subset}}\to \Gr^{\omega^\rho}_{T,\Ran}$ is $\varphi_{\on{small}}$,
see \secref{sss:unitality Gr} for the notation.

\medskip

With respect to this identification, the composition
$$(\Gr^{\omega^\rho}_{T,\Ran})^{\on{neg}}\underset{\Conf}\times S^{\Conf}_\Ran \to
S^{\Conf}_\Ran  \overset{\text{\eqref{e:2nd map to Conf}}}\longrightarrow \Conf$$
identifies with
$$S^{\Conf}_{\Conf}\underset{\Conf}\times 
(\Gr^{\omega^\rho}_{T,\Ran})^{\on{neg}}\underset{\Gr^{\omega^\rho}_{T,\Ran}}\times \Gr^{\omega^\rho}_{T,(\Ran\times \Ran)^{\subset}}
\to (\Gr^{\omega^\rho}_{T,\Ran})^{\on{neg}}\to \Conf.$$

Hence, the pullback of $\CG^\Lambda$ along this map identifies with the pullback of $\CG^T$ along the map
$$S^{\Conf}_{\Conf}\underset{\Conf}\times 
(\Gr^{\omega^\rho}_{T,\Ran})^{\on{neg}}\underset{\Gr^{\omega^\rho}_{T,\Ran}}\times \Gr^{\omega^\rho}_{T,(\Ran\times \Ran)^{\subset}}
\to \Gr^{\omega^\rho}_{T,(\Ran\times \Ran)^{\subset}} \overset{\varphi_{\on{small}}}\longrightarrow \Gr^{\omega^\rho}_{T,\Ran}.$$

The composition
$$(\Gr^{\omega^\rho}_{T,\Ran})^{\on{neg}}\underset{\Conf}\times S^{\Conf}_\Ran \to
S^{\Conf}_\Ran  \overset{\text{\eqref{e:1st map to GrT}}}\longrightarrow \Gr^{\omega^\rho}_{T,\Ran}$$
identifies with 
 $$S^{\Conf}_{\Conf}\underset{\Conf}\times 
(\Gr^{\omega^\rho}_{T,\Ran})^{\on{neg}}\underset{\Gr^{\omega^\rho}_{T,\Ran}}\times \Gr^{\omega^\rho}_{T,(\Ran\times \Ran)^{\subset}}
\to \Gr^{\omega^\rho}_{T,(\Ran\times \Ran)^{\subset}} \overset{\varphi_{\on{big}}}\longrightarrow \Gr^{\omega^\rho}_{T,\Ran}.$$

Hence, the required isomorphism follows from \corref{c:gerbe unital}.

\end{proof}

\sssec{}

From \propref{p:gerbe on stratum}, we obtain that we have a canonically defined pullback functor
$$p_\Ran^!:\Shv_{\CG^\Lambda}((\Conf\times \Ran)^{\subset})\to 
\Shv_{\CG^G}(S^{\Conf}_\Ran),$$
and for every individual $\lambda\in \Lambda^{\on{neg}}$, a functor 
$$(p^\lambda_\Ran)^!:\Shv_{\CG^\Lambda}((\Conf^\lambda\times \Ran)^{\subset})\to 
\Shv_{\CG^G}(S^\lambda_\Ran),$$

As in \lemref{l:semiinf strata} (see also \cite[Lemma 1.4.8]{Ga7}, we have:

\begin{lem} \label{l:semiinf strata Ran}
For every $\lambda\in \Lambda^{\on{neg}}$, 
functor $(p^\lambda_\Ran)^!$ induces an equivalence
$$\Shv_{\CG^\Lambda}((\Conf^\lambda\times \Ran)^{\subset})\to \SI_{q,\Ran}(G)^{=\lambda}.$$
\end{lem}

\ssec{The unital subcategory}

As we just saw in \lemref{l:semiinf strata Ran}, the category $\SI_{q,\Ran}(G)^{=\lambda}$ is equivalent to
that of (gerbe-twisted) sheaves on $(\Conf^\lambda\times \Ran)^{\subset}$. This category is too large for
our needs: it contains the ``junk" directions that have to do with $\Ran$, and we would like to cut those down.

\medskip

A device to do so is the \emph{unitality structure}.

\sssec{}

Let us recall the setting of \secref{ss:unital}. Note that according to \lemref{l:UHC unital}, the pullback functor
$$\varphi_{\on{small}}:\Shv_{\CG^G}(\Gr_{G,\Ran})\to \Shv_{\CG^G}(\Gr^{\omega^\rho}_{G,(\Ran\times \Ran)^{\subset}})$$
is fully faithful.

\medskip

We define the category $\Shv_{\CG^G}(\Gr_{G,\Ran})_{\on{untl}}$ to consist of objects
$\CF\in \Shv_{\CG^G}(\Gr^{\omega^\rho}_{G,\Ran})$ equipped with an isomorphism
$$\varphi_{\on{small}}^!(\CF)\simeq \varphi_{\on{big}}^!(\CF)$$
and an identification of the composite map
$$\CF\simeq \Delta_{\Ran}!\circ \varphi_{\on{small}}^!(\CF)\simeq \Delta_{\Ran}!\circ \varphi_{\on{big}}^!(\CF)\simeq \CF$$
with the identity map on $\CF$. 

\medskip

Note that due to the fully faithfulness result quoted above, the forgetful functor
$$\Shv_{\CG^G}(\Gr^{\omega^\rho}_{G,\Ran})_{\on{untl}}\to \Shv_{\CG^G}(\Gr^{\omega^\rho}_{G,\Ran})$$
is fully faithful.

\sssec{}

Define
$$\SI_{q,\Ran}(G)_{\on{untl}}:=\SI_{q,\Ran}(G)\cap \Shv_{\CG^G}(\Gr^{\omega^\rho}_{G,\Ran})_{\on{untl}}\subset \Shv_{\CG^G}(\Gr_{G,\Ran}).$$

\sssec{}

Along with $\Gr^{\omega^\rho}_{G,(\Ran\times \Ran)^{\subset}}$, we can consider the prestacks $$\ol{S}^0_{(\Ran\times \Ran)^{\subset}},\,\, 
\ol{S}^{\Conf}_{(\Ran\times \Ran)^{\subset}},\,\, S^{\Conf}_{(\Ran\times \Ran)^{\subset}},$$ etc. 

\medskip

Proceeding as above, we define the full subcategories
$$\Shv_{\CG^G}(\ol{S}^0_{(\Ran\times \Ran)^{\subset}})_{\on{untl}}\subset 
\Shv_{\CG^G}(\ol{S}^0_{(\Ran\times \Ran)^{\subset}}),$$
$$\Shv_{\CG^G}(\ol{S}^\lambda_{(\Ran\times \Ran)^{\subset}})_{\on{untl}}\subset 
\Shv_{\CG^G}(\ol{S}^\lambda_{(\Ran\times \Ran)^{\subset}})$$
and $$\Shv_{\CG^G}(S^\lambda_{(\Ran\times \Ran)^{\subset}})_{\on{untl}}\subset 
\Shv_{\CG^G}(S^\lambda_{(\Ran\times \Ran)^{\subset}}),$$
as well as
$$\SI_{q,\Ran}(G)^{\leq 0}_{\on{untl}}\subset \SI_{q,\Ran}(G)^{\leq 0},$$
$$\SI_{q,\Ran}(G)_{\on{untl}}^{\leq \lambda}\subset \SI_{q,\Ran}(G)^{\leq \lambda},$$
$$\SI_{q,\Ran}(G)_{\on{untl}}^{=\lambda}\subset \SI_{q,\Ran}(G)^{=\lambda}.$$

\medskip

It is clear that the functors $(\ol\bi{}^\lambda_\Ran)^!$, $(\bj^\lambda_\Ran)^!$ (and hence also $(\bi{}^\lambda_\Ran)^!$)
as well as $(\ol\bi{}^\lambda_\Ran)_*$, $(\bj^\lambda_\Ran)_*$ (and hence also $(\bi{}^\lambda_\Ran)_*$)
preserve the corresponding unital subcategories. 

\medskip

In addition, as in \cite[Proposition 4.2.2]{Ga7}, we have:

\begin{prop}
The functors $(\bi{}^\lambda_\Ran)_!$ and $(\bi^\lambda_\Ran)^*$
also preserve the unital subcategories.
\end{prop}

\sssec{}

By a similar token we define the full subcategory
$$\Shv_{\CG^\Lambda}((\Conf\times \Ran)^{\subset})_{\on{untl}}\subset
\Shv_{\CG^\Lambda}((\Conf\times \Ran)^{\subset}).$$

\medskip

It follows from \lemref{l:semiinf strata Ran} that the functor $(p^\lambda_\Ran)^!$ induces an equivalence
\begin{equation} \label{e:pullback to stratum unital}
\Shv_{\CG^\Lambda}((\Conf^\lambda\times \Ran)^{\subset})_{\on{untl}}\to \SI_{q,\Ran}(G)^{=\lambda}_{\on{untl}}.
\end{equation} 

\sssec{}

The following is proved in \cite[Proposition 4.2.7]{Ga7}:

\begin{prop} \label{p:config contr}
Pullback with respect to
$$\on{pr}^\lambda_\Ran:(\Conf^\lambda\times \Ran)^{\subset}\to \Conf^\lambda$$
defines an equivalencce
$$\Shv_{\CG^\Lambda}(\Conf^\lambda)\to \Shv_{\CG^\Lambda}((\Conf^\lambda\times \Ran)^{\subset})_{\on{untl}}.$$
\end{prop}

\sssec{}

Combining, we obtain the following explicit description of the category $\SI_{q,\Ran}(G)^{=\lambda}_{\on{untl}}$:

\begin{cor} \label{c:unital semiinf on stratum}
For every $\lambda\in \Lambda^{\on{neg}}$, pullback with respect to $\on{pr}^\lambda_\Ran\circ p^\lambda_\Ran$ defines an equivalence
$$\Shv_{\CG^\Lambda}(\Conf^\lambda)\to \SI_{q,\Ran}(G)^{=\lambda}_{\on{untl}}.$$
\end{cor}

\sssec{Example}

For $\lambda=0$ the above corollary says that the functor
$$\sfe\mapsto \omega_{S^0_{\Ran}}$$
defines an equivalence
$$\Vect\to \SI_{q,\Ran}(G)^{=0}_{\on{untl}}.$$

\section{The metaplectic semi-infinite IC sheaf}  \label{s:ICs}

In this section we finally construct the main object of study in this Part: the metaplectic semi-infinite IC sheaf $\ICs_{q,\Ran}$. 

\ssec{The t-structure on the semi-infinite category}  \label{ss:t-str semiinf}

In this subsection we will introduce a t-structure on $\SI_{q,\Ran}(G)^{\leq 0}_{\on{untl}}$, and as a result
we will define the metaplectic semi-infinite IC sheaf $\ICs_{q,\Ran}$. 

\sssec{}  \label{sss:t-str on config Ran}

We define a t-structure on the category $\Shv_{\CG^\Lambda}((\Conf^\lambda\times \Ran)^{\subset})_{\on{untl}}$ 
via the equivalence of \propref{p:config contr}:
$$(\on{pr}^\lambda_\Ran)^!:\Shv_{\CG^\Lambda}(\Conf^\lambda)\overset{\sim}\to 
\Shv_{\CG^\Lambda}((\Conf^\lambda\times \Ran)^{\subset})_{\on{untl}},$$
i.e., we transfer the (perverse) t-structure from $\Shv_{\CG^\Lambda}(\Conf^\lambda)$. 

\sssec{}

We define a t-structure on $\SI_{q,\Ran}(G)^{=\lambda}_{\on{untl}}$ via the equivalence of \eqref{e:pullback to stratum unital},
\emph{up to a cohomological shift by $\langle \lambda,2\check\rho \rangle$.}

\medskip

Namely, we declare on object of $\SI_{q,\Ran}(G)^{=\lambda}_{\on{untl}}$ to be connective/coconnective
if and only if its
image in $\Shv_{\CG^\Lambda}((\Conf^\lambda\times \Ran)^{\subset})_{\on{untl}}$, shifted cohomologically by 
$[\langle \lambda,2\check\rho \rangle]$, is connective/coconnective. 

\medskip

We apply a similar procedure to define a t-structure on 
$\SI_{q,\Ran}(G)^{=0}_{\on{untl}}$
via its identification with $\Vect$.

\sssec{}  \label{sss:t-str semiinf}

We define a structure on $\SI_{q,\Ran}(G)^{\leq 0}_{\on{untl}}$ by declaring an object to be coconnective if and only if its
image under $(\bi^\lambda_{\Ran})^!$ is coconnective in $\SI_{q,\Ran}(G)^{=\lambda}_{\on{untl}}$ 
for every $\lambda$. 

\sssec{}

By construction, the functors 
$$(\bi^\lambda_{\Ran})_!:\SI_{q,\Ran}(G)^{=\lambda}_{\on{untl}}\to \SI_{q,\Ran}(G)^{\leq 0}_{\on{untl}}$$
are right t-exact, and the functors
$$(\bi^\lambda_{\Ran})_*:\SI_{q,\Ran}(G)^{=\lambda}_{\on{untl}}\to \SI_{q,\Ran}(G)^{\leq 0}_{\on{untl}}$$
are left t-exact. From here, we obtain that the functors 
$$(\bi^\lambda_{\Ran})^*:\SI_{q,\Ran}(G)^{\leq 0}_{\on{untl}}\to \SI_{q,\Ran}(G)^{=\lambda}_{\on{untl}}$$
are right t-exact.

\medskip

Moreover, as in \cite[Lemma 2.1.8]{Ga7}, we have:

\begin{lem}
An object $\CF\in \SI_{q,\Ran}(G)^{\leq 0}_{\on{untl}}$ is connective if and only if
$(\bi^\lambda_{\Ran})^*(\CF)\in \SI_{q,\Ran}(G)^{=\lambda}_{\on{untl}}$
is connective for every $\lambda$.
\end{lem}

\sssec{}  \label{sss:IC semiinf}

We let 
$$\ICs_{q,\Ran}\in (\SI_{q,\Ran}(G)^{\leq 0}_{\on{untl}})^\heartsuit\subset \SI_{q,\Ran}(G)^{\leq 0}_{\on{untl}}$$ 
be the object equal to the image of the resulting map
$$H^0((\bj_{\Ran})_!(\omega_{S^0_{\Ran}}))\to 
H^0((\bj_{\Ran})_*(\omega_{S^0_{\Ran}})),$$
where $H^0$ refers to the t-structure on $\SI_{q,\Ran}(G)^{\leq 0}_{\on{untl}}$ introduced in \secref{sss:t-str semiinf} above. 

\begin{rem}
For the purposes of this paper we will largely forget the unitality property of $\ICs_{q,\Ran}$, and regard it 
simply as an object of $\SI_{q,\Ran}(G)^{\leq 0}$. 

\medskip

We emphasize, however, that we needed the unitality device in order to have a 
well-behaved t-structure. 
\end{rem} 

\sssec{}

For future use we record:

\begin{lem} \label{l:non sharp strata}
The !-restriction of $\ICs_{q,\Ran}$ to $S^\lambda_\Ran$ is zero unless $\lambda\in \Lambda^\sharp$.
\end{lem}

The proof will be given in \secref{sss:proof of non-sharp}. 

\sssec{}

In the rest of this section we will discuss various properties of $\ICs_{q,\Ran}$:

\medskip

\noindent--We will explicitly describe its restriction to 
$$\Gr_{G,x}^{\omega^\rho}\subset \Gr_{G,\Ran};$$

\medskip

\noindent--We will endow it with a \emph{factorization structure};

\medskip

\noindent--We will describe its relationship with the \emph{global} metaplectic
semi-infinite IC sheaf. 

\ssec{Description of the fiber}  \label{ss:fiber of semiinf}

The definition of $\ICs_{q,\Ran}$ might appear as very abstract, and the result object may seem 
hardly calculable. This turns out to be not so. 

\medskip

In this subsection we will describe explicitly its restriction
to $\Gr^{\omega^\rho}_{G,x}\subset \Gr^{\omega^\rho}_{G,\Ran}$, i.e., the fiber of $\Gr^{\omega^\rho}_{G,\Ran}$
over the point $\{x\}\in \Ran$. 

\sssec{}

For an element $\gamma\in \Lambda^\sharp$ consider the corresponding point 
$$t^{\gamma}\in S^{\gamma} \subset \Gr^{\omega^\rho}_{G,x}.$$

The trivialization of the gerbe $\CG_{T_H,x}$ in \secref{sss:trivialize gerbe at point} gives rise to trivialization of 
 $\CG^G|_{t^\gamma}$. Hence, $\delta_{t^\gamma,\Gr}$ makes sense as an object of $\Shv_{\CG}(\Gr^{\omega^\rho}_{G,x})$. 

\sssec{}

Assume now that $\gamma$ is dominant. Let $V^{\gamma}$ denote the corresponding irreducible
object in $\Rep(H)$ (see our conventions in \secref{sss:descr of weight spaces}). It follows from 
\eqref{e:descr h.w.} that the !-fiber of $\Sat_{q,G}(V^{\gamma})$ at $t^{\gamma}$ identifies canonically with 
$\sfe[-\langle \gamma,2\check\rho\rangle]$.

\medskip

Hence, we obtain a canonically defined map
\begin{equation} \label{e:transition map pre initial}
\delta_{t^\gamma,\Gr}[-\langle \gamma,2\check\rho\rangle]\to \Sat_{q,G}(V^\gamma). 
\end{equation}

By adjunction, we obtain a map
\begin{equation} \label{e:transition map initial}
\delta_{1,\Gr}\to \delta_{t^{-\gamma},\Gr}\star \Sat_{q,G}(V^{\gamma})[\langle \gamma,2\check\rho\rangle].
\end{equation}

\sssec{}

We consider $(\Lambda^\sharp)^+$ as a poset with respect to the order relation that
$$\gamma_2\preceq \gamma_1\, \Leftrightarrow\, \gamma_2-\gamma_1\in (\Lambda^\sharp)^+.$$

Note that this is a \emph{different} order relation than the standard one denoted $\leq$.
Note also that, taken with respect to $\preceq$, the set $(\Lambda^\sharp)^+$ is \emph{filtered}.

\sssec{}

For $\gamma_2\succeq \gamma_1$ with $\gamma_2-\gamma_1=\gamma\in (\Lambda^\sharp)^+$
we define the map
$$\delta_{t^{-\gamma_1},\Gr}\star \Sat_{q,G}(V^{\gamma_1})[\langle \gamma_1,2\check\rho\rangle]\to 
\delta_{t^{-\gamma_2},\Gr}\star \Sat_{q,G}(V^{\gamma_2})[\langle \gamma_2,2\check\rho\rangle]$$ to be the composite
\begin{multline}  \label{e:transition maps ICs}
\delta_{t^{-\gamma_1},\Gr}\star \Sat_{q,G}(V^{\gamma_1})[\langle \gamma_1,2\check\rho\rangle] \simeq
\delta_{t^{-\gamma_1},\Gr}\star \delta_{1,\Gr}\star \Sat_{q,G}(V^{\gamma_1})[\langle \gamma_1,2\check\rho\rangle] 
\overset{\text{\eqref{e:transition map initial}}}\longrightarrow \\
\to \delta_{t^{-\gamma_1},\Gr}\star \delta_{t^{-\gamma},\Gr}\star \Sat_{q,G}(V^{\gamma}) \star \Sat_{q,G}(V^{\gamma_1})
[\langle \gamma_2,2\check\rho\rangle] \simeq \delta_{t^{-\gamma_2},\Gr}\star 
\Sat_{q,G}(V^{\gamma}) \star \Sat_{q,G}(V^{\gamma_1})[\langle \gamma_2,2\check\rho\rangle] \to \\
\to \delta_{t^{-\gamma_2},\Gr}\star \Sat_{q,G}(V^{\gamma_2})[\langle \gamma_2,2\check\rho\rangle]
\end{multline} 
where the last arrow is induced by the Pl\"ucker map
$$V^{\gamma}\otimes V^{\gamma_1}\to V^{\gamma_2},$$
normalized so that it is the \emph{idenitity} map on the (trivialized) highest weight lines. 

\sssec{}

As in \cite[Sect. 2.3]{Ga6}, one shows that the assingment
$$\gamma \rightsquigarrow \delta_{t^{-\gamma},\Gr}\star \Sat_{q,G}(V^{\gamma})[\langle \gamma,2\check\rho\rangle]$$
together with the above transition maps define a functor of $\infty$-categories
$$((\Lambda^\sharp)^+,\preceq) \to \Shv_{\CG^G}(\Gr^{\omega^\rho}_{G,x}).$$

Define:
\begin{equation} \label{e:ICs as colim}
'\!\ICs_{q,x}:=\underset{\gamma\in ((\Lambda^\sharp)^+,\leq)}{\on{colim}}\, 
\delta_{t^{-\gamma},\Gr}\star \Sat_{q,G}(V^{\gamma})[\langle \gamma,2\check\rho\rangle]\in \Shv_{\CG^G}(\Gr^{\omega^\rho}_{G,x}).
\end{equation}

As in \cite[Proposition 2.3.7]{Ga6}, one shows: 

\begin{prop} \hfill

\smallskip

\noindent{\em(a)} The object $'\!\ICs_{q,x}$ belongs to $\SI_{q,x}(G)$. 

\smallskip

\noindent{\em(b)} The object $'\!\ICs_{q,x}$ is supported on $\ol{S}^0\subset \Gr^{\omega^\rho}_{G,x}$. 

\smallskip

\noindent{\em(c)} There is a canonical identification $(\bj^0)^!({}'\!\ICs_{q,x})\simeq \omega_{S^0}$. 

\end{prop}

\sssec{}

A metaplectic analog of the isomorphism of \cite[Corollary 2.7.7]{Ga7} reads:

\begin{thm}  \label{t:fiber of Ics}
There exists a canonical isomorphism between 
$$\ICs_{q,x}:=\ICs_{q,\Ran}|_{\Gr^{\omega^\rho}_{G,x}}$$
and $'\!\ICs_{q,x}$.
\end{thm}

\sssec{}

In what follows we will need the following property of $\ICs_{q,x}$:

\begin{prop}  \label{p:sharper estimate}
The !-restriction of $\ICs_{q,x}$ to a stratum $S^\lambda$ with $\lambda\neq 0$ is of the form 
$$\omega_{S^\lambda}[-\langle \lambda,2\check\rho\rangle]\otimes \sK_{x,\lambda},$$
where $\sK_\lambda$ is an object of $\Shv_{\CG^\Lambda_{\lambda\cdot x}}(\on{pt})$ that
lives in cohomological degrees $\geq 2$.
\end{prop}

\begin{proof}

Consider the restriction of $\ICs_{q,\Ran}$ to the stratum
$$X\underset{\Conf^\lambda}\times S^\lambda_\Ran,$$
where $X\to \Conf^\lambda$ is the diagonal map $x\mapsto \lambda\cdot x$.

\medskip

This restriction is the pullback of an object 
$$\sK_{X,\lambda}[-\langle \lambda,2\check\rho\rangle]\in \Shv_{\CG^\Lambda|_X}(X),$$
and $\sK_{x,\lambda}$ is the !-fiber of $\sK_{X,\lambda}$ at $x$. A priori, $\sK_{X,\lambda}$
lives in cohomological degrees $>0$. 

\medskip

Now, the expression of $\ICs_{q,x}$ given by \eqref{e:ICs as colim} implies that $\sK_{x,\lambda}$ 
would be (non-canonically) the same for \emph{another choice} of a curve $X$ and a point $x$ on it.
So for the purposes of proving the cohomological estimate, we can assume that $X=\BA^1$ with
a choice of geometric metaplectic datum which is translation-invariant with the same quadratic 
form $q$.  
 
\medskip

Since the assignment $X\rightsquigarrow \sK_{X,\lambda}$ respects automorphisms of
the situation, we obtain that $\sK_{X,\lambda}$ is translation-invariant, and in particular lisse. 
This implies that the !-fiber $\sK_{X,\lambda}$ at $x$ lives in degrees $>1$, as required.

\end{proof}

\ssec{Factorization}

We will now specify a key structure possessed by $\ICs_{q,\Ran}$: that of factorization algebra.
This structure will be used to define a factorization structure on the Jacquet functor
$$\Whit_{q,x}(G)\to \Shv_{\CG^T}(\Gr_T).$$

\sssec{}

Parallel to the factorization structure on $\Gr^{\omega^\rho}_{G,\Ran}$ (see \secref{sss:aff Gr}), 
we also have a factorization on $\ol{S}^0_{\Ran}$ and $\ol{S}_{\Ran}$:  
\begin{equation} \label{e:fact S0}
\ol{S}^0_{\Ran} \underset{\Ran}\times (\Ran^J)_{\on{disj}}
\simeq (\ol{S}^0_{\Ran})^J \underset{\Ran^J}\times (\Ran^J)_{\on{disj}}. 
\end{equation}
\begin{equation} \label{e:fact no bar S0}
S^0_{\Ran} \underset{\Ran}\times (\Ran^J)_{\on{disj}}
\simeq (S^0_{\Ran})^J \underset{\Ran^J}\times (\Ran^J)_{\on{disj}}. 
\end{equation}


\sssec{}

The following assertion is an extension of \cite[Sect. 4.6]{Ga7}:

\begin{thm} \label{t:fact ICs}
The object $\ICs_{q,\Ran}\in \Shv_{\CG^G}(\ol{S}^0_{\Ran})$ has a structure of factorization algebra
(see \secref{sss:fact alg geom}),
uniquely characterized by the requirement that the induced structure of factorization algebra on 
$\ICs_{q,\Ran}|_{S^0_{\Ran}}\in \Shv_{\CG^G}(S^0_{\Ran})$
corresponds to the tautological one on $\omega_{S^0_{\Ran}}$ (see \secref{sss:dualizing as an example of FA}) with respect to
the identification $$\ICs_{q,\Ran}|_{S^0_{\Ran}}\simeq \omega_{S^0_{\Ran}}$$
\end{thm} 

\begin{rem}
In fact, a slightly stronger assertion is true: for an individual finite set $J$, there exists 
factorization isomorphism
\begin{equation} \label{e:factor IC}
\ICs_{q,\Ran}|_{\ol{S}^0_{\Ran} \underset{\Ran}\times (\Ran^J)_{\on{disj}}}
\simeq   
(\ICs_{q,\Ran})^{\boxtimes J}|_{(\ol{S}^0_{\Ran})^J \underset{\Ran^J}\times (\Ran^J)_{\on{disj}}},
\end{equation} 
is \emph{uniquely characterized} by the requirement that the induced isomorphism 
$$\ICs_{q,\Ran}|_{S^0_{\Ran} \underset{\Ran}\times (\Ran^J)_{\on{disj}}}
\simeq   (\ICs_{q,\Ran})^{\boxtimes J}|_{(S^0_{\Ran})^J \underset{\Ran^J}\times (\Ran^J)_{\on{disj}}}$$
corresponds to the tautological one under
on $$\ICs_{q,\Ran}|_{S^0_{\Ran}}\simeq \omega_{S^0_{\Ran}}.$$ 
\end{rem} 

\begin{rem}
It is in the proof of \thmref{t:fact ICs} that the unitality property/structure on $\ICs_{q,\Ran}$ plays
the most essential role, see \cite[Sect. 4.6]{Ga7}: one shows that both sides in \eqref{e:factor IC}
are Goresky-MacPherson extensions of their respective restrictions to $S^0_{\Ran} \underset{\Ran}\times (\Ran^J)_{\on{disj}}$
for the appropriately defined t-structure. 
\end{rem}

\begin{rem}
The unitality structure is also helpful to combat difficulties of homotopy-theoretic nature: it allows us to 
consider all the objects involved in \eqref{e:factor IC}
as living inside appropriately defined \emph{abelian categories}; hence the
\emph{homotopy coherence} of factorization isomorphisms is automatic. 

\end{rem}

\ssec{Comparison with the global metaplectic IC sheaf}  \label{ss:semiinf glob}

In this subsection we will show how to express $\IC_{q,\Ran}$ in terms of finite-dimensional
algebraic geometry (specifically, in terms of gerbe-twisted sheaves on $\BunNbom$). 

\medskip

This description will be subsequently used in order to establish one of the key properties of the
Jacquet functor, namely, its commutation with Verdier duality. 

\sssec{}

Consider the algebraic stack $\BunNbom$. As in the case of 
$$\Whit_{q,\on{glob}}(G)\subset \Shv_{\CG^G}(\BunNbom),$$ 
one singles out a full subcategory denoted 
\begin{equation} \label{e:semiinf global}
\SI_{q,\on{glob}}(G)^{\leq 0}\subset \Shv_{\CG^G}(\BunNbom),
\end{equation} 
by imposing equivariance with respect to a certain unipotent groupoid, see \cite[Sects. 4.4-4.7]{Ga9}.
The subcategory \eqref{e:semiinf global} is compatible with the (perverse) t-structure on $\Shv_{\CG^G}(\BunNbom)$. 

\medskip

We will give an explicit description of this subcategory below, see \secref{sss:descr strata glob}. 

\sssec{}

For $\lambda\in \Lambda^{\on{neg}}$, let 
$$(\BunNbom)^{\leq \lambda}\overset{\ol\bi{}^\lambda_{\on{glob}}}\hookrightarrow \BunNbom$$
be the closed substack corresponding to the locus where the generalized $B$ structure has total defect at least $-\lambda$. 
Let 
$$(\BunNbom)^{=\lambda}\overset{\bj^\lambda_{\on{glob}}}\hookrightarrow (\BunNbom)^{\leq \lambda}$$
denote the open substack, where we the generalized $B$ structure has total defect exactly $\lambda$. 
Denote $$\bi^\lambda_{\on{glob}}=\ol\bi{}^\lambda_{\on{glob}}\circ \bj^\lambda_{\on{glob}}.$$

\medskip

In particular, for $\lambda=0$, we have $\ol\bi{}^0_{\on{glob}}=\on{id}$, and
$$(\BunNbom)^{\leq 0}=\ol\Bun^{\omega^\rho}_N \text{ and } (\BunNbom)^{=0}=\Bun^{\omega^\rho}_N,$$
and the map $\bj^0_{\on{glob}}=\bi^0_{\on{glob}}$ is 
the open embedding
$$\Bun_N^{\omega^\rho}\overset{\bj_{\on{glob}}}\hookrightarrow \ol\Bun^{\omega^\rho}_N.$$

\medskip

For $\lambda\in \Lambda^{\on{neg}}-0$ we have a canonical isomorphism
$$(\BunNbom)^{=\lambda}\simeq \Bun_B\underset{\Bun_T}\times \Conf^\lambda,$$
where 
\begin{equation} 
\on{AJ}^{\omega^\rho}:\Conf\to \Bun_T, \quad \underset{k}\Sigma\, \lambda_k\cdot x_k\mapsto 
\omega^\rho(\underset{k}\Sigma\, \lambda_k\cdot x_k).
\end{equation} 

\sssec{}

We have the corresponding full subcategories
$$\SI_{q,\on{glob}}(G)^{\leq \lambda}\subset \Shv_{\CG^G}((\BunNbom)^{\leq \lambda})$$
and
$$\SI_{q,\on{glob}}(G)^{=\lambda}\subset \Shv_{\CG^G}((\BunNbom)^{=\lambda}),$$
and the corresponding pairs of adjoint functors
$$(\bj^\lambda_{\on{glob}})_!:\SI_{q,\on{glob}}(G)^{=\lambda}\rightleftarrows \SI_{q,\on{glob}}(G)^{\leq \lambda}:(\bj^\lambda_{\on{glob}})^!,$$
$$(\bj^\lambda_{\on{glob}})^*=(\bj^\lambda_{\on{glob}})^!:\SI_{q,\on{glob}}(G)^{\leq \lambda}
\rightleftarrows \SI_{q,\on{glob}}(G)^{=\lambda}:(\bj^\lambda_{\on{glob}})_*;$$
$$(\ol\bi{}^\lambda_{\on{glob}})_!=(\ol\bi^\lambda_{\on{glob}})_*:\SI_{q,\on{glob}}(G)^{\leq \lambda}\rightleftarrows 
\SI_{q,\on{glob}}(G)^{\leq 0}:(\ol\bi{}^\lambda_{\on{glob}})^!,$$
$$(\ol\bi{}^\lambda_{\on{glob}})^*:\SI_{q,\on{glob}}(G)^{\leq 0}\rightleftarrows \SI_{q,\on{glob}}(G)^{\leq \lambda}:
(\ol\bi{}^\lambda_{\on{glob}})_*.$$

\sssec{}  \label{sss:descr strata glob}

The full subcategory $\SI_{q,\on{glob}}(G)^{\leq 0}\subset \Shv_{\CG^G}(\BunNbom)$ is characterized as follows:

\medskip

An object $\CF\in \Shv_{\CG^G}(\BunNbom)$ belongs to $\SI_{q,\on{glob}}(G)^{\leq 0}$ if and only if each 
$(\bi^\lambda_{\on{glob}})^!(\CF)$ (or, equivalently $(\bi^\lambda_{\on{glob}})^*(\CF)$) 
belongs to the corresponding full subcategory
$$\SI_{q,\on{glob}}(G)^{=\lambda}\subset \Shv_{\CG^G}((\BunNbom)^{=\lambda}).$$

\medskip

In its turn, pullback along
$$(\BunNbom)^{=\lambda}\to \Conf^\lambda$$
is fully faithful, and its essential image is exactly $\SI_{q,\on{glob}}(G)^{=\lambda}$. 

\sssec{}

Consider the map
$$\pi_{\Ran}:\ol{S}^0_{\Ran}\to \BunNbom.$$

Note that it follows from \thmref{t:N contr} that the pullback functor
$$\pi_{\Ran}^!: \Shv_{\CG^G}(\BunNbom) \to \Shv_{\CG^G}(\ol{S}^0_\Ran)$$
is fully faithful.

\medskip

The following is a metaplectic analog of \cite[Corollary 3.5.7 and Theorem 4.3.2]{Ga7}:

\begin{thm} \label{t:local-to-global semiinf} 
The functor
$$\pi_{\Ran}^!: \Shv_{\CG^G}(\BunNbom) \to \Shv_{\CG^G}(\ol{S}^0_\Ran)$$
induces an equivalence between
$$\SI_{q,\on{glob}}(G)^{\leq 0}\subset \Shv_{\CG^G}(\BunNbom)$$
and
$$\SI_{q,\Ran}(G)^{\leq 0}_{\on{untl}}\subset \Shv_{\CG^G}(\ol{S}^0_\Ran).$$
Moreover, after applying the cohomological shift by 
$$d_g:=\dim(\Bun_N^{\omega^\rho})=(g-1)(d-\langle 2\check\rho,2\rho\rangle),\quad d=\dim(\fn),$$
it is t-exact. 
\end{thm} 

\sssec{}

Let $\ICs_{q,\on{glob}}\in  (\Shv_{\CG^G}(\BunNbom))^\heartsuit$ be the IC-extension (in the category of $\CG^G$-twisted sheaves)
of the constant perverse sheaf on $\Bun_N^{\omega^\rho}$. From \thmref{t:local-to-global semiinf} we obtain:

\begin{cor}  \label{c:ICs}
We have a unique isomorphism $(\pi_{\Ran})^!(\ICs_{q,\on{glob}})[d_g]\simeq \ICs_{q,\Ran}$
extending the tautological identification over $S^0_{\Ran}$.
\end{cor}

\section{Torus equivariance and the Hecke property of the semi-infinite IC sheaf}  \label{s:Hecke ICs}

The goal of this section is two-fold.
First, we will introduce a twisting construction, which will allow us to put $\ICs_{q,\Ran}$ on twisted versions 
of $\ol{S}^0_\Ran$. Second, we will study the Hecke eigen-property of $\ICs_{q,\Ran}$, which is a crucial
ingredient for the Hecke enhancement of the Jacquet functor. 

\ssec{Adding $T$-equivariance}

In this subsection we will show that $\ICs_{q,\Ran}$ has a natural structure of equivariance with respect to the group
$\fL^+(T)_{\Ran}$. 

\sssec{}

The constructions in \secref{s:semiinf categ} all have a variant when we replace the \emph{condition} of
equivariance with respect to $\fL(N)_\Ran^{\omega^\rho}$ by the \emph{structure} of equivariance with
respect to $\fL(N)_\Ran^{\omega^\rho}\cdot \fL^+(T)_\Ran$. 

\medskip

Thus, we obtain a full subcategory
$$(\SI_{q,\Ran}(G))^{\fL^+(T)_\Ran}:=\Shv_{\CG^G}(\Gr^{\omega^\rho}_{G,\Ran})^{\fL(N)_\Ran^{\omega^\rho}\cdot \fL^+(T)_\Ran}
\subset \Shv_{\CG^G}(\Gr^{\omega^\rho}_{G,\Ran})^{\fL^+(T)_\Ran},$$
along with its variant
$$(\SI_{q,\Ran}(G)^{\leq 0})^{\fL^+(T)_\Ran}:= \Shv_{\CG^G}(\ol{S}^0_\Ran)^{\fL(N)_\Ran^{\omega^\rho}\cdot \fL^+(T)_\Ran}\subset 
\Shv_{\CG^G}(\ol{S}^0_\Ran)^{\fL^+(T)_\Ran}$$ as well as
$$(\SI_{q,\Ran}(G)^{\leq \lambda})^{\fL^+(T)_\Ran}:=
\Shv_{\CG^G}(\ol{S}^\lambda_\Ran)^{\fL(N)_\Ran^{\omega^\rho}\cdot \fL^+(T)_\Ran}\subset 
\Shv_{\CG^G}(\ol{S}^\lambda_\Ran)^{\fL^+(T)_\Ran}$$
and
$$(\SI_{q,\Ran}(G)^{= \lambda})^{\fL^+(T)_\Ran}=\Shv_{\CG^G}(S^\lambda_\Ran)^{\fL(N)_\Ran^{\omega^\rho}\cdot \fL^+(T)_\Ran}
\subset \Shv_{\CG^G}(S^\lambda_\Ran)^{\fL^+(T)_\Ran}.$$

\sssec{}

For future use we note:

\begin{lem} \label{l:non sharp strata bis}
The category $(\SI_{q,\Ran}(G)^{=\lambda})^{\fL^+(T)_\Ran}$ is zero unless $\lambda\in \Lambda^\sharp$.
\end{lem}

\begin{proof}

The proof follows by analyzing the stabilizers:

\medskip

By factorization, we may consider the case of a single point $x\in X$, i.e., we have to show that the category
$(\SI_{q,x}(G)_{=\lambda})^{\fL^+(T)_x}$ is zero unless $\lambda\in \Lambda^\sharp$.

\medskip

Consider the (unique) $T$-fixed point $t^\lambda \in S^\lambda \subset \Gr^{\omega^\rho}_{G,x}$. 
Restriction defines an equivalence
$$\SI_{q,x}(G)_{=\lambda}\to \Shv_{\CG^G}(\{t^\lambda\}),$$
and hence an equivalence
$$(\SI_{q,x}(G)_{= \lambda})^{\fL^+(T)_x}\simeq \Shv_{\CG^G}(\{t^\lambda\})^{\fL^+(T)_x}.$$

The $\fL^+(T)_x$-equivariance structure on the gerbe $\CG^G|_{\{t^\lambda\}}$ corresponds to
a character sheaf on $\fL^+(T)_x$, and we have to show that this character sheaf is non-trivial if
$\lambda\notin \Lambda^\sharp$.

\medskip 

Now, by \cite[Sect. 7.4]{GLys}, the above character sheaf is described as follows: it is the pullback
along $\fL^+(T)_x\to T$ of the character sheaf on $T$ arising by Kummer theory from the map
$$b(\lambda,-):\Lambda\to \sfe(-1).$$
The triviality of the latter means that $\lambda\in \Lambda^\sharp$. 

\end{proof}

\sssec{}

We also have the corresponding unital variants:
$$(\SI_{q,\Ran}(G)^{\leq 0}_{\on{untl}})^{\fL^+(T)_\Ran}\subset (\SI_{q,\Ran}(G)^{\leq 0})^{\fL^+(T)_\Ran},$$
$$(\SI_{q,\Ran}(G)^{\leq \lambda}_{\on{untl}})^{\fL^+(T)_\Ran}\subset (\SI_{q,\Ran}(G)^{\leq \lambda})^{\fL^+(T)_\Ran},$$
$$(\SI_{q,\Ran}(G)^{= \lambda}_{\on{untl}})^{\fL^+(T)_\Ran}\subset (\SI_{q,\Ran}(G)^{=\lambda})^{\fL^+(T)_\Ran}.$$ 

Note that the corresponding variant of \corref{c:unital semiinf on stratum} says that pullback with respect to 
$\on{pr}^\lambda_\Ran\circ p^\lambda_\Ran$ defines an equivalence
$$\Shv_{\CG^\Lambda}(\Conf^\lambda)^{\fL^+(T)_{\Conf^\lambda}}\to (\SI_{q,\Ran}(G)^{= \lambda}_{\on{untl}})^{\fL^+(T)_\Ran},$$
where $\fL^+(T)_{\Conf^\lambda}$ acts trivially on its base $\Conf^\lambda$, but the equivariance structure for the gerbe $\CG^\Lambda$
is non-trivial. 

\medskip

In particular, the recipe of \secref{ss:t-str semiinf} defines a t-structure on $(\SI_{q,\Ran}(G)^{\leq 0}_{\on{untl}})^{\fL^+(T)_\Ran}$
so that the forgetful functor
$$(\SI_{q,\Ran}(G)^{\leq 0}_{\on{untl}})^{\fL^+(T)_\Ran}\to \SI_{q,\Ran}(G)^{\leq 0}_{\on{untl}}$$
is t-exact, and its restriction to the hearts is fully faithful.

\sssec{}  \label{sss:proof of non-sharp}

Thus, we obtain that $\ICs_{q,\Ran}$ naturally upgrades to an object of $(\SI_{q,\Ran}(G)^{\leq 0}_{\on{untl}})^{\fL^+(T)_\Ran}$. 

\medskip

In particular, we note that the assertion of \lemref{l:non sharp strata} follows from \lemref{l:non sharp strata bis}. 

\sssec{}

Note that the colimit \eqref{e:ICs as colim} is naturally an object of
$$(\SI_{q,x}(G))^{\fL^+(T)_x}:=\Shv_{\CG^G}(\Gr^{\omega^\rho}_{G,x})^{\fL(N)^{\omega^\rho}_x\cdot \fL^+(T)_x}.$$

Thus, by \thmref{t:fiber of Ics}, we obtain that \eqref{e:ICs as colim} gives a description of 
$$\ICs_{q,x}:=\ICs_{q,\Ran}|_{\Gr^{\omega^\rho}_{G,x}}$$
as an object of $(\SI_{q,x}(G))^{\fL^+(T)_x}$. 

\sssec{}

Consider the category $\Shv_{\CG^G}(\ol{S}^0_{\Ran})^{\fL^+(T)_\Ran}$. Thinking of it as
$$\Shv_{\CG^G}(\fL^+(T)_\Ran\backslash \ol{S}^0_{\Ran}),$$
we can talk about factorization algebras in $\Shv_{\CG^G}(\ol{S}^0_{\Ran})^{\fL^+(T)_\Ran}$.

\medskip

The next result is be proved along with \thmref{t:fact ICs}:

\begin{thm} \label{t:fact ICs equiv}
There exists a factorization structure on $\ICs_{q,\Ran}$ thought of as an object of $(\SI_{q,\Ran}(G)^{\leq 0}_{\on{untl}})^{\fL^+(T)_\Ran}$,
uniquely characterized by the requirement that it gives rise to the 
factorization structure on $\ICs_{q,\Ran}$, thought of as an object of $\SI_{q,\Ran}(G)^{\leq 0}_{\on{untl}}$,
specified by \thmref{t:fact ICs}. 
\end{thm} 

\ssec{Hecke action on the semi-infinite category}  

In this subsection we will explore a key structure possessed by $\ICs_{q,\Ran}$, namely, that it is a Hecke eigensheaf. 

\sssec{}

Let us perform the base change for all the objects involved with respect to the forgetful map
$$\Ran_x\to \Ran.$$

\medskip

Consider the resulting category
$$(\SI_{q,\Ran_x}(G))^{\fL^+(T)_{\Ran_x}}=\Shv_{\CG^G}(\Gr^{\omega^\rho}_{G,\Ran_x})^{\fL(N)_\Ran^{\omega^\rho}\cdot \fL^+(T)_{\Ran_x}},$$
and the object
$$\ICs_{q,\Ran_x}\in (\SI_{q,\Ran_x}(G))^{\fL^+(T)_{\Ran_x}},$$
equal to the !-pullback of $\ICs_{q,\Ran}$ along the map $\Gr^{\omega^\rho}_{G,\Ran_x}\to \Gr^{\omega^\rho}_{G,\Ran}$. 

\medskip

Note that $(\SI_{q,\Ran_x}(G))^{\fL^+(T)_{\Ran_x}}$ is acted on naturally on the right by 
$$\Shv_{\CG^G}(\Gr^{\omega^\rho}_{G,\Ran_x})^{\fL^+(G)_{\Ran_x}},$$
and on the left by 
$$\Shv_{\CG^T}(\Gr^{\omega^\rho}_{T,\Ran_x})^{\fL^+(T)_{\Ran_x}}.$$

\sssec{} 

Recall the groupoid $\on{Hecke}^{\on{loc}}_{G,x}$ acting on $\fL^+(G)_{\Ran_x}\backslash \Ran_x$,
see \secref{sss:restr Hecke}. Consider the resulting monoidal functor 
\begin{equation} \label{e:Hecke G action first}
\Sph_{q,x}(G):=\Shv_{\CG^G}(\Gr^{\omega^\rho}_{G,x})^{\fL^+(G)_x^{\omega^\rho}}\to 
\Shv_{\CG^G}(\Gr^{\omega^\rho}_{G,\Ran_x})^{\fL^+(G)_{\Ran_x}}.
\end{equation} 

\medskip

Composing with the geometric Satake functor
$$\Sat_{q,G}:\Rep(H)\to \Sph_{q,x}(G)$$
we obtain that $(\SI_{q,\Ran_x}(G))^{\fL^+(T)_{\Ran_x}}$ is acted on by $\Rep(H)$ (on the right). 

\sssec{}  \label{sss:Sat'}

Similarly, we obtain a monoidal functor 
\begin{equation} \label{e:Hecke T action first}
\Sph_{q,x}(T):=\Shv_{\CG^T}(\Gr^{\omega^\rho}_{T,x})^{\fL^+(T)_x}\to \Shv_{\CG^T}(\Gr^{\omega^\rho}_{T,\Ran_x})^{\fL^+(T)_{\Ran_x}}.
\end{equation} 

\medskip

In what follows we will replace the usual geometric Satake functor for $T$
$$\Sat_{q,T}:\Rep(T_H)\to \Sph_{q,x}(T),$$
by its variant that takes into account a cohomological shift: 
$$\Sat'_{q,T}:\Rep(T_H)\to \Sph_{q,x}(T),$$
where 
$$\Sat'_{q,T}(\sfe^\lambda)=\Sat(\sfe^\lambda)[-\langle \lambda,2\check\rho\rangle].$$

\medskip

Composing, with \eqref{e:Hecke T action first}, we obtain that $(\SI_{q,\Ran_x}(G))^{\fL^+(T)_{\Ran_x}}$ is acted on by $\Rep(T_H)$
(on the left), in a way commuting with the above $\Rep(H)$-action. 

\sssec{}

Thus, we find ourselves in the paradigm of \secref{ss:rel Hecke}, and it makes sense to consider the corresponding
category $$\bHecke_{\on{rel}}((\SI_{q,\Ran_x}(G))^{\fL^+(T)_{\Ran_x}}).$$

\medskip

We have the following result, which is a metaplectic version of \cite[Theorem 5.1.8]{Ga7}:
 
\begin{thm} \label{t:Hecke ICs}  
The object $\ICs_{q,\Ran_x}\in (\SI_{q,\Ran_x}(G))^{\fL^+(T)_{\Ran_x}}$
naturally lifts to object of the category $$\bHecke_{\on{rel}}((\SI_{q,\Ran_x}(G))^{\fL^+(T)_{\Ran_x}}).$$
\end{thm} 

\sssec{}

Restricting to the fiber over $x\in \Ran_x$, we obtain that the object
$$\ICs_{q,x}\in (\SI_{q,x}(G))^{\fL^+(T)_x}$$
lifts to an object of $\bHecke_{\on{rel}}((\SI_{q,x}(G))^{\fL^+(T)_x})$.

\begin{rem}
In \secref{ss:Dr-Pl semiinf} we will show how this structure follows from the identification of 
$\ICs_{q,x}$ with $'\!\ICs_{q,x}$ via the \emph{Drinfeld-Pl\"ucker formalism}.
\end{rem}

\ssec{Hecke structure and factorization}

For what follows we will need to complement \thmref{t:Hecke ICs} by the following statement. 

\sssec{}

The factorization isomorphisms \eqref{e:factor IC} (viewed as taking place in the $\fL^+(T)$-equivariant category)
give rise to factorization isomorphisms 
\begin{multline} \label{e:factor IC x}
\ICs_{q,\Ran_x}|_{\ol{S}^0_{\Ran_x} \underset{\Ran_x}\times (\Ran^J\times \Ran_x)_{\on{disj}}} \simeq \\
\simeq   ((\ICs_{q,\Ran})^{\boxtimes J}\boxtimes \ICs_{q,\Ran_x})
|_{((\ol{S}^0_{\Ran})^J\times \ol{S}^0_{\Ran_x}) \underset{\Ran^J\times \Ran_x}\times (\Ran^J \times \Ran_x)_{\on{disj}}}.
\end{multline}

I.e., $\ICs_{q,\Ran_x}$ is naturally an object of the category 
\begin{equation} \label{e:IC mod}
\ICs_{q,\Ran}\on{-FactMod}(\Shv_{\CG^\Lambda}(\Gr^{\omega^\rho}_{G,\Ran_x})^{\fL^+(T)_{\Ran_x}}).
\end{equation}

\sssec{}  \label{sss:Hecke action on fact mod}

Recall now (see \secref{sss:fact funct mod geom}) that the actions of $\Sph_{q,x}(G)$ and $\Sph_{q,x}(T)$ on 
$$\Shv_{\CG^\Lambda}(\Gr^{\omega^\rho}_{G,\Ran_x})^{\fL^+(T)_{\Ran_x}}$$ give rise to ones on \eqref{e:IC mod}. 

\medskip

Composing with the metaplectic geometric Satake functors for $G$ and $T$, we obtain that \eqref{e:IC mod}
acquires a structure of module category over $\Rep(H)\otimes \Rep(T_H)$. 

\sssec{}

The following comes along with the construction of the Hecke structure on $\ICs_{q,\Ran_x}$:

\begin{thm} \label{t:Hecke ICs fact}  
The relative Hecke structure on $\ICs_{q,\Ran_x}$ given by \thmref{t:Hecke ICs}
is compatible in the natural sense with the factorization isomorphisms \eqref{e:factor IC} and \eqref{e:factor IC x}. 
I.e.,
$$\ICs_{q,\Ran_x}\in \ICs_{q,\Ran}\on{-FactMod}(\Shv_{\CG^\Lambda}(\Gr^{\omega^\rho}_{G,\Ran_x})^{\fL^+(T)_{\Ran_x}})$$
lifts to an object of
$$\bHecke_{\on{rel}}\left(\ICs_{q,\Ran}\on{-FactMod}(\Shv_{\CG^\Lambda}(\Gr^{\omega^\rho}_{G,\Ran_x})^{\fL^+(T)_{\Ran_x}})\right).$$
\end{thm}

\ssec{A $T$-twisting construction} \label{ss:twist T bundle}

In this subsection we will perform a twisting construction that will allow us to consider variants of 
$\ICs_{q,\Ran}$ on spaces that are fibered over a base with typical fiber being $\ol{S}^0_{\Ran}$. 
This construction will be used in the definition of the Jacquet functor, and also for the local-to-global 
comparison.  

\sssec{}

Let $\CY$ be a prestack equipped with a map to $\fL^+(T)_\Ran\backslash \Ran$.  We let
$$_\CY\Gr^{\omega^\rho}_{G}$$
denote the fiber product
$$\CY\underset{\fL^+(T)_\Ran\backslash \Ran}\times \left(\fL^+(T)_\Ran\backslash \Gr^{\omega^\rho}_{G,\Ran}\right).$$

We will use a similar notation for $_\CY\ol{S}^0$, etc. 

\sssec{}

Pulling back with respect to the forgetful map
$$_\CY\Gr^{\omega^\rho}_{G}\to \Gr^{\omega^\rho}_{G,\Ran},$$
from $\CG^G$, we obtain a gerbe on $_\CY\Gr^{\omega^\rho}_{G}$, which we denote by $_\CY\!\CG^G$.

\medskip

Further, pulling back $\ICs_{q,\Ran}$, we obtain an object
$$_\CY\!\ICs_{q}\in \Shv_{_\CY\!\CG^G}({}_\CY\ol{S}^0).$$

\medskip

In the sequel we will need the following two particular cases of this construction.

\sssec{}  \label{sss:twist by Gr T}

Take 
$$\CY=\Gr^{\omega^\rho}_{T,\Ran},$$
equipped with its tautological map to $\fL^+(T)_\Ran\backslash \Ran$,
$$(\CI,\CP_G,\alpha)\mapsto (\CI,\CP_G).$$

\medskip

Note that we have a canonical identification
\begin{equation}  \label{e:compose Gr T}
_\CY\Gr^{\omega^\rho}_{G} \simeq 
\Gr^{\omega^\rho}_{T,\Ran}\underset{\Ran}\times \Gr^{\omega^\rho}_{G,\Ran}.
\end{equation}

The multiplicativity property of the gerbe $\CG^{G,G,\on{ratio}}$ (see \secref{sss:gerbe on Hecke}) implies that the resulting 
gerbe $_\CY\!\CG^G$ on $_\CY\Gr^{\omega^\rho}_{G}$ goes over under the identification \eqref{e:compose Gr T}
to the gerbe on $\Gr^{\omega^\rho}_{T,\Ran}\underset{\Ran}\times \Gr^{\omega^\rho}_{G,\Ran}$ equal to
$$\CG^{G,T,\on{ratio}}:=(\CG^T)^{-1} \boxtimes \CG^G.$$

\medskip

Thus, we obtain that 
$$_\CY\!\ICs_{q}\in \Shv_{_\CY\!\CG^G}({}_\CY\ol{S}^0)$$
can be thought of as an object, to be denoted
$$_{\Gr_T}\!\ICs_{q,\Ran}\in \Shv_{\CG^{G,T,\on{ratio}}}(\Gr^{\omega^\rho}_{T,\Ran}\underset{\Ran}\times \Gr^{\omega^\rho}_{G,\Ran}).$$

\medskip

From we obtain:

\begin{cor}  \label{c:twisted IC fact}
The object $_{\Gr_T}\!\ICs_{q,\Ran}$ has a natural structure of factorization algebra in 
$$\Shv_{\CG^{G,T,\on{ratio}}}(\Gr^{\omega^\rho}_{T,\Ran}\underset{\Ran}\times \Gr^{\omega^\rho}_{G,\Ran}).$$
\end{cor} 

\sssec{}

Let us now take
$$\CY=\Bun_T\times \Ran,$$
equipped with its natural map to $\fL^+(T)_\Ran\backslash \Ran$. Consider the corresponding object
$$_{\Bun_T}\!\ICs_{q,\Ran}:={}_{\Bun_T\times \Ran}\!\ICs_q\in \Shv_{_{\Bun_T\times \Ran}\CG^G}(_{\Bun_T\times \Ran}\ol{S}^0).$$

\medskip

Note that the prestack 
$_{\Bun_T}\ol{S}^0_\Ran:={}_{\Bun_T\times \Ran}\ol{S}^0$ admits a natural map 
$$\pi_\Ran:{}_{\Bun_T}\ol{S}^0_\Ran\to \BunBb,$$
so that we have a Cartesian square
$$
\CD
\ol{S}^0_\Ran    @>>>    _{\Bun_T}\ol{S}^0_\Ran \\
@V{\pi_\Ran}VV  @VV{\pi_\Ran}V   \\
\BunNbom    @>>>  \BunBb  \\
@VVV   @VV{\ol\sfq}V   \\
\on{pt}    @>{\omega^\rho}>>   \Bun_T.
\endCD
$$

\sssec{}

Let $\CG^{G,T,\on{ratio}}$ denote the gerbe on $\BunBb$ equal to
$$\CG^G\otimes (\CG^T)^{-1},$$
where $\CG^G$ is the pullback of the same-named gerbe on $\Bun_G$ (see \secref{sss:global gerbe G}) along the map
$$\ol\sfp:\BunBb\to \Bun_G$$
and $\CG^T$ is the pullback of the same-named gerbe on $\Bun_T$ (see \secref{sss:global gerbe G})
along the map $$\ol\sfq:\BunBb\to \Bun_T.$$

\medskip

Note that we have a canonical identification of gerbes on $_{\Bun_T}\ol{S}^0_\Ran$
\begin{equation} \label{e:global loc ratio}
_{\Bun_T\times \Ran}\CG^G\simeq (\pi_\Ran)^*(\CG^{G,T,\on{ratio}}).
\end{equation}

\sssec{}

Note also that we have a canonical trivialization of the restriction of $\CG^{G,T,\on{ratio}}$ along
$$\Bun_B\overset{\bj_{\on{glob}}}\hookrightarrow \BunBb.$$
This trivialization is compatible with the trivializations of the restriction of both sides of \eqref{e:global loc ratio} 
to $_{\Bun_T}S^0_\Ran\subset {}_{\Bun_T}\ol{S}^0_\Ran$. 

\medskip

We let $_{\Bun_T}\!\ICs_{q,\on{glob}}$ denote the IC extension (in the category of $\CG^{G,T,\on{ratio}}$-twisted sheaves on $\BunBb$) 
of the constant perverse sheaf on $\Bun_B$ (the latter makes sense due to the above trivialization of $\CG^{G,T,\on{ratio}}|_{\Bun_B}$).

\medskip

The following is a metaplectic analog of \cite[Theorem 6.3.2]{Ga7}:

\begin{thm}  \label{t:ICs glob}
There is a unique isomorphism
$$(\pi_\Ran)^!({}_{\Bun_T}\!\ICs_{q,\on{glob}})[d_g+\dim(\Bun_T)+\deg]\simeq {}_{\Bun_T}\!\ICs_{q,\Ran}$$
extending the tautological identification over $_{\Bun_T}S^0_\Ran$, where the value of $\deg$
equals $\langle \lambda,2\check\rho\rangle$ over the connected component $\Bun^\lambda_T$ of $\Bun_T$. 
\end{thm}

\begin{rem}
We normalize the bijection
$$\pi_0(\Bun_T)\simeq \Lambda$$ so that
the map $\Gr_{T,x}\to \Bun_T$ sends
$$\Gr_{T,x}^\lambda\to \Bun^\lambda_T,$$
where $\Gr_{T,x}^\lambda$ is the connected component of $\Gr_{T,x}$ that contains the point $t^\lambda$.
\end{rem} 

\ssec{Hecke property in the twisted context}  \label{ss:twisted Hecke}

In this subsection we will show how the Hecke property of $\ICs_{q,\Ran_x}$ translates to a Hecke
property of its twisted version $_\CY\!\ICs_{q}$ constructed in \secref{ss:twist T bundle}. This is a necessary 
ingredient for establishing the Hecke property of the Jacquet functor. 

\sssec{}

Let us base change the discussion in \secref{ss:twist T bundle} along the map $\Ran_x\to \Ran$.
I.e., let us assume having a prestack $\CY$, equipped with a map
$$\CY\to \fL^+(T)_{\Ran_x}\backslash \Ran_x.$$

\medskip

Consider the corresponding prestack 
$$_\CY\Gr^{\omega^\rho}_{G}:=\CY\underset{\fL^+(T)_{\Ran_x}\backslash \Ran_x}\times 
\left(\fL^+(T)_{\Ran_x}\backslash \Gr^{\omega^\rho}_{G,\Ran_x}\right).$$

\sssec{}  \label{sss:Hecke G action} 

Recall again the groupoid $\on{Hecke}^{\on{loc}}_{G,x}$ acting on $\fL^+(G)_{\Ran_x}\backslash \Ran_x$,
see \secref{sss:restr Hecke}. 

\medskip 

By construction, this action lifts to one on $_\CY\Gr^{\omega^\rho}_{G}$. Hence, we obtain a monoidal action of $\Sph_{q,x}(G)$ on
$\Shv_{_\CY\!\CG^G}({}_\CY\Gr^{\omega^\rho}_{G})$. 

\medskip

Composing with the geometric Satake functor $\Sat_{q,G}$, we obtain an action of $\Rep(H)$ on 
$\Shv_{_\CY\!\CG^G}({}_\CY\Gr^{\omega^\rho}_{G})$. 

\sssec{}  \label{sss:action on Y}

Consider now the action of the groupoid $\on{Hecke}^{\on{loc}}_{T,x}$ acting on $\fL^+(T)_{\Ran_x}\backslash \Ran_x$. 
Assume that we are given a lift of this action to one on $\CY$.  

\medskip

I.e., we assume 
being given a prestack $_\CY\!\on{Hecke}^{\on{loc}}_{T,x}$ equipped with maps
$$\CY \overset{\hl_T}\longleftarrow {}_\CY\!\on{Hecke}^{\on{loc}}_{T,x} \overset{\hr_T}\longrightarrow \CY$$
and a map
$$_\CY\!\on{Hecke}^{\on{loc}}_{T,x}\to {}\on{Hecke}^{\on{loc}}_{T,x}$$
that make both square in the diagram
$$
\CD
\CY  @<{\hl_T}<< _\CY\!\on{Hecke}^{\on{loc}}_{T,x} @>{\hr_T}>> \CY  \\
@VVV  @VVV  @VVV  \\
\fL^+(T)_{\Ran_x}\backslash \Ran_x @<{\hl_T}<< 
\on{Hecke}^{\on{loc}}_{T,x}  @>{\hr_T}>> \fL^+(T)_{\Ran_x}\backslash \Ran_x,
\endCD
$$
Cartesian. 

\sssec{}  \label{sss:Hecke T action} 

Under the above circumstances, the above action of $\on{Hecke}^{\on{loc}}_{T,x}$ on $\CY$ can be 
further lifted to an action on
$_\CY\Gr^{\omega^\rho}_{G}$ (given by leaving the $G$-bundle $\CP_G$ intact). 

\medskip

In other words, we have a prestack $_{_\CY\Gr^{\omega^\rho}_{G}}\!\on{Hecke}^{\on{loc}}_{T,x}$ equipped with maps
$$_\CY\Gr^{\omega^\rho}_{G} \leftarrow {}_{_\CY\Gr^{\omega^\rho}_{G}}\!\on{Hecke}^{\on{loc}}_{T,x}\to
{}_\CY\Gr^{\omega^\rho}_{G}$$
and a map
$$_{_\CY\Gr^{\omega^\rho}_{G}}\!\on{Hecke}^{\on{loc}}_{T,x}\to 
{}_\CY\!\on{Hecke}^{\on{loc}}_{T,x}$$
that make both squares in the diagram
$$
\CD 
_\CY\Gr^{\omega^\rho}_{G} @<{\hl_T}<<  _{_\CY\!\Gr^{\omega^\rho}_{G}}\!\on{Hecke}^{\on{loc}}_{T,x}  
@>{\hr_T}>> _\CY\Gr^{\omega^\rho}_{G}  \\
@VVV   @VVV   @VVV   \\
\CY  @<{\hl_T}<< _\CY\!\on{Hecke}^{\on{loc}}_{T,x} @>{\hr_T}>> \CY 
\endCD
$$
Cartesian.

\sssec{}

From here we obtain that the monoidal category $\Sph_{q,x}(T)$ acts on the left $\Shv_{_\CY\!\CG^G}({}_\CY\Gr^{\omega^\rho}_{G})$.
Moreover, this action commutes with the right action of $\Sph_{q,x}(G)$. 

\medskip

Composing with the geometric Satake functor $\Sat'_{q,T}$, we obtain an action of $\Rep(T_H)$ on 
$\Shv_{_\CY\!\CG^G}({}_\CY\Gr^{\omega^\rho}_{G})$, which commutes with the $\Rep(H)$-action defined above. 

\sssec{}

Thus, we obtain that $\Shv_{_\CY\!\CG^G}({}_\CY\Gr^{\omega^\rho}_{G})$ can be viewed as a category equipped
with an action of $\Rep(H)\otimes \Rep(T_H)$, and we find ourselves in the context of \secref{ss:rel Hecke}. 

\medskip

From \thmref{t:Hecke ICs} we obtain:

\begin{cor} \label{c:twisted Hecke IC}
The object $$_\CY\!\ICs_{q,\Ran_x}\in \Shv_{_\CY\!\CG^G}({}_\CY\Gr^{\omega^\rho}_{G})$$
lifts to object of $\bHecke_{\on{rel}}(\Shv_{_\CY\!\CG^G}({}_\CY\Gr^{\omega^\rho}_{G}))$.
\end{cor} 

\ssec{Hecke action over $\Gr_T$}   \label{ss:Hecke on Gr T}

We will now consider an example of the situation described in \secref{ss:twisted Hecke} (another example will be described
in \secref{ss:Hecke global} below). The construction in this subsection has a direct import on the Hecke property
of the Jacquet functor considered in the next Part. 

\sssec{}

Let us be in the context of \secref{sss:twist by Gr T}, but with a marked point $x$. I.e., we take
$$\CY:=\Gr^{\omega^\rho}_{T,\Ran_x},$$
together with its natural map to $\fL^+(T)_{\Ran_x}\backslash \Ran_x$.

\medskip

The action of $\on{Hecke}^{\on{loc}}_{T,x}$ on $\fL^+(T)_{\Ran_x}\backslash \Ran_x$ naturally lifts to
a \emph{right} action on $\Gr^{\omega^\rho}_{T,\Ran_x}$. Using the inversion involution, we obtain a 
left action of $\on{Hecke}^{\on{loc}}_{T,x}$ on $\Gr^{\omega^\rho}_{T,\Ran_x}$. Hence, 
we find ourselves in the context of \secref{sss:action on Y}. 

\sssec{} \label{sss:Hecke action on Gr T}

Recall the identification \eqref{e:compose Gr T}:
$$_\CY\Gr^{\omega^\rho}_{G} \simeq 
\Gr^{\omega^\rho}_{T,\Ran_x}\underset{\Ran_x}\times \Gr^{\omega^\rho}_{G,\Ran_x}.$$

\medskip

This identification intertwines the left action of $\on{Hecke}^{\on{loc}}_{G,x}$
on $_\CY\Gr^{\omega^\rho}_G$ of \secref{sss:Hecke G action} and the natural right
action of $\on{Hecke}^{\on{loc}}_{G,x}$ on $\Gr^{\omega^\rho}_{T,\Ran_x}\underset{\Ran}\times \Gr^{\omega^\rho}_{G,\Ran_x}$
via the second factor. 

\medskip

In addition, the above identification intertwines the left action of $\on{Hecke}^{\on{loc}}_{T,x}$
on $_\CY\Gr^{\omega^\rho}_G$ of \secref{sss:Hecke T action} and the natural right
action of $\on{Hecke}^{\on{loc}}_{T,x}$ on $\Gr^{\omega^\rho}_{T,\Ran_x}\underset{\Ran}\times \Gr^{\omega^\rho}_{G,\Ran_x}$
via the first factor, \emph{precomposed} with the inversion involution on $\on{Hecke}^{\on{loc}}_{T,x}$
(note that inversion turns a left action into a right action). 

\medskip

Hence, we can view $\Shv_{\CG^{G,T,\on{ratio}}}(\Gr^{\omega^\rho}_{T,\Ran_x}\underset{\Ran_x}\times \Gr^{\omega^\rho}_{G,\Ran_x})$
as a module over $\Rep(H)\otimes \Rep(T_H)$, \emph{where we precompose} the action of $\Rep(T_H)$ with the Cartan
involution $\tau^{T_H}$, see \secref{sss:Cartan inv}. 

\sssec{}

From \corref{c:twisted Hecke IC}, combined with \secref{sss:Satake and inversion}, we obtain: 

\begin{cor} \label{c:Hecke IC Gr T}  
The object 
$$_{\Gr_T}\!\ICs_{q,\Ran_x}\in 
\Shv_{\CG^{G,T,\on{ratio}}}(\Gr^{\omega^\rho}_{T,\Ran_x}\underset{\Ran_x}\times \Gr^{\omega^\rho}_{G,\Ran_x})$$ 
lifts to an object of 
$$\bHecke_{\on{rel}}(\Shv_{\CG^{G,T,\on{ratio}}}(\Gr^{\omega^\rho}_{T,\Ran_x}\underset{\Ran_x}\times \Gr^{\omega^\rho}_{G,\Ran_x})).$$
\end{cor} 

\sssec{}

We will now need to complement of statement of \corref{c:Hecke IC Gr T} to take into account the factorization structure.

\medskip

By \corref{c:twisted IC fact} we can regard $_{\Gr_T}\!\ICs_{q,\Ran}$ as a factorization algebra in the category
$$\Shv_{\CG^{G,T,\on{ratio}}}(\Gr^{\omega^\rho}_{T,\Ran}\underset{\Ran}\times \Gr^{\omega^\rho}_{G,\Ran})$$ 
and $_{\Gr_T}\!\ICs_{q,\Ran_x}$ as an object of
\begin{equation} \label{e:IC twisted mod fact}
_{\Gr_T}\!\ICs_{q,\Ran}\on{-FactMod}
\left(\Shv_{\CG^{G,T,\on{ratio}}}(\Gr^{\omega^\rho}_{T,\Ran_x}\underset{\Ran_x}\times \Gr^{\omega^\rho}_{G,\Ran_x})\right).
\end{equation} 

%

As in \secref{sss:Hecke action on fact mod}, we can regard \eqref{e:IC twisted mod fact} as a module category
over $\Rep(H)\otimes \Rep(T_H)$. From \thmref{t:Hecke ICs fact} we obtain:

\begin{cor} \label{c:Hecke IC Gr T fact}
The relative Hecke structure on $_{\Gr_T}\!\ICs_{q,\Ran_x}$ given by \corref{c:Hecke IC Gr T}
is compatible with the factorization structure in a natural sense. I.e., 
$_{\Gr_T}\!\ICs_{q,\Ran_x}$ naturally lifts to an object of the category
$$\bHecke_{\on{rel}}
\left({}_{\Gr_T}\!\ICs_{q,\Ran}\on{-FactMod}
(\Shv_{\CG^{G,T,\on{ratio}}}(\Gr^{\omega^\rho}_{T,\Ran_x}\underset{\Ran_x}\times \Gr^{\omega^\rho}_{G,\Ran_x}))\right).$$
\end{cor} 

\ssec{Global Hecke property}  \label{ss:Hecke global}

In this subsection we will consider another example of the paradigm of \secref{ss:twisted Hecke}, which contains a global 
aspect, incarnated by the stack $\Bun_T$ of $T$-bundles on a complete curve $X$. 

\medskip

The material in this subsection is needed for the local-to-global comparison for the (Hecke version of the) Jacquet functor, 
which is, in turn, used to establish the commutation of the (Hecke version of the) local Jacquet functor with Verdier duality. 

\sssec{}

Consider the ind-algebraic stack $(\BunBb)_{\infty\cdot x}$. By construction, it is equipped with a pair of maps
$$\on{pt}/\fL^+(G)_x \leftarrow (\BunBb)_{\infty\cdot x} \to \on{pt}/\fL^+(T)_x.$$

The action on $\on{Hecke}^{\on{loc}}_{T,x}$ on $\on{pt}/\fL^+(T)_x$ naturally lifts to a (left) action on 
$(\BunBb)_{\infty\cdot x}$: we modify the $T$-bundle, while leaving the $G$-bundle intact.

\medskip

Similarly, the right action of $\on{Hecke}^{\on{loc}}_{G,x}$ on $\on{pt}/\fL^+(G)_x$ naturally lifts to an action on 
$(\BunBb)_{\infty\cdot x}$: we modify the $G$-bundle, while leaving the $T$-bundle intact.

\medskip

Thus, the category $\Shv_{\CG^{G,T,\on{ratio}}}((\BunBb)_{\infty\cdot x})$ acquires a left action of
$\Sph_{q,x}(T)$ and a commuting right action of $\Sph_{q,x}(G)$. 

\begin{rem}  \label{r:global Hecke action}

Concretely, the action of $\Sph_{q,x}(G)$ on $\Shv_{\CG^{G,T,\on{ratio}}}((\BunBb)_{\infty\cdot x})$ is given by
the formula
$$\CF,\CS_G\mapsto (\hl_G)_*\left((\hr_G)^!(\CF)\sotimes \CS_G|_{_{(\BunBb)_{\infty\cdot x}}\!\on{Hecke}^{\on{loc}}_{G,x}}\right),
\quad \CF\in \Shv_{\CG^{G,T,\on{ratio}}}((\BunBb)_{\infty\cdot x}), \,\, \CS_G\in \Sph_{q,x}(G)$$
for the maps in the following diagram
$$
\CD
(\BunBb)_{\infty\cdot x}    @<{\hl_G}<<  _{(\BunBb)_{\infty\cdot x}}\!\on{Hecke}^{\on{loc}}_{G,x}  @>{\hr_G}>> (\BunBb)_{\infty\cdot x} \\
@VVV    @VVV  @VVV  \\
\on{pt}/\fL^+(G)_x   @<{\hl_G}<<  \on{Hecke}^{\on{loc}}_{G,x}  @>{\hr_G}>>  \on{pt}/\fL^+(G)_x. 
\endCD
$$ 

\medskip

Similarly, the action of $\Sph_{q,x}(T)$ on $\Shv_{\CG^{G,T,\on{ratio}}}((\BunBb)_{\infty\cdot x})$ is given by the formula 
$$\CF,\CS_T\mapsto (\hr_T)_*\left((\hl_T)^!(\CF)\sotimes \CS_T|_{_{(\BunBb)_{\infty\cdot x}}\!\on{Hecke}^{\on{loc}}_{T,x}}\right),\quad
\CF\in \Shv_{\CG^{G,T,\on{ratio}}}((\BunBb)_{\infty\cdot x}), \,\, \CS_T\in \Sph_{q,x}(T)$$
for the maps in the following diagram
$$
\CD
(\BunBb)_{\infty\cdot x}    @<{\hl_T}<<  _{(\BunBb)_{\infty\cdot x}}\!\on{Hecke}^{\on{loc}}_{T,x}  @>{\hr_T}>> (\BunBb)_{\infty\cdot x} \\
@VVV    @VVV  @VVV  \\
\on{pt}/\fL^+(T)_x   @<{\hl_T}<<  \on{Hecke}^{\on{loc}}_{T,x}  @>{\hr_T}>>  \on{pt}/\fL^+(T)_x. 
\endCD
$$ 

\end{rem} 

\sssec{}

Applying geometric Satake, we obtain an action of $\Rep(H)$ and a commuting action of $\Rep(T_H)$
on $\Shv_{\CG^{G,T,\on{ratio}}}((\BunBb)_{\infty\cdot x})$. 

\medskip 

\noindent{\it Warning:} Here, unlike the local version, we use the usual geometric Satake functor $\Sat_{q,T}$
(and not the cohomologically shifted version $\Sat'_{q,T}$, see \secref{sss:Sat'}).

\medskip

The following is a metaplectic version of \cite[Theorem 3.1.4]{BG}:

\begin{thm}  \label{t:global Hecke}
The object $_{\Bun_T}\!\ICs_{q,\on{glob}}\in \Shv_{\CG^{G,T,\on{ratio}}}((\BunBb)_{\infty\cdot x})$ naturally lifts to an object of 
$\bHecke_{\on{rel}}(\Shv_{\CG^{G,T,\on{ratio}}}((\BunBb)_{\infty\cdot x}))$.
\end{thm} 

\begin{rem} \label{r:Rep H abelian}
In \cite{BG}, the Hecke property of $\IC_{\BunBb}$ was established with respect to objects $V\in \Rep(\cG)$ that 
\emph{lie in the abelian category} $(\Rep(\cG))^\heartsuit$. This is, however, sufficient because 
$\Rep(\cG)$ is the derived category if its heart. The same remark applies in the metaplectic case. 
\end{rem}

\sssec{}

Let us return to the setting of \secref{ss:twisted Hecke} with
$$\CY:=\Bun_T\times \Ran_x.$$

\medskip

Consider the resulting prestack
$$_{\Bun_T}\!\Gr^{\omega^\rho}_{G,\Ran_x}:={}_{\Bun_T\times \Ran_x}\!\Gr^{\omega^\rho}_G.$$
and the corresponding closed sub-prestack
$$({}_{\Bun_T}\!\ol{S}^0_{\Ran_x})_{\infty\cdot x}\subset {}_{\Bun_T}\!\Gr^{\omega^\rho}_{G,\Ran_x},$$
see \secref{sss:tw grass}. 

\medskip

Note that the actions of the groupoids $\on{Hecke}^{\on{loc}}_{T,x}$ and
$\on{Hecke}^{\on{loc}}_{G,x}$ preserve this sub-prestack. In particular,
we can consider $\Shv_{_{\Bun_T\times \Ran}\CG}(({}_{\Bun_T}\!\ol{S}^0_{\Ran_x})_{\infty\cdot x})$
as equipped with a right action of $\Sph_{q,x}(G)$ and a commuting left action of $\Sph_{q,x}(T)$. 

\medskip

In particular, we obtain that $\Shv_{_{\Bun_T\times \Ran}\CG}(({}_{\Bun_T}\!\ol{S}^0_{\Ran_x})_{\infty\cdot x})$
is a module over $\Rep(H)\otimes \Rep(T_H)$. 
It follows from \corref{c:twisted Hecke IC} that the resulting object
$$_{\Bun_T}\!\ICs_{q,\Ran_x}\in \Shv_{_{\Bun_T\times \Ran}\CG}(({}_{\Bun_T}\!\ol{S}^0_{\Ran_x})_{\infty\cdot x})$$
naturally lifts to an object of 
$$\bHecke_{\on{rel}}(\Shv_{_{\Bun_T\times \Ran}\CG}(({}_{\Bun_T}\!\ol{S}^0_{\Ran_x})_{\infty\cdot x})).$$

\sssec{}

Note that the projection
$$\pi_{\Ran_x}:({}_{\Bun_T}\!\ol{S}^0_{\Ran_x})_{\infty\cdot x}\to (\BunBb)_{\infty\cdot x}$$
intertwines the actions of $\on{Hecke}^{\on{loc}}_{T,x}$ and
$\on{Hecke}^{\on{loc}}_{G,x}$ on $({}_{\Bun_T}\!\ol{S}^0_{\Ran_x})_{\infty\cdot x}$
and $(\BunBb)_{\infty\cdot x}$. 

\medskip

Hence, the pullback functor
$$(\pi_{\Ran_x})^!: \Shv_{\CG^{G,T,\on{ratio}}}((\BunBb)_{\infty\cdot x}) \to 
\Shv_{_{\Bun_T\times \Ran}\CG}(({}_{\Bun_T}\!\ol{S}^0_{\Ran_x})_{\infty\cdot x})$$
gives rise to a functor
$$\bHecke_{\on{rel}}(\Shv_{\CG^{G,T,\on{ratio}}}((\BunBb)_{\infty\cdot x}))\to 
\bHecke_{\on{rel}}(\Shv_{_{\Bun_T\times \Ran}\CG}(({}_{\Bun_T}\!\ol{S}^0_{\Ran_x})_{\infty\cdot x})).$$

\sssec{}

The following is a metaplectic version of \cite[Theorem 6.3.5]{Ga7}: 

\begin{thm} \label{t:local vs global IC Hecke}
The isomorphism 
$$(\pi_{\Ran_x})^!({}_{\Bun_T}\!\ICs_{q,\on{glob}})[d_g+\dim(\Bun_T)+\deg]\simeq  {}_{\Bun_T}\!\ICs_{q,\Ran_x}$$
of \thmref{t:ICs glob} lifts to an isomorphism of objects in 
$\bHecke_{\on{rel}}(\Shv_{_{\Bun_T\times \Ran}\CG}(({}_{\Bun_T}\!\ol{S}^0_{\Ran_x})_{\infty\cdot x}))$. 
\end{thm} 

\ssec{Local vs global Hecke property}

In this subsection we will study the compatibility of the constructions in Sects. \ref{ss:Hecke on Gr T} and \ref{ss:Hecke global}, respectively. 

\sssec{}  \label{sss:two contexts}

Recall the closed subfunctor
$${}_{\Gr^{\omega^\rho}_{T,\Ran}}\!\ol{S}^0\subset {}_{\Gr^{\omega^\rho}_{T,\Ran}}\!\Gr_G^{\omega^\rho}
\simeq \Gr^{\omega^\rho}_{T,\Ran}\underset{\Ran}\times \Gr^{\omega^\rho}_{G,\Ran}.$$

\medskip

We have a Cartesian square
$$
\CD
_{\Gr^{\omega^\rho}_{T,\Ran}}\!\ol{S}^0  @>>>  _{\Bun_T}\!\ol{S}^0_{\Ran}  \\
@VVV   @VVV   \\
\Gr^{\omega^\rho}_{T,\Ran} @>>>  \Bun_T. 
\endCD
$$

Composing with the map $\pi_{\Ran}:{}_{\Bun_T}\!\ol{S}^0_{\Ran} \to  \BunBb$, 
we obtain a map  
$$\pi_{\Gr_T}:{}_{\Gr^{\omega^\rho}_{T,\Ran}}\!\ol{S}^0\to \BunBb$$
so that the diagram
\begin{equation} \label{e:S and BunT}
\CD
_{\Gr^{\omega^\rho}_{T,\Ran}}\!\ol{S}^0   @>{\pi_{\Gr_T}}>>   \BunBb   \\
@VVV @VV{\sfq}V   \\
\Gr^{\omega^\rho}_{T,\Ran} @>>>  \Bun_T
\endCD
\end{equation} 
commutes. 

\medskip

Note that the pullback of the gerbe $\CG^{G,T,\on{ratio}}$ along $\pi_{\Gr_T}$ goes
over to the restriction of the gerbe denoted $\CG^{G,T,\on{ratio}}$ on 
$\Gr^{\omega^\rho}_{T,\Ran}\underset{\Ran}\times \Gr^{\omega^\rho}_{G,\Ran}$. 

\medskip

Unwinding the constructions, from \thmref{t:ICs glob} we obtain:

\begin{cor} \label{c:ICs glob sym}
There exists canonical isomorphism in $\Shv_{\CG^{G,T,\on{ratio}}}({}_{\Gr^{\omega^\rho}_{T,\Ran}}\!\ol{S}^0)$
$$(\pi_{\Gr_T})^!({}_{\Bun_T}\!\ICs_{q,\on{glob}})[d_g+\dim(\Bun_T)+\deg]\simeq {}_{\Gr_T}\!\ICs_{q,\Ran},$$
where the value of $\deg$ equals $\langle \lambda,2\check\rho\rangle$ over the connected component $\Bun_T^\lambda$
of $\Bun_T$. 
\end{cor} 

\sssec{}

Consider now the closed subfunctor
$${}_{\Gr^{\omega^\rho}_{T,\Ran_x}}\!(\ol{S}^0)_{\infty\cdot x}\subset {}_{\Gr^{\omega^\rho}_{T,\Ran_x}}\!\Gr_G^{\omega^\rho}
\simeq \Gr^{\omega^\rho}_{T,\Ran_x}\underset{\Ran_x}\times \Gr^{\omega^\rho}_{G,\Ran_x}.$$

\medskip

Similarly, to the above, we have a map 
$$\pi_{\Gr_T}:{}_{\Gr^{\omega^\rho}_{T,\Ran_x}}\!(\ol{S}^0)_{\infty\cdot x}
\to (\BunBb)_{\infty\cdot x},$$
compatible with the gerbes, and with the actions of the groupoids $\on{Hecke}^{\on{loc}}_{G,x}$
and $\on{Hecke}^{\on{loc}}_{T,x}$. 

\medskip

In particular, the functor of pullback 
$$(\pi_{\Gr_T})^!:\Shv_{\CG^{G,T,\on{ratio}}}((\BunBb)_{\infty\cdot x}) \to 
\Shv_{\CG^{G,T,\on{ratio}}}({}_{\Gr^{\omega^\rho}_{T,\Ran_x}}\!(\ol{S}^0)_{\infty\cdot x})$$
gives rise to a functor
$$\bHecke_{\on{rel}}(\Shv_{\CG^{G,T,\on{ratio}}}((\BunBb)_{\infty\cdot x})) \to 
\bHecke_{\on{rel}}(\Shv_{\CG^{G,T,\on{ratio}}}({}_{\Gr^{\omega^\rho}_{T,\Ran_x}}\!(\ol{S}^0)_{\infty\cdot x}).$$

\medskip

From \thmref{t:local vs global IC Hecke} we obtain: 

\begin{cor} \label{c:ICs glob sym Hecke}
The isomorphism
$$(\pi_{\Gr_T})^!({}_{\Bun_T}\!\ICs_{q,\on{glob}})[d_g+\dim(\Bun_T)+\deg]\simeq {}_{\Gr_T}\!\ICs_{q,\Ran_x}$$
of \corref{c:ICs glob sym} lifts to an isomorphism of objects of 
$$\bHecke_{\on{rel}}(\Shv_{\CG^{G,T,\on{ratio}}}({}_{\Gr^{\omega^\rho}_{T,\Ran_x}}\!(\ol{S}^0)_{\infty\cdot x}).$$
\end{cor} 

\newpage 

\centerline{\bf Part V: The Jacquet functor} 

\bigskip

In this Part we begin the process of relating the two categories involved in our main theorem:
the Hecke category of $\Whit_{q,x}(G)$ and a certain category of factorization modules. The
relationship will be realized by a functor from the former to the latter, the main ingredient of 
which is (one of the versions of) the Jacquet functor 
$$\Whit_{q,x}(G)\to \Shv_{\CG^T}(\Gr^{\omega^\rho}_{R,\Ran_x}).$$

\section{Construction of the Jacquet functor}   \label{s:Jacquet}

The goal of this section is to construct the Jacquet functor
$$\fJ_{!*,\on{Fact}}:\Whit_{q,x}(G) \to \Omega_q^{\Whit_{!*}}\on{-FactMod}(\Shv_{\CG^T}(\Gr^{\omega^\rho}_{T,\Ran_x})),$$
where $\Omega_q^{\Whit_{!*}}\mod$ is a certain factorization algebra in $\Shv_{\CG^T}(\Gr^{\omega^\rho}_{T,\Ran})$,
and $$\Omega_q^{\Whit_{!*}}\on{-FactMod}(\Shv_{\CG^T}(\Gr^{\omega^\rho}_{T,\Ran_x}))$$ denotes the category of factorization modules over it in 
$\Shv_{\CG^T}(\Gr^{\omega^\rho}_{T,\Ran_x})$. 

\medskip

In the next section, we will upgrade the functor $\fJ_{!*,\on{Fact}}$ to a functor
$$\fJ^{\bHecke}_{!*,\on{Fact}}:\bHecke(\Whit_{q,x}(G)) \to \Omega_q^{\Whit_{!*}}\on{-FactMod}(\Shv_{\CG^T}(\Gr^{\omega^\rho}_{T,\Ran_x})).$$

\ssec{The bare version of the Jacquet functor}

The definition of the functor $\fJ_{!*,\on{Fact}}$ will proceed in stages. In this subsection we will
define the functor
$$\fJ_{!*,\on{sprd}}:\Whit_{q,x}(G) \to \Shv_{\CG^T}(\Gr^{\omega^\rho}_{T,\Ran_x}),$$
which we should think of as the composition of $\fJ_{!*,\on{Fact}}$ and the forgetful functor
$$\Omega_q^{\Whit_{!*}}\on{-FactMod}(\Shv_{\CG^T}(\Gr^{\omega^\rho}_{T,\Ran_x}))\to \Shv_{\CG^T}(\Gr^{\omega^\rho}_{T,\Ran_x}).$$

\sssec{}

Recall the object
$$_{\Gr^{\omega^\rho}_T}\!\ICs_{q,\Ran}\in \Shv_{\CG^{G,T,\on{ratio}}}(\Gr^{\omega^\rho}_{T,\Ran}\underset{\Ran}\times \Gr^{\omega^\rho}_{G,\Ran}).$$ 

We will consider its counterpart denoted
$$_{\Gr^{\omega^\rho}_T}\!\ICsm_{q^{-1},\Ran}\in \Shv_{\CG^{G,T,\on{ratio}}}(\Gr^{\omega^\rho}_{T,\Ran}\underset{\Ran}\times \Gr^{\omega^\rho}_{G,\Ran}),$$ 
in which we replace $N\rightsquigarrow N^-$ and $\CG^G\rightsquigarrow (\CG^G)^{-1}$. 

\sssec{}

We define the most basic version of the Jacquet functor
$$\fJ_{!*,\Ran}: \Shv_{\CG^G}(\Gr^{\omega^\rho}_{G,\Ran}) \to \Shv_{\CG^T}(\Gr^{\omega^\rho}_{T,\Ran})$$
as
$$\CF\mapsto (p_T)_*({}_{\Gr^{\omega^\rho}_T}\!\ICsm_{q^{-1},\Ran}\sotimes (p_G)^!(\CF)),$$
where $p_T$ and $p_G$ are the projections
\begin{equation} \label{e:G T proj}
\Gr^{\omega^\rho}_{T,\Ran} \leftarrow
\Gr^{\omega^\rho}_{T,\Ran}\underset{\Ran}\times \Gr^{\omega^\rho}_{G,\Ran}\to \Gr^{\omega^\rho}_{G,\Ran},
\end{equation}
respectively. 

\medskip

Note that the functor $(p_T)_*$ makes sense since the tensor product 
$$_{\Gr^{\omega^\rho}_T}\!\ICsm_{q^{-1},\Ran}\sotimes (p_G)^!(\CF)$$
belongs to the category
$$\Shv_{\CG^T}(\Gr^{\omega^\rho}_{T,\Ran}\underset{\Ran}\times \Gr^{\omega^\rho}_{G,\Ran}).$$

\medskip

In other words, the functor $\fJ_{!*}$ is defined using the correspondence \eqref{e:G T proj} with kernel 
$_{\Gr^{\omega^\rho}_T}\!\ICsm_{q^{-1},\Ran_x}$. 

\begin{rem}
Let us explain the origin of the name ``Jacquet functor". Suppose that instead of $\ICs_{q,\Ran}$ we use
$$(\bj_\Ran)_*(\omega_{S^0_\Ran})\in \SI_{q,\Ran}(G)^{\fL^+(T)_\Ran},$$
along with its variants. 

\medskip

Denote the resulting functor 
$$\Shv_{\CG^G}(\Gr^{\omega^\rho}_{G,\Ran}) \to \Shv_{\CG^T}(\Gr^{\omega^\rho}_{T,\Ran})$$
by $\fJ_{*,\Ran}$. 

\medskip

Then it is easy to see that $\fJ_{*,\Ran}$ is given by 
$$\Shv_{\CG^G}(\Gr^{\omega^\rho}_{G,\Ran}) \overset{!\on{-pullback}}\longrightarrow 
\Shv_{\CG^G}(\Gr^{\omega^\rho}_{B^-,\Ran}) \overset{*\on{-pushforward}}\longrightarrow \Shv_{\CG^T}(\Gr^{\omega^\rho}_{T,\Ran}).$$

\end{rem}

\sssec{}

Along with $\fJ_{!*,\Ran}$, we will consider its variants 
$$\fJ_{!*,\Ran_x}:\Shv_{\CG^G}(\Gr^{\omega^\rho}_{G,\Ran_x}) \to \Shv_{\CG^T}(\Gr^{\omega^\rho}_{T,\Ran_x})$$
and
$$\fJ_{!*,x}:\Shv_{\CG^G}(\Gr^{\omega^\rho}_{G,x})\to \Shv_{\CG^T}(\Gr^{\omega^\rho}_{T,x}),$$
defined using the objects
$$_{\Gr^{\omega^\rho}_T}\!\ICsm_{q^{-1},\Ran_x}\in \Shv_{\CG^{G,T,\on{ratio}}}(\Gr^{\omega^\rho}_{T,\Ran_x}\underset{\Ran_x}\times \Gr^{\omega^\rho}_{G,\Ran_x})
\text{ and } _{\Gr^{\omega^\rho}_T}\!\ICsm_{q^{-1},x}\in \Shv_{\CG^{G,T,\on{ratio}}}(\Gr^{\omega^\rho}_{T,x}\times \Gr^{\omega^\rho}_{G,x}),$$
respectively, obtained as pullbacks of  $_{\Gr^{\omega^\rho}_T}\!\ICsm_{q,\Ran}$. 

\sssec{}

Recall now the functor 
$$\on{sprd}_{\Ran_x}: \Whit_{q,x}(G) \to \Shv_{\CG^G}(\Gr^{\omega^\rho}_{G,\Ran_x}),$$
see \secref{sss:spread functor}. 

\medskip

We define the functor
$$\fJ_{!*,\on{sprd}}:\Whit_{q,x}(G) \to \Shv_{\CG^T}(\Gr^{\omega^\rho}_{T,\Ran_x})$$
as the composite 
$$\fJ_{!*,\Ran_x}\circ \on{sprd}_{\Ran_x}.$$

\ssec{The factorization algebra $\Omega_q^{\Whit_{!*}}$}

In this subsection we define a factorization algebra
$$\Omega_q^{\Whit_{!*}}\in \Shv_{\CG^T}(\Gr^{\omega^\rho}_{T,\Ran}).$$

It will be essentially equivalent to the factorization algebra on $\Conf$ used to define the
category of factorization modules that appears in the right-hand side in our main theorem. 

\sssec{}

Recall that the object
$$\on{Vac}_{\Whit,\Ran}\in \Shv_{\CG^G}(\Gr^{\omega^\rho}_{G,\Ran}).$$

Recall that, according to \thmref{t:fact on Vac Whit}(a), $\on{Vac}_{\Whit,\Ran}$ has a structure of
factorization algebra in $\Shv_{\CG^G}(\Gr^{\omega^\rho}_{G,\Ran})$

\medskip

Recall also that, according to \corref{c:twisted IC fact}, 
 $_{\Gr_T}\!\ICsm_{q^{-1},\Ran}$ has a structure of factorization algebra
in $\Shv_{(\CG^{G,T,\on{ratio}})^{-1}}(\Gr^{\omega^\rho}_{T,\Ran}\underset{\Ran}\times \Gr^{\omega^\rho}_{G,\Ran})$. 

\medskip

Hence, by \secref{sss:fact funct geom pullback}, we obtain that
$$_{\Gr_T}\!\ICsm_{q^{-1},\Ran}\sotimes (p_G)^!(\on{Vac}_{\Whit,\Ran})$$
acquires a natural structure of factorization algebra in 
$\Shv_{\CG^T}(\Gr^{\omega^\rho}_{T,\Ran}\underset{\Ran}\times \Gr^{\omega^\rho}_{G,\Ran})$. 

\sssec{} 

Set
$$\Omega_q^{\Whit_{!*}}:=(p_T)_*({}_{\Gr_T}\!\ICsm_{q^{-1},\Ran}\sotimes (p_G)^!(\on{Vac}_{\Whit,\Ran})).$$

By \secref{sss:fact funct geom pushforward}, we obtain that $\Omega_q^{\Whit_{!*}}$ acquires a natural structure of factorization
algebra in $\Shv_{\CG^T}(\Gr^{\omega^\rho}_{T,\Ran})$.

\begin{rem}
In \secref{ss:Omega expl} we will a very explicit description of $\Omega_q^{\Whit_{!*}}$.
\end{rem}

\ssec{Adding the factorization structure}

In this subsection we will upgrade $\fJ_{!*,\on{sprd}}$ to a functor 
$$\fJ_{!*,\on{Fact}}:\Whit_{q,x}(G) \to \Omega_q^{\Whit_{!*}}\on{-FactMod}(\Shv_{\CG^T}(\Gr^{\omega^\rho}_{T,\Ran_x})).$$

\sssec{}

Recall that \thmref{t:fact on Vac Whit}(b) says that the functor $\on{sprd}_{\Ran_x}$
canonically lifts to a functor 
\begin{equation} \label{e:spread fact}
\on{sprd}_{\on{Fact}}: \Whit_{q,x}(G)\to \on{Vac}_{\Whit,\Ran}\on{-FactMod}(\Shv_{\CG^G}(\Gr_{G,\Ran_x})).
\end{equation}

%

\medskip

Consider $_{\Gr_T}\!\ICsm_{q^{-1},\Ran_x}$ as an object of 
$$_{\Gr_T}\!\ICsm_{q^{-1},\Ran}\on{-FactMod}
\left(\Shv_{\CG^T}(\Gr^{\omega^\rho}_{T,\Ran_x}\underset{\Ran_x}\times \Gr^{\omega^\rho}_{G,\Ran_x})\right).$$

\medskip

Using \secref{sss:fact funct mod geom}, we obtain that the functor
$$\CF\mapsto {} _{\Gr_T}\!\ICsm_{q^{-1},\Ran}\sotimes p_G^!\circ \on{sprd}_{\Ran_x}(\CF)$$
upgrades to a functor
\begin{multline} \label{e:funct to ten prod fact mod}
\Whit_{q,x}(G)\to \\
\to ({}_{\Gr_T}\!\ICsm_{q^{-1},\Ran}\sotimes (p_G)^!(\on{Vac}_{\Whit,\Ran}))\on{-FactMod}
\left(\Shv_{(\CG^{G,T,\on{ratio}})^{-1}}(\Gr^{\omega^\rho}_{T,\Ran_x}\underset{\Ran_x}\times \Gr^{\omega^\rho}_{G,\Ran_x})\right).
\end{multline}

\sssec{}

Composing \eqref{e:funct to ten prod fact mod} with the functor $(p_T)_*$ (see \secref{sss:fact funct mod geom}), we obtain that the functor
$\fJ_{!*,\on{sprd}}$ upgrades to the sought-for functor
$$\fJ_{!*,\on{Fact}}:\Whit_{q,x}(G)\to \Omega_q^{\Whit_{!*}}\on{-FactMod}(\Shv_{\CG^T}(\Gr^{\omega^\rho}_{T,\Ran_x})).$$

\section{Hecke enhancement of the Jacquet functor}   \label{s:Hecke Jacquet}

The goal of this section is to perform the key construction of this paper, namely, to extend the functor 
$$\fJ_{!*,\on{Fact}}:\Whit_{q,x}(G)\to \Omega_q^{\Whit_{!*}}\on{-FactMod}(\Shv_{\CG^T}(\Gr^{\omega^\rho}_{T,\Ran_x}))$$
constructed in the previous section to a functor
$$\fJ^{\bHecke}_{!*,\on{Fact}}:\bHecke(\Whit_{q,x}(G)) \to \Omega_q^{\Whit_{!*}}\on{-FactMod}(\Shv_{\CG^T}(\Gr^{\omega^\rho}_{T,\Ran_x})).$$

\ssec{Extension of the bare version of the functor}

In this subsection, as a warm-up, we will extend the functor 
$$\fJ_{!*,\on{sprd}}: \Whit_{q,x}(G)\to  \Shv_{\CG^T}(\Gr^{\omega^\rho}_{T,\Ran_x})$$
to a functor
$$\fJ^{\bHecke}_{!*,\on{sprd}}: \Whit_{q,x}(G)\to  \Shv_{\CG^T}(\Gr^{\omega^\rho}_{T,\Ran_x}).$$

\sssec{}

By \secref{sss:graded Hecke univ}, the construction of the sought-for functor $\fJ^{\bHecke}_{!*,\on{sprd}}$
is equivalent to the following:

\begin{thmconstr}  \label{t:Hecke ppty of J} 
The functor $$\fJ_{!*,\on{sprd}}: \Whit_{q,x}(G)\to  \Shv_{\CG^T}(\Gr^{\omega^\rho}_{T,\Ran_x})$$
intertwines the $\Rep(H)$-action on $\Whit_{q,x}(G)$ and the $\Rep(T_H)$-action
on $\Shv_{\CG^T}(\Gr^{\omega^\rho}_{T,\Ran_x})$.
\end{thmconstr}

The rest of this subsection is devoted to the proof of this theorem. 

\sssec{}

First, we note that the functor 
$$\on{sprd}_{\Ran_x}:\Whit_{q,x}(G)\to \Shv_{\CG^G}(\Gr_{G,\Ran_x})$$
intertwines the actions of $\Sph_{q,x}(G)$ on $\Whit_{q,x}(G)$ and $\Shv_{\CG^G}(\Gr_{G,\Ran_x})$.

\medskip

Hence, \thmref{t:Hecke ppty of J} follows from the next more general result:
\begin{thmconstr}  \label{t:Hecke ppty of J gen} 
The functor $$\fJ_{!*,\Ran_x}: \Shv_{\CG^G}(\Gr_{G,\Ran_x})\to  \Shv_{\CG^T}(\Gr^{\omega^\rho}_{T,\Ran_x})$$
intertwines the $\Rep(H)$-action on $\Shv_{\CG^G}(\Gr_{G,\Ran_x})$ and the $\Rep(T_H)$-action
on $\Shv_{\CG^T}(\Gr^{\omega^\rho}_{T,\Ran_x})$.
\end{thmconstr}

We will see that the proof of \thmref{t:Hecke ppty of J gen} amounts to no more than diagram chase once we know \corref{c:Hecke IC Gr T}. 

\sssec{}  \label{sss:paradigm for Hecke}

The proof of \thmref{t:Hecke ppty of J gen} will fit into the following general paradigm:

\medskip

Let $\bC$ be a category equipped with an action of $\Rep(H)$,
and let $\bD$ be a category equipped with an action of $\Rep(H)\otimes \Rep(T_H)$. Let $\bE$ be a category,
equipped with an action of $\Rep(T_H)$. 

\medskip

Let us be given a functor
$$\Psi:\bC\otimes \bD\to \bE,$$ 
equipped with a factorization
\begin{equation} \label{e:factor Psi}
\bC\otimes \bD \to \bC\underset{\Rep(H)}\otimes \bD\overset{\wt\Psi}\longrightarrow \bE.
\end{equation}

\medskip

Note that $\bC\underset{\Rep(H)}\otimes \bD$ is acted on by $\Rep(T_H)$. Assume that 
$\wt\Psi$ intertwines this action and the given one on $\bE$. 

\medskip

Let $\bd\in \bD$ be an object equipped with a structure of object of $\bHecke_{\on{rel}}(\bD)$. By unwinding the
definitions, in this case we obtain that the functor 
$$\Phi:=\Psi_\bd:\bC\to \bE, \quad \bc\mapsto \Psi(\bc\otimes \bd)$$
intertwines the $\Rep(H)$-action on $\bC$ and the $\Rep(T_H)$-action on $\bE$. 

\sssec{}

We apply the above paradigm as follows. We take
$$\bC:=\Shv_{\CG^G}(\Gr_{G,\Ran_x}), \,\, \bE:=\Shv_{\CG^T}(\Gr^{\omega^\rho}_{T,\Ran_x}),$$
and
$$\bD:=\Shv_{(\CG^{G,T,\on{ratio}})^{-1}}(\Gr^{\omega^\rho}_{T,\Ran_x}\underset{\Ran_x}\times \Gr^{\omega^\rho}_{G,\Ran_x}).$$

We will now supply the above categories with the required pieces of structure. 
First off, the functor $\Psi$ is the functor
\begin{equation} \label{e:our Psi}
\CF,\CF'\mapsto 
(p_T)_*\left(\CF'\sotimes (p_G)^!(\CF)\right).
\end{equation}

\sssec{}

The action of $\Rep(H)$ on $\Shv_{\CG^G}(\Gr_{G,\Ran_x})$ (resp., 
the action of $\Rep(T_H)$ on $\Shv_{\CG^T}(\Gr^{\omega^\rho}_{T,\Ran_x})$) is given by composing the
$\Sph_{q,x}(G)$-action on $\Shv_{\CG^G}(\Gr_{G,\Ran_x})$ (resp., 
the action of $\Sph_{q,x}(T)$ on $\Shv_{\CG^T}(\Gr^{\omega^\rho}_{T,\Ran_x})$) and the geometric Satake functor
$\Sat_{q,G}$ (resp., $\Sat'_{q,T}$).

\medskip

The action of $\Rep(T_H)$ on 
$\Shv_{(\CG^{G,T,\on{ratio}})^{-1}}(\Gr^{\omega^\rho}_{T,\Ran_x}\underset{\Ran_x}\times \Gr^{\omega^\rho}_{G,\Ran_x})$
is also given by $\Sat'_{q,T}$. 

\medskip

Now, the action of $\Rep(H)$ on 
$\Shv_{(\CG^{G,T,\on{ratio}})^{-1}}(\Gr^{\omega^\rho}_{T,\Ran_x}\underset{\Ran_x}\times \Gr^{\omega^\rho}_{G,\Ran_x})$
is given by \emph{composing}
$$\Sat_{q,G}:\Rep(H)\to \Sph_{q,x}(G)$$ 
with the inversion anti-homomorphism
$$\on{inv}^G:\Sph_{q,x}(G)\to \Sph_{q^{-1},x}(G),$$
see \secref{sss:inversion}.  Note that by \secref{sss:Satake and inversion}, the latter is the same as the action of 
$\Sph_{q^{-1},x}(G)$, \emph{precomposed} with the Cartan involution $\tau^H$ on $\Rep(H)$.

\medskip

The factorization \eqref{e:factor Psi} follows from the formula
$$(p_T)_*\left(\CF'\sotimes (p_G)^!(\CF\star \CS_G)\right)\simeq 
(p_T)_*\left((\CF'\star \on{inv}^G(\CS_G))\sotimes \CF\right), \quad \CS_G\in \Sph_{q,x}(G).$$

The fact that the resulting functor $\wt\Psi$ is compatible with the actions of $\Rep(T_H)$ follows from the formula
$$(p_T)_*\left((\CF'\star \CS_T)\sotimes (p_G)^!(\CF)\right)\simeq 
\left((p_T)_*\left(\CF'\sotimes \CF\right)\right)\star \CS_T , \quad \CS_T\in \Sph_{q,x}(T).$$

\sssec{}

Finally, we take 
$$\bd:={}_{\Gr_T}\!\ICsm_{q^{-1},\Ran}\in 
\Shv_{(\CG^{G,T,\on{ratio}})^{-1}}(\Gr^{\omega^\rho}_{T,\Ran_x}\underset{\Ran_x}\times \Gr^{\omega^\rho}_{G,\Ran_x}).$$

The structure on $\bd$ of an object of $\bHecke_{\on{rel}}(\bD)$ is provided by \corref{c:Hecke IC Gr T}: 

\medskip

The Cartan involution
on $T_H$ (appearing in \corref{c:Hecke IC Gr T}), the Cartan involution on $H$ (involved in the action of $\Rep(H)$ on $\bD$,
see above), and the swap of $B$ and $B^-$ (involved in the passage $_{\Gr_T}\!\ICs_{q^{-1},\Ran}\rightsquigarrow
{}_{\Gr_T}\!\ICsm_{q^{-1},\Ran}$) cancel each other out. 

\ssec{Hecke structure on the factorizble version}

In this subsection we will upgrade the construction of the previous subsection to obtain the functor
$$\fJ^{\bHecke}_{!*,\on{Fact}}:\bHecke(\Whit_{q,x}(G)) \to \Omega_q^{\Whit_{!*}}\on{-FactMod}(\Shv_{\CG^T}(\Gr^{\omega^\rho}_{T,\Ran_x})).$$

\sssec{}

By \secref{sss:fact funct mod geom}, the category $\Omega_q^{\Whit_{!*}}\on{-FactMod}(\Shv_{\CG^T}(\Gr^{\omega^\rho}_{T,\Ran_x}))$ 
carries a (right) monoidal action of
$\Sph_{q,x}(T)$, and hence an action of $\Rep(T_H)$.

\medskip

By  \secref{sss:graded Hecke univ}, the construction of the sought-for functor $\fJ^{\bHecke}_{!*,\on{Fact}}$
is equivalent to the following:

\begin{thmconstr}  \label{t:Hecke ppty of J Fact} 
The functor $$\fJ_{!*,\on{Fact}}: \Whit_{q,x}(G)\to \Omega_q^{\Whit_{!*}}\on{-FactMod}(\Shv_{\CG^T}(\Gr^{\omega^\rho}_{T,\Ran_x}))$$
intertwines the $\Rep(H)$-action on $\Whit_{q,x}(G)$ and the $\Rep(T_H)$-action
on $$\Omega_q^{\Whit_{!*}}\on{-FactMod}(\Shv_{\CG^T}(\Gr^{\omega^\rho}_{T,\Ran_x})).$$
\end{thmconstr}

The rest of this subsection is devoted to the proof of this theorem. 

\sssec{}

First, by \secref{sss:fact funct mod geom}, the category $\on{Vac}_{\Whit,\Ran}\on{-FactMod}(\Shv_{\CG^G}(\Gr_{G,\Ran_x}))$
carries a (right) monoidal action of $\Sph_{q,x}(G)$. From the construction of the functor 
$$\on{sprd}_{\on{Fact}}:\Whit_{q,x}(G)\to \on{Vac}_{\Whit,\Ran}\on{-FactMod}(\Shv_{\CG^G}(\Gr_{G,\Ran_x}))$$
in \thmref{t:fact on Vac Whit}(b), we obtain that it intertwines the actions of $\Sph_{q,x}(G)$ on the two sides. 

\medskip

Hence, \thmref{t:Hecke ppty of J Fact} follows from the next more general result:

\begin{thmconstr}  \label{t:Hecke ppty of J Fact gen} 
The functor 
\begin{equation} \label{m:dir im fact}
\on{Vac}_{\Whit,\Ran}\on{-FactMod}(\Shv_{\CG^G}(\Gr_{G,\Ran_x})) \to 
\Omega_q^{\Whit_{!*}}\on{-FactMod}(\Shv_{\CG^T}(\Gr^{\omega^\rho}_{T,\Ran_x}))
\end{equation}
induced by $\fJ_{!*,\Ran_x}$ intertwines the $\Rep(H)$-action on $\on{Vac}_{\Whit,\Ran}\on{-FactMod}(\Shv_{\CG^G}(\Gr_{G,\Ran_x}))$
and the $\Rep(T_H)$-action on $\Omega_q^{\Whit_{!*}}\on{-FactMod}(\Shv_{\CG^T}(\Gr^{\omega^\rho}_{T,\Ran_x}))$. 
\end{thmconstr}

\sssec{}

We will prove \thmref{t:Hecke ppty of J Fact gen} by applying the paradigm of \secref{sss:paradigm for Hecke}. 

\medskip

We take 
$$\bC=\on{Vac}_{\Whit,\Ran}\on{-FactMod}(\Shv_{\CG^G}(\Gr_{G,\Ran_x})),$$
$$\bE:=\Omega_q^{\Whit_{!*}}\on{-FactMod}(\Shv_{\CG^T}(\Gr_{T,\Ran_x})),$$
$$\bD:= \left({}_{\Gr_T}\!\ICsm_{q^{-1},\Ran}\sotimes (p_G)^!(\on{Vac}_{\Whit,\Ran})\right)\on{-FactMod}
(\Shv_{(\CG^{G,T,\on{ratio}})^{-1}}(\Gr^{\omega^\rho}_{T,\Ran_x}\underset{\Ran_x}\times \Gr^{\omega^\rho}_{G,\Ran_x})).$$

We take $\Psi$ to be the functor, induced by the functor \eqref{e:our Psi}.  

\sssec{}

Now the proof of \thmref{t:Hecke ppty of J Fact gen} follows verbatim that of \thmref{t:Hecke ppty of J gen}, 
using \corref{c:Hecke IC Gr T fact} as an additional input. 

\newpage 

\centerline{\bf Part VI: Interpretation via configuration spaces} 

\bigskip

In the previous Part, we constructed the functor
$$\fJ^{\bHecke}_{!*,\on{Fact}}:\bHecke(\Whit_{q,x}(G)) \to \Omega_q^{\Whit_{!*}}\on{-FactMod}(\Shv_{\CG^T}(\Gr^{\omega^\rho}_{T,\Ran_x})).$$

This functor is \emph{almost} an equivalence: we will need to apply a certain renormalization procedure to the right-hand side
in order to turn it into one. However, in order to make sense of this renormalization procedure, we will need to interpret 
the above category of factorization modules in terms of the configuration space, following the procedure of \secref{ss:Conf vs Gr}.
The passage $\Gr_T\rightsquigarrow \Conf$ will have the additional advantage of placing us in the context of finite-dimensional 
algebraic geometry, thereby making the category of factorization modules more accessible to calculations (see \secref{ss:mod over conf}).  

\medskip

Once we have reinterpreted the Jacquet functor as taking place in the category of factorization modules over the configuration
space, we will be able to state our main theorem. 

\section{Factorization algebra $\Omega_q^{\on{small}}$}  \label{s:Omega small}

In this section we put ourselves in the context of \secref{ss:fact on config}. We will describe a particular
factorization algebra, denoted $\Omega_q^{\on{small}}$, in $\Shv_{\CG^\Lambda}(\Conf)$ that exists under some additional conditions on
the geometric metaplectic data $\CG^T$. 

\medskip

The right-hand side in our main theorem will be the (renormalized version of the) category of factorization modules
over $\Omega_q^{\on{small}}$. 

\ssec{Additional requirements on the geometric metaplectic data}

In this subsection we start with a geometric metaplectic data $\CG^T$ for the torus $T$ and describe
the additional condition that allows to define the factorization algebra $\Omega_q^{\on{small}}$. 

\sssec{}  \label{sss:simple root triv}

Recall (see \secref{sss:gerbe Lambda}) that the geometric metaplectic data $\CG^T$ for $T$ gives rise
to a factorization gerbe, denoted $\CG^\Lambda$, on $\Conf$.

\medskip

We will require the following:

\medskip

\hskip0.5cm{\em For every vertex $i$ of the Dynkin diagram, the gerbe $\CG^\Lambda$ evaluated on 
$$\Conf^{-\alpha_i}=X^{-\alpha_i}\simeq X$$
is trivialized.}

\sssec{}  \label{sss:open part triv}

Note that the datum of trivialization as above uniquely extends to a trivialization of the restriction of $\CG^\Lambda$
$$\overset{\circ}\Conf\subset \Conf,$$
in a way compatible with the factorization structure on $\CG^\Lambda$. 

\sssec{}  \label{sss:gerbes from G}

Let us show that if the geometric metaplectic data $\CG^T$ for $T$ arises from a a geometric metaplectic data $\CG^G$,
then the above condition is automatically satisfied. 

\medskip

First, we note that by \cite[Sect. 5.1]{GLys}, this question reduces to the case of $G=SL_2$. Next, 
by \cite[Proposition 3.1.9 and Theorem 3.2.6]{GLys} every factorization gerbe on $\Gr_{SL_2}$ is canonically of the form
$$(\det_{SL_2})^a, \quad a\in \sfe^\times(-1).$$

\medskip

The resulting gerbe on $X=X^{-\alpha_i}$ equals $\CL^a$ (see \cite[Sect. 1.4.2]{GLys} for the notation), where $\CL$ is the line bundle
on $X$ that sends $x\in X$ to the line
$$\on{rel.det.}(\omega^{\otimes \frac{1}{2}}(x),\omega^{\otimes \frac{1}{2}}) \otimes 
\on{rel.det.}(\omega^{\otimes -\frac{1}{2}}(-x),\omega^{\otimes -\frac{1}{2}}).$$

\sssec{}   \label{sss:det calc}

We claim that $\CL$ is canonically trivial. Indeed:
$$\on{rel.det.}(\omega^{\otimes \frac{1}{2}}(x),\omega^{\otimes \frac{1}{2}}) \simeq \omega^{\otimes -\frac{1}{2}}_x$$
and
$$\on{rel.det.}(\omega^{\otimes -\frac{1}{2}},\omega^{\otimes -\frac{1}{2}}(-x))\simeq \omega^{\otimes -\frac{1}{2}}_x$$
as required. 

\ssec{Construction of $\Omega_q^{\on{small}}$ as a sheaf}

In this subsection we will define the (gerbe-twisted) perverse sheaf underlying the factorization algebra $\Omega_q^{\on{small}}$. 

\sssec{}  \label{sss:fi}

As our initial datum, for each vertex of the Dynkin diagram $i$, let us choose an $\sfe$-line denoted
$\sff^{i,\on{fact}}$.

\medskip

In \secref{sss:sfi from Whit} we will see the geometric meaning of these lines (they appear as \emph{geometric Gauss sums}). 

\sssec{}

Since $\Conf$ is a scheme locally of finite type, we have a well-defined full subcategory of twisted
\emph{perverse} sheaves 
$$\on{Perv}_{\CG^\Lambda}(\Conf)\subset 
\Shv_{\CG^\Lambda}(\Conf).$$

\medskip

We are going to define an object
$$\Omega^{\on{small}}_q\in \on{Perv}_{\CG^\Lambda}(\Conf).$$

\medskip

We first define its restriction to $\overset{\circ}\Conf$, 
$$\overset{\circ}\Omega{}^{\on{small}}_q\in \on{Perv}_{\CG^\Lambda}(\overset{\circ}\Conf).$$

\sssec{}

By \secref{sss:open part triv}, the gerbe 
$\CG^\Lambda$ is trivialized over $\overset{\circ}\Conf$, so
\begin{equation} \label{e:trivialize sheaves on open}
\Shv_{\CG^\Lambda}(\overset{\circ}\Conf)\simeq 
\Shv(\overset{\circ}\Conf)
\end{equation} 
and 
$$\on{Perv}_{\CG^\Lambda}(\overset{\circ}\Conf)\simeq 
\on{Perv}(\overset{\circ}\Conf).$$

\medskip

So, we can regard the sought-for twisted perverse sheaf $\overset{\circ}\Omega{}^{\on{small}}_q$ as an object of
$\on{Perv}(\overset{\circ}\Conf)$.

\sssec{}

Recall that 
$$\overset{\circ}\Conf= \underset{\lambda\in \Lambda^{\on{neg}}}\sqcup\,  \overset{\circ}\Conf^\lambda
\simeq \underset{\lambda\in \Lambda^{\on{neg}}}\sqcup\, \overset{\circ}{X}{}^\lambda,$$
where $\overset{\circ}{X}{}^\lambda$ is obtained from 
$$X^\lambda=\underset{i}\prod\, X^{(n_i)} \text{ if } \lambda=\underset{i}\Sigma\, n_i\cdot (-\alpha_i)$$
by removing the diagonal divisor. 

\medskip

For a fixed $\lambda$ and each $i$ we consider the $n_i$-th \emph{anti-symmetric power} of the
constant sheaf $(\sff^{i,\on{fact}})_X$ on $X$, as a local system on 
$$\overset{\circ}{X}{}^{(n_i)}=X^{(n_i)} -\on{Diag}.$$

\medskip

Denote it by $(\sff^{i,\on{fact}})_X^{(n_i),\on{sign}}$. Note that up to a trivialization of $\sff^{i,\on{fact}}$, this is just the sign local system on
$X^{(n_i)} -\on{Diag}$.

\sssec{}

Consider the perverse sheaf
$$(\sff^{i,\on{fact}})_X^{(n_i),\on{sign}}[n_i]\in \on{Perv}(\overset{\circ}{X}{}^{(n_i)}).$$

\medskip 

We set  
$$\overset{\circ}\Omega{}^{\on{small}}_q\in \on{Perv}(\overset{\circ}\Conf)=
\on{Perv}(\overset{\circ}{X}{}^\lambda)$$
to be the external product
$$\left(\underset{i}\boxtimes\, (\sff_{i,\on{fact}})_X^{(n_i),\on{sign}}[n_i]\right)|_{\overset{\circ}{X}{}^\lambda}.$$

\sssec{}

Finally, we define $\Omega^{\on{small}}_q$ to be the Goresky-MacPherson extension of $\overset{\circ}\Omega{}^{\on{small}}_q$
\emph{in the category of twsited perverse sheaves} $\on{Perv}_{\CG^\Lambda}(\Conf)$. 

\ssec{Factorization structure on $\Omega^{\on{small}}_q$}

In this subsection we endow $\Omega^{\on{small}}_q$ with a structure of factorization algebra. 

\sssec{}

We have to construct a system of isomorphisms
\begin{equation} \label{e:fact Omega}
\Omega^{\on{small}}_q|_{(\Conf^J)_{\on{disj}}}\simeq (\Omega^{\on{small}}_q)^{\boxtimes J}|_{(\Conf^J)_{\on{disj}}},
\end{equation}
satisfying a homotopy coherent system of compatibilities. 

\sssec{}

Note that both sides of \eqref{e:fact Omega} are \emph{perverse sheaves} that are Goresky-MacPherson extensions
of their respective restrictions to 
$$(\overset{\circ}\Conf)^J_{\on{disj}}:=(\overset{\circ}\Conf)^J\cap (\Conf^J)_{\on{disj}}\subset \Conf^J.$$

This is due to the fact that the addition map
$$\Conf^J\to \Conf$$
is \'etale when restricted to $(\Conf^J)_{\on{disj}}$.

\sssec{}

Hence, instead of \eqref{e:fact Omega}, it suffices to construct the corresponding isomorphisms 
\begin{equation} \label{e:fact Omega circ}
\overset{\circ}\Omega{}^{\on{small}}_q|_{(\overset{\circ}\Conf^J)_{\on{disj}}}\simeq 
(\overset{\circ}\Omega{}^{\on{small}}_q)^{\boxtimes J}|_{(\overset{\circ}\Conf^J)_{\on{disj}}}.
\end{equation}

However, the latter results directly from the construction. 

\section{Jacquet functor to the configuration space}  \label{s:Jacquet to conf}

In this section we will perform one of the main constructions in this paper: we will use the functor
$$\fJ^{\bHecke}_{!*,\on{Fact}}:\Whit_{q,x}(G) \to \Omega_q^{\Whit_{!*}}\on{-FactMod}(\Shv_{\CG^T}(\Gr^{\omega^\rho}_{T,\Ran_x}))$$
constructed earlier into a functor to produce a functor
$$\Phi^{\bHecke}_{\on{Fact}}:\Whit_{q,x}(G) \to \Omega^{\on{small}}_q\on{-FactMod}(\Shv_{\CG^\Lambda}(\Conf_{\infty\cdot x})).$$

\ssec{Support of the Jacquet functor}

As a warm-up for what follows, we will show that the support of objects in the image of the functor $\fJ_{!*,\on{sprd}}$
is contained in the \emph{non-positive} part of $\Gr^{\omega^\rho}_{T,\Ran_x}$, see \secref{ss:Conf Gr}. 

\sssec{}

We will presently prove: 

\begin{prop} \label{p:support} \hfill

\smallskip

\noindent{\em(a)} The support of the object 
$$\Omega_q^{\Whit_{!*}}\in \Shv_{\CG^T}(\Gr^{\omega^\rho}_{T,\Ran})$$
is contained in $(\Gr^{\omega^\rho}_{T,\Ran})^{\on{non-pos}}$. 

\smallskip

\noindent{\em(b)} 
The unit map 
\begin{equation} \label{e:unit Omega}
\on{unit}_*(\omega_\Ran)\to \Omega_q^{\Whit_{!*}}
\end{equation} 
gives rise to an isomorphism
$$\omega_\Ran\to \on{unit}^!(\Omega_q^{\Whit_{!*}}).$$ 
\end{prop}

\sssec{Proof of \propref{p:support}(a)}

By factorization, the assertion of point (a) of the proposition is equivalent to the following: the support of the object
$$\fJ_{!*,x}(\on{Vac}_{\Whit,x}) \in \Shv_{\CG^T}(\Gr^{\omega^\rho}_{T,x})$$
is contained in the union of the connected components of $\Gr^{\omega^\rho}_{T,x}$ corresponding to the elements of $\Lambda^{\on{neg}}$. 

\medskip

To prove this, it suffices to show that if
$$\left(\{t^\lambda\} \times \on{supp}(W^{0,*})\right) \cap \on{supp}({}_{\Gr_T}\!\ICsm_{q^{-1},x})\neq \emptyset,$$
then $\lambda\in \Lambda^{\on{neg}}$. 

\medskip

The intersection
$$(p_T)^{-1}(\{t^\lambda\})\cap \on{supp}({}_{\Gr_T}\!\ICsm_{q^{-1},x})$$
equals $\ol{S}^{-,\lambda}$, viewed as a subset of $\Gr^{\omega^\rho}_{G,x}$.  This is while 
$$\on{supp}(W^{0,*})=\ol{S}^0\subset \Gr^{\omega^\rho}_{G,x}.$$

Now, the assertion follows from the fact that 
$$\ol{S}^{-,\lambda}\cap \ol{S}^0\neq \emptyset \, \Rightarrow\, \lambda\in \Lambda^{\on{neg}}.$$

\sssec{Proof of \propref{p:support}(b)}

To prove point (b) we note that the restriction of $_{\Gr_T}\!\ICsm_{q^{-1},x}$ to
$$\{1\}\times \Gr^{\omega^\rho}_{G,x}\subset \Gr^{\omega^\rho}_{T,x}\times \Gr^{\omega^\rho}_{G,x}$$
identifies with $\ICsm_{q^{-1},x}$. Hence, we need to show that the map
$$\sfe\to H(\Gr^{\omega^\rho}_{G,x},\ICsm_{q^{-1},x}\sotimes W^{0,*}),$$
induced by \eqref{e:unit Omega}, is an isomorphism.

\medskip

We recall that $\on{supp}(\ICsm_{q^{-1},x})=\ol{S}^{-,0}$, while $\ol{S}^{-,0}\cap \ol{S}^0=\{1\}$.
Hence, 
$$\ICsm_{q^{-1},x}\sotimes W^{0,*}\simeq \delta_{1,\Gr},$$
and the assertion follows. 

\sssec{}

As an immediate corollary of \propref{p:support}, we obtain:

\begin{cor}  \label{x:support}
The support of the objects in the essential image of the functor 
$$\fJ_{!*,\on{sprd}}:\Whit_{q,x}(G)\to \Shv_{\CG^T}(\Gr^{\omega^\rho}_{T,\Ran_x})$$
is contained in $(\Gr^{\omega^\rho}_{T,\Ran_x})^{\on{non-pos}}_{\infty\cdot x}$.
\end{cor}

Thus, we can consider the functors $\fJ_{!*,\on{Fact}}$ and $\fJ^{\bHecke}_{!*,\on{Fact}}$ as taking values in
$$\Omega_q^{\Whit_{!*}}\mod\on{-FactMod}\left(\Shv_{\CG^T}((\Gr^{\omega^\rho}_{T,\Ran})^{\on{non-pos}}_{\infty\cdot x})\right).$$

\ssec{Jacquet functor as mapping to configuration spaces}

In this subsection we will finally construct the functors
$$\Phi_{\on{Fact}}:\Whit_{q,x}(G) \to \Omega^{\on{small}}_q\on{-FactMod}(\Shv_{\CG^\Lambda}(\Conf_{\infty\cdot x}))$$
and
$$\Phi^{\bHecke}_{\on{Fact}}:\bHecke(\Whit_{q,x}(G)) \to \Omega^{\on{small}}_q\on{-FactMod}(\Shv_{\CG^\Lambda}(\Conf_{\infty\cdot x}))$$

\sssec{}

Recall the factorization subspace
\begin{equation} \label{e:neg non pos}
(\Gr^{\omega^\rho}_{T,\Ran})^{\on{neg}}\hookrightarrow (\Gr^{\omega^\rho}_{T,\Ran})^{\on{non-pos}},
\end{equation}
see \secref{ss:Conf Gr}. 

\medskip

Let 
$$\Omega_{q,\on{red}}^{\Whit_{!*}}\in \Shv_{\CG^T}((\Gr^{\omega^\rho}_{T,\Ran})^{\on{neg}})$$ denote the factorization algebra 
obtained from $\Omega_q^{\Whit_{!*}}$ by applying the functor of !-pullback with respect to \eqref{e:neg non pos},
(see \secref{sss:fact funct geom pullback}).

\sssec{}

Consider the closed embedding
\begin{equation} \label{e:neg non pos mod}
(\Gr^{\omega^\rho}_{T,\Ran_x})^{\on{neg}}_{\infty\cdot x}\hookrightarrow (\Gr^{\omega^\rho}_{T,\Ran_x})^{\on{non-pos}}_{\infty\cdot x},
\end{equation}
see \secref{sss:neg part Gr ptd}.

\medskip

Let $\fJ_{!*,\on{Fact,red}}$ (resp., $\fJ^{\bHecke}_{!*,\on{Fact,red}}$) be the functor obtained from 
$\fJ_{!*,\on{Fact}}$ (resp., $\fJ^{\bHecke}_{!*,\on{Fact}}$) by composing with the functor
\begin{equation} \label{e:red modules}
\Omega_q^{\Whit_{!*}}\on{-FactMod}(\Shv_{\CG^T}((\Gr^{\omega^\rho}_{T,\Ran})^{\on{non-pos}}_{\infty\cdot x}))\to 
\Omega_{q,\on{red}}^{\Whit_{!*}}\on{-FactMod}(\Shv_{\CG^T}((\Gr^{\omega^\rho}_{T,\Ran})^{\on{neg}}_{\infty\cdot x})),
\end{equation}
given by !-pullback along \eqref{e:neg non pos mod} (see \secref{sss:fact funct mod geom}).

\begin{rem}  \label{r:unitality Gr}
Although we will not use this in the present work, one can show that the functor \eqref{e:red modules} is \emph{almost}
an equivalence. Namely, inside 
$$\Omega_q^{\Whit_{!*}}\on{-FactMod}(\Shv_{\CG^T}((\Gr^{\omega^\rho}_{T,\Ran})^{\on{non-pos}}_{\infty\cdot x}))$$
one singles out a \emph{full subcategory} (one that contains the essential image of the functors 
$\fJ_{!*,\on{Fact}}$ and $\fJ^{\bHecke}_{!*,\on{Fact}}$) that consists of \emph{unital} factorization modules, to be denoted
$$\Omega_q^{\Whit_{!*}}\on{-FactMod}(\Shv_{\CG^T}((\Gr^{\omega^\rho}_{T,\Ran})^{\on{non-pos}}_{\infty\cdot x}))_{\on{untl}}.$$
Now, the functor \eqref{e:red modules} defines an equivalence from this subcategory to 
$$\Omega_q^{\Whit_{!*}}\on{-FactMod}(\Shv_{\CG^T}((\Gr^{\omega^\rho}_{T,\Ran})^{\on{non-pos}}_{\infty\cdot x}))_{\on{untl}}\to 
\Omega_{q,\on{red}}^{\Whit_{!*}}\on{-FactMod}(\Shv_{\CG^T}((\Gr^{\omega^\rho}_{T,\Ran})^{\on{neg}}_{\infty\cdot x})).$$
This is due to the fact that the objects $\Omega_q^{\Whit_{!*}}$ and $\Omega_{q,\on{red}}^{\Whit_{!*}}$, both
viewed as factorization algebras in $\Shv_{\CG^T}(\Gr^{\omega^\rho}_{T,\Ran})$, are related by the mutually
inverse procedures of ``adding the unit" and ``passage to the augmentation ideal".
\end{rem}

\sssec{}

We now apply the equivalence of \secref{sss:fact Gr T vs Conf}. Let
$$\Omega_{q,\Conf}^{\Whit_{!*}}\in \Shv_{\CG^\Lambda}(\Conf)$$
be the factorization algebra that corresponds to
$$\Omega_{q,\on{red}}^{\Whit_{!*}}\in \Shv_{\CG^T}((\Gr^{\omega^\rho}_{T,\Ran})^{\on{neg}})$$
under the above equivalence.

\medskip

Consider the resulting equivalence of \eqref{e:compare fact mod}
\begin{equation} \label{e:modules from Gr to Conf}
\Omega_{q,\on{red}}^{\Whit_{!*}}\on{-FactMod}(\Shv_{\CG^T}((\Gr^{\omega^\rho}_{T,\Ran})^{\on{non-pos}}_{\infty\cdot x}))\simeq
\Omega_{q,\Conf}^{\Whit_{!*}}\on{-FactMod}(\Shv_{\CG^\Lambda}(\Conf_{\infty\cdot x})).
\end{equation}

Let
$$\fJ_{!*,\on{Fact},\Conf} :\Whit_{q,x}(G)\to \Omega_{q,\Conf}^{\Whit_{!*}}\on{-FactMod}$$
and
$$\fJ^{\bHecke}_{!*,\on{Fact},\Conf} :\bHecke(\Whit_{q,x}(G))\to \Omega_{q,\Conf}^{\Whit_{!*}}\on{-FactMod}$$
denote the functors obtained from $\fJ_{!*,\on{Fact,red}}$ and $\fJ^{\bHecke}_{!*,\on{Fact,red}}$, respectively,
by composing with the equivalence \eqref{e:modules from Gr to Conf}.

\sssec{}

We will denote by
$$'\Omega_q^{\on{small}}$$
the factorization algebra in $\Shv_{\CG^\Lambda}(\Conf)$ obtained from $\Omega_{q,\Conf}^{\Whit_{!*}}$ by applying
the cohomological shift $[\langle \lambda,2\check\rho\rangle]$ on the connected component $\Conf^\lambda$ of $\Conf$. 

\medskip

The functor
\begin{equation} \label{e:shift mod}
\Shv_{\CG^\Lambda}(\Conf_{\infty\cdot x})\to \Shv_{\CG^\Lambda}(\Conf_{\infty\cdot x}),
\end{equation}
given by cohomological shift $[\langle \mu,2\check\rho\rangle]$ on the connected component
$\Conf^\mu_{\infty\cdot x}$ of $\Conf_{\infty\cdot x}$
defines an equivalence
$$\Omega_{q,\Conf}^{\Whit_{!*}}\on{-FactMod}\to {}'\Omega_q^{\on{small}}\on{-FactMod}.$$

\sssec{}

Let 
$$\Phi_{\on{Fact}}:\Whit_{q,x}(G)\to {}'\Omega_q^{\on{small}}\on{-FactMod}$$
and
$$\Phi^{\bHecke}_{\on{Fact}}:\bHecke(\Whit_{q,x}(G))\to {}'\Omega_q^{\on{small}}\on{-FactMod}$$
be the functors obtained from $\fJ_{!*,\on{Fact},\Conf}$ and $\fJ^{\bHecke}_{!*,\on{Fact},\Conf}$,
respectively, by composing with \eqref{e:shift mod}. 

\sssec{}

Let
$$\Phi:\Whit_{q,x}(G)\to \Shv_{\CG^\Lambda}(\Conf_{\infty\cdot x}) \text{ and }
\Phi^{\bHecke}:\bHecke(\Whit_{q,x}(G))\to \Shv_{\CG^\Lambda}(\Conf_{\infty\cdot x})$$
denote the functors obtained from $\Phi_{\on{Fact}}$ and $\Phi^{\bHecke}_{\on{Fact}}$, respectively,
by composing with the forgetful functor 
$$\oblv_{\on{Fact}}:  \Omega_{q,\Conf}^{\Whit_{!*}}\on{-FactMod}\to \Shv_{\CG^\Lambda}(\Conf_{\infty\cdot x}).$$

\sssec{}

Recall the subcategory
$$\Shv_{\CG^\Lambda}(\Conf_{\infty\cdot x})^{\on{loc.c}}\subset \Shv_{\CG^\Lambda}(\Conf_{\infty\cdot x}),$$
see \secref{sss:locally compact}. Recall also the Verdier duality equivalence 
$$\BD^{\on{Verdier}}: (\Shv_{\CG^\Lambda}(\Conf_{\infty\cdot x})^{\on{loc.c}})^{\on{op}} \to 
\Shv_{(\CG^\Lambda)^{-1}}(\Conf_{\infty\cdot x})^{\on{loc.c}},$$
see \secref{sss:locally compact}. 

\medskip

A key technical assertion, which we will prove in \secref{ss:proof of Verdier} is the following:

\begin{thm}  \label{t:duality pres}
The functor $\Phi$ sends compact objects in $\Whit_{q,x}(G)$ to $\Shv_{\CG^\Lambda}(\Conf_{\infty\cdot x})^{\on{loc.c}}$,
and the diagram 
$$
\CD 
(\Whit_q(G)^c)^{\on{op}} @>{\BD^{\on{Verdier}}}>>  \Whit_{q^{-1}}(G)^c \\
@V{\Phi}VV   @VV{\Phi}V   \\
(\Shv_{\CG^\Lambda}(\Conf_{\infty\cdot x})^{\on{loc.c}})^{\on{op}}  @>{\BD^{\on{Verdier}}}>> 
\Shv_{(\CG^\Lambda)^{-1}}(\Conf_{\infty\cdot x})^{\on{loc.c}},
\endCD
$$
commutes, where the upper horizontal arrow is the equivalence \eqref{e:self duality for Whit comp}.
\end{thm} 

\begin{rem}
A curious aspect of \thmref{t:duality pres} is that we will use global methods to prove it. In fact, it is this theorem
that was the reason to introduce the global version of the Whittaker category, $\Whit_{q,\on{glob}}(G)$. 
\end{rem} 

\ssec{Explicit description of the functor $\Phi$}

We will now make a pause and describe explicitly the functor $\Phi$ and the sheaf $'\Omega_q^{\on{small}}$. 

\sssec{}

Recall the context of \secref{ss:twist T bundle}. Take
$$\CY:=(\Gr^{\omega^\rho}_{T,\Ran})^{\on{neg}}.$$

The corresponding twist $_\CY\Gr^{\omega^\rho}_G$ of $\Gr^{\omega^\rho}_{G,\Ran}$ identifies with
$$\Gr^{\omega^\rho}_{G,\Ran}\underset{\Ran}\times (\Gr^{\omega^\rho}_{T,\Ran})^{\on{neg}}\simeq 
\Gr^{\omega^\rho}_{G,\Conf}\underset{\Conf}\times (\Gr^{\omega^\rho}_{T,\Ran})^{\on{neg}}.$$

\medskip

Since the map 
$$_\CY\Gr^{\omega^\rho}_G\simeq \Gr^{\omega^\rho}_{G,\Conf}\underset{\Conf}\times (\Gr^{\omega^\rho}_{T,\Ran})^{\on{neg}}\to
\Gr^{\omega^\rho}_{G,\Conf}$$
is a base-change of the map $(\Gr^{\omega^\rho}_{T,\Ran})^{\on{neg}}\to \Conf$, it induces an equivalence
between spaces of gerbes and gerbe-twisted sheaves.

\medskip

Let $\CG^{G,T,\on{ratio}}$ be the gerbe on $\Gr^{\omega^\rho}_{G,\Conf}$ whose pullback to $_\CY\Gr^{\omega^\rho}_G$
gives $_\CY\CG^G$. Let $\CG^\Lambda$ be the pullback of the same-named gerbe along $\Gr^{\omega^\rho}_{G,\Conf}\to \Conf$. 
Note that 
$$\CG^{G,T,\on{ratio}}\simeq \CG^G\otimes (\CG^\Lambda)^{-1},$$
where $\CG^G$ is the gerbe on $\Gr^{\omega^\rho}_{G,\Conf}$ whose further pullback to $_\CY\Gr^{\omega^\rho}_G$
is the pullback of the same-named gerbe on $\Gr^{\omega^\rho}_{G,\Ran}$ along
\begin{equation} \label{e:twist Y and Gr}
_\CY\Gr^{\omega^\rho}_G
\simeq \Gr^{\omega^\rho}_{G,\Ran}\underset{\Ran}\times (\Gr^{\omega^\rho}_{T,\Ran})^{\on{neg}}\to \Gr^{\omega^\rho}_{G,\Ran}.
\end{equation} 

\sssec{}  \label{sss:ICs conf}

Consider the object
$$_\CY\!\ICsm_{q^{-1}}\in \Shv_{({}_\CY\CG^G)^{-1}}({}_\CY\Gr^{\omega^\rho}_G).$$

Let 
$$\ICsm_{q^{-1},\Conf}\in \Shv_{(\CG^{G,T,\on{ratio}})^{-1}}(\Gr^{\omega^\rho}_{G,\Conf})$$
be the object that corresponds to it under the equivalence
$$\Shv_{(\CG^{G,T,\on{ratio}})^{-1}}(\Gr^{\omega^\rho}_{G,\Conf})\to \Shv_{({}_\CY\CG)^{-1}}({}_\CY\Gr^{\omega^\rho}_G).$$

\sssec{}

Let 
$$\on{Vac}_{\Whit,\Conf}\in \Shv_{\CG^G}(\Gr^{\omega^\rho}_{G,\Conf})$$
be the object that corresponds under the equivalence
$$\Shv_{\CG^G}(\Gr^{\omega^\rho}_{G,\Conf})\to \Shv_{\CG^G}({}_\CY\Gr^{\omega^\rho}_G)$$
to the pullback of $\on{Vac}_{\Whit,\Ran}$ along \eqref{e:twist Y and Gr}. 

\sssec{}  \label{sss:Omega expl}

Note that the tensor product
\begin{equation} \label{e:Vac ten semiinf}
\on{Vac}_{\Whit,\Conf}\sotimes \ICsm_{q^{-1},\Conf}
\end{equation} 
is naturally an object of 
$$\Shv_{\CG^\Lambda}(\Gr^{\omega^\rho}_{G,\Conf}).$$

Unwinding the constructions, we obtain that 
$$'\Omega^{\on{small}}_q\in \Shv_{\CG^\Lambda}(\Conf)$$
is the direct image of \eqref{e:Vac ten semiinf} along the projection
$$\Gr^{\omega^\rho}_{G,\Conf}\to \Conf,$$
cohomologically shifted by $[\langle \lambda,2\check\rho\rangle]$ on the connected component $\Conf^\lambda$ of $\Conf$. 

\sssec{}

Let us now modify the above discussion by taking
$$\CY:=(\Gr^{\omega^\rho}_{T,\Ran_x})^{\on{neg}}_{\infty\cdot x}.$$

Consider the corresponding version of the affine Grassmannian $\Gr^{\omega^\rho}_{G,\Conf_{\infty\cdot x}}$, so that
$$ \Gr^{\omega^\rho}_{G,\Conf_{\infty\cdot x}} \underset{\Conf_{\infty\cdot x}}\times 
(\Gr^{\omega^\rho}_{T,\Ran_x})^{\on{neg}}_{\infty\cdot x}\simeq \Gr^{\omega^\rho}_{G,\Ran_x}
\underset{\Ran_x}\times (\Gr^{\omega^\rho}_{T,\Ran_x})^{\on{neg}}_{\infty\cdot x}.$$

\sssec{}

Applying the same procedure as above, we obtain an object
$$\ICsm_{q^{-1},\Conf_{\infty\cdot x}}\in \Shv_{(\CG^{G,T,\on{ratio}})^{-1}}(\Gr^{\omega^\rho}_{G,\Conf_{\infty\cdot x}}),$$
and a functor
$$\on{sprd}_{\Conf_{\infty\cdot x}}:\Whit_{q,x}(G)\to 
\Shv_{\CG^G}(\Gr^{\omega^\rho}_{G,\Conf_{\infty\cdot x}}).$$

\sssec{}  \label{sss:Phi expl}

For $\CF\in \Whit_{q,x}(G)$, the tensor product
\begin{equation} \label{e:Whit ten semiinf}
\on{sprd}_{\Conf_{\infty\cdot x}}(\CF)\sotimes \ICsm_{q^{-1},\Conf_{\infty\cdot x}}
\end{equation} 
is naturally an object of $\Shv_{\CG^\Lambda}(\Gr^{\omega^\rho}_{G,\Conf_{\infty\cdot x}})$. 

\medskip

Then for $\CF$ as above the object 
$$\Phi(\CF)\in \Shv_{\CG^\Lambda}(\Conf_{\infty\cdot x})$$
is the direct image of \eqref{e:Whit ten semiinf} along the projection
$$\Gr^{\omega^\rho}_{G,\Conf_{\infty\cdot x}}\to \Conf_{\infty\cdot x},$$
cohomologically shifted by $[\langle \lambda,2\check\rho\rangle]$ on the connected component 
$\Conf^\lambda_{\infty\cdot x}$ of $\Conf_{\infty\cdot x}$. 

\ssec{Identification of factorization algebras}   \label{ss:Omega expl} 

We now come to the first crucial computational results of this paper (which makes everything work).

\sssec{}

Recall the factorization algebra $\Omega_q^{\on{small}}$ introduced in \secref{s:Omega small}. We claim:

\begin{thm} \label{t:ident vacuum}
There exists a canonical isomorphism of factorization algebras 
$$'\Omega_q^{\on{small}}\simeq \Omega_q^{\on{small}}$$
for the choice of the lines $\sff^{i,\on{fact}}$ specified in \secref{sss:sfi from Whit}. 
\end{thm}

The assertion of \thmref{t:ident vacuum} naturally splits into two parts:

\begin{prop} \label{p:ident vacuum open}
For the choice of the lines $\sff^{i,\on{fact}}$ specified in \secref{sss:sfi from Whit}, 
we have a canonical isomorphism of factorization algebras in $\Shv_{\CG^\Lambda}(\overset{\circ}{\Conf})$
$$'\Omega_q^{\on{small}}|_{\overset{\circ}{\Conf}}\simeq \Omega_q^{\on{small}}|_{\overset{\circ}{\Conf}}.$$
\end{prop} 

\begin{thm} \label{t:GM from open}
The object $'\Omega^{\on{small}}_q\in \Shv_{\CG^\Lambda}(\Conf)$ is a perverse sheaf,
which is the Goresky-MacPherson extension of its restriction to $\overset{\circ}{\Conf}$.
\end{thm} 

We will prove \propref{p:ident vacuum open} in the rest of this subsection. The proof of 
\thmref{t:GM from open} will be given in \secref{ss:proof of GM}. 

\sssec{}  \label{sss:sfi from Whit} 

Recall that Kummer theory attaches to an element $c\in \sfe^\times(-1)$ a character sheaf on $\BG_m$;
to be denoted $\Psi_c$. When $c\neq 0$, the extension of $\Psi_c$ along $\BG_m\to \BA^1$ is \emph{clean};
by a slight abuse of notation we will denote it by the same character $\Psi_c$. 

\medskip

Recall also that a geometric metaplectic datum defines for each vertex of the Dynkin diagram an element
$q_i\in \sfe^\times(-1)$. Recall also that we impose a non-degeneracy condition on the geomtetric metaplectic datum,
which says that all $q_i$ are non-trivial, see Definition \ref{d:non-deg q}.

\medskip

Finally, recall that for our definition of the Whittaker category, we \emph{chose} an Artin-Schreier character sheaf
$\chi$ on $\BG_a$.

\medskip

We let $\sff^{i,\on{fact}}$ be the line (a ``Gauss sum")
$$H^1(\BA^1,\Psi_{q_i}\otimes \chi).$$

Note this cohomology is a geomtric Gauss sum, in particular, $H^0=H^2=0$. 

\medskip

By factorization, in order to prove \propref{p:ident vacuum open}, it suffices to perform the calculation on each individual
$\Conf^{-\alpha_i}\simeq X$. To simplify the notation, we will perform the calculation at a fixed point $x\in X$. 

\sssec{}  \label{sss:shifted IC}

For $\lambda\in \Lambda$, consider the action of the element $t^\lambda\in \fL(T)_x$ on $\Gr^{\omega^\rho}_{G,x}$.
It induces a functor
$$\Shv_{(\CG^G)^{-1}}(\Gr^{\omega^\rho}_{G,x})^{\fL^+(T)_x}\to 
\Shv_{(\CG^G)^{-1}\otimes \CG^\Lambda_{\lambda\cdot x}}(\Gr^{\omega^\rho}_{G,x})^{\fL^+(T)_x}.$$

Note that $t^\lambda\star \ICsm_{q^{-1},x}$ identifies with the restriction of $\ICsm_{q^{-1,\Conf}}$ to
the fiber over the point $$\lambda\cdot x\in \Conf.$$

\medskip

Set
$$\ICslm_{q^{-1},x}=t^\lambda\star \ICsm_{q^{-1},x}[\langle \lambda,2\check\rho\rangle]
\in \Shv_{(\CG^G)^{-1}\otimes \CG^\Lambda_{\lambda\cdot x}}(\Gr^{\omega^\rho}_{G,x}).$$

Note that the restriction of the gerbe $(\CG^G)^{-1}\otimes \CG^\Lambda_{\lambda\cdot x}$ to $S^{-,\lambda}\subset \Gr^{\omega^\rho}_{G,x}$
is canonically trivialized, and in terms of this trivialization, we have
$$\ICslm_{q^{-1},x}|_{S^{-,\lambda}}\simeq \omega_{S^{-,\lambda}}[\langle \lambda,2\check\rho\rangle].$$

\sssec{}

Thus, we have to show that
\begin{equation} \label{e:vac calc one}
H(\Gr_{G,x}^{\omega^\rho},W^{0,*}\sotimes \on{IC}^{-\alpha_i+\frac{\infty}{2},-}_{q^{-1},x})\simeq 
\Gamma(\BA^1,\Psi_{q_i}\otimes \chi).
\end{equation}

Note that the tensor product $W^{0,*}\sotimes \on{IC}^{-\alpha_i+\frac{\infty}{2},-}_{q^{-1},x}$
is an \emph{untwisted} sheaf on $\Gr_{G,x}^{\omega^\rho}$ due to the trivialization of the restriction of
$\CG^\Lambda$ to the point $-\alpha_i\cdot x\in \Conf$, given by \secref{sss:gerbes from G}. 

\sssec{}

First, by \corref{c:Whittaker strata}(b), the object $W^{0,*}$ is the *-extension of its restriction to $S^0$. 

\medskip

The object $\on{IC}^{-\alpha_i+\frac{\infty}{2},-}_{q^{-1},x}$ is supported on $\ol{S}^{-,-\alpha_i}$. Note now that the residue map
defines an isomorphism
\begin{equation} \label{e:simple semiinf intersect}
S^0\cap \ol{S}^{-,-\alpha_i}\to \BA^1.
\end{equation}

\medskip

We will show that we have a canonical isomorphism 
\begin{equation} \label{e:vac calc two}
(W^{0,*}\sotimes \ICsm_{q^{-1},\Conf})|_{S^0\cap \ol{S}^{-,-\alpha_i}} \simeq \Psi_{q_i}\otimes \chi
\end{equation} 
taking place in $\Shv(\BA^1)$. 

\sssec{}

Note that the open subset
\begin{equation} \label{e:open orbit}
S^0\cap S^{-,-\alpha_i}\subset S^0\cap \ol{S}^{-,-\alpha_i}
\end{equation}
corresponds to
$$\BA^1-0 \subset \BA^1,$$
while $\{0\}\subset \BA^1$ corresponds to 
\begin{equation} \label{e:smaller orbit}
S^0\cap S^{-,0}\subset S^0\cap \ol{S}^{-,-\alpha_i}.
\end{equation}

\medskip

We claim that the object 
$$(W^{0,*}\sotimes \ICsm_{q^{-1},\Conf})|_{S^0\cap \ol{S}^{-,-\alpha_i}}$$
is the *-extension of its own restriction along the open embedding \eqref{e:open orbit}. I.e., we claim that its !-restriction
along \eqref{e:smaller orbit} vanishes. However, this follows from \lemref{l:non sharp strata} and the assumption that $q_i$
is non-trivial. 

\medskip

Hence, it suffices to establish an isomorphism 
\begin{equation} \label{e:vac calc three}
(W^{0,*}\sotimes \ICsm_{q^{-1},\Conf})|_{S^0\cap S^{-,-\alpha_i}} \simeq \Psi_{q_i}\otimes \chi 
\end{equation} 
taking place in $\Shv(\BA^1-0)$. 

\sssec{}

Recall that the gerbe $\CG^G|_{S^0}$ is trivialized, and in terms of this trivialization,
the object $W^{0,*}|_{S^0}$ corresponds to $\omega_{S^0}$
tensored by the pullback of $\chi$. Hence, in terms of \emph{this} trivialization, 
$$W^{0,*}|_{S^0\cap S^{-,-\alpha_i}}\simeq \chi\otimes \omega_{\BA^1-0}.$$

\medskip

Recall also that the gerbe $(\CG^G)^{-1}\otimes \CG^T|_{S^{-,-\alpha_i}}$ is also canonically
trivialized, and in terms of \emph{this} trivialization, $\on{IC}^{-\alpha_i+\frac{\infty}{2},-}_{q^{-1},x}|_{S^{-,-\alpha_i}}$
corresponds to $\omega_{S^{-,-\alpha_i}}[-2]$.  Hence, in terms of \emph{this} trivialization,
$$\ICsm_{q^{-1},\Conf}|_{S^0_{\Conf^{-\alpha_i}}\cap S^{-,-\alpha_i}_{\Conf^{-\alpha_i}}}\simeq \sfe_{\BA^1-0}.$$

\medskip

Hence, in order to establish \eqref{e:vac calc three}, we need to show that the resulting trivialization of
$$\CG^G|_{S^0\cap S^{-,-\alpha_i}} \otimes 
((\CG^G)^{-1}\otimes \CG^T)|_{S^0\cap S^{-,-\alpha_i}}\simeq
\CG^T|_{S^0\cap S^{-,-\alpha_i}}$$ differs 
from the trialization of $\CG^T|_{S^{-,-\alpha_i}}$ of \secref{sss:gerbes from G}
by the local system equal to the pullback of $\Psi_{q_i}$ along the residue map 
$$S^0\cap S^{-,-\alpha_i}\to \BA^1-0.$$

\medskip

This is a calculation performed in the next subsection. 

\ssec{Calculation of the discrepancy}

\sssec{}

As in \secref{sss:gerbes from G}, the calculation reduces to the case of $G=SL_2$, in which case the gerbe $\CG^G$
is canonically of the form $\det_{SL_2}^a$ for $a\in \sfe^\times(-1)$. 

\medskip

The corresponding quadratic form $q$ takes the value $a$ on the (unique)
coroot $\alpha_i$, i.e., $q_i=a$.

\medskip

The line bundle $\det_{SL_2}$ admits a canonical trivialization when restricted to $S^0$ and also, by 
\secref{sss:det calc}, to $S^{-,-\alpha_i}$. We need to show that the discrepancy of these two trivializations,
viewed as a function,
$$S^0 \cap S^{-,-\alpha_i}\to \BG_m$$
equals the residue map. 

\sssec{}

A point of $S^0 \cap S^{-,-\alpha_i}$ is a rank $2$-bundle $\CM$ on $X$ that fits into a diagram
$$
\CD
& & \omega^{\otimes \frac{1}{2}}(x)  \\
& & @AAA \\
\omega^{\otimes \frac{1}{2}}  @>>>  \CM   @>>>  \omega^{\otimes -\frac{1}{2}} \\
& & @AAA  \\
& & \omega^{\otimes -\frac{1}{2}}(-x) 
\endCD
$$
where the row and the column are exact sequences. Such an $\CM$ is uniquely determined by the choice of a line
$$\ell\subset \omega^{\otimes \frac{1}{2}}_x\oplus  \omega^{\otimes -\frac{1}{2}}(-x)_x$$
that projects isomorphically to both factors
\begin{equation} \label{e:two proj}
\omega^{\otimes -\frac{1}{2}}(-x)_x \leftarrow \ell \to \omega^{\otimes \frac{1}{2}}_x.
\end{equation}

\medskip

The resulting isomorphism
\begin{equation} \label{e:discrep}
\omega^{\otimes -\frac{1}{2}}(-x)_x\to \omega^{\otimes \frac{1}{2}}_x
\end{equation}
can be regarded as a non-zero element of $\omega^{\otimes \frac{1}{2}}(x)_x$, where $\omega^{\otimes \frac{1}{2}}(x)_x\simeq \BA^1$
by the residue map. It is easy to see that the thus constructed map $S^0 \cap S^{-,-\alpha_i}\to \BA^1-0$ is the residue map 
of \eqref{e:simple semiinf intersect}. 

\sssec{}

The two embeddings
$$\omega^{\otimes \frac{1}{2}}(x) \oplus \omega^{\otimes -\frac{1}{2}}(-x) \hookleftarrow 
\omega^{\otimes \frac{1}{2}}\oplus \omega^{\otimes -\frac{1}{2}}(-x) \hookrightarrow 
\omega^{\otimes \frac{1}{2}}\oplus \omega^{\otimes -\frac{1}{2}}$$
induce the identifications
\begin{multline} \label{e:det1}
\on{rel.det}(\omega^{\otimes \frac{1}{2}}\oplus \omega^{\otimes -\frac{1}{2}}(-x),\omega^{\otimes \frac{1}{2}}\oplus \omega^{\otimes -\frac{1}{2}}) \simeq \\
\simeq \on{rel.det}(\omega^{\otimes \frac{1}{2}}(x) \oplus \omega^{\otimes -\frac{1}{2}}(-x),\omega^{\otimes \frac{1}{2}}\oplus \omega^{\otimes -\frac{1}{2}})
\otimes \omega^{\otimes \frac{1}{2}}_x\simeq \omega^{\otimes \frac{1}{2}}_x,
\end{multline}
(where the last isomorphism comes from \secref{sss:det calc})
and
\begin{multline} \label{e:det2}
\on{rel.det}(\omega^{\otimes \frac{1}{2}}\oplus \omega^{\otimes -\frac{1}{2}}(-x),\omega^{\otimes \frac{1}{2}}\oplus \omega^{\otimes -\frac{1}{2}}) \simeq \\
\simeq \on{rel.det}(\omega^{\otimes \frac{1}{2}}\oplus \omega^{\otimes -\frac{1}{2}},\omega^{\otimes \frac{1}{2}}\oplus \omega^{\otimes -\frac{1}{2}})
\otimes \omega^{\otimes \frac{1}{2}}_x\simeq \omega^{\otimes \frac{1}{2}}_x.
\end{multline}

However, by the construction of the isomorphism in \secref{sss:det calc}, it follows that the above two isomorphisms 
$$\on{rel.det}(\omega^{\otimes \frac{1}{2}}\oplus \omega^{\otimes -\frac{1}{2}}(-x),\omega^{\otimes \frac{1}{2}}\oplus \omega^{\otimes -\frac{1}{2}})\simeq
\omega^{\otimes \frac{1}{2}}_x$$
coincide. 

\sssec{}

The fiber of $\det_{SL_2}$ at $\CM$ is given by
$$\on{det.rel.}(\CM,\omega^{\otimes \frac{1}{2}}\oplus \omega^{\otimes -\frac{1}{2}}).$$

We have a canonical short exact sequence
$$0\to \omega^{\otimes \frac{1}{2}} \oplus \omega^{\otimes -\frac{1}{2}}(-x) \to \CM\to \ell\otimes \omega^{\otimes -1}_x\to 0,$$
where the line $\ell\otimes \omega^{\otimes -1}_x$ is regarded as a skyscraper sheaf at $x$. Hence, the fiber of $\det_{SL_2}$ at $\CM$
can be further identified with 
$$\on{det.rel.}(\omega^{\otimes \frac{1}{2}}\oplus \omega^{\otimes -\frac{1}{2}}(-x),\omega^{\otimes \frac{1}{2}}\oplus \omega^{\otimes -\frac{1}{2}})
\otimes  \ell\otimes \omega^{\otimes -1}_x,$$ 
and its two trivializations are given by 
\begin{multline*}
\on{det.rel.}(\omega^{\otimes \frac{1}{2}}\oplus \omega^{\otimes -\frac{1}{2}}(-x),\omega^{\otimes \frac{1}{2}}\oplus \omega^{\otimes -\frac{1}{2}})
\otimes  \ell\otimes \omega^{\otimes -1}_x \overset{\leftarrow \text{in \eqref{e:two proj}}}\simeq \\
\simeq \on{det.rel.}(\omega^{\otimes \frac{1}{2}}\oplus \omega^{\otimes -\frac{1}{2}}(-x),\omega^{\otimes \frac{1}{2}}\oplus \omega^{\otimes -\frac{1}{2}})
\otimes \omega^{\otimes -\frac{1}{2}}(-x)_x \otimes \omega^{\otimes -1}_x 
\overset{\text{\eqref{e:det1}}} 
\simeq \omega^{\otimes \frac{1}{2}}_x\otimes \omega^{\otimes -\frac{1}{2}}_x\simeq k
\end{multline*}
and
\begin{multline*}
\on{det.rel.}(\omega^{\otimes \frac{1}{2}}\oplus \omega^{\otimes -\frac{1}{2}}(-x),\omega^{\otimes \frac{1}{2}}\oplus \omega^{\otimes -\frac{1}{2}})
\otimes  \ell\otimes \omega^{\otimes -1}_x \overset{\to \text{in \eqref{e:two proj}}}\simeq \\
\simeq \on{det.rel.}(\omega^{\otimes \frac{1}{2}}\oplus \omega^{\otimes -\frac{1}{2}}(-x),\omega^{\otimes \frac{1}{2}}\oplus \omega^{\otimes -\frac{1}{2}})
\otimes \omega^{\otimes \frac{1}{2}}_x \otimes \omega^{\otimes -1}_x \overset{\text{\eqref{e:det2}}} 
\simeq \omega^{\otimes \frac{1}{2}}_x\otimes \omega^{\otimes -\frac{1}{2}}_x\simeq k,
\end{multline*}
respectively. 

\medskip

Hence, the discrepancy between the two is given by \eqref{e:discrep}, as required. 

\ssec{Properties of the functor $\Phi$ and its variants}

By \thmref{t:ident vacuum}, we can consider $\Phi_{\on{Fact}}$ and $\Phi^{\bHecke}_{\on{Fact}}$
as taking values in the category
$\Omega_q^{\on{small}}\on{-FactMod}$. In this subsection we will study some basic properties
of this functor. 

\sssec{} 

Recall (see \secref{ss:t on fact}) that sense $\Omega^{\on{small}}_q$ is perverse, the category 
$\Omega^{\on{small}}_q\on{-FactMod}$ has a t-structure for which the functor $\oblv_{\on{Fact}}$
is t-exact. Recall also the irreducible objects
$$\CM^{\lambda,!*}_{\on{Fact}}\in (\Omega^{\on{small}}_q\on{-FactMod})^\heartsuit.$$

Here is the key result, proved in \secref{ss:proof of GM} below:

\begin{thm}  \label{t:irred pres}
Let $\lambda\in \Lambda^+$ be restricted. Then
$$\Phi(W^{\lambda,!*})\simeq \oblv(\CM^{\lambda,!*}_{\on{Fact}}).$$
\end{thm}

The above theorem has a slew of consequences pertaining to the properties of the functors 
$\Phi$, $\Phi_{\on{Fact}}$ and $\Phi^{\bHecke}_{\on{Fact}}$.

\begin{cor}  \label{c:irred pres}
Let $\lambda\in \Lambda^+$ be restricted. Then
$$\Phi_{\on{Fact}}(W^{\lambda,!*})\simeq \CM^{\lambda,!*}_{\on{Fact}}.$$
\end{cor}

\begin{proof}
Follows from the fact that if $'\CM^{\lambda,!*}_{\on{Fact}}\in \Omega_q^{\on{small}}\on{-FactMod}$
is such that
$$\oblv({}'\CM^{\lambda,!*}_{\on{Fact}})\simeq \oblv(\CM^{\lambda,!*}_{\on{Fact}}),$$
then $'\CM^{\lambda,!*}_{\on{Fact}}\simeq \CM^{\lambda,!*}_{\on{Fact}}$.
\end{proof}

\begin{cor}  \label{c:irred pres Hecke}
For any $\mu\in \Lambda$,
$$\Phi^{\bHecke}_{\on{Fact}}(\bCM^{\mu,!*}_{\Whit})\simeq \CM^{\mu,!*}_{\on{Fact}}.$$
\end{cor}

\begin{proof}

As in \secref{sss:red to sc} and using \secref{sss:isogenies conf fact}, 
we can reduce the assertion of the proposition to the case when $H$
is such that its derived group is simply-connected. In this case we can write 
$\mu=\lambda+\gamma$, where $\lambda\in \Lambda^+$ is restricted and $\gamma\in \Lambda^\sharp$.
We have:
\begin{multline*}
\Phi^{\bHecke}_{\on{Fact}}(\bCM^{\mu,!*}_{\Whit})\simeq 
\Phi^{\bHecke}_{\on{Fact}}(\bCM^{\lambda,!*}_{\Whit})\otimes \sfe^\gamma)\simeq
\Phi^{\bHecke}_{\on{Fact}}(\bCM^{\lambda,!*}_{\Whit}))\star \delta_{t^\gamma}\simeq \\
\simeq \Phi^{\bHecke}_{\on{Fact}}(\ind_{\bHecke}(\bCM^{\lambda,!*}_{\Whit}))\star \delta_{t^\gamma}\simeq
\Phi_{\on{Fact}}(\bCM^{\lambda,!*}_{\Whit})\star \delta_{t^\gamma}\simeq
\CM^{\lambda,!*}_{\on{Fact}}\star \delta_{t^\gamma}\simeq \CM^{\mu,!*}_{\on{Fact}}.
\end{multline*}
\end{proof}

\begin{cor} \label{c:Phi Hecke exact} \hfill 

\smallskip

\noindent{\em(a)} The functor $\Phi^{\bHecke}_{\on{Fact}}$ is t-exact. 

\smallskip

\noindent{\em(b)} The functor $\Phi_{\on{Fact}}$ is t-exact. 

\smallskip

\noindent{\em(c)} The functor $\Phi$ is t-exact. 

\end{cor}

\begin{proof}

Point (a) is immediate from \corref{c:irred pres Hecke} using the fact that the t-structure on 
the category $\bHecke(\Whit_{q,x}(G))$ is Artinian (the latter by \propref{p:irred Hecke Whit}(a)).  

\medskip

Point (b)
follows from point (a) since
$$\Phi_{\on{Fact}}\simeq \Phi^{\bHecke}_{\on{Fact}}\circ \ind_{\bHecke},$$
where $\ind_{\bHecke}$ is also t-exact.

\medskip

Point (c) is logically equivalent to point (b). 

\end{proof}

\begin{rem} \label{r:Phi exact}
We will give an alternative proof of \corref{c:Phi Hecke exact}(c) in \secref{ss:proof of GM}, in the course of the proof
of \thmref{t:irred pres}.
\end{rem}

\ssec{Computation of stalks and proofs of Theorems \ref{t:GM from open} and \ref{t:irred pres}}  \label{ss:proof of GM}

In this subsection we will assume \thmref{t:duality pres} and deduce from it Theorems \ref{t:GM from open} and \ref{t:irred pres}. 

\sssec{}    \label{sss:stalk Phi vac}

In order to prove \thmref{t:GM from open}, using the Verdier autoduality statement established in \thmref{t:duality pres}
and factorization, it suffices to show the following:

\medskip

For $\lambda\in \Lambda^{\on{neg}}$ with $\lambda\neq -\alpha_i$, the !-fiber of $\Phi(W^0)$ at the 
point $\lambda\cdot x\in \Conf$ lives in cohomological degrees $\geq 2$. 

\medskip

We will now derive a general expression for the !-fiber, denoted 
$$\Phi(\CF)_{\lambda\cdot x}\in \Shv_{\CG^G_{\lambda\cdot x}}(\on{pt})=:\Vect_{\CG^G_{\lambda\cdot x}}$$
of $\Phi(\CF)$ at $\lambda\cdot x\in \Conf_{\infty\cdot x}$
for $\lambda\in \Lambda$ and $\CF\in \Whit_q(G)$. By factorization, this would give an answer to the !-fiber of $\Phi(\CF)$ at 
any other point of $\Conf_{\infty\cdot x}$. 

\medskip

We will show that for any $\CF\in (\Whit_q(G))^\heartsuit$, 
the object $\Phi(\CF)_{\lambda\cdot x}$ lives in cohomological degrees $\geq 0$.
Taking into account \thmref{t:duality pres} and \corref{c:duality irred}, this will imply that the functor $\Phi$ is t-exact
(see Remark \ref{r:Phi exact} above).

\medskip

Furthermore, we will show that for $\mu$ restricted, $\Phi(W^{\mu,!*})_{\lambda\cdot x}$ lives in cohomological degrees $\geq 1$
for $\lambda\neq \mu$. Using \thmref{t:duality pres} and \corref{c:duality irred}, this will imply \thmref{t:irred pres}.

\sssec{}

Recall the notation 
$$\ICslm_{q^{-1},x}\in \Shv_{(\CG^G)^{-1}\otimes \CG^\Lambda_{\lambda\cdot x}}(\Gr^{\omega^\rho}_{G,\Ran}),$$ 
see \secref{sss:shifted IC}. 

\medskip

Unwinding the definitions, we obtain:
\begin{equation} \label{e:fiber of Phi}
\Phi(\CF)_{\lambda\cdot x}\simeq \Gamma(\Gr^{\omega^\rho}_{G,x},\CF\sotimes \ICslm_{q^{-1},x})\in \Vect_{\CG^G_{\lambda\cdot x}}.
\end{equation}

\sssec{}

We now claim: 

\begin{prop}   \label{p:fiber of Phi top} For $\CF\in (\Whit_q(G))^\heartsuit$, the cohomologies 
$H^i(\Gr^{\omega^\rho}_{G,x},\CF\sotimes \ICslm_{q^{-1},x})$ satisfy:

\smallskip

\noindent{\em(a)} $H^i=0$ for $i<0$. 

\smallskip

\noindent{\em(b)} $H^0$ identifies with $H^0(S^{-,\lambda},\CF|_{S^{-,\lambda}}[\langle \lambda,2\check\rho\rangle])$.

\smallskip

\noindent{\em(c)} $H^1$ injects into 
$H^1(S^{-,\lambda},\CF|_{S^{-,\lambda}}[\langle \lambda,2\check\rho\rangle])$.
\end{prop} 

\begin{proof}

Note that the support of $\ICslm_{q^{-1},x}|_{S^{-,\lambda}}$ is $\ol{S}^{-,\lambda}$, which is the union of
$S^{-,\lambda'}$ for $\lambda'\in \lambda+\Lambda^{\on{pos}}$. Using the Cousin decomposition, it suffices
to show that for any $\lambda'\in \lambda+(\Lambda^{\on{pos}}-0)$, the cohomologies 
$$H^i(S^{-,\lambda'},\CF|_{S^{-,\lambda'}}\sotimes \ICslm_{q^{-1},x}|_{S^{-,\lambda'}})$$
vanish in degrees $\leq 1+\langle \lambda,2\check\rho\rangle$ and that the cohomolgies 
\begin{equation} \label{e:cohomology lambda}
H^i(S^{-,\lambda},\CF|_{S^{-,\lambda}}\sotimes \ICslm_{q^{-1},x}|_{S^{-,\lambda}})
\end{equation}
vanish in degrees $<\langle \lambda,2\check\rho\rangle$. 

\medskip

Note that by \propref{p:sharper estimate}, the restriction
$\ICslm_{q^{-1},x}|_{S^{-,\lambda'}}$ is isomorphic to 
$$\omega_{S^{-,\lambda'}}[\langle \lambda',2\check\rho\rangle]\otimes \sK_{\lambda,\lambda'},$$
where $\sK_{\lambda,\lambda'}$ is an object of the category 
$\Shv_{\CG^\Lambda_{\lambda\cdot x}\otimes \CG^\Lambda_{-\lambda'\cdot x}}(\on{pt})$ that lives in cohomological degrees $>1$. 

\medskip

Hence, it suffices to show that the cohomologies 
$$H^i(S^{-,\lambda'},\CF|_{S^{-,\lambda'}})\in \Shv_{\CG^\Lambda_{\lambda'\cdot x}}(\on{pt})$$
vanish in degrees $<\langle \lambda',2\check\rho\rangle$. 

\medskip

Applying the Cousin decomposition with respect to the orbits $S^\mu$, it suffices to show that the cohomologies 
$$H^i(S^\mu\cap S^{-,\lambda'}, \CF|_{S^\mu\cap S^{-,\lambda'}})$$
vanish in degrees $<\langle \lambda',2\check\rho\rangle$. 

\medskip

By the definition of the t-structure on $\Whit_q(G)$, the restriction $\CF|_{S^\mu}$ is isomorphic to
$$\chi_N\otimes \omega_{S^\mu}[-\langle \mu,2\check\rho\rangle]\otimes \sK'_\mu,$$
where $\sK'_\mu$ is object of the category $\Shv_{\CG^\Lambda_{\mu\cdot x}}(\on{pt})$ 
that lives in cohomological degrees $\geq 0$. 

\medskip

The discrepancy of the identifications
$$\CG^\Lambda_{\mu\cdot x} \simeq \CG^\Lambda|_{S^\mu\cap S^{-,\lambda'}}\simeq \CG^\Lambda_{\lambda'\cdot x}$$
is given by a local system in 
$\Shv_{\CG^\Lambda_{\lambda'\cdot x} \otimes (\CG^\Lambda_{\mu\cdot x})^{-1}}(S^\mu\cap S^{-,\lambda'})$. 
Hence, up to tensoring by \emph{lisse} sheaves, the restriction $\CF|_{S^\mu\cap S^{-,\lambda'}}$ identifies with 
$$\omega_{S^\mu\cap S^{-,\lambda'}}[-\langle \mu,2\check\rho\rangle]\otimes \sK'_\mu,$$ and hence  
lives in perverse cohomological degrees
$$\geq \langle \mu,2\check\rho\rangle -\dim(S^\mu\cap S^{-,\lambda'}).$$

Hence, its cohomologies live in degrees
\begin{multline*}
\geq \langle \mu,2\check\rho\rangle -\dim(S^\mu\cap S^{-,\lambda'})-\dim(S^\mu\cap S^{-,\lambda'})= \\
=\langle \mu,2\check\rho\rangle -2\dim(S^\mu\cap S^{-,\lambda'})=\langle \mu,2\check\rho\rangle - \langle \mu-\lambda',2\check\rho\rangle=
\langle \lambda',2\check\rho\rangle,
\end{multline*}
as required. 

\medskip

The cohomologies \eqref{e:cohomology lambda} are analyzed similarly. 

\end{proof}

\sssec{}

We are now ready to prove the cohomological estimates stated in \secref{sss:stalk Phi vac}. 
The fact that for any $\CF\in (\Whit_q(G))^\heartsuit$, 
the object $\Phi(\CF)_{\lambda\cdot x}$ lives in cohomological degrees $\geq 0$ follows from
\propref{p:fiber of Phi top}(a). 

\medskip

Let us take $\CF=W^{\mu,!*}$ with $\mu$ restricted. Applying \propref{p:fiber of Phi top}(b), we need to
show that 
$$H^0(S^\mu\cap S^{-,\lambda},\chi_N|_{S^\mu\cap S^{-,\lambda}}\otimes \Psi_q\otimes \omega_{S^\mu\cap S^{-,\lambda}}
[\langle \lambda-\mu,2\check\rho\rangle])=0 \text{ if } \lambda\neq \mu$$
where $\Psi_q$ is the local system on $S^\mu \cap S^{-,\lambda}$ from \secref{sss:Psiq}. 
However, this is the statement of \thmref{t:restr}.

\sssec{}

Finally, let us take $\CF=W^0$. Applying \propref{p:fiber of Phi top}(c), for the proof of \thmref{t:GM from open}
it remains to show that the \emph{sup-bottom} cohomology 
$$H^1(S^0\cap S^{-,\lambda},\chi_N|_{S^0\cap S^{-.\lambda}}\otimes \Psi_q\otimes \omega_{S^0\cap S^{-,\lambda}}
[\langle \lambda,2\check\rho\rangle])$$
vanishes. 

\medskip

This vanishing was established in \cite[Theorem 1.1.5]{Lys} under the (mild) assumption that the order of each $q_i$ is 
large enough. In a subsequent publication we will give a proof in the general case (assuming $q_i\neq 1$). 

\section{Statement of the main theorem}  \label{s:statement}

In this section we will finally formulate our main theorem, which compares $\bHecke(\Whit_{q,x}(G))$
with a certain modification of the category $\Omega^{\on{small}}_q\on{-FactMod}$.

\ssec{Renormalization of $\Omega^{\on{small}}_q\on{-FactMod}$} \label{ss:renorm fact}

In this subsection we will introduce a \emph{renormalized version} of $\Omega^{\on{small}}_q\on{-FactMod}$, denoted 
$\Omega^{\on{small}}_q\on{-FactMod}^{\on{ren}}$. It is this
modified version that will end up being equivalent to $\bHecke(\Whit_{q,x}(G))$. 

\medskip

The nature of this modification is that we \emph{declare a larger class of objects as compact}; it mimics the procedure 
that produces $\IndCoh$ from $\QCoh$, see \cite[Sect. 1]{Ga3}. 

\sssec{} 

We define $\Omega^{\on{small}}_q\on{-FactMod}^{\on{ren}}$ to be the ind-completion of the non-cocomplete full subcategory of
$\Omega^{\on{small}}_q\on{-FactMod}$ that consists of objects that are finite extensions of (shifts of) the objects $\CM^{\mu,!*}_{\Conf}$.

\medskip

Ind-extension of the tautological inclusion defines a functor
\begin{equation} \label{e:un-ren}
\on{un-ren}: \Omega^{\on{small}}_q\on{-FactMod}^{\on{ren}}\to \Omega^{\on{small}}_q\on{-FactMod}.
\end{equation} 

\begin{rem}
Note that the functor $\on{un-ren}$ is \emph{not fully faithful}, even though it is such when restricted
to the full subcategory of compact objects.
\end{rem}

\sssec{}

Recall that the category $\Omega^{\on{small}}_q\on{-FactMod}$ is compactly generated by the objects $\CM^{\mu,!}_{\Conf}$
for $\mu\in \Lambda$. We will need the following result:

\begin{prop} \label{p:fin length}
The objects $\CM^{\mu,!}_{\Conf}$ and $\CM^{\mu,*}_{\Conf}$ have finite length.
\end{prop}

The proof will be given in \secref{sss:proof fin length}.

\begin{rem}
When $k=\BC$, an alternative proof of \propref{p:fin length}
will be given in \secref{sss:proof fin length quant}.
\end{rem}

\sssec{}

By \propref{p:fin length}, the category of compact objects in $\Omega^{\on{small}}_q\on{-FactMod}$ can be thought 
as a full subcategory in the category of compact objects in $\Omega^{\on{small}}_q\on{-FactMod}^{\on{ren}}$. 

\medskip

Therefore, the procedure of ind-extension defines a fully faithful functor
$$\on{ren}:\Omega^{\on{small}}_q\on{-FactMod}\to \Omega^{\on{small}}_q\on{-FactMod}^{\on{ren}},$$
which is easily seen to be the left adjoint of the functor $\on{un-ren}$ of \eqref{e:un-ren}. 

\sssec{}

As in \cite[Sect. 1.2]{Ga3} we have: 

\begin{prop} \label{p:prop of ren} 
The category $\Omega^{\on{small}}_q\on{-FactMod}^{\on{ren}}$ has a t-structure, uniquely characterized by the property that 
an object is connective if and only if its image under the functor $\on{un-ren}$ is connective. Moreover, the functor
$\on{un-ren}$ has the following properties with respect to this t-structure:

\medskip

\noindent{\em(a)} It is t-exact;

\medskip

\noindent{\em(b)} It induces an equivalence 
$$(\Omega^{\on{small}}_q\on{-FactMod}^{\on{ren}})^{\geq n}\to (\Omega^{\on{small}}_q\on{-FactMod})^{\geq n}$$ for any $n$;

\medskip

\noindent{\em(c)} It induces an equivalence of the hearts.

\end{prop}

\begin{cor}
The kernel of the functor $\on{un-ren}$ consists of infinitely coconnective objects, i.e., 
$$\underset{n}\bigcap\, (\Omega^{\on{small}}_q\on{-FactMod}^{\on{ren}})^{\leq -n}.$$
\end{cor}


\begin{rem} \label{r:two meanings}

Note that there are two possible meaning for the notation $\CM^{\mu,!}_{\Conf}$ and $\CM^{\mu,*}_{\Conf}$ and $\CM^{\mu,!*}_{\Conf}$.
On the one hand, we can view them as objects of the original category $\Omega^{\on{small}}_q\on{-FactMod}$. 

\medskip

On the other hand, we can view them (using \propref{p:fin length} for $\CM^{\mu,!}_{\Conf}$ and $\CM^{\mu,*}_{\Conf}$)
as (compact) objects in $\Omega^{\on{small}}_q\on{-FactMod}^{\on{ren}}$.

\medskip

We will not distinguish the two usages notationally, but will explicitly mention which one we mean if a confusion is likely to occur.

\medskip

We note that the functor $\on{un-ren}$ sends $\CM^{\mu,!}_{\Conf}$ (resp., $\CM^{\mu,*}_{\Conf}$,  $\CM^{\mu,!*}_{\Conf}$), viewed 
as an object of $\Omega^{\on{small}}_q\on{-FactMod}^{\on{ren}}$, to $\CM^{\mu,!}_{\Conf}$ (resp., $\CM^{\mu,*}_{\Conf}$,  $\CM^{\mu,!*}_{\Conf}$),
viewed as an object of $\Omega^{\on{small}}_q\on{-FactMod}$.

\medskip

The functor $\on{ren}$ sends $\CM^{\mu,!}_{\Conf}$, viewed 
as an object of $\Omega^{\on{small}}_q\on{-FactMod}$, to $\CM^{\mu,!}_{\Conf}$,
viewed as an object of $\Omega^{\on{small}}_q\on{-FactMod}^{\on{ren}}$.
 
\medskip

However, the functor $\on{ren}$ \emph{does not} send $\CM^{\mu,*}_{\Conf}$ (resp., $\CM^{\mu,!*}_{\Conf}$), viewed 
as an object of $\Omega^{\on{small}}_q\on{-FactMod}$, to $\CM^{\mu,!}_{\Conf}$,
viewed as an object of $\Omega^{\on{small}}_q\on{-FactMod}^{\on{ren}}$. This is due to the fact that $\CM^{\mu,*}_{\Conf}$ and $\CM^{\mu,!*}_{\Conf}$
are \emph{not} compact in $\Omega^{\on{small}}_q\on{-FactMod}$. In fact, the images of these objects under $\on{ren}$ do not lie in 
the heart of $\Omega^{\on{small}}_q\on{-FactMod}^{\on{ren}}$.  

\end{rem} 

\sssec{}

Recall (see \secref{sss:properties of t}) what it means for a t-structure on a DG category to be Artinian. 

\medskip

We obtain that the t-structure on $\Omega^{\on{small}}_q\on{-FactMod}^{\on{ren}}$ is Artinian. This is the main point of difference 
between $\Omega^{\on{small}}_q\on{-FactMod}$ and $\Omega^{\on{small}}_q\on{-FactMod}^{\on{ren}}$. 

\begin{rem}
Since introducing the configuration space may seem a bit artificial, let us remark that one 
define the renormalization $\Omega^{\on{small}}_q\on{-FactMod}^{\on{ren}}$ purely in terms 
of the factorization algebra $\Omega_q^{\Whit_{!*}}\in \Shv_{\CG^T}(\Gr^\omega_{T,\Ran})$. 

\medskip

Namely, recall (see Remark \ref{r:unitality Gr}) that the category $\Omega^{\on{small}}_q\on{-FactMod}$ is equivalent to 
$$\Omega_q^{\Whit_{!*}}\on{-FactMod}(\Shv_{\CG^T}((\Gr^{\omega^\rho}_{T,\Ran})^{\on{non-pos}}_{\infty\cdot x}))_{\on{untl}}.$$

We introduce the t-structure on this category by declaring an object co-connective if its !-restriction to $\{x\}\subset \Ran_x$,
viewed as an object of 
$$\Shv_{\CG^T}(\Gr^{\omega^\rho}_{T,x}),$$ 
is coconnective with respect to the following t-structure:

\medskip

For $\mu\in \Lambda$, on the connected component of $\Gr^{\omega^\rho}_{T,x}\simeq \Lambda$, we shift
the standard t-structure by $[-\langle \mu, 2\check\rho \rangle]$. 

\medskip

Having defined the t-structure, we can define the sought-for renormalization: 

\medskip

The renormalized category is the ind-completion
of objects that are finite extensions of (shifts of) irreducible objects in
$$\left(\Omega_q^{\Whit_{!*}}\on{-FactMod}(\Shv_{\CG^T}((\Gr^{\omega^\rho}_{T,\Ran})^{\on{non-pos}}_{\infty\cdot x}))_{\on{untl}}\right)^\heartsuit.$$

\end{rem} 

\ssec{Statement of the theorem}

In this subsection we define a renormalized version of the functor $\Phi^{\bHecke}_{\on{Fact}}$, which will allow us to state
our main result, \thmref{t:main}.

\sssec{}

Note that \corref{c:irred pres Hecke} implies that the restriction of the functor $\Phi^{\bHecke}_{\on{Fact}}$ to
$$\bHecke(\Whit_{q,x}(G))^c\subset \Whit_{q,x}(G)$$
takes values in 
$$(\Omega^{\on{small}}_q\on{-FactMod}^{\on{ren}})^c\subset \Omega^{\on{small}}_q\on{-FactMod}.$$

Hence, ind-extension defines a functor
$$\Phi^{\bHecke,\on{ren}}_{\on{Fact}}:\bHecke(\Whit_{q,x}(G))\to \Omega^{\on{small}}_q\on{-FactMod}^{\on{ren}}$$
so that 
$$\on{un-ren}\circ \Phi^{\bHecke,\on{ren}}_{\on{Fact}}\simeq \Phi^{\bHecke}_{\on{Fact}}.$$

\sssec{}

By construction, the functor $\Phi^{\bHecke,\on{ren}}_{\on{Fact}}$ preserves compactness. This is the main difference between
$\Phi^{\bHecke,\on{ren}}_{\on{Fact}}$ and $\Phi^{\bHecke}_{\on{Fact}}$. 

\sssec{}

By construction, we have:
\begin{equation} \label{e:pres irred ren}
\Phi^{\bHecke,\on{ren}}_{\on{Fact}}(\bCM^{\mu,!*}_{\Whit})\simeq \CM^{\mu,!*}_{\on{Fact}},
\end{equation}
as objects of $\Omega^{\on{small}}_q\on{-FactMod}^{\on{ren}}$.

\medskip

Since the t-structure on $\bHecke(\Whit_{q,x}(G))$ is Artinian, we obtain that the functor $\Phi^{\bHecke,\on{ren}}_{\on{Fact}}$
is t-exact.

\sssec{}

The main result of this work is the following:

\begin{thm} \label{t:main}
The functor $\Phi^{\bHecke,\on{ren}}_{\on{Fact}}$ is an equivalence. 
\end{thm} 

\ssec{Outline of the proof}

In this subsection we outline the main steps involved in the proof of \thmref{t:main}. 

\sssec{}

First, by Propositions \ref{p:isogenies trick} and \ref{p:isogen Gr}, we can assume that the derived
group of $H$ is simply-connected. 

\sssec{}  \label{sss:strategy}

We will introduce objects
$$\mu \rightsquigarrow \bCM{}_{\Whit}^{\mu,!}, \bCM{}_{\Whit}^{\mu,*}\in \bHecke(\Whit_q(G))$$
with the following properties:

\begin{itemize}

\item(i)
$$\CHom_{\bHecke(\Whit_q(G))}(\bCM{}^{\mu',!}_{\Whit},\bCM{}^{\mu,*}_{\Whit})=
\begin{cases}
&\sfe \text{ if } \mu'=\mu \\
&0 \text{ otherwise}.
\end{cases}$$

\item(ii)
The unique (up to a scalar) map $\bCM{}^{\mu,!}\to \bCM{}^{\mu,*}$
factors through the irreducible object $\bCM{}_{\Whit}^{\mu,!*}$, and the fiber (resp., cofiber) of the map
from $\bCM{}^{\mu,!}$ to it (resp., from it to $\bCM{}_{\Whit}^{\mu,*}$) has a finite filtration with
subquotients being irreducibles $\bCM{}_{\Whit}^{\mu',!*}$ with $\mu'<\mu$.

\bigskip 

\item(iii)
$$\Phi^{\bHecke}_{\on{Fact}}(\bCM{}^{\mu,!}_{\Whit})\simeq \bCM{}^{\mu,!}_{\Conf} \text{ and }
\Phi^{\bHecke}_{\on{Fact}}(\bCM{}^{\mu,*}_{\Whit})\simeq \bCM{}^{\mu,*}_{\Conf}$$
as objects of $\Omega_q^{\on{small}}\on{-FactMod}$. 

\end{itemize} 

\medskip

Let us show how having a collection of objects with the above properties implies \thmref{t:main}.

\sssec{} \label{sss:proof fin length}

First, property (ii) above implies that the objects $\bCM{}^{\mu,!}$ and $\bCM{}^{\mu,*}$
are compact.

\medskip

Let us note that this, combined with property (iii) and \eqref{e:pres irred ren}, 
implies the assertion of \propref{p:fin length}. 

\sssec{}

Next we note that the isomorphisms of property (iii) above imply that we also have
$$\Phi^{\bHecke,\on{ren}}_{\on{Fact}}(\bCM{}^{\mu,!}_{\Whit})\simeq \bCM{}^{\mu,!}_{\Conf} \text{ and }
\Phi^{\bHecke,\on{ren}}_{\on{Fact}}(\bCM{}^{\mu,*}_{\Whit})\simeq \bCM{}^{\mu,*}_{\Conf}$$
as objects of $\Omega_q^{\on{small}}\on{-FactMod}^{\on{ren}}$, see Remark \ref{r:two meanings}. 

\begin{proof}[Proof of \thmref{t:main}]

In view of \eqref{e:pres irred ren}, it only remains to show that the functor $\Phi^{\bHecke,\on{ren}}_{\on{Fact}}$ 
is fully faithful. Since the objects $\bCM{}^{\mu,!*}$ compactly 
generate the category $\bHecke(\Whit_q(G))$, it suffices to show that the functor $\Phi^{\bHecke,\on{ren}}_{\on{Fact}}$ 
induces an isomorphism
$$\CHom_{\bHecke(\Whit_q(G))}(\bCM{}^{\mu',!*}_{\Whit},\bCM{}^{\mu,!*}_{\Whit})\to 
\CHom_{ \Omega_q^{\on{small}}\on{-FactMod}^{\on{ren}}}
(\Phi^{\bHecke,\on{ren}}_{\on{Fact}}(\bCM{}^{\mu',!}_{\Whit}),
\Phi^{\bHecke,\on{ren}}_{\on{Fact}}(\bCM{}^{\mu,*}_{\Whit}))$$
for all $\mu,\mu'\in \Lambda$.  

\medskip

However, by a standard argument with the 5-lemma, it follows from (ii) above that it is sufficient to show that 
the functor $\Phi^{\bHecke,\on{ren}}_{\on{Fact}}$ 
induces an isomorphism
$$\CHom_{\bHecke(\Whit_q(G))}(\bCM{}^{\mu',!}_{\Whit},\bCM{}^{\mu,*}_{\Whit})\to 
\CHom_{ \Omega_q^{\on{small}}\on{-FactMod}^{\on{ren}}}
(\Phi^{\bHecke,\on{ren}}_{\on{Fact}}(\bCM{}^{\mu',!}_{\Whit}),
\Phi^{\bHecke,\on{ren}}_{\on{Fact}}(\bCM{}^{\mu,*}_{\Whit}))$$
for all $\mu,\mu'\in \Lambda$.  

\medskip

The latter is equivalent to showing that the functor $\Phi^{\bHecke}_{\on{Fact}}$ 
induces an isomorphism
$$\CHom_{\bHecke(\Whit_q(G))}(\bCM{}^{\mu',!}_{\Whit},\bCM{}^{\mu,*}_{\Whit})\to 
\CHom_{ \Omega_q^{\on{small}}\on{-FactMod}}
(\Phi^{\bHecke}_{\on{Fact}}(\bCM{}^{\mu',!}_{\Whit}),
\Phi^{\bHecke}_{\on{Fact}}(\bCM{}^{\mu,*}_{\Whit}))$$
for all $\mu,\mu'\in \Lambda$.  

\medskip

In view of properties (i) and (iii), and \eqref{e:Ext standard fact}, both sides vanish for $\mu\neq \mu'$. So, it remains
to consider the case of $\mu=\mu'$. We have to show that the resulting map
$$\sfe\simeq \CHom_{\bHecke(\Whit_q(G))}(\bCM{}^{\mu',!}_{\Whit},\bCM{}^{\mu,*}_{\Whit})\to 
\CHom_{ \Omega_q^{\on{small}}\on{-FactMod}^{\on{ren}}}
(\bCM{}^{\mu',!}_{\on{Fact}},\bCM{}^{\mu,*}_{\on{Fact}})\simeq \sfe$$
is non-zero. 

\medskip

We argue by contradiction. If this map where zero, by property (ii) and the t-exactness of 
$\Phi^{\bHecke,\on{ren}}_{\on{Fact}}$, this would imply that $\Phi^{\bHecke,\on{ren}}_{\on{Fact}}(\bCM{}^{\mu,!*}_{\Whit})=0$,
which would contradict \eqref{e:pres irred ren}. 

\end{proof}

\newpage 

\centerline{\bf Part VII: Zastava spaces and global interpretation of the Jacquet functor} 

\bigskip

The goal of this part is to prove \thmref{t:duality pres}, and its Hecke extension, \thmref{t:duality pres Hecke}.
A salient feature of the proof is that it uses global methods, i.e., we will be working with a complete curve $X$.

\section{Zastava spaces}  \label{s:Zastava}

In this section we will rewrite the Jacquet functor $\Phi$ in terms of the global version of the 
Whittaker category, $\Whit_{q,x}(G)$. The link between the local geometry (which involves 
$\fL(N)$-orbits on $\Gr_G$) and the global one (which involves $(\ol\Bun^{\omega^{\rho}}_N)_{\infty\cdot x}$)
is provided by the \emph{Zastava spaces}.

\ssec{Zastava spaces: local definition}

In this subsection we will interpret the spaces involved in the construction of the functor 
$\Phi$ as \emph{Zastava spaces}. 

\sssec{}

The (completed) Zastava space, denoted $\bCZ$ is defined to be
$$\ol{S}^0_{\Conf}\cap \ol{S}^{-,\Conf}_{\Conf},$$
where the intersection is taking place in $\Gr^{\omega^\rho}_{G,\Conf}$.

\medskip 

Let $\sv_{\Conf}$ denote the projection $\bCZ\to \Conf$. The properness of the affine Grassmannian
implies that the map $\sv_{\Conf}$ is ind-proper. However, we will soon see that $\bCZ$ is actually
a \emph{scheme}, so that the map $\sv_{\Conf}$ is actually a proper map of schemes. 

\sssec{}

Let $\CZ\subset \bCZ$
be the open subfunctor equal to
$$\ol{S}^0_{\Conf}\cap S^{-,\Conf}_{\Conf}.$$

\medskip

Let $\boCZ\subset \bCZ$ be the open subfunctor equal to
$$S^0_{\Conf}\cap \ol{S}^{-,\Conf}_{\Conf}.$$

\medskip

Finally, let 
$$\oCZ:=\CZ\cap \boCZ\subset \bCZ.$$

\sssec{}

We introduce the polar version of the Zastava spaces as follows:
$$\bCZ_{\infty\cdot x,\infty\cdot x}:=(\ol{S}^0_{\Conf_{\infty\cdot x}})_{\infty\cdot x}\cap 
(\ol{S}^{-,\Conf_{\infty\cdot x}}_{\Conf_{\infty\cdot x}})_{\infty\cdot x},$$
where the intersection is taking place in $\Gr^{\omega^\rho}_{G,\Conf_{\infty\cdot x}}$.

\medskip 

Let $\sv_{\Conf_{\infty\cdot x}}$ denote the projection $\bCZ_{\infty\cdot x,\infty\cdot x}\to \Conf_{\infty\cdot x}$.
The properness of the affine Grassmannian implies that the map $\sv_{\Conf_{\infty\cdot x}}$ is ind-proper. 

\sssec{}

We introduce the closed subspace $\bCZ_{\infty\cdot x}\subset \bCZ_{\infty\cdot x,\infty\cdot x}$
to be 
$$\bCZ_{\infty\cdot x}:=(\ol{S}^0_{\Conf_{\infty\cdot x}})_{\infty\cdot x}\cap 
(\ol{S}^{-,\Conf_{\infty\cdot x}}_{\Conf_{\infty\cdot x}}).$$

\begin{rem}
Note that the space $\bCZ_{\infty\cdot x}$ was used in the definition of the functor $\Phi$, see \secref{sss:Phi expl}. 
Indeed, the support of
objects of the form
$$\on{sprd}_{\Conf_{\infty\cdot x}}(\CF)\sotimes \ICsm_{q^{-1},\Conf_{\infty\cdot x}}$$
is contained in $\bCZ_{\infty\cdot x}$.
\end{rem}

\ssec{Global interpretation of the Zastava spaces}   \label{ss:Zastava} 

In this subsection we let the curve $X$ be complete. We will reinterpret the various versions of the Zastava 
space in terms of moduli spaces of bundles on $X$.

\sssec{}

Consider the algebraic stacks
$$\ol\Bun^{\omega^{\rho}}_N\overset{\ol\sfp}\longrightarrow \Bun_G \overset{\ol\sfp^-}\longleftarrow \BunBbm.$$

Consider the open substack
$$(\ol\Bun^{\omega^{\rho}}_N\underset{\Bun_G}\times \BunBbm)^{\on{gen}}\subset 
\ol\Bun^{\omega^{\rho}}_N\underset{\Bun_G}\times \BunBbm$$
corresponding to the condition that the (generic) $N$-reduction and the (generic) $B^-$-reduction of a given
$G$-bundles are \emph{transversal at the generic point of $X$}.

\medskip

The maps
$$\pi:\ol{S}^0_{\Conf}\to \BunNb$$
and
$$\pi^-_{\Conf}:\ol{S}^{-,\Conf}_{\Conf}\to \BunBbm$$
(see \secref{sss:two contexts}) define a map
$$\bCZ\to \ol\Bun^{\omega^{\rho}}_N\underset{\Bun_G}\times \BunBbm,$$
and it is easy to see that its image lands in $(\BunNb\underset{\Bun_G}\times \BunBbm)^{\on{gen}}$.

\medskip

We claim:

\begin{prop} \label{p:Zastava via global}
The map 
$$\bCZ \to (\BunNb\underset{\Bun_G}\times \BunBbm)^{\on{gen}}$$
is an isomorphism. Under this identification, the subspaces
$$\CZ\subset \bCZ\supset \boCZ$$ correspond to
$$(\ol\Bun^{\omega^{\rho}}_N\underset{\Bun_G}\times \Bun_{B^-})^{\on{gen}} \subset 
(\ol\Bun^{\omega^{\rho}}_N\underset{\Bun_G}\times \BunBbm)^{\on{gen}}
\supset (\Bun^{\omega^{\rho}}_N\underset{\Bun_G}\times \BunBbm)^{\on{gen}},$$
respectively. 
\end{prop}

\begin{proof}

Proceeding as in \cite[Sect. 7]{Sch}, we can assume that the derived group of $G$ is simply-connected. 
In this case, we will explicitly construct an inverse map. 

\medskip

An $S$-point of $(\ol\Bun^{\omega^{\rho}}_N\underset{\Bun_G}\times \BunBbm)^{\on{gen}}$
consists of a $G$-bundle $\CP_G$, equipped with a Pl\"ucker data for $N$ and $B^-$ 
$$(\CP_G,\{\kappa^{\clambda},\clambda\in \cLambda^+\},\{\kappa^{-,\clambda},\clambda\in \cLambda^+\}),$$
where:

\smallskip

\begin{itemize}

\item $(\CP_G,\{\kappa^{\clambda},\clambda\in \cLambda^+\})$ is as in the definition of $\BunNbox$, i.e., this is a data of a
$G$-bundle on $X$, equipped with a generalized reduction (a.k.a. Drinfeld structure) to $N^{\omega^\rho}$;

\smallskip

\item $\{\kappa^{-,\clambda},\clambda\in \cLambda^+\}$ is a generalized reduction of $\CP_G$ to $B^-$, i.e., these are maps
$$'\CV^{\clambda}_{\CP_G}\to \clambda(\CP_T)$$
for \emph{some} $T$-bundle $\CP_T$ on $X$ that satisfy the Plucker relations (here $'\CV^{\clambda}$ denotes the dual Weyl module
with highest weight $\clambda$);

\smallskip

\item The above generalized reductions of $\CP_G$ to $N^{\omega^\rho}$ and $B^-$ are mutually transversal away from a closed
subset $S\times X$ that is \emph{finite over} $S$.  

\end{itemize}

The transversality condition means that the composite maps
\begin{equation} \label{e:maps of line bundles}
(\omega^{\frac{1}{2}})^{\langle \clambda,2\rho\rangle}\to \CV^{\clambda}_{\CP_G}
\to {}'\CV^{\clambda}_{\CP_G}\to \clambda(\CP_T), \clambda\in \cLambda^+
\end{equation}
are isomorphisms away from a closed subset $S\times X$ that is \emph{finite over} $S$.  

\medskip

The assumption that the derived group of $G$ is simply connected implies (see Remark \ref{r:config as divisors}) 
that the system of maps \eqref{e:maps of line bundles} is equivalent to a datum of an $S$-point $D$ of $\Conf$. 

\medskip

By construction, over the open subset
$$S\times X-D \subset S\times X,$$
the maps $\kappa^{\clambda}$ and the maps $\kappa^{-,\clambda}$ are \emph{bundle maps}, hence 
the (generalized) $N$-reduction and the (generalized) $B^-$-reduction of $\CP_G|_{S\times X}$ are \emph{genuine}
and mutually transversal. Hence, over this open subset these two reductions uniquely determine a trivialization of
$\CP_G$. 

\medskip

I.e., we obtain an $S$-point of $\Gr^{\omega^\rho}_{G,\Conf}$ that projects to the $S$-point of $\Conf$, given by
$D$. By construction, the above $S$-point of $\Gr^{\omega^\rho}_{G,\Conf}$ actually belongs to the subfunctor
$\bCZ$, as desired.

\end{proof} 

As an immediate corollary of \propref{p:Zastava via global}, we obtain:

\begin{cor}
The prestack $\bCZ$ is a scheme.
\end{cor}

\begin{proof}

Indeed, by construction, $\bCZ$ is an ind-scheme, but \propref{p:Zastava via global} implies
that it is also an algebraic stack.

\end{proof} 

\sssec{}

We will also need a variant of the above picture with pole points. Consider the maps
$$(\ol\Bun^{\omega^{\rho}}_N)_{\infty\cdot x}\overset{\ol\sfp}\longrightarrow \Bun_G \overset{\ol\sfp^-}\longleftarrow 
(\BunBbm)_{\infty\cdot x}.$$

We consider the corresponding open subfunctors
$$\left((\ol\Bun^{\omega^{\rho}}_N)_{\infty\cdot x}\underset{\Bun_G}\times (\BunBbm)_{\infty\cdot x}\right)^{\on{gen}}
\subset (\ol\Bun^{\omega^{\rho}}_N)_{\infty\cdot x}\underset{\Bun_G}\times (\BunBbm)_{\infty\cdot x}$$
and 
$$\left((\ol\Bun^{\omega^{\rho}}_N)_{\infty\cdot x}\underset{\Bun_G}\times \BunBbm\right)^{\on{gen}}
\subset (\ol\Bun^{\omega^{\rho}}_N)_{\infty\cdot x}\underset{\Bun_G}\times \BunBbm.$$

We have the naturally defined maps
\begin{equation} \label{e:from local to global Zastava poles poles}
\bCZ_{\infty\cdot x,\infty\cdot x}\to 
\left((\ol\Bun^{\omega^{\rho}}_N)_{\infty\cdot x}\underset{\Bun_G}\times (\BunBbm)_{\infty\cdot x}\right)^{\on{gen}}
\end{equation}
and 
\begin{equation} \label{e:from local to global Zastava poles}
\bCZ_{\infty\cdot x}\to 
\left((\ol\Bun^{\omega^{\rho}}_N)_{\infty\cdot x}\underset{\Bun_G}\times \BunBbm\right)^{\on{gen}}.
\end{equation}

As in \propref{p:Zastava via global}, one shows:

\begin{prop} \label{p:Zastava via global poles}
The maps \eqref{e:from local to global Zastava poles poles} and \eqref{e:from local to global Zastava poles}
are isomorphisms.
\end{prop} 

\sssec{}

Let $\CG^G$ denote the gerbe on $\bCZ$ (resp., $\bCZ_{\infty\cdot x}$, $\bCZ_{\infty\cdot x,\infty\cdot x}$) obtained by
restriction from the same-named gerbe on $\Gr^{\omega^\rho}_{G,\Conf}$. Let $\CG^\Lambda$ denote the gerbe on 
$\bCZ$ (resp., $\bCZ_{\infty\cdot x}$, $\bCZ_{\infty\cdot x,\infty\cdot x}$) obtained as a pullback from the same-named gerbe
on $\Conf$ (resp., $\Conf_{\infty\cdot x}$). Denote
$$\CG^{G,T,\on{ratio}}:=\CG^G\otimes (\CG^T)^{-1}.$$

\medskip

We obtain that under the identification of \propref{p:Zastava via global} 
(resp., \eqref{e:from local to global Zastava poles poles} and \eqref{e:from local to global Zastava poles}),
$\CG^G$ corresponds to the pullback of the same-named gerbe on $\ol\Bun^{\omega^{\rho}}_N$ 
(resp., $(\ol\Bun^{\omega^{\rho}}_N)_{\infty\cdot x}$),
and $\CG^{G,T,\on{ratio}}$ corresponds to the pullback of the same-named gerbe on $\BunBbm$ 
(resp., $(\BunBbm)_{\infty\cdot x}$).

\ssec{Global interpretation of the functor $\Phi$}

In this subsection we will use \propref{p:Zastava via global poles} to give a global interpretation of the
functor $\Phi$.

\sssec{}

Consider the Cartesian square
$$
\CD
\ol\Bun^{\omega^{\rho}}_N\underset{\Bun_G}\times \BunBbm    @>{'\ol\sfp}>>  \BunBbm  \\
@V{'\ol\sfp^-}VV    @VV{\ol\sfp^-}V    \\
\ol\Bun^{\omega^{\rho}}_N   @>{\ol\sfp}>>   \Bun_G.
\endCD
$$

By a slight abuse of notation let us denote by the same symbols $'\ol\sfp^-$ and $'\ol\sfp$ the resulting maps
$$\ol\Bun^{\omega^{\rho}}_N  \overset{'\ol\sfp^-}\longleftarrow \bCZ \overset{'\ol\sfp}\longrightarrow \BunBbm$$
arising from the identification of \propref{p:Zastava via global}.

\medskip

We will use a similar notation also for the spaces $\bCZ_{\infty\cdot x,\infty\cdot x}$ and $\bCZ_{\infty\cdot x}$. 

\sssec{}   \label{sss:Phi glob}

Define the functor
$$\Phi_{\on{glob}}:\Whit_{q,\on{glob}}(G)\to  \Shv_{\CG^\Lambda}(\Conf_{\infty\cdot x})$$
to be the composition 
\begin{multline}  \label{e:global composition}
\Whit_{q,\on{glob}}(G) \hookrightarrow \Shv_{\CG^G}((\ol\Bun^{\omega^{\rho}}_N)_{\infty\cdot x})
\overset{({}'\ol\sfp^-)^!}\longrightarrow \Shv_{\CG^G}(\bCZ_{\infty\cdot x}) 
\overset{-\sotimes ({}'\ol\sfp)^!({}_{\Bun_T}\!\ICsm_{q^{-1},\on{glob}})[\dim(\Bun_G)]}\longrightarrow \\
\to \Shv_{\CG^\Lambda}(\bCZ_{\infty\cdot x}) \overset{(\sv_{\Conf})_*} \longrightarrow 
\Shv_{\CG^\Lambda}(\Conf_{\infty\cdot x}).
\end{multline} 

\sssec{}

Recall the equivalence
$$\pi_x^!:\Whit_{q,\on{glob}}(G)\to \Whit_q(G).$$

We claim:

\begin{prop} \label{p:functor Phi glob}
The functor 
$$\Phi\circ \pi_x^![d_g]: \Whit_{q,\on{glob}}(G)\to \Shv_{\CG^\Lambda}(\Conf_{\infty\cdot x})$$
identifies canonically with $\Phi_{\on{glob}}$. 
\end{prop} 

This proposition follows by applying $\sv_{\Conf_x}$ from the following more precise result:

\begin{prop} \label{p:functor Phi glob upstairs}
The composite functor
\begin{multline}  \label{e:global composition upstairs}
\Whit_{q,\on{glob}}(G) \hookrightarrow \Shv_{\CG^G}((\ol\Bun^{\omega^{\rho}}_N)_{\infty\cdot x})
\overset{({}'\ol\sfp^-)^!}\longrightarrow \Shv_{\CG^G}(\bCZ_{\infty\cdot x}) 
\overset{-\sotimes ({}'\ol\sfp)^!({}_{\Bun_T}\!\ICsm_{q^{-1},\on{glob}})[\dim(\Bun_G)]}\longrightarrow \\
\to \Shv_{\CG^\Lambda}(\bCZ_{\infty\cdot x})
\end{multline}
identifies canonically with the composition of $\pi_x^![d_g]$ and the functor 
\begin{equation}  \label{e:Phi first step}
\CF\mapsto \on{sprd}_{\Conf_{\infty\cdot x}}(\CF)\sotimes \ICsm_{q^{-1},\Conf_{\infty\cdot x}}[\deg],\quad 
\Whit_{q,x}(G)\to \Shv_{\CG^\Lambda}(\bCZ_{\infty\cdot x}).
\end{equation}
\end{prop}


\begin{proof}

First, we note that by \corref{c:ICs glob sym} (or, rather, its version with $N$ replaced by $N^-$), the object
$$\ICsm_{q^{-1},\Conf_{\infty\cdot x}}|_{\bCZ_{\infty\cdot x}}[\deg+d_g]
\in  \Shv_{(\CG^G)^{-1}\otimes \CG^\Lambda}(\bCZ_{\infty\cdot x})$$
identifies with
$$({}'\ol\sfp)^!({}_{\Bun_T}\!\ICsm_{q^{-1},\on{glob}})[\dim(\Bun_G)].$$

\medskip

Hence, it remains to show that the composition 
$$\Whit_{q,\on{glob}}(G) \overset{\pi_x^!}\longrightarrow \Whit_q(G) \overset{\on{sprd}_{\Conf_{\infty\cdot x}}}\longrightarrow
\Shv_{\CG^G}((\ol{S}^0_{\Conf_{\infty\cdot x}})_{\infty\cdot x})\overset{!\on{-restriction}}\longrightarrow
\Shv_{\CG^G}(\bCZ_{\infty\cdot x})$$
identifies with the functor
$$\Whit_{q,\on{glob}}(G) \hookrightarrow \Shv_{\CG^G}((\ol\Bun^{\omega^{\rho}}_N)_{\infty\cdot x})
\overset{({}'\ol\sfp^-)^!}\longrightarrow \Shv_{\CG^G}(\bCZ_{\infty\cdot x}).$$

In fact, we claim that the functor
$$\Whit_{q,\on{glob}}(G) \overset{\pi_x^!}\longrightarrow \Whit_q(G) \overset{\on{sprd}_{\Conf_{\infty\cdot x}}}\longrightarrow
\Shv_{\CG^G}((\ol{S}^0_{\Conf_{\infty\cdot x}})_{\infty\cdot x})$$
identifies canonically with
$$\Whit_{q,\on{glob}}(G) \hookrightarrow \Shv_{\CG^G}((\ol\Bun^{\omega^{\rho}}_N)_{\infty\cdot x})
\overset{\pi_{\Conf_{\infty\cdot x}}^!}\longrightarrow \Shv_{\CG^G}((\ol{S}^0_{\Conf_{\infty\cdot x}})_{\infty\cdot x}).$$

For this, it suffices to show that the composition
$$\Whit_{q,\on{glob}}(G) \overset{\pi_x^!}\longrightarrow \Whit_q(G) \overset{\on{sprd}_{\Ran_x}} \longrightarrow
\Shv_{\CG^G}((\ol{S}^0_{\Ran_x})_{\infty\cdot x})$$
identifies canonically with 
$$\Whit_{q,\on{glob}}(G) \hookrightarrow \Shv_{\CG^G}((\ol\Bun^{\omega^{\rho}}_N)_{\infty\cdot x})
\overset{\pi_{\Ran_x}^!}\longrightarrow \Shv_{\CG^G}((\ol{S}^0_{\Conf_{\Ran_x}})_{\infty\cdot x}).$$

The latter follows from \thmref{t:restr to unit} by composing both sides with the functor
$$\on{unit}^!: \Shv_{\CG^G}((\ol{S}^0_{\Conf_{\Ran_x}})_{\infty\cdot x}) \to
 \Shv_{\CG^G}(\Ran_x\times \Gr^{\omega^\rho}_{G,x}).$$

Indeed, both sides are given by the forgetful functor $\Whit_{q,\on{glob}}(G) \hookrightarrow \Shv_{\CG^G}((\ol\Bun^{\omega^{\rho}}_N)_{\infty\cdot x})$,
followed by pullback along
$$\Ran_x\times \Gr^{\omega^\rho}_{G,x} \to \Gr^{\omega^\rho}_{G,x}  \overset{\pi_x}\longrightarrow 
(\ol\Bun^{\omega^{\rho}}_N)_{\infty\cdot x}.$$

\end{proof} 

\sssec{}  \label{sss:proof of finite length pres}

Let us observe that \propref{p:functor Phi glob} immediately implies the first assertion \thmref{t:duality pres}, namely that 
the functor $\Phi$ sends compact objects in $\Whit_{q,x}(G)$ to locally compact ones in $\Shv_{\CG^\Lambda}(\Conf_{\infty\cdot x})$:

\medskip

Indeed, over each connected component $\Conf^\lambda_{\infty\cdot x}$ of $\Conf_{\infty\cdot x}$, 
all the functors in the composition \eqref{e:global composition} preserve compactness.

\ssec{A sharpened version of \thmref{t:duality pres}}  \label{ss:proof of Verdier}

In this subsection we will formulate a sharpened version of \thmref{t:duality pres}, and deduce from it the original
statement.

\sssec{}   \label{sss:proof of duality pres}

Let
$$\Shv_{\CG^\Lambda}(\CZ_{\infty\cdot x})^{\on{loc.c}} \subset \Shv_{\CG^\Lambda}(\CZ_{\infty\cdot x}).$$
be the subcategory of \secref{sss:locally compact}; it consists of objects whose restriction to the preimage of every individual 
$\Conf^\lambda_{\infty\cdot x}\subset \Conf_{\infty\cdot x}$ is compact. 

\medskip

Since $\CZ_{\infty\cdot x}$ is an ind-scheme, we have a well-defined Verdier duality functor
\begin{equation} \label{e:Verdier on Zastava}
\BD^{\on{Verdier}}:(\Shv_{\CG^\Lambda}(\CZ_{\infty\cdot x})^{\on{loc.c}})^{\on{op}} \to 
\Shv_{(\CG^\Lambda)^{-1}}(\CZ_{\infty\cdot x})^{\on{loc.c}}.
\end{equation} 

\medskip

We are going to prove the following result:

\begin{thm} \label{t:Verdier upstairs}
The following diagram of functors commutes
$$
\CD
(\Whit_q(G)^c)^{\on{op}}     @>{\BD^{\on{Verdier}}}>>    \Whit_{q^{-1}}(G)^c  \\
@V{\text{\eqref{e:Phi first step}}}VV   @V{\text{\eqref{e:Phi first step}}}VV \\
(\Shv_{\CG^\Lambda}(\CZ_{\infty\cdot x})^{\on{loc.c}})^{\on{op}} @>{\BD^{\on{Verdier}}}>>  
\Shv_{(\CG^\Lambda)^{-1}}(\CZ_{\infty\cdot x})^{\on{loc.c}}
\endCD
$$
commutes. 
\end{thm}

Note that combined with the properness of the map $\sv_{\Conf_x}$, the assertion of \thmref{t:Verdier upstairs}
implies that of \thmref{t:duality pres}. 

\sssec{}

Note that statement of \thmref{t:Verdier upstairs} is \emph{local} (i.e., does not appeal to a global curve $X$).
However, we will use global methods to prove it. Namely, we will use the interpretation of the functor \eqref{e:Phi first step}
as \eqref{e:global composition upstairs} in order to prove it. 

\medskip

Note that, according to Remark \ref{r:local vs global duality}, the functor $\pi_x^![d_g]$
intertwines the duality \eqref{e:self duality for Whit comp} with the Verdier duality functor
$$(\Whit_{q,\on{glob}}(G)^c)^{\on{op}}\to \Whit_{q^{-1},\on{glob}}(G)^c$$
induced by Verdier duality
$$\BD^{\on{Verdier}}:(\Shv_{\CG^G}((\ol\Bun^{\omega^{\rho}}_N)_{\infty\cdot x})^{\on{loc.c}})^{\on{op}}\to
\Shv_{(\CG^G)^{-1}}((\ol\Bun^{\omega^{\rho}}_N)_{\infty\cdot x})^{\on{loc.c}}.$$

Thus, using \propref{p:functor Phi glob upstairs}, we obtain that \thmref{t:Verdier upstairs} is equivalent to the following one:

\begin{thm} \label{t:Verdier upstairs global}
The following diagram of functors commutes
$$
\CD
(\Whit_{q,\on{glob}}(G)^c)^{\on{op}}     @>{\BD^{\on{Verdier}}}>>    \Whit_{q^{-1},\on{glob}}(G)^c  \\
@V{\text{\eqref{e:global composition upstairs}}}VV   @V{\text{\eqref{e:global composition upstairs}}}VV \\
(\Shv_{\CG^\Lambda}(\CZ_{\infty\cdot x})^{\on{loc.c}})^{\on{op}} @>{\BD^{\on{Verdier}}}>>  
\Shv_{(\CG^\Lambda)^{-1}}(\CZ_{\infty\cdot x})^{\on{loc.c}}
\endCD
$$
commutes. 
\end{thm}

We will prove \thmref{t:Verdier upstairs global} in the next section, using the notion of ULA 
(universal local acyclicity). 

\sssec{}

For future reference we record the following corollary of \thmref{t:Verdier upstairs global}
(which, given \propref{p:functor Phi glob}, is equivalent to \thmref{t:duality pres}):

\begin{cor} \label{c:duality pres global}
The diagram 
$$
\CD 
(\Whit_{q,\on{glob}}(G)^c)^{\on{op}}     @>{\BD^{\on{Verdier}}}>>    \Whit_{q^{-1},\on{glob}}(G)^c  \\
@V{\Phi_{\on{glob}}}VV   @VV{\Phi_{\on{glob}}}V   \\
(\Shv_{\CG^\Lambda}(\Conf_{\infty\cdot x})^{\on{loc.c}})^{\on{op}}  @>{\BD^{\on{Verdier}}}>> 
\Shv_{(\CG^\Lambda)^{-1}}(\Conf_{\infty\cdot x})^{\on{loc.c}}.
\endCD
$$
\end{cor} 

\section{Proof of the local Verdier duality theorem}   \label{s:proof of Verdier upstairs}

The goal of this section is to prove \thmref{t:Verdier upstairs global}. The main ingredient in the proof
is a certain local acyclicity (a.k.a. ULA) property of 
$$_{\Bun_T}\!\ICsm_{q^{-1},\on{glob}}\in \Shv_{(\CG^{G,T,\on{ratio}})^{-1}}(\BunBbm)$$
with respect to the projection
$$\ol\sfp^-:\BunBbm\to \Bun_G.$$

\ssec{Construction of the natural transformation}

In this subsection we will describe a framework that leads to the construction of a natural
transformation in \thmref{t:Verdier upstairs global}. 

\sssec{}  \label{sss:pullback morphism general}

Let us be given a Cartesian diagram of algebraic stacks
$$
\CD
\CY_1\underset{\CY}\times \CY_2  @>{'\!f}>>  \CY_2 \\
@V{'\!f^-}VV   @VV{f^-}V   \\
\CY_1 @>{f}>>  \CY,
\endCD
$$
where $\CY$ is smooth of dimension $d$.  Note that for $\CF_i\in \Shv_{\CG_i}(\CY_i)^{\on{loc.c}}$ there exists a canonically defined map
in $\Shv_{\CG_1\otimes \CG_2}(\CY_1\underset{\CY}\times \CY_2 )$
\begin{equation} \label{e:!* map}
({}'\!f^-)^*(\CF_1)\overset{*}\otimes ({}'\!f)^*(\CF_2)[-d]\to ({}'\!f^-)^!(\CF_1)\sotimes ({}'\!f)^!(\CF_2)[d],
\end{equation} 
see \cite[Sect. 5.1]{BG}. 

\medskip
 
In particular, we obtain that there exists a natural transformation
\begin{equation} \label{e:pullback duality}
\BD^{\on{Verdier}}(({}'\!f^-)^!(\CF_1)\sotimes ({}'\!f)^!(\CF_2)[d])\to 
({}'\!f^-)^!(\BD^{\on{Verdier}}(\CF_1))\sotimes ({}'\!f)^!(\BD^{\on{Verdier}}(\CF_2))[d].
\end{equation} 

\sssec{}

We apply the above paradigm to 
\begin{equation} \label{e:Zastava as cart}
\CD
\BunNbox\underset{\Bun_G}\times   \ol\Bun_{B^-}  @>{'\ol\sfp}>>   \ol\Bun_{B^-}  \\
@V{'\ol\sfp^-}VV   @VV{\ol\sfp^-}V   \\
\BunNbox    @>{\ol\sfp}>>  \Bun_G,
\endCD
\end{equation}
where we take $\CF_1=\CF\in \Whit_{q,\on{glob}}(G)$ and $\CF_2={}_{\Bun_T}\!\ICsm_{q^{-1},\on{glob}}$.

\medskip

We compose the two functors in \eqref{e:pullback duality}
with restriction along the open embedding
$$\bCZ_{\infty\cdot x}
\hookrightarrow \BunNbox\underset{\Bun_G}\times   \ol\Bun_{B^-}.$$

\medskip

We will prove the following more precise version of \thmref{t:Verdier upstairs}:

\begin{thm} \label{t:Verdier upstairs bis}
The natural transformation
\begin{multline}  \label{e:reform duality upstairs bis}
\BD^{\on{Verdier}}\left(({}'\ol\sfp^-)^!(\CF) \sotimes ({}'\ol\sfp)^!({}_{\Bun_T}\!\ICsm_{q^{-1},\on{glob}})[\dim(\Bun_G)]\right)\to \\
\to ({}'\ol\sfp^-)^!(\BD^{\on{Verdier}}(\CF)) \sotimes ({}'\ol\sfp)^!({}_{\Bun_T}\!\ICsm_{q^{-1},\on{glob}})[\dim(\Bun_G)],
\end{multline} 
arising from \eqref{e:pullback duality} is an isomorphism in $\Shv_{(\CG^{G,T,\on{ratio}})^{-1}}(\bCZ_{\infty\cdot x})$
for any $\CF\in \Whit_{q,\on{glob}}(G)^c$. 
\end{thm}

\ssec{Proof of \thmref{t:Verdier upstairs bis}, Step 1}

\sssec{}

With no restriction of generality, we can assume that $\CF\in \Whit_{q,\on{glob}}(G)$ is perverse and irreducible. 
Hence it is of the form $W^{\mu,!*}_{\on{glob}}$ for some $\mu\in \Lambda^+$, see \secref{sss:irred global}.

\medskip

Then $\CF$ is the Goresky-MacPherson extension of its restriction to the locally closed substack 
$(\ol\Bun^{\omega^{\rho}}_N)_{=\mu\cdot x}$. Let $\bCZ_{=\mu\cdot x}$ be the corresponding locally closed
substack of $\bCZ_{\infty\cdot x}$. 

\medskip

The key step in the proof of \thmref{t:Verdier upstairs bis} is the following:

\medskip

\noindent(*) \hskip1cm {\it The object 
\begin{equation} \label{e:suspect}
({}'\ol\sfp^-)^!(\CF) \sotimes ({}'\ol\sfp)^!({}_{\Bun_T}\!\ICsm_{q^{-1},\on{glob}})[\dim(\Bun_G)]\in \Shv_{\CG^\Lambda}(\bCZ_{\infty\cdot x})
\end{equation} 
is perverse and is isomorphic to the Goresky-MacPherson extension of its restriction to $\bCZ_{=\mu\cdot x}$.}

\medskip

In the rest of this subsection we will show how (*) implies the assertion of \thmref{t:Verdier upstairs bis}. 

\sssec{}

By (*), we know that both sides 
in \eqref{e:reform duality upstairs bis} are the Goresky-MacPherson extensions of their respective restrictions to 
$\bCZ_{=\mu\cdot x}$.  Hence, it is enough to show that the corresponding map in $\Shv_{\CG^\Lambda}(\bCZ_{=\mu\cdot x})$
is an isomorphism.  Thus we can replace the diagram \eqref{e:Zastava as cart} by 
$$
\CD
(\ol\Bun^{\omega^{\rho}}_N)_{=\mu\cdot x} \underset{\Bun_G}\times   \ol\Bun_{B^-}  @>{'\ol\sfp}>>   \ol\Bun_{B^-}  \\
@V{'\ol\sfp^-}VV   @VV{\ol\sfp^-}V   \\
(\ol\Bun^{\omega^{\rho}}_N)_{=\mu\cdot x}    @>{\ol\sfp}>>  \Bun_G,
\endCD
$$

Since $\CF|_{(\ol\Bun^{\omega^{\rho}}_N)_{=\mu\cdot x}}$ is lisse and $(\ol\Bun^{\omega^{\rho}}_N)_{=\mu\cdot x}$ is smooth,
the statement about the isomorphism is equivalent to one when $\CF$ is replaced by the constant/dualizing sheaf.

\medskip

Thus, we  have to show that the map
\begin{multline*}
\BD^{\on{Verdier}}\left(({}'\ol\sfp)^!({}_{\Bun_T}\!\ICsm_{q^{-1},\on{glob}})[\dim(\Bun_G)-d_{g,\mu}]\right)\to \\
\to ({}'\ol\sfp^-)^!(\BD^{\on{Verdier}}({}_{\Bun_T}\!\ICsm_{q^{-1},\on{glob}}))[\dim(\Bun_G)-d_{g,\mu}]
\end{multline*}
is an isomorphism, where $d_{g,\mu}=\dim((\ol\Bun^{\omega^{\rho}}_N)_{=\mu\cdot x}).$

\sssec{}

By (the metaplectic analog of) \cite[Proposition 3.6.5]{Ga6}, the object 
$$({}'\ol\sfp)^!({}_{\Bun_T}\!\ICsm_{q^{-1},\on{glob}})[\dim(\Bun_G)-d_{g,\mu}]$$
is perverse and is isomorphic to the Goresky-MacPherson extension of its restriction to 
$$\CZ_{=\mu\cdot x}:=\bCZ_{=\mu\cdot x}\underset{\BunBbm}\times \Bun_{B^-}.$$

Hence, it is enough to show that the map
\begin{multline} \label{e:e:reform duality upstairs bis const}
\BD^{\on{Verdier}}\left(({}'\sfp)^!({}_{\Bun_T}\!\ICsm_{q^{-1},\on{glob}}|_{\Bun_{B^-}})[\dim(\Bun_G)-d_{g,\mu}]\right)\to \\
\to ({}'\sfp^-)^!(\BD^{\on{Verdier}}({}_{\Bun_T}\!\ICsm_{q^{-1},\on{glob}}|_{\Bun_{B^-}}))[\dim(\Bun_G)-d_{g,\mu}]
\end{multline}
is an isomorphism in $\Shv_{\CG^\Lambda}(\CZ_{=\mu\cdot x})$. 

\medskip

Note now that the object $_{\Bun_T}\!\ICsm_{q^{-1},\on{glob}}|_{\Bun_{B^-}}$ is the constant sheaf (up to a cohomological shift).
Now the desired assertion follows from the next result:

\begin{lem}
In the setting of \secref{sss:pullback morphism general}, assume that $\CY_1$, $\CY_2$ and $\CY$ are equidimensional and smooth, 
and let
$$\CY'\hookrightarrow \CY_1\underset{\CY}\times \CY_2$$
be a smooth open substack of dimension $\dim(\CY_1)+\dim(\CY_2)-\dim(\CY)$. Then for $\CF_1$ and $\CF_2$ lisse, the restriction
of the map \eqref{e:pullback duality} to $\CY'$ is an isomorphism.
\end{lem}

\ssec{Proof of \thmref{t:Verdier upstairs bis}, Step 2}

In this subsection we will show how a ULA property of the global $\ICs$ implies statement (*) used in the previous subsection. 

\sssec{}

Let us be in the situation of \secref{sss:pullback morphism general}. Let $\overset{\circ}\CY_1\subset \CY_1$ be a locally closed
substack. Assume that $\CF_1\in \Shv_{\CG_1}(\CY_1)$ is perverse and is isomorphic to the Goresky-MacPherson
extension of its restriction to $\overset{\circ}\CY_1$.

\medskip

Recall the notion of \emph{ULA object} for $\CF\in \Shv_{\CG}(\CY')$ relative to a morphism $\CY'\to \CY$,
see \cite[Sect. 5.1]{BG}. A key observation that we will use in the proof of \thmref{t:Verdier upstairs bis} is the following:

\begin{lem}  \label{l:ULA}
Assume that $\CF_2\in \on{Perv}_{\CG_2}(\CY_2)$ is ULA with respect to $f^-$. Then the object
$$({}'\!f^-)^!(\CF_1)\sotimes ({}'\!f)^!(\CF_2)[d]$$ is perverse and is isomorphic to the 
Goresky-MacPherson extension of its restriction to $\overset{\circ}\CY_1\underset{\CY}\times \CY_2$.
\end{lem}

The proof of property (*) will consist of reducing to the situation in which we can apply
\lemref{l:ULA}.

\sssec{}

We now proceed with the proof of (*). For $\nu\in \Lambda$, denote
$$\bCZ_{\infty\cdot x}^\nu:=\bCZ_{\infty\cdot x}\underset{\Conf_{\infty\cdot x}}\times \Conf^\nu_{\infty\cdot x}.$$
Let $\CF^\nu$ denote the restriction of \eqref{e:suspect} to $\bCZ_{\infty\cdot x}^\nu$. 

\medskip

For an element $\nu'\in \Lambda^{\on{neg}}$, consider the factorization isomorphism 
$$
(\bCZ^{\nu'}\times \bCZ^\nu_{\infty\cdot x}) \underset{\Conf^{\nu'}\times \Conf^\nu_{\infty\cdot x}}\times 
\left(\Conf^{\nu'}\times \Conf^\nu_{\infty\cdot x}\right)_{\on{disj}} 
\simeq 
\bCZ^{\nu'+\nu}_{\infty\cdot x}\underset{\Conf^{\nu'+\nu}_{\infty\cdot x}}\times 
\left(\Conf^{\nu'}\times \Conf^\nu_{\infty\cdot x}\right)_{\on{disj}}$$
and consider the open substack of the LHS equal to
\begin{equation} \label{e:good factor}
(\oCZ{}^{\nu'}\times \bCZ^\nu_{\infty\cdot x}) \underset{\Conf^{\nu'}\times \Conf^\nu_{\infty\cdot x}}\times 
\left(\Conf^{\nu'}\times \Conf^\nu_{\infty\cdot x}\right)_{\on{disj}}.
\end{equation}

\medskip

We will consider the following two maps from \eqref{e:good factor}
to $\bCZ_{\infty\cdot x}$. One, denoted $f^{\nu'}_1$, is the projection on the second factor
\begin{equation} \label{e:good factor 2nd}
(\oCZ{}^{\nu'}\times \bCZ^\nu_{\infty\cdot x}) \underset{\Conf^{\nu'}\times \Conf^\nu_{\infty\cdot x}}\times 
\left(\Conf^{\nu'}\times \Conf^\nu_{\infty\cdot x}\right)_{\on{disj}}
\to \oCZ{}^{\nu'}\times \bCZ^\nu_{\infty\cdot x}\to \bCZ^\nu_{\infty\cdot x}
\end{equation}
and the other, denoted $f^{\nu'}_2$, the composite
\begin{multline} \label{e:good factor comp}
(\oCZ{}^{\nu'}\times \bCZ^\nu_{\infty\cdot x}) \underset{\Conf^{\nu'}\times \Conf^\nu_{\infty\cdot x}}\times 
\left(\Conf^{\nu'}\times \Conf^\nu_{\infty\cdot x}\right)_{\on{disj}}\hookrightarrow \\
\to (\bCZ^{\nu'}\times \bCZ^\nu_{\infty\cdot x}) \underset{\Conf^{\nu'}\times \Conf^\nu_{\infty\cdot x}}\times 
\left(\Conf^{\nu'}\times \Conf^\nu_{\infty\cdot x}\right)_{\on{disj}} \simeq \\
\simeq 
\bCZ^{\nu'+\nu}_{\infty\cdot x}\underset{\Conf^{\nu'+\nu}_{\infty\cdot x}}\times 
\left(\Conf^{\nu'}\times \Conf^\nu_{\infty\cdot x}\right)_{\on{disj}}  \to \bCZ^{\nu'+\nu}_{\infty\cdot x}.
\end{multline} 

\medskip

Both these maps are smooth.

\sssec{}

Fix an element $\lambda\in \Lambda^{\on{pos}}$. We will consider an open substack in \eqref{e:good factor}, to be denoted 
$\CZ^{\on{Fact},\nu',\leq \lambda}$, that consists of points satisfying the following conditions: 

\medskip

\noindent(i) We require that for the point of $\bCZ_{\infty\cdot x}$ (obtained by either \eqref{e:good factor 2nd} or \eqref{e:good factor comp}), 
the generalized $B^-$-reduction has total order of degeneracy is $\leq \lambda$. 

\medskip

\noindent(ii) We require that $-\nu'-\nu-\lambda$ be ``deep enough" in the dominant chamber, in the sense that 
$$\langle -\nu'-\nu-\lambda,\check\alpha_i\rangle >d$$
for some fixed integer $d$ (specified in \thmref{t:ULA} below). 

\medskip

It is easy to see that the union of all $\lambda$'s and $\nu'$ of the images of $\CZ^{\on{Fact},\nu',\leq \lambda}$ under the maps $f^{\nu'}_1$
cover $\bCZ^\nu_{\infty\cdot x}$. 

\sssec{}

Hence, in order to prove (*), it suffices to show that for all $\nu'$, the object
\begin{equation} \label{e:factorization pullback 1}
(f^{\nu'}_1)^!(\CF^\nu)|_{\CZ^{\on{Fact},\nu',\leq \lambda}},
\end{equation}
shifted cohomologically by $[\langle \nu',2\check\rho\rangle]=[-\dim(\oCZ{}^{\nu'})]$,
is perverse and is isomorphic to the Goresky-MacPherson extension of its restriction to 
$$\CZ^{\on{Fact},\nu',\leq \lambda}_{=\mu\cdot x}:=\CZ^{\on{Fact},\nu',\leq \lambda}
\underset{f^{\nu'}_1,\bCZ^\nu_{\infty\cdot x}}\times \bCZ^\nu_{\mu\cdot x}.$$

\sssec{}

Note now that the locally closed substack $\CZ^{\on{Fact},\nu',\leq \lambda}_{=\mu\cdot x}\subset \CZ^{\on{Fact},\nu',\leq \lambda}$ also equals
$$\CZ^{\on{Fact},\nu',\leq \lambda}\underset{f^{\nu'}_2,\bCZ^{\nu'+\nu}_{\infty\cdot x}}\times \bCZ^{\nu'+\nu}_{\mu\cdot x}.$$

\medskip

Now, by repeating the argument of \cite[Sect. 3.9]{Ga6}, one shows that there exists an isomorphism
$$(f^{\nu'}_2)^!(\CF^\nu)\simeq \CE \otimes (f^{\nu'}_1)^!(\CF^\nu),$$
where $\CE$ is \emph{lisse sheaf}, pulled back from the $\oCZ{}^{\nu'}$ factor, and  
placed in perverse cohomological degree $-\langle \nu',2\check\rho\rangle$. 

\medskip

Hence, in order to prove the desired property of \eqref{e:factorization pullback 1}, it suffices to establish the same property of 
$$(f^{\nu'}_2)^!(\CF^\nu)|_{\CZ^{\on{Fact},\nu',\leq \lambda}}.$$

\medskip

Now, the required assertion follows from \lemref{l:ULA} and the following result of \cite{Camp}:

\begin{thm} \label{t:ULA}
There exists an integer $d$ that only depends on the genus of $X$ with the following property: 
for any $\lambda\in \Lambda^{\on{pos}}$ and $\mu\in \Lambda$, 
satisfying $\langle \mu-\lambda,\check\alpha_i\rangle >d$, the restriction of
$_{\Bun_T}\!\ICsm_{q,\on{glob}}$ to the open locus of $\ol\Bun_{B^-}$ consisting of generalized $B^-$-reductions of total 
order of degeneracy $\leq \lambda$ and degree $\mu$ is ULA with respect to the projection
$$\ol\sfp^-:\ol\Bun_{B^-}\to \Bun_G.$$
\end{thm} 

\section{Hecke enhancement of the Verdier duality theorem}  \label{s:Hecke Zast}

Recall the functor
$$\Phi^{\bHecke}=\oblv_{\on{Fact}}\circ \Phi^{\bHecke}_{\on{Fact}}:
\bHecke(\Whit_{q,x}(G))\to \Shv_{\CG^\Lambda}(\Conf_{\infty\cdot x}).$$

For the proof of the main theorem (more precisely, for property (iii) for $\bCM^{\mu,!}_{\Whit}$ in \secref{sss:strategy}),
we will need an extension of \thmref{t:duality pres} for the functor $\Phi^{\bHecke}$. This is the subject of the present section.

\medskip

More precisely, we will state and prove \thmref{t:duality pres Hecke}, which will be used in the proof of 
\corref{c:fiber of standard}. 

\ssec{Hecke enhancement and duality}

In this subsection we will state \thmref{t:duality pres Hecke}, which expresses the commutation property of
the functor $\Phi^{\bHecke}$ with Verdier duality.

\sssec{}

Note that \thmref{t:Hecke ppty of J} supplies a system of functorial isomorphisms
\begin{equation} \label{e:before duality}
\Phi(\CF\star \Sat_{q,G}(V))\simeq \Sat_{q,T}(\Res^H_{T_H}(V))\star \Phi(\CF), \quad \CF\in \Whit_q(G),\,\, V\in \Rep(H),
\end{equation}
compatible with tensor products of objects $V\in \Rep(H)$. 

\medskip

In addition, \thmref{t:duality pres} establishes functorial isomorphisms
\begin{multline*} 
\BD^{\on{Verdier}}\left(\Phi(\CF\star \Sat_{q,G}(V))\right) \simeq
\Phi\left(\BD^{\on{Verdier}}(\CF\star \Sat_{q,G}(V))\right)\simeq \Phi\left(\BD^{\on{Verdier}}(\CF)\star \BD^{\on{Verdier}}(\Sat_{q,G}(V))\right) \simeq \\
\simeq \Phi\left(\BD^{\on{Verdier}}(\CF)\star \Sat_{q^{-1},G}(\tau^H(V^*))\right)
\end{multline*} 
and
\begin{multline*} 
\BD^{\on{Verdier}}\left(\Sat_{q,T}(\Res^H_{T_H}(V))\star \Phi(\CF)\right) \simeq
\BD^{\on{Verdier}}(\Sat_{q,T}(\Res^H_{T_H}(V)))\star \BD^{\on{Verdier}}(\Phi(\CF))\simeq  \\
\simeq \BD^{\on{Verdier}}(\Sat_{q,T}(\Res^H_{T_H}(V)))\star \Phi(\BD^{\on{Verdier}}(\CF))
\simeq \Sat_{q^{-1},T}(\tau^{T_H}(\Res^H_{T_H}(V^*)))\star \Phi(\BD^{\on{Verdier}}(\CF)).
\end{multline*}
where we assume that both $\CF$ and $V$ are compact. 

\sssec{}

Thus, on the one hand, applying $\BD^{\on{Verdier}}$ to both sides of \eqref{e:before duality}, 
we obtain a system of functorial isomorphisms
\begin{equation} \label{e:after duality}
\Phi\left(\BD^{\on{Verdier}}(\CF)\star \Sat_{q^{-1},G}(\tau^H(V^*))\right) \simeq 
\Sat_{q^{-1},T}(\tau^{T_H}(\Res^H_{T_H}(V^*)))\star \Phi(\BD^{\on{Verdier}}(\CF)),
\end{equation}
compatible with tensor products of objects $V\in \Rep(H)$.  

\medskip

On the other hand, applying \thmref{t:Hecke ppty of J} to the functor
$$\Phi:\Whit_{q^{-1}}(G)\to \Shv_{(\CG^\Lambda)^{-1}}(\Conf_{\infty\cdot x}),$$
we obtain a system of identifications 
\begin{equation} \label{e:after duality bis}
\Phi\left(\BD^{\on{Verdier}}(\CF)\star \Sat_{q^{-1},G}(\tau^H(V^*))\right) \simeq
\Sat_{q^{-1},T}(\Res^H_{T_H}(\tau^H(V^*))) \star \Phi(\BD^{\on{Verdier}}(\CF)),
\end{equation}
compatible with tensor products of objects $V\in \Rep(H)$.  

\medskip

We claim:

\begin{thm}  \label{t:duality and Hecke}
The identifications \eqref{e:after duality} and \eqref{e:after duality bis} are compatible via 
the canonical isomorphism 
$$\tau^{T_H}\circ \Res^H_{T_H} \simeq \Res^H_{T_H}\circ \tau^H.$$
These identifications satisfy a homotopy-coherent system of comptatibilities for tensor products
of the objects $V\in \Rep(H)^c$. 
\end{thm}

\sssec{}

As a formal consequence of \thmref{t:duality and Hecke} we obtain:

\begin{thm} \label{t:duality pres Hecke}
We have a commutative diagram
$$
\CD
(\bHecke(\Whit_q(G))^c)^{\on{op}}  @>{\BD^{\on{Verdier}}}>>  \bHecke(\Whit_{q^{-1}}(G))  \\
@V{\Phi^{\bHecke}}VV   @VV{\Phi^{\bHecke}}V  \\
(\Shv_{\CG^\Lambda}(\Conf_{\infty\cdot x})^{\on{loc.c}})^{\on{op}} @>{\BD^{\on{Verdier}}}>> 
\Shv_{(\CG^\Lambda)^{-1}}(\Conf_{\infty\cdot x})^{\on{loc.c}},
\endCD
$$ 
where the upper horizontal arrow is the equivalence \eqref{e:duality for Hecke Whit comp}. 
\end{thm} 

\sssec{}

The rest of this subsection is devoted to the proof of \thmref{t:duality and Hecke}. 
As the isomorphism of \thmref{t:duality pres} was proved by global methods, we will have to resort to global methods to prove 
\thmref{t:duality and Hecke}.  Thus, for the rest of this section the curve $X$ will be complete.

\ssec{Hecke structure on the functor $\Phi_{\on{glob}}$}

Recall the functor 
$$\Phi_{\on{glob}}: \Whit_{q,\on{glob}}(G)\to  \Shv_{\CG^\lambda}(\Conf_{\infty\cdot x}),$$
see \secref{sss:Phi glob}. As a first step towards the proof of \thmref{t:duality and Hecke},
we will replace it by a statement that involves $\Phi_{\on{glob}}$ instead of $\Phi$. 

\sssec{} \label{sss:Hecke structure glob}

The Hecke action on $\BunNbox$ makes $\Whit_{q,\on{glob}}(G)$ into a category acted on by $\Sph_{q,x}(G)$ on the right
(see Remark \ref{r:global Hecke action}). In particular, we obtain a $\Rep(H)$ on $\Whit_{q,\on{glob}}(G)$ via $\on{Sat}_{q,G}$. 

\medskip

We claim that an analog of \thmref{t:Hecke ppty of J} holds in this situation:

\begin{thmconstr}  \label{t:Hecke ppty of J glob} 
The functor $\Phi_{\on{glob}}$ intertwines the $\Rep(H)$-action on $\Whit_{q,x}(G)$ and the $\Rep(T_H)$-action
on $\Shv_{\CG^\lambda}(\Conf_{\infty\cdot x})$ given by translation functors (see \secref{sss:translation functors}). 
\end{thmconstr}

\begin{proof}

We employ the paradigm of \secref{sss:paradigm for Hecke}. We take
$\bC:=\Whit_{q,\on{glob}}(G)$, $\bE:=\Shv_{\CG^\Lambda}(\Conf_{\infty\cdot x})$ and
$$\bD:=\Shv_{(\CG^{G,T,\on{ratio}})^{-1}}((\BunBbm)_{\infty\cdot x}).$$ 

\medskip

We take the functor
$$\Whit_{q,\on{glob}}(G) \otimes \Shv_{(\CG^{G,T,\on{ratio}})^{-1}}((\BunBbm)_{\infty\cdot x})
\to \Shv_{\CG^\Lambda}(\Conf_{\infty\cdot x})$$
to be
\begin{equation} \label{e:global Psi}
\CF,\CF'\mapsto (\sv_{\Conf_{\infty\cdot x}})_*\left(({}'\ol\sfp^-)^!(\CF)\sotimes ({}'\ol\sfp)^!(\CF')[\dim(\Bun_G)]\right),
\end{equation}
where $'\ol\sfp^-$ and $'\ol\sfp$ denote the two projections
$$(\ol\Bun^{\omega^{\rho}}_N)_{\infty\cdot x} \overset{'\ol\sfp^-}\longleftarrow \bCZ_{\infty\cdot x,\infty\cdot x} 
\overset{'\ol\sfp}\longrightarrow  (\BunBbm)_{\infty\cdot x}.$$

\medskip

The Hecke actions for $G$ and $T$ on $(\BunBbm)_{\infty\cdot x}$ make 
$\Shv_{(\CG^{G,T,\on{ratio}})^{-1}}((\BunBbm)_{\infty\cdot x})$ into a module category for $\Rep(H)\otimes \Rep(T_H)$
(see Remark \ref{r:global Hecke action}).

\medskip 
 
We take $\bd\in \bHecke(\bD)$ to be the object
$$_{\Bun_T}\!\ICs_{q^{-1},\on{glob}}\in \Shv_{(\CG^{G,T,\on{ratio}})^{-1}}((\BunBb)_{\infty\cdot x})$$
of \thmref{t:global Hecke}. 

\medskip

One shows that the required compatibilities hold as in the local case. This gives the functor $\Phi_{\on{glob}}$ the required
structure. 

\end{proof}

\sssec{}

Recall now that according to \propref{p:functor Phi glob}, we have a canonical isomorphism
$$\Phi_{\on{glob}}\simeq \Phi\circ \pi_x^![d_g].$$

\medskip

Since the functor 
$$\pi_x^![d_g]:\Whit_{q,\on{glob}}(G)\to \Whit_q(G)$$
commutes with Hecke actions, the structure on of commutation with the action of $\Rep(H)$ on $\Phi$,
provided by \thmref{t:Hecke ppty of J} induces one on $\Phi_{\on{glob}}$. 

\medskip

However, if follows from the constructions that this structure on $\Phi_{\on{glob}}$ equals one constructed 
in \secref{sss:Hecke structure glob} above.

\sssec{}

Since $\pi_x^![d_g]$ is an equivalence that commutes with duality, we obtain that \thmref{t:duality and Hecke} is equivalent to
the corresponding statement for $\Phi_{\on{glob}}$:

\medskip

\begin{thm}  \label{t:duality and Hecke glob}
The following diagram commutes for $\CF\in \Whit_{q,\on{glob}}(G)^c$ and $V\in \Rep(H)^c$. 
$$
\CD
\BD^{\on{Verdier}}\left(\Phi_{\on{glob}}(\CF\star \Sat_{q,G}(V))\right)   @>{\text{\thmmref{t:Hecke ppty of J glob}}}>>  
\BD^{\on{Verdier}}\left(\Sat_{q,T}(\Res^H_{T_H}(V))\star \Phi_{\on{glob}}(\CF)\right) \\
@V{\text{\corrref{c:duality pres global}}}VV   @VVV  \\
\Phi_{\on{glob}}\left(\BD^{\on{Verdier}}(\CF\star \Sat_{q,G}(V))\right) & & 
\BD^{\on{Verdier}}(\Sat_{q,T}(\Res^H_{T_H}(V)))\star \BD^{\on{Verdier}}(\Phi_{\on{glob}}(\CF))  \\
@VVV   @VV{\text{\corrref{c:duality pres global}}}V   \\
\Phi_{\on{glob}}\left(\BD^{\on{Verdier}}(\CF)\star \BD^{\on{Verdier}}(\Sat_{q,G}(V))\right) & & 
\BD^{\on{Verdier}}(\Sat_{q,T}(\Res^H_{T_H}(V)))\star \Phi_{\on{glob}}(\BD^{\on{Verdier}}(\CF))  \\
& &   @VVV   \\
@VVV   \Sat_{q^{-1},T}(\tau^{T_H}(\Res^H_{T_H}(V^*)))\star \Phi_{\on{glob}}(\BD^{\on{Verdier}}(\CF)) \\ 
& & @VVV  \\
\Phi_{\on{glob}}\left(\BD^{\on{Verdier}}(\CF)\star \Sat_{q^{-1},G}(\tau^H(V^*))\right)  @>\text{\thmmref{t:Hecke ppty of J glob}}>>
\Sat_{q^{-1},T}(\Res^H_{T_H}(\tau^H(V^*)))\star \Phi_{\on{glob}}(\BD^{\on{Verdier}}(\CF)).
\endCD
$$
The commutation identifications satisfy a homotopy-coherent system of comptatibilities for tensor products
of the objects $V\in \Rep(H)^c$. 
\end{thm}

The rest of this section is devoted to \thmref{t:duality and Hecke glob}.

\ssec{A framework for commutation of Hecke structure with duality}  \label{ss:Hecke and duality abs}

In this subsection we will describe a general categorical framework for the proof of \thmref{t:duality and Hecke glob}.

\sssec{}   

Let us be in the paradigm of \secref{sss:paradigm for Hecke}. Assume that all categories involved are
compactly generated. 

\medskip

Let us consider $\bC^\vee$ and $\bD^\vee$ equipped with actions of $\Rep(H)$ given on compact objects
by the formula
$$\bc^\vee \star V=(\bc\star \tau^H(V^*))^\vee, \,\, V\star \bd^\vee=(\tau^H(V^*)\star \bd)^\vee,$$
and let us equip $\bD^\vee$ with an action of $\Rep(T_H)$ by the formula
$$\bd^\vee \star W=(\bd\star \tau^H(W^*))^\vee.$$

\sssec{}  \label{sss:Psi op}

Assume that the functor $\Psi:\bC\otimes \bD\to \bE$ preserves compactness, and let $\Upsilon$ denote
the resulting functor
$$\bC^\vee\otimes \bD^\vee\to \bE^\vee, \quad \Upsilon(\bc^\vee\otimes \bd^\vee):=(\Psi(\bc\otimes \bd))^\vee.$$ 

\medskip

Identifying 
$$\bC^\vee\underset{\Rep(H)}\otimes \bD^\vee\simeq (\bC\underset{\Rep(H)}\otimes \bD)^\vee,\quad
\bc^\vee\otimes \bd^\vee\mapsto (\bc\otimes \bd)^\vee,$$
we obtain that the functor $\Upsilon$ is also equipped with a factorization
$$\bC^\vee\otimes \bD^\vee\to \bC^\vee\underset{\Rep(H)}\otimes \bD^\vee\overset{\wt{\Upsilon}}\longrightarrow \bE$$
for some canonically defined functor $\wt{\Upsilon}$. 

\medskip 

Let now $\bd$ be an object of $\bHecke_{\on{rel}}(\bD)$. (We are not assuming that $\bd$
is compact in $\bHecke_{\on{rel}}(\bD)$, and a fortiori not in $\bD$.)
 
\sssec{}

Recall that to \emph{any} object $\bbf$ in a compactly generated category $\bF$ one
can attach its dual $\bbf^\vee \in \bF^\vee$, characterized uniquely by the property that 
$$\CHom_{\bF^\vee}(\bbf_1^\vee,\bbf^\vee)=\CHom_{\bF}(\bbf,\bbf_1), \quad \bbf_1\in \bF^c.$$

Explicitly, if $\bbf=\underset{k}{\on{colim}}\, \bbf_k$, then 
$$\bbf^\vee \simeq \underset{k}{\on{lim}}\, \bbf^\vee_k.$$

Let $\Psi:\bF\to \bE$ be a continuous functor that preserves compactness. Consider the corresponding functor 
$$\Upsilon:\bF^\vee\to \bE^\vee, \quad \Upsilon(\bbf^\vee_1):=(\Psi(\bbf_1))^\vee, \quad \bbf_1\in \bF^c.$$
Then we have a natural map
\begin{equation} \label{e:bad dual}
\Upsilon(\bbf^\vee)\to \Psi(\bbf)^\vee.
\end{equation}

This map is an isomorphism if $\Psi$ also commutes with \emph{limits}. 

\sssec{}  \label{sss:d dual}

Let $\bd^\vee$ be the corresponding object of $\bD^\vee$. Note that for any finite-dimensional
$V\in \Rep(H)$ (resp., $W\in \Rep(T_H)$), the canonical maps 
$$V\star \bd^\vee\to (\tau^H(V^*)\star \bd)^\vee \text{ and } \bd^\vee \star W=(\bd\star \tau^{T_H}(W^*))^\vee$$
coming from \eqref{e:bad dual} are isomorphisms (indeed, the functors $V\star -$ and $-\star W$
admit left adjoints and hence commute with limits).

\medskip

Thus, we obtain that the object $\bd^\vee\in \bD^\vee$ is a system of isomorphisms
$$V\star \bd^\vee \simeq \bd^\vee \star \tau^{T_H}(\Res^H_{T_H}(\tau^H(V))).$$

We identify 
$$\tau^{T_H}(\Res^H_{T_H}(\tau^H(V)))\simeq \Res^H_{T_H}(V).$$

This identification defines a lift of $\bd^\vee$ to an object of $\bHecke_{\on{rel}}(\bD^\vee)$.

\sssec{}

Consider the resulting functors
$$\Psi_\bd:\bC\to \bE \text{ and } \Upsilon_{\bd^\vee}:\bC^\vee\to \bE^\vee.$$

From \eqref{e:bad dual}, for a compact $\bc\in \bC$ we obtain a naturally defined map
\begin{equation} \label{e:bad duality and Psi}
\Upsilon_{\bd^\vee}(\bc^\vee)\to (\Psi_{\bd}(\bc))^\vee.
\end{equation}

Assume that these maps are isomorphisms. 

\sssec{}

By \secref{sss:paradigm for Hecke}, we have the isomorphisms
$$\Psi_\bd(\bc\star V) \simeq \Res^H_{T_H}(V)\star \Psi_\bd(\bc),$$
from which by duality we obtain the isomorphisms
\begin{equation} \label{e:after duality abs}
(\Psi_\bd(\bc\star V))^\vee \simeq \tau^{T_H}(\Res^H_{T_H}(V^*))\star (\Psi_\bd(\bc))^\vee.
\end{equation}

We also have the isomorphisms
\begin{equation} \label{e:other Hecke abs}
\Upsilon_{\bd^\vee}(\bc^\vee\star V) \simeq \Res^H_{T_H}(V) \star \Upsilon_{\bd^\vee}(\bc^\vee).
\end{equation}

\medskip

Unwinding the constructions, we obtain that the following diagrams are commutative:
\begin{equation} \label{e:Hecke and duality abs}
\CD
(\Psi_\bd(\bc\star V))^\vee   @>{\text{\eqref{e:after duality abs}}}>> \tau^{T_H}(\Res^H_{T_H}(V^*))\star (\Psi_\bd(\bc))^\vee \\
@A{\text{\eqref{e:bad duality and Psi}}}AA    @AA{\text{\eqref{e:bad duality and Psi}}}A   \\
\Upsilon_{\bd^\vee}((\bc\star V)^\vee) & &  \tau^{T_H}(\Res^H_{T_H}(V^*))\star \Upsilon_{\bd^\vee}(\bc^\vee) \\
@A{\sim}AA  @AA{\sim}A   \\
\Upsilon_{\bd^\vee}(\bc^\vee \star \tau^H(V^*))   @>{\text{\eqref{e:other Hecke abs}}}>>
\Res^H_{T_H}(\tau^H(V^*))\star  \Upsilon_{\bd^\vee}(\bc^\vee), 
\endCD
\end{equation}
where lower right vertical arrow comes from the identification
$$\tau^{T_H}\circ \Res^H_{T_H} \simeq \Res^H_{T_H}\circ \tau^H.$$
Moreover, the commutation identifications satisfy a homotopy-coherent system of comptatibilities for tensor products
of the objects $V\in \Rep(H)^c$. 

\ssec{Proof of \thmref{t:duality and Hecke glob}}

In this subsection we will prove \thmref{t:duality and Hecke glob} by showing that it fits into the paradigm of
\secref{ss:Hecke and duality abs} above. 

\sssec{}

We take $\bC$ to be the category $\Whit_{q,\on{glob}}(G)$, and $\bE$ to be $\Shv_{\CG^\Lambda}(\Conf_{\infty\cdot x})$.

\medskip

Recall that in \secref{sss:Hecke structure glob}, the category $\bD$ was taken to be 
$\Shv_{(\CG^{G,T,\on{ratio}})^{-1}}((\BunBbm)_{\infty\cdot x})$. However, here we will have
to somewhat modify this choice.

\sssec{}

Note that the algebraic stack $\BunBbm$ is disconnected, and its individual connected
components are \emph{not} quasi-compact. So, the question of compact generation of
$\Shv_{(\CG^{G,T,\on{ratio}})^{-1}}(\BunBbm)$ may be non-trivial. We will skirt this problem
as follows:

\medskip

Consider the full (but non-cocomplete) subcategory of $\Shv_{(\CG^{G,T,\on{ratio}})^{-1}}((\BunBbm)_{\infty\cdot x})$ 
generated by $T$-Hecke translates of the direct summands of $_{\Bun_T}\!\ICs_{q^{-1},\on{glob}}$
(these direct summands correspond to the different connected components of $\BunBbm$). We let $\bD$
(to be henceforth denoted $\bD_q$)
be the ind-completion of this category. By construction, we have a tautological functor
$$\bD_q\to \Shv_{(\CG^{G,T,\on{ratio}})^{-1}}((\BunBbm)_{\infty\cdot x}).$$

\medskip

Note that compact objects of $\bD_q$ map to \emph{locally compact} objects of 
$\Shv_{(\CG^{G,T,\on{ratio}})^{-1}}((\BunBbm)_{\infty\cdot x})$. In particular, Verdier duality on
$\Shv_{(\CG^{G,T,\on{ratio}})^{-1}}((\BunBbm)_{\infty\cdot x})$ is well-defined for these objects.

\medskip

This allows to identify $\bD_q^\vee$ with the category $\bD_{q^{-1}}$ so that we have a commutative diagram
$$
\CD
(\bD_q^c)^{\on{op}} @>>>  \left(\left(\Shv_{(\CG^{G,T,\on{ratio}})^{-1}}((\BunBbm)_{\infty\cdot x})\right)^{\on{loc.c}}\right)^{\on{op}}  \\
@VVV   @VVV   \\
\bD_{q^{-1}}^c  @>>> \left(\Shv_{\CG^{G,T,\on{ratio}}}((\BunBbm)_{\infty\cdot x})\right)^{\on{loc.c}}. 
\endCD
$$

\sssec{}

We take the functor $\Psi$ to be the composite of 
$$\Whit_{q,\on{glob}}(G)\otimes \bD_q\to
\Whit_{q,\on{glob}}(G)\otimes \Shv_{(\CG^{G,T,\on{ratio}})^{-1}}((\BunBbm)_{\infty\cdot x})$$
with the functor \eqref{e:global Psi}.  Let us denote this functor by $\Psi_q$. 

\medskip

By \thmref{t:Verdier upstairs bis}, the resulting functor
$$\Upsilon:\Whit_{q^{-1},\on{glob}}(G)\otimes \bD^\vee \to \Shv_{\CG^\Lambda}(\Conf_{\infty\cdot x})$$
from \secref{sss:Psi op} identifies with $\Psi_{q^{-1}}$, i.e., with 
\begin{multline*}
\Whit_{q^{-1},\on{glob}}(G)\otimes \bD_{q^{-1}}\simeq 
\Whit_{q^{-1},\on{glob}}(G)\otimes \Shv_{\CG^{G,T,\on{ratio}}}((\BunBbm)_{\infty\cdot x}) \to \\
\overset{\text{\eqref{e:global Psi}}}\longrightarrow \Shv_{(\CG^\Lambda)^{-1}}(\Conf_{\infty\cdot x}).
\end{multline*}

\sssec{}

Consider the functor
$$\wt\Psi_q:=\wt\Psi:\Whit_{q,\on{glob}}(G)\underset{\Rep(H)}\otimes \bD_q \to \Shv_{\CG^\Lambda}(\Conf_{\infty\cdot x}).$$

We claim that the resulting functor
$$\wt\Upsilon: \Whit_{q^{-1},\on{glob}}(G)\underset{\Rep(H)}\otimes \bD_{q^{-1}} \to \Shv_{\CG^\Lambda}(\Conf_{\infty\cdot x})$$
from \secref{sss:Psi op} identifies with $\wt\Psi_{q^{-1}}$. 

\medskip

This statement amounts to the fact that for $\CF\in \Whit_{q,\on{glob}}(G)^c$, $\CF'\in \bD_q$ and $V\in \Rep(H)^c$,
the isomorphisms
$$\Psi_q(\CF\star V,\CF')\simeq \Psi_q(\CF, V \star \CF')$$
and
$$\Psi_{q^{-1}}(\BD^{\on{Verdier}}(\CF)\star \tau^H(V^*),\BD^{\on{Verdier}}(\CF'))\simeq 
\Psi_{q^{-1}}(\BD^{\on{Verdier}}(\CF), \tau^H(V^*) \star \BD^{\on{Verdier}}(\CF'))$$
make the following diagram is commutative (in a way compatible with tensor products of objects $V$):
$$
\CD
\BD^{\on{Verdier}}(\Psi_q(\CF\star V,\CF'))  @>>> \BD^{\on{Verdier}}(\Psi_q(\CF, V \star \CF')) \\
@VVV   @VVV  \\
\Psi_{q^{-1}}(\BD^{\on{Verdier}}(\CF)\star \tau^H(V^*),\BD^{\on{Verdier}}(\CF')) @>>> 
\Psi_{q^{-1}}(\BD^{\on{Verdier}}(\CF), \tau^H(V^*) \star \BD^{\on{Verdier}}(\CF')). 
\endCD
$$

This follows from the fact that the natural transformations \eqref{e:pullback duality} involved
in the construction of the isomophism 
$$\Upsilon\simeq \Psi_{q^{-1}}$$
commute with proper pushforwards. 

\medskip

Similarly, the data of commutation with the action of $\Rep(T_H)$ on the functor 
$\wt\Upsilon$ that arises from one on $\wt\Psi_q$ agrees with the corresponding 
data on $\wt\Psi_{q^{-1}}$. 

\sssec{}

We take the object $\bd\in \bHecke_{\on{rel}}(\bD_q)$ (to be henceforth denoted $\bd_q$) to be 
$$_{\Bun_T}\!\ICs_{q^{-1},\on{glob}}.$$

Note that although $\bd_q$, viewed as an object of $\bD_q$, is not compact (indeed, it is spread over all connected 
components on $\BunBbm$), that its image in $\Shv_{(\CG^{G,T,\on{ratio}})^{-1}}((\BunBbm)_{\infty\cdot x})$ is locally compact. 

\medskip

The corresponding object $\bd_q^\vee\in \bD^\vee_q$ identifies with $\bd_{q^{-1}}$, i.e.,
$$_{\Bun_T}\!\ICs_{q,\on{glob}}.$$

\medskip

Recall now that according to \secref{sss:d dual}, the Hecke structure on $\bd_q$ gives rise to one on $\bd_q^\vee\in \bD^\vee_q$;
i.e., it lifts to an object of $\bHecke_{\on{rel}}(\bD^\vee_q)$. We have the following key assertion:

\medskip

\begin{thm} \label{t:Hecke on ICs and duality} 
Under the identifications $\bD^\vee_q\simeq \bD_{q^{-1}}$ and $\bd_q^\vee\simeq \bd_{q^{-1}}$, the structure on 
$\bd_q^\vee$ of object of $\bHecke_{\on{rel}}(\bD^\vee_q)$ coincides with the structure on $\bd_{q^{-1}}$ of object 
of $\bHecke_{\on{rel}}(\bD_{q^{-1}})$.
\end{thm}

\sssec{}

Assuming for a moment \thmref{t:Hecke on ICs and duality} we complete the proof of the desired global version of 
\thmref{t:duality and Hecke} by invoking the system of commutative diagrams \eqref{e:Hecke and duality abs}. 

\ssec{Hecke structure on the global IC sheaf and Verdier duality}

This rest of this section is devoted to the proof of \thmref{t:Hecke on ICs and duality}. In order to unburden the notation we
will write $\ICs_{q,\on{glob}}$ instead of $_{\Bun_T}\!\ICs_{q,\on{glob}}$. We will also switch from $B^-$ back to $B$. 

\medskip

In this subsection we will explain what \thmref{t:Hecke on ICs and duality} says in ``down-to-earth" terms. 

\sssec{}

Let us write down what \thmref{t:Hecke on ICs and duality} says in concrete terms. According to \thmref{t:global Hecke},
for $V\in \Rep(H)$ we have canonical isomorphisms
$$\ICs_{q,\on{glob}} \star \Sat_{q,G}(V) \simeq 
\Sat_{q,T}(\Res^H_{T_H}(V))\star {}\ICs_{q,\on{glob}}$$
and
$$\ICs_{q^{-1},\on{glob}} \star \Sat_{q^{-1},G}(V) \simeq 
\Sat_{q^{-1},T}(\Res^H_{T_H}(V))\star {}\ICs_{q^{-1},\on{glob}}.$$

The claim is that for $V\in \Rep(H)^c$, the following diagram commutes:

\medskip

\begin{equation} \label{e:Hecke IC glob concr}
\CD
\BD^{\on{Verdier}}({}\ICs_{q,\on{glob}} \star \Sat_{q,G}(V))   @>>>  
\BD^{\on{Verdier}}(\Sat_{q,T}(\Res^H_{T_H}(V))\star {}\ICs_{q,\on{glob}}) \\
& & @VVV   \\
@VVV  \BD^{\on{Verdier}}(\Sat_{q,T}(\Res^H_{T_H}(V))) \star \BD^{\on{Verdier}}({}\ICs_{q,\on{glob}}))  \\
& & @VVV  \\
\BD^{\on{Verdier}}({}\ICs_{q,\on{glob}}) \star \BD^{\on{Verdier}}(\Sat_{q,G}(V)))  & & 
\Sat_{q^{-1},T}(\tau^{T_H}(\Res^H_{T_H}(V^*)))  \star {}\ICs_{q^{-1},\on{glob}}  \\
@VVV  @VVV  \\
\ICs_{q^{-1},\on{glob}} \star \Sat_{q^{-1},G}(\tau^H(V^*))   @>>> 
\Sat_{q^{-1},T}(\Res^H_{T_H}(\tau^H(V^*))) \star {} \ICs_{q^{-1},\on{glob}},
\endCD
\end{equation}
where the lower right vertical arrow is given by
$$\tau^{T_H}\circ \Res^H_{T_H} \simeq \Res^H_{T_H}\circ \tau^H.$$
Moreover, the data of commutation is compatible with tensor products of the objects $V$. 

\sssec{}

Some simplifying remarks are in order:

\medskip

\noindent(i) As in Remark \ref{r:Rep H abelian} it is sufficient to establish the commutativity of the diagrams
\eqref{e:Hecke IC glob concr} for $V\in (\Rep(H))^\heartsuit$. 

\medskip

\noindent(ii) For $V$ in the abelian category, all the objects involved in \eqref{e:Hecke IC glob concr}
lie in the heart of the perverse
t-structure on $\Shv_{\CG^{G,T,\on{ratio}}}((\BunBb)_{\infty\cdot x})$. Hence, once we check
the commutation for individual objects $V$, the higher compatibilities would follow. 

\ssec{Digression: gluing different components of $\BunBb$ together}

Note that $\ICs_{q,\on{glob}}$ is not an irreducible perverse sheaf for the simple reason that it is supported
on all the different connected components of $\BunBb$. 

\medskip

In this subsection we will introduce a geometric device that allows to ``sew together" the various components
of $\BunBb$. More precisely, we will define a category such that, when regarded as an object in it, $\ICs_{q,\on{glob}}$
will be irreducible. 

\sssec{}

Fix another point $x\neq y\in X$. Let 
$$(\BunBb)_{\infty\cdot x,\on{good}\,\on{at}\,y}\subset (\BunBb)_{\infty\cdot x}$$
be the open sub-functor, where we require that our generalized $B$-reduction be 
non-degenerate at $y$. Restriction to the formal disc around $y$ defines a map
$$(\BunBb)_{\infty\cdot x,\on{good}\,\on{at}\,y}\to \on{pt}/\fL^+(B)_y.$$

\medskip

Consider the corresponding Hecke groupoid
$$
\CD
(\BunBb)_{\infty\cdot x,\on{good}\,\on{at}\,y}  @<{\hl_B}<<
_{(\BunBb)_{\infty\cdot x,\on{good}\,\on{at}\,y}}\!\on{Hecke}^{\on{loc}}_{B,y}   @>{\hr_B}>> (\BunBb)_{\infty\cdot x,\on{good}\,\on{at}\,y} \\
@VVV   @VVV  @VVV  \\
\on{pt}/\fL^+(B)_y  @<<<  \on{Hecke}^{\on{loc}}_{B,y}  @>>> \on{pt}/\fL^+(B)_y,
\endCD
$$
where in this diagram both squares are Cartesian. 

\medskip

The pullbacks of the gerbe $\CG^{G,T,\on{ratio}}$ with respect to $\hl_B$ and $\hr_B$ are naturally identified. 

\sssec{}

The natural projection
$$\on{Hecke}^{\on{loc}}_{B,y}\to \on{Hecke}^{\on{loc}}_{T,y}\to \Lambda$$
defines a decomposition of $\on{Hecke}^{\on{loc}}_{B,y}$ into connected components, indexed by the elements of $\Lambda$;
denote them by $\on{Hecke}^{\on{loc},\lambda}_{B,y}$.  Denote
$$\on{Hecke}^{\on{loc},+}_{B,y}=\underset{\lambda\in \Lambda^+}\sqcup\, \on{Hecke}^{\on{loc},\lambda}_{B,y}$$.

\medskip

For $\lambda\in \Lambda^+$ let 
$$\on{Hecke}^{\on{loc},\lambda,\on{restr}}_{B,y}\subset \on{Hecke}^{\on{loc},\lambda}_{B,y}$$
be the subfunctor
$$\fL^+(N)_y\backslash (\fL^+(N)_y \cdot t^\lambda)/\fL^+(N)_y\subset 
\fL^+(N)_y\backslash (\fL(N)_y \cdot t^\lambda)/\fL^+(N)_y.$$

\medskip

Then the map
$$\on{Hecke}^{\on{loc},\lambda,\on{restr}}_{B,y} \overset{\hr^{\lambda,\on{restr}}_B}\longrightarrow \on{pt}/\fL^+(B)_y$$
is an isomorphism and the map
$$\on{pt}/\fL^+(B)_y \overset{\hl^{\lambda,\on{restr}}_B}\longleftarrow \on{Hecke}^{\on{loc},\lambda,\on{restr}}_{B,y}$$
is a fibration into affine spaces of dimensions $\langle \lambda,2\check\rho\rangle$. 

\sssec{}

Consider the corresponding substacks
$$_{(\BunBb)_{\infty\cdot x,\on{good}\,\on{at}\,y}}\!\on{Hecke}^{\on{loc},\lambda,\on{restr}}_{B,y}\subset
{}_{(\BunBb)_{\infty\cdot x,\on{good}\,\on{at}\,y}}\!\on{Hecke}^{\on{loc}}_{B,y}.$$

\medskip

The resulting map
$${}_{(\BunBb)_{\infty\cdot x,\on{good}\,\on{at}\,y}}\!\on{Hecke}^{\on{loc},\lambda,\on{restr}}_{B,y}
\overset{\hr^{\lambda,\on{restr}}_B}\longrightarrow (\BunBb)_{\infty\cdot x,\on{good}\,\on{at}\,y}$$
is an isomorphism, and the map
$$(\BunBb)_{\infty\cdot x,\on{good}\,\on{at}\,y} 
\overset{\hl^{\lambda,\on{restr}}_B}\longleftarrow   {}_{(\BunBb)_{\infty\cdot x,\on{good}\,\on{at}\,y}}\!\on{Hecke}^{\on{loc},\lambda,\on{restr}}_{B,y}$$
is a fibration into affine spaces of dimensions $\langle \lambda,2\check\rho\rangle$. 
 
\sssec{}

The (ind)-algebraic stack $(\BunBb)_{\infty\cdot x,\on{good}\,\on{at}\,y}$ splits into connected components
$$(\BunBb)^\lambda_{\infty\cdot x,\on{good}\,\on{at}\,y}, \quad \lambda\in \Lambda.$$

The above maps $\hr^{\lambda,\on{restr}}_B,\hl^{\lambda,\on{restr}}_B$ define a system of maps 
$$m^{\lambda_2,\lambda_1}:
(\BunBb)^{\lambda_2}_{\infty\cdot x,\on{good}\,\on{at}\,y}\to (\BunBb)^{\lambda_1}_{\infty\cdot x,\on{good}\,\on{at}\,y}, 
\quad \lambda_2-\lambda_1 \in \Lambda^+.$$

\sssec{}

We can view the assignment 
$$\lambda \mapsto (\BunBb)^\lambda_{\infty\cdot x,\on{good}\,\on{at}\,y}$$ as a functor from (the opposite of)
$\Lambda$ viewed as a poset $$\lambda_1\preceq \lambda_2 \, \Leftrightarrow\, \lambda_2-\lambda_1\in \Lambda^+$$
to the category of (ind)-algebraic stacks. 

\medskip

Consider the functor 
$$(\Lambda,\preceq)\to \on{DGCat}$$ that sends
$$\lambda\mapsto \Shv_{\CG^{G,T,\on{ratio}}}((\BunBb)^\lambda_{\infty\cdot x,\on{good}\,\on{at}\,y})$$
and $\lambda_1\preceq \lambda_2$ to the functor $(m^{\lambda_2,\lambda_1})^![\langle \lambda_1-\lambda_2,2\check\rho\rangle]$. 
 
\sssec{}
 
Define
\begin{equation} \label{e:Hecke at y}
\Shv_{\CG^{G,T,\on{ratio}}}((\BunBb)_{\infty\cdot x,\on{good}\,\on{at}\,y})^{\on{Hecke}_{T,y}}:
=\underset{(\Lambda,\preceq)^{\on{op}}}{\on{lim}}\, \Shv_{\CG^{G,T,\on{ratio}}}((\BunBb)^\lambda_{\infty\cdot x,\on{good}\,\on{at}\,y}).
\end{equation} 
 
Informally, objects of this category are objects $\CF\in \Shv_{\CG^{G,T,\on{ratio}}}((\BunBb)_{\infty\cdot x,\on{good}\,\on{at}\,y})$
equipped with a homotopy-compatible system of identifications
$$(m^{\lambda_2,\lambda_1})^!(\CF^{\lambda_1})[\langle \lambda_1-\lambda_2,2\check\rho\rangle]\simeq \CF^{\lambda_2},$$
where 
$$\CF^\lambda:=\CF|_{(\BunBb)^\lambda_{\infty\cdot x,\on{good}\,\on{at}\,y}}.$$ 

\sssec{}

By a slight abuse of notation let us continue to denote by $\ICs_{q,\on{glob}}$ its restriction along the open
embedding
$$(\BunBb)_{\infty\cdot x,\on{good}\,\on{at}\,y}\hookrightarrow (\BunBb)_{\infty\cdot x}.$$

\medskip

It is clear that $\ICs_{q,\on{glob}}$ naturally lifts to an object of 
$\Shv_{\CG^{G,T,\on{ratio}}}((\BunBb)_{\infty\cdot x,\on{good}\,\on{at}\,y})^{\on{Hecke}_{T,y}}$.

\sssec{}

The key observation now is that for any $\gamma_1,\gamma_2\in \Lambda^\sharp$, we have
\begin{equation} \label{e:ICs irred}
\Hom_{\Shv_{\CG^{G,T,\on{ratio}}}((\BunBb)_{\infty\cdot x,\on{good}\,\on{at}\,y})^{\on{Hecke}_{T,y}}}
(\sfe^{\gamma_1}\star \ICs_{q,\on{glob}},\sfe^{\gamma_2}\star \ICs_{q,\on{glob}})=
\begin{cases}
&\sfe \text{ if } \gamma_1=\gamma_2 \text{ and } \\
&0 \text{ otherwise}.
\end{cases}
\end{equation} 

We emphasize that in the above formula, we are taking $\Hom(-,-)$, i.e., $H^0(\CHom(-,-))$. 

\begin{rem}
The isomorphism \eqref{e:ICs irred} takes place for $\Hom$ taken in the category \eqref{e:Hecke at y}, but not in 
$\Shv_{\CG^{G,T,\on{ratio}}}((\BunBb)_{\infty\cdot x})$, because in the latter each connected component would
contribute its own factor of $\sfe$. This was the reason for introducing the category \eqref{e:Hecke at y}.
\end{rem}

\sssec{}

We will use \eqref{e:ICs irred} as follows:

\medskip

First off, it follows from the definitions that the $G$- and $T$- Hecke actions at $x$ lift naturally to actions on the category 
\eqref{e:Hecke at y}. 

\medskip

Now, from \eqref{e:ICs irred} and \thmref{t:global Hecke} we obtain 

\medskip

\noindent(A) There exists a \emph{monoidal} functor $_q\!\Res^H_{T_H}:\Rep(H)^\heartsuit\to \Rep(T_H)^\heartsuit$, uniquely 
characterized by the the system of isomorphisms 
\begin{equation} \label{e:Res q}
\ICs_{q,\on{glob}} \star \Sat_{q,G}(V) \simeq \Sat_{q,T}({}_q\!\Res^H_{T_H}(V))\star {}\ICs_{q,\on{glob}}, \quad V\in \Rep(H)^\heartsuit,
\end{equation} 
taking place in $\Shv_{\CG^{G,T,\on{ratio}}}((\BunBb)_{\infty\cdot x,\on{good}\,\on{at}\,y})^{\on{Hecke}_{T,y}}$.  Indeed, for
$W\in \Rep(T_H)^\heartsuit$ we have
\begin{multline*}
\Hom(W,{}_q\!\Res^H_{T_H}(V)):= \\
=\Hom_{\Shv_{\CG^{G,T,\on{ratio}}}((\BunBb)_{\infty\cdot x,\on{good}\,\on{at}\,y})^{\on{Hecke}_y}}
(\Sat_{q,T}(W)\star \ICs_{q,\on{glob}},\ICs_{q,\on{glob}}\star \Sat_{q,G}(V)).
\end{multline*} 

\medskip

\noindent(B) There exists an isomorphism between monoidal functors $_q\!\Res^H_{T_H}\simeq \Res^H_{T_H}$.

\ssec{Proof \thmref{t:Hecke on ICs and duality}}

In this subsection we will finally prove \thmref{t:Hecke on ICs and duality}. 

\sssec{}

It is easy to see that Verdier duality is well-defined on objects of \eqref{e:Hecke at y} that
are locally compact as objects of $\Shv_{\CG^{G,T,\on{ratio}}}((\BunBb)_{\infty\cdot x,\on{good}\,\on{at}\,y})$.

\medskip

This implies that the objects appearing in the diagram \eqref{e:Hecke IC glob concr} can be considered as objects in 
\eqref{e:Hecke at y}. Hence, it is sufficient to establish the commutativity of the diagram \eqref{e:Hecke IC glob concr} in this context.

\sssec{}

We prove the required equality as follows: 

\medskip

Applying Verdier duality to \eqref{e:Res q}, we obtain \emph{an} isomorphism 
$$_q\!\Res^H_{T_H} \simeq {}_{q^{-1}}\!\Res^H_{T_H}$$
as monoidal functors $\Rep(H)\to \Rep(T_H)$. 

\medskip

We need to show that the composite isomorphism
$$\Res^H_{T_H}\simeq {}_q\!\Res^H_{T_H} \simeq {}_{q^{-1}}\!\Res^H_{T_H}\simeq \Res^H_{T_H}$$
is the identity map. 

\sssec{}

A priori, the above composite map is given by an element $t\in T_H$, and we need to see that $t=1$.
For that it is sufficient to see that $t$ acts as identity on the highest weight lines for each $V=V^\gamma$. 

\medskip

However, the latter is easy to see from the constructions. 

\newpage 

\centerline{\bf Part VIII: Baby Verma objects} 

\bigskip

The goal of this Part is to carry out the program indicated in \secref{sss:strategy}, i.e., to construct the 
``(dual) baby Verma" objects $\bCM^{\mu,!}_{\Whit}$ and $\bCM^{\mu,*}_{\Whit}$, and establish their properties. 
 
 \medskip
 
 The term ``baby Verma"  is due to the fact that under the equivalence 
with the category of modules over the small quantum group, the objects $\bCM^{\mu,!}_{\Whit}$ 
(resp., $\bCM^{\mu,*}_{\Whit}$) correspond to baby Verma (resp., dual baby Verma) mdoules, 
see Remark \ref{r:baby Verma}.

\section{The $B$-Hecke category and the Drinfeld-Pl\"ucker formalism}  \label{s:DrPl}

The construction of the objects $\bCM^{\mu,*}_{\Whit}$ is based on the 
\emph{Drinfeld-Pl\"ucker formalism}\footnote{Both the name ``Drinfeld-Pl\"ucker" and the mathematical idea belong to S.~Raskin.}, 
which is the subject of the present section. 

\ssec{The $B$-Hecke category}

Recall the setting of Sects. \ref{ss:Hecke} and \ref{ss:graded Hecke}. In this subsection we will need to complement that discussion by introducing 
yet another version of the Hecke category, this time relative to the Borel subgroup $B_H\subset H$. 

\sssec{}

Let $\bC$ be a category acted on by $\Rep(H)$. We define the category $\BHecke(\bC)$ to be
$$\bC\underset{\Rep(H)}\otimes \Rep(B_H),$$
where $B_H$ is the Borel subgroup of $H$, see \secref{ss:ten prod Rep(H)}.

\medskip

Tautologically, we can rewrite 
$$\BHecke(\bC)\simeq (\Hecke(\bC))^{B_H},$$
where we view $\Hecke(\bC)$ as acted on by $H$, see \secref{sss:de-eq}. 

\medskip

In what follows we will assume that $\bC$ is compactly generated, in which case 
$\BHecke(\bC)$ is also compactly generated. 

\sssec{}

The pair of adjoint functors
$$\Res^{B_H}_{T_H}:\Rep(B_H)\rightleftarrows \Rep(T_H):\coInd^{B_H}_{T_H}$$
gives rise to the (same named) functors 
$$\Res^{B_H}_{T_H}:\BHecke(\bC)\rightleftarrows \bHecke(\bC):\coInd^{B_H}_{T_H}.$$

\medskip

Being the left adjoint of a continuous functor, the functor $\Res^{B_H}_{T_H}$ preserves 
compactness. However, we claim that more is true:

\begin{lem} \label{l:detect compact by restr}
If $\bc\in \BHecke(\bC)$ is such that $\Res^{B_H}_{T_H}(\bc)\in \bHecke(\bC)$ is compact,
then $\bc$ is compact. 
\end{lem}

\begin{proof}

This is a general phenomenon: for a category $\bD$ acted on by an algebraic group $H'$ (in out  
case $\bD=\Hecke(\bC)$ and $H'=B_H$), if an object $\bd\in \bD^{H'}$ is such that the underlying
object $\Res^{H'}(\bd)\in \bD$ is compact, then $\bd$ itself is compact.

\end{proof} 

\sssec{}

We now consider the functor
$$\Res^H_{B_H}:\Rep(H)\to \Rep(B_H),$$
and its right and left adjoints, denoted $\coInd^H_{B_H}$ and $\Ind^H_{B_H}$, respectively. The functor $\Res^H_{B_H}$
induces the (same named) functor:
$$\Res^H_{B_H}:\bC\simeq \bC\underset{\Rep(H)}\otimes \Rep(H)\to \bC\underset{\Rep(H)}\otimes \Rep(B_H)\simeq 
\BHecke(\bC)$$
and its right and left adjoints
\begin{equation} \label{e:ind and coind}
\coInd^H_{B_H},\Ind^H_{B_H}: \BHecke(\bC)\to \bC.
\end{equation} 

Being left adjoints of continuous functors, the functors $\Res^H_{B_H}$ and $\Ind^H_{B_H}$ preserve compactness. 

\sssec{}

Recall also that Serre duality for $H/B_H$ implies that we have a canonical isomorphism 
\begin{equation} \label{e:Serre duality}
\coInd^H_{B_H}(-)\simeq \Ind^H_{B_H}(\otimes k^{-2\rho_H})[-d],
\end{equation}
(here $d=\dim(H/B_H)$) as $\Rep(H)$-linear functors $\Rep(B_H)\to \Rep(H)$. 

\medskip

This implies a similar relationship between the functors \eqref{e:ind and coind}. In particular, we obtain
that the functor $\coInd^H_{B_H}$ also preserves compactness. 

\sssec{} 

Let us consider $\bC^\vee$ as a category acted on by $\Rep(H)$ as in \secref{ss:duality on Hecke gen}. 

\medskip

Then we obtain a canonical identification 
\begin{equation} \label{e:BB- duality}
\BHecke(\bC)^\vee \simeq \BmHecke(\bC^\vee),
\end{equation}
or equivalently
\begin{equation} \label{e:duality B Hecke}
(\BHecke(\bC)^c)^{\on{op}}\simeq \BmHecke(\bC^\vee)^c,\quad \bc\mapsto \bc^\vee,
\end{equation} 
for which the diagram
\begin{equation} \label{e:duality B T Hecke}
\CD
((\bC\otimes \Rep(B_H))^c)^{\on{op}}  @>>>   (\BHecke(\bC)^c)^{\on{op}}  \\
@VVV   @VV{\text{\eqref{e:duality B Hecke}}}V   \\  
((\bC^\vee\otimes \Rep(B^-_H))^c)^{\on{op}}   @>>>   \BmHecke(\bC^\vee)^c,
\endCD
\end{equation} 
where the left vertical arrow is the tensor product of 
$$(\bC^c)^{\on{op}} \to (\bC^\vee)^c, \quad \bc\mapsto \bc^\vee,$$
and the functor
$$(\Rep(B_H)^c)^{\on{op}}\to \Rep(B^-_H)^c, \quad V\mapsto \tau^H(V^*),$$
where we use $\tau^H$ as an isomorphism $B_H\to B^-_H$. 

\sssec{}

Note that from \eqref{e:Serre duality} we obtain
\begin{equation}  \label{e:Serre duality again}
(\coInd^H_{B_H}(\bc))^\vee \simeq \coInd^H_{B^-_H}(\bc^\vee\otimes \sfe^{-2\rho_H})[d],\quad \bc\in \Rep(B_H)^c.
\end{equation}

Note also, that we have a canonical identification
$$(\Res^{B_H}_{T_H}(\bc))^\vee \simeq \Res^{B^-_H}_{T_H}(\bc^\vee),\quad \bc\in \Rep(B_H)^c$$
where we use the identification
$$(\bHecke(\bC)^c)^{\on{op}}\to \bHecke(\bC^\vee)^c$$
as in \secref{sss:duality on Hecke graded} (i.e., we combine the usual duality for $T_H$ with 
$\tau^{T_H}$). 

\ssec{Behavior of the t-structure}

In this subsection we will study the behavior of the t-structure on $\BHecke(\bC)$.

\sssec{}

Assume that $\bC$ is equipped with a t-structure so that the action of $\Rep(H)$ on $\bC$ is given
by t-exact functors. Then, according to \secref{sss:t on Hecke}, the category $\BHecke(\bC)$
also acquires a t-structure. 

\medskip

By construction, the functors 
$$\Res^H_{B_H}:\bC\to \BHecke(\bC) \text{ and } \Res^{B_H}_{T_H}:\BHecke(\bC)\to \bHecke(\bC)$$
are t-exact.

\sssec{}  \label{sss:ampl coInd}

By adjunction, the functor 
$\Ind^H_{B_H}:\BHecke(\bC) \to \bC$
is right t-exact, while the functor $\coInd^H_{B_H}:\BHecke(\bC) \to \bC$ is left t-exact.

\medskip

Note, however, that it follows from \eqref{e:Serre duality} that the right cohomological amplitude
of $\coInd^H_{B_H}$ (resp., left cohomological amplitude of $\Ind^H_{B_H}$) is bounded by $d$.

\sssec{}

In what follows we will assume that the t-structure on $\bC$ is compactly generated (see \secref{sss:properties of t} for
what this means). We are going to prove the following analog of Serre's theorem on coherent sheaves on the projective space.

\medskip

Note that for $\bc\in \BHecke(\bC)$ and $\gamma$ we have a canonically defined map
\begin{equation} \label{e:gen by glob sect}
\Res^H_{B_H}(\coInd^H_{B_H}(\bc\otimes \sfe^{-\gamma}))\otimes \sfe^\gamma\to \bc.
\end{equation} 

We have: 

\begin{prop}   \label{p:Serre}
Let $\bc$ be an object of $\BHecke(\bC)^c\cap (\BHecke(\bC))^{\leq 0}$. Then for
all $\gamma$ deep enough in the dominant chamber (i.e., $\gamma\in \gamma_0+\Lambda^+_H$
for some fixed $\gamma_0$) we have: 

\smallskip

\noindent{\em(a)} The object $\coInd^H_{B_H}(\bc\otimes \sfe^{-\gamma})$ is connective. 

\smallskip

\noindent{\em(b)} The cofiber of the map \eqref{e:gen by glob sect} belongs to $(\BHecke(\bC))^{<0}$.

\end{prop} 

\begin{proof}

We can find $\bc'\in \bC^c\cap (\bC)^{\leq 0}$ and $V\in \Rep(B_H)^c\cap (\Rep(B_H))^{\leq 0}$ together with a map
$$\Res^H_{B_H}(\bc')\otimes V=:\bc_1\to \bc$$
whose cofiber belongs to $(\BHecke(\bC))^{<0}$. Let $\bc_2$ denote the fiber of the above map. 

\medskip

It is easy to see that both points of the proposition hold for $\bc_1$. From here we obtain that point (a) for $\bc_2$
implies point (b) for $\bc$.

\medskip

We will prove point (a) by descending induction. Namely, we claim that $\coInd^H_{B_H}(\bc\otimes \sfe^{-\gamma})$
belongs to $(\BHecke(\bC))^{\leq i}$ for $\gamma$ deep enough in the dominant chamber. The assertion for $i>d$
follows from \secref{sss:ampl coInd}. The induction step follows the fiber sequence
$$\bc_2\to \bc_1\to \bc$$
and the fact that the assertion holds for $\bc_1$. 

\end{proof} 


\sssec{}

Let us assume that the t-structure on $\bC$ in Artinian and the action of $\Rep(H)$ is accessible 
(see \secref{sss:accessible} for what this means). 

\medskip 

Recall (see \corref{c:irred Hecke graded bis}) that in this case the t-structure on $\bHecke(\bC)$ is also
Artinian. From here, combining with \lemref{l:detect compact by restr} we obtain:

\begin{cor} \label{c:B Artinian}
Under the above circumstances, the t-structure on $\BHecke(\bC)$ is Artinian.
\end{cor} 

\ssec{Drinfeld-Pl\"ucker formalism}  \label{ss:Dr-Pl}

In this subsection we will finally introduce the Drinfeld-Pl\"ucker formalism. 

\sssec{}  \label{sss:rel B}

Let now $\bC$ be as in \secref{ss:rel Hecke}, i.e., it is acted on by $\Rep(H)\otimes \Rep(T_H)$. Define
$$\BHecke_{\on{rel}}(\bC):=\bC\underset{\Rep(H)\otimes \Rep(T_H)}\otimes \Rep(B_H),$$
where $\Rep(T_H)\to \Rep(B_H)$ is the functor of restriction along the \emph{projection} $B_H\to T_H$. 

\medskip

We have the following diagram of categories
$$
\CD 
\BHecke_{\on{rel}}(\bC)  @>\Res^{B_H}_{T_H}>>  \bHecke_{\on{rel}}(\bC)  \\
@V{\oblv_{\on{rel}}}VV  @VV{\oblv_{\on{rel}}}V   \\
\BHecke(\bC) @>{\Res^{B_H}_{T_H}}>>  \bHecke(\bC),
\endCD
$$
where the vertical arrows are the functors
$$\bC\underset{\Rep(H)\otimes \Rep(T_H)}\otimes \bD \to \bC\underset{\Rep(H)}\otimes \bD,$$
right adjoint to the projections  
$$\bC\underset{\Rep(H)}\otimes \bD \to \bC\underset{\Rep(H)\otimes \Rep(T_H)}\otimes \bD,$$
for $\bD=\Rep(B_H)$ and $\bD=\Rep(T_H)$.  

\medskip

Let $\oblv_{\BHecke_{\on{rel}}}$ denote the forgetful functor
$$\BHecke_{\on{rel}}(\bC) \overset{\Res^{B_H}_{T_H}}\longrightarrow \bHecke_{\on{rel}}(\bC) 
\overset{\oblv_{\bHecke_{\on{rel}}}}\longrightarrow  \bC.$$

\sssec{}

Consider the base affine space $\ol{H/N_H}$ for the group $H$. This is an affine scheme acted on by 
$H\times T_H$. We consider the algebra of regular functions
$\on{Fun}(\ol{H/N_H})$ on $\ol{H/N_H}$ as an algebra object inside the monoidal category $\Rep(H)\otimes \Rep(T_H)$. 

\medskip

For $\bC$ as in \secref{sss:rel B}, define
$$\on{DrPl}(\bC):=\on{Fun}(\ol{H/N_H})\mod(\bC)
\simeq \bC\underset{\Rep(H)\otimes \Rep(T_H)}\otimes 
\on{Fun}(\ol{H/N_H})\mod(\Rep(H)\otimes \Rep(T_H))$$

\sssec{}  \label{sss:DrPl expl}

Explicitly, we can think about an object of $\on{DrPl}(\bC)$ as follows: this is an object $\bc\in \bC$ 
endowed with a system of maps
\begin{equation} \label{e:DrPl maps}
\bc \star (V^\gamma)^*\to \sfe^{-\gamma}\star \bc, \quad \gamma\in \Lambda^+_H
\end{equation}
that satisfy a homotopy-coherent system of compatibilities, starting from the commutative diagram
\begin{equation} \label{e:DrPl comp}
\CD 
\bc \star ((V^{\gamma_1})^*\otimes (V^{\gamma_2})^*)   @>>>  \bc\star (V^{\gamma_1+\gamma_2})^* \\
@V{\sim}VV    @VVV   \\
(\bc \star (V^{\gamma_1})^*) \star (V^{\gamma_2})^* & & \sfe^{-\gamma_1-\gamma_2} \star \bc \\
@VVV  @VV{\sim}V  \\
(\sfe^{-\gamma_1}\star \bc) \star (V^{\gamma_2})^*  @>>>  (\sfe^{-\gamma_1}\otimes\sfe^{-\gamma_2} )\star \bc \\
@V{\sim}VV    @VV{\sim}V   \\
\sfe^{-\gamma_1}\star (\bc \star (V^{\gamma_2})^*)  @>>>  \sfe^{-\gamma_1}\star (\sfe^{-\gamma_2}\star \bc). 
\endCD
\end{equation}

\begin{rem}
In the above commutative diagram, the upper horizontal arrow comes from the Pl\"ucker map 
\begin{equation}  \label{e:dual Plucker map}
(V^{\gamma_1})^*\otimes (V^{\gamma_2})^*\to (V^{\gamma_1+\gamma_2})^*
\end{equation} 
dual to the map
$$V^{\gamma_1+\gamma_2}\to V^{\gamma_1}\otimes V^{\gamma_2}$$
which induces the \emph{identity} map on the \emph{trivialized} highest weight lines. 

\medskip

Equivalently, by definition
$$V^\gamma:=\Ind^H_{B_H}(\sfe^\gamma) \text{ and } (V^\gamma)^*\simeq \coInd^H_{B_H}(\sfe^{-\gamma}),$$
and the map \eqref{e:dual Plucker map} is the canonical map
$$\coInd^H_{B_H}(\sfe^{-\gamma_1})\otimes \coInd^H_{B_H}(\sfe^{-\gamma_2})\to 
\coInd^H_{B_H}(\sfe^{-\gamma_1-\gamma_2}).$$
\end{rem} 

\sssec{}

Let $\sfj$ denote the open embedding 
$$H\backslash (H/N_H)/T_H)\hookrightarrow H\backslash (\ol{H/N_H})/T_H.$$
The pair of adjoint functors
\begin{equation} \label{e:base affine space emb}
\sfj^*:\QCoh(H\backslash (\ol{H/N_H})/T_H) \rightleftarrows \QCoh(H\backslash (H/N_H)/T_H):\sfj_*
\end{equation} 
induces an adjoint pair 
$$\sfj^*:\bC\underset{\Rep(H)\otimes \Rep(T_H)}\otimes \QCoh(H\backslash (\ol{H/N_H})/T_H) \rightleftarrows
\bC\underset{\Rep(H)\otimes \Rep(T_H)}\otimes \QCoh(H\backslash (H/N_H)/T_H):\sfj_*.$$

We identify
$$H\backslash (H/N_H)/T_H\simeq \on{pt}/B_H$$
and 
$$\QCoh(H\backslash (\ol{H/N_H})/T_H) \simeq \on{Fun}(\ol{H/N_H})\mod(\Rep(H)\otimes \Rep(T_H)).$$

Hence, we obtain an adjunction
\begin{equation} \label{e:DrPl adj}
\sfj^*: \on{DrPl}(\bC) \rightleftarrows \BHecke_{\on{rel}}(\bC):\sfj_*.
\end{equation} 

Since the co-unit of the adjunction
$$\sfj^*\circ \sfj_*\to \on{Id}$$ 
is an isomorphism in \eqref{e:base affine space emb}, the same is true for \eqref{e:DrPl adj}. I.e., the functor
$\sfj_*$ in \eqref{e:DrPl adj} is fully faithful. 

\sssec{} \label{sss:DrPl adj expl}

The composite functor
$$\on{DrPl}(\bC) \overset{\sfj^*}\longrightarrow \BHecke_{\on{rel}}(\bC) \overset{\oblv_{\BHecke_{\on{rel}}}}\longrightarrow \bC$$
can be explicitly described as follows (see \cite[Proposition 6.2.4]{Ga6}): 

\medskip

If we think of an object of $\on{DrPl}(\bC)$ as in \secref{sss:DrPl expl}, then the resulting object of $\bC$ identifies with
\begin{equation} \label{e:DrPl adj expl}
\underset{\gamma\in \Lambda^+_H}{\on{colim}}\, \sfe^{-\gamma}\star \bc\star V^\gamma,
\end{equation}
where we regard $\Lambda^+_H$ as a (filtered!) poset with respect to
$$\gamma_1\preceq \gamma_2\, \Leftrightarrow\, \gamma_2-\gamma_1=:\gamma\in \Lambda^+_H,$$
and the transition maps are given by
\begin{multline*} 
\sfe^{-\gamma_1}\star \bc\star V^{\gamma_1} \to 
\sfe^{-\gamma_1}\star \bc\star ((V^{\gamma})^* \otimes V^{\gamma}\otimes V^{\gamma_1}) \simeq
(\sfe^{-\gamma_1}\star \bc\star (V^{\gamma})^*)\star (V^{\gamma}\otimes V^{\gamma_1}) \to \\
\to (\sfe^{-\gamma_1}\star \sfe^{-\gamma}\star \bc) \star (V^{\gamma}\otimes V^{\gamma_1}) \to
\sfe^{-\gamma_1-\gamma}\star \bc \star V^{\gamma_1+\gamma}.
\end{multline*} 

\ssec{The relative vs non-relative case}  \label{ss:rel non-rel}

In this subsection we will discuss some variants of the construction in \secref{ss:Dr-Pl}.

\sssec{}

Let now $\bC_0$ be a category equipped just with an action of $\Rep(H)$. Set
$$\bC:=\Rep(T_H)\otimes \bC_0,$$
so that 
$$\BHecke_{\on{rel}}(\Rep(T_H) \otimes \bC_0)\simeq \BHecke(\bC_0).$$

\medskip

Note that the functor 
$$\oblv_{\BHecke_{\on{rel}}}:\BHecke_{\on{rel}}(\Rep(T_H) \otimes \bC_0)\to \Rep(T_H)\otimes \bC_0$$
identifies with the composite
$$\BHecke(\bC_0)\overset{\Rep^H_{B_H}}\longrightarrow \bHecke(\bC_0) \overset{\oblv_{\bHecke}}\longrightarrow 
\Rep(T_H)\otimes \bC_0$$.

\sssec{}   \label{sss:DrPl adj expl non-rel}

Consider the adjunction
$$\sfj^*:\on{DrPl}(\Rep(T_H)\otimes \bC_0)\rightleftarrows \BHecke(\bC_0):\sfj_*.$$

\medskip

We can think of objects of $\on{DrPl}(\Rep(T_H)\otimes \bC_0)$ as follows. These are families
of objects
$$\{\bc^\gamma\in \bC,\gamma\in \Lambda_H\}$$
equipped with a system of maps
\begin{equation} \label{e:DrPl maps non-rel}
\bc^{\gamma_1}\star (V^\gamma)^*\to \bc^{\gamma_1-\gamma}, \quad \gamma\in \Lambda^+_H,
\end{equation}
satisfying an appropriate system of compatibilities. 

\medskip

In terms of this description, the corresponding functor 
$$\oblv_{\BHecke_{\on{rel}}} \circ \sfj^*:\on{DrPl}(\Rep(T_H)\otimes \bC_0)\to \Rep(T_H)\otimes \bC_0$$
sends a system $\{\bc^\gamma\}$ as above to an object 
$$\{\wt\bc^\gamma\}\in \Rep(T_H)\otimes \bC_0$$
with
\begin{equation} \label{e:DrPl adj expl non-rel}
\wt\bc^{\gamma'}\simeq  \underset{\gamma\in \Lambda^+_H}{\on{colim}}\, \bc^{-\gamma+\gamma'}\star V^\gamma,
\end{equation}

\sssec{}

The functor 
$$\sfj_*:\BHecke(\bC_0)\to \on{DrPl}(\Rep(T_H)\otimes \bC_0)$$
can be described as follows. It sends an object $\bc\in \BHecke(\bC_0)$ to the system $\{\bc^\lambda\}$ with
$$\bc^\gamma=\coInd^{H}_{B_H}(\bc\otimes \sfe^\gamma),$$
where the maps \eqref{e:DrPl maps non-rel} are given by
\begin{multline*}
\coInd^{H}_{B_H}(\bc\otimes \sfe^{\gamma_1})\star (V^\gamma)^*\simeq 
\coInd^{H}_{B_H}\left(\bc\otimes \sfe^{\gamma_1}\otimes \Res^H_{B_H}((V^\gamma)^*)\right)\to  \\
\to \coInd^{H}_{B_H}\left(\bc\otimes \sfe^{\gamma_1}\otimes \sfe^{-\gamma}\right)=
\coInd^{H}_{B_H}(\bc\otimes \sfe^{\gamma_1-\gamma}).
\end{multline*}

\sssec{}

Note that by adjunction we obtain that for any
$$\{\bc^\gamma\}\in \on{DrPl}(\Rep(T_H)\otimes \bC_0)$$
and $\bc:=\sfj^*(\{\bc^\gamma\})\in \BHecke(\bC)$, there exists a canonical map
\begin{equation} \label{e:to glob sect}
\bc^\gamma \to \coInd^H_{B_H}(\bc\otimes \sfe^\gamma).
\end{equation}

\sssec{}  \label{sss:DrPl and duality}

Recall the duality \eqref{e:BB- duality}.
For a compact object $\bc\in \BHecke(\bC_0)$, consider its dual $\bc^\vee\in  \BmHecke(\bC_0^\vee)$,
and consider the corresponding object
$$\sfj_*(\bc^\vee)\in \on{DrPl}^-(\Rep(T_H)\otimes \bC_0).$$

\medskip

From \eqref{e:Serre duality again} we obtain that $\sfj_*(\bc^\vee)$ is given by the system $\{\bc^\gamma\}$ with 
$$\bc^\gamma:=(\coInd^H_{B_H}(\bc\otimes \sfe^{\gamma+2\rho_H}))^\vee[-d].$$

\sssec{}

Let $\bC$ be again a category acted on by $\Rep(H)\otimes \Rep(T_H)$, and take $\bC_0:=\bC$,
where we disregard the $\Rep(T_H)$-action. Consider the right adjoint of the action functor
$$\bC\to \Rep(T_H)\otimes \bC_0.$$
This is a functor of $\Rep(T_H)$-module categories.

\medskip

Hence, it induces functors
\begin{equation} \label{e:forget rel Pl}
\on{DrPl}(\bC)\to \on{DrPl}(\Rep(T_H)\otimes \bC_0)
\end{equation}
\begin{equation} \label{e:forget rel B}
\BHecke_{\on{rel}}(\bC)\to  \BHecke(\bC_0)
\end{equation}
that make the diagrams
$$
\CD
\on{DrPl}(\bC)   @>>>  \on{DrPl}(\Rep(T_H)\otimes \bC_0)  \\
@V{\sfj^*}VV  @VV{\sfj^*}V  \\
\BHecke_{\on{rel}}(\bC)  @>>>  \BHecke(\bC_0)
\endCD
$$
and 
$$
\CD
\on{DrPl}(\bC)   @>>>  \on{DrPl}(\Rep(T_H)\otimes \bC_0)  \\
@A{\sfj_*}AA  @AA{\sfj_*}A \\
\BHecke_{\on{rel}}(\bC)  @>>>  \BHecke(\bC_0)
\endCD
$$
commute.  

\medskip

Note that the functor \eqref{e:forget rel B} is the same as the functor that we denoted $\oblv_{\on{rel}}$
in \secref{sss:rel B}. 

\medskip

The functor \eqref{e:forget rel Pl} sends an object $\bc$ as in \secref{sss:DrPl expl} to the system $\{\bc^\gamma\}$
with
$$\bc^\gamma:=\sfe^\gamma\star \bc.$$

\sssec{}

Applying \eqref{e:DrPl adj expl non-rel}, we obtain that the functor 
$$\oblv_{\bHecke}\circ \Res^{B_H}_{T_H}\circ \oblv_{\on{rel}}\circ \sfj^*:\on{DrPl}(\bC) \to \Rep(T_H)\otimes \bC$$
sends $\bc\in \on{DrPl}(\bC)$
to the object $\{\wt\bc^\gamma\}\in \Rep(T_H)\otimes \bC$ with 
\begin{equation} \label{e:DrPl adj expl rel}
\wt\bc^{\gamma'}\simeq  \underset{\gamma\in \Lambda^+_H}{\on{colim}}\, \sfe^{-\gamma+\gamma'}\star \bc \star V^\gamma. 
\end{equation}

\ssec{The semi-infinite IC sheaf via the Drinfeld-Pl\"ucker formalism}  \label{ss:Dr-Pl semiinf}

We will now show that the object $'\!\ICs_{q,x}\in \Shv_{\CG^G}(\Gr_G)$ defined in 
\secref{ss:fiber of semiinf} can be obtained via the Drinfeld Pl\"ucker formalism. 

\medskip

We will change the notation slightly and denoted the object $'\!\ICs_{q,x}$, when viewed as 
equipped with the relative Hecke structure (that latter thanks to \thmref{t:Hecke ICs}) by
$$'\bICs_{q,x}\in \bHecke_{\on{rel}}(\Shv_{\CG^G}(\Gr_G)).$$

\medskip

From now on we will omit the superscript $\omega^\rho$ and the subscript $x$.
So we will write $\Gr_G$ instead of $\Gr^{\omega^\rho}_{G,x}$, and $\fL(N)$ instead of
$\fL(N)^{\omega^\rho}_x$, etc. 

\sssec{}

Consider the category $\Shv_{\CG^G}(\Gr_G)^{\fL^+(T)}$. We regard it as acted on from the right by
$\Rep(H)$ (via $\Sat_{q,G}$) and on the left by $\Rep(T_H)$ (via $\Sat'_{q,T}$, see \secref{sss:Sat'}). 

\medskip

We claim that the object $\delta_{1,\Gr}\in \Shv_{\CG^G}(\Gr_G)^{\fL^+(T)}$ naturally upgrades to an object 
$$(\delta_{1,\Gr})^{\on{DrPl},\fL^+(T)}\in \on{DrPl}(\Shv_{\CG^G}(\Gr_G)^{\fL^+(T)}).$$

The corresponding maps \eqref{e:DrPl maps} 
$$\IC_{q,\ol\Gr^{-w_0(\gamma)}_G} \simeq \delta_{1,\Gr}\star \Sat((V^\gamma)^*)\to 
\sfe^{-\gamma}\star \delta_{1,\Gr}\simeq \delta_{t^{-\gamma},\Gr}[\langle \gamma,2\check\rho\rangle]$$
are obtained from the maps \eqref{e:transition map pre initial} by adjunction. 

\medskip

The higher compatibilities for these maps are explained in \cite[Sect. 2.7]{Ga6}.

\sssec{}

Consider the resulting object
$$\sfj^*((\delta_{1,\Gr})^{\on{DrPl},\fL^+(T)})\in \BHecke_{\on{rel}}(\Shv_{\CG^G}(\Gr_G)^{\fL^+(T)}).$$

We claim:

\begin{prop} \label{p:ident semiinf as DrPl}
The object
$$\Res^{B_H}_{T_H}(\sfj^*((\delta_{1,\Gr})^{\on{DrPl},\fL^+(T)}))\in \bHecke_{\on{rel}}(\Shv_{\CG^G}(\Gr_G)^{\fL^+(T)})$$
identifies canonically with $'\bICs_{q,x}$.
\end{prop}

\begin{proof}

The fact that
$$\oblv_{\BHecke_{\on{rel}}}(\sfj^*((\delta_{1,\Gr})^{\on{DrPl},\fL^+(T)}))\simeq {}'\!\ICs_{q,x}$$
as plain objects of $\Shv_{\CG^G}(\Gr_G)^{\fL^+(T)}$ follows from \secref{sss:DrPl adj expl}.

\medskip

The fact that the isomorphism respects the graded Hecke structure follows from the construction
of the latter on $'\!\ICs_{q,x}$, see \cite[Theorem 5.1.8]{Ga7}. 

\end{proof}

\sssec{}

From now on we will denote
$$'\tICs_{q,x}:=\sfj^*((\delta_{1,\Gr})^{\on{DrPl},\fL^+(T)})\in \BHecke_{\on{rel}}(\Shv_{\CG^G}(\Gr_G)^{\fL^+(T)}),$$
so that the identification of \propref{p:ident semiinf as DrPl} says that 
$$\Res^{B_H}_{T_H}({}'\tICs_{q,x})\simeq {}'\bICs_{q,x}.$$
 
\sssec{} 

Recall now that $'\!\ICs_{q,x}$ was actually an object of the full subcategory 
$$((\SI_{q,x}(G))^{\fL^+(T)}:=
\Shv_{\CG^G}(\Gr_G)^{\fL(N)\cdot \fL^+(T)}\subset 
\Shv_{\CG^G}(\Gr_G)^{\fL^+(T)}.$$

Since the above embeddings are fully faithful, we obtain that $'\tICs_{q,x}$ automatically belongs to 
$$\BHecke_{\on{rel}}(((\SI_{q,x}(G))^{\fL^+(T)})\subset \BHecke_{\on{rel}}(\Shv_{\CG^G}(\Gr_G)^{\fL^+(T)}).$$

We will now show that when viewed as such, the object  $'\tICs_{q,x}$ can also be obtained by 
applying $\sfj^*$ to an object in $\on{DrPl}((\SI_{q,x}(G))^{\fL^+(T)})$.

\sssec{}

Namely, consider the object 
$$(\bi_{0})_!(\omega_{S^0})\in (\SI_{q,x}(G))^{\fL^+(T)}.$$

We claim that it naturally upgrades to an object of
$$((\bi_{0})_!(\omega_{S^0}))^{\on{DrPl}}\in \on{DrPl}((\SI_{q,x}(G))^{\fL^+(T)}),$$
so that the resulting object $\sfj^*(((\bi_{0})_!(\omega_{S^0}))^{\on{DrPl}})\in \BHecke_{\on{rel}}((\SI_{q,x}(G))^{\fL^+(T)})$
goes over under the embedding
$$\BHecke_{\on{rel}}((\SI_{q,x}(G))^{\fL^+(T)})\hookrightarrow \BHecke_{\on{rel}}(\Shv_{\CG^G}(\Gr_G)^{\fL^+(T)})$$
to $\sfj^*((\delta_{1,\Gr})^{\on{DrPl},\fL^+(T)}):={}'\tICs_{q,x}$. 

\sssec{}

In order to construct the Drinfeld-Pl\"ucker structure on $(\bi_{0})_!(\omega_{S^0})$, 
we consider the (partially defined) left adjoint $\on{Av}_!^{\fL(N)}$ to the embedding 
$$(\SI_{q,x}(G))^{\fL^+(T)}\hookrightarrow \Shv_{\CG^G}(\Gr_G)^{\fL^+(T)}.$$

\medskip

Note that for $\gamma\in \Lambda^\sharp$, we have a canonical identification
$$\on{Av}_!^{\fL(N)}(\delta_{t^\gamma,\Gr})\simeq
(\bi_{\gamma})_!(\omega_{S^{\gamma}}).$$

\medskip

Since Hecke convolutions (for $G$ and for $T$) are given by proper pushforwards, level-wise
application of $\on{Av}_!^{\fL(N)}$ induces a (partially defined) left adjoint to
$$\on{DrPl}((\SI_{q,x}(G))^{\fL^+(T)}) \hookrightarrow \on{DrPl}(\Shv_{\CG^G}(\Gr_G)^{\fL^+(T)})$$
and to 
$$\BHecke_{\on{rel}}((\SI_{q,x}(G))^{\fL^+(T)})\hookrightarrow \BHecke_{\on{rel}}(\Shv_{\CG^G}(\Gr_G)^{\fL^+(T)}).$$
Moreover, these functors are intertwined by the corresponding functors $\sfj^*$, $\sfj_*$, etc.

\sssec{}

Applying $\on{Av}_!^{\fL(N)}$ to $(\delta_{1,\Gr})^{\on{DrPl},\fL^+(T)}$ we obtain the desired object $((\bi_{0})_!(\omega_{S^0}))^{\on{DrPl}}$.
Moreover,
\begin{equation} \label{e:Av N DrPl}
\on{Av}_!^{\fL(N)}(\sfj^*((\delta_{1,\Gr})^{\on{DrPl},\fL^+(T)}))\simeq \sfj^*(((\bi_{0})_!(\omega_{S^0}))^{\on{DrPl}}).
\end{equation}

However, since $\sfj^*((\delta_{1,\Gr})^{\on{DrPl},\fL^+(T)})$ already belongs to $\BHecke_{\on{rel}}((\SI_{q,x}(G))^{\fL^+(T)})$,
we obtain that the left-hand side in \eqref{e:Av N DrPl} identifies with 
$\sfj^*((\delta_{1,\Gr})^{\on{DrPl},\fL^+(T)})$, as desired. 

\section{The dual baby Verma object in $\Shv_{\CG^G}(\Gr_G)^I$}  \label{s:Fl}

In order to construct the objects $\bCM^{\mu,*}_{\Whit}$, we will need to make a detour and discuss
the category of Iwahori-equivariant sheaves on $\Gr_G$. We will construct a particular ``dual baby Verma" object
$$\wt\CF^\semiinf_{\on{rel}}\in \BHecke(\Shv_{\CG^G}(\Gr_G)^I),$$
and study its properties. 

\medskip

Results of this section are of independent interest as the object $\wt\CF^\semiinf_{\on{rel}}$ (and its descendents
$\overset{\bullet}\CF{}^\semiinf$, $\CF^\semiinf$) are quite ubiquitous in this branch of representation theory, see
e.g., \cite{ABBGM}, \cite{FG2}, \cite{FG3}, \cite{Ga8}. 

\medskip

In the next section, we apply a simple manipulation to $\wt\CF^\semiinf_{\on{rel}}$ and produce from it
the sought-for objects $\bCM^{\mu,*}_{\Whit}$.

\ssec{The Iwahori-equivariant category}

In this subsection we recollect some facts pertaining to the behavior of the category of Iwahori-equivariant sheaves on $\Gr_G$. 

\sssec{}

Let $I\subset \fL^+(G)$ be the Iwahori subgroup. We consider the category 
$$\Shv_{\CG^G}(\Gr_G)^I.$$

In what follows we will use a slightly renormalized version of the category, namely the ind-completion
of $(\Shv_{\CG^G}(\Gr_G)^I)^{\on{loc.c}}$. We denote it $\Shv_{\CG^G}(\Gr_G)^{I,\on{ren}}$. 
The category $\Shv_{\CG^G}(\Gr_G)^I$ carries a t-structure, for which the tautological functor
\begin{equation} \label{e:unren I}
\on{un-ren}:\Shv_{\CG^G}(\Gr_G)^{I,\on{ren}}\to \Shv_{\CG^G}(\Gr_G)^I
\end{equation} 
is t-exact.

\medskip

The advantage of $\Shv_{\CG^G}(\Gr_G)^{I,\on{ren}}$ is that the t-structure on it is Artinian
(see \secref{sss:properties of t} for what this means).

\sssec{}

The category $\Shv_{\CG^G}(\Gr_G)^{I,\on{ren}}$ is acted on from the right by $\Sph_{q,x}(G)$ by convolutions. 

\medskip

Consider the affine flag space
$$\Fl_G:=\fL(G)/I.$$

Convolution on the left defines an action of the monoidal category $\Shv_{\CG^G}(\Fl_G)^I$ on 
$\Shv_{\CG^G}(\Gr_G)^{I,\on{ren}}$. This action commutes with the above right action of $\Sph_{q,x}(G)$.

\medskip

To avoid notational confusion, henceforth, we will denote these two convolution functors by
$$\underset{\fL^+(G)}\star- \text{ and } -\underset{I}\star,$$
respectively. 

\sssec{}

We claim that there exists a canonically defined monoidal functor
\begin{equation} \label{e:BMW}
\Rep(T_H)\to \Shv_{\CG^G}(\Fl_G)^I, 
\end{equation}

Namely, for an element $\gamma\in \Lambda^\sharp= \Lambda_H\subset \Lambda$ consider the
corresponding orbit
$$I\cdot t^\gamma \cdot I/I \subset \on{Fl}_G.$$

Our choice of trivialization of $\CG_{T_H,x}$ defines an $I$-equivariant  
trivialization of the restriction $\CG^G|_{I\cdot t^\gamma\cdot I/I}$.
Let 
$$j_{\gamma,!},j_{\gamma,*}\in \Shv_{\CG^G}(\Fl_G)^I$$
denote the corresponding standard (resp., costandard) objects. 

\begin{prop}  \label{p:BMW}
We have canonical isomorphisms
$$j_{\gamma,!}\underset{I}\star j_{\gamma,*}\simeq \delta_{1,\Fl}\simeq  j_{\gamma,*}\underset{I}\star j_{\gamma,!}.$$
\end{prop}

\begin{proof}

First, we notice that the functor of convolution
$$j_{\gamma,!}\underset{I}\star -:\Shv_{\CG^G}(\Fl_G)^I\to \Shv_{\CG^G}(\Fl_G)^I$$
is the left adjoint of 
$$j_{-\gamma,*}\underset{I}\star -:\Shv_{\CG^G}(\Fl_G)^I\to \Shv_{\CG^G}(\Fl_G)^I.$$

\medskip

Now the assertion follows from the fact that these functors are self-equivalences, hence ``adjoint'' is the
same as ``inverse". 

\end{proof}

\sssec{}

The sought-for functor \eqref{e:BMW} is uniquely determined by the condition that it sends
$$\sfe^\gamma \mapsto j_{\gamma,!} \text{ for } \gamma\in \Lambda_H^+.$$

From \propref{p:BMW} it follows that \eqref{e:BMW} sends
$$\sfe^{-\gamma} \mapsto j_{-\gamma,*} \text{ for } \gamma\in \Lambda_H^+.$$

For general $\gamma\in \Lambda_H$, let us denote by $J_\gamma$ the image of
$\sfe^\gamma$ under \eqref{e:BMW}. 

\sssec{}

Thus, we obtain that the category $\Shv_{\CG^G}(\Gr_G)^{I,\on{ren}}$ is acted on by $\Rep(H)\otimes \Rep(T_H)$. 

\begin{rem}  \label{r:t on Hecke Iw}
Note, however, that while the action of $(\Rep(H))^\heartsuit$ on $\Shv_{\CG^G}(\Gr_G)^{I,\on{ren}}$ is given by t-exact functors,
this is \emph{not} the case for $(\Rep(T_H))^\heartsuit$. So, while the categories
$$\BHecke(\Shv_{\CG^G}(\Gr_G)^{I,\on{ren}}) \text{ and } \bHecke(\Shv_{\CG^G}(\Gr_G)^{I,\on{ren}})$$
carry well-behaved t-structures, the relative versions 
$$\BHecke_{\on{rel}}(\Shv_{\CG^G}(\Gr_G)^{I,\on{ren}}) \text{ and } \bHecke_{\on{rel}}(\Shv_{\CG^G}(\Gr_G)^{I,\on{ren}})$$
do not. 
\end{rem} 

\ssec{Construction of the dual baby Verma object}

In this subsection we define the main player in this section--the dual baby Verma object in $\Shv_{\CG^G}(\Gr_G)^{I,\on{ren}}$. 

\sssec{}

Consider the object $\delta_{1,\Gr}\in \Shv_{\CG^G}(\Gr_G)^{I,\on{ren}}$. We claim that it naturally upgrades
to an object 
$$(\delta_{1,\Gr})^{\on{DrPl},I}\in \on{DrPl}(\Shv_{\CG^G}(\Gr_G)^{I,\on{ren}}).$$

The corresponding system of maps \eqref{e:DrPl maps} is given by
\begin{equation} \label{e:DrPl maps I}
\IC_{q,\ol\Gr^{-w_0(\gamma)}_G}\to j_{-\gamma,*}\underset{I}\star \delta_{1,\Gr}
\end{equation}
obtained by adjunction from the maps 
\begin{equation} \label{e:DrPl maps I adj}
j_{\gamma,!}\underset{I}\star \delta_{1,\Gr} \to \IC_{q,\ol\Gr^\gamma_G},
\end{equation}
the latter being maps corresponding to the open embeddings
$$I\cdot t^\gamma \fL^+(G)/\fL^+(G)\hookrightarrow \fL^+(G)\cdot t^\gamma \fL^+(G)/\fL^+(G)=\Gr_G^\gamma.$$

One easily checks that the maps \eqref{e:DrPl maps I} satisfy the compatibility expressed by diagram
\eqref{e:DrPl comp}; namely one checks the commutativity of the 
corresponding diagram for the maps \eqref{e:DrPl maps I adj}:
$$
\CD
j_{\gamma_1,!}\underset{I}\star j_{\gamma_2,!} \underset{I}\star \delta_{1,\Gr} @>>>  
j_{\gamma_1,!}\underset{I} \star \IC_{q,\ol\Gr^{\gamma_2}_G} \\
@V{\sim}VV  @VV{\sim}V    \\
j_{\gamma_1+\gamma_2,!} \underset{I}\star \delta_{1,\Gr}  @>>>  
j_{\gamma_1,!}\underset{I} \star \delta_{1,\Gr} \underset{\fL^+(G)} \star \IC_{q,\ol\Gr^{\gamma_2}_G} \\ 
@VVV   @VV{\sim}V  \\
\IC_{q,\ol\Gr^{\gamma_1+\gamma_2}_G}  @>>> \IC_{q,\ol\Gr^{\gamma_1}_G}\underset{\fL^+(G)}\star \IC_{q,\ol\Gr^{\gamma_2}_G}. 
\endCD
$$

The higher compatibilities hold automatically, as the objects involved in \eqref{e:DrPl maps I adj} belong to the heart
of the t-structure. 

\sssec{}

Let $\wt\CF^\semiinf_{\on{rel}}$ denote the object of $\BHecke_{\on{rel}}(\Shv_{\CG^G}(\Gr_G)^{I,\on{ren}})$ equal to
$$\sfj^*((\delta_{1,\Gr})^{\on{DrPl},I}).$$

\medskip

We will also consider several objects obtained from $\wt\CF^\semiinf_{\on{rel}}$ by applying the various forgetful functors:
$$\overset{\bullet}\CF{}^\semiinf_{\on{rel}}:=\Res^{B_H}_{T_H}(\wt\CF^\semiinf_{\on{rel}})\in \bHecke_{\on{rel}}(\Shv_{\CG^G}(\Gr_G)^{I,\on{ren}});$$
$$\wt\CF^\semiinf:=\oblv_{\on{rel}}(\wt\CF^\semiinf_{\on{rel}})\in \BHecke(\Shv_{\CG^G}(\Gr_G)^{I,\on{ren}});$$
$$\overset{\bullet}\CF{}^\semiinf:=\Res^{B_H}_{T_H}(\wt\CF^\semiinf)\simeq \oblv_{\on{rel}}(\overset{\bullet}\CF{}^\semiinf_{\on{rel}})\in 
\bHecke(\Shv_{\CG^G}(\Gr_G)^{I,\on{ren}}).$$
 
\begin{rem}
We can also consider the object
$$\CF^\semiinf:=\Res^{B_H}(\wt\CF^\semiinf)\simeq \Res^{T_H}(\overset{\bullet}\CF{}^\semiinf)\in
\Hecke(\Shv_{\CG^G}(\Gr_G)^{I,\on{ren}}).$$
But this object will not play a prominent role in this paper. 
\end{rem}

\sssec{}

Recall now (see Remark \ref{r:t on Hecke Iw}) that the categories 
$$\BHecke(\Shv_{\CG^G}(\Gr_G)^{I,\on{ren}}) \text{ and } \bHecke(\Shv_{\CG^G}(\Gr_G)^{I,\on{ren}})$$
each carries a well-behaved t-structure, and the restriction functor $\Rep^{B_H}_{T_H}$ t-exact. 

\medskip

We claim:
\begin{prop}  \label{p:baby in heart} 
The object $\wt\CF^\semiinf$ (resp., $\overset{\bullet}\CF{}^\semiinf$) 
belongs to $(\BHecke(\Shv_{\CG^G}(\Gr_G)^{I,\on{ren}}))^\heartsuit[d]$
(resp., $(\bHecke(\Shv_{\CG^G}(\Gr_G)^{I,\on{ren}}))^\heartsuit[d]$). 
\end{prop}

\begin{proof}

It suffices to prove the assertion for $\overset{\bullet}\CF{}^\semiinf$.  Applying \eqref{e:DrPl adj expl rel}, 
and using the fact that the poset $(\Lambda_H,\preceq)$ is filtered, it suffices to
see that for a fixed $\gamma'$ and cofinal set of $\gamma$'s, the objects
$$j_{-\gamma+\gamma',*}\underset{I}\star  \IC_{q,\ol\Gr^\gamma_G}$$
belong to $(\Shv_{\CG^G}(\Gr_G)^{I,\on{ren}}))^\heartsuit[d]$. 

\medskip

We claim that this happens as soon as $-\gamma+\gamma'=:\gamma_0$ is dominant regular. Indeed, convolution with $\Sph_{q,x}(G)$ is t-exact,
so it suffices to see that the objects
$$j_{-\gamma_0,*}\underset{I}\star \delta_{1,\Gr}$$
belong to $(\Shv_{\CG^G}(\Gr_G)^{I,\on{ren}}))^\heartsuit[d]$ for $\gamma_0$ dominant regular. Indeed, in this case,
the map
$$I\cdot t^{-\gamma_0}\cdot I/I\to I\cdot t^{-\gamma}\cdot \fL^+(G)/\fL^+(G)$$
is a fibration into affine spaces of dimension $d$, while the inclusion 
$$I\cdot t^{-\gamma_0}\cdot \fL^+(G)/\fL^+(G)\hookrightarrow \Gr_G$$
is affine. 

\end{proof}

\ssec{Relation to the semi-infinite IC sheaf}  \label{ss:F and semiinf}

In this subsection we will see that the object $$\wt\CF{}^\semiinf_{\on{rel}}\in \BHecke_{\on{rel}}(\Shv_{\CG^G}(\Gr_G)^{I,\on{ren}})$$
introduced above, is closely related to the semi-infinite IC sheaf
$$'\tICs_{q,x}\in \BHecke_{\on{rel}}(\Shv_{\CG^G}(\Gr_G)^{\fL^+(T)}).$$

A similar relationship will hold for their descendants, in particular for 
$$\overset{\bullet}\CF{}^\semiinf_{\on{rel}}\in \bHecke_{\on{rel}}(\Shv_{\CG^G}(\Gr_G)^{I,\on{ren}}) \text{ and }
'\bICs_{q,x}\in \bHecke_{\on{rel}}(\Shv_{\CG^G}(\Gr_G)^{I,\on{ren}}).$$

\sssec{}  \label{sss:semiinf Iw}

We recall (see, e.g., \cite[Proposition 5.2.2]{Ga6}) that the (partially defined) functor 
$$\on{Av}_!^{\fL(N)}:\Shv_{\CG^G}(\Gr_G)^{\fL^+(T)}\to \Shv_{\CG^G}(\Gr_G)^{\fL(N)\cdot \fL^+(T)}=:\SI_{q,x}(G)^{\fL^+(T)}.$$
restricted to 
$$\Shv_{\CG^G}(\Gr_G)^I\subset \Shv_{\CG^G}(\Gr_G)^{\fL^+(T)}$$
defines an equivalence
$$\Shv_{\CG^G}(\Gr_G)^I\to \SI_{q,x}(G)^{\fL^+(T)}.$$

\medskip

The inverse equivalence is given by $\on{Av}_*^{\overset{\circ}{I}}$, where $\overset{\circ}{I}$ is the unipotent radical of $I$. 

\medskip

This equivalence is compatible with the right action of $\Rep(H)$ and with the left action of $\Rep(T_H)$. 
Hence, it induces equivalences 
$$\on{Av}_!^{\fL(N)}:\on{DrPl}(\Shv_{\CG^G}(\Gr_G)^I)\to \on{DrPl}(\SI_{q,x}(G)^{\fL^+(T)})$$
and
$$\on{Av}_!^{\fL(N)}:\BHecke_{\on{rel}}(\Shv_{\CG^G}(\Gr_G)^I)\to \BHecke_{\on{rel}}(\on{DrPl}(\SI_{q,x}(G)^{\fL^+(T)}),$$
which intertwine the functors $\sfj^*$ and $\sfj_*$. 

\sssec{}

By a slight abuse of notation, let us denote by the same symbol $(\delta_{1,\Gr})^{\on{DrPl},I}$ the image of 
(what was previously denoted) $(\delta_{1,\Gr})^{\on{DrPl},I}$ under the functor
$$\on{DrPl}(\Shv_{\CG^G}(\Gr_G)^{I,\on{ren}}\to \on{DrPl}(\Shv_{\CG^G}(\Gr_G)^I)$$
induced by the functor $\on{un-ren}$ of \eqref{e:unren I}.  We will use a similar convention for 
$\wt\CF^\semiinf_{\on{rel}}$. 

\medskip

We claim:
\begin{prop}  \label{p:F and semiinf}
The functor $\on{Av}_!^{\fL(N)}$ sends 
$$(\delta_{1,\Gr})^{\on{DrPl},I}\in \on{DrPl}(\Shv_{\CG^G}(\Gr_G)^I)$$
to 
$$((\bi_{0})_!(\omega_{S^0}))^{\on{DrPl}}\in \on{DrPl}(\SI_{q,x}(G)^{\fL^+(T)}).$$
\end{prop}

\begin{proof} 

The proof follows from the fact that the objects $j_{-\gamma,*}$ for $\gamma\in \Lambda^+_H$
equipped with the maps 
$$\IC_{q,\ol\Gr_G^{-w_0(\gamma)}}\to j_{-\gamma,*}$$
can be obtained from the objects
$\on{Av}_!^{\fL(N)}(\delta_{t^\gamma,\Gr})[\langle -\gamma,2\check\rho\rangle]$ 
equipped with the maps 
$$\on{Av}_!^{\fL(N)}(\delta_{1,\Gr})\underset{\fL^+(G)}\star \IC_{q,\ol\Gr_G^{-w_0(\gamma)}}\to 
\on{Av}_!^{\fL(N)}(\delta_{t^{-\gamma},\Gr})[\langle -\gamma,2\check\rho\rangle]$$
by applying the functor $\on{Av}_*^{\overset{\circ}{I}}$. 

\end{proof}

\begin{cor}  \label{c:F and semiinf}
The functor $\on{Av}_!^{\fL(N)}$ sends 
$$\wt\CF{}^\semiinf_{\on{rel}}\in \BHecke_{\on{rel}}(\Shv_{\CG^G}(\Gr_G)^I) \mapsto 
{}'\tICs_{q,x}\in \BHecke_{\on{rel}}(\on{DrPl}(\SI_{q,x}(G)^{\fL^+(T)})$$
and 
$$\overset{\bullet}\CF{}^\semiinf_{\on{rel}}\in \bHecke_{\on{rel}}(\Shv_{\CG^G}(\Gr_G)^I) \mapsto 
{}'\bICs_{q,x}\in \bHecke_{\on{rel}}(\on{DrPl}(\SI_{q,x}(G)^{\fL^+(T)}).$$
\end{cor}

\ssec{A twist by $w_0$}

We will now perform a (rather elementary) manipulation with $\wt\CF{}^\semiinf_{\on{rel}}$--a twist by the longest element
of the Weyl group. This way, we will define the object $\wt\CF{}^{\semiinf,-}_{\on{rel}}$ that we are actually after. 

\sssec{}

Consider the object
$$j_{w_0,*}\in \Shv(G/B)^B\subset \Shv_{\CG^G}(\Fl_G)^I.$$

Define
$$\wt\CF^{\semiinf,w_0}:=j_{w_0,*}\underset{I}\star \wt\CF^\semiinf[-d]\in \BHecke(\Shv_{\CG^G}(\Gr_G)^{I,\on{ren}});$$
$$\overset{\bullet}\CF{}^{\semiinf,w_0}:=j_{w_0,*}\underset{I}\star \overset{\bullet}\CF{}^\semiinf[-d]\in \bHecke(\Shv_{\CG^G}(\Gr_G)^{I,\on{ren}}),$$
where $d=\dim(G/B)$.

\sssec{}

Note that by \secref{ss:rel non-rel}, the object $\wt\CF^{\semiinf,w_0}$ can be thought of as obtained by applying the
functor $\sfj^*$ to an object of 
$$\on{DrPl}(\Rep(T_H)\otimes \Shv_{\CG^G}(\Gr_G)^{I,\on{ren}})$$
with components
\begin{equation} \label{e:components key} 
j_{w_0,*}\underset{I}\star J_{\gamma,!} \underset{I}\star \delta_{1,\Gr}[-d].
\end{equation} 

Note that for $\gamma\in \Lambda^+_H$, we have
$$j_{w_0,*}\underset{I}\star J_{-\gamma,!}\simeq j_{w_0,*}\underset{I}\star j_{-\gamma,*}\simeq
j_{-w_0(\gamma),*} \underset{I}\star j_{w_0,*}.$$

Hence, the above object identifies with
\begin{equation} \label{e:components key dom}
j_{-w_0(\gamma),*} \underset{I}\star \delta_{1,\Gr},
\end{equation}
where we have used the fact that $j_{w_0,*} \underset{I}\star \delta_{1,\Gr} \simeq \delta_{1,\Gr}[d]$. 

\sssec{}

We claim:

\begin{prop}  \label{p:baby in heart w0}
The object $\wt\CF^{\semiinf,w_0}$ belongs to $(\BHecke(\Shv_{\CG^G}(\Gr_G)^{I,\on{ren}}))^\heartsuit$; the same
is true for $\overset{\bullet}\CF{}^{\semiinf,w_0}$.
\end{prop}

\begin{proof}

As in the proof of \propref{p:baby in heart}, it suffices to show that the objects 
$$j_{w_0,*}\underset{I}\star J_{-\gamma,!} \underset{I}\star \delta_{1,\Gr}\underset{\fL^+(G)}\star \IC_{q,\ol\Gr^\gamma_G}[-d]$$
belong to $(\Shv_{\CG^G}(\Gr_G)^{I,\on{ren}}))^\heartsuit$ for $\gamma \in \Lambda^+$. Using \eqref{e:components key dom},
we rewrite the above object as
$$j_{-w_0(\gamma),*} \underset{I}\star \delta_{1,\Gr} \star \IC_{q,\ol\Gr^\gamma_G}.$$

\medskip 

Since the functor $-\underset{\fL^+(G)}\star \IC_{q,\ol\Gr^\gamma_G}$ is t-exact, it suffices to show that 
$$j_{\gamma,*} \underset{I}\star \delta_{1,\Gr}\in (\Shv_{\CG^G}(\Gr_G)^{I,\on{ren}}))^\heartsuit$$
for $\gamma\in \Lambda^+_H$. 

\medskip

The latter follows from the fact that the projection
$$I\cdot t^{\gamma}\cdot I/I\to I\cdot t^{\gamma}\cdot \fL^+(G)/\fL^+(G)$$
is an isomorphism, while the inclusion
$$I\cdot t^{\gamma}\cdot \fL^+(G)/\fL^+(G)\hookrightarrow \Gr_G$$
is affine.

\end{proof}

\sssec{Digression}

Note that we have a \emph{canonical} equivalence
$$\Rep(B_H)\to \Rep(B^-_H)$$
as $\Rep(H)$-module categories.

\medskip

Indeed, choose of a representative $\sw'_0$ of $w_0\in W_H$. Then conjugation by
$\sw'_0$ defines the required functor. Now, as two such choices differ by an element in
$B_H$, the corresponding functors are canonically identified.

\medskip

Similarly, the action of $w_0$ defines a canonical self-equivalence of $\Rep(T_H)$ as a $\Rep(H)$-module 
category. 

\medskip

In particular, for $\bC$ acted on by $\Rep(H)$, we have canonical equivalences 
$$\BHecke(\bC)\overset{w_0}\to \BmHecke(\bC),$$
and
$$\bHecke(\bC)\overset{w_0}\to \bHecke(\bC),$$
that make the diagram
$$
\CD
\BHecke(\bC)  @>{w_0}>>  \BmHecke(\bC) \\
@V{\Res^{B_H}_{T_H}}VV   @VV{\Res^{B^-_H}_{T_H}}V   \\
\bHecke(\bC) @>{w_0}>> \bHecke(\bC) 
\endCD
$$
commute. 

\sssec{}

Define the objects
$$\wt\CF^{\semiinf,-}:=w_0(\wt\CF^{\semiinf,w_0})\in \BmHecke(\Shv_{\CG^G}(\Gr_G)^{I,\on{ren}})$$
and
$$\overset{\bullet}\CF{}^{\semiinf,-}:=w_0(\overset{\bullet}\CF{}^{\semiinf,w_0})\in \bHecke(\Shv_{\CG^G}(\Gr_G)^{I,\on{ren}}).$$

Note that we have
$$\overset{\bullet}\CF{}^{\semiinf,-}\simeq \Res^{B^-_H}_{T_H}(\wt\CF^{\semiinf,-}).$$

For completeness, define also
$$\CF^{\semiinf,-}:\Res^{B^-_H}(\wt\CF^{\semiinf,-})\simeq \Res^{T_H}(\overset{\bullet}\CF{}^{\semiinf,-})\in 
\Hecke(\Shv_{\CG^G}(\Gr_G)^{I,\on{ren}}).$$

\sssec{}  \label{sss:system for F-}

By construction, the object $\wt\CF^{\semiinf,-}$ is obtained by applying the functor $\sfj^*$ to an object of
$\on{DrPl}^-(\Rep(T_H)\otimes \Shv_{\CG^G}(\Gr_G)^{I,\on{ren}})$ that corresponds, in terms of \secref{sss:DrPl adj expl non-rel},
to the system
$$\gamma\mapsto j_{w_0,*}\underset{I}\star J_{w_0(\gamma),!} \underset{I}\star \delta_{1,\Gr}[-d].$$

Note that for $\gamma\in \Lambda^+_H$, the above object identifies with
\begin{equation} \label{e:components key dom w0}
j_{\gamma,*} \underset{I}\star \delta_{1,\Gr}.
\end{equation} 

Hence, according to \eqref{e:DrPl adj expl rel}, the object
$$\oblv_{\bHecke}\circ \Res^{B^-_H}_{T_H}(\wt\CF^{\semiinf,-})\in \Rep(T_H)\otimes \Shv_{\CG^G}(\Gr_G)^{I,\on{ren}}$$
is given by
\begin{equation} \label{e:key object as colimit}
\gamma'\rightsquigarrow 
\underset{\gamma}
{\on{colim}}\, j_{\gamma+\gamma',*} \underset{I}\star \IC_{q,\ol\Gr^{-w_0(\gamma)}_G},
\end{equation} 
where the colimit runs over the set $\gamma\in (-\gamma'+\Lambda^+_H)\cap \Lambda^+_H$. 


\ssec{Finiteness properties of $\wt\CF^{\semiinf,-}$}

In this subsection we will state a crucial finiteness property of the object $\wt\CF^{\semiinf,-}$ constructed above. 
It will instrumental in establishing the required properties of the objects $\bCM^{\mu,*}_{\Whit}$.

\sssec{}

Recall, following \eqref{e:to glob sect}, that according to \secref{sss:system for F-}, 
for $\gamma\in \Lambda^+_H$, there exists a canonical map
\begin{equation} \label{e:coind baby Verma}
j_{\gamma,*} \underset{I}\star \delta_{1,\Gr} \to \coInd^H_{B^-_H}(\wt\CF^{\semiinf,-}\otimes \sfe^\gamma).
\end{equation}  

We are going to prove:

\begin{thm} \label{t:finiteness of baby Verma}  \hfill

\smallskip

\noindent{\em(a)} The object $\wt\CF^{\semiinf,-}\in \BmHecke(\Shv_{\CG^G}(\Gr_G)^{I,\on{ren}})$
is compact.

\smallskip

\noindent{\em(b)} The maps \eqref{e:coind baby Verma} (with $\gamma\in \Lambda^+_H$) are isomorphisms.

\end{thm} 

As a special case of point (b) of the theorem we obtain:

\begin{cor}  \label{c:coind of baby Verma zero}
The map  $\delta_{1,\Gr} \to \coInd^H_{B^-_H}(\wt\CF^{\semiinf,-})$ is an isomorphism.
\end{cor}

We can now amplify point (b) of \thmref{t:finiteness of baby Verma} as follows:

\begin{cor} \label{c:coind of baby Verma}
For \emph{any} $\gamma\in \Lambda_H$ we have
$$\coInd^H_{B^-_H}(\wt\CF^{\semiinf,-}\otimes \sfe^\gamma)\simeq 
j_{w_0,*}\underset{I}\star J_{w_0(\gamma)} \underset{I}\star \delta_{1,\Gr}[-d].$$
\end{cor}

\begin{proof}

By construction, for any $\gamma\in \Lambda_H$, we have
$$J_\gamma \underset{I}\star  \wt\CF^\semiinf \simeq \wt\CF^\semiinf\otimes \sfe^\gamma.$$

Hence,
$$j_{w_0,*}\underset{I}\star J_\gamma \underset{I}\star j_{w_0,!}\underset{I}\star \wt\CF^{\semiinf,-}
\simeq \wt\CF^{\semiinf,-}\otimes \sfe^{w_0(\gamma)}.$$

This implies the required assertion by applying $\coInd^H_{B^-_H}$ to both sides and using 
\corref{c:coind of baby Verma zero}. 

\end{proof}

\begin{rem}  \label{r:coind of baby Verma}
Note that for $\gamma\in \Lambda^+_H$, the object
$$j_{w_0,*}\underset{I}\star J_{w_0(-\gamma)} \underset{I}\star \delta_{1,\Gr}[-d]$$
that appears in \corref{c:coind of baby Verma} identifies with 
$$j_{w_0,*}\underset{I}\star j_{w_0(-\gamma),!} \underset{I}\star \delta_{1,\Gr}[-d].$$

Note also that if $\gamma$ is regular, then the latter object identifies with 
$$j_{-\gamma\cdot w_0,!} \underset{I}\star \delta_{1,\Gr}[-d],$$
where we note that
$$j_{-\gamma\cdot w_0,!} \underset{I}\star \delta_{1,\Gr}\in (\Shv_{\CG^G}(\Gr_G)^{I,\on{ren}})^\heartsuit.$$
\end{rem}


%
%
%
%
%
%

\sssec{}

The rest of this section is essentially devoted to the proof of \thmref{t:finiteness of baby Verma}.
We will give two proofs: the first one by mimicking certain arguments from the paper \cite{ABBGM}.
And the second one, which actually explains ``what is going on" via a metaplectic version of
(some aspects of) the Arkhipov-Bezrukavnikov theory. 

\ssec{First proof of \thmref{t:finiteness of baby Verma}}  \label{ss:ABBGM}

\sssec{}

In \secref{ss:proof accessible Iw} we will prove:

\begin{thm} \label{t:accessible Iwahori}
The action of $\Rep(H)$ on $\Shv_{\CG^G}(\Gr_G)^{I,\on{ren}}$ is accessible.
\end{thm}

In the rest of this subsection we will show how \thmref{t:accessible Iwahori} implies 
\thmref{t:finiteness of baby Verma}.  

\sssec{}

First, we have the following assertion, whose proof is a verbatim repetition of the proof
of the argument in \cite[Proposition 3.2.6]{ABBGM}:

\begin{prop}  \label{p:ABBGM surj}
The map
\begin{equation} \label{e:ABBGM map}
\Res^H_{B^-_H}(j_{\gamma,*} \underset{I}\star \delta_{1,\Gr})\to \wt\CF{}^{\semiinf,-}\otimes \sfe^\gamma
\end{equation} 
arising by adjunction from \eqref{e:coind baby Verma}, is surjective, as a map of objects in 
$(\BmHecke(\Shv_{\CG^G}(\Gr_G)^{I,\on{ren}}))^\heartsuit$, for any $\gamma$ which is regular.
\end{prop}

\sssec{}

Note that by \corref{c:B Artinian}, from \propref{p:ABBGM surj} we immediately obtain that 
$\wt\CF^{\semiinf,-}$ is compact. This is point (a) of \thmref{t:finiteness of baby Verma}.

\medskip

To prove point (b), by the argument in \corref{c:coind of baby Verma}, it is enough to prove the assertion for
some $\gamma$ that is sufficiently dominant. Fix a dominant \emph{regular} $\gamma_0$. We will show that 
the map \eqref{e:coind baby Verma} is an isomorphism for $\gamma_0+\gamma$ for all $\gamma$ that are 
sufficiently dominant. 

\sssec{}

Consider the map \eqref{e:ABBGM map} for our $\gamma_0$. By Propositions \ref{p:Serre} and \ref{p:ABBGM surj}, 
for all $\gamma$ that are sufficiently dominant, the map
\begin{multline*} 
j_{\gamma_0,*}\underset{I}\star \IC_{q,\ol\Gr^\gamma_G}\simeq
(j_{\gamma_0,*}\underset{I}\star \delta_{1,\Gr}) \underset{\fL^+(G)}\star \Sat_{q,G}(V^\gamma)\simeq 
\coInd^H_{B^-_H}(\Res^H_{B^-_H}(j_{\gamma_0,*} \underset{I}\star \delta_{1,\Gr})\otimes \sfe^\gamma)\to \\
\to \coInd^H_{B^-_H}(\wt\CF^{\semiinf,-}\otimes \sfe^{\gamma_0}\otimes \sfe^\gamma) \simeq
\coInd^H_{B^-_H}(\wt\CF^{\semiinf,-}\otimes \sfe^{\gamma_0+\gamma})
\end{multline*}
is surjective. 

\medskip

However, it is easy to see that the latter map factors as
$$j_{\gamma_0,*}\underset{I}\star \Sat_{q,G}(V^\gamma) \to j_{\gamma_0+\gamma,*}\to 
\coInd^H_{B^-_H}(\wt\CF^{\semiinf,-}\otimes \sfe^{\gamma_0+\gamma}),$$
where:

\begin{itemize}

\item The first arrow is the map given by applying
$$j_{\gamma_0,*} \underset{I}\star j_{w_0,*}\underset{I}\star -[-d]$$
to the map \eqref{e:DrPl maps I}; 

\medskip

\item The second arrow is the map \eqref{e:coind baby Verma} for $\gamma_0+\gamma$. 

\end{itemize} 

Thus, we obtain that the map
\begin{equation} \label{e:coind baby Verma shift}
j_{\gamma_0+\gamma,*}\to 
\coInd^H_{B^-_H}(\wt\CF^{\semiinf,-}\otimes \sfe^{\gamma_0+\gamma})
\end{equation} 
above is surjective, as a map of objects in $(\Shv_{\CG^G}(\Gr_G)^{I,\on{ren}})^\heartsuit$. 
Hence, it remains to show that the map \eqref{e:coind baby Verma shift} is injective
(as a map of objects in $(\Shv_{\CG^G}(\Gr_G)^{I,\on{ren}})^\heartsuit$).

\sssec{}

We will show that the map \eqref{e:coind baby Verma} is injective for any $\gamma\in \Lambda_H^+$.

\medskip

Note that the map
$$\IC_{q,\ol\Gr^{\gamma}_G}\to j_{\gamma,*}$$
identifies $\IC_{q,\ol\Gr^{\gamma}_G}$ with the \emph{socle} of $j_{\gamma,*}$.
Hence, it is sufficient show that the composite
$$\IC_{q,\ol\Gr^{\gamma}_G}\to j_{\gamma,*}\to 
\coInd^H_{B^-_H}(\wt\CF^{\semiinf,-}\otimes \sfe^{\gamma})$$
is injective (equivalently, non-zero). Equivalently, we have to show that the map
\begin{equation} \label{e:from irred to F}
\Res^H_{B^-_H}(\IC_{q,\ol\Gr^{\gamma}_G}) \to \wt\CF^{\semiinf,-}\otimes \sfe^{\gamma}
\end{equation}
is non-zero.  

\medskip

After passing from $\wt\CF^{\semiinf,-}$ back to $\wt\CF^\semiinf$, the latter map fits 
into the following paradigm. 

\sssec{}

We start from an object $\bc^{\on{DrPl}}\in \on{DrPl}(\bC)$ with the underlying object of $\bC$ denoted by $\bc$.
Denote 
$$\bc':=\oblv_{\on{rel}}\circ \sfj^*(\bc^{\on{DrPl}})\in \bC\underset{\Rep(H)}\otimes \Rep(B_H).$$

Then, according to \eqref{e:to glob sect}, we have a canonical map
\begin{equation} \label{e:intial map DrPl}
\Res^H_{B_H}(\bc)\to \bc'.
\end{equation} 

The map \eqref{e:from irred to F} equals the composite
\begin{equation} \label{e:Vgamma map prel} 
\Res^H_{B_H}(\bc\star (V^\gamma)^*) \to \bc' \otimes \Res^H_{B_H}((V^\gamma)^*)\to \bc' \otimes \sfe^{-\gamma}.
\end{equation} 

\medskip

We claim that this map is non-zero under the following circumstances:

\begin{itemize}

\item(i) $\bC$ is equipped 
with a t-structure such that $(\Rep(H))^\heartsuit$ acts by t-exact functors;

\item(ii) $\sfe^{-\gamma}\star \bc\in (\bC)^\heartsuit$ for $\gamma$
\emph{sufficiently dominant};

\item(iii) $\bc$ is compact

\item(iv) The maps \eqref{e:DrPl maps} are non-zero.

\end{itemize} 

\sssec{}

Indeed, the map \eqref{e:Vgamma map prel} comes by adjunction from the map
\begin{equation} \label{e:Vgamma map} 
\Res^H_{B_H}(\bc)\to \bc' \to \bc'\otimes \sfe^{-\gamma}\otimes \Res^H_{B_H}(V^\gamma).
\end{equation} 

So it is sufficient to show that the latter map is non-zero. By \eqref{e:DrPl adj expl rel}, and assumptions (i) and (ii) on the
object $\bc$,
$$\oblv_{\bHecke}\circ \Res^{B_H}_{T_H}(\bc')\in (\Rep(T_H)\otimes \bC)^\heartsuit.$$

Hence, $\bc'\in (\BHecke(\bC))^\heartsuit$ and the second arrow in \eqref{e:Vgamma map} is injective. Hence, it 
remains to show that the first arrow in \eqref{e:Vgamma map}, i.e., map \eqref{e:intial map DrPl}, is non-zero.

\sssec{}

To prove the latter, it suffices to show that the induced map
$$\ind_{\bHecke}(\bc)\simeq \Res^H_{T_H}(\bc)\simeq \Res^{B_H}_{T_H}\circ \Res^H_{B_H}(\bc)\to  \Res^{B_H}_{T_H}(\bc')$$
is non-zero, i.e., the map
$$\sfe\otimes \bc\to \oblv_{\bHecke}\circ \Res^{B_H}_{T_H}(\bc')$$
is non-zero. 

\medskip

Applying \eqref{e:DrPl adj expl rel}, this equivalent to the fact that the map
$$\bc\to \underset{\gamma\in \Lambda^+_H}{\on{colim}}\, \sfe^{-\gamma}\star \bc\star V^\gamma$$
is non-zero in $\Rep(T_H)\otimes \bC$. 

\medskip

Since $\bc$ is compact and $(\Lambda^+_H,\preceq)$ is filtered, it suffices to show that the individual maps
$$\bc\to \sfe^{-\gamma}\star \bc\star V^\gamma$$
are non-zero. However, these maps are obtained by adjunction from the maps 
\eqref{e:DrPl maps}. 

\ssec{Proof of \thmref{t:accessible Iwahori}}  \label{ss:proof accessible Iw}

The proof will be an adaptation of the argument of \cite[Theorem 1.3.5]{ABBGM}. 

\sssec{}

Recall that $I$-orbits on $\Gr_G$ are in bijection with cosets $W^{\on{aff,ext}}/W$. 
Note that for an element $\wt{w}=w\cdot \lambda$ of $W^{\on{aff,ext}}$, the orbit
$$I\cdot \wt{w}\cdot \fL^+(G)/\fL^+(G)$$ 
carries an $I$-equivariant $\CG^G$-twisted sheaf if and only if $\lambda\in \Lambda^+_H$.

\medskip

For $\gamma\in \Lambda^+_H$, let 
$$\IC_{q,w\cdot \gamma}\in (\Shv_{\CG^G}(\Gr_G)^{I,\on{ren}}))^\heartsuit$$
denote the corresponding irreducible object. 

\medskip 

We claim:

\begin{thm} \label{t:restricted Iwahori}
Let $\gamma\in \Lambda^+_H$ and $w\in W$ be a pair of elements such that
$$\langle \gamma,\check\alpha_i\rangle=
\begin{cases}
&0, \text{ if } w(\alpha_i)\in \Lambda^{\on{pos}}; \\
&\on{ord}(q_i), \text{ if } w(\alpha_i)\in \Lambda^{\on{neg}}.
\end{cases}
$$
Then the object $\IC_{q,w\cdot \gamma}\in (\Shv_{\CG^G}(\Gr_G)^{I,\on{ren}}))^\heartsuit$
is restricted. In fact,
$$\IC_{q,w\cdot \gamma} \underset{\fL^+(G)}\star \IC_{q,\Gr^{\gamma'}_G}\simeq \IC_{q,w\cdot (\gamma+\gamma')}.$$
\end{thm} 

This theorem implies \thmref{t:accessible Iwahori} in the same way as \cite[Theorem 1.3.5(1)]{ABBGM} implies
\cite[Theorem 1.3.5(2)]{ABBGM} using the assumption that the derived group of $H$ is simply-connected. 

\sssec{}

To prove \thmref{t:restricted Iwahori} we repeat the argument in \cite[Sect. 2.1]{ABBGM}. The only difference 
is that instead of \cite[Theorem 7.1.7]{FGV} we use \thmref{t:restr}, applied to the weight $\lambda$ defined by
the formula
$$
\langle \lambda,\check\alpha_i\rangle=
\begin{cases}
&0, \text{ if } w(\alpha_i)\in \Lambda^{\on{pos}}; \\
&\on{ord}(q_i)-1, \text{ if } w(\alpha_i)\in \Lambda^{\on{neg}}.
\end{cases}
$$

\ssec{Metaplectic Arkhipov-Bezrukavnikov theory}  \label{ss:AB}

In this section we will make preparations for another proof of \thmref{t:finiteness of baby Verma} by
introducing a metaplectic analog of (some aspects of) the theory developed in \cite{AB}. 

\sssec{}

Consider the following example of a category equipped with an action of $\Rep(H)\otimes \Rep(T_H)$:
$$\QCoh(H\backslash (\ol{H/N_H})/T_H).$$

The self-duality of $\QCoh(H\backslash (\ol{H/N_H})/T_H)$ as a module over $\Rep(H)\otimes \Rep(T_H)$ implies
that for any $\Rep(H)\otimes \Rep(T_H)$-module category $\bC$, we have:
\begin{equation} \label{e:univ base affine space}
\on{DrPl}(\bC)\simeq \on{Funct}_{\Rep(H)\otimes \Rep(T_H)}(\QCoh(H\backslash (\ol{H/N_H})/T_H),\bC).
\end{equation}

In particular, taking the identity functor on $\QCoh(H\backslash (\ol{H/N_H})/T_H)$, we obtain that the object
$$\CO_{H\backslash (\ol{H/N_H})/T_H}\in \QCoh(H\backslash (\ol{H/N_H})/T_H)$$
admits a canonical lift to an object
$$(\CO_{H\backslash (\ol{H/N_H})/T_H})^{\on{DrPl}}\in \on{DrPl}(\QCoh(H\backslash (\ol{H/N_H})/T_H)).$$

Under the equivalence \eqref{e:univ base affine space}, for a functor
$$\QCoh(H\backslash (\ol{H/N_H})/T_H)\to \bC$$ the corresponding object of $\on{DrPl}(\bC)$ is the image
of $(\CO_{H\backslash (\ol{H/N_H})/T_H})^{\on{DrPl}}$ under this functor. 

\sssec{}

Next consider 
$$\Rep(B_H)\simeq \QCoh(H\backslash (H/N_H)/T_H)$$
as a category equipped with an action of $\Rep(H)\otimes \Rep(T_H)$.  

\medskip

The self-duality of $\Rep(B_H)$ as a module over $\Rep(H)\otimes \Rep(T_H)$ implies
that for any $\bC$, we have:
\begin{equation} \label{e:univ G/B}
\BHecke_{\on{rel}}(\bC)\simeq \on{Funct}_{\Rep(H)\otimes \Rep(T_H)}(\Rep(B_H),\bC).
\end{equation}

In particular, the object $\sfe\in \Rep(B_H)$ admits a canonical lift to an object
$$\sfe^{\BHecke_{\on{rel}}}\in \BHecke_{\on{rel}}(\Rep(B_H)),$$
which corresponds to the identity functor on $\Rep(B_H)$. 

\medskip

Under the equivalence \eqref{e:univ G/B}, for a functor $\Rep(B_H)\to \bC$, the corresponding object of 
$\BHecke_{\on{rel}}(\bC)$ is the image of $\sfe^{\BHecke_{\on{rel}}}$ under the above functor.

\sssec{}

One can describe the object $\sfe^{\BHecke_{\on{rel}}}$ explicitly. Namely, we identify
$$\BHecke_{\on{rel}}(\Rep(B_H))\simeq \QCoh((H/N_H)/\on{Ad}(B_H)),$$
and $\sfe^{\BHecke_{\on{rel}}}$ is the image of the structure sheaf along the following composition of closed embeddings
$$\on{pt}/B_H\to T_H/\on{Ad}(B_H)\hookrightarrow (H/N_H)/\on{Ad}(B_H).$$

In particular, the object
$$\sfe^{\BHecke}:=\oblv_{\on{rel}}(\sfe^{\BHecke_{\on{rel}}})\in \BHecke(\Rep(B_H))$$
with respect to the identification
$$\BHecke(\Rep(B_H))\simeq \QCoh(B_H\backslash H/B_H)$$
is the image of the structure sheaf along the map
$$\sfi:\on{pt}/B_H\simeq B_H\backslash B_H/B_H \hookrightarrow B_H\backslash H/B_H.$$

\medskip

From here it follows that given an object $\bc\in \BHecke_{\on{rel}}(\bC)$, the resulting functor $\Rep(B_H)\to \bC$ sends 
\begin{equation} \label{e:coInd e}
\sfe\in \Rep(B_H)  \rightsquigarrow \coInd^{H}_{B_H}(\oblv_{\on{rel}}(\bc)).
\end{equation} 

\sssec{}

Under the identifications \eqref{e:univ base affine space} and \eqref{e:univ G/B}, 
the functor $\sfj^*:\on{DrPl}(\bC)\to \BHecke_{\on{rel}}(\bC)$ corresponds to precomposition with
$$\sfj_*:\QCoh(H\backslash (H/N_H)/T_H)\to \QCoh(H\backslash (\ol{H/N_H})/T_H).$$

\sssec{}

Take $\bC=\Shv_{\CG^G}(\Gr_G)^{I,\on{ren}}$ and 
$$(\delta_{1,\Gr})^{\on{DrPl},I}\in \on{DrPl}(\Shv_{\CG^G}(\Gr_G)^{I,\on{ren}}).$$

Consider the corresponding functor, denoted 
\begin{equation} \label{e:from base affine flags to Fl}
\wt{\on{AB}}:\QCoh(H\backslash (\ol{H/N_H})/T_H)\to \Shv_{\CG^G}(\Gr_G)^{I,\on{ren}}.
\end{equation} 

By the above, we obtain that
$$\wt\CF^\semiinf_{\on{rel}}\simeq \wt{\on{AB}}\circ \sfj_*(\sfe^{\BHecke_{\on{rel}}}),$$
and hence
\begin{equation} \label{e:F as image}
\wt\CF^\semiinf\simeq \wt{\on{AB}}\circ \sfj_*((\sfi)_*(\sfe)).
\end{equation} 

\sssec{}

The following is the metaplectic analog of the result of \cite[Theorem 3.1.4]{AB}: 

\begin{thm} \label{t:AB}
The functor $\wt{\on{AB}}$ of \eqref{e:from base affine flags to Fl} factors through the localization
$$\sfj^*: \QCoh(H\backslash (\ol{H/N_H})/T_H) \to \QCoh(H\backslash (H/N_H)/T_H).$$
\end{thm}

\begin{rem}
Another way to formulate \thmref{t:AB} is that the functor \eqref{e:from base affine flags to Fl} 
sends $\on{ker}(\sfj^*)$ to zero.
\end{rem}

\sssec{}

\thmref{t:AB} can be proved by repeating verbatim the proof in \cite{AB} of the usual (i.e., nin-metaplectic)
version. Alternatively, in the next subsection we will see that \thmref{t:AB} is logically equivalent to 
\thmref{t:finiteness of baby Verma}. 

\ssec{Second proof of \thmref{t:finiteness of baby Verma}}   \label{ss:AB bis}

We will now show that Theorems \ref{t:AB} and \ref{t:finiteness of baby Verma} tautologically imply one another. 

\sssec{}   \label{sss:AB bis}

First, let us assume \thmref{t:AB}.

\medskip

Let 
\begin{equation} \label{e:from G/B to Fl}
\on{AB}:\Rep(B_H)\simeq \QCoh(H\backslash (H/N_H)/T_H)\to \Shv_{\CG^G}(\Gr_G)^{I,\on{ren}}
\end{equation} 
denote the resulting functor. Note, however, that since $\sfj^*\circ \sfj_*\simeq \on{Id}$, we have
\begin{equation} \label{e:AB tilde}
\on{AB}\simeq \wt{\on{AB}}\circ \sfj_*.
\end{equation} 

\sssec{}

Point (a) of \thmref{t:finiteness of baby Verma}
is equivalent to the assertion that $\wt\CF^\semiinf\in \BHecke(\Shv_{\CG^G}(\Gr_G)^{I,\on{ren}})$ is compact. 

\medskip

The functor $\wt{\on{AB}}$ sends the generator of $\QCoh(H\backslash (\ol{H/N_H})/T_H)$, viewed as
a $\Rep(H)\otimes \Rep(T_H)$-module category, i.e., $\CO_{H\backslash (\ol{H/N_H})/T_H}$, to
a compact object of $\Shv_{\CG^G}(\Gr_G)^{I,\on{ren}}$, i.e., $\delta_{1,\Gr}$. Hence, $\wt{\on{AB}}$ sends
compacts to compacts. 

\medskip

The latter formally implies that the functor $\on{AB}$ also sends compacts to compacts. Hence, $\on{AB}$ admits 
a continuous right adjoint, which automatically also respects the action of $\Rep(H)\otimes \Rep(T_H)$.  The adjoint
pair
$$\on{AB}:\QCoh(H\backslash (H/N_H)/T_H)\rightleftarrows \Shv_{\CG^G}(\Gr_G)^{I,\on{ren}}:\on{AB}^R$$
induces the (same-named) adjoint pair
$$\BHecke(\QCoh(H\backslash (H/N_H)/T_H))\rightleftarrows \BHecke(\Shv_{\CG^G}(\Gr_G)^{I,\on{ren}}).$$

\medskip

In particular, the functor 
$$\on{AB}:\BHecke(\Rep(B_H))\to \BHecke(\Shv_{\CG^G}(\Gr_G)^{I,\on{ren}})$$
admits a continuous right adjoint, and hence sends compacts to compacts. In particular,
the image of $(\sfi)_*(\sfe)$ is compact. However,
$$\on{AB}\circ (\sfi)_*(\sfe)\overset{\text{\eqref{e:AB tilde}}}\simeq
\wt{\on{AB}}\circ \sfj_*\circ (\sfi)_*(\sfe)\overset{\text{\eqref{e:F as image}}} \simeq
\wt\CF^\semiinf,$$
whence the latter is compact, as required.  

\sssec{}

Point (b) of \thmref{t:finiteness of baby Verma} is equivalent to the assertion that the maps
$$J_\gamma\underset{I}\star \delta_{1,\Gr} \to \coInd^H_{B_H}(\wt\CF^\semiinf)$$
that arise from \eqref{e:to glob sect} are isomorphisms.

\medskip

To prove this, by applying the functor $\on{AB}$, it suffices to show that the corresponding maps
$$\sfe^\gamma \to \coInd^H_{B_H}((\sfi)_*(\sfe)\otimes \sfe^\gamma)$$
are isomorphisms in $\Rep(B_H)$. Note that 
$$\coInd^H_{B_H}:\BHecke(\Rep(B_H))\to \Rep(B_H)$$
identifies with the direct image functor along 
$$B_H\backslash H/B_H \to \on{pt}/B_H.$$

This makes the assertion obvious. 

\sssec{}

Vice versa, let us assume \thmref{t:finiteness of baby Verma}(b) and deduce \thmref{t:AB}. We need to show that the
natural transformation
\begin{equation} \label{e:tilde AB}
\wt{\on{AB}}\to \wt{\on{AB}}\circ \sfj_*\circ \sfj^*,
\end{equation} 
induced by the unit of the adjunction $\on{Id}\to \sfj_*\circ \sfj^*$ is an isomorphism.

\medskip

Since both functors respect the action of $\Rep(H)\otimes \Rep(T_H)$, it suffices to show that the natural transformation
\eqref{e:tilde AB} induces an isomorphism
$$\wt{\on{AB}}(\CO_{H\backslash (\ol{H/N_H})/T_H})\to 
\wt{\on{AB}}\circ \sfj_*\circ \sfj^*(\CO_{H\backslash (\ol{H/N_H})/T_H}).$$

However, it is easy to see that we have a commutative diagram
$$
\CD
\wt{\on{AB}}(\CO_{H\backslash (\ol{H/N_H})/T_H}) @>>>  \wt{\on{AB}}\circ \sfj_*\circ \sfj^*(\CO_{H\backslash (\ol{H/N_H})/T_H})  \\
& &  @VV{\sim}V   \\
@V{\sim}VV \wt{\on{AB}}\circ \sfj_*(\CO_{H\backslash (H/N_H)/T_H})  \\
& &  @V{\sim}V{\text{\eqref{e:F as image} and \eqref{e:coInd e}}}V   \\
\delta_{1,\Gr}  @>{\text{\eqref{e:coind baby Verma}}}>> 
 \coInd^{H}_{B_H}(\wt\CF^\semiinf).
\endCD
$$

Now, the bottom arrow is the above diagram is an isomorphism by \corref{c:coind of baby Verma zero}. 

\section{Baby Verma objects in the Whittaker category}  \label{s:baby Verma in Whit}

In this section we will realize a part of the program indicated in \secref{sss:strategy}: we will
construct the objects $\bCM^{\mu,*}_{\Whit}$ and $\bCM^{\mu,!}_{\Whit}$ and verify proprties (i) and (ii). 

\ssec{Construction of dual baby Verma objects in the Whittaker category}

In this subsection we will construct the objects $\bCM^{\mu,*}_{\Whit}$. 

\sssec{}

Consider the category
$$\Shv_{\CG^G}(\Fl_G)^{\fL(N),\chi_N}\subset \Shv_{\CG^G}(\Fl_G),$$
defined in the same way as 
$$\Whit_q(G)=\Shv_{\CG^G}(\Gr_G)^{\fL(N),\chi_N},$$
but with $\Gr_G$ replaced by $\Fl_G$.

\sssec{}

Note now that we have a well-defined convolution functor 
\begin{equation} \label{e:Fl conv}
\Shv_{\CG^G}(\Fl_G)\otimes \Shv_{\CG^G}(\Gr_G)^{I,\on{ren}}\to \Shv_{\CG^G}(\Gr_G),
\quad \CF_1,\CF_2\to \CF_1 \underset{I}\star \CF_2,
 \end{equation}
which respects the action of $\Sph_{q,x}(G)$ on the right, and the convolution action of $\Shv_{\CG^G}(\fL(G))$
on the left.

\medskip

In particular, the above functor induces a functor
\begin{equation} \label{e:Whit conv}
\Shv_{\CG^G}(\Fl_G)^{\fL(N),\chi_N} \otimes \Shv_{\CG^G}(\Gr_G)^{I,\on{ren}} \to \Shv_{\CG^G}(\Gr_G)^{\fL(N),\chi_N} 
=:\Whit_q(G),
\end{equation}
which respects the action of $\Sph_{q,x}(G)$ on the right. 

\medskip

We notice:

\begin{lem}  \label{l:conv with unit}
The functor 
$$-\underset{I}\star \delta_{1,\Gr}:\Shv_{\CG^G}(\Fl_G)\to \Shv_{\CG^G}(\Gr_G)$$
identifies with direct image along $\Fl_G\to \Gr_G$.
\end{lem} 

\sssec{}

For $\lambda\in \Lambda^+$, 
denote $$S^\lambda_{\Fl}=\fL(N)\cdot t^\lambda\cdot I/I\subset \Fl_G.$$

As in the case of the affine Grassmannian, the functor of taking the fiber at $t^\lambda \in \Fl_G$ defines
an equivalence
\begin{equation} \label{e:N orbit in Fl}
\Shv_{\CG^G}(S^\lambda_{\Fl})^{\fL(N),\chi_N}\to \Vect.
\end{equation}

Let 
$$W^{\lambda,*}_{\Fl}\in \Shv_{\CG^G}(\Fl_G)^{\fL(N),\chi_N}$$
be the *-extension of the image of $\sfe[-\langle \lambda,2\check\rho\rangle]\in \Vect$ under the equivalence \eqref{e:N orbit in Fl}. 

\sssec{}

Denote
$$\wt\CM^{\lambda,*}_{\Whit}:=W^{\lambda,*}_{\Fl}\underset{I}\star \wt\CF^{\semiinf,-}\in
\BmHecke(\Whit_q(G)).$$

Denote
$$\bCM{}^{\lambda,*}_{\Whit}:=\Res^{B^-_H}_{T_H}(\CM^\lambda_{\Whit})\in \bHecke(\Whit_q(G)).$$

\sssec{}

We observe:

\begin{lem}   \label{l:baby Verma shift}
For $\gamma\in \Lambda^+_H$ we have an isomorphism
$$\wt\CM^{\lambda,*}_{\Whit} \otimes \sfe^\gamma \simeq \wt\CM^{\lambda+\gamma,*}_{\Whit}.$$ 
\end{lem}

\begin{proof}

By the construction of $\wt\CF^\semiinf$, we have
$$J_\gamma \underset{I}\star \wt\CF^\semiinf\simeq \wt\CF^\semiinf \otimes \sfe^\gamma.$$

Hence,
$$j_{w_0,*} \underset{I}\star J_\gamma \underset{I}\star j_{w_0,!} \underset{I}\star \wt\CF^{\semiinf,-} 
\simeq \wt\CF^{\semiinf,-} \otimes \sfe^{w_0(\gamma)}.$$

Hence, for $\gamma\in \Lambda^+_H$, we obtain 
$$j_{\gamma,*}\underset{I}\star \wt\CF^{\semiinf,-}  \simeq \wt\CF^{\semiinf,-} \otimes \sfe^\gamma.$$

From here we obtain that $\wt\CM^{\lambda,*}_{\Whit} \otimes \sfe^\gamma$ identifies with
$$W^{\lambda,*}_{\Fl} \underset{I}\star  j_{\gamma,*} \underset{I}\star \wt\CF^{\semiinf,-}.$$

Finally, we notice that for $\gamma\in \Lambda^+_H$, we have:
\begin{equation} \label{e:shift Verma on flags}
W^{\lambda,*}_{\Fl} \underset{I}\star  j_{\gamma,*}\simeq W^{\lambda+\gamma,*}_{\Fl},
\end{equation}
which implies the assertion of the lemma.

\end{proof}

Let now $\mu$ be an arbitrary element of $\Lambda$. Write 
\begin{equation} \label{e:mu and lambda}
\mu=\lambda-\gamma, \quad \lambda\in \Lambda^+_H.
\end{equation}

Define:
$$\wt\CM^{\mu,*}_{\Whit}:=\wt\CM^{\lambda,*}_{\Whit} \otimes \sfe^{-\gamma}.$$

Note that \lemref{l:baby Verma shift} implies that this definition is independent of the choice
of a presentation of $\mu$ as in \secref{e:mu and lambda}. Moreover, for any $\mu$ and $\gamma$ we have:
\begin{equation} \label{e:baby Verma shift}
\wt\CM^{\mu+\gamma,*}_{\Whit}\simeq \wt\CM^{\mu,*}_{\Whit}\otimes \sfe^\gamma.
\end{equation}

Define:
$$\bCM{}^{\mu,*}_{\Whit}:=\Res^{B^-_H}_{T_H}(\wt\CM^\mu_{\Whit})\in \bHecke(\Whit_q(G)).$$

We also have:
$$\bCM{}^{\mu+\gamma,*}_{\Whit}\simeq \bCM{}^{\mu,*}_{\Whit} \otimes \sfe^\gamma.$$

\sssec{}

For completeness, define 
Define:
$$\CM{}^{\mu,*}_{\Whit}:=\Res^{T_H}(\bCM^\mu_{\Whit})\in \Hecke(\Whit_q(G)).$$

However, these objects will not be used in this paper. 

\ssec{Properties of dual baby Verma objects in the Whittaker category}

In this subsection we investigate some basic properties of the objects $\wt\CM^{\mu,*}_{\Whit}$
constructed in the previous subsection. 

\sssec{}

First, we claim: 

\begin{lem} \label{l:baby Verma Whit expl}
The object 
$$\oblv_{\bHecke}(\overset{\bullet}\CM{}^{\mu,*}_{\Whit})\in \Rep(T_H)\otimes \Whit_q(G)$$
is given by 
$$\gamma'\rightsquigarrow \underset{\gamma}
{\on{colim}}\,  W^{\mu+\gamma+\gamma',*}\underset{\fL^+(G)}\star \IC_{q,\ol\Gr^{-w_0(\gamma)}_G},$$
where the colimit runs over the set $\gamma\in (-\mu-\gamma'+\Lambda^+_H)\cap \Lambda^+_H$. 
\end{lem}

\begin{proof}

By \eqref{e:baby Verma shift}, with no restriction of generality we can assume that $\mu=\lambda\in \Lambda^+$.

\medskip

By \eqref{e:key object as colimit}, the object 
$$\oblv_{\bHecke}(\overset{\bullet}\CM{}^{\lambda,*}_{\Whit})$$
is given by
$$\gamma'\rightsquigarrow \underset{\gamma}
{\on{colim}}\,  W^{\lambda,*}_{\Fl} \underset{I}\star 
j_{\gamma+\gamma',*} \underset{I}\star \IC_{q,\ol\Gr^{-w_0(\gamma)}_G},$$
where the colimit runs over the set $\gamma\in (-\gamma'+\Lambda^+_H)\cap \Lambda^+_H$. 

\medskip

Using \eqref{e:shift Verma on flags}, we have:
$$W^{\lambda,*}_{\Fl} \underset{I}\star 
j_{\gamma+\gamma',*} \underset{I}\star \IC_{q,\ol\Gr^{-w_0(\gamma)}_G}\simeq
W^{\lambda+\gamma+\gamma',*}_{\Fl} \underset{I}\star \delta_{1,\Gr} \underset{\fL^+(G)}\star \IC_{q,\ol\Gr^{-w_0(\gamma)}_G}.$$

\medskip

Finally, we notice that by \lemref{l:conv with unit} for any $\lambda'\in \Lambda^+$
$$W^{\lambda',*}_{\Fl} \underset{I}\star \delta_{1,\Gr} \simeq W^{\lambda',*}.$$

\end{proof}

\begin{cor}  \label{c:baby Whit in heart}
The objects $\wt\CM^{\mu,*}_{\Whit}$ belong to $(\BmHecke(\Whit_q(G)))^\heartsuit$.
\end{cor}

\begin{proof}

It is sufficient to show that
$$\oblv_{\bHecke}(\bCM^{\lambda,*}_{\Whit})\in \Whit_q(G)$$
belongs to $(\Whit_q(G))^\heartsuit$ for $\lambda\in \Lambda^+$. 

\medskip

Since the poset $\Lambda^+_H$ is filtered, it suffices to show that each term in the colimit
in \lemref{l:baby Verma Whit expl} belongs to $(\Whit_q(G))^\heartsuit$.
Now the assertion follows from \propref{p:conv is exact}.

\end{proof}

Next, we claim: 

\begin{lem}  \label{l:coind baby Whit}
For $\lambda\in \Lambda^+$, we have
$$\coInd^{H}_{B^-_H}(\wt\CM^{\lambda,*}_{\Whit})\simeq W^{\lambda,*}.$$
\end{lem}

\begin{proof}

Follows immediately from \corref{c:coind of baby Verma zero} using \lemref{l:conv with unit}. 

\end{proof}

\sssec{}

For $\lambda\in \Lambda^+$ consider now the map
\begin{equation} \label{e:from Weyl to baby Whit}
\Res^H_{B^-_H}(W^{\lambda,*})\to \wt\CM^{\lambda,*}_{\Whit}
\end{equation} 
arising by adjunction from the isomorphism of \lemref{l:coind baby Whit}. 

\medskip

\begin{prop}  \label{p:from Weyl to baby Whit}  
For $\lambda\in \Lambda^+$ and $\gamma\in \Lambda_H$ 
consider the map 
$$\Res^H_{B^-_H}(W^{\lambda+\gamma,*})\otimes \sfe^{-\gamma}\to \wt\CM^{\lambda+\gamma,*}_{\Whit}.$$
If $\gamma$ is deep enough in the dominant chamber, this map has the following properties: 

\smallskip

\noindent{\em(a)} 
It is surjective (in the abelian category $(\BmHecke(\Whit_q(G)))^\heartsuit$).

\smallskip

\noindent{\em(b)} 
Its kernel admits a finite left resolution each of whose terms admits a filtration with subquotients of the form 
$$\Res^H_{B^-_H}(W^{\lambda'+\gamma',*})\otimes \sfe^{-\gamma'}$$ for $\lambda'\in \lambda-(\Lambda^{\on{pos}}-0)$, 
$\gamma'\in \Lambda^+_H$ and $\lambda'+\gamma'\in \Lambda^+$. 

\end{prop}

\begin{proof}

Recall the functor
$$\on{AB}:\Rep(B_H) \to \Shv_{\CG^G}(\Gr_G)^{I,\on{ren}}$$
of \eqref{e:from G/B to Fl}. We will denote by the same character the resulting functor 
$$\QCoh(B_H\backslash H/B_H) \to \BHecke(\Shv_{\CG^G}(\Gr_G)^{I,\on{ren}}).$$

Recall that
$$\wt\CF^\semiinf\simeq \on{AB}((\sfi)_*(\sfe)).$$

We note now that for $\gamma\in \Lambda_H$ deep enough in the dominant chamber, the object 
$$(\sfi)_*(\sfe)\in \QCoh(B_H\backslash H/B_H)\simeq \Rep(B_H)\underset{\Rep(H)}\otimes \Rep(B_H)$$
admits a finite left resolution whose initial term is 
$$\sfe^{\gamma}\otimes \sfe^{-\gamma},$$
and each of the other terms admits a filtration with terms
$$\sfe^{\gamma'+\gamma_0}\otimes \sfe^{-\gamma'},\quad \gamma_0\in \Lambda_H^{\on{neg}}-0,\quad  
\gamma'+\gamma_0\in \Lambda^+_H.$$

\medskip

Applying to this resolution the functor $\on{AB}$ term-wise, then convolving with $j_{w_0,*}[-d]$ on the left and with
$W^{\lambda,*}_{\Fl}$ on the right, and finally applying the functor $$w_0:\BHecke(\Whit_q(G))\to \BmHecke(\Whit_q(G)),$$ 
we obtain that $\wt\CM^{\lambda+\gamma,*}_{\Whit}$ admits a left resolution of
the form specified in the proposition.

\end{proof} 

\ssec{Jordan-Holder series of dual baby Verma modules}

\sssec{}

First, we claim:

\begin{lem} \label{l:Baby Whit compact}
The objects $\wt\CM{}^{\mu,*}_{\Whit}\in \BmHecke(\Whit_q(G))$ are compact.
\end{lem} 

\begin{proof}

It is enough to show that the object $\wt\CM{}^{\mu+\gamma,*}_{\Whit}\in \BmHecke(\Whit_q(G))$ is compact for \emph{some} $\gamma$. 
By \propref{p:from Weyl to baby Whit}, we can choose $\gamma$ large enough so that 
$\wt\CM{}^{\mu+\gamma,*}_{\Whit}$ admits a surjection from $\Res^H_{B_H}(W^{\mu+\gamma,*})$. Hence, 
$\wt\CM{}^{\mu+\gamma,*}_{\Whit}$ is compact as the t-structure on $\BmHecke(\Whit_q(G))$ is Artinian. 

\end{proof} 

\begin{cor} \label{c:Baby Whit compact}
The objects $\bCM^{\mu,*}_{\Whit}\in \BmHecke(\Whit_q(G))$ are compact.
\end{cor}

\sssec{}

Since the t-structure on $\bHecke(\Whit_q(G))$, is Artinian, 
\corref{l:Baby Whit compact} implies that the objects $\bCM{}^{\mu,*}_{\Whit}\in (\bHecke(\Whit_q(G)))^\heartsuit$ have finite
length. 

\medskip

We claim:

\begin{prop} \label{p:Jordan-Holder}
There exists a non-zero map $\bCM{}^{\mu,!*}_{\Whit}\to \bCM{}^{\mu,*}_{\Whit}$, such that the Jordan-Holder constituents
of the quotient are of the form 
$$\bCM{}^{\mu',!*}_{\Whit}, \quad \mu'\in \mu-(\Lambda^{\on{pos}}-0).$$
\end{prop} 

\begin{proof} 

We can replace the initial $\mu$ by any $\mu+\gamma$ for $\gamma\in \Lambda_H$. Let $\gamma$ be as in 
\propref{p:from Weyl to baby Whit}. Applying this proposition, we can assume that there exists a surjection
\begin{equation} \label{e:ind Weyl to baby}
\Res^H_{T_H}(W^{\mu,*})\twoheadrightarrow \bCM{}^{\mu,*}_{\Whit},
\end{equation} 
whose kernel admits a left resolution each of which terms admits a filtration with subquotients
of the form 
\begin{equation} \label{e:constit kernel}
\Res^H_{T_H}(W^{\mu'+\gamma',*})\otimes \sfe^{-\gamma'},\quad \mu'\in \mu-(\Lambda^{\on{pos}}-0), \quad \gamma'\in \Lambda^+_H. 
\end{equation} 

\medskip

Recall the (injective) map
\begin{equation} \label{e:irr to ind Weyl}
\bCM{}^{\mu,!*}_{\Whit}\to \Res^H_{T_H}(W^{\mu,*})
\end{equation}
(see \corref{c:isotypics in Weyl}).  Composing, we obtain a map
$$\bCM{}^{\mu,!*}_{\Whit}\to \Res^H_{T_H}(W^{\mu,*})\to \bCM{}^{\mu,*}_{\Whit}.$$

We claim that this composite in non-zero. Indeed, if it were zero, the image of \eqref{e:irr to ind Weyl}
would hit the kernel of \eqref{e:ind Weyl to baby}. However, the Jordan-Holder constituents of the latter 
are among those of \eqref{e:constit kernel}, and by \corref{c:isotypics in Weyl} those are of the form
$$\bCM{}^{\mu',!*}_{\Whit},\quad \mu'\in \mu-(\Lambda^{\on{pos}}-0).$$

\medskip

Finally, the fact that the cokernel of the map $\bCM{}^{\mu,!*}_{\Whit}\to \bCM{}^{\mu,*}_{\Whit}$ has constituents 
of the form specified in the proposition follows by applying \corref{c:isotypics in Weyl} again. 

\end{proof}

\ssec{The (actual) baby Verma objects in the Whittaker category}

In this subsection we will finally define the objects $\bCM{}^{\mu,!}_{\Whit}\in \bHecke(\Whit_q(G))$. 

\sssec{}

Recall the duality operation of \eqref{e:duality B Hecke}:
$$(\BmHecke(\Whit_q(G))^c)^{\on{op}} \to \BHecke(\Whit_{q^{-1}}(G))^c.$$

Applying this functor to
$$\wt\CM^{\mu,*}_{\Whit}\in \BmHecke(\Whit_q(G))^c,$$
and up to replacing $q^{-1}$ by $q$, we obtain an object that we will denote
$$\wt\CM^{\mu,!}_{\Whit}\in \BHecke(\Whit_q(G))^c.$$

\sssec{}

Denote 
$$\bCM^{\mu,!}_{\Whit}:=\Res^{B_H}_{T_H}(\wt\CM^{\mu,!}_{\Whit})\in \bHecke(\Whit_q(G))^c.$$

By \eqref{e:duality B T Hecke}, the object $\bCM^{\mu,!}_{\Whit}$ is obtained from $\bCM^{\mu,*}_{\Whit}$
by the duality functor 
\begin{equation} \label{e:duality Hecke Whit}
(\bHecke(\Whit_q(G))^c)^{\on{op}} \to \bHecke(\Whit_{q^{-1}}(G))^c.
\end{equation} 

\sssec{}

One can deduce many of the properties of $\wt\CM^{\mu,!}_{\Whit}$ (resp., $\bCM^{\mu,!}_{\Whit}$) 
from those of $\wt\CM^{\mu,*}_{\Whit}$ (resp., $\bCM^{\mu,*}_{\Whit}$) by duality. We will need the following few:

\begin{cor}  \label{c:Jordan-Holder}
There exists a non-zero map $\bCM{}^{\mu,!}_{\Whit}\to \bCM{}^{\mu,!*}_{\Whit}$, such that the Jordan-Holder constituents
of the kernel are of the form 
$$\bCM{}^{\mu',!*}_{\Whit}, \quad \mu'\in \mu-(\Lambda^{\on{pos}}-0).$$
\end{cor} 

\begin{cor}  \label{c:from Weyl to baby Whit}  
For every $\mu\in \Lambda$ one can find $\gamma\in \Lambda_H$ sufficiently deep in the 
dominant chamber, such that there exists a map 
$$\wt\CM^{\mu+\gamma,!}_{\Whit}\to \Res^H_{B_H}(W^{\mu+\gamma,!})\otimes \sfe^{-\gamma}$$
with the the following properties: 

\smallskip

\noindent{\em(a)} 
It is injective.

\smallskip

\noindent{\em(b)} 
Its cokernel admits a finite right resolution each of whose terms admits a filtration with subquotients of the form 
$$\Res^H_{B_H}(W^{\mu'+\gamma',!})\otimes \sfe^{-\gamma'}$$ for $\mu'\in \mu-(\Lambda^{\on{pos}}-0)$, 
$\gamma'\in \Lambda^+_H$ and $\mu'+\gamma'\in \Lambda^+$. 

\end{cor}

\begin{cor} For $\lambda \in \Lambda^+$, we have
$$\Ind^{H}_{B_H}(\bCM{}^{\lambda,!}_{\Whit})\simeq W^{\lambda,!}.$$
\end{cor} 

\ssec{Orthogonality}

In this subsection we will show that the $\bCM{}^{\lambda,!}_{\Whit}$, 
$\bCM{}^{\lambda,*}_{\Whit}$ satisfy properties (i) and (ii) from \secref{sss:strategy}.

\sssec{}

First, by combining \propref{p:Jordan-Holder} and \corref{c:Jordan-Holder}, we obtain that there exist maps
$$\bCM{}^{\mu,!}_{\Whit}\twoheadrightarrow \bCM{}^{\mu,!*}_{\Whit}\hookrightarrow \bCM{}^{\mu,*}_{\Whit}$$
whose kernel/cokernel have Jordan-Holder constituents of the form 
$$\bCM{}^{\mu',!*}_{\Whit}, \quad \mu'\in \mu-(\Lambda^{\on{pos}}-0).$$

\medskip

We now claim:

\begin{thm}  \label{t:orthogonality} We have:
$$\CHom_{\bHecke(\Whit_q(G))}(\bCM{}^{\mu',!}_{\Whit},\bCM{}^{\mu,*}_{\Whit})=
\begin{cases}
&\sfe \text{ if } \mu'=\mu \\
&0 \text{ otherwise}.
\end{cases}$$
\end{thm}

The rest of this subsection is devoted to the proof of this theorem. 

\sssec{Reduction step 1}

First, we claim that it is sufficient to consider the case when
$$\mu'\notin \mu+(\Lambda^{\on{pos}}-0).$$

Indeed, $\mu'\in \mu+(\Lambda^{\on{pos}}-0)$, applying the duality functor \eqref{e:duality Hecke Whit},
we obtain:
$$\CHom_{\bHecke(\Whit_q(G))}(\bCM{}^{\mu',!}_{\Whit},\bCM{}^{\mu,*}_{\Whit})\simeq
\CHom_{\bHecke(\Whit_{q^{-1}}(G))}(\bCM{}^{\mu,!}_{\Whit},\bCM{}^{\mu',*}_{\Whit}),$$
and now $\mu\notin \mu'+(\Lambda^{\on{pos}}-0)$. 

\sssec{Reduction step 2}

Let us choose $\gamma\in \Lambda^+_H$ deep enough in the dominant chamber, so that $\bCM{}^{\mu'+\gamma,!}_{\Whit}$
admits a resolution as in \corref{c:from Weyl to baby Whit}. Replacing 
$$\mu'\rightsquigarrow \mu'+\gamma, \,\, \mu\rightsquigarrow \mu+\gamma,$$
we obtain that it is sufficient to show that for $\lambda'\in \Lambda^+$ and $\lambda'\notin \lambda+(\Lambda^{\on{pos}}-0)$, we have:
$$\CHom_{\bHecke(\Whit_q(G))}(\Res^H_{T_H}(W^{\lambda',!}),\bCM{}^{\lambda,*}_{\Whit})=
\begin{cases}
&\sfe \text{ if } \lambda'=\lambda \\
&0 \text{ otherwise}.
\end{cases}
$$

\sssec{}

Using \lemref{l:baby Verma Whit expl}, we rewrite
$$\CHom_{\bHecke(\Whit_q(G))}(\Res^H_{T_H}(W^{\lambda',!}),\bCM{}^{\lambda,*}_{\Whit})$$
as
$$\underset{\gamma\in \Lambda^+_H}{\on{colim}}\,\, 
\CHom_{\Whit_q(G)}(W^{\lambda',!},W^{\lambda+\gamma,*}\underset{\fL^+(G)}\star \IC_{q,\ol\Gr^{-w_0(\gamma)}_G}).$$

\medskip 

It is therefore sufficient to show that for $\lambda'\in \Lambda^+$ and $\lambda'\notin \lambda+(\Lambda^{\on{pos}}-0)$, we have:
\begin{equation} \label{e:orthogonality final}
\CHom_{\Whit_q(G)}(W^{\lambda',!},W^{\lambda+\gamma,*}\underset{\fL^+(G)}\star \IC_{q,\ol\Gr^{-w_0(\gamma)}_G})=
\begin{cases}
&\sfe \text{ if } \lambda'=\lambda \\
&0 \text{ otherwise}.
\end{cases}
\end{equation} 

\sssec{}

We rewrite 
$$\CHom_{\Whit_q(G)}(W^{\lambda',!},W^{\lambda+\gamma,*}\underset{\fL^+(G)}\star \IC_{q,\ol\Gr^{-w_0(\gamma)}_G})
\simeq
\CHom_{\Whit_q(G)}(W^{\lambda',!} \underset{\fL^+(G)}\star \IC_{q,\ol\Gr^\gamma_G},W^{\lambda+\gamma,*}).$$

Note that $$W^{\lambda',!} \underset{\fL^+(G)}\star \IC_{q,\ol\Gr^\gamma_G}$$ is supported on $\ol{S}^{\lambda'+\gamma}$,
and that its restriction to $S^{\lambda'+\gamma}$ is the generator of $\Whit_q(G)_{=\lambda'+\gamma}$. This implies 
\eqref{e:orthogonality final}:

\medskip

Indeed, the case $\lambda'=\lambda$ is obvious. For $\lambda'\neq\lambda$, the condition that 
$\lambda'\notin \lambda+(\Lambda^{\on{pos}}-0)$ implies that 
$$\ol{S}^{\lambda'+\gamma}\cap S^{\lambda+\gamma}=\emptyset.$$

\section{Calculation of stalks}  \label{s:stalks}

The goal of this section is to complete the program indicated in \secref{sss:strategy}
by proving property (iii) in {\it loc.cit.}

\ssec{Statement of the result}

In this subsection we will state precisely the calculation that we will perform. 

\sssec{}  \label{sss:trivialize gerbe lambda}

Fix an element $\mu \in \Lambda$. In order to simply the notation, we will trivialize the fiber
of $\CG^\Lambda$ at the point $\mu\cdot x\in \Conf_{\infty\cdot x}$. Note that this fiber identifies
canonically with the fiber of $\CG^G$ at $t^\mu\in \fL(G)^{\omega^\rho}_x$.

\medskip

Due to this trivialization, we have a \emph{well-defined} functor
\begin{equation} \label{e:fiber at lambda}
\Shv_{\CG^\Lambda}(\Conf_{\infty\cdot x})\to \Vect,
\end{equation}
given by taking the !-fiber at $\mu\cdot x$. In particular, we a \emph{well-defined} objects 
$$\CM^{\mu,!}_{\Conf},\CM^{\mu,*}_{\Conf}\in \Omega^{\on{small}}_q\on{-FactMod}.$$

\medskip

The trivialization of the fiber of $\CG^G$ at $t^\mu\in \fL(G)^{\omega^\rho}_x$ gives rise to
an identification
$$\Whit_q(G)_{=\mu}\simeq \Vect,$$
and hence to a pair of \emph{well-defined} objects
$$W^{\mu,!},W^{\mu,*}\in \Whit_q(G).$$

We normalize $\bCM^{\mu,*}_{\Whit}$ so that
$$\coind^H_{B^-_H(\bCM^{\mu,*}_{\Whit})}\simeq W^{\mu,*}.$$

\sssec{}

The main result of the present section is the following theorem:

\begin{thm} \label{t:fiber of costandard}
The functor 
$$\Phi^{\bHecke}_{\on{Fact}}:\bHecke(\Whit_q(G))\to \Omega_q^{\on{small}}\on{-FactMod}$$
sends $\bCM^{\mu,*}_{\Whit}$ to $\CM^{\mu,*}_{\Conf}\in \Omega^{\on{small}}_q\on{-FactMod}$.
\end{thm}

\sssec{}

Before we proceed further, let us show how \thmref{t:fiber of costandard} completes the outline
in \secref{sss:strategy}. In fact, it remains to prove the following:

\begin{cor} \label{c:fiber of standard}
The functor $\Phi^{\bHecke}_{\on{Fact}}$ sends 
$\bCM^{\mu,!}_{\Whit}$ to $\CM^{\mu,!}_{\Conf}\in \Omega^{\on{small}}_q\on{-FactMod}$.
\end{cor}

\begin{proof}

Note that if an object $\CF\in \Omega^{\on{small}}_q\on{-FactMod}$ is equipped with an isomorphism
$$\oblv_{\on{Fact}}(\CF)\simeq \oblv_{\on{Fact}}(\CM^{\mu,!}_{\Conf}),$$
then this isomorphism lifts uniquely to an isomorphism
$$\CF\simeq \CM^{\mu,!}_{\Conf}.$$

Hence, in order to prove the corollary, it suffices establish an isomorphism
\begin{equation} \label{e:standard to standard rough}
\Phi^{\bHecke}(\bCM^{\mu,!}_{\Whit})\simeq \oblv_{\on{Fact}}(\CM^{\mu,!}_{\Conf}).
\end{equation}

We start with the isomorphism 
$$\Phi^{\bHecke}(\bCM^{\mu,*}_{\Whit})\simeq \oblv_{\on{Fact}}(\CM^{\mu,*}_{\Conf}),$$
provided by \thmref{t:fiber of costandard}. Now \eqref{e:standard to standard rough} follows
by applying \thmref{t:duality pres Hecke}. 

\end{proof}

\sssec{}

The rest of this section is devoted to the proof of \thmref{t:fiber of costandard}. We will deduce it from 
the following result

\begin{thm}  \label{t:Hom from baby}
The functor
\begin{equation} \label{e:Phi mu}
\bHecke(\Whit_q(G)) \overset{\Phi^{\bHecke}}\longrightarrow \Shv_{\CG^\Lambda}(\Conf_{\infty\cdot x})
\overset{!\on{-fiber}\,\on{at}\,\mu\cdot x}\longrightarrow \Vect
\end{equation}
identifies canonically with 
$$\CHom_{\bHecke(\Whit_q(G))}(\bCM^{\mu,!}_{\Whit},-).$$
\end{thm}

Let us see how \thmref{t:Hom from baby} implies \thmref{t:fiber of costandard}:

\begin{proof}[Proof of \thmref{t:fiber of costandard}]

The object $\CM^{\mu,*}_{\Conf}\in \Omega^{\on{small}}_q\on{-FactMod}$ is uniquely characterized
by the property that the underlying object
$$\oblv_{\on{Fact}}(\CM^{\mu,*}_{\Conf})\in \Shv_{\CG^\Lambda}(\Conf_{\infty\cdot x})$$
has !-fiber $\sfe$ at $\mu\cdot x$ and has a zero !-fiber at $\mu'\cdot x$ for $\mu'\neq \mu$.

\medskip

Now the assertion follows from \thmref{t:orthogonality}. 

\end{proof}

\sssec{}

Thus, the rest of this section is devoted to the proof of \thmref{t:Hom from baby}. We should remark, however,
that the proof will essentially be a formal manipulation, given the relationship between $'\bICs_{q,x}$ and
$\overset{\bullet}\CF{}^{\semiinf}_{\on{rel}}$ explained in \secref{ss:F and semiinf}. 

\ssec{Framework for the proof of \thmref{t:Hom from baby}}

In this subsection we will explain a general categorical framework in which \thmref{t:Hom from baby} will be proved. 

\sssec{}  \label{sss:another pairing}

Let $\bC$ and $\bD$
be module categories over $\Rep(H)$, and let us be given a pairing 
$$\wt\Psi:\bC\underset{\Rep(H)}\otimes \bD\to \bE.$$

We claim that the datum of $\Psi$ gives rise to the datum of a pairing 
$$\overset{\bullet}\Psi: \bHecke(\bC)\otimes \bHecke(\bD)\to \bE.$$

Indeed, assume for simplicity that $\bD$ is compactly generated. Then we can interpret 
$\wt\Psi$ as a $\Rep(H)$-linear functor 
$$\bC\to \bD^\vee\otimes \bE.$$

The latter gives rise to a functor
$$\bHecke(\bC)\to \bHecke(\bD^\vee)\otimes \bE.$$

Now, identifying 
$$\bHecke(\bD^\vee)\simeq \bHecke(\bD)^\vee$$
as in \secref{sss:duality on Hecke graded}, we thus obtain a functor 
$$\bHecke(\bC)\to \bHecke(\bD)^\vee\otimes \bE,$$
hence the desired pairing $\overset{\bullet}\Psi$. 

\sssec{}

Let us now be in the context of \secref{sss:paradigm for Hecke}. It is easy to see that the functor
$$\bHecke(\bC) \otimes \bHecke_{\on{rel}}(\bD)\to \bE$$
constructed in {\it loc. cit.} identifies with
$$\bHecke(\bC) \otimes \bHecke_{\on{rel}}(\bD) \overset{\on{Id}\otimes \oblv_{\on{rel}}}\longrightarrow 
\bHecke(\bC) \otimes \bHecke(\bD) \overset{\overset{\bullet}\Psi}\longrightarrow \bE.$$

\ssec{Applying the framework: the left-hand side}

\sssec{}

Let $\Psi$ denote the pairing
\begin{multline} \label{e:our Psi again}
\Whit_q(G)\otimes \Shv_{(\CG^G)^{-1}}(\Gr_G)^{\fL^+(T)}\to
\Shv_{\CG^G}(\Gr_G)\otimes \Shv_{(\CG^G)^{-1}}(\Gr_G) \overset{-\sotimes-} \longrightarrow  \\
\to \Shv(\Gr_G)\overset{\Gamma(\Gr_G,-)}\longrightarrow \Vect.
\end{multline}

By construction, it factors via
$$\wt\Psi:\Whit_q(G)\underset{\Rep(H)}\otimes \Shv_{(\CG^G)^{-1}}(\Gr_G)^{\fL^+(T)} \to \Vect,$$
where $\Rep(H)$ acts on $\Shv_{(\CG^G)^{-1}}(\Gr_G)^{\fL^+(T)}$ by
$$V,\CF\mapsto \CF\underset{\fL^+(G)}\star \on{inv}^G(\Sat_{q,G})(V)\simeq
\CF\underset{\fL^+(G)}\star \Sat_{q^{-1},G}(\tau^H(V)).$$

\sssec{}

Consider the corresponding functor
$$\overset{\bullet}\Psi: \bHecke(\Whit_q(G))\otimes \bHecke(\Shv_{(\CG^G)^{-1}}(\Gr_G)^{\fL^+(T)})\to \Vect.$$

\medskip

Recall the object
$$'\bICs\in  \bHecke_{\on{rel}}(\Shv_{\CG^G}(\Gr_G)^{\fL^+(T)}),$$
and consider the corresponding object.
$$'\bICsm\in  \bHecke_{\on{rel}}(\Shv_{(\CG^G)^{-1}}(\Gr_G)^{\fL^+(T)}).$$

Due to the trivialization in \secref{sss:trivialize gerbe lambda}, the translate 
$$t^\mu \cdot {}'\bICsm$$
makes sense as on object of $\bHecke_{\on{rel}}(\Shv_{(\CG^G)^{-1}}(\Gr_G)^{\fL^+(T)})$. 

\medskip

Unwinding the definitions and using \secref{sss:another pairing} above, we obtain that the functor 
\eqref{e:Phi mu} identifies with the functor
\begin{equation} \label{e:Phi mu 1}
\overset{\bullet}\Psi(-,\oblv_{\on{rel}}(t^\mu \cdot {}'\bICsm))[\langle \mu,2\check\rho\rangle].
\end{equation}

\sssec{}

Note that since objects of $\Whit_q(G)$ are $\fL^+(N)$-equivariant, the pairing $\Psi$ of \eqref{e:our Psi again} is isomorphic to
its precomposition with the endo-functor 
$$\Whit_q(G)\otimes \Shv_{(\CG^G)^{-1}}(\Gr_G)^{\fL^+(T)}\overset{\on{Id}\otimes \on{Av}_*^{\fL^+(N)}}\longrightarrow 
\Whit_q(G)\otimes \Shv_{(\CG^G)^{-1}}(\Gr_G)^{\fL^+(T)}.$$

\medskip

Hence, the same is true for the functors $\wt\Psi$ and $\overset{\bullet}\Psi$. 
Therefore, we can rewrite the functor in \eqref{e:Phi mu 1} as
\begin{equation} \label{e:Phi mu 2}
\overset{\bullet}\Psi(-,\on{Av}_*^{\fL^+(N)}(\oblv_{\on{rel}}(t^\mu \cdot {}'\bICsm)))[\langle \mu,2\check\rho\rangle].
\end{equation}

\ssec{Applying the framework: the right-hand side}

\sssec{}

By \secref{sss:another pairing}, the pairing
$$\Upsilon:\Whit_q(G)\otimes \Whit_{q^{-1}}(G)\to \Vect$$
arising from \eqref{e:self duality for Whit}, gives rise to a pairing
$$\overset{\bullet}\Upsilon:\bHecke(\Whit_q(G)) \otimes \bHecke(\Whit_{q^{-1}}(G))\to \Vect.$$

\medskip 

By definition, the functor 
$$\CHom_{\bHecke(\Whit_q(G))}(\bCM^{\mu,!}_{\Whit},-):\bHecke(\Whit_q(G))\to \Vect$$
is given by
$$\overset{\bullet}\Upsilon(-,\bCM^{\mu,*}_{\Whit}),$$

\sssec{}

Unwinding the definitions of $\bCM^{\mu,*}_{\Whit}$ and of $\Upsilon$ we obtain:

\begin{lem}
The pairings 
$$\Whit_q(G)\otimes \Shv_{(\CG^G)^{-1}}(\Gr_G)^{I,\on{ren}}\to \Vect$$
given by 
$$\CF,\CF'\mapsto \Psi(\CF,t^\mu\cdot \CF')[\langle \mu,2\check\rho\rangle]$$
and 
$$\CF,\CF'\mapsto \Upsilon(\CF,W^{\mu,*}_{\Fl}\underset{I}\star \CF')$$
are canonically isomorphic.
\end{lem}

\begin{cor}
For $\mu\in \Lambda$, the pairings 
$$\bHecke(\Whit_q(G))\otimes \bHecke(\Shv_{(\CG^G)^{-1}}(\Gr_G)^{I,\on{ren}})\to \Vect$$
given by 
$$\CF,\CF'\mapsto \overset{\bullet}\Psi(\CF,t^\mu\cdot \CF')[\langle \mu,2\check\rho\rangle]$$
and 
$$\CF,\CF'\mapsto \overset{\bullet}\Upsilon(\CF,W^{\mu,*}_{\Fl}\underset{I}\star \CF')$$
are canonically isomorphic.
\end{cor}

\sssec{}

Thus, we obtain that the functor $\CHom_{\bHecke(\Whit_q(G))}(\bCM^{\mu,!}_{\Whit},-)$ can be identified
with
\begin{equation} \label{e:Phi mu 3}
\overset{\bullet}\Psi(-,t^\mu\cdot \overset{\bullet}\CF{}^{\semiinf,-})[\langle \mu,2\check\rho\rangle].
\end{equation}

Using once again $\fL^+(N)$-equivariance of objects of $\Whit_q(G)$, we obtain that the latter expression cab be rewritten as
\begin{equation} \label{e:Phi mu 4}
\overset{\bullet}\Psi(-,\on{Av}_*^{\fL^+(N)}(t^\mu\cdot \overset{\bullet}\CF{}^{\semiinf,-}))[\langle \mu,2\check\rho\rangle].
\end{equation}

\ssec{Conclusion of proof of \thmref{t:Hom from baby}}

\sssec{}

Comparing \eqref{e:Phi mu 2} and \eqref{e:Phi mu 4}, we obtain that \thmref{t:Hom from baby} follows from the next assertion: 

\begin{prop}  \label{p:IC vs F again}
For $\mu\in \Lambda^+$, there exists a canonical isomorphism in $\bHecke(\Shv_{\CG^G}(\Gr_G))$:
$$\on{Av}_*^{\fL^+(N)}(t^\mu\cdot \overset{\bullet}\CF{}^{\semiinf,-})\simeq
\on{Av}_*^{\fL^+(N)}(\oblv_{\on{rel}}(t^\mu \cdot {}'\bICsm))).$$
\end{prop}

The rest of this subsection is devoted to the proof of \propref{p:IC vs F again}.

\sssec{}

Recall (see \secref{sss:semiinf Iw}) that the functor $\on{Av}^{\fL(N)}_!$ defines an equivalence 
$$\Shv_{\CG^G}(\Gr_G)^I\to \Shv_{\CG^G}(\Gr_G)^{\fL(N)\cdot \fL^+(T)}.$$

Note that we have a commutative diagram of functors
\begin{equation} \label{e:favorite diagram}
\CD
\Shv_{\CG^G}(\Gr_G)^{\fL(N)\cdot \fL^+(T)}   @>{w_0\cdot }>> \Shv_{\CG^G}(\Gr_G)^{\fL(N^-)\cdot \fL^+(T)}  \\
@A{\on{Av}^{\fL(N)}_!}AA   @AA{\on{Av}^{\fL(N^-)}_!}A   \\
\Shv_{\CG^G}(\Gr_G)^I   @>{j_{w_0,!}\underset{I}\star-[d]}>>  \Shv_{\CG^G}(\Gr_G)^I.
\endCD
\end{equation} 

In particular, we obtain that the functor 
\begin{equation} \label{e:Av N-}
\on{Av}^{\fL(N^-)}_!:\Shv_{\CG^G}(\Gr_G)^I\to \Shv_{\CG^G}(\Gr_G)^{\fL(N^-)\cdot \fL^+(T)}
\end{equation}
is also an equivalence. 

\sssec{}

It follows formally that the right adjoint of the functor \eqref{e:Av N-} is given by $\on{Av}^{\overset{\circ}{I}}_*$,
where $\overset{\circ}{I}$ is the unipotent radical of $I$. Since \eqref{e:Av N-} is an equivalence, we obtain that 
$\on{Av}^{\overset{\circ}{I}}_*$ defines an inverse equivalence. 

\medskip

From the decomposition
$$\overset{\circ}{I}=\overset{\circ}{I}{}^+\cdot \overset{\circ}{I}{}^0\cdot \overset{\circ}{I}{}^-,$$
where
$$\overset{\circ}{I}{}^+:=\overset{\circ}{I}\cap \fL(N)=\fL^+(N),\,\, \overset{\circ}{I}{}^0:=
\overset{\circ}{I}\cap \fL(T),\,\, \overset{\circ}{I}{}^:=\overset{\circ}{I}\cap \fL(N^-),$$
it follows that when applied to objects of $\Shv_{\CG^G}(\Gr_G)^{\fL(N)\cdot \fL^+(T)}$, the natural transformation
$$\on{Av}^{\overset{\circ}{I}}_*\to \on{Av}^{\fL^+(N)}_*$$
is an isomorphism.

\sssec{}

It follows from \propref{p:F and semiinf} and the commutative diagram \eqref{e:favorite diagram} that
$$\on{Av}^{\fL(N^-)}_!(\overset{\bullet}\CF{}^{\semiinf,-})\simeq \oblv_{\on{rel}}({}'\bICsm).$$

Hence, we obtain an isomorphism
\begin{equation} \label{e:IC vs F again}
\on{Av}_*^{\fL^+(N)}(\oblv_{\on{rel}}({}'\bICsm)))\simeq  \overset{\bullet}\CF{}^{\semiinf,-}
\end{equation}
as objects of $\bHecke(\Shv_{\CG^G}(\Gr_G))$. 

\medskip

We claim that this implies the isomorphism stated in \propref{p:IC vs F again}. 

\sssec{}

Indeed, since $\mu$ is dominant, we have 
$$\fL^+(N)\subset \on{Ad}_{t^{-\mu}}(\fL^+(N)).$$

Hence, \eqref{e:IC vs F again} implies 
$$\on{Av}_*^{\on{Ad}_{t^{-\mu}}(\fL^+(N))}(\oblv_{\on{rel}}({}'\bICsm)))\simeq
\on{Av}_*^{\on{Ad}_{t^{-\mu}}(\fL^+(N))}(\overset{\bullet}\CF{}^{\semiinf,-}).$$

Translating both sides by $t^\mu$ we arrive at the isomorphism of \propref{p:IC vs F again}. 

\qed

\newpage 

\centerline{\bf Part IX: Relation to quantum groups} 

\bigskip

This Part is disjoint from the rest of this work. Here we will establish an equivalence between the category 
$\Omega^{\on{small}}_q\on{-FactMod}^{\on{ren}}$ and a (renormalized) version of the category of modules
over the small quantum group. 

\medskip

The contents of this Part can be thought of recasting the modern language the equivalence of categories,
which was the subject of the work \cite{BFS}. 

\section{Modules over the small quantum group}  \label{s:quantum}

In this section we define the category of modules over the small quantum group, as a $\sfe$-linear category. 
It will be introduced as \emph{relative Drinfeld center} of the category of modules over the \emph{positive part},
denoted $\fu_q(\cN)$. 

\ssec{The small quantum group: the positive part}

In this subsection we will introduce the positive part of the small quantum group, $\fu_q(\cN)$.

\sssec{}  \label{sss:b'}

We start with a datum of a bilinear form 
$$b':\Lambda\otimes \Lambda\to \sfe^{\times,\on{tors}}.$$

Let $q$ be the associated quadratic form $\Lambda\to \sfe^{\times,\on{tors}}$. 
We will assume that $q$ is \emph{non-degenerate}, i.e., $q(\alpha)\neq 1$ for all coroots $\alpha$. 

\medskip

We will also assume that the quadratic form $q$ belongs to the subset
$$\on{Quad}(\Lambda,\sfe^{\times,\on{tors}})^W_{\on{restr}}\subset 
\on{Quad}(\Lambda,\sfe^{\times,\on{tors}}),$$
introduced in \cite[Sect. 3.2.2]{GLys}. Namely, this is the subset consisting of $W$-invariant quadratic forms
that satisfy the additional condition that for any coroot $\alpha$ and any $\lambda\in \Lambda$, we have
\begin{equation} \label{e:restr cond}
b(\alpha,\lambda)=q(\alpha)^{\langle \lambda,\check\alpha\rangle}.
\end{equation} 

\sssec{}

Starting from this data, we will eventually define the category of modules over the \emph{Langlands dual} small quantum group,
to be denoted $\overset{\bullet}\fu_q(\cG)\mod$. 

\begin{rem}
The dot $\bullet$ over $\fu_q$ is meant to emphasize that we will be dealing with the category of $\Lambda$-\emph{graded} modules.
\end{rem}

\begin{rem}
The fact that our quantum group corresponds to the Langlands dual group manifests itself in that its lattice of \emph{weights}
is the lattice $\Lambda$ of coweights of $G$.
\end{rem} 

\sssec{}

Consider the category $\Rep(\cT)\simeq \Vect^{\Lambda}$, where $\Vect^\Lambda$
is the category of $\Lambda$-graded vector spaces. For $\lambda\in \Lambda$ we let 
$$\sfe^\lambda\in \Vect^\lambda\subset \Vect^\Lambda$$
denote the vector space $\sfe$ placed in the graded component $\lambda$. 

\medskip

We consider $\Vect^\Lambda$ as endowed with the standard monoidal structure. 
Now, the data of $b'$ defines a new braiding on $\Vect^\Lambda$. Denote the resulting braided monoidal category $\Vect_q^\Lambda$. 

\sssec{}  \label{sss:e lines}

Choose \emph{some} 1-dimensional objects 
$$\sfe^{i,\on{quant}}\in \Vect^{\alpha_i}\subset \Vect^\Lambda.$$ 
I.e., each $\sfe^{i,\on{quant}}$ is \emph{non-canonically} isomorphic to $\sfe^{\alpha_i}$. 

\medskip

Let $U_q(\cN)^{\on{free}}$ be the free associative algebra in $\Vect^{\Lambda^{\on{pos}}}\subset \Vect^\Lambda$ on 
$$\underset{i}\oplus\, \sfe^{i,\on{quant}}\in \Vect^{\Lambda}.$$
I.e., to specify a map from
$U_q(\cN)^{\on{free}}$ to an associative algebra $A$ in $\Vect^\Lambda$ is equivalent to specifying maps
$$\sfe^{i,\on{quant}} \to A.$$ 
Let $e_i$ denote the tautological map $\sfe^{i,\on{quant}}\to U_q(\cN)^{\on{free}}$. 

\medskip

We endow $U_q(\cN)^{\on{free}}$ with a Hopf algebra structure in $\Vect_q^{\Lambda^{\on{pos}}}$ by letting the comultiplication 
$$U_q(\cN)^{\on{free}}\to U_q(\cN)^{\on{free}}\otimes U_q(\cN)^{\on{free}}$$
correspond to the maps
$$\sfe_{i,\on{quant}} \overset{e_i\otimes \on{unit}+\on{unit}\otimes e_i}\longrightarrow U_q(\cN)^{\on{free}}\otimes U_q(\cN)^{\on{free}}.$$

\sssec{}

Let $U_q(\cN)^{\on{co-free}}$ denote the co-free graded co-associative co-algebra in $\Vect^{\Lambda^{\on{pos}}}$
on the co-generators $\sfe^{i,\on{quant}}$. I.e., to specify a map from a co-associative co-algebra $A$ in $\Vect^{\Lambda^{\on{pos}}}$
to $U_q(\cN)^{\on{co-free}}$ is equivalent to specifying maps
$$A\to  \sfe^{i,\on{quant}}.$$ 
Let $e^*_i$ denote the tautological map $U_q(\cN)^{\on{co-free}}\to \sfe_{i,\on{quant}}$.

\medskip

We endow $U_q(\cN)^{\on{co-free}}$ with a Hopf algebra structure in $\Vect_q^\Lambda$ by letting the multiplication
$$U_q(\cN)^{\on{co-free}}\otimes U_q(\cN)^{\on{co-free}}\to U_q(\cN)^{\on{co-free}}$$
correspond to the maps 
$$U_q(\cN)^{\on{co-free}}\otimes U_q(\cN)^{\on{co-free}} \overset{e^*_i\otimes \on{aug}+\on{aug}\otimes e^*_i}\longrightarrow 
\sfe_{i,\on{quant}}.$$

\sssec{}

We define a map of Hopf algebras
\begin{equation} \label{e:free to cofree}
U_q(\cN)^{\on{free}}\to U_q(\cN)^{\on{co-free}}
\end{equation} 
to correspond to the projections 
$$U_q(\cN)^{\on{free}} \to \sfe_{i,\on{quant}}$$
onto the $\alpha_i$ components.

\sssec{}

We define $\fu_q(\cN^+)$ to be the image of the map \eqref{e:free to cofree}. The key fact that uses the non-degeneracy 
assumption on $q$ is that $\fu_q(\cN^+)$ is \emph{finite-dimensional}. 

\medskip

Moreover, it can be explicitly described (as an algebra)
as a quotient of $U_q(\cN)^{\on{free}}$ by the quantum Serre relations and the relations 
$$(e_i)^{\on{ord}(q_i)}=1.$$

\ssec{Digression: the notion of relative Drinfeld center}   \label{ss:rel Dr center}

In this subsection we recollect the general framework for defining the notion of relative Drinfeld center. 

\sssec{}

Recall that if $\bA$ is a category, it makes sense to talk about an action of a monoidal category $\bO$ on $\bA$.
The category of such pairs $(\bO,\bA)$ itself forms a symmetric monoidal category under the operation of
Cartesian product. 

\medskip

Consider the category of associative algebra objects in the above category. If $(\bO,\bA)$ is such an algebra object,
the forgetful functor $(\bO,\bA)\mapsto \bA$ endows $\bA$ with a structure of monoidal category, and the forgetful 
functor $(\bO,\bA)\mapsto \bO$ endows $\bO$ with a structure of associative algebra in the category of monoidal
categories. In other words $\bO$ acquires a structure of \emph{braided monoidal category}. 
In this case we shall say that $\bO$ acts on $\bA$ (sometimes, for emphasis, we shall say that $\bO$ acts on $\bA$
\emph{on the left}).

\medskip

Given $\bA$, there exists a universal braided monoidal category that acts on $\bA$ in the above sense. It is
called the \emph{Drinfeld center} of $\bA$, and is denoted $Z_{\on{Dr}}(\bA)$. 

\medskip

The objects of $Z_{\on{Dr}}(\bA)$ are $z\in \bA$, equipped with a family of isomorphisms
$$R_{z,\ba}:z\otimes \ba\simeq \ba\otimes z,$$
compatible with tensor products of the $\ba$'s. 

\sssec{}

Similar definitions apply for \emph{right} actions of monoidal categories. In this way we obtain the notion
of \emph{right} action of a braided monoidal category on a monoidal category. 

\medskip

Thus, we can talk about triples $(\bO,\bA,\bO')$, where $\bA$ is a monoidal category, and $\bO$ and $\bO'$
are braided monoidal categories, acting compatibly on the left and right on $\bA$, respectively. 

\medskip

Given $\bA$
equipped with an action of $\bO'$ on the right, 
there exists a universal braided monoidal category $\bO$ that acts on $\bA$ (on the left) in a way compatible with the right
action of $\bO'$. It is called the \emph{relative (to $\bO'$) Drinfeld center} of $\bA$, and it is denoted $Z_{\on{Dr},\bO'}(\bA)$.

\medskip

Objects of $Z_{\on{Dr},\bO'}(\bA)$ are $z\in \bA$, equipped with a family of isomorphisms
$$R_{z,\ba}:z\otimes \ba\simeq \ba\otimes z,$$
compatible with tensor products of the $\ba$'s, and compatible with the action of $\bO'$ in the sense that
for $\bo'\in \bO'$ the map
$$R_{z,\bo'}:z\otimes \bo'\simeq \bo'\otimes z,$$
agrees with the one induced by the right action of $\bO'$ on $\bA$. 

\medskip

We have a natural forgetful functor
$$Z_{\on{Dr},\bO'}(\bA)\to \bA.$$

\begin{rem}
Note that unless the braided monoidal structure on $\bO'$ is \emph{symmetric}, there is no naturally defined homomorphism
from $\bO'$ to $Z_{\on{Dr},\bO'}(\bA)$.
\end{rem} 

\sssec{}  \label{sss:Hopf algebra}

Let $\bO$ be a braided monoidal category, and let $A$ be a Hopf algebra in $\bO$. In this case, the category
$$\bA:=A\mod$$ of $A$-modules in $\bO$ acquires a natural monoidal structure, compatible 
with the forgetful monoidal functor 
$$\oblv_A:A\mod\to \bO.$$

\medskip

We also note that tensoring on the right defines a \emph{right} action of $\bO$ on $A\mod$. Thus, we
can talk about the braided monoidal category
$$Z_{\on{Dr},\bO}(A\mod).$$

\medskip

Note that the monoidal forgetful functor 
\begin{equation} \label{e:from center to O}
Z_{\on{Dr},\bO}(A\mod)\to A\mod\to \bO
\end{equation} 
is \emph{not} compatible with the braided structures. 

\sssec{}  \label{sss:induce center}

Suppose now that $A$ is dualizable as an object of $\bO$. In this case, the forgetful functor
\begin{equation} \label{e:forget center}
Z_{\on{Dr},\bO}(A\mod)\to A\mod
\end{equation} 
admits both a left and right adjoints.

\medskip 

The composition of the left adjoint to \eqref{e:forget center} with the forgetful functor \eqref{e:from center to O} identifies with
$$M\mapsto \oblv_A(M)\otimes A^\vee,$$
The composition of the right adjoint to \eqref{e:forget center} with the forgetful functor \eqref{e:from center to O} identifies with
$$M\mapsto \oblv_A(M)\otimes A.$$

\ssec{The category of modules over the small quantum group}

In this subsection we will finally introduce the category $\overset{\bullet}\fu_q(\cG)\mod$ of modules 
over the (Langlands dual) small quantum group. 

\sssec{}

We apply the discussion in \secref{sss:Hopf algebra} to the braided monoidal category $\Vect_q^\Lambda$ and the 
Hopf algebra $\fu_q(\cN^+)$ \emph{in} $\Vect_q^\Lambda$.

\medskip

We introduce the (braided monoidal) category $\overset{\bullet}\fu_q(\cG)\mod$ to be
$$Z_{\on{Dr},\Vect_q^\Lambda}(\fu_q(\cN^+)\mod).$$

\medskip

We emphasize that in the above formula, and elsewhere, $\fu_q(\cN^+)\mod$ denotes the category of $\fu_q(\cN^+)$-modules \emph{in} $\Vect_q^\Lambda$.

\sssec{}

Let $\oblv_{\fu_q(\cG)}$ denote the (conservative) forgetful functor
\begin{equation} \label{e:oblv qntm}
\overset{\bullet}\fu_q(\cG)\mod\to \fu_q(\cN^+)\mod\to \Vect_q^\Lambda,
\end{equation}
i.e., the functor \eqref{e:from center to O}. 

\medskip

This functor admits a left adjoint, which we denote by
$$\ind_{\fu_q(\cG)}:\Vect_q^\Lambda \rightleftarrows \overset{\bullet}\fu_q(\cG)\mod.$$

\medskip

The resulting monad $\oblv_{\fu_q(\cG)}\circ \ind_{\fu_q(\cG)}$ is t-exact with respect to the (obvious) t-structure on 
$\Vect_q^\Lambda$. This implies that $\overset{\bullet}\fu_q(\cG)\mod$ acquires a t-structure, uniquely characterized
by the requirement that the forgetful functor $\oblv_{\fu_q(\cG)}$ is t-exact. Moreover, the functor $\ind_{\fu_q(\cG)}$
is also t-exact.

\sssec{}

Denote
$$\fu_q(\cG)^\mu:=\ind_{\fu_q(\cG)}(\sfe^\mu),$$
where we remind that $\sfe^\mu\in \Vect^\Lambda$ is the vector space $\sfe$ placed in degree $\mu$.

\medskip

The objects $\fu_q(\cG)^\mu$ form a set of compact generators of $\overset{\bullet}\fu_q(\cG)\mod$;
they are projective as objects of $(\overset{\bullet}\fu_q(\cG)\mod)^\heartsuit$. 

\medskip

From here we obtain that the canonically defined functor 
$$D^+((\overset{\bullet}\fu_q(\cG)\mod)^\heartsuit)\to \overset{\bullet}\fu_q(\cG)\mod$$
extends to an equivalence
$$D((\overset{\bullet}\fu_q(\cG)\mod)^\heartsuit)\simeq \overset{\bullet}\fu_q(\cG)\mod.$$

\ssec{Standard and costandard objects}

\sssec{}

The forgetful functor 
$$\oblv^{\fu_q(\cG)}_{\fu_q(\cN^+)}:\overset{\bullet}\fu_q(\cG)\mod\to \fu_q(\cN^+)\mod$$
also admits a t-exact left adjoint, denoted $\ind^{\fu_q(\cG)}_{\fu_q(\cN^+)}$, 
see \secref{sss:induce center}.

\medskip

For $\mu\in \Lambda$, consider $\sfe^\mu$ as an object of $\fu_q(\cN^+)\mod$, where the action of 
$\fu_q(\cN^+)$ is trivial. Set
$$\CM^{\mu,!}_{\on{quant}}:=\ind^{\fu_q(\cG)}_{\fu_q(\cN^+)}(\sfe^\mu).$$

We call it the \emph{standard object} of $\overset{\bullet}\fu_q(\cG)\mod$ of highest weight $\mu$. It belongs to
the heart of the t-structure. 

\begin{rem}  \label{r:baby Verma}
The objects $\CM^{\mu,!}_{\on{quant}}$ are sometimes called the ``baby Verma modules". 
\end{rem}

\sssec{} \label{sss:compact by standard}

Note that each  $\fu_q(\cG)^\mu$ has a finite filtration with subquotients $\CM^{\mu+\lambda,!}_{\on{quant}}$ with 
$\lambda\in \Lambda^{\on{pos}}$. 

\sssec{}

It is well-known that $(\overset{\bullet}\fu_q(\cG)\mod)^\heartsuit$ has a structure of highest weight category, in which $\CM^{\mu,!}_{\on{quant}}$
are the standard objects. In particular, for every $\mu$ there exists a \emph{co-standard} object 
$$\CM^{\mu,*}_{\on{quant}}\in (\overset{\bullet}\fu_q(\cG)\mod)^\heartsuit,$$
uniquely characterized by the requirement that
\begin{equation} 
\CHom_{\overset{\bullet}\fu_q(\cG)\mod}(\CM^{\mu',!}_{\on{quant}},\CM^{\mu,*}_{\on{quant}})=
\begin{cases}
&\sfe \text{ if } \mu'=\mu \\
&0 \text{ otherwise}.
\end{cases}
\end{equation}

\sssec{}

We let $\CM^{\mu,!*}_{\on{quant}}$ denote the image of the canonical map
$$\CM^{\mu,!}_{\on{quant}}\to \CM^{\mu,*}_{\on{quant}}.$$

The objects $\CM^{\mu,!*}_{\on{quant}}$ are the irreducibles of $(\overset{\bullet}\fu_q(\cG)\mod)^\heartsuit$.

\section{Renormalization for quantum groups}  \label{s:renorm}

Recall (see \secref{ss:renorm fact}) that we modified the category of factorization modules in order to obtained
a category that eventually turned to be equivalent to $\bHecke(\Whit_{q,x}(G))$. In this section we will apply
a similar renormalization procedure to $\overset{\bullet}\fu_q(\cG)\mod$. 
 
\medskip

In fact, we will define two different renormalizations of $\overset{\bullet}\fu_q(\cG)\mod$: one will be equivalent to
the original category of factorization modules, and the other two its renormalized version.

\ssec{The ``obvious" renormalization}

Along with $\overset{\bullet}\fu_q(\cG)\mod$ we will consider its renormalized version $\overset{\bullet}\fu_q(\cG)\mod^{\on{ren}}$, endowed
with a pair of adjoint functors
$$\on{ren}:\overset{\bullet}\fu_q(\cG)\mod \rightleftarrows \overset{\bullet}\fu_q(\cG)\mod^{\on{ren}}:\on{un-ren}.$$

The material here is parallel to \secref{ss:renorm fact}. 

\sssec{}

The category $\overset{\bullet}\fu_q(\cG)\mod^{\on{ren}}$ is defined as the the ind-completion of the full 
(but not cocomplete) subcategory of 
$$\overset{\bullet}\fu_q(\cG)\mod^{\on{fin.dim}}\subset \overset{\bullet}\fu_q(\cG)\mod$$
that consists of that go to compact objects in 
$\Vect_q^\Lambda$ under the forgetful functor $\oblv_{\fu_q(\cG)}$ of \eqref{e:oblv qntm}.

\medskip

Note that $\overset{\bullet}\fu_q(\cG)\mod^{\on{fin.dim}}$ can also be characterized as consisting of 
finite extensions of (shifts of) irreducible objects.

\medskip

Ind-extension of the tautological embedding defines a functor 
$$\on{un-ren}:\overset{\bullet}\fu_q(\cG)\mod^{\on{ren}}  \to \overset{\bullet}\fu_q(\cG)\mod.$$

\sssec{}

Since the objects $\fu_q(\cG)^\mu$ have finite length, we have
$$\overset{\bullet}\fu_q(\cG)\mod^c\subset \overset{\bullet}\fu_q(\cG)\mod^{\on{fin.dim}}.$$

Ind-extension of this embedding defines a fully faithful functor 
$$\on{ren}: \overset{\bullet}\fu_q(\cG)\mod\to \overset{\bullet}\fu_q(\cG)\mod^{\on{ren}},$$
which is a left adjoint to $\on{un-ren}$.

\sssec{}

We have the following assertion parallel to \propref{p:prop of ren}:

\begin{prop} \label{p:prop of ren quant} 
The category $\overset{\bullet}\fu_q(\cG)\mod^{\on{ren}}$ has a t-structure, uniquely characterized by the property that 
an object is connective if and only if its image under the functor $\on{un-ren}$ is connective. Moreover, the functor
$\on{un-ren}$ has the following properties with respect to this t-structure:

\medskip

\noindent{\em(a)} It is t-exact;

\medskip

\noindent{\em(b)} It induces an equivalence 
$$(\overset{\bullet}\fu_q(\cG)\mod^{\on{ren}})^{\geq n}\to (\overset{\bullet}\fu_q(\cG)\mod)^{\geq n}$$ for any $n$;

\medskip

\noindent{\em(c)} It induces an equivalence of the hearts.

\end{prop}

\begin{cor}
The kernel of the functor $\on{un-ren}$ consists of infinitely coconnective objects, i.e., 
$$\underset{n}\bigcap\, (\overset{\bullet}\fu_q(\cG)\mod^{\on{ren}})^{\leq -n}.$$
\end{cor}

\begin{rem}
We will use the notation $\CM^{\mu,!}_{\on{quant}}$, $\CM^{\mu,*}_{\on{quant}}$ and $\CM^{\mu,!*}_{\on{quant}}$ for the corresponding
objects of \emph{either} $\overset{\bullet}\fu_q(\cG)\mod$ \emph{or} $\overset{\bullet}\fu_q(\cG)\mod^{\on{ren}}$, see 
Remark \ref{r:two meanings}.
\end{rem}

\sssec{}

It follows from the construction that the t-structure on $\overset{\bullet}\fu_q(\cG)\mod^{\on{ren}}$ is Artinian (see 
\secref{sss:properties of t} for what this means). 

\ssec{A different renormalization}

We will now introduce a different category, denoted $$\overset{\bullet}\fu_q(\cG)\mod^{\on{baby-ren}},$$ which
will be sandwiched between $\overset{\bullet}\fu_q(\cG)\mod$ and
$\overset{\bullet}\fu_q(\cG)\mod^{\on{ren}}$. 

\sssec{}  \label{sss:another renorm}

Consider the category $\fu_q(\cN^+)\mod$. 

\medskip

We define its renormalized version, denoted $\fu_q(\cN^+)\mod^{\on{ren}}$, 
by the same procedure as above. I.e., $\fu_q(\cN^+)\mod^{\on{ren}}$ is the
ind-completion of the full (but not cocomplete) subcategory
$$(\fu_q(\cN^+)\mod^{\on{fin.dim}}\subset \fu_q(\cN^+)\mod$$
consisting of objects that map to compact objects under the forgetful functor 
$$\fu_q(\cN^+)\mod\to \Vect_q^\Lambda.$$ 

\medskip

We have a pair of adjoint functors
$$\on{ren}:\fu_q(\cN^+)\mod\rightleftarrows \fu_q(\cN^+)\mod^{\on{ren}}:\on{un-ren}.$$

\sssec{}

The subcategory $\fu_q(\cN^+)\mod^{\on{fin.dim}}\subset \fu_q(\cN^+)\mod$ is preserved by the
tensor product operation. Hence, it inherits a monoidal structure.  

\medskip

Ind-extending, we obtain that $\fu_q(\cN^+)\mod^{\on{ren}}$ acquires a monoidal structure,
which can be uniquely characterized by the requirement that the functor $\on{un-ren}$ be monoidal.

\medskip

Furthermore, we have a right action of (the braided monoidal category) $\Vect_q^\Lambda$ on 
$\fu_q(\cN^+)\mod^{\on{ren}}$, so that the functor $\on{un-ren}$ is compatible with the actions. 

\sssec{}

We define $\overset{\bullet}\fu_q(\cG)\mod^{\on{baby-ren}}$ to be (the relative to $\Vect_q^\Lambda$)
Drinfeld center of $\fu_q(\cN^+)\mod^{\on{ren}}$, i.e., 
$$\overset{\bullet}\fu_q(\cG)\mod^{\on{baby-ren}}:=Z_{\on{Dr},\Vect_q^\Lambda}(\fu_q(\cN^+)\mod^{\on{ren}}).$$

\sssec{}

Note now that the monoidal operation
$$\fu_q(\cN^+)\mod^{\on{ren}}\otimes \fu_q(\cN^+)\mod^{\on{ren}}\to \fu_q(\cN^+)\mod^{\on{ren}}$$
factors as 
$$\fu_q(\cN^+)\mod^{\on{ren}}\otimes \fu_q(\cN^+)\mod^{\on{ren}}\to \fu_q(\cN^+)\mod^{\on{ren}}\otimes \fu_q(\cN^+)\mod
\to \fu_q(\cN^+)\mod^{\on{ren}}$$
and also
$$\fu_q(\cN^+)\mod^{\on{ren}}\otimes \fu_q(\cN^+)\mod^{\on{ren}}
\to \fu_q(\cN^+)\mod\otimes \fu_q(\cN^+)\mod^{\on{ren}}\to \fu_q(\cN^+)\mod^{\on{ren}}.$$

\medskip

From here, it is easy to see that $\overset{\bullet}\fu_q(\cG)\mod^{\on{baby-ren}}$ is endowed by a pair of adjoint functors
$$\ind^{\fu_q(\cG)}_{\fu_q(\cN^+)}:\fu_q(\cN^+)\mod^{\on{ren}}\rightleftarrows \overset{\bullet}\fu_q(\cG)\mod^{\on{baby-ren}}:\oblv^{\fu_q(\cG)}_{\fu_q(\cN^+)}$$
and also a pair of adjoint functors
$$\on{ren}':\overset{\bullet}\fu_q(\cG)\mod \rightleftarrows \overset{\bullet}\fu_q(\cG)\mod^{\on{baby-ren}}:\on{un-ren}'$$
that make all the circuits in the diagram
$$\xymatrix{
\fu_q(\cN^+)\mod^{\on{ren}} \ar[rr]<2pt>^{\ind^{\fu_q(\cG)}_{\fu_q(\cN^+)}} \ar[d]<2pt>^{\on{un-ren}} &&
\overset{\bullet}\fu_q(\cG)\mod^{\on{baby-ren}}  \ar[ll]<2pt>^{\oblv^{\fu_q(\cG)}_{\fu_q(\cN^+)}} \ar[d]<2pt>^{\on{un-ren}'}  \\
\fu_q(\cN^+)\mod \ar[rr]<2pt>^{\ind^{\fu_q(\cG)}_{\fu_q(\cN^+)}} \ar[u]<2pt>^{\on{ren}} && 
\overset{\bullet}\fu_q(\cG)\mod \ar[ll]<2pt>^{\oblv^{\fu_q(\cG)}_{\fu_q(\cN^+)}} \ar[u]<2pt>^{\on{ren}'}}$$
commute.

\sssec{}

The above commutative diagram 
implies in particular that the monad $\oblv^{\fu_q(\cG)}_{\fu_q(\cN^+)}\circ \ind^{\fu_q(\cG)}_{\fu_q(\cN^+)}$
acting on $\fu_q(\cN^+)\mod^{\on{ren}}$ is t-exact. Since the forgetful functor 
$$\oblv^{\fu_q(\cG)}_{\fu_q(\cN^+)}:\overset{\bullet}\fu_q(\cG)\mod^{\on{baby-ren}}\to \fu_q(\cN^+)\mod^{\on{ren}}$$
is conservative, we obtain that
$\overset{\bullet}\fu_q(\cG)\mod^{\on{baby-ren}}$ acquires a t-structure, uniquely characterized by the condition that 
$\oblv^{\fu_q(\cG)}_{\fu_q(\cN^+)}$ 
is t-exact. 

\medskip

In particular, we obtain that the functor 
$$\on{un-ren}':\overset{\bullet}\fu_q(\cG)\mod^{\on{baby-ren}}\to \overset{\bullet}\fu_q(\cG)\mod$$
is t-exact. 

\sssec{}

Thus, we obtain that $\overset{\bullet}\fu_q(\cG)\mod^{\on{baby-ren}}$ can also be obtained as a renormalization of 
$\overset{\bullet}\fu_q(\cG)\mod$. 

\medskip

Namely, $\overset{\bullet}\fu_q(\cG)\mod^{\on{baby-ren}}$ identifies with the ind-completion of the full subcategory 
$$\overset{\bullet}\fu_q(\cG)\mod^{\on{baby}}\subset \overset{\bullet}\fu_q(\cG)\mod$$
that consists of objects that finite extensions of (shifts of) standard (i.e., baby Verma) objects. 

\medskip 

The functor 
$$\on{un-ren}':\overset{\bullet}\fu_q(\cG)\mod^{\on{baby-ren}}  \to \overset{\bullet}\fu_q(\cG)\mod$$
is the ind-extension of the above tautological embedding. 

\medskip

The functor 
$$\on{ren}': \overset{\bullet}\fu_q(\cG)\mod\to \overset{\bullet}\fu_q(\cG)\mod^{\on{baby-ren}}$$
is the ind-extension of the embedding 
$$\overset{\bullet}\fu_q(\cG)\mod^c\subset \overset{\bullet}\fu_q(\cG)\mod^{\on{baby}},$$
the latter due to the fact that each $\fu_q(\cG)^\mu$ has a filtration by standards, see \secref{sss:compact by standard}. 

\medskip

By \cite[Sect. 23]{FG1}, the functor $\on{un-ren}':\overset{\bullet}\fu_q(\cG)\mod^{\on{baby-ren}}  \to \overset{\bullet}\fu_q(\cG)\mod$
induces an equivalence
$$(\overset{\bullet}\fu_q(\cG)\mod^{\on{baby-ren}})^{\geq n}\to (\overset{\bullet}\fu_q(\cG)\mod)^{\geq n}$$ for any $n$,
and thus also an equivalence of the hearts. 

\begin{rem}
We will use the notation $\CM^{\mu,!}_{\on{quant}}$, $\CM^{\mu,*}_{\on{quant}}$ and $\CM^{\mu,!*}_{\on{quant}}$ for the corresponding
objects of \emph{either} $\overset{\bullet}\fu_q(\cG)\mod$ \emph{or} $\overset{\bullet}\fu_q(\cG)\mod^{\on{baby-ren}}$, see 
Remark \ref{r:two meanings}. 
\end{rem}

\begin{rem}

The definition of $\overset{\bullet}\fu_q(\cG)\mod^{\on{baby-ren}}$ can also be rephrased as follows:

\medskip

The monad $\oblv^{\fu_q(\cG)}_{\fu_q(\cN^+)}\circ \ind^{\fu_q(\cG)}_{\fu_q(\cN^+)}$ acting on $\fu_q(\cN^+)\mod$
preserves the subcategories 
$$\fu_q(\cN^+)\mod^c\subset \fu_q(\cN^+)\mod^{\on{fin.dim}}\subset \fu_q(\cN^+)\mod.$$ 

We have:
$$\overset{\bullet}\fu_q(\cG)\mod^c\simeq (\oblv^{\fu_q(\cG)}_{\fu_q(\cN^+)}\circ \ind^{\fu_q(\cG)}_{\fu_q(\cN^+)})\mod(\fu_q(\cN^+)\mod^c)$$
and 
$$\overset{\bullet}\fu_q(\cG)\mod^{\on{baby}}\simeq (\oblv^{\fu_q(\cG)}_{\fu_q(\cN^+)}\circ \ind^{\fu_q(\cG)}_{\fu_q(\cN^+)})\mod(\fu_q(\cN^+)\mod^{\on{fin.dim}}),$$
where the functor $\on{ren}'$ is induced by the embedding
\begin{multline*}
(\oblv^{\fu_q(\cG)}_{\fu_q(\cN^+)}\circ \ind^{\fu_q(\cG)}_{\fu_q(\cN^+)})\mod(\fu_q(\cN^+)\mod^c)\hookrightarrow  \\
\hookrightarrow (\oblv^{\fu_q(\cG)}_{\fu_q(\cN^+)}\circ \ind^{\fu_q(\cG)}_{\fu_q(\cN^+)})\mod(\fu_q(\cN^+)\mod^{\on{fin.dim}}).
\end{multline*}

\end{rem}

\ssec{Relationship between the two renormalizations}

\sssec{}

Note now that the category of compact objects in 
$\overset{\bullet}\fu_q(\cG)\mod^{\on{baby-ren}}$ can be thought of as a subcategory of the category of compact objects in 
$\overset{\bullet}\fu_q(\cG)\mod^{\on{ren}}$. Hence, we obtain a fully faithful functor
$$\on{ren}'':\overset{\bullet}\fu_q(\cG)\mod^{\on{baby-ren}}\to \overset{\bullet}\fu_q(\cG)\mod^{\on{ren}}$$
that admits a continuous right adjoint, denoted $\on{un-ren}''$. 

\medskip

The composition
$$\overset{\bullet}\fu_q(\cG)\mod\overset{\on{ren}'}\longrightarrow \overset{\bullet}\fu_q(\cG)\mod^{\on{baby-ren}} \overset{\on{ren}''}\longrightarrow 
\overset{\bullet}\fu_q(\cG)\mod^{\on{ren}}$$
identifies with the functor $\on{ren}$, and the composition
$$\overset{\bullet}\fu_q(\cG)\mod^{\on{ren}} \overset{\on{un-ren}''}\longrightarrow 
\overset{\bullet}\fu_q(\cG)\mod^{\on{baby-ren}}\overset{\on{un-ren}'}\longrightarrow \overset{\bullet}\fu_q(\cG)\mod$$
identifies with the functor $\on{ren}$.

\medskip

As in \cite[Corollary 4.4.3]{AG}, we obtain that the functor $\on{un-ren}''$ is t-exact.

\sssec{}

Thus, we can think of $\overset{\bullet}\fu_q(\cG)\mod^{\on{ren}}$ as a  renormalization of $\overset{\bullet}\fu_q(\cG)\mod^{\on{baby-ren}}$. 

\medskip

Namely,
$\overset{\bullet}\fu_q(\cG)\mod^{\on{ren}}$ is the ind-completion of the full subcategory of $\overset{\bullet}\fu_q(\cG)\mod^{\on{baby-ren}}$
consisting of objects that are finite extensions of (shifts of) the the objects $\CM^{\mu,!*}_{\on{quant}}$, 
\emph{where the latter are viewed as objects in} $(\overset{\bullet}\fu_q(\cG)\mod^{\on{baby-ren}})^\heartsuit$.

\section{Quantum groups vs factorization modules equivalence}  \label{s:BFS}

In this section we will take the ground field to be $\BC$, the curve $X$ to be $\BA^1$, 
and the sheaf theory to be that of constructible
sheaves in the classical topology. We will relate the category $\overset{\bullet}\fu_q(\cG)\mod$ 
(or rather its renormalized version) to the category of factorizations modules over $\Omega_q^{\on{small}}$. 

\ssec{Matching the parameters}

In this subsection we will explain how to match the parameters needed to define the quantum group
with those needed to define $\Omega_q^{\on{small}}$. 

\sssec{} \label{sss:match}

Recall (see \secref{sss:b'}) that in order to define the category of modules over the quantum group, we started
with a bilinear form 
$$b':\Lambda\otimes \Lambda\to \sfe^{\times,\on{tors}},$$
such that the corresponding quadratic form belongs to $\on{Quad}(\Lambda,\sfe^{\times,\on{tors}})^W_{\on{restr}}$. 

\medskip

Since the ground field is $\BC$, the Tate twist is canonically trivialized, so we can regard $b'$ as a map
\begin{equation} \label{e:b' fact}
b':\Lambda\otimes \Lambda\to \sfe^{\times,\on{tors}}(-1).
\end{equation} 

We claim that the datum of \eqref{e:b' fact} gives rise to a geometric metaplectic datum for $T$ that satisfies
the additional condition of \secref{sss:simple root triv}\footnote{The discussion in the rest of this subsection 
applies to a general sheaf theory of a general ground field.}. 

\sssec{}

First, we claim that a bilinear form \eqref{e:b' fact} (without the extra condition on $q$) gives rise to a geometric
metaplectic datum $\CG^T$ for $T$. 

\medskip

Recall the description of factorization gerbes for tori, given in \cite[Sect. 4.1.3]{GLys}. Namely, to specify a factorization gerbe $\CG^T$,
we need to specify for every finite set $J$ and a map $\lambda_J:J\to \Lambda$ a gerbe $\CG^T_{\lambda^J}$ on $X^J$, along with the
compatibilities of \cite[Equations (4.3) and (4.4)]{GLys}.

\medskip

For $j$, denote $\lambda_j=\lambda_J(j)$. For un \emph{unordered} pair of elements $j_1\neq j_2$, let $\Delta_{j_1,j_2}$ denote the 
corresponding diagonal divisor in $X^J$.
We set
\begin{equation} \label{e:comb gerbe}
\CG^T_{\lambda^J}:=\left(\underset{j\in J}\boxtimes\, \omega^{q(\lambda_j)}\right) \bigotimes \left(\underset{j_1\neq j_2/\Sigma_2}\otimes\, 
\CO(-\Delta_{j_1,j_2})^{b(\lambda_{j_1},\lambda_{j_2})}\right).
\end{equation}

\medskip

The isomorphisms of \cite[Equations (4.4)]{GLys} are automatic. In order to construct the isomorphisms of \cite[Equations (4.3)]{GLys},
to simplify the notation we will consider the case $J=\{1,2\}$. Thus, we need to construct an isomorphism of gerbes
\begin{equation} \label{e:gerbes on X2}
\omega^{q(\lambda_1)} \otimes \omega^{q(\lambda_2)} \otimes (\CO(-\Delta)|_\Delta)^{b(\lambda_{j_1},\lambda_{j_2})}\simeq 
\omega^{q(\lambda_1+\lambda_2)}.
\end{equation}

\sssec{}

Let us note that for every element $c\in \sfe^\times(-1)$ we have a well-defined $\sfe^\times$-torsor, denoted $(-1)^c$, constructed as follows:

\medskip

To $c$ we associate the corresponding Kummer sheaf $\Psi_c$ on $\BG_m$. 
We set $(-1)^c$ to be equal to the fiber of $\Psi_c$ at $(-1)\in \BG_m$. 

\sssec{}

To define \eqref{e:gerbes on X2} let us first choose an ordering, namely $(1,2)$ on $\{1,2\}$. 
This ordering identifies the line bundle $\CO(-\Delta)|_\Delta$ with $\omega$. 

\medskip

We let \eqref{e:gerbes on X2} be the tautological isomorphism coming from the identity 
$$q(\lambda_1)\cdot q(\lambda_2)\cdot b(\lambda_{j_1},\lambda_{j_2})=q(\lambda_1+\lambda_2),$$ tensored with the line
$$(-1)^{b'(\lambda_1,\lambda_2)}.$$

Note that this is the only place in the construction where we use the data of a bilinear form $b'$, as opposed to that
of a quadratic form $q$. 

\sssec{}

Let us now show that the isomorphism \eqref{e:gerbes on X2} is canonically independent of the choice of the ordering. Indeed,
the swap of two factors multiplies the identification 
$$\CO(-\Delta)|_\Delta\simeq \omega$$
by $(-1)$. 

\medskip

The required isomorphism follows now from
$$(-1)^{b(\lambda_1,\lambda_2)}\otimes (-1)^{b'(\lambda_1,\lambda_2)}\simeq (-1)^{b'(\lambda_2,\lambda_1)}.$$

\begin{rem}
The the map 
$$\on{Bilin}(\Lambda,\sfe^\times(-1))\to \on{Quad}(\Lambda,\sfe^\times(-1))$$ is surjective, 
it follows that any geometric metaplectic data for $T$ can be obtained as a tensor product 
of one coming from a bilinear form $b'$ as above, with one with a vanishing quadratic form, i.e.,
one coming from a $\cT(\sfe)^{\on{tors}}$-gerbe (see \secref{sss:neutral case}). 
\end{rem}

\sssec{}    \label{sss:from bilin to gerbes with triv}

Let us now restore the condition that $q$ belong to $\on{Quad}(\Lambda,\sfe^{\times,\on{tors}})$. Let us show that
the resulting factorization gerbe on $\Conf$ satisfies the condition from \secref{sss:simple root triv}. 

\medskip

Indeed, according to formula \eqref{e:expl shifted gerbe vac}, we need construct a trivialization of the gerbe
$$\omega^{q(-\alpha_i)}\otimes (\omega^{\otimes \frac{1}{2}})^{b(-\alpha_i,2\rho)}.$$
For this, it suffices to show that 
$$q(-\alpha_i)^2\cdot b(-\alpha_i,2\rho)=1,$$
but this follows from condition \eqref{e:restr cond}. 

\ssec{Statement of the result}

In this subsection we will state the main theorem of this Part that establishes an equivalence between the categories
$\overset{\bullet}\fu_q(\cG)\mod^{\on{baby-ren}}$ and $\Omega^{\on{small}}_q\on{-FactMod}$. 

\sssec{}  

We start with a form $b'$ as in \secref{sss:match}. Choose the lines $\sfe^{i,\on{quant}}$, see \secref{sss:e lines}.
Let $\sff^{i,\on{fact}}$ be the dual lines. 

\medskip

To the data $(b',\{\sfe^{i,\on{quant}}\})$ be we associate the category 
$\overset{\bullet}\fu_q(\cG)\mod$, and its renormalized version $\overset{\bullet}\fu_q(\cG)\mod^{\on{baby-ren}}$.

\medskip

To the data of $b'$ we associate a geometric metaplectic data $\CG^T$ for $T$ (see \secref{sss:match}), and to the
data $(\CG^T,\{\sff^{i,\on{fact}}\})$ we attach the category $\Omega^{\on{small}}_q\on{-FactMod}$, see 
\secref{s:Omega small}. 

\sssec{}

The goal of this section is to prove the following:

\begin{thm}  \label{t:factorization vs quantum}
There exists a canonical equivalence 
$$\Omega^{\on{small}}_q\on{-FactMod}\simeq \overset{\bullet}\fu_q(\cG)\mod^{\on{baby-ren}},$$
which is t-exact and maps standards to standards.
\end{thm} 

We note that the assertion of \thmref{t:factorization vs quantum} at the level of the hearts of the corresponding categories
is the main result of the book \cite{BFS}, specifically Theorem 17.1 in Part III of {\it loc. cit}.

\sssec{}

As a formal corollary of \thmref{t:factorization vs quantum} we obtain: 

\begin{cor}  \label{c:factorization vs quantum}
There exists a canonical equivalence 
$$\Omega^{\on{small}}_q\on{-FactMod}^{\on{ren}}\simeq \overset{\bullet}\fu_q(\cG)\mod^{\on{ren}},$$
which is t-exact and maps standards to standards.
\end{cor} 

The rest of this section is devoted to the proof of \thmref{c:factorization vs quantum}\footnote{The proof of 
\thmref{t:factorization vs quantum} given below is the result of discussions between the first-named
author and J.~Lurie. However, the responsibility for any shortcomings that may result from its
publication lie with D.G.}.

\sssec{} \label{sss:proof fin length quant} 

Note also that \thmref{t:factorization vs quantum} gives a proof of \propref{p:fin length} when $k=\BC$. 

\ssec{Koszul duality for Hopf algebras}  \label{ss:Koszul abstract}

In this section we will perform the first step in the proof of \thmref{t:factorization vs quantum}: it consists of
passing from a Hopf algebra in $\Vect^\Lambda_q$ (such as $\fu_q(\cN^+)$) to its Koszul dual $\BE_2$-algebra. 

\sssec{}   \label{sss:Koszul abstract}

Let $A$ be a Hopf algebra in $\Vect^\Lambda_q$ such that its augmentation ideal is contained in
$\Vect_q^{\Lambda^{\on{pos}}-0}$, and each graded component is finite-dimensional.

\medskip

Consider the category $A\mod$, and let $A\mod^{\on{ren}}$ denote its renormalized version
defined as in \secref{sss:another renorm}. 

\medskip

By the definition of $A\mod^{\on{ren}}$, the functor of trivial action
$$\triv_{A}:\Vect^\Lambda_q\to A\mod^{\on{ren}},$$
admits a \emph{continious} right adjoint. 

\medskip

This right adjoint is also conservative, because by the condition on $A$, the essential image of
the functor $\triv_{A}$ generates $A\mod^{\on{ren}}$ under colimits (indeed, the essential image 
of $(\Vect^\Lambda_q)^{\on{fin.dim}}$ along $\triv_{A}$ generates $A\mod^{\on{fin.dim}}$ 
under finite colimits). 

\sssec{}

The monad
$$\inv_{A}\circ \triv_{A}$$
acting on $\Vect^\Lambda_q$ commutes with right multiplication. Hence, it is given by an 
associative algebra in $\Vect^\Lambda_q$, to be denoted $\on{Inv}_{A}$. The underlying 
object of $\Vect^\Lambda_q$ identifies with $\inv_A\circ \triv_A(\sfe)$, i.e., 
$$\lambda\mapsto \CHom_{A\mod}(\triv_{A}(\sfe^0),\triv_{A}(\sfe^\lambda)).$$

\medskip

Tautologically, the functor $\inv_{A}$ upgrades to a functor 
\begin{equation} \label{e:Koszul N^+}
\inv^{\on{enh}}_{A}:A\mod^{\on{ren}}\to \on{Inv}_{A}\mod.
\end{equation} 

By the Barr-Beck-Lurie theorem, the above functor \eqref{e:Koszul N^+} is an equivalence.

\sssec{}

Let now $A$ be a Hopf algebra in $\Vect^\Lambda_q$. This structure is equivalent to giving 
$A\mod$ (or $A\mod^{\on{ren}}$) a structure of monoidal category, for which the forgetful
functor 
$$A\mod^{\on{ren}} \to \Vect^\Lambda_q$$
is monoidal, in a way compatible with the right action of the braided monoidal category $\Vect^\Lambda_q$
(see \secref{sss:Hopf algebra}). The unit in $A\mod^{\on{ren}}$ is given by $\triv_{A}(k)$. 

\medskip

Hence, the equivalence of \eqref{e:Koszul N^+} induces a monoidal structure on $\on{Inv}_{A}\mod$,
for which the unit object is
$$\on{Inv}_{A}\in  \on{Inv}_{A}\mod.$$

\medskip

Such a structure is equivalent to a structure on $\on{Inv}_{A}$ of $\BE_2$-algebra in $\Vect^\Lambda_q$, 
so that \eqref{e:Koszul N^+}.

\sssec{}

Thus, we obtain an equivalence of the corresponding (relative to $\Vect^\Lambda_q$) Drinfeld centers
\begin{equation} \label{e:Koszul Dr}
Z_{\on{Dr},\Vect^\Lambda_q}(A\mod^{\on{ren}})\to Z_{\on{Dr},\Vect^\Lambda_q}(\on{Inv}_{A}\mod).
\end{equation} 

\sssec{}

By \cite[Proposition 4.36]{Fra} we have:
$$Z_{\on{Dr},\Vect^\Lambda_q}(\on{Inv}_{A}\mod)\simeq \on{Inv}_{A}\mod_{\BE_2},$$
where the latter denotes the category of $\BE_2$-modules over the $\BE_2$-algebra 
$\on{Inv}_{A}$ in the braided monoidal category $\Vect^\Lambda_q$.

\sssec{}

To summarize, we obtain an equivalence
\begin{equation} \label{e:E2 as center abstract}
Z_{\on{Dr},\Vect^\Lambda_q}(A\mod^{\on{ren}})\simeq \on{Inv}_{A}\mod_{\BE_2}.
\end{equation}

\ssec{Factorization algebras  vs $\BE_2$-algebras}  \label{ss:factorization vs E2}

We will now perform a crucial step in the transition between $\overset{\bullet}\fu_q(\cG)\mod^{\on{ren}}$
and $\Omega^{\on{small}}_q\on{-FactMod}$: we will relate (a certain kind of) $\BE_2$-algebras in $\Vect^\Lambda_q$
and factorization algebras in $\Shv_{\CG^\Lambda}(\Conf)$. 

\sssec{}  

Recall that in this section the curve $X$ is taken to be $\BA^1$ (with $x\in X$ being $0\in \BA^1$). 

\medskip

According to \cite{Lur}, to a braided monoidal category $\bO$ one can attach a \emph{factorization category} over 
the Ran space of $\BA^1$, denoted $\on{Fact}(\bO)$. 

\medskip

Futhermore, if $B$ is an $\BE_2$-algebra in $\bO$, then to it there corresponds a \emph{factorization algebra} 
$\Omega_B$ in $\on{Fact}(\bO)$, and we have an equivalence between the category of $\BE_2$-modules
with respect to $B$ in $\bO$ and factorization $\Omega_B$-modules in $\on{Fact}(\bO)$, i.e.,
\begin{equation} \label{e:E2 vs factorization abs}
\Omega_B\on{-FactMod} \simeq B\mod_{\BE_2}. 
\end{equation}

\sssec{}

We take $\bO=\Vect^\Lambda_q$. In this case $\on{Fact}(\bO)$ identifies with $\Shv_{\CG^T}(\Gr_{T,\Ran})$. 

\medskip

Let now $B$ be a (non-unital) $\BE_2$-algebra $\Vect^\Lambda_q$, which is contained in $\Vect_q^{\Lambda^{\on{neg}}-0}$.
Then $\Omega_B$, viewed as an object of $\Shv_{\CG^T}(\Gr_{T,\Ran})$, is supported on $(\Gr_{T,\Ran})^{\on{neg}}$. Hence, by
\secref{sss:fact Gr T vs Conf}, we can think of $\Omega_B$ as a $\CG^\Lambda$-twisted 
factorization algebra on $\Conf$. 

\medskip

Furthermore, according to \eqref{e:compare fact mod}, we can think of factorization $\Omega_B$-modules on 
$\Gr_{T,\Ran}$ as factorization $\Omega_B$-modules on $\Conf_{\infty\cdot x}$.

\medskip

Hence, \eqref{e:E2 vs factorization abs} becomes an equivalence
\begin{equation} \label{e:E2 vs factorization}
\Omega_B\on{-FactMod} \simeq B\mod_{\BE_2},
\end{equation}
where $\Omega_B\on{-FactMod}$ denotes the category of factorization $\Omega_B$-modules in 
$\Shv_{\CG^\Lambda}(\Conf_{\infty\cdot x})$.

\medskip

Let us write how certain functors on one side of the equivalence \eqref{e:E2 vs factorization} 
translate to the other side. 

\sssec{}

Since our curve $X$ is $\BA^1$, the canonical line bundle $\omega$ on $X$ is trivialized. In addition, we have a 
canonical generator for the diagonal divisor $\Delta\subset X\times X$. In particular, by formulas \eqref{e:comb gerbe}
and \eqref{e:expl shifted gerbe vac}, the fiber of $\CG^\Lambda$ at any point $\mu\cdot x\in \Conf_{\infty\cdot x}$
admits a canonical trivialization. 

\begin{rem}
It is easy to see that the gerbe $\CG^\Lambda$ on all of $\Conf_{\infty\cdot x}$ (and on $\Conf$)
admits a canonical trivialization. However, these trivializations are \emph{incompatible} with factorization.

\medskip

The above trivialization of $\CG^\Lambda$ on $\Conf$ is also incompatible with the trivialization of $\CG^\Lambda$ on
$\overset{\circ}{\Conf}$ of \secref{sss:open part triv}. The discrepancy of these two trivialization  
is given by a \emph{non-trivial} local system on $\overset{\circ}{\on{Conf}}(X,\Lambda^{\on{neg}})$. This is the braiding local
system of \cite[Part III, Sect. 3.1]{BFS}. As a result, the twisted perverse sheaf $\overset{\circ}\Omega{}^{\on{small}}_q$, 
viewed as a plain perverse sheaf on $\overset{\circ}{\Conf}$ (via the above trivialization of $\CG^\Lambda$ specific to $\BA^1$)
is not just the sign local system, but has a monodromy that depends on $q$. 
\end{rem} 

\sssec{}  \label{sss:!-fiber}

First off, for $\mu\in \Lambda$, the functor
$$B\mod_{\BE_2} \to \Vect^\Lambda_q \to \Vect^\mu_q \simeq \Vect,$$
(where the second arrow is the projection on the $\mu$-component) 
corresponds in terms of the equivalence \eqref{e:E2 vs factorization} to the composite:
$$\Omega_B\on{-FactMod}\overset{\oblv_{\on{Fact}}}\longrightarrow \Shv_{\CG^\Lambda}(\Conf_{\infty\cdot x})
\overset{\text{!-fiber at }\mu\cdot x}\longrightarrow \CG^\lambda|_{\mu\cdot x}\simeq \Vect.$$

\sssec{} \label{sss:hyperbolic fiber}

Consider now the functor $\Omega_B\on{-FactMod}\to \Vect$ equal to the composite: 

\begin{itemize}

\item 
The forgetful functor $\Omega_B\on{-FactMod}\to \Shv_{\CG^\Lambda}(\Conf_{\infty\cdot x})$;

\item The functor of !-restriction to the subspace of \emph{non-negative real configurations} 
$$\Conf^{\BR^{\geq 0}}_{\infty\cdot x}\subset \Conf_{\infty\cdot x};$$

\item The functor of *-fiber at $\mu\cdot x \in \Conf^{\BR^{\geq 0}}_{\infty\cdot x}$. 

\item The identification $\CG^\lambda|_{\mu\cdot x}\simeq \Vect$. 

\end{itemize}

The corresponding functor $B\mod_{\BE_2}\to \Vect$ is the composite
$$B\mod_{\BE_2}\to B\mod\overset{\be\underset{B}\otimes -}\longrightarrow \Vect^\Lambda_q \to \Vect^\mu_q \simeq \Vect.$$

\ssec{Hopf algebras vs factorization algebras}  \label{ss:Hopf vs factorization}

In this subsection we will supply some explicit information on the factorization algebra $\Omega_B$ corresponding via
\eqref{e:E2 vs factorization abs} to the augmentation ideal in $\on{Inv}_A$, where $A$ is a Hopf algebra as in \secref{sss:Koszul abstract}. 

\sssec{}

Let $B$ be the augmentation ideal in $\on{Inv}_{A}$. Composing \eqref{e:E2 vs factorization}
with \eqref{e:E2 as center abstract}, we obtain an equivalence
\begin{equation} \label{e:center to factor}
Z_{\on{Dr},\Vect^\Lambda_q}(A\mod^{\on{ren}}) \simeq \Omega_B\on{-FactMod}.
\end{equation}

\sssec{}

Let us see how the obvious forgetful functors on one side of the equivalence \eqref{e:center to factor}
look on the other side. 

\medskip

The functor of !-fiber at the point $\mu\cdot x\in \Conf_{\infty\cdot x}$ on $\Omega_B\on{-FactMod}$
(i.e., the composite described in \secref{sss:!-fiber}) corresponds to the composite 
$$Z_{\on{Dr},\Vect^\Lambda_q}(A\mod^{\on{ren}}) \to   A\mod^{\on{ren}} \overset{\inv_A} \to \Vect^\Lambda_q \to \Vect^\mu_q \simeq \Vect,$$
where the first arrow in the forgetful functor and the third arrow is the projection onto the $\mu$-component.

\medskip

Consider now the composite functor $\Omega_B\on{-FactMod}\to \Vect$, described in \secref{sss:hyperbolic fiber}. It corresponds to the 
forgetful functor 
$$Z_{\on{Dr},\Vect^\Lambda_q}(A\mod^{\on{ren}}) \to \Vect^\Lambda_q \to \Vect^\mu_q \simeq  \Vect,$$
where the second arrow is the projection onto the $\mu$-component.

\sssec{}  \label{sss:dual of Hopf}

We have the following additional two properties of the assignment $A\rightsquigarrow \Omega_B$:

\begin{prop} \label{p:dual of Hopf} \hfill

\smallskip

\noindent{\em(a)}
For $\mu\in \Lambda^{\on{neg}}$, the vector space equal to the *-fiber at $\mu\cdot x$ of $\Omega_B$
identifies with the $-\mu$-component of $\on{Inv}_{A^\vee}$, where $A^\vee$ is the component-wise linear
dual of $A$, viewed as a Hopf algebra in $\Vect^\Lambda_{q^{-1}}$. 

\smallskip

\noindent{\em(b)} For $\mu\in \Lambda^{\on{neg}}$, the vector space equal to the *-fiber at $\mu\cdot x$ of the 
!-restriction of $\Omega_B$ to $\Conf^{\BR}\subset \Conf$
identifies with the vector space \emph{dual} to the $-\mu$-component of $A$. 

\end{prop} 

\sssec{}

We now claim:

\begin{cor} \label{c:exactness Kosz}  \hfill

\smallskip

\noindent{\em(a)} If $A$ is concentrated in cohomological degrees $\geq 0$ (resp., $\leq 0$), with respect to the
obvious t-structure on $\Vect^\Lambda_q$, then $\Omega_B$, viewed as an object of $\Shv_{\CG^\Lambda}(\Conf)$,
is concentrated in perverse cohomological degrees $\geq 0$ (resp., $\leq 0$). 

\smallskip

\noindent{\em(a')} If $A$ is concentrated in cohomological degree $0$,
then $\Omega_B$ is perverse. 

\smallskip

\noindent{\em(b)} If $A_1$ and $A_2$ are both concentrated in cohomological degree $0$ and 
$A_1\to A_2$ is a surjective (resp., injective) map of Hopf algebras, then the induced map
$\Omega_{B_2}\to \Omega_{B_1}$ is injective (resp., surjective).

\smallskip

\noindent{\em(c)} Under the assumption of \emph{(a')}, the equivalence \eqref{e:center to factor} is t-exact.

\end{cor}

\begin{proof}

For a fixed $\lambda\in \Lambda^{\on{neg}}$ consider the full subcategory of $\Shv(X^\lambda)$ that consists of complexes
locally constant along the diagonal stratification. Then the functor that sends an object to the *-fiber at $\mu\cdot x$ 
of its !-restriction to $X^\lambda_\BR$ is conservative and t-exact (in the perverse t-structure).

\medskip

This implies points (a) and (b) in view of \propref{p:dual of Hopf}(b). Point (a') is a particular case of (a). 

\medskip

For point (c) we consider the full subcategory of $\Shv(X^\lambda)$ that consists of complexes
locally constant along the stratification given by diagonals and incidence with $x$. Then on this subcategory, the 
functor that sends an object to the *-fiber at $\mu\cdot x$ of its !-restriction to
$X^\lambda_{\BR^{\geq 0}}$ is t-exact (in the perverse t-structure).

\end{proof}

\ssec{The case of quantum groups}

We will now combine the contents of Sects. \ref{ss:Koszul abstract}-\ref{ss:Hopf vs factorization}, 
and complete the proof of \thmref{t:factorization vs quantum}. 

\sssec{}

We apply the above discussion to the Hopf algebra $A=\fu_q(\cN^+)$. First, we claim that the resulting factorization algebra
$\Conf$ $\Omega_B$ on $\Conf$ identifies with $\Omega^{\on{small}}_q$.  

\medskip

Using \corref{c:exactness Kosz}(b), it suffices to show that
the (twisted) sheaf $\Omega_{A'}$ on $\Conf$ corresponding to $A'=U_q(\cN)^{\on{free}}$ 
(resp., $A'=U_q(\cN)^{\on{co-free}}$) is given by extension by * (resp., !) of $\overset{\circ}\Omega{}^{\on{small}}_q$
along the embedding
\begin{equation}  \label{e:open config}
\overset{\circ}{\on{Conf}}\hookrightarrow \Conf.
\end{equation}
We will prove the assertion regarding $U_q(\cN)^{\on{free}}$; the one about $U_q(\cN)^{\on{co-free}}$ follows similarly 
by applying \propref{p:dual of Hopf}(a). 

\sssec{}

We have to show that !-fibers of $\Omega_B$ for $B=\on{Inv}_A$ with $A=U_q(\cN)^{\on{free}}$ are zero on the
complement to \eqref{e:open config}. By factorization, this is equivalent to showing that !-fibers of $\Omega_B$
are zero at points $\mu\cdot x$ for $\mu$ not being a negative simple root. 

\medskip

By \secref{sss:!-fiber}, we need to show that $\inv_{A}(\sfe)$ lives only in degrees that are negative simple roots. However,
this follows from the fact that $A$ is free as an associative algebra on the generators in degrees equal to simple roots. 

\sssec{}

Thus, the equivalence \eqref{e:E2 vs factorization} translates in our case to the equivalence
$$\Omega^{\on{small}}_q\on{-FactMod}\simeq Z_{\on{Dr},\Vect_q^\Lambda}(\fu_q(\cN^+)\mod^{\on{ren}}),$$
where the right-hand side is by definition $\overset{\bullet}\fu_q(\cG)\mod^{\on{baby-ren}}$.

\medskip

This establishes the equivalence of categories claimed in \thmref{t:factorization vs quantum}. 
The t-exactness property of the above equivalence follows from \corref{c:exactness Kosz}(c). 
The fact that this equivalence maps standards to standards follows from the construction. 

\newpage

\appendix

\section{Whittaker vs Kirillov models}

In this Appendix we will explain a device that replaces the Whittaker model when the Artin-Schreier sheaf 
does not exist, e.g., in the $\ell$-adic context when the ground field has characteristic zero, or in the Betti context. 
This device is called the Kirillov model. 

\ssec{The context}

\sssec{}

Let $\CY$ be an algebraic stack of finite type over the ground field $k$. Let first $\CY$ be equipped with an action of
$\BG_a$. If $k$ is of characteristic zero and we are working with D-modules over if $k$ is of positive characteristic 
and we are working with $\ell$-adic sheaves, we have a well-defined \emph{full subcategory}
$$\Whit(\CY)\subset \Shv(\CY)$$
that consists of objects that are $\BG_a$-invariant against a chosen Artin-Schreier sheaf $\chi$ on $\BG_a$
(in the context of D-modules, $\chi$ is the exponential D-module). 

\medskip

I.e., this is the category whose objects are those $\CF\in  \Shv(\CY)$ for which there \emph{exists} an isomorphism
$$\on{act}^*(\CF)\simeq \chi\boxtimes \CF,$$
whose further *-pullback along 
$$\{0\}\times \CY\overset{i_0}\longrightarrow \BG_a\times \CY$$ is the identity map 
$$\CF\simeq i_0^*\circ \on{act}^*(\CF)\simeq  i_0^*\circ (\chi\boxtimes \CF)\simeq \CF.$$

\sssec{}  \label{sss:ppties Whit}

The subcategory $\Whit(\CY)\subset \Shv(\CY)$ has some favorable properties:

\begin{enumerate}

\item It is compatible with the t-structure (i.e., is preserved by the truncation functors);

\smallskip

\item For a map $f:\CY_1\to \CY_2$, the functors $f^!$ and $f_*$ send the categories 
$\Whit(\CY_1)$ and $\Whit(\CY_2)$ to one another. Furthermore, the partially defined left adjoints of these
functors, i.e., the functors $f_!$ and $f^*$ also send $\Whit(\CY_i)$ to $\Whit(\CY_j)$ on those
objects on which they are defined;

\smallskip

\item Verdier duality (which is defined on $\Shv(\CY)^{\on{loc.c.}}$, see \secref{sss:locally compact})
sends $\Whit(\CY)$ to a similarly defined character for the opposite choice of the Artin-Schreier sheaf.

\end{enumerate} 

\sssec{}

Assume now that we are given an extension of the action of $\BG_a$ on $\CY$ to an action of the
semi-direct product 
$$\BG_m\ltimes \BG_a.$$

In this case we will be able to define another category, denoted $\on{Kir}_*(\CY)$. This category $\on{Kir}_*(\CY)$
will be defined in an arbitrary sheaf-theoretic context (in that it does not require the existence of the Artin-Schreier
sheaf). 

\medskip

In the context when Artin-Schreier is defined, we will have a canonical equivalence 
$$\on{Kir}_*(\CY)\simeq \Whit(\CY).$$

\medskip

In addition, in the context of constructible sheaves or holonomic D-modules, we will be able to define another category,
denoted $\on{Kir}(\CY)_!$, and we will have also an equivalence
$$\on{Kir}_!(\CY)\simeq \on{Kir}_*(\CY).$$

\medskip

Furthermore, restricting to locally compact objects, the Verdier duality functor on $\Shv(\CY)^{\on{loc.c}}$ 
induces an equivalence
$$(\on{Kir}_!(\CY)^{\on{loc.c}})^{\on{op}}\to \on{Kir}_*(\CY)^{\on{loc.c}}.$$

\begin{rem}
That said, the category $\on{Kir}_*(\CY)$ (or $\on{Kir}_!(\CY)$) does not enjoy the favorable properties of $\Whit(\CY)$
mentioned in \secref{sss:ppties Whit}. 

\medskip

Most importantly, for a map $f:\CY_1\to \CY_2$, the functors $f^!$ and $f_*$
do send $\on{Kir}_*(\CY_1)$ and $\on{Kir}_*(\CY_2)$ to one another, but the functors $f_!$ and $f^*$ do not. The situation
with $\on{Kir}_!(\CY)$ is the opposite.
\end{rem} 

\sssec{}

For the purposes of this work, we take the stacks $\CY$ to be the following ones: 

\medskip 

Recall the subschemes $Y_j\subset \Gr^{\omega^\rho}_{G,x}$ and the subgroups $N_k\subset \fL(N)^{\omega^\rho}_x$ 
of \secref{ss:def Whit}, so that the action of $N_k$ on $Y_j$ factors through some finite-dimensional quotient $N_{k,l}$.

\medskip

Consider the action of $\BG_m$ on $\fL(N)^{\omega^\rho}_x$ obtained from the adjoint action of
$T_{\on{ad}}\subset \fL^+(T)_x$ on $\fL(N)^{\omega^\rho}_x$ and the cocharacter $\rho:\BG_m\to T_{\on{ad}}$. 

\medskip

With no restriction of generality, we can assume that the subgroup $N_k$ is preserved by this action, as well
as the kernel of the projection to $N_{k,l}$. Finally, we can assume that the restriction of the canonical homomorphism
$\fL(N)^{\omega^\rho}_x\to \BG_a$ to $N_k$ factors through $N_{k,l}$.  Let $N'_{k,l}$ be the kernel of the resulting
homomorphism $N_{k,l}\to \BG_a$. 

\medskip

Set
$$\CY:=N'_{k,l}\backslash Y_j.$$

By construction, $\CY$ carries a residual action of $\BG_a$. In addition, we 
note that the action of $T\subset \fL^+(T)_x$ on $\Gr^{\omega^\rho}_{G,x}$ also factors through $T_{\on{ad}}$. So
we obtain a well-defined action of $\BG_m$ on $Y_j$ via $\rho$, and hence on $\CY$.

\medskip

It is easy to, however, that the above $\BG_a$- and $\BG_m$-actions on $\CY$ combine to an action of the semi-direct
product $\BG_m\ltimes \BG_a$.

\ssec{Definition of the Kirillov model}

\sssec{}

Consider the category $\Shv(\CY)^{\BG_m}$ and its full subcategory $\Shv(\CY)^{\BG_m\ltimes \BG_a}$. The forgetful functor
\begin{equation} \label{e:semidir}
\Shv(\CY)^{\BG_m\ltimes \BG_a}\hookrightarrow \Shv(\CY)^{\BG_m}
\end{equation}
admits a right adjoint, denoted $\on{Av}^{\BG_a}_*$, which makes the following diagram commute:
$$
\CD 
\Shv(\CY)^{\BG_m}  @>{\on{Av}^{\BG_a}_*}>>  \Shv(\CY)^{\BG_m\ltimes \BG_a}  \\
@VVV   @VVV   \\
\Shv(\CY)  @>{\on{Av}^{\BG_a}_*}>>  \Shv(\CY)^{\BG_a},
\endCD
$$
where the vertical arrows are the forgetful functors. 

\medskip

We let $\on{Kir}_*(\CY)$ be the full subcategory of $\Shv(\CY)^{\BG_m}$ equal to the kernel of the functor $\on{Av}^{\BG_a}_*$.

\sssec{}

Note that the embedding 
$$\on{Kir}_*(\CY)\hookrightarrow \Shv(\CY)^{\BG_m}$$
admits a \emph{left} adjoint, given by
\begin{equation} \label{e:left adj *}
\CF\mapsto \on{coFib}(\on{Av}^{\BG_a}_*(\CF)\to \CF).
\end{equation}

\sssec{}

Similarly, in the constructible situation, the functor \eqref{e:semidir} admits a left adjoint, 
denoted $\on{Av}^{\BG_a}_!$, which makes the following diagram commute:
$$
\CD 
\Shv(\CY)^{\BG_m}  @>{\on{Av}^{\BG_a}_!}>>  \Shv(\CY)^{\BG_m\ltimes \BG_a}  \\
@VVV   @VVV   \\
\Shv(\CY)  @>{\on{Av}^{\BG_a}_!}>>  \Shv(\CY)^{\BG_a}. 
\endCD
$$

\medskip

We let $\on{Kir}_!(\CY)$ be the full subcategory of $\Shv(\CY)^{\BG_m}$ equal to the kernel of the functor $\on{Av}^{\BG_a}_!$.

\sssec{}

Note that now the embedding 
$$\on{Kir}_!(\CY)\hookrightarrow \Shv(\CY)^{\BG_m}$$
admits a \emph{right} adjoint, given by
\begin{equation} \label{e:right adj !}
\CF\mapsto \on{Fib}(\CF\to \on{Av}^{\BG_a}_!(\CF)).
\end{equation}

\sssec{}

It is clear that the Verdier duality functor
$$\left((\Shv(\CY)^{\BG_m})^{\on{loc.c}}\right)^{\on{op}}\to (\Shv(\CY)^{\BG_m})^{\on{loc.c}}$$
sends 
$$\on{Kir}_!(\CY)^{\on{loc.c}}:=\on{Kir}_!(\CY) \cap (\Shv(\CY)^{\BG_m})^{\on{loc.c}}\subset (\Shv(\CY)^{\BG_m})^{\on{loc.c}}$$
to 
$$\on{Kir}_*(\CY)^{\on{loc.c}}:=\on{Kir}_*(\CY) \cap (\Shv(\CY)^{\BG_m})^{\on{loc.c}}\subset (\Shv(\CY)^{\BG_m})^{\on{loc.c}}$$
and vice versa. 

\ssec{The Whittaker vs Kirillov equivalence}

\sssec{}

Let us be again in the situation when the Artin-Schreier sheaf is defined. Consider the functor
\begin{equation} \label{e:Whit to Kir prel}
\Whit(\CY)\to \Shv(\CY)^{\BG_m}
\end{equation}
equal to
$$\Whit(\CY)\hookrightarrow \Shv(\CY) \overset{\on{Av}^{\BG_m}_*}\longrightarrow \Shv(\CY)^{\BG_m},$$
where the first arrow is the forgetful functor. 

\medskip

It is easy to see that the image of \eqref{e:Whit to Kir prel} belongs to $\on{Kir}_!(\CY)\subset \Shv(\CY)^{\BG_m}$: this follows from
the fact that the diagram
$$
\CD
\Shv(\CY) @>{\on{Av}^{\BG_m}_*}>> \Shv(\CY)^{\BG_m}  \\
@V{\on{Av}^{\BG_a}_*}VV  @VV{\on{Av}^{\BG_a}_*}V   \\
\Shv(\CY)^{\BG_a} @>{\on{Av}^{\BG_m}_*}>> \Shv(\CY)^{\BG_m}
\endCD
$$
commutes.

\medskip

Hence, we obtain a functor
\begin{equation} \label{e:Whit to Kir}
\Whit(\CY)\to \on{Kir}_*(\CY)
\end{equation}

We claim:

\begin{prop} \label{p:Whit Kir}  \hfill

\smallskip

\noindent{\em(a)} The functor \eqref{e:Whit to Kir} is an equivalence.

\smallskip

\noindent{\em(b)} The (partially defined) functor $\on{Av}^{\BG_a,\chi}_!$ left adjoint to
the embedding $\Whit(\CY)\to \Shv(\CY)$ is defined and maps isomorphically to $\on{Av}^{\BG_a,\chi}_*[2]$
on the essential image of the forgetful functor $\Shv(\CY)^{\BG_m}\to \Shv(\CY)$.

\smallskip

\noindent{\em(c)} The resulting functor
$$\Shv(\CY)^{\BG_m}\to \Shv(\CY) \overset{\on{Av}^{\BG_a,\chi}_!}\longrightarrow \Whit(\CY)$$
factors as
$$\Shv(\CY)^{\BG_m}\overset{\text{\eqref{e:left adj *}}}\longrightarrow \on{Kir}_*(\CY)\to \Whit(\CY),$$
and the resulting functor $\on{Kir}_*(\CY)\to \Whit(\CY)$ is the inverse of \eqref{e:Whit to Kir}. 

\end{prop}

\begin{proof}

The proof follows from the Fourier-Deligne transform picture: we interpret the *-convolution action of
$\Shv(\BG_m)$ on $\Shv(\CY)$ as an action of the monoidal category $\Shv(\BA^1)$ with respect to the
pointwise $\sotimes$ tensor product. 

\end{proof}

\sssec{}

In the constructible situation we have a similarly defined equivalence
\begin{equation} \label{e:Whit to Kir dual}
\Whit(\CY)\to \on{Kir}_!(\CY)
\end{equation}
using the dual functors, i.e., 
$$\Whit(\CY)\hookrightarrow \Shv(\CY) \overset{\on{Av}^{\BG_m}_!}\longrightarrow \Shv(\CY)^{\BG_m}.$$

The inverse functor makes the following diagram commutative
$$
\CD
\Shv(\CY)^{\BG_m}  @>{\text{\eqref{e:right adj !}}}>>   \on{Kir}_!(\CY)  \\
@VVV   @VVV   \\
\Shv(\CY)  @>{\on{Av}^{\BG_a,\chi}_*}>>  \Whit(\CY). 
\endCD
$$

\ssec{The *-Kirillov vs !-Kirillov equivalence}

\sssec{}

Let us first be in the constructible situation when the Artin-Schreier sheaf is defined. Combining the equivalences
\eqref{e:Whit to Kir} and \eqref{e:Whit to Kir dual}, we obtain an equivalence
\begin{equation} \label{e:Kir !*}
\on{Kir}_*(\CY) \simeq  \on{Kir}_!(\CY). 
\end{equation}

Let us describe the corresponding functors
$$\on{Kir}_*(\CY) \leftrightarrow  \on{Kir}_!(\CY)$$
explicitly.

\sssec{}

Note that for any $\CG\in \Shv(\BG_a)^{\BG_m}$, we have the well-defined endo-functors of $\Shv(\CY)^{\BG_m}$
$$\CF\mapsto \CG\overset{*}\star \CF \text{ and } \CF\mapsto \CG\overset{!}\star \CF$$
intertwined by the forgetful functor $\Shv(\CY)^{\BG_m}\to \Shv(\CY)$ with the same-named endo-functors of $\Shv(\CY)$.

\sssec{}  \label{sss:descr Kir !*}

By unwinding the definitions, we obtain that the resulting functor
$$\on{Kir}_*(\CY) \to  \on{Kir}_!(\CY)$$
fits into the commutative diagram
$$
\CD
\Shv(\CY)^{\BG_m}  @>{j_*(\sfe)[3] \overset{!}\star-}>> \Shv(\CY)^{\BG_m}   \\
@V{\text{\eqref{e:left adj *}}}VV     @AAA   \\
\on{Kir}_*(\CY) @>>> \on{Kir}_!(\CY).
\endCD
$$

The functor 
$$\on{Kir}_*(\CY) \leftarrow  \on{Kir}_!(\CY)$$
fits into the commutative diagram
$$
\CD
\Shv(\CY)^{\BG_m}  @<{j_!(\sfe)[-1] \overset{*}\star-}<<  \Shv(\CY)^{\BG_m}   \\
@AAA     @VV{\text{\eqref{e:right adj !}}}V   \\
\on{Kir}_*(\CY) @<<< \on{Kir}_!(\CY).
\endCD
$$

In the above formulas $j$ denotes the open embedding 
$$\BG_a-0\hookrightarrow \BG_a,$$
and $\sfe\in \Shv(\BG_a-0)$ stands for the constant sheaf on $\BG_a-0$.

\sssec{}

Let us now be in the constructible situation, but where the Artin-Schreier sheaf is not necessarily defined. Note that the endo-functors of
$\Shv(\CY)^{\BG_m}$ given by $j_*(\sfe)[3] \overset{!}\star-$ and $j_!(\sfe)[3] \overset{*}\star-$ respectively, define a pair of mutually 
adjoint functors
\begin{equation} \label{e:Kir !* expl}
\on{Kir}_*(\CY) \rightleftarrows  \on{Kir}_!(\CY)
\end{equation}

We claim:

\begin{prop} 
The functors \eqref{e:Kir !* expl} are mutually inverse equivalences.
\end{prop}

\begin{proof}

By Lefschetz principle, we can reduce to the situation when $k=\BC$ and the sheaf theory is that
of constructible sheaves in the classical topology with coefficients in $\sfe=\BC$. In the latter case,
we can apply Riemann-Hilbert and thus embed our situation into that of holonomic D-modules.
Now the assertion follows from \secref{sss:descr Kir !*} above.

\end{proof}

\end{document}